\documentclass[a4paper,11pt]{amsbook}
\usepackage{amsmath}
\usepackage{amsthm}
\usepackage{amsfonts}
\usepackage{amssymb}
\usepackage{latexsym}
\usepackage{dutchcal} 
\usepackage{enumitem}

\usepackage[section]{placeins}

\usepackage[abs]{overpic}
\usepackage{imakeidx}

\usepackage[titletoc]{appendix}

\usepackage[usenames]{color}
\usepackage[all]{xy}

\usepackage{graphicx}
\usepackage{latexsym}
\usepackage[cp850]{inputenc}
\usepackage{epsfig}
\usepackage{psfrag}

\usepackage{thmtools}
\usepackage{thm-restate}

\usepackage{hyperref}
\usepackage{cleveref}

\numberwithin{section}{chapter}
\numberwithin{figure}{section}

\newtheorem{theorem}{Theorem}[section]{\bf}{\it}
\newtheorem{lemma}[theorem]{Lemma}{\bf}{\it}
\newtheorem{proposition}[theorem]{Proposition}{\bf}{\it}
\newtheorem{corollary}[theorem]{Corollary}{\bf}{\it}
{\bf}{\it} 
{\bf}{\it}
\newtheorem*{theorem*}{Theorem}

\newtheorem{remark}[theorem]{Remark}
\newtheorem*{remark*}{Remark}
{\bf}{\it}
{\bf}{\it}

{\bf}{\it}
{\bf}{\it}
\newtheorem{convention}[theorem]{Convention}{\bf}{\it}
\newtheorem{example}[theorem]{Example}
\newtheorem{definition}[theorem]{Definition}
\newtheorem{notation}[theorem]{Notation}

\newtheorem{standing}[theorem]{Standing assumptions}{\bf}{\it}

\theoremstyle{remark}

\theoremstyle{definition}

\theoremstyle{remark}

\newcommand{\diam}{\operatorname{diam}}

\newcommand{\R}{\mathbb R}
\newcommand{\Z}{\mathbb Z}

\newcommand{\N}{\mathbb N}

\newcommand{\loc}{{\operatorname{loc}}}

\newcommand{\dist}{{\operatorname{dist}\,}}
\newcommand{\id}{{\operatorname{id}}}

\newcommand{\pr}{{\operatorname{pr}}}

\newdimen\vintkern\vintkern11pt
\def\vint{-\kern-\vintkern\int}

\newcommand{\norm}[1]{\lVert #1 \rVert}

\newcommand{\bS}{\mathbb{S}}

\newcommand{\cI}{\mathcal{I}}

\newcommand{\cF}{\mathcal{F}}
\newcommand{\cL}{\mathcal{L}}

\newcommand{\cP}{\mathcal{P}}

\newcommand{\interior}{\mathrm{int}}

\newcommand{\Star}{\mathrm{St}}
\newcommand{\Link}{\mathrm{Lk}}

\newcommand{\cdiff}{\mathcal{Diff}}
\newcommand{\cD}{\mathcal{D}}

\newcommand{\ind}{\mathrm{indent}}

\newcommand{\cT}{\mathcal{T}}

\newcommand{\cC}{\mathcal{C}}
\newcommand{\cQ}{\mathcal{Q}}
\newcommand{\cR}{\mathcal{R}}

\newcommand{\sfK}{\mathsf{K}}

\newcommand{\Diff}{\mathsf{Diff}}

\newcommand{\Wedge}{\mathrm{Wedge}}

\newcommand{\cA}{\mathcal{A}}

\newcommand{\fL}{\mathfrak{L}}

\newcommand{\cU}{\mathcal{U}}

\newcommand{\CW}{\mathrm{CW}}

\newcommand{\Clover}{\mathrm{Clover}}
\newcommand{\cl}{\mathrm{cl}}

\newcommand{\Exp}{\mathsf{Exp}}
\newcommand{\EXP}{\mathsf{EXP}}

\newcommand{\X}{\mathsf{X}}

\newcommand{\SL}{\mathrm{sl}}

\newcommand{\cS}{\mathcal S}

\newcommand{\sK}{\mathsf{K}}

\renewcommand{\emptyset}{\varnothing}

\newcommand{\Refine}{\mathrm{Ref}}
\newcommand{\Core}{\mathrm{Core}}

\newcommand{\sL}{\mathsf{L}}

\newcommand{\Comp}{\mathrm{Comp}}
\newcommand{\Real}{\mathrm{Real}}

\newcommand{\Rec}{\mathsf{Rec}}
\newcommand{\Tr}{\mathsf{Tr}}

\newcommand{\Span}{\mathrm{Span}}
\newcommand{\sfJ}{\mathsf{J}}

\newcommand{\sfE}{\mathsf{E}}
\newcommand{\sfC}{\mathsf{C}}
\newcommand{\sfQ}{\mathsf{Q}}
\newcommand{\sfR}{\mathsf{R}}
\newcommand{\sfT}{\mathsf{T}}

\newcommand{\sfRC}{\mathsf{RC}}

\newcommand{\sF}{\mathcal{F}}

\newcommand{\Channel}{\mathsf{Ch}}

\newcommand{\Dent}{\mathsf{Dent}}

\newcommand{\TOP}{\mathrm{top}}
\newcommand{\BOT}{\mathrm{bot}}

\newcommand{\scrS}{\mathcal S}
\newcommand{\PS}{\mathsf{PS}}

\newcommand{\Lfundamental}{\sL^\dagger}

\newcommand{\Ldef}{\sL_{\mathrm{def}}}
\newcommand{\op}{\mathrm{op}}
\newcommand{\tr}{\mathrm{tr}}

\newcommand{\Layer}{\mathrm{Layer}}

\newcommand{\wall}{\mathrm{wall}}
\newcommand{\ceiling}{\mathrm{ceiling}}
\newcommand{\floor}{\mathrm{floor}}

\newcommand{\PC}{\mathsf{PC}}

\newcommand{\spare}{\mathrm{spare}}

\title[Quasiregular cobordism]{Quasiregular cobordism theorem}

\author{Pekka Pankka}
\address{Department of Mathematics and Statistics, P.O. Box 68 (Pietari Kalmin katu 5), FI-00014 University of Helsinki, Finland}

\email{pekka.pankka@helsinki.fi}

\author{Jang-Mei Wu}
\address{Department of Mathematics, University of Illinois,  1409 West Green Street, Urbana, IL 61822, USA}
\email{jmwu@illinois.edu}

\date{\today}

\thanks{This work was supported in part by the Academy of Finland projects \#256228, \#297258 and \#332671, and a grant from the Simons Foundation \#353435. This material is partly based upon work supported by the National Science Foundation under Grant No. DMS-1928930 while P.P. participated in a program hosted
by the Mathematical Sciences Research Institute in Berkeley, California, during
the Spring 2022 semester.}
\date{\today}
\subjclass[2010]{Primary 30C65; Secondary 57M12, 30L10}

\makeindex

\begin{document}

\maketitle

\cleardoublepage
\thispagestyle{empty}
\begin{center}
Seppo Rickman (1935--2017) in memoriam.
\end{center}
\cleardoublepage

\setcounter{tocdepth}{1}
\tableofcontents

\chapter{Introduction}

In this article, we prove a quantitative theorem on the existence of quasiregular mappings from a compact oriented Riemannian $n$-manifolds with boundary to the $n$-sphere $\bS^n$ with finitely many mutually disjoint $n$-balls removed.
The result in its simplest form states as follows.  

\begin{restatable}{theorem}{introthm:cobordism_short}
\label{intro-thm:cobordism_short}
Let $n\geq 3$, $m\geq 2$, $M$ be a compact connected oriented Riemannian $n$-manifold with $m$ boundary components, and $B_1,\ldots, B_m$ be pairwise disjoint closed Euclidean $n$-balls in $\bS^n$. 
Then there exists a constant $\sK=\sK(n,M, B_1,\ldots, B_m)\ge 1$ such that, for each $d_0\in \N$, there exists a surjective mapping
\[
M \to \bS^n \setminus \interior (B_1 \cup \cdots \cup B_m)
\]
of degree at least $d_0$, which is $\sK$-quasiregular in the interior of $M$.
\end{restatable}

Quasiregular maps between manifolds are natural generalization of planar holomorphic maps. Non-constant quasiregular mappings are sense preserving branched covering maps (discrete and open maps) with controlled geometry. More precisely, a continuous mapping $f\colon M \to N$ between oriented Riemannian $n$-manifolds ($n\geq 2$) is \emph{$\sK$-quasiregular} for $\sK \ge 1$, if $f$ belongs to the Sobolev space $W^{1,n}_\loc(M;N)$ and satisfies the distortion inequality
\[
\norm{Df}^n \le \sK J_f\quad \text{a.e.}\ M,
\]
where $\norm{Df}$ is the norm of the weak differential $Df$ of $f$ and $J_f$ is the Jacobian determinant of $f$; we refer to Rickman's monograph \cite{Rickman_book} for the theory of quasiregular mappings. \index{quasiregular map} 

Theorem \ref{intro-thm:cobordism_short} states that, in dimensions $n \ge 3$, there is a threshold for distortion, above which the global degree of a quasiregular mapping is not limited by its distortion. The interest to this statement stems from results of Martio \cite{Martio_capacity_1970} and Sarvas \cite{Sarvas_Hausdorff_1975} on the relationship between  local degree and distortion of quasiregular mappings in dimensions $n\ge 3$. In particular, the distortion of the mapping bounds the Hausdorff dimension of the sets of large local index; see \cite[Chapter III]{Rickman_book} for a detailed discussion on these results.

The mapping $f\colon M \to \bS^n\setminus \interior(B_1\cup \cdots \cup B_m)$ in Theorem \ref{intro-thm:cobordism_short} is a \emph{(generalized) branched cover} in the sense that it is orientation preserving, discrete, and open map. Recall that a mapping is \emph{discrete} if its pre-images are discrete and \emph{open} if it maps open sets to open sets.

From a purely topological point of view, extension of branched covering maps on the boundary of a $3$-manifold  to the interior  has been studied by Hirsch \cite{Hirsch-1977}  and Berstein and Edmonds \cite{Berstein-Edmonds_TAMS}. In particular, Theorem 6.2 in \cite{Berstein-Edmonds_TAMS} states that if $M$ is a compact oriented $3$-manifold $M$ with two  boundary components, $\Sigma_1$ and $ \Sigma_2$, and $f_j \colon \Sigma_j \to \bS^2\times \{j\}$,  $j=0,1$, are oriented branched coverings of  degree $d \geq 3$, then there is a PL branched covering $ M\to \bS^2\times [0,1]$ which extends $f_0$ and $f_1$. This result has been extended by Heinonen and Rickman  \cite{Heinonen-Rickman_Duke} to $3$-manifolds $M$ with at least two boundary components; see also \cite{Pankka-Rajala-Wu}. 

Theorem \ref{intro-thm:cobordism_short} may be viewed as a higher dimensional analog of this theorem of Heinonen and Rickman. In both theorems, the methods stem from construction methods used to show the sharpness of Rickman's Picard theorem; see Rickman \cite{Rickman_Acta} for $n=3$ and Drasin and the first named author \cite{Drasin-Pankka} for $n\ge 4$.

The theorem of Heinonen and Rickman in  \cite{Heinonen-Rickman_Duke} and Theorem \ref{intro-thm:cobordism_short} are far from optimal degree. Regarding optimal degree Piergallini and Zuddas  \cite{Piergallini_Zuddas_2018}  proved, in dimension $n=4$, that every compact connected oriented  PL $4$-manifold $M$ with $p (\geq 0)$ boundary components admits a 'simple' branched covering map $M\to \bS^4\setminus  \interior (B_1\cup \cdots \cup B_p )$, where $B_i$'s are pairwise disjoint Euclidean balls, of degree either $4$ or $5$.

\medskip

Before we discuss applications of Theorem \ref{intro-thm:cobordism_short} and its proof, we restate the result in a stronger form, which provides information of the  boundary map, and allows the domain and the target to have different numbers of boundary components. We recall the terminology after the statement.

\begin{restatable}[Quasiregular cobordism theorem]{theorem}{introthmqrcobordism}
\label{intro-thm:qrcobordism}
Let $n\geq 2$,  $m\ge p\geq 2 $, and $M$ be a compact connected oriented Riemannian $n$-manifold with $m$ boundary components. Let $N=\bS^n \setminus \interior (B_1\cup \cdots \cup B_p)$ be an oriented manifold, where $B_1,\ldots, B_p$ are pairwise disjoint closed Euclidean balls, and  $c\colon \cC(K) \to \{1,\ldots,p\}$ be a surjection defined on the collection $\cC(K)$ of connected components of $\partial M$. Then there exists a constant $\sK=\sK(n,M, N)\ge 1$ for the following.

For each $d_0\in \mathbb N$, there exists a $\sK$-quasiregular map
\[
f\colon M \to \bS^n \setminus \interior (B_1\cup \cdots \cup B_p),
\]
of degree at least $ d_0$, for which the boundary map $f|_{\partial M}$ is BLD and for each $\Sigma\in \cC(K)$,  $f|_{\Sigma} \colon \Sigma\to \partial B_{c(\Sigma)}$  is an Alexander map expanded by simple covers. 
\end{restatable}

The term Alexander map stems from a 1920 article of J.~W.~Alexander \cite{Alexander}, in which he showed that every closed oriented piecewise linear  $n$-manifold $M$ can be triangulated to admit an orientation-preserving  branched covering map  $M\to \bS^n$, which maps neighboring $n$-simplices to the upper and the lower hemispheres of $\bS^n$, respectively. In the literature, any such map from a triangulated $n$-manifold (not necessarily simplicial and possibly having boundary) to $\bS^n$ is called an \emph{Alexander map}.

A cubical $n$-complex $K$ is a complex, analogous to a simplicial complex, whose elements are $k$-cubes,  $0\leq k \leq n$. 
Every Riemannian $n$-manifold $(M,g)$ supports a cubical complex $K$ with a flat structure and  a flat metric $d_K$ for which $(|K|, d_K)$ is quasi-similar to $(M,g)$; see Proposition \ref{prop:Riemannian-to-cubical}. Under the flat metric, all $n$-cubes are isometric. Thus the barycentric triangulation $K^\Delta$ of $K$ yields a simplicial structure in which all $n$-simplices are bilipschitz equivalent.
It is with respect to this structure and this metric, the Alexander maps on $\partial M$ are defined in  Section \ref{sec:Alexander-maps}

The notion of  simple cover stems from Rickman's article \cite{Rickman_Acta}. 
We use simple covers  to adjust degrees of mappings - each simple cover increases the multiplicity of a map by one; see Figure \ref{fig:Alexander_map}. This concept is formalized  in Section \ref{sec:simple-cover}.

\begin{figure}[h!]
\begin{overpic}[scale=0.8,unit=1mm]{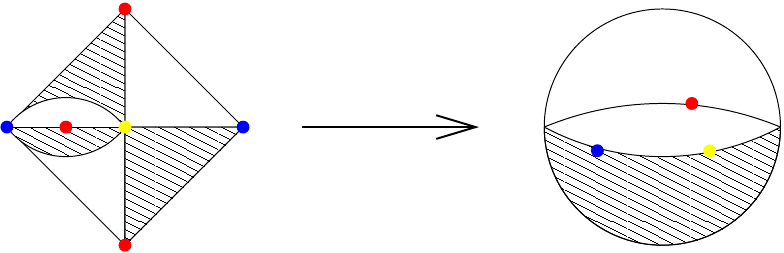} 
\end{overpic}
\caption{A degree $2$ Alexander map expanded by a simple cover; }
\label{fig:Alexander_map}
\end{figure}

Finally, a mapping $f\colon X\to Y$ between metric spaces is a \emph{mapping of bounded length distortion}, or an \emph{$L$-BLD mapping for $L\ge 1$}, if $f$ is discrete and open and \index{BLD map}
\[
\frac{1}{L} \ell(\gamma) \le \ell(f\circ \gamma) \le L \ell(\gamma)
\]
for all paths $\gamma$ in $X$, where $\ell(\cdot)$ is the length of the path. BLD mappings are quasiregular; we refer to Heinonen and Rickman \cite{Heinonen-Rickman_Duke} about other properties of BLD maps.

\begin{remark}[Remark on the name of the theorem]
We call Theorem \ref{intro-thm:qrcobordism} a 'cobordism theorem' as it provides a (virtual) extension of cubical Alexander maps over a cobordism $M$ of the boundary $\partial M$ in the following sense. The space $M$ admits a cubical structure, whose barycentric subdivision admits simplicial Alexander maps from the boundary components of $\partial M$ onto spheres $\partial B_j$ for $j=1,\ldots, p$. These Alexander maps extend over $M$ modulo a modification with a number of so-called \emph{simple covers}. Heuristically these simple covers balance these Alexander maps so that their extensions can be combined inside $M$. 
\end{remark}

\section{Applications}

An immediate consequence of Theorem \ref{intro-thm:qrcobordism} is an observation that in dimensions $n\ge 3$, there exist quasiregular maps $\bS^n \to \bS^n$ with preassigned preimages, at finitely many points. 

\begin{restatable}[Preassigned preimages]{theorem}{introthmpreimages}
\label{intro-thm:QR_preimages}
Let $n\ge 3$, $p\ge 1$, and let $z_1,\ldots, z_p$ be distinct points in $\bS^n$ and let $Z_1,\ldots, Z_p$ be mutually disjoint finite non-empty sets in $\bS^n$.  Then there exists a quasiregular map $f\colon \bS^n \to \bS^n$ satisfying $f^{-1}(z_i) = Z_i$ for each $i=1,\ldots, p$.
\end{restatable}
This property of quasiregular mappings $\bS^n\to\bS^n$ for $n \ge 3$ is in a striking contrast to quasiregular mappings $\bS^2\to \bS^2$; see Section \ref{sec:prescribed_preimages}.

\medskip

It was generally expected that the complexity of  branched sets and the high local index on a large set will  drive up the distortion of a quasiregular map. Martio, Rickman, and V\"ais\"al\"a conjectured \cite{MRV} in 1971 that:
\begin{quote}
\emph{Given $n\ge 3$ and $\sK\ge 1$, there exists a constant $c=c(n,\sK)\ge 1$ having the property: for any $\sK$-quasiregular mapping $f\colon \R^n \to \R^n$, the set $E_f=\{ x\in \R^n \colon i(x,f) \ge c\}$ does not have accumulation points.}
\end{quote}
Here $i(x,f)$ is the local index of the map $f$ at a point $x\in \R^n$.

Rickman's local index theorem \cite{Rickman-AASF} provides a counterexample to this conjecture in dimension $n=3$. 
Applying Theorem \ref{intro-thm:qrcobordism}, we obtain counterexamples, similar to Rickman's, for dimension $n\geq 3$.

\begin{restatable}[Large local index]{theorem}{introthmRickmanAASF}
\label{intro-thm:Rickman-AASF}
Given  $n\ge 3$,  there exists a constant $\sK=\sK(n)\ge 1$ having the property:  for each $c \ge 1$ there exists a $\sK$-quasiregular mapping $F\colon \bS^n \to \bS^n$ of degree at least $c$ for which
\[
E_F = \{x \in \bS^n \colon i(x,F) = \deg(F)\}
\]
is a Cantor set.
\end{restatable}

The expected tension between the distortion $\sK$ and the local index $i(\cdot,F)$ behind the conjecture of Martio, Rickman, and V\"is\"al\"a stems from the observation that, for a quasiregular map $f\colon M\to N$ between $n$-manifolds, a large local index $i(x,f)$ at a point $x\in M$ has the effect of shrinking small neighborhoods of $x$ severely. Since the branch set $B_f$ of $f$, where $f$ is not a local homeomorphism, is either empty or has an image $f(B_f)$ of positive $(n-2)$-Hausdorff measure, this makes excessive shrinking impossible. A quantitative statement to this effect was obtained by Martio in \cite{Martio_capacity_1970}. Rickman and Srebro \cite{Rickman-Srebro} have shown that a  $\mathsf K$-quasiregular map can not maintain a high local index  on a large set of evenly distributed points in some quantitative sense.

\medskip

In a more topological vein, by applying  Rickman's $2$-dimensional deformation theory, Heinonen and Rickman constructed in \cite{Heinonen-Rickman_Topology}  a quasiregular mapping $\bS^3 \to \bS^3$ whose branch set contains an Antoine's necklace. Using  Theorem \ref{intro-thm:qrcobordism}, together with a quasi-self-similar wild Cantor set in $\R^4$ constructed in the Appendix, we prove the existence of a Heinonen-Rickman type map in dimension $n=4$.

\begin{restatable}[Wildly branching quasiregular map]{theorem}{introthmHeinonenRickmanTopologyfourdim}
\label{intro-thm:Heinonen-Rickman_Topology_fourdim} There exist a wild Cantor set $X\subset \R^4$, and constants $\sK\ge 1$, $c_0\ge 1$, and $m_0\ge 1$ for the following. For each $c\ge c_0$, there exist $c'\ge c$ and a $\sK$-quasiregular mapping $F \colon \bS^4\to \bS^4$ whose branched set is $X$ and whose local index   $i(x,F)= c'$ for each $x\in X$, and $i(x,F) \le m_0$ for each $x\in \bS^4\setminus X$. 

Furthermore, given $s_0\ge 1$, the mapping $F$ may be chosen for which 
\[
\frac{1}{C} \dist(x, X)^{s}\leq J_F(x) \leq C\dist (x,X)^{s}
\]
for some $s\ge s_0$ and  almost every $x\in \bS^4\setminus X$.
\end{restatable}

It is known that, for each $n\geq 3$ and $n\neq 4$, there exists a metric $n$-sphere $(S,d)$, nearly indistinguishable from $\bS^n$ by classical analysis in the sense advocated by Semmes \cite{Semmes}, which is  not a bilipschitz copy of $\bS^n$ but can be mapped onto $\bS^n$ by a BLD map. The example in dimension $3$  is due to  Semmes and \cite{Semmes} and Heinonen and Rickman \cite{Heinonen-Rickman_Topology,Heinonen-Rickman_Duke}, and in dimension $n\geq 5$ is the double suspension of a non-trivial homology $(n-2)$-sphere, observed by Siebenmann and Sullivan in\cite{Siebenmann-Sullivan}.

Theorem \ref{intro-thm:Heinonen-Rickman_Topology_fourdim} may be used to furnish an example of this type for dimension $n=4$; see Section \ref{sec:intro_BLD_parametrization}.

\bigskip

Theorem \ref{intro-thm:qrcobordism} can also be used to construct examples in quasiregular dynamics. The composition of two $\sK$-quasiregular maps is again  quasiregular but often has  distortion  larger than $\sK$. Thus in dynamics,  it is natural to consider only \emph{uniformly quasiregular maps}, or \emph{UQR maps}, $f\colon M \to M$ between Riemannian manifolds for which the mapping $f$ and all its iterates are $\sK$-quasiregular for a number $\sK >1$. \index{UQR map}

Using conformal traps, Iwaniec and Martin  \cite{Iwaniec-Martin} constructed UQR maps $\bS^n\to \bS^n$ whose Julia sets are tame Cantor sets. UQR maps in $\bS^3$ whose Julia sets are wild Cantor sets are constructed  by Fletcher and the second named author in \cite{Fletcher-Wu}. Theorem \ref{intro-thm:qrcobordism} may be used to extend the example of wild Julia set to $\bS^4$.

\begin{restatable}[Wild Julia set]{theorem}{introthmwildJuliasetdimfour}
\label{intro-thm:wild_Julia_set_dim4}\index{Julia set}
For each $k\in \N$, there exists a uniformly quasiregular map $\bS^4 \to \bS^4$ of degree at least $k$, whose Julia set is a wild Cantor set.
\end{restatable}

The restriction to dimension $n=4$ here and in Theorem \ref{intro-thm:Heinonen-Rickman_Topology_fourdim}  stems from the fact that at present we are only able to construct \emph{quasi-self-similar} Antoine-Blankinship's necklaces in $\bS^4$.
We refer to Blankinship \cite{Blankinship} and \cite{Pankka-Vellis} for the topological construction of wild Cantor sets in $\R^n$, $ n\ge 4$; see also \cite{Pankka-Vellis}.

These applications are discussed in more detail in Part \ref{part:Applications}.

\section{Outline of the proof}

The proof of the Quasiregular cobordism theorem spans from Part \ref{part:separating_complexes} to Part \ref{part:Alexander-Rickman}. Whereas parts \ref{part:separating_complexes} and \ref{part:deformation} are independent, parts \ref{part:QR-extension} and \ref{part:Alexander-Rickman} are methodologically independent but use the particular constructions in parts \ref{part:separating_complexes} and \ref{part:deformation}. We discuss now, in heuristic terms, the methods used in the proofs. For brevity, we discuss the case $p=m$ only.

Let $K$ be a cubical $n$-complex on the $n$-manifold $(M,g)$ with a flat structure and flat metric $d_K$ for which $(|K|, d_K)$ is quasi-similar to $(M,g)$.

Unlike the typical procedure for constructing maps - from  boundary  towards  interior - we build maps from the interior towards the boundary. Topology of the boundary of $M$ has no role in the construction.

We begin by proving that  there exists  an $(n-1)$-dimensional (separating) subcomplex $Z$ of $\Refine(K)$, which separates the space $|K|$ of $K$ into connected components $D_\Sigma$, one associated to each boundary component $\Sigma\in \cC(K)$, for which  $D_\Sigma$  is a collar $|\Sigma| \times [0,1)$; see Theorem \ref{thm:separating-complex-existence}. Refinement $\Refine(K)$  is a complex obtained from $K$ by subdividing each $\ell$-cube in $K$ into $N^\ell $ subcubes where $N$ is a fixed integer, and $\cC(K)$ is the collection of boundary component of $K$.

To construct mappings of arbitrarily large degrees having a uniformly bounded distortion, we modify the separating complexes iteratively in higher order refinements $\Refine^k(K)$ in such a way that the complementary components of the separating complexes $Z_k$ in $|K|$ are progressively  more intertwined, yet quasiconformally stable, as stated in Theorem \ref{theorem:evolution-short}.

This is achieved with an inductive channeling procedure: making indentation to the components of $|K|\setminus |Z_{k-1}|$,  then redistributing the space vacant by the dents in the form of tunnels to the dented components. Tunnels are complexes whose adjacency graphs are trees.  Channeling, resembling a geometrically controlled construction of Lakes of Wada,  is explained   in Part \ref{part:separating_complexes}.

As an initial thought, the construction of mappings, for a fixed $k$, would start with an Alexander maps $f \colon (Z_k)^\Delta \to \bS^{n-1}$ on the triangulated $Z_k$, and  follow with an extension of $f$ over  each component $D_\Sigma (=D_{k;\Sigma})$ separated by $Z_k$. An immediately obstacle occurs from the fact that $Z_k$ is rarely a manifold. Hence, in general, $(Z_k)^\Delta$ does not support a well-defined Alexander map and  the closure $\overline D_\Sigma$ of $D_\Sigma$ in $|K|$ is not $\Sigma \times [0,1]$. See see Figure \ref{fig:Separating_complex}. On the other hand, the closure $\widetilde D_\Sigma$ of $D_\Sigma$ with respect to the inner path metric in $D_\Sigma$ is always $|\Sigma|\times [0,1]$.

For this reason,  we lift the problem of constructing maps to 
a complex $\Real_\Sigma$, called a realization, supported on $\widetilde D_\Sigma$, which projects canonically down to the complex $\Refine^k(K)|_{\overline D_\Sigma}$  and whose     
 $n$-cubes are in one-to-one correspondence with those in $\Refine^k(K)|_{\overline D_\Sigma}$. The realization $\Real_\Sigma$ has an outer boundary component $\Sigma$ and  an inner boundary component $\Upsilon_\Sigma$ lifted from $Z_k \cap \overline D_\Sigma$. 
The inner component $\Upsilon_\Sigma$ is a manifold and supports a well-defined BLD Alexander map  $f_\Sigma\colon \Upsilon_\Sigma \to \bS^{n-1}$.   

The iterative channeling process endows the realization $\Real_\Sigma$ a tree-like structure $\cT_\Sigma$,
with the root at a dented collar of $\Sigma$ and the $j$-th generation vertices at the tunnels added at the $j$-th channeling step,  $1\le j\le k$, as described in Proposition  \ref{prop:general-localized-realization-structure}.

For the extension of Alexander map  $f_\Sigma \colon \Upsilon_\Sigma \to \bS^{n-1}$ to a BLD map $F_\Sigma \colon|\Real_\Sigma| \to \bS^{n-1} \times [0, N^k]$,  we establish a dimension-free deformation method in Part \ref{part:deformation} for Alexander mappings defined on shellable cubical complexes.  Deformation is largely topological, however, when applied on specific cubical modules, yields a BLD-controlled branched cover homotopy  from Alexander map $f_\Sigma$ to a  simpler map, modulo  degree adjustment.

Extension to $|\Real_\Sigma|$ is built along the tree $\cT_\Sigma$ on $\Real_\Sigma$  from leaves to root, one dented tunnel at a time, by local BLD deformations; this procedure is explained in Part \ref{part:QR-extension}. 
The extended maps $F_\Sigma$ on $|\Real_\Sigma|$, $\Sigma\in \cC(K)$, serve as the building blocks for the mapping asserted in Theorem \ref{intro-thm:qrcobordism}.

Gluing mappings $F_\Sigma, \Sigma\in \cC(K)$, together naively along $Z_k$ by projecting them down to $|K|$ will  produce at most a map $|K|\to \bS^{n-1} \times [0,N^k]$ whose target  has only two boundary components.
To  combine mappings $F_\Sigma$ properly, 
we partition the target $\bS^n\setminus \interior (B_1\cup\cdots \cup B_m)$, by a branched sphere $\mathbf{S}$, into spherical shells $E_1,\ldots, E_m$.
We construct, on the domain side,  a weaved approximation $X$ of $Z_k$ in $|K|$ whose complementary components $\delta_\Sigma, \,\Sigma\in  \cC(K),$ in $|K|$ are in one-to-one correspondence with $D_\Sigma$ bilipschitzly. Moreover, the space $X$ admits  a so-called Alexander-Rickman map $X \to \mathbf{S}$ which extends to a map  $|K|\to \bS^n\setminus \interior (B_1\cup\cdots \cup B_m)$ that maps $\delta_\Sigma$ to  $E_{c(\Sigma)}$.
The construction of  weaved approximation spans the entire Part \ref{part:Alexander-Rickman}.

The mapping $f$ asserted in Theorem \ref{intro-thm:qrcobordism} is obtained by gluing the domains $|\Real_\Sigma|$ of mappings $F_\Sigma$ along $X$;  $f$ has degree at least $N^k$.

\begin{figure}[htp]
\begin{overpic}[scale=0.16,unit=1mm]{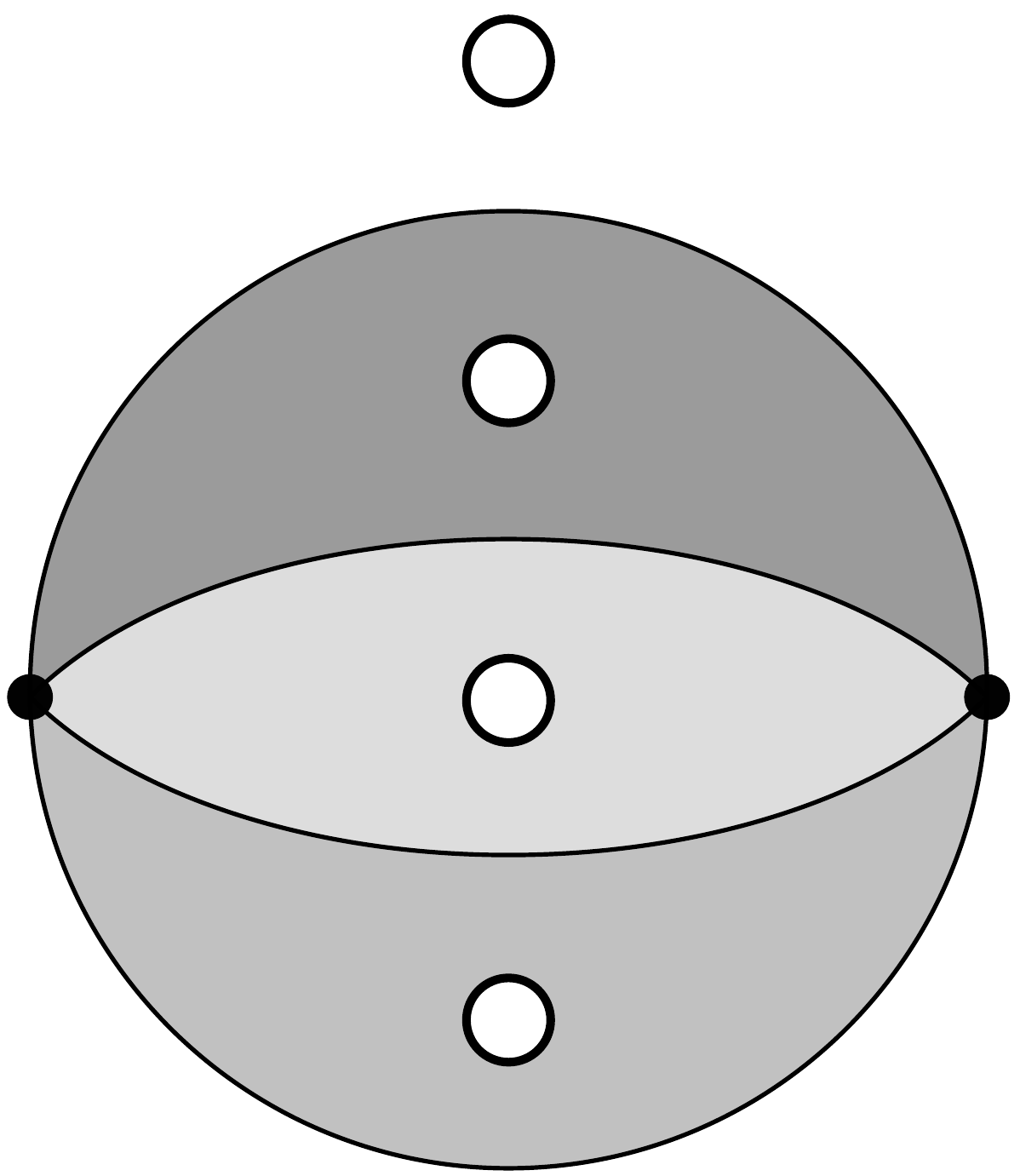} 
\put(18,36){\tiny$B_1$}
\put(18,5){\tiny$B_2$}
\put(18,15.5){\tiny$B_3$}
\put(18,25.5){\tiny$B_4$}

\put(27,12){$\mathbf{S}$}
\end{overpic}
\caption{$\bS^2\setminus \interior (B_1\cup\cdots \cup B_4)$ partitioned by branched sphere $\mathbf{S}$ into $4$ spherical shells.}
\label{fig:intro-bs-balls}
\end{figure}

A remark on the distortion is in order.
Extension to $|\Real_k(\Sigma)|$ is made in essentially disjoint dented tunnels in succession. 
Within each dented tunnel, the extension is built by BLD deformations together with local bilipschitz homeomorphisms. During this process, every point in the domain is moved by non-conformal maps at most a bounded number of times, with the number depending only on $n$ and $K$, before being mapped to the target.  Therefore the distortion of the final extension remains bounded by a constant depending only on $n$ and $K$, in particular independent of $k$. 

This concludes the heuristic description of the proof of Theorem \ref{intro-thm:qrcobordism}.

\section{Remarks on deformation and partitioning}

To close the introduction, we state two theorems on deformation of Alexander maps on shellable cubical complexes to give a hint of the higher dimension theory. 
We also state an example, of independent interest, emerging from the evolution of separating complexes.

\medskip

One of the methods Rickman introduced in his 1985 paper \cite{Rickman_Acta} on the sharpness of the Picard theorem for quasiregular mappings is the deformation of $2$-dimensional Alexander maps. Although it is possible to avoid this deformation method in \cite{Drasin-Pankka}, we do not see how it could be by-passed in the proof of Theorem \ref{intro-thm:qrcobordism} due to the conditions posed to the boundary map. To obtain a dimension free deformation theory, we restrict ourselves to deformation of Alexander maps on shellable cubical complexes.

A finite cubical $n$-complex $K$, whose space $|K|$ is an $n$-cell, is \emph{shellable} if there exists an order $Q_1,\ldots, Q_\ell$ for the $n$-cubes in $K$  for which the intersection $(Q_1\cup \cdots \cup Q_{k-1}) \cap Q_k$ is an $(n-1)$-cell for each $k=1,\ldots, \ell$. Not all cubical $n$-complexes of cells are shellable; see Remark \ref{rmk:non_shellable}.

The flat structure on a cubical complex $K$ provides the canonical triangulation 
$K^\Delta$ the stability needed for deformation; it is a local regularity condition. The shellability, on the other hand,  is a global condition which allows the underlying complexes and the mappings to be reduced inductively.

\begin{figure}[htp]
\begin{overpic}[scale=0.50,unit=1mm]{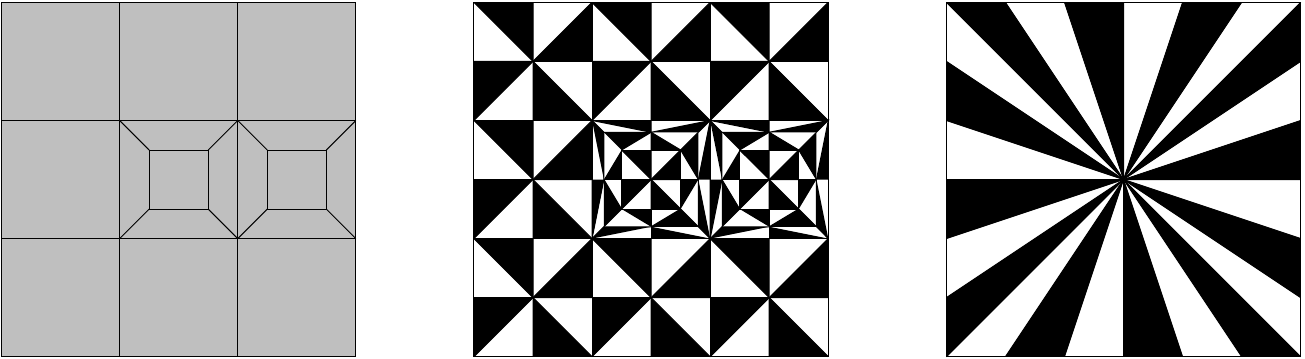} 
\put(-5,28){\tiny $K$}
\put(35,28){\tiny $K^\Delta$}
\put(75,28){\tiny $K^*$}
\end{overpic}
\caption{A shellable cubical $2$-complex $K$, its barycentric triangulation $K^\Delta$, and a star-replacement $K^*$ of $K^\Delta$.}
\label{fig:Cubical-shellable-Alexander_map}
\end{figure}

Given a cubical complex $K$ on an $n$-cell, we denote by $K^*$  a (combinatorially unique) simplicial complex which is  star of a vertex in the interior of $|K|$ and agrees with $\partial (K^\Delta)$ on the boundary, also denote by $(K^\Delta)^{[k]}$ the collection of $k$-simplices in the simplicial complex $K^\Delta$.

\begin{restatable}[Deformation of Alexander maps]{theorem}{introthmdeformation}
\label{thm:QR-deformation-Alexander} 
Let $n\ge 2$, $K$ be a shellable cubical $n$-complex, $K^\Delta$ the barycentric triangulation of $K$, and $K^*$ a star-replacement of $K^\Delta$. Let $f\colon |K|\to \bS^n$ be a $K^\Delta$-Alexander map, and  
\[
m = \left( \# (K^\Delta)^{[n]}-\# (K^*)^{[n]} \right)/2.
\]
Then there exists an $\sL(n,K,K^*)$-BLD homotopy 
\[ F\colon |K|\times [0,1]\to \bS^n\times [0,1]\quad (\text{modulo}\,\,\, |\partial K|) \]
from $f$ to an expansion of a $K^*$-Alexander map $|K|\to \bS^n$ by $m$ simple covers.
\end{restatable}

The classical Hopf degree theorem states that \emph{for a closed, connected, and oriented $n$-manifold $M$, two maps $M\to \bS^n$ are homotopic if and only if they have the same degree.} In particular, the homotopy class of a continuous map $M\to \bS^n$ is classified by a single integer.

Theorem \ref{thm:QR-deformation-Alexander}  leads to a version of the Hopf degree theorem for cubical Alexander maps between spheres.
We say a cubical complex $K$ on the sphere $\bS^n$ is \emph{shellable} if there exists an $n$-cube $Q\in K$ for which $K\setminus \{Q\}$ is a shellable complex on the $n$-cell $\bS^n\setminus \interior Q$.

\begin{restatable}[Hopf theorem for Alexander maps]{theorem}{introthmHopftheoremvone}
\label{intro-thm:Hopf_theorem_v1}
Let $n\geq 2$ and $K_1$ and $K_2$ be two shellable cubical $n$-complexes on $\bS^n$ having the same number of $n$-cubes. Then a $K_1^\Delta$-Alexander map and a $K_2^\Delta$-Alexander map, with the same orientation, are quasiregularly homotopic.
\end{restatable}

\medskip

The partition methods in Part \ref{part:separating_complexes} can be used to produce geometric Wada continua in dimensions $n\ge 3$. Recall that a \emph{Wada continuum} is a compact connected subset of $\bS^n$ which is the common boundary of its (at least three) complementary components. These complementary domains are called Lakes of Wada; for historical background, see  Section \ref{sec:Wada}.

An outcome of evolution of separating complexes, repeated indefinitely,  is the paradoxical example of Lakes of Wada in $\bS^n, n\geq 3,$ which are quasiconformal to Euclidean balls. 
This  result is also contained implicitly in \cite{Drasin-Pankka}, where it is obtained by a different method.

\begin{restatable}[Lakes of Wada]{theorem}{introthmWadaRiemannianmanifold}
\label{intro-thm:Wada-on-sphere} 
Given $n\ge 3$ and $m \ge 2$, there exists a compact connected continuum $X$ in $\bS^n$ whose complement  has exactly $m$ components, $ M_1,\ldots, M_m$,  for which
\[ 
X= \partial M_1 = \cdots = \partial M_m  =\bigcap_{i=1}^m\,  \overline{M_i},
\]  
and each component $M_j $ is quasiconformal to the open Euclidean  $B^n(0,1)$.
\end{restatable}

\medskip

\noindent {\bf Acknowledgments} We thank David Drasin for his continuous support on this project. We also thank Julie Kaufman for the $3$-dimensional drawings.


\part{Preliminaries}

\chapter{Preliminaries on complexes}

\section{Cubical complexes}\label{sec:prelim-cubical-complex}
In this section we define a class of cubical complexes, which give the underlying structure of our constructions.

In what follows, a \emph{complex $K$} is a collection of sets with the property that, for $Q,Q'\in K$, the intersection $Q\cap Q'$ belongs to $K$. As usual, a subset $K'\subset K$ is a \emph{subcomplex} if $K'$ is a complex. In particular, given $Q\in K$, the restriction of the complex $K$ to $Q$, that is, the collection 
\[
K|_Q = \{ Q' \cap Q\colon Q'\in K\},
\]
is a complex and a subcomplex of $K$. The \emph{space $|K|$ of a complex} is the union of all its elements, that is, $|K| = \bigcup_{Q \in K} Q$. The topology of $|K|$ is the topology induced by the closed subsets of elements in $K$.  

We say that $(K_1,\ldots, K_\ell)$ is an essential partition of a complex $K$, if $K_1,\ldots, K_\ell$ are subcomplexes of $K$ for which $(K_i \cap K_j)^{[n]} = \emptyset$ for $1\leq i< j\leq \ell$, and $K= K_1\cup \cdots \cup K_\ell$.

We say that a complex $K$ is \emph{connected} if for elements $Q,Q'\in K$, there exists a chain $Q=Q_0,\ldots, Q_m=Q'$ of elements of $K$ for which $Q_{i-1}\cap Q_i \ne \emptyset$ for $i\in \{1,\ldots, m\}$. Equivalently, a complex $K$ is connected if and only if its space $|K|$ is connected. A non-empty subcomplex $P \subset K$ is a \emph{connected component of $K$} if $P$ is connected and $P \cap \Span_K(K\setminus P) = \emptyset$. Equivalently, a non-empty subcomplex $P$ is connected component of $K$ if and only if $|P|$ is connected and open in $|K|$. \index{cubical complex!connected}

Given a homeomorphism $\phi \colon X\to |K|$, we denote 
\[
\phi^*K = \{ \phi^{-1}Q \colon Q\in K \}
\]
the \emph{pull-back of the complex $K$ under $\phi$}; clearly $|\phi^*K| = X$.

As usual, a \emph{simplicial complex} is a complex whose elements are cells with the standard simplicial structure. The collection of non-negative integers is denoted by $\N$.

For the definition of cubical complexes, we call, for each $n\in \N\cup \{-1\}$, the complex $\sfC_n=\{\{0\},\{1\}, [0,1]\}^n$ a \emph{standard cubical structure on $[0,1]^k$}; in addition we set $\sfC_{-1}=\{\{0\},\{1\}, [0,1]\}^{-1}=\emptyset$ and $\sfC_0 = \{\{0\},\{1\}, [0,1]\}^0=\{\{0\}\}$. In particular, the cubical structure $\sfC_n$ of $[0,1]^n$ is the product 
\[
\sfC_{n-1}\times \sfC_1 = \{ q \times q' \colon q \in \sfC_{n-1},\ q' \in \sfC_1\}
\]
of the standard cubical structures $\sfC_{n-1}$ and $\sfC_1$ of $[0,1]^{n-1}$ and $[0,1]$, respectively. Given an $n$-cell $Q$ and a homeomorphism $\phi \colon Q\to [0,1]^n$, we call the complex 
\[
\sfC_\phi(Q) = \{ \phi^{-1}q \colon q\in \sfC_n\} = \phi^*\sfC_n
\]
a \emph{standard cubical structure (induced by $\phi$) on $Q$}. If there is no ambiguity on the parametrizing homeomorphism, we simply denote $\sfC(Q) = \sfC_\phi(Q)$. The $(n-1)$-cubes and $(n-2)$-cubes in $\sfC_\phi(Q)$ are called \emph{faces and edges of $Q$ (with respect to $\sfC_\phi(Q)$)}, respectively. Each $k$-cube $q\subset Q$ has a unique \emph{opposite $k$-cube $q^\op$ in $Q$} satisfying $q^\op \cap q' = \emptyset$ for all $(n-k)$-cubes $q'\in Q$ which intersect $q$.

\begin{definition}\index{cubical complex}
A complex $K$ is a \emph{cubical $n$-complex (for $n\in \N$)} if the following conditions are satisfied:
\begin{enumerate}
\item $K$ is finite;
\item elements of $K$ are cells;
\item for each $n$-cell $Q\in K$, the restriction $K|_Q$ is a standard cubical structure on $Q$; 
\item each cell in $K$ is contained in an $n$-cell; and
\end{enumerate}
For each $k\in \N$, the $k$-cells of $K$ are called $k$-cubes.
\end{definition}

Given a cubical $n$-complex $K$ and $k\in \N$, we denote
\[
K^{[k]} = \{ q \in K \colon q \text{ is a $k$-cube}\}
\quad
\text{and}
\quad
K^{(k)} = \bigcup_{\ell \le k} K^{[\ell]}.
\]
We call the subcomplex $K^{(k)} \subset K$ the \emph{$k$-skeleton of $K$}.

Similarly as for cubes, given cubical complexes $K$ and $K'$, we have the product cubical complex
\[
K \times K' = \{ q \times q' \colon q\in K,\ q'\in K'\}.
\]

Our terminology is mostly standard with a few exceptions. 
Let $S\subset K$ be a subset of a cubical $n$-complex $K$. The subcomplex $\Span_K(S)$ \emph{spanned by $S$} is the smallest subcomplex of $K$ containing $S$. The \emph{star $\Star_K(q)$ of a $k$-cube $q$ in $K$} is the smallest subcomplex spanned by all cubes in $K$ containing $q$, that is,
\[
\Star_K(q)=\Span_K(\{Q\in K\colon q\subset Q\}).
\]
 Note that this definition is more restrictive that the usual definition, which states that a star is spanned by all cubes meeting $q$. The \emph{star $\Star_K(S)$ of a subset $S\subset K$} is \index{$\Star_K(S)$} 
\[
\Star_K(S) = \bigcup_{q\in S} \Star_K(q).
\]

Given subcomplexes $P$ and $N$ of $K$, we also denote 
\[
P - N = \Span_K(\{ q \in P \colon q\not \in N\})
\]
the \emph{difference of $P$ and $N$}. \index{cubical complex!difference $P - N$}

Regarding boundary of a cubical complex,  we say that an $(n-1)$-cube $q$ in a cubical $n$-complex $K$ is a \emph{one-sided} if there exists only one $n$-cube in $K$ containing $q$. 
The boundary $\partial K$ is the subcomplex of $K$ spanned by all one-sided cubes, that is,
\[
\partial K = \Span_K\left( \{ q\in K^{[n-1]} \colon q \text{ is a boundary cube}\right).
\]

Given cubical complexes $K$ and $K'$, a homeomorphism $\phi \colon |K|\to |K'|$ is a \emph{cubical isomorphism} if $\phi^*(K') = K$.  We say a continuous map $\varphi \colon |K|\to |K'|$ is \emph{cubical} if the map $\varphi_* \colon K\to K', Q\mapsto \varphi(Q)$, is a well-defined map of cubical complexes.

\subsubsection*{Flat Structure and flat metric}

To keep track of the distortion of the cubes resulted from a subdivision or an expansion of a complex, we first define a flat structure  as follows.

\begin{definition}
\label{def:flat-structure}
A \emph{flat structure $\sF$ on a cubical $n$-complex $K$} is a family of cubical isomorphisms $\{ \phi_Q \colon Q\to [0,1]^n \colon Q\in K^{[n]}\}$ having the property that, for $n$-cubes $Q$ and $Q'$ having a common face $q=Q\cap Q' \in K^{[n-1]}$, the transition map
\[
\phi_{Q'} \circ \phi_Q^{-1}|_{\phi_Q(q)} \colon \phi_Q(q) \to \phi_{Q'}(q)
\]
is a Euclidean isometry. \index{cubical complex!flat structure}
\end{definition}

Let  $K$ be a cubical $n$-complex with a flat structure $\sF$. For a continuum $\beta\subset |K|$, define \index{cubical complex!flat metric}
\[
\diam_{\sF}(\beta) = \sum_{Q\in K^{[n]}} \diam \phi_Q(\beta\cap Q), 
\]
where $\phi_Q \colon Q\to [0,1]^n$ belongs to $\sF$. The \emph{flat metric  $d_\sF\colon |K|\times |K|\to [0,\infty)$ induced by $\sF$} is defined by the formula,
\[
d_\sF(x,x') = \inf_{\beta} \diam_{\sF}(\beta),
\]
for $x,x'\in |K|$, where the infimum is taken over all continua $\beta\subset |K|$ containing $x$ and $x'$. \index{metric!flat}

Metric $d_\sF$ has the following property resembling that of $d_T$: for an $n$-cube $Q\in K^{[n]}$ and $x, x'\in Q$,
\[
d_\sF(x,x') = \inf_\gamma \ell_\cF(\gamma),
\]
where the infimum is taken over all paths in $Q$ connecting $x$ and $x'$, and $\ell_\sF(\cdot)$ is the length of a path in metric $d_\sF$.

\begin{convention}
Unless otherwise stated, we assume that we have fixed a flat structure $\sF_K$ for a cubical complex $K$ and that its space $|K|$ is a metric space with flat metric $d_{\sF_K}$. 
\end{convention}

\subsection{Passing from PL structure to a cubical structure}
\label{sec:congruent-cubulation} 

In the subdivisions below,  we use the common terminology that $n$-cells $C_1,\ldots, C_\ell$ \emph{subdivide an $n$-cell $C$} if cells $C_1,\ldots, C_\ell$ have mutually disjoint interiors and $C=C_1\cup \cdots \cup C_\ell$. 

Let $T$ be a simplicial complex for which $|T|$ is an $n$-manifold with boundary. \index{cubical complex!subdivision from PL}

Let $\Delta_n = [e_1,\ldots, e_{n+1}] \subset \R^{n+1}$ be the standard $n$-simplex, and fix for each $n$-simplex $\sigma\in T^{[n]}$ an affine simplicial isomorphism $\phi_\sigma \colon \sigma \to \Delta_n$. Let
\[
\sF(T) = \{ \phi_\sigma \colon  \, \sigma\in T^{[n]}\}.
\]
We observe that, for maps $\phi \colon \sigma \to \Delta_n$ and $\phi'\colon \sigma' \to \Delta_n$ in $\sF(T)$ and $\tau = \sigma\cap \sigma'\ne \emptyset$, the transition map $\phi' \circ \phi^{-1}|_{\phi(\tau)} \colon \phi(\tau) \to \phi'(\tau)$ is a Euclidean isometry. In this sense, we may consider $\sF(T)$ as a flat structure on $T$. 

We now associate to $T$ and $\sF(T)$ a cubical $n$-complex $T^\boxplus$ and an induced flat structure as follows.

To pass from a simplicial structure to a cubical structure, we subdivide each $n$-simplex in $T$ into $n$ polyhedra with isomorphic cubical structures. Although the construction is well-known, we recall some details on the geometry of the obtained cubes which have an important role in what follows. 

In the following statement, an $n$-cube $Q$ is a polyhedron with a cubical structure $\sfC(Q)$ isomorphic to $\sfC_n$ and that two $n$-cubes $Q$ and $Q'$ are \emph{congruent} if there exists a cubical isomorphism $\psi \colon Q \to Q'$ which is an isometry.

\begin{lemma}
\label{lemma:congruent-cubulation} 
There exists a cubical complex $\Delta_n^\boxplus$ having $n$-simplex $\Delta_n= [v_0,v_1,\ldots, v_n]$ as its space and $(n+1)$ congruent $n$-cubes $Q_0,\ldots, Q_n$, for which
the restriction $\Delta_n^\boxplus|_\sigma$, of $\Delta_n^\boxplus$  to each face $\sigma = [v_0,\ldots, \widehat{v_i}, \ldots, v_n]$ of $\Delta_n$, is a well-defined cubical $(n-1)$-complex and with $n$ congruent $(n-1)$-cubes.

Moreover, there exist cubical isomorphisms $h_i \colon Q_i \to [0,1]^n$ for $i=0, 1,\ldots,n,$ for which the transition maps $h_i\circ h_j^{-1}|_{h_j(Q_i\cap Q_j)} \colon h_j(Q_i\cap Q_j) \to h_i(Q_i\cap Q_j)$ are Euclidean isometries.
\end{lemma}
\begin{proof}

Let $X$ be the barycentric subdivision of $\Delta_n$ for which simplices of the same dimension are congruent.  For each vertex $v_i$ of $\Delta_n$, let $\Star(v_i)$ be the star of $v_i$ in the complex $X$ and let $Q_i=|\Star(v_i)|$ be its space. 

We claim  that each $Q_i$ admits a cubical structure $\sfC(Q_i)$ isomorphic to the standard one $\sfC_n$ on $[0,1]^n$ for which each  $k$-cube in $\sfC(Q_i)$, $0\leq k\leq n$, 
is the union of certain $k$-simplices in $X$. The argument is an induction on dimension.

For $n=1$, the claim clearly holds. Suppose the claim holds in dimension $n-1$. For dimension $n$, it suffices to consider the vertex $v_0 = e_1$. By the induction assumption,  for each $i=1,\ldots, n$,  the restriction of the star $\Star(v_0)$ to the face $[v_0,v_1,\ldots,v_{i-1}, v_{i+1}, \ldots, v_n]$  admits a cubical structure $K_{0\, i}$ which has exactly one $(n-1)$-cube and whose $k$-cubes, $0 \leq k\leq  n-1,$ are obtained by taking unions of $k$-simplices in $X|_{Q_0}$.  Thus the complex $K_0 = K_{0\, 1}\cup \cdots \cup K_{0\, n}$ is a cubical $(n-1)$-complex having $n$ cubes of dimension $(n-1)$.

Let now $v$ be the unique vertex of $X$ in the interior of $\Delta_n$. Let $L_0$ be the link of vertex $v_0$ in $X$, and $\tau_0$ be the face of $\Delta_n$ opposite to vertex $v_0$. Since $X$ is the barycentric subdivision of $\Delta_n$, the link $L_0$, as a simplicial complex,  is isomorphic  to the barycentric subdivision $X|_{\tau_0}$ of $\tau_0$. 
By the induction assumption,  $X|_{\tau_0}$ admits a cubical $(n-1)$-complex which consists of $n$ cubes of dimension $(n-1)$, and  is the star of the unique vertex in the interior of $\tau_0$.
Thus the same holds true for the link $L_0$ and vertex $v$. We denote this complex on $|L_0|$ by $K'_0$.

Then the union $K_0 \cup K'_0$ is a cubical complex on the boundary of the star $\Star(v_0)$ in $X$, obtained by taking unions of simplices in $X$, which consists of $2n$  dimension $n-1$ cubes.  Moreover, $|K_0|$ and $|K'_0|$ are $(n-1)$-cells and $\partial Q_0= |K_0\cup K'_0|$ is an $(n-1)$-sphere. Thus  $Q_0$ is an $n$-cell and is the space of the cubical complex $\sfC(Q_0) = K_0\cup K'_0 \cup \{Q_0\}$ isomorphic to $\sfC_n$.

We now fix, for each $i=1,\ldots, n$, an isometry $\rho_i \colon \Delta_n \to \Delta_n$ which maps $v_i$ to $v_0$.  Then $\rho_i$ fixes the barycenter $v$ of $\Delta_n $ and satisfies $\rho_i(Q_i)=Q_0$. Thus  $Q_i$ admits a cubical structure $\sfC(Q_i)$ for which $\rho_i^*(\sfC(Q_i)) = \sfC(Q_0)$.

The  congruence relation among the $n$-cubes $Q_0,\ldots, Q_n$ follows from definition.  The fact that  $\sfC(Q_i)|_{Q_i\cap Q_j} = \sfC(Q_j)|_{Q_i\cap Q_j}$ for $i\ne j$ follows  from the induction assumption and the Euclidean barycentric subdivision of $\Delta_n$. 

We set $\Delta_n^\boxplus$ to be the union of complexes $\sfC(Q_0),\ldots, \sfC(Q_n)$. 

Based on the proof of Lemma \ref{lemma:congruent-cubulation}, we may associate to   $\Delta_n^\boxplus$ a flat structure as follows.  Let $h \colon Q_0 \to [0,1]^n$ be a cubical isomorphism with the property that $h(v_0) = 0$, $h(v) = e_1 + \cdots +e_n$, and that $h$ conjugates each isometry of $Q_0$ fixing $v_0$ and $v$ to an isometry of $[0,1]^n$, that is, if $\iota \colon Q_0 \to Q_0$ is an isometry fixing $v_0$ and $v$ then $h \circ \iota \circ h^{-1} \colon [0,1]^n \to [0,1]^n$ is an isometry fixing $0$ and $e_1+\cdots + e_n$. We set  $h_0 = h \colon Q_0 \to [0,1]^n$ and, for $i>0$, and $h_i = h \circ \rho_i \colon Q_i \to [0,1]^n$. 

Then $h_i \cap h_j^{-1}|_{h_j(Q_j\cap Q_i)} \colon h_j(Q_j\cap Q_i) \to h_i(Q_j\cap Q_i)$ is a Euclidean isometry. Indeed, $\rho_j(Q_j\cap Q_i)$ and $\rho_i(Q_j\cap Q_i)$ are unions of faces of $Q_0$ in $K_0'$ which differ by an isometry. Thus $(h_i\circ h_j^{-1})|_{h_j(Q_j\cap Q_i})$ is a Euclidean isometry.
\end{proof}

An immediate consequence of the previous lemma is the existence of a cubical complex $T^\boxplus$ subdividing $T$, which admits a flat structure.

For the statement of the proposition, we recall that a map $f\colon X\to Y$ between metric spaces is \emph{$L$-quasisimilarity for $L\ge 1$} if there exist $\lambda>0$ having the property that, for all points $x,x'\in X$, we have \index{quasisimilarity}
\[
\frac{\lambda}{L}d(x,y) \le d(f(x),f(y)) \le \lambda L d(x,y).
\]

\begin{proposition}
\label{prop:PL-to-cubical}
Let $T$ be a simplicial $n$-complex whose space $(|T|, d)$ is a Riemannian $n$-manifold with boundary. Then there exists a cubical complex $K=T^\boxplus$ having the same space for which the restriction $T^\boxplus|_\sigma$ is a well-defined cubical $n$-complex for each $n$-simplex $\sigma$. Moreover, there exists a flat structure $\sF_{T^\boxplus}$ of $T^\boxplus$ with the following property: for each $n$-simplex $\sigma$ in $T$ and $n$-cube $Q\in T^\boxplus|_\sigma$, the map $\phi_Q \colon Q\to [0,1]^n$ is an $L$-quasisimilarity, where $L=L(\sigma)$. \index{$T^\boxplus$}
\end{proposition}

\begin{proof}
We fix for each $n$-simplex $\sigma\in T$, an affine simplicial homeomorphism $\psi_\sigma \colon \sigma \to \Delta_n$. Since the transition maps $\psi_{\sigma'} \circ \psi_{\sigma}^{-1}$ are Euclidean isometries on common faces, the structures $\psi_\sigma^*(\Delta_n^\boxplus)$ and $\psi_{\sigma'}^*(\Delta_n^\boxplus)$ agree on $\sigma\cap \sigma'$. Thus we may take 
\[
T^\boxplus = \bigcup_{\sigma \in T^{[n]}} \,\psi_\sigma^*(\Delta_n^\boxplus).
\]

We now construct a flat structure $\sF_{T^\boxplus}$ on $T^\boxplus$. Recall first from Lemma \ref{lemma:congruent-cubulation} that 
 all $n$-cubes $Q_0,\ldots, Q_n$ in $\Delta_n^\boxplus$ are congruent to $Q_0$ and that  maps $h_i \colon Q_i \to [0,1]^n$ are cubical isomorphisms.

Let  $\sigma$ be an $n$-simplex in $T$ and  $Q$ be an $n$-cube in $ T^\boxplus|_\sigma$. Let $\ell_Q \in \{0,1\ldots, n\}$ be the index for which $\psi_\sigma(Q) = Q_{\ell_Q}$. Set $\phi_Q = h_{\ell_Q} \circ \psi_\sigma|_Q \colon Q\to [0,1]^n$, and take
\[
\sF_{T^\boxplus}  = \bigcup_{\sigma \in T^{[n]}} \, ( \bigcup_{Q\in T^\boxplus|_{|\sigma|}} \,\, \{ \phi_Q\} \,\,). 
\]

To check that transition maps $\phi_{Q'} \circ \phi_Q^{-1} \colon \phi_Q(Q\cap Q') \to \phi_{Q'}(Q\cap Q')$ are Euclidean isometries for cubes $Q, Q'$ sharing a face $q= Q\cap Q'$, we consider two cases: (i)  $Q$ and $Q'$ belonging to the same $n$-simplex $\sigma$, and (ii) $Q$ and $Q'$ belonging to two different $n$-simplices $\sigma$ and $\sigma'$. We omit the straightforward details. Finally, the geometrical property of the flat structure follows from the fact that $\psi_\sigma$ is a quasisimilarity  with a constant depending on $\sigma$, and that $h$ is a bilipschitz map, in Euclidean metric, with a constant depending on the dimension $n$.
\end{proof}

In view of Proposition \ref{prop:PL-to-cubical}, we may associate to $|T|$ a flat metric $d_{\sF_{T^\boxplus}}$ induced by  $\sF_{T^\boxplus}$.

An immediate corollary of Proposition \ref{prop:PL-to-cubical} is the following comparison of metrics.
\begin{corollary}
\label{cor:flat-metrics}
Let $T$ be a simplicial $n$-complex whose space $(|T|, d)$ is a Riemannian $n$-manifold with boundary. Let $T^\boxplus$ be the cubical complex associated to $T$, and $\sF_{T^\boxplus}$ be a flat structure on $T^\boxplus$. Then, the flat metric $d_{\sF_{T^\boxplus}}$ and the Riemannian metric $d$ are quasisimilar with a constant depending only on the constant $L$ in Proposition \ref{prop:PL-to-cubical}.
\end{corollary}

\subsection{Standard refinements and standard metrics}
\label{sec:refinement-metric}

The standard cubical structure $\sfC_n$ of the Euclidean $n$-cube $[0,1]^n$ admits a natural subdivision into $3^n$ congruent $n$-cubes, and a barycentric subdivision into congruent $n$-simplices.
We begin by fixing these subdivisions on $\sfC_n$, and then extend these notions to cubical complexes.
\index{cubical complex!refinement} 

For each $v\in \{0,1,2\}^n$, we set $q_v = \frac{1}{3}\left( v+[0,1]^n \right)$ and $\iota_v \colon q_v \to [0,1]^n$ to be the congruence $x\mapsto 3^n(x-v)$; we denote $\sfC(q_v) = \sfC_{\iota_v}(q_v)$. Then cubes $q_v$, $v\in \{0,1,2\}^n$, are congruent Euclidean cubes which cover $[0,1]^n$, have mutually disjoint interiors, and satisfy $q_v \cap q_{v'} \in \sfC(q_v) \cap \sfC(q_{v'})$ for $v,v'\in \{0,1,2\}^n$. Thus
\[
\sfR_n = \bigcup_{v\in \{0,1,2\}^n} \sfC(q_v)
\]
is a cubical $n$-complex with space $[0,1]^n$. We call $\sfR_n$ the \emph{standard refinement of $\sfC_n$}.

Let now $K$ be a cubical complex with a flat structure $\Phi_K=\{ \phi_Q \colon Q\to [0,1]^n \colon Q\in K^{[n]}\}$. For each $Q\in K^{[n]}$, we take   
\[
\Refine(Q) = \phi_Q^*(\sfR_n).
\]
Since the transition maps of the flat structure are Euclidean isometries, we observe that, for $n$-cubes $Q$ and $Q'$ having a common face $q=Q\cap Q'\in K^{[n-1]}$, we have that
\[
\Refine(Q)|_q = \Refine(Q')|_q.
\]
Thus
\[
\Refine(K) = \bigcup_{Q\in K^{[n]}} \Refine(Q)
\]
is a well-defined cubical complex. We call  $Q$ the  \emph{parent of the cubes in $\Refine(Q)$}.

A flat structure $\sF_{\Refine(K)}$ on $\Refine(K)$ may also be defined as follows. Let $Q\in K^{[n]}$ and $Q'\in \Refine(K)^{[n]}$ be an $n$-cube contained in $Q$. Then $\phi_Q(Q')=q_v$ for some $v\in \{0,1,2\}^n$. Let $\phi_{Q'}\colon Q' \to [0,1]^n$  be the map $\phi_{Q'} = \iota_v \circ \phi_Q|_{Q'}$, and set  
\[ \sF_{\Refine(K)} = \{ \phi_{Q'} \colon Q'\to [0,1]^n \colon Q'\in \Refine(K)^{[n]}\}.\]
Again, the transition maps for $\sF_{\Refine(K)}$
 are Euclidean isometries. Moreover, the maps in the flat structure $\sF_{\Refine(Q)}$ inherit their quasisimilarity constants from their parents.

\begin{definition}\label{def:standard-metric} 
The cubical $n$-complex $\Refine(K)$ is called the \emph{standard refinement of a cubical $n$-complex $K$ (with respect to the flat structure $\sF_K$)}. For $k\ge 1$, the complex $\Refine^k(K) = \Refine(\Refine^{k-1}(K))$ is called the \emph{$k$th iterated refinement of $K$}.

The metric on $\Refine^k(K)$ given by
\[d_K \equiv d_{\sF_K} \colon |\Refine^k(K)|\times |\Refine^k(K)| \to [0,\infty)\]
is called the \emph{standard metric on $\Refine^k(K)$}.  \index{metric!on refinement}
\end{definition} 

\begin{remark}
We observe that the iterated refinement may be defined, equivalently, by the formula \index{$\Refine^k(K)$}
\[
\Refine^k(K) = \bigcup_{Q\in K^{[n]}} \phi_Q^*(\Refine^k(\sfC_n)).
\]
Also, when complex $K$ has a flat structure, its refinement $\Refine^k(K)$ has a well-defined flat structure inherited from $K$.
\end{remark}

\begin{remark}The flat metric $d_{\sF_{\Refine^k(K)}}$ in $\Refine^k(K)$ induced by the flat structure $\sF_{\Refine^k(K)}$  satisfies
\[ d_{\sF_{\Refine^k(K)}} (x,x')= 3^k\,  d_K(x,x')\]
for $x, x' \in \Refine^k(K)$.
\end{remark}

\begin{remark}
Passing to a refinement $\Refine(K)$ of a complex $K$ respects subcomplexes and skeleta, that is,  the refinement $\Refine(K')$ of a subcomplex $K'$ of $K$ is the restriction of $\Refine(K)$ to $|K'|$. We may unambiguously  say that a $k$-cube in $\Refine(K')$ is contained in $\Refine(K)$, or a $k$-cube in $\Refine(K')$  is contained in a $k$-cube in $K$. 
\end{remark}

We finish this discussion by fixing the notion of center cube for the forthcoming sections. Let $K$ be a cubical $n$-complex with a flat structure $\sF_K$. We denote by
\[
c \colon K \to \Refine(K),
\]
the map which takes  $q \in K^{[k]}$ to  the unique \emph{center cube} $c(q)$ in $(\Refine(K)|_q)^{[k]}$.

\subsection{Passing from Riemannian structure to cubical structure}
\label{sec:From-Riem-To-Cubical}

Every Riemannian $n$-manifold $(M,g)$ admits a triangulation $T$ in which every $n$-simplex $(\sigma, g|_{\sigma})$ is bilipschitz to an affine Euclidean $n$-simplex. See \cite{Cairns}, \cite{WhiteheadC1}, and \cite{Munkres}.

The following proposition, which follows immediately from Proposition \ref{prop:PL-to-cubical}, explains our reason in focusing on cubical complexes when studying the quasiconformal and quasiregular mappings on manifolds.

\begin{proposition}\label{prop:Riemannian-to-cubical}Associated to each  Riemannian manifold $(M,g)$, there exist a cubical complex $K$ with a flat metric $d_K$ for which each cubic is isometric a Euclidean unit cube, and a quasi-similar homeomorphism $(M,g) \to (K,d_K)$.
\end{proposition}

\subsection{Barycentric triangulation of a cubical complex}
\label{sec:triangulation}

As the last subdivision to be discussed, we recall the standard barycentric subdivision $K^\Delta$ of a cubical complex $K$; see Figure \ref{fig:canonical} for an example.

\begin{figure}[h!]
\begin{overpic}[scale=0.8,unit=1mm]{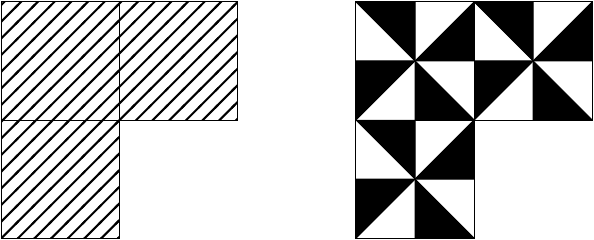} 
\end{overpic}
\caption{A cubical complex $K$ and its triangulation $K^\Delta$.}
\label{fig:canonical}
\end{figure}

Let $\sfC_n$ be the standard cubical structure on the unit cube $[0,1]^n$.
We say that a simplicial complex $T_n= \sfC_n^\Delta$ is a \emph{barycentric triangulation of $\sfC_n$, or of the Euclidean cube $[0,1]^n$}, if all vertices  of $\sfC_n$ are vertices of $T_n$, and $T_n$ has exactly one vertex in the interior of each cube $q \in \sfC_n$. We may obtain this triangulation, for example, as follows.

 The barycentric triangulation $T_n$ of $\sfC_n$ is obtained inductively with respect to the skeleta $\sfC_n^{(k)}$ as follows. For the construction, we denote, for each cube $q\in \sfC_n$, $b_q\in q$ the Euclidean barycenter of $q$, that is, average of the corners of $q$.

For $k=0$, we fix $\tau_0 = \{b_q \colon q\in \sfC_n\}$. Then $\tau_0 = \sfC_n^{(0)}$. Suppose that for $k\in (0,n]$, we have constructed a sequence of simplicial complexes $\tau_0 \subset \cdots \subset \tau_{k-1}$, for which  $\tau_\ell$ has $|\sfC_n^{(\ell)}|$ as its space and  the restriction $\tau_{\ell}|_q$ to each $q\in \sfC_n^{[\ell]}$   is a barycentric triangulation of $q$.

We fix the complex $\tau_k$ as follows. Let $q\in \sfC_n^{[k]}$. By assumption, the restriction $\tau_{k-1}|_{\partial q}$ is a simplicial subcomplex of $\tau_{k-1}$. We extend each $\ell$-simplex $\xi$ in $\tau_{k-1}|_{\partial q}$ to an $(\ell+1)$-simplex $\xi_q = \xi * b_q$, where $*$ is the join of $\tau$ and $b_q$. Then, for each $q\in \sfC_n$, 
\[
\tau(q)=\{ \xi_q \colon \xi \in \tau_{k-1}|_{\partial q}\} \cup (\tau_{k-1}|_{\partial q})
\]
is a simplicial complex with having $q$ as its space. Moreover, for $q,q'\in \sfC_n^{[k]}$, we have that $\tau(q)|_{q\cap q'} = \tau(q')|_{q\cap q'} = \tau_{k-1}|_{q\cap q'}$. Thus
\[
\tau_k = \bigcup_{q\in \sfC_n^{[k]}} T(q)
\]
is a simplicial complex with space $|\sfC_n^{(k)}|$, and the sequence $\tau_0\subset \cdots \subset \tau_{k-1}\subset \tau_k$   satisfies the induction assumption. Finally we set $T_n = \tau_n$. This ends the construction of the barycentric triangulation $T_n=\sfC^\Delta_n$ of $\sfC_n$.

\medskip

Given now a cubical complex $K$ with a flat structure $\sF_K=\{ \phi_Q \colon Q\to [0,1]^n \colon Q\in K^{[n]}\}$, the barycentric subdivision of $K$ is then the simplicial complex \index{$K^\Delta$} \index{cubical complex!barycentric triangulation $K^\Delta$} \index{$K^\Delta$}
\[
K^\Delta = \bigcup_{Q\in K^{[n]}} \phi_Q^*(\sfC^\Delta_n).
\]
Again, since the transition functions of the flat structure are Euclidean isometries and the $n$-simplices in $T^\Delta_n$ are congruent, we have that $K^\Delta$ is a well-defined complex.

\subsubsection{Barycentric subdivisions of triangulations}
\label{sec:subdivision-triangulation}
 
As we work mostly with cubical complexes in what follows, we only comment briefly the usual barycentric subdivision on triangulations.

As for cubical complexes, we may define, for a simplicial $n$-complex, a flat structure $\cF(T) = \{ \phi_\sigma \colon |\sigma|\to \Delta_n \colon \sigma \in T^{[n]} \} $ of a simplicial $n$-complex $T$. Here we -- and in what follows -- tacitly assume that $|T|$ is the union of $n$-simplices.

Let $B$ be a barycentric subdivision of $\Delta_n$. Having the flat structure $d_{\cF_T}$ at our disposal, we may define a geometric barycentric subdivision $T^\Delta$ of $T$ by $T^\Delta = \bigcup_{\sigma\in T^{[n]}} \phi_\sigma^*(B)$.

\section{Adjacency graph of a complex}
\label{sec:adjacency-graph}

In this section, we introduce the notion of an adjacency graph of a cubical complex and consider cubical complexes which are realizations of subgraphs of adjacency graphs. For the definition of the adjacency, let $K$ be  a cubical $n$-complex. The corresponding definition for simplicial complexes is similar.

\index{cubical complex!adjacency graph}
\begin{definition}
\label{def:adjacent}
\index{adjacency graph}
Two $n$-cubes $Q$ and $Q'$ in a cubical $n$-complex $K$ are \emph{adjacent} if $Q\cap Q'$ is a common face of $Q$ and $Q'$, that is, $Q\cap Q'\in K^{[n-1]}$. The graph 
\[
\Gamma(K) = \left( K^{[n]}, \left\{ \{Q,Q'\} \colon Q\ne Q',\ Q\cap Q'\in K^{[n-1]}\right\}\right)
\]
is called the \emph{adjacency graph of $K$}. 
\end{definition}

\begin{definition}\index{cubical complex!adjacently connected}
\label{def:adjacently-connected}
\index{adjacently connected complex}
A cubical $n$-complex $K$ is \emph{adjacently connected} if its adjacency graph $\Gamma(K)$ is connected.
\end{definition}

A remark on this definition is in order.

\begin{remark}
Clearly, two $n$-cubes $Q$ and $Q'$ of $K$ may have non-empty intersection without being adjacent. Thus, given a collection $S\subset K^{(n]}$ of $n$-cubes, the space $|\Span_K(S)|$ may be connected as a topological space, but the graph $\Gamma(\Span_K(S))$ is not connected.
\end{remark}

To each subgraph of the adjacency graph $\Gamma(K)$, we associate a subcomplex of $K$ as follows. 

\begin{definition}\index{cubical complex!subcomplex spanned by a graph} 
\label{def:G-span}
\index{subcomplex spanned by a graph $\Span_K(G)$}
Let $K$ be a cubical $n$-complex and $G = (V,E) \subset \Gamma(K)$ a subgraph. We call
\[
 K_G = \Span_K(G) =  \Span_K\left( \{ Q\in K^{[n]} \colon Q\in V\}\right) 
\]
the \emph{subcomplex of $K$ spanned by $G$}.
\end{definition}

\begin{example}
\label{ex:adjacency}
Let $K$ be a $2$-complex consisting of four unit squares in $\R^2$ and having space $[0,2]\times [0,2]$, and $G\subset \Gamma(K)$ a spanning tree in $K$. In this case, $\Gamma(K)$ is a loop, hence $G\subsetneq \Gamma(K)$. Since $K_G=K$, we have $\Gamma(K_G)=\Gamma(K)$.
\end{example}

\begin{figure}[htp]
\begin{overpic}[scale=1,unit=1mm]{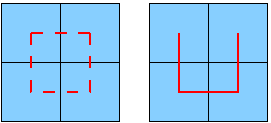} 
\end{overpic}
\label{fig:Realization-1}
\caption{Left: the adjacency graph $\Gamma(K)$ of complex $K$ in Example \ref{ex:adjacency}. Right: A spanning tree $G$ of $\Gamma(K)$.}
\end{figure}

\subsection{Realization of a subgraph}
\label{sec:realization-subgraph}

Each subgraph $G$ of an adjacency graph $\Gamma(K)$ is an adjacency graph of a cubical $n$-complex $R_G=\Real(G)$; we call the complex $\Real(G)$  a \emph{realization of $G$}. For this construction, we give first some notations. 

Let $K$ be a cubical $n$-complex and let $G = (V_G,E_G) \subset \Gamma(K)$ be a subgraph. Denote by
\[
S_G = \bigsqcup_{Q\in V_G} Q
\]
the disjoint union of all the $n$-cubes in the graph $G$, and by
\[
F_G = \bigsqcup_{Q\in V_G} \sfC(Q)
\]
a cubical $n$-complex having $S_G$ as its space.

Let $\sim_G$ be the equivalence relation in $S_G$ generated by relations $x\sim_G x'$ between points  $x\in Q$ and $x'\in Q'$ for which  $\{Q,Q'\}\in E_G$ is an edge and $x = x'$ in $|K|$.  Denote by $[x]$  the equivalence class of $x\in S_G$ induced in $S_G/{\sim_G}$. We denote $[q] = \{ [x]\colon x\in q\}$ for each $q\in F_G$,

We are now ready to give the definition of a realization of a subgraph; see Figure \ref{fig:Realization-2} for an illustration of a realization related to Example \ref{ex:adjacency}.

\begin{definition}
\label{def:realization}
\index{realization of a subgraph $\Real(G)$}
The \emph{realization $\Real(G)$} of a subgraph $G \subset \Gamma(K)$ of the adjacency graph of a cubical $n$-complex $K$ is the cubical $n$-complex
\[
\Real(G)= F_G/{\sim_G} = \{ [q] \colon q\in F_G\}.
\]
\end{definition}

\begin{figure}[htp]
\begin{overpic}[scale=1,unit=1mm]{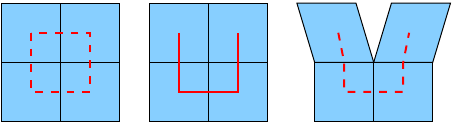} 
\end{overpic}
\label{fig:Realization-2}
\caption{From left to right: the adjacency graph $\Gamma(K)$ of complex $K$ in Example \ref{ex:adjacency}, a spanning tree $G$ of $\Gamma(K)$, and the realization $\Real(G)$ of $G$.}
\end{figure}

\subsubsection{Quotient maps associated to realizations}

We associate to he realization $R_G=\Real(G)$ of a subgraph $G\subset \Gamma(K)$ canonical maps $R_G \to K_G$ and $|R_G|\to |K_G|$ defined as follows. Let $\iota_G \colon S_G \to |K_G|$ be the map induced by inclusion of each $n$-cube $Q\in V_G$ into $|K_G|$. We denote by the same symbol also the induced map $\iota_G \colon F_G \to K_G$ between complexes. By construction, the mapping $\iota_G \colon S_G \to |K_G|$ factorizes through the quotient map $S_G \to |R_G|$, $x\mapsto [x]$, and the mapping $\iota_G \colon F_G \to K_G$ through the induced quotient map $F_G \to R_G$, $q\mapsto [q]$. Thus we obtain diagrams
\[
\xymatrix{
S_G \ar[rr]^{\iota_G} \ar[dr]_{x\mapsto [x]} & & |K_G| \\
& |R_G| \ar[ru]_{\pi_G} &
}
\quad \text{and} \quad
\xymatrix{
F_G \ar[rr]^{\iota_G} \ar[dr]_{q\mapsto [q]} & & K_G \\
& R_G \ar[ru]_{\pi_G} &
}
\]
where we denote the both factors $|R_G| \to |K_G|$ and $R_G \to K_G$ with the same symbol $\pi_G$, respectively. We follow this convention from now on and call both canonical maps 
\[\pi_G \colon |R_G|\to |K_G|, \quad \text{and}\,\,\,\, \pi_G \colon R_G \to K_G\]
 as \emph{$G$-quotient maps}. 

Since the equivalence relation $\sim_G$ does not make non-trivial identifications among points in a single $n$-cube $Q$ in $V_G$, the image of a cube under either map is a cube and, similarly, a preimage of a cube is a cube. We call -- in the case of both maps -- the image $\pi_G(c)$ of a cube $c\in R_G$ a \emph{projection} and the preimage $\pi_G^{-1}(q)$  of a cube $q\in K_G$ a \emph{lift}. Note that, we further obtain that the restrictions $\pi_G|_Q \colon Q\to |K_G|$  and  $\pi_G|_{\sfC(Q)} \colon \sfC(Q) \to K_G$ are embeddings for each $Q\in V_G$, that is, a homeomorphism and an isomorphism to their images, respectively.

Finally, we note that the cubical map $\pi_G \colon R_G \to K_G$ induces a map of adjacency graphs $\Gamma(\pi_G) \colon \Gamma(R_G) \to \Gamma(K_G)$, determined by the map $C\mapsto \pi_G(C)$ on the vertices and by the map $\{C,C'\}\mapsto \{ \pi_G(C),\pi_G(C')\}$ on the edges. Recall that $\pi_G([Q]_j) = Q$.

\subsubsection{Elementary properties of realizations}

In the future, we consider mainly realizations $R_G=\Real(G)$ associated to spanning trees $G$ of adjacency graphs $\Gamma(K)$; in such case, the spaces $|R_G|$ are $n$-cells. We now consider realizations more generally. As a motivating example, in Figure \ref{fig:Realization-3} we have illustrated a simple example of a cubical $2$-complex $K$ for which $|K|$ is not a $2$-cell but $\Gamma(R_G) = \Gamma(K)$ and $|R_G|$ is a $2$-cell for $G=\Gamma(K)$. 

\begin{figure}[htp]
\begin{overpic}[scale=1,unit=1mm]{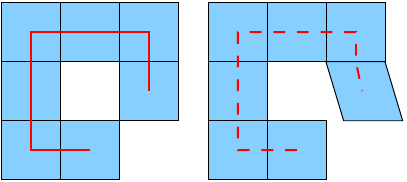} 
\end{overpic}
\caption{Left: Cubical $2$-complex $K$ consisting of $6$ squares and graph $G=\Gamma(K)$. Right: The realization $\Real(G)$ of $G$.}
\label{fig:Realization-3}
\end{figure}

Our first observation is that the adjacency graph of the realization $R_G$ is isomorphic to $G$.

\begin{lemma}
\label{lemma:realization-1}
Let $K$ be a cubical $n$-complex and let $G=(V_G,E_G)\subset \Gamma(K)$ be a subgraph. Then the graph $\Gamma(R_G)$ is isomorphic to $G$. Moreover, the map $\Gamma(\pi_G) \colon \Gamma(R_G) \to \Gamma(K_G)$ is an isomorphism onto the subgraph $G$ of $\Gamma(K_G)$.
\end{lemma}

\begin{proof}
We observe first that, for each $Q\in V_G$, we have $\pi_G([Q]) = Q$. Thus the induced map $\Gamma(\pi_G) \colon \Gamma(R_G) \to G$ is a bijection on vertices. It suffices to show that it is also an isomorphism on edges.

We show first that $\Gamma(R_G)$ is isomorphic to $G$. Suppose that $Q$ and $Q'$ are $n$-cubes in $K_G$ for which $\{Q,Q'\}\in E_G$. Then $[Q]\cap [Q'] = [Q\cap Q']$. Thus $[Q]$ and $[Q']$ are adjacent in $R_G$ and $\Gamma(\pi_G)(\{[Q],[Q']\}) = \{Q,Q'\}$. 

Suppose now that $[Q]$ and $[Q']$ are adjacent in $R_G$. Then $\xi = [Q]\cap [Q']\in R_G^{[n-1]}$ is a common face of $[Q]$ and $[Q']$ and there exists $q\in K^{[n-1]}$ for which $\xi = [q]$. Thus $\{Q,Q'\}\in E_G$ and $\Gamma(\pi_G)(\{ [Q],[Q']\}) = \{Q,Q'\}$. We conclude that $\Gamma(\pi_G)$ is an isomorphism $\Gamma(R_G) \to G$.
\end{proof}

We show in the next lemma that the connected components of the realization $R_G$ are in one-to-one correspondence with the connected components of $G$.
\begin{lemma}
\label{lemma:realization-2}
Let $K$ be a cubical $n$-complex and let $G=(V_G,E_G)\subset \Gamma(K)$ be a subgraph. Then the connected components of $|R_G|$ are in one-to-one correspondence with those of $G$.
\end{lemma}

\begin{proof}
Let $Q\in V_G$, and let $\cC$ be the connected component of $G$ containing $Q$ and let $P_{\cC}$ be the subcomplex of $R_G$ spanned by the $n$-cubes in $\cC$.

Let $\Omega$ be the connected component of $|R_G|$ containing $[Q]$. We observe first that $\Omega$ is a union of $n$-cubes in $R_G$. Indeed, if $[Q]\cap \Omega\ne \emptyset$ for $Q\in V_G$, then $[Q]\subset \Omega$ by connectedness of $[Q]$. Let now $P_\Omega$ be the subcomplex of $R_G$ for which $\Omega = |P|$. Since $|P_{\cC}|$ is connected, we have that $|P_{\cC}|\subset \Omega$.

Let now $[Q']$ be an $n$-cube in $P_\Omega$ for which $[Q']\cap |P_{\cC}|\ne \emptyset$. Then there exists an $n$-cube $[Q'']$ in $\cC$ for which $[Q']\cap [Q'']\ne \emptyset$. Thus there are points $x'$ and $x''$ in cubes $Q'$ and $Q''$, which are identified. By the definition of $\sim_G$, there exists a chain of adjacent cubes $Q'=Q_0, \ldots, Q_m=Q''$ in $G$ for which $[x']\subset [Q_i]$ for each $i=0,\ldots, m$; recall that $[x']=[x'']$. Thus $[Q_0],\ldots, [Q_m]$ is a chain in $\Gamma(R_G)$ from $[Q']$ to $[Q'']$ and they belong to the same component of $\Gamma(R_G)$. Thus $[Q']\in P_{\cC}$. We conclude that $\Omega \subset |P_{\cC}|$ by connectedness of $\Omega$.
\end{proof}

We now formalize the local connectedness property of $R_G$ illustrated in Figure \ref{fig:Realization-2}.

\begin{lemma}
\label{lemma:realization-3}
Let $K$ be a cubical $n$-complex and let $G=(V_G,E_G)\subset \Gamma(K)$ be a subgraph. Let also $C,C'\in R_G^{[n]}$ be $n$-cubes with a non-empty intersection $C\cap C'$ which is not an $n$-cube. Then there exists a chain $C=C_0,\ldots, C_m = C'$ of adjacent $n$-cubes in $R_G$ for which $C\cap C' \subset C_i$ for each $i=0,\ldots, m$. In particular, $C \cap C'$ is contained in an $(n-1)$-cube $c \in R_G$ which is an edge in $\Gamma(R_G)$. 
\end{lemma}

\begin{proof}
Suppose first that points $x,y\in S_G$ are identified under $\sim_G$. Since $\sim_G$ is generated by relations $x\sim_G x'$ between points $x\in Q$ and $x'\in Q'$ for which $\iota_G(x)=\iota_G(x')$ and $\{Q,Q'\}\in E_G$, we have that $\iota_G(x)=\iota_G(y)$ and that there exists $n$-cubes $Q_0,\ldots, Q_m$ in $V_G$, each containing the point $x$, for which $\{Q_i,Q_{i+1}\}\in G$ for each $i=0,\ldots, m-1$. Then $[Q_0],\ldots, [Q_m]$ is a chain of adjacent $n$-cubes in $R_G$ and $[x]\in [Q_i]$ for each $i=0,\ldots, m$. In particular, $[x]\in [Q_0]\cap [Q_1]$, where $\{[Q_0],[Q_1]\}\in \Gamma(R_G)$.

Suppose next that $C$ and $C'$ are $n$-cubes in $R_G$ having a non-empty intersection. Let $Q$ and $Q'$ be $n$-cubes in $V_G$ for which $C=[Q]$ and $C'=[Q']$. Since $C\cap C'$ is not an $n$-cube, we have that $Q\ne Q'$. If $C\cap C'$ is a singleton $[x]$, both claims follow from the above observation. Suppose that $C\cap C'$ is not a singleton. Then we may fix $[x]\in C\cap C'$ in the interior of $C\cap C'$. By the above observation, there exists a chain $C_0,\ldots, C_m$ of adjacent $n$-cubes in $R_G$ each containing $[x]$. Since $C\cap C'$ is a cube, we have that $C\cap C'\subset C_i$ for each $i=0,\ldots, m$. In particular, $C\cap C'\subset C_0\cap C_1$. Since $C_0$ and $C_1$ are adjacent in $R_G$, the claim follows.
\end{proof}

Lemma \ref{lemma:realization-2} yields that the realization of a tree in $\Gamma(K)$ is a cell.

\begin{lemma}
\label{lemma:realization-cell}
Let $K$ be a cubical $n$-complex and  $G \subset \Gamma(K)$ be a tree. Then $|R_G|$ is an $n$-cell. 
\end{lemma}

\begin{proof}
Since $G$ is a tree, so is $\Gamma(R_G)=(V,E)$. Let $Q\in R_G^{[n]}$ be a leaf of $\Gamma(R_G)$ and  $Q'\in R_G^{[n]}$ be the unique $n$-cube adjacent to $Q$. Let  
$R'_G =\Span\{R_G\setminus \sfC(Q)\}$ be the cubical $n$-complex obtained by removing from $R_G$ all the lower dimensional cubes contained in $Q$ but not in any other $n$-cubes. Lemma \ref{lemma:realization-3} 
implies that $Q$ does not meet any other leaf in $R_G$. Thus there exists a PL homeomorphism $h_Q \colon |R_G| \to  |R'_G|$ for which  $h_Q|_q=\id$ for all $n$-cubes $q\in R_G$, $q \neq Q, Q'$. 

By removing leaves of $G$ iteratively, we may reduce $R_G$ to an $n$-cube, which is of course an $n$-cell.
\end{proof}

\subsection{Tunnels}

We  single out a particular family of complexes, called \emph{tunnels}, which are crucial in the forthcoming constructions.

\begin{definition}
\label{def:tunnel}
\index{tunnel}
A cubical $n$-complex $K$ is a \emph{tunnel} if its adjacency graph $\Gamma(K)$ is a tree and $K$ is isomorphic to the realization of $\Gamma(K)$.
\end{definition}

Realization of a tree is always a tunnel. We record this observation as a corollary.
\begin{corollary}
\label{cor:tunnel}
Let $K$ be a cubical $n$-complex, and  $G \subset \Gamma(K)$  a tree. Then the realization $R_G$ of $G$ is a tunnel.
\end{corollary}

It follows from Lemma \ref{lemma:realization-cell} that the space of a tunnel is an $n$-cell. Moreover, we have the following (tunnel-contracting) bilipschitz homeomorphism from a complex expanded with a tunnel to the original complex.

\begin{proposition}[Tunnel-contracting]
\label{prop:tunnel-contracting}
Let $K$ be a cubical $n$-complex and let $T$ be a cubical $n$-complex which is a tunnel for which the intersection $q_T=K\cap T$ is an $(n-1)$-cube. Then there exists $L=L(n,\# T^{[n]})\ge 1$ and a piecewise linear $L$-bilipschitz homeomorphism $\phi_T \colon |K\cup T| \to |K|$ for which $\phi_T$ is an identity on $|K-Q_T|$, where $Q_T$ is the unique $n$-cube in $K$ having $q_T$ as a face.
\end{proposition}

\begin{proof}
Since also $T'=T\cup Q_T$ is a tunnel, $\Gamma(T')$ is a tree. We give $\Gamma(T')$ a partial order $<$, induced by the tree structure, in which $Q_T$ is the unique smallest element. The proof is by induction with respect to this order. It suffices to consider the case of leaves, i.e.~maximal elements, of $\Gamma(T')$; other cases follow by induction.

Let $Q\in {T'}^{[n]}$ be a maximal element in $\Gamma(T')$ and let $Q'$ be the unique element in $T'$ sharing a face $q=Q\cap Q'$ with $Q$, i.e.~$\{Q,Q'\}$ is the unique edge in $\Gamma(T')$ containing $Q$. Since $T$ is isomorphic to the realization of $\Gamma(T)$, we have that $Q\cap T' = Q\cap Q'$. Thus there exists a piecewise linear homeomorphism $\varphi_Q \colon |K\cup T| \to |(K\cup T) -Q|$, which is bilipschitz with a constant depending only on $n$ and which is the identity on $|(K \cup T) - (Q\cup Q')| \cup |\partial Q' - q|$.

The wanted homeomorphism is now obtained as a composition of such homeomorphism. Thus its bilipschitz constant depends only on the dimension $n$ and the size of the tree $\Gamma(T)$, which is $\#T^{[n]}$. 
\end{proof}

\subsection{Cut graphs of adjacency graph}
\label{sec:cut-graphs}
\index{cubical complex!cut-graph}

In Part \ref{part:separating_complexes} we use a particular class of subgraphs of the adjacency graph $\Gamma(K)$ obtained by cutting $\Gamma(K)$ along a codimension one subcomplex of $K$.

\begin{definition}
\label{def:complementary-graph}
\index{adjacency graph!cut-graphs} \index{$\Gamma(K;Y)$}
Let $K$ be a cubical $n$-complex and $Y \subset K$ be a cubical $(n-1)$-subcomplex. The subgraph
\[
\Gamma(K;Y) = \left( K^{[n]}, \{ \{ Q\cap Q'\} \in \Gamma(K) \colon Q\cap Q'\not \in Y\}\right)
\]
of $\Gamma(K)$ is called the \emph{cut-graph of $\Gamma(K)$ along $Y$}. \index{cut-graph}
\end{definition}

We denote 
\[
\Real_K(Y) = \Real(\Gamma(K;Y))
\]
the associated realization, and 
\index{$\pi_{(K,Y)}$}
\[
\pi_{(K,Y)} \colon \Real_K(Y) \to K
\]
the associated canonical map.

Suppose that space $|Y|$ of $Y$ is contained in the interior of $|K|$. For a connected component $\Sigma$ of $\partial K$, we denote by $\Gamma(K;Y,\Sigma)$  the unique connected component of $\Gamma(K;Y)$ for which
\[
\Sigma \subset \Span_K(\Gamma(K;Y,\Sigma)).
\]
In this case, we call the subcomplex \index{$\Comp_K(Y;\Sigma)$} \index{cubical complex!components separated by codimension one subcomplex}
\[
\Comp_K(Y;\Sigma) = \Span_K(\Gamma(K;Y, \Sigma))
\]
the \emph{$\Sigma$-component of $K$ separated by $Y$}, and the realization
\[
\Real_K(Y; \Sigma) = \Real(\Gamma(K;Y,\Sigma))
\]
the \emph{$\Sigma$-realization of $\Gamma(K;Y,\Sigma)$}. \index{realization of cut-graph $\Real_K(Y;\Sigma)$} 

We may canonically associate  $\Real_K(Y;\Sigma)$ with the connected component  $\pi_{(K,Y)}^{-1}(\Comp_K(Y;\Sigma))$ containing $\Sigma$ in $\Real_K(Y)$. In particular, for each boundary component $\Sigma \subset \partial K$, we have that  \index{$\pi_{(K,Y;\Sigma)}$}
\[
\pi_{(K,Y)}|_{\Real_K(Y;\Sigma)} = \pi_{(K,Y;\Sigma)} \colon \Real_K(Y;\Sigma) \to \Comp_K(Y;\Sigma).
\]

\begin{remark}
As the complex $Y$ need not separate all boundary components of $K$, we may have $\Comp_K(Y;\Sigma) = \Comp_K(Y;\Sigma')$, and hence also $\Real_K(Y;\Sigma) = \Real_K(Y;\Sigma')$, for different components $\Sigma$ and $\Sigma'$ of $\partial K$. 

As $\Gamma(K;Y)$ may have a connected component whose vertices do not meet $\partial K$, we may have $\bigcup_\Sigma \Comp_K(Y;\Sigma)\subsetneq K$.
\end{remark}

\begin{remark}
For a $(n-1)$-subcomplex $Y$ of $K$ and $\Sigma$ a boundary component of $ K$, we have
\[
\Comp_{\Refine^k(K)}(\Refine^k(Y);\Refine^k(\Sigma)) = \Refine^k(\Comp_K(Y; \Sigma)),
\]
and
\[
\Real_{\Refine^k(K)}(\Refine^k(Y);\Refine^k(\Sigma)) = \Refine^k(\Real_K(Y; \Sigma)).
\]
\end{remark}

\begin{convention}
For a $(n-1)$-subcomplex $Y'$ of a refinement $\Refine^k(K)$ and $\Sigma$ a boundary component of $K$, we do not repeat the refinement $\Refine^k(\Sigma)$ in the notation of $\Sigma$-component, and simply write 
\[
\Comp_{\Refine^k(K)}(Y';\Sigma) = \Comp_{\Refine^k(K)}(Y';\Refine^k(\Sigma)),
\]
and
\[
\Real_{\Refine^k(K)}(Y';\Sigma) = \Real_{\Refine^k(K)}(Y';\Refine^k(\Sigma)).
\]
\end{convention}

\subsubsection{Metrics on  realizations}\label{sec:metric-realization}

Let $K$ be a cubical $n$-complex with a flat metric $d_K$, and $\Sigma$ a boundary component of $\partial K$.  

Suppose that  $Y$  is an $(n-1)$-subcomplex of $K$.  Since the $\Sigma$-component $\Comp_K(Y;\Sigma)$ is a subcomplex of $K$, the distance between  points $x$ and $y$ in $|\Comp_K(Y;\Sigma)|$ is  $d_K(x,y)$.  

The realization $\Real_K(Y; \Sigma)$, on the other hand, is equipped with a flat metric $d_{\Real_K(Y; \Sigma)}$ associated to its own cubical flat structure. Spaces $(\interior (|\Comp_K(Y;\Sigma)|\setminus |Y|),d_K)$ and  $(\interior |\Real_K(Y; \Sigma)|,d_{\Real_K(Y; \Sigma)})$ are locally isometric, therefore conformal, via   $\pi_{(K,Y;\Sigma)}$; here interior is the manifold interior.  However they are not bilipschitz equivalent.

When the complexes are refined, the refinements $\Refine^k(\Comp_K(Y;\Sigma))$ and $\Refine^k(\Real_K(Y; \Sigma))$ are equipped with the standard metrics $d_K$ and $d_{\Real_K(Y; \Sigma)}$, respectively, inherited from the complexes $(\Comp_K(Y;\Sigma),d_K)$ and $(\Real_K(Y; \Sigma),d_{\Real_K(Y; \Sigma)})$, respectively, in which each $n$-cube has size $3^{-k}$. Recall the definition of the standard metric on the refinement of a complex in Definition \ref{def:standard-metric}.

\begin{convention}\label{convention:realization-metric}
Suppose  that $Y'$ is an $(n-1)$-subcomplex of the refined complex $\Refine^k(K)$. Then the distance in the space $|\Comp_{\Refine^k(K)}(Y';\Sigma)|$ is measured by the standard metric $d_K$ on $\Refine^k(K)$.

The  realization $\Real_{\Refine^k(K)}(Y';\Sigma)$ is equipped with the \emph{lifted standard metric}, that is, the path metric for which each $n$-cube in $\Real_{\Refine^k(K)}(Y';\Sigma)$ is locally isometric, via $\pi^{-1}_{(\Refine^k(K),Y';\Sigma)}$, to the corresponding $n$-cube in the space $(|\Comp_{\Refine^k(K)}(Y';\Sigma)|, d_K)$. Therefore,  $\interior (|\Comp_{\Refine^k(K)}(Y';\Sigma)|\setminus |Y'|)$ and $\interior (|\Real_{\Refine^k(K)}(Y';\Sigma)|)$ are conformal.
\end{convention}


\chapter{Preliminaries on quasiregular maps}
\label{part:preliminary_branched_covers}

\section{Alexander maps}
\label{sec:Alexander-maps}

In 1920, J. W. Alexander \cite{Alexander} showed that every closed oriented piecewise linear  $n$-manifold $M$ can be triangulated to admit an orientation-preserving  branched covering map  $M\to \bS^n$, which maps adjacent $n$-simplices to the upper and the lower hemispheres of $\bS^n$ respectively. Such maps from a triangulated $n$-manifold (possibly having boundary) to $\bS^n$ are called Alexander maps. For the definition, we first fix a structure on $\bS^n$.

\subsubsection*{Structures $\bS^n$} Let $\Delta_n = [e_1,\ldots, e_n, e_{n+1} ]$ be the standard Euclidean $n$-simplex in $\R^{n+1}$, and  
\[
\Sigma^2(\Delta_n) = \left( \Delta_n\times \{-1,1\}\right)\Big/{\sim},
\]
be the quotient space for which $\sim$ is the smallest equivalence relation satisfying $(x,1)\sim (x,-1)$ for $x\in \partial \Delta_n$. \index{$\Sigma^2(\Delta_n)$}
Let $[(x,t)]$ be the equivalence class of $(x,t)$ under $\sim$, and 
\[\Pi \colon \Delta_n\times \{-1,1\}\to \Sigma^2(\Delta_n)\]
 be the quotient map $(x,t)\mapsto [(x,t)]$.

The structure on $\Sigma^2(\Delta_n)$, consisting of two simplices (hemispheres) $(\Delta_n)_{\pm 1}= \Pi(\Delta_n\times \{\pm 1\})$ having all vertices 
\[
w_0,w_1,\ldots, w_n
\]
in common  and  intersection $(\Delta_n)_{+1}\cap (\Delta_n)_{-1} = \partial ((\Delta_n)_{+1}) = \partial ((\Delta_n)_{-1})$, is  a (non-simplicial) $\CW$-complex.
 
The space of $\Sigma^2(\Delta_n)$ is homeomorphic to $\bS^n$. In fact, $\Sigma^2(\Delta_n)$ admits a path metric $d_{\Sigma^2(\Delta_n)}$ in which $n$-simplices $(\Delta_n)_{\pm 1}$ are isometric to the standard $n$-simplex $\Delta_n$. For this metric, $\Sigma^2(\Delta_n)$ is bilipschitz homeomorphic to $\bS^n$ with a bilipschitz homeomorphism $\Sigma^2(\Delta_n) \to \bS^n$, whose bilipschitz constant depends only on the dimension $n$ and which maps simplices $(\Delta_n)_{\pm 1}$ to upper and lower hemispheres $\bS^n_+ = \bS^n \cap \R^n \times [0,\infty)$ and $\bS^n_-= \bS^n \cap \R^n\times (-\infty,0]$, respectively. \index{$\Sigma^2(\Delta_n)$}

We give $\bS^n$ also another simplicial structure $\Sigma^2(\Delta^\square_n)$, constructed similarly, where $\Delta^\square_n$ is an $n$-simplex congruent to $n$-simplices in a barycentric triangulation of the unit cube $[0,1]^n$. We also give $\Sigma^2(\Delta^\square_n)$ the associated length metric $d_{\Sigma^2(\Delta^\square_n)}$. \index{$\Sigma^2(\Delta^\square_n)$} 

An affine map $\Delta_n \to \Delta^\square_n$ induces a bilipschitz equivalence of of $\Sigma^2(\Delta_n)$ and $\Sigma^2(\Delta^\square_n)$.  The notations $\Sigma^2(\Delta_n)$ and $\Sigma^2(\Delta^\square_n)$ anticipate the discussion in Part \ref{part:Alexander-Rickman}.

\begin{convention}
In what follows, we consider $\bS^n$ as the flat $n$-sphere $(\Sigma^2(\Delta_n), d_{\Sigma^2(\Delta_n)})$ or the flat $n$-sphere $(\Sigma^2(\Delta^\square_n), d_{\Sigma^2(\Delta^\square_n)})$ depending on the setting.
\end{convention}

We are now ready to define simplicial Alexander maps and cubical Alexander maps.

\begin{definition}
Let $K$ be a simplicial $n$-complex. A map $f\colon |K|\to \bS^n$ is 
\emph{$(K, \Sigma^2(\Delta_n))$-simplicial} if $f$ maps each $n$-simplex in $K$ to an $n$-simplex in $\Sigma^2(\Delta_n)$ by an affine map.
\end{definition}

Alexander maps are simplicial maps, which map adjacent simplices to opposite hemispheres. 

\begin{definition}
\label{def:simplicial-Alexander-map}
Let $K$ be a simplicial $n$-complex whose space $|K|$ is an orientable connected manifold. A map $f\colon |K|\to \bS^n$ 
is a \emph{$K$-Alexander map} \index{Alexander map} if $f$ is $(K, \Sigma^2(\Delta_n))$-simplicial and $f$ maps adjacent $n$-simplices to opposite hemispheres. We call the pair $(K,f)$ a (simplicial) Alexander complex.
\end{definition}

Cubical Alexander maps are Alexander maps defined on the barycentric triangulations $K^\Delta$ of cubical complexes $K$ and whose target $\bS^n$ is equipped with the structure $\Sigma^2(\Delta^\square_n)$.

\begin{definition}\label{def:cubical-Alexander-map}
Let $K$ be an $n$-dimensional cubical complex  on an orientable connected manifold. A mapping $f\colon |K|\to \bS^n$ is called a \emph{cubical $K$-Alexander map} (or \emph{$K^\Delta$-Alexander map})  if $f$ is $(K^\Delta, \Sigma^2(\Delta^\square_n))$-simplicial and $f$ maps adjacent $n$-simplices to opposite hemispheres. We call the pair $(K,f)$ a (cubical) Alexander complex. \index{cubical Alexander map}
\end{definition}

\begin{remark}
\label{rmk:Alexander-is-simplicial-isometry}
The reason for separate definitions for simplicial and cubical Alexander is geometrical. In these definitions, a $K$-Alexander map $|K|\to \bS^n$, $K$ either simplicial or cubical, is an isometry on simplices when $|K|$ is equipped with the flat metric associated to the flat structure of $K$.
\end{remark}

All cubical complexes, whose spaces are orientable connected manifolds, admit cubical Alexander maps. We formulate this observation as follows.

\begin{proposition}
\label{prop:orientation_Alexander}
Let  $K$ be an $n$-dimensional cubical complex whose space is  an orientable connected manifold. Then there exists a cubical $K$-Alexander map $f \colon |K| \to \bS^n$.
\end{proposition}

Readers familiar with Alexander maps may notice immediately that this proposition is essentially a result of Alexander \cite{Alexander} stated in our terminology. The main difference is that we consider here maps with respect to a fixed triangulation $K^\Delta$. For other results on the existence of Alexander maps with prescribed local behavior; see e.g.\;Rickman \cite[Section 2]{Rickman_Acta} and Peltonen \cite{Peltonen}.

\begin{proof}[Proof of Proposition \ref{prop:orientation_Alexander}]
Since we may map the vertices of $K^\Delta$ by the formula $v\mapsto w_k$ if $v$ is in the interior of a $k$-cube in $K$, it suffices to construct a function $\nu \colon (K^\Delta)^{[n]} \to \{-1,1\}$, which assigns opposite signs for adjacent $n$-simplices.

Let $\omega = \sum_{\sigma\in (K^\Delta)^{(n)}} \omega_\sigma$ be a representative of a fundamental class in $H_n(|K|, \partial |K|)$, where $\omega_\sigma$ is an oriented $n$-simplex, i.e.\;an $n$-simplex $\sigma$ with a choice of a permutation class of vertices.

Let now $\sigma\in (K^\Delta)^{(n)}$. Then $\sigma$ has a (unique) preferred choice for the order of its vertices induced by the canonical triangulation, namely $\sigma = [v_0,\ldots, v_n]$, where each $v_k$ is in the interior of a $k$-cube for $k>0$. Let $\nu \colon (K^\Delta)^{(n)} \to \{-1,1\}$ be the function satisfying $[v_0,\ldots, v_n] = \nu(\sigma) \omega_\sigma$, as oriented simplices, for each $\sigma = [v_0,\ldots, v_n]$.

To verify that $\nu$ assigns opposite sign for adjacent $n$-simplices, let $\sigma$ and $\sigma'$ be adjacent $n$-simplices in $K^\Delta$. Then $\sigma = [v_0,\ldots, v_n]$ and $\sigma' = [v'_0,\ldots, v'_n]$, where $v_k$ and $v'_k$ are vertices in $K^\Delta$ which are in the interior of a $k$-cube in $k$ for $k>0$. Since $\sigma \cap \sigma'$ is an $(n-1)$-simplex, we conclude that there exists unique $j\in \{0,\ldots, n\}$ for which $v'_j \ne v_j$ and $v_k = v'_k$ for $k\ne j$. Since $\sigma \cap \sigma' = [v_0,\ldots, \widehat v_j, \ldots, v_n]$, exactly one of $[v_0,\ldots, v_n]$ and $[v'_0,\ldots, v'_n]$ is positively oriented with respect to $\omega_\sigma$ and $\omega_{\sigma'}$ in the fixed chain $\omega$. Thus $\nu$ has the wanted property.
\end{proof}

\begin{remark}
Let $f\colon |K|\to \bS^n$ be a cubical $K$-Alexander map. From now on, we always assume as we may, after relabeling the vertices of $\Sigma^2(\Delta_n)$ if necessary, that $f(v) = w_k$ for each $k\in \{0,\ldots, n\}$ and each $v\in (K^\Delta)^{[0]}$ that is in the interior of a $k$-cube in $K$.
\end{remark}

The proof of Proposition \ref{prop:orientation_Alexander} is also applicable to a barycentric subdivision $T^\Delta$ of a simplicial $n$-complex $T$; see Section \ref{sec:subdivision-triangulation}. Thus we also obtain the following classical result, which we record as a corollary.

\begin{corollary}
\label{cor:orientation_Alexander-simplicial}
Let $T$ be a simplicial $n$-complex whose space is an orientable connected manifold. Then there exists a $T^\Delta$-Alexander map $|T^\Delta|\to \bS^{n}$.
\end{corollary}

\newcommand{\sptMod}{\mathrm{supp}}
\newcommand{\Alex}{\mathrm{Alex}}

\section{Quasiregular mappings on simplicial complexes}

For the forthcoming discussion, we give the following definition of a quasiregular map $|K|\to \bS^n$, where $K$ is a simplicial $n$-complex on an oriented manifold equipped with a flat structure and flat metric $d_K$.

An orientation preserving continuous map $f\colon |K|\to \bS^n$ is \emph{$\sK$-quasiregular} if, for each $n$-simplex $\sigma \in K$, $f|_{\interior |\sigma|} \colon \interior |\sigma|\to\bS^n$ is $\sK$-quasiregular in the usual sense of manifolds with respect to the flat metrics on $\sigma$ and $\bS^n=\Sigma^2(\Delta_n)$, that is,  $f$ belongs to the Sobolev space $W^{1,n}_\loc(|K|,\bS^n)$ and satisfies the distortion inequality 
\[
\norm{Df}^n \le K J_f
\]
almost everywhere in $|K|$, where $\norm{DF}$ is the operator norm of the differential and $J_f$ the Jacobian determinant.

Recall that a quasiregular mapping $f$ is called a \emph{mapping of $\sL$-bounded length distortion (or $\sL$-BLD)} if 
\[
\frac{1}{\sL} \le \norm{DF} \le \sL
\]
almost everywhere. Also, $f$ is $\sL$-BLD if $f\colon |K|\to \bS^n$ is a discrete, open, and sense-preserving map satisfying 
\[
\frac{1}{\sL} \ell(\gamma) \le \ell(\gamma \circ f) \le \sL \ell(\gamma)
\]
for all paths $\gamma$, where $\ell(\cdot)$ is the length of a path. We refer to Martio--V\"ais\"al\"a \cite{Martio-Vaisala_BLD} and Heinonen--Rickman \cite{Heinonen-Rickman_Duke} for a detailed discussion on definitions of BLD mappings.

We make two remarks on BLD mappings.

\begin{remark}
An $L$-BLD mapping $f\colon |K|\to \bS^n$ is $L^{2n}$-quasiregular. Indeed, by definition $f$ is quasiregular and by the BLD property
\[
\norm{Df}^n \le L^n \le L^{2n} J_f.
\]
\end{remark}

\begin{remark}
\label{rmk:Alexander-BLD}
Related to Alexander maps, we note that, for a cubical $n$-complex $K$, cubical $K$-Alexander maps $|K|\to \Sigma^2(\Sigma^\square_n)$ are $1$-BLD.
\end{remark}

In what follows, the BLD mappings play an important role in the homotopy theory of quasiregular mappings. Indeed, it is easy to check that a product mapping $F\colon |K|\times [0,1] \to \bS^n \times [0,1]$, $(x,t) \mapsto (f(x),t)$, is quasiregular if and only if the mapping $f \colon |K|\to \bS^n$ is BLD. 

In the following, an orientation preserving continuous mapping $F\colon |K|\times [a,b]\to \bS^n\times [a,b]$ is said to be \emph{$\sK$-quasiregular} if $F|_{\interior |\sigma| \times (0,1)} \colon \interior |\sigma| \times (a,b) \to \bS^n \times (a,b)$ is $\sK$-quasiregular in the usual sense of manifolds with respect to product metrics.

\begin{remark}
Suppose that $K$ is a simplicial complex on a Riemannian manifold $M$ as in Section \ref{sec:From-Riem-To-Cubical}. Since the flat sphere is bilipschitz to the standard unit $n$-sphere in $\R^{n+1}$, quasiregular mappings $|K|\to \bS^n$ in this sense are quasiregular, quantitatively, in the usual sense. The same observation applies also to quasiregular mappings $|K|\times [a,b]\to \bS^n\times [a,b]$. 
\end{remark}

\begin{definition}
A mapping $F\colon |K|\times [a,b]\to \bS^n\times [a,b]$ is a \emph{$\sK$-quasiregular homotopy from $f \colon |K|\to \bS^n$ to $f' \colon |K|\to \bS^n$} if $F(x,a) = (f(x),a)$ and $F(x,b) = (f'(x),b)$ for $x\in |K|$. We also say that $F$ is a \emph{quasiregular homotopy modulo $|\partial K|$} if, in addition, $F(x,t)=F(x,a)$ for each $x\in |\partial K|$.
\end{definition}

\subsection{Quasiregular expansion of Alexander maps}

In this section, we introduce the terminology for quasiregular modifications of Alexander maps. Let again $K$ be a simplicial complex, whose space is an oriented manifold, with a flat structure and flat metric $d_K$.

\begin{definition}
A quasiregular map $f\colon |K|\to \bS^n$ is a \emph{quasiregular modification of a $K$-Alexander map $h \colon |K|\to \bS^n$} if there exists a closed  $G \subsetneq |K|$ for which $|K^{[n-1]}| \cap \interior G = \emptyset$  and  $f|_{|K|\setminus G} = h|_{|K|\setminus G}$.  
\end{definition}

Associated to a modification $f$ of $h$, we denote $\sptMod f$ the minimal $G$ for which $f|_{|K|\setminus G} = h|_{|K|\setminus G}$ and call $\sptMod f$ the \emph{support of the modification}.

A simplicial complex $K$, whose space is an $n$-manifold, admits either none, or exactly two Alexander maps for which one is orientation preserving and the other orientation reversing. Thus, given a quasiregular modification $f\colon |K|\to \bS^n$ of a $K$-Alexander map, the original Alexander map $|K|\to \bS^n$ is uniquely determined. We denote $f_\Alex$ the Alexander map associated to $f$.

We now give the main definition for studying homotopy of simple covers.

\begin{definition}
\label{def:expansion} \index{Alexander map!expansion of}
A quasiregular map $f \colon |K|\to \bS^n$ is an \emph{expansion of a $K$-Alexander map $h\colon |K|\to \bS^n$ on $K$}, if $f$ is a modification of $h$ and its support $\sptMod f$ is a union of mutually disjoint $n$-cells $\widetilde E = \widetilde E_1 \cup \cdots \cup \widetilde E_k$ such that 
for each $n$-cell $\widetilde E_i$, there exists an $n$-cell $E_i \subset \widetilde E_i$ for which \index{Alexander map!support of expansion} \index{$\sptMod f$}
\begin{enumerate}
\item $f(E_i) = \bS^n$, 
\item $f|_{\interior E_i} \colon \interior E_i \to \bS^n$, and
\item $f(\widetilde E_i \setminus \interior E_i)\ne \bS^n$.
\end{enumerate}

In this case, we  say that  \emph{$f$ is an expansion on $K$}, for short. \index{expansion on $K$}
\end{definition}

We call a pair $(\widetilde E_i, E_i)$ of $n$-cells, as in the definition of quasiregular expansion, a \emph{simple cover in $f$}, and say that $f$ is an \emph{expansion on $K$ by simple covers $(\widetilde E_1, E_1),\ldots, (\widetilde E_k, E_k)$}. \index{simple cover}

 The terms 'expansion' and 'simple cover' stem from the following observation.
\begin{remark} 
Let $f\colon |K|\to \bS^n$ be a quasiregular expansion. Since $f$ is a orientation preserving, each component of $\sptMod f$ adds the degree of $f_\Alex$ by one. Thus, in notation of the previous definition, $\deg f = \deg f_\Alex + k$.
\end{remark}

\section{Homotopy theorems for simple covers}\label{sec:simple-cover}

We now discuss homotopy theorems on shrinking and relocating simple covers. The main results in this section are Proposition \ref{prop:shrink} and \ref{prop:move}; see Figure \ref{fig:simple_cover_move} for an illustration.

\begin{figure}[h!]
\begin{overpic}[scale=.50,unit=1mm]{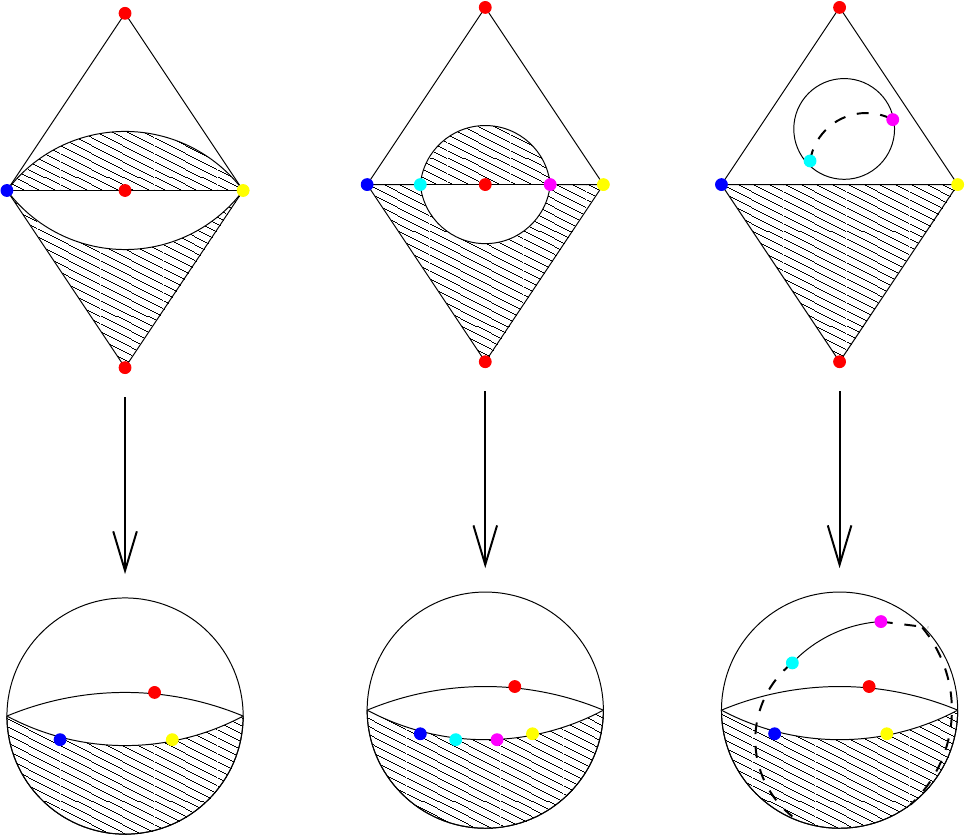} 
\end{overpic}
\caption{Shrinking and relocating a simple cover.}
\label{fig:simple_cover_move}
\end{figure}

As already noted that, in general, products of quasiregular mappings need not be quasiregular and, even if they are, the distortion constant of the product may depend on the mappings in question as is indicated in the following example.

\begin{example}
Let $A \colon |K|\to \bS^n$ be a $K$-Alexander map. Let now $\sigma\in K^{[n]}$ be an $n$-simplex in $K$ and let $B \subset \sigma$ be a Euclidean ball in $\sigma$. We fix a quasiregular expansion $f_1 \colon |K|\to \bS^n$ of $A$ by a single simple cover $(\widetilde E, E)$ contained in $B$, then take a small concentric ball $B'\subset B$ and scale the simple cover $(\widetilde E, E)$ to a congruent pair $(\widetilde E',E')$ in $B'$. we expand the Alexander map $A \colon |K|\to \bS^n$ again in the pair $(\widetilde E',E')$ to obtain a quasiregular expansion $f_2 \colon |K|\to \bS^n$ of $A$. 
Mappings $f_1$ and $f_2$ can be taken to have the same distortion constant $\sK$ but, if  ball $B'$ is small enough, then they have very different BLD constants $\sL_1$ and $\sL_2$.

Also the distortion constant of the trivial homotopy $F_i \colon |K|\times [0,1]\to \bS^n\times [0,1]$, $(x,t) \mapsto (f_i(x),t)$, from $f_i$ to itself depends on the BLD constant of $f_i$, and hence on the expansion. In particular, the distortion constant of $F_2$ is larger than $\sL_2$ for a small ball $B'$.
\end{example}

 We give concrete examples of BLD expansions of Alexander mappings when we discuss deformation theory in Part \ref{part:deformation}.

As this example shows, in order to obtain quantitative results, it is preferable to work with homotopies of BLD mappings instead of general quasiregular mappings. For this reason, we continue our discussion with BLD expansions.

\subsection{Bilipschitz framing of a simple cover}

We define now a geometric structure for simple covers in a BLD expansion $|K|\to \bS^n$ on $K$ for Propositions \ref{prop:shrink} and \ref{prop:move}.

Let $e_\Diamond = \bar B^{n-1}\times \{0\}$ and let $E_\Diamond$ and $\widetilde E_\Diamond $ be the convex hulls of $\{\pm e_n,  e_\Diamond\}$ and $\{\pm 2e_n, e_\Diamond\}$, respectively; see Figure \ref{fig:Euclidean-package}. We call the pair $(\widetilde E_\Diamond, E_\Diamond)$ a \emph{model package}.
\index{$(\widetilde E_\Diamond, E_\Diamond)$}
Let $\rho_\Diamond \colon \widetilde E_\Diamond \to \widetilde E_\Diamond$ be the map satisfying $\rho_\Diamond|_{\partial  \widetilde E_\Diamond} = \id$, and for each $x\in \bar B^{n-1}$,
\begin{enumerate}
\item $\rho_\Diamond((\{ x\} \times \R) \cap E_\Diamond) = \{x\}$, and
\item $\rho_\Diamond|_{(\{x\} \times \R) \cap (\widetilde E_\Diamond\setminus \interior E_\Diamond)}\colon (\{x\} \times \R) \cap (\widetilde E_\Diamond\setminus \interior E_\Diamond)\to  (\{x\} \times \R) \cap \widetilde E_\Diamond$ is linear.
\end{enumerate}
\index{$\rho_\Diamond$}

Heuristically, $\rho_\Diamond$ projects the inner diamond $E_\Diamond$ onto $\bar B^{n-1}$ while keeping the boundary of the larger diamond $\widetilde E_\Diamond$ fixed.

\begin{figure}[h!]
\begin{overpic}[scale=1,unit=1mm]{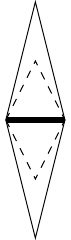} 
\put(4, 22){\tiny $e_\Diamond$}
\put(4, 41){\tiny $\widetilde E_\Diamond$}
\put(4,31){\tiny $E_\Diamond$}
\end{overpic}
\caption{Model package in dimension $n=2$.}
\label{fig:Euclidean-package}
\end{figure}

\begin{definition}\index{simple cover!bilipschitz framing}
Let $f \colon |K|\to \bS^n$ be a BLD expansion on $K$ and $(\widetilde E, E)$ a simple cover in $f$. 
We say that $(\widetilde E, E)$ is a \emph{framed $\sL$-bilipschitz simple cover} if there exists an
$\sL$-bilipschitz embedding $\theta \colon \widetilde E_\Diamond \to |K|$ for which $\theta(\widetilde E_\Diamond)=\widetilde E$, $\theta(E_\Diamond) = E$, 
and 
\[
f \circ \theta|_{\widetilde E_\Diamond \setminus E_\Diamond} = f_\Alex \circ \theta \circ \rho_\Diamond|_{\widetilde E_\Diamond \setminus E_\Diamond}. 
\]
In this case, we call $\theta$ an \emph{$\sL$-bilipschitz framing of $(\widetilde E, E)$}, and call $\widetilde E$ the \emph{support of the simple cover $(\widetilde E, E)$}. \index{simple cover!support}
\end{definition}

We define now a flow of a framing $\theta \colon \widetilde E_\Diamond\to |K|$. For the statement, we denote, for a level preserving mapping $H \colon X\times [a,b]\to Y\times [a,b]$, by $H_t \colon X\to Y$ the mapping defined by $H(x,t) = (H_t(x),t)$ for $(x,t)\in X\times [a,b]$.

\begin{definition} Let $A\colon |K|\to \bS^n$ be a $K$-Alexander map.
A pair $(\Theta, \Psi)$ is a \emph{$K$-flow of a framing $\theta \colon \widetilde E_\Diamond \to |K|$ (for time $[a,b]$)} if there exist
\begin{enumerate}
\item a level preserving bilipschitz embedding $\Theta \colon \widetilde E_\Diamond \times [a,b]\to |K|\times [a,b]$, where $\Theta_a = \theta$ and
\item a level preserving bilipschitz homeomorphism $\Psi \colon \bS^n \times [a,b]\to \bS^n \times [a,b]$, where $\Psi_a = \id$, 
\end{enumerate}
for which 
\[
\Psi_t \circ A \circ \theta|_{e_\Diamond} = A \circ \Theta_t|_{e_\Diamond}. 
\]
for each $t\in [a,b]$.
\end{definition}

The mapping $\Theta_t$ is called a 'relocation of the framing $\theta$' and $\Psi_t$ a 'rotation of the framing $\theta$'. The term 'flow' stems from the following lemma.

\begin{lemma}
\label{lemma:flow}
Let $f\colon |K|\to \bS^n$ be a BLD expansion on $K$, $(\widetilde E, E)$ a simple cover of $f$, $\theta$ a framing of $(\widetilde E, E)$, and $(\Theta, \Psi)$ an $K$-flow of $\theta$ on $[a,b]$. Denote $\widetilde E_t = \Theta_t(\widetilde E_\Diamond)$ and $E_t = \Theta_t(E_\Diamond)$ for $t\in [a,b]$.

Suppose that $(\sptMod f\setminus \widetilde E) \cap \widetilde E_t = \emptyset$  for all $t\in [a,b]$. Then the mappings $f_t \colon |K|\to \bS^n$, defined by 
\begin{enumerate}
\item $f_t|_{|K|\setminus \widetilde E_t} = f|_{|K|\setminus \widetilde E_t}$,
\item $f_t|_{\widetilde E_t \setminus E_t} = f_\Alex \circ \Theta_t \circ \rho_\Diamond \circ \Theta_t^{-1}|_{\widetilde E_t \setminus E_t}$, and
\item $f_t|_{E_t} = \Psi_t \circ f \circ \theta \circ \Theta_t^{-1}|_{E_t}$,
\end{enumerate}
are well-defined BLD expansions on $K$ with $\sptMod f_t = (\sptMod f\setminus \widetilde E) \cup \widetilde E_t$, and the mapping $F\colon |K|\times [a,b]\to \bS^n \times [a,b]$, $(x,t) \mapsto (f_t(x),t)$, is a BLD map.
\end{lemma}

The formula for $f_t$ shows that the result is quantitative in the sense that the BLD constant of $f_t$ depends only on the BLD constant of $f$ and bilipschitz constants of $\theta$, $\Theta_t$, and $\Psi_t$. Similarly, the BLD constant of $F$ depends on BLD constants of $f$ and bilipschitz constants of $\Theta$ and $\Psi$.

\begin{proof}[Proof of Lemma \ref{lemma:flow}]
It suffices to check that the mapping $f_t$ is well-defined. Since $f|_{\partial \widetilde E} = f_\Alex|_{\partial \widetilde E}$ and $\rho_\Diamond = \id$ on $\partial \widetilde E_\Diamond$, we have that $f_t$ is well-defined on $\partial \widetilde E_t$.
 
On $\partial E_\Diamond$, we have that $f \circ \theta = f_\Alex \circ \theta \circ \rho_\Diamond$. Thus, on $\partial E_t$, we have that
\[
\Psi_t \circ f \circ \theta \circ \Theta_t^{-1} = \Psi_t \circ f_\Alex\circ \theta \circ \rho_\Diamond \circ \Theta_t^{-1}. 
\]
Thus $f_t$ is well-defined. 
\end{proof}

\subsection{Two particular flows}

We discuss now two particular flows of simple covers - shrinking and relocating.

\begin{definition}
Let $f \colon |K|\to \bS^n$ be a BLD expansion on $K$. A simple cover $(\widetilde E, E)$ of $f$ is \emph{shrinkable} if there exists a framing $\theta \colon \widetilde E_\Diamond \to |K|$ of $(\widetilde E,E)$ and a flow $(\Theta,\Psi)$ of $\theta$ on $[0,1]$ for which $\diam(\Theta_1(\widetilde E_\Diamond)) < 1$. 
\end{definition}

We call $(\Theta,\Psi)$ an \emph{shrinking flow of $\theta$},  and $\theta$ an shrinkable framing. 

Although the following statement follows immediately from Lemma \ref{lemma:flow}, we state it as a proposition due to its importance.

\begin{proposition}[Shrinking] 
\label{prop:shrink}
Let $K$ be a simplicial Alexander $n$-complex, let $f \colon |K|\to \bS^n$ be a BLD expansion on $K$ and $(\widetilde E, E)$ a shrinkable simple cover of $f$. Then, for each $\varepsilon\in (0,1]$, there exists a BLD homotopy $F \colon |K|\times [0,1] \to \bS^n \times [0,1]$ from $f$ to a BLD expansion $f'\colon |K|\to \bS^n$ on $K$ satisfying $f'|_{|K|\setminus \widetilde E} = f|_{|K|\widetilde E}$, $\sptMod f' \cap \widetilde E = \widetilde E'$ and $\diam \widetilde E' < \varepsilon$. The result is quantitative in the sense that the distortion constant of $F$ and BLD constant of $f'$ depend only the BLD constant of $f$ and bilipschitz constants of a framing of $(\widetilde E,E)$ and its $\varepsilon$-shrinking flow.
\end{proposition}

In the proposition, we may assume that all $\varepsilon$-shrinkable simple covers, for sufficiently small $\varepsilon$, are tame in the following sense.

\begin{definition}
A simple cover $(\widetilde E,E)$ of a BLD expansion $f\colon |K|\to \bS^n$ is \emph{tame} if $\widetilde E$ is contained in a ball $B$ satisfying $2B \cap |K^{(n-2)}\cup \partial K| =\emptyset$.
\end{definition}

We combine now the properties of quasiregular expansions to one definition.

\begin{definition}
A BLD expansion $f \colon |K|\to \bS^n$ on $K$ is \emph{$\sL$-controlled} if  \index{Alexander map!BLD-controlled expansion}
\begin{enumerate}
\item  $f$ is an $\sL$-BLD expansion on $K$ by tame simple covers,
\item $\sptMod f$ is covered by balls $B_i$, $i\in I$, for which each $B_i$  contains exactly one simple cover and  balls $2B_i$ are mutually disjoint, and
\item each simple cover $(\widetilde E,E)$ of $f$ admits an $\sL$-bilipschitz shrinkable framing.
\end{enumerate}
\end{definition}

The other application of the flow is the relocation of simple cover.

\begin{proposition}
\label{prop:move}
Let $f\colon |K|\to \bS^n$ be a BLD expansion on $K$, $(\widetilde E, E)$ a tame simple cover of $f$, and $\theta \colon \widetilde E_\Diamond \to |K|$ a framing of $(\widetilde E, E)$. Let also $T\colon \widetilde E \times [a,b]\to |K|\times [a,b]$ be a level preserving bilipschitz embedding for which
\begin{enumerate}
\item $T_a = \id$,
\item $(\sptMod f \setminus \widetilde E) \cap T_t (\widetilde E)= \emptyset$ for each $t\in [a,b]$, and 
\item $T_t \colon \widetilde E \to |K|$ is an isometry for which $T_t(\widetilde E)$ is contained in a pair of adjacent simplices of $K$ for each $t\in [a,b]$. 
\end{enumerate}
Then there exists a BLD homotopy $F \colon |K|\times [a,b]\to \bS^n\times [a,b]$ from $f$ to a BLD expansion $f'\colon |K|\to \bS^n$ on $K$ for which $\sptMod f' = (\sptMod f\setminus \widetilde E) \cup T_b(\widetilde E)$. The result is quantitative in the sense that the mappings $f$ and $f'$ have the same BLD constant and the distortion constant of $F$ depends only on the BLD constant of $f$ and bilipschitz constant of $T$. Furthermore, if $f$ is an $\sL$-controlled expansion, so is $f'$.
\end{proposition}

\begin{proof}
Let $\Theta \colon \widetilde E_\Diamond \times [a,b]\to |K|\times [a,b]$ be the composition $\Theta = T \circ (\theta \times \id)$. For the flow of $\theta$ it suffices now to construct the rotation $\Psi \colon \bS^n\times [a,b] \to \bS^n \times [a,b]$. Let $\widetilde E_t = T_t(\widetilde E)$ for each $t\in [a,b]$. Recall that we have modeled the $n$-sphere as $\bS^n = \Sigma^2(\Delta^\square_n)$.

Since each $\widetilde E_t$ is contained in a pair of adjacent $n$-simplices in $K$, we have that the Alexander map $f_\Alex$ is an isometry on $\widetilde E_t$. We define, for each $t\in [a,b]$, a mapping $\Psi_t = f_\Alex\circ T_t \circ (f_\Alex|_{\widetilde E})^{-1} \colon f_\Alex(\widetilde E) \to \bS^n$. 

Let now $P$ be the $(n-2)$-skeleton of $\Sigma^2(\Delta^\square_n)$. We extend now each $\Psi_t$ as a continuous map $\bS^n \to \bS^n$ in such a way that $\Psi_t$ is a local isometry for each $y\in \bS^n$ for which $y\not \in P$ and $\Psi_t(y)\not \in P$. Then $\Psi_t$ is a bilipschitz map. Since the extension is unique in $P\cup \Psi_t^{-1}P$, we have that the mapping $\Psi \colon \bS^n \times [a,b]\to \bS^n \times [a,b]$, $(x,t) \mapsto (\Psi_t(x),t)$, is a bilipschitz map. Moreover, we have that 
\[
\Psi_t \circ f_\Alex \circ \theta|_{e_\Diamond} = f_\Alex \circ T_t \circ \theta|_{e_\Diamond} = f_\Alex \circ \Theta_t|_{e_\Diamond}.
\]
Thus $(\Theta, \Psi)$ is a flow of $\theta$. The claim follows now from Lemma \ref{lemma:flow}.
\end{proof}

An immediate corollary of Proposition \ref{prop:move} and its proof is the following result.

\begin{corollary} 
Let $B\subset |K|$ be a ball having the property that $2B$ is contained in an $n$-simplex of $K$. Then, for an $\sL$-controlled BLD expansion $f\colon |K|\to \bS^n$ on $K$, there exists a BLD homotopy $F\colon |K|\times [0,D] \to \bS^n \times [0,D]$, where 
$D = \max\{ \dist(x,B)\colon x\in \sptMod f\}$,  from $f$ to an $\sL$-controlled BLD expansion $f'\colon |K|\to \bS^n$ satisfying $\sptMod f' \subset B$. The result is quantitative in the sense that the distortion constant of $F$ depends only on $\sL$.
\end{corollary}



\part{Partition} 
\label{part:separating_complexes}

\chapter{Indentation in good complexes}


\section{Goal of this chapter}

The goal of this chapter is to prove the following result. 
For the statement, we define the notion of local multiplicity of a complex; the definition of a good complex is defined in the next section.

\begin{definition}
\label{def:mu}
\index{local multiplicity of a complex}
\index{$\mu(K)$}
Given a cubical $n$-complex $K$, the quantity
\[
\mu(K) = \max\left\{\#\{q\in K^{[n-1]}\colon e\in q^{[n-2]}\} \colon e\in K^{[n-2]}  \right\}
\]
is called the \emph{local multiplicity of $K$}. 
\end{definition}

\begin{theorem}\label{theorem:RC-vague}
Let $K$ be a good cubical $n$-complex and $Y\subset K$ an adjacently connected $(n-1)$-dimensional subcomplex. Given $m\ge 1$ and $\nu\in \N$ satisfying $3^\nu \ge 3^{10} m \mu(K)^2$,  there exist $\lambda=\lambda(m, \# Y^{[n-1]})\ge 1$,$L=L(n,\mu(K))\ge 1$,  an $n$-subcomplex $D$ in $\Refine^\nu(K)$, and an $(n-1)$-complex $S$ in $ D$ for which 
\begin{enumerate}
\item the cut-graph $\Gamma(D; S)$ has $m$ components $\tau_1,\ldots, \tau_m$, each of which is a tree of size at most $\lambda$,
\item for each $q\in Y^{[n-1]}$ and $i=1,\ldots, m$, there exists an $n$-cube in $\Span_{\Refine^\nu(K)}(\tau_i)$ having a face in $q$, and 
\item there exists an $L$-bilipschitz homeomorphism 
\[|\Real_K(Y)-\widetilde D| \to |\Real_K(Y)|,\]
where $\widetilde D \subset \Real_K(Y)$
 is the $n$-dimensional subcomplex for which $\pi_{(K,Y)}(\widetilde D) = D$.
\end{enumerate}
\end{theorem}

Heuristically this result says that we may remove a collection of  tunnels -- separated by complex $S$ -- which visit every $(n-1)$-cube of a given subcomplex $Y$, without changing the bilipschitz geometry of the complement of $Y$ with respect to realization. In the forthcoming chapters we use this observation to modify the complementary components of a given (separating) $(n-1)$-subcomplex; see Chapter \ref{chap:Separating-complexes}.

To give  a precise description on the subcomplex $D$, we introduce the notions of indentation and bent indentation; the complex $\widetilde D$ in Theorem \ref{theorem:RC-vague}  is a bent indentation in $\Real_K(Y)$. Then, in Section \ref{sec:reservoir-canal-system}, we give a concrete construction of  a reservoir-canal system, which proves the existence of the tunnels above and gives also a more detailed version of  Theorem \ref{theorem:RC-vague}.

\section{Good complexes} 

We formalize a notion of good complexes by collecting the essential technical assumptions needed in building stable complexes.

\begin{definition}
\label{def:good-complex} \index{good complex}
A  cubical $n$-complex $K$ is \emph{good} if 
\begin{enumerate}
\item $K$ has a flat structure $\sF_K$, \label{item:good-complex-flat}
\item $\Gamma(K)$ is connected,  \label{item:good-complex-adjacency-1}
\item for each component $\Sigma \subset |\partial K|$, the graph $\Gamma(\partial K|_\Sigma)$ is connected, \label{item:good-complex-adjacency-2}
\item the subcomplex $\Star_K(\partial K)$ is isomorphic to $\partial K \times [0,1]$, and \label{item:good-complex-collar}
\item each $(n-1)$-cube is a face of at most two $n$-cubes, and \label{item:good-complex-regularity-a}
\item star of each $(n-2)$-cube in $K$ is an adjacently connected $n$-complex. \label{item:good-complex-regularity-b} 
\end{enumerate}
\end{definition}


Condition \eqref{item:good-complex-adjacency-2} implies that components of $|\partial K|$ are in one-to-one correspondence with the components of $\Gamma(\partial K)$. \emph{A fortiori} conditions \eqref{item:good-complex-adjacency-2} and \eqref{item:good-complex-collar} together imply that, if $\Sigma$ and $\Sigma'$ are components of $\partial K$, then $\Star_K(\Sigma) \cap \Star_K(\Sigma') = \emptyset$.

By condition \eqref{item:good-complex-regularity-a}, each $(n-1)$-cube in a good cubical $n$-complex is a face of at most two $n$-cubes. Thus, for a good cubical $n$-complex $K$, we may identify an edge $\{Q,Q'\}$ of $\Gamma(K)$ with the face $Q\cap Q'\in K^{[n-1]}$, hence the set of edges of $\Gamma(K)$ with a subset of $K^{[n-1]}$. Specifically, edges of $\Gamma(K)$ correspond precisely to the  $(n-1)$-cubes in $K$ which are not boundary cubes, i.e, we may identify $\Gamma(K)$ with the pair
\[
\Gamma(K) = \left( K^{[n]}, \{ q \in K^{[n-1]} \colon q\not \in \partial K\} \right).
\]

Conditions \eqref{item:good-complex-regularity-a} and \eqref{item:good-complex-regularity-b} together force the adjacency graphs of the stars of codimension $2$ simplices to be cyclic. For the statement, we say that a graph $G=(V,E)$ is \emph{cyclic} if it is isomorphic to either
\[
(\{1,\ldots, m\}, \{ \{1,2\}, \{2,3\},\ldots, \{m-1,m\}\})
\]
or 
\[
(\{1,\ldots, m\}, \{ \{1,2\}, \{2,3\},\ldots, \{m-1,m\}, \{m,1\}\}).
\]

\begin{lemma}
\label{lemma:star-cyclic}
Let $K$ be a good cubical $n$-complex. Then the adjacency graph $\Gamma(\Star_K(\xi))$ of the star of an $(n-2)$-cube $\xi\in K^{[n-2]}$ is cyclic.
\end{lemma}

\begin{proof}
Let $\xi\in K^{[n-2]}$. By condition \eqref{item:good-complex-regularity-b} in Definition \ref{def:good-complex}, the adjacency graph $\Gamma(\Star_K(\xi))$ is connected. Let $m\in \Z_+$ be the number of $n$-cubes in $\Gamma(\Star_K(\xi))$. 

We observe that the valence of $\Gamma(\Star_K(\xi))$ at each vertex is at most two. Indeed, let $Q\in \Star_K(\xi)$ be an $n$-cube. Then $\xi$ is intersection of exactly two faces of $Q$. Thus $Q$ is adjacent to at most two other $n$-cubes in $\Star_K(\xi)$. Thus $\Gamma(\Star_K(\xi))$ is a connected graph having valence at vertices at most two. We conclude that $\Gamma(\Star_K(\xi))$ is cyclic. 
\end{proof}

The notion of a good complex is stable under refinement: refinements of good complexes are also good complexes. We omit the straightforward proof.

\begin{lemma}
\label{lemma:refinement-strongly-connected-good-complex}
Let $K$ be an adjacently connected cubical $n$-complex. Then also $\Refine(K)$ is also an adjacently connected cubical $n$-complex. If $K$ is a good cubical $n$-complex then  $\Refine(K)$ is a good cubical $n$-complex.
\end{lemma}

\subsection{Existence of good cubical complexes}
\label{sec:existence-good-complexes}

The cubical structure $T^\boxplus$, associated to a triangulation $T$ of a manifold with boundary, is a good cubical complex modulo condition \eqref{item:good-complex-collar}.

\begin{proposition}
\label{prop:cubical-starting-point}
Let $T$ be a simplicial complex whose space $|T|$ is a connected $n$-manifold with boundary. Let $\iota \colon \partial T^\boxplus \to \partial T^\boxplus \times [0,1]$ be the cubical inclusion $q \mapsto q\times \{0\}$, and let $K$ be the cubical $n$-complex
\[
K = T^\boxplus \, {\bigcup}_\iota \left(\partial T^\boxplus \times [0,1]\right).
\]
Then $K$ is a good cubical $n$-complex.
\end{proposition}

\begin{proof}
We verify the conditions of a good complex. condition \eqref{item:good-complex-collar}  is satisfied by the construction, so it suffices to check the other conditions.

By Proposition \ref{prop:PL-to-cubical}, the complex $T^\boxplus$ has a flat structure $\mathcal F_{T^\boxplus}$. It is straightforward to extend the flat structure $\mathcal F_{T^\boxplus}$ to a flat structure of $K$.

To show that $\Gamma(K)$ is connected we proceed as follows. Since $|T|$ is a connected manifold with boundary, $T$ is adjacently connected in terms of $n$-simplices, that is, for $n$-simplices $\sigma$ and $\sigma'$ in $T$, there exists a chain $\sigma = \sigma_0, \ldots, \sigma_m=\sigma'$ of $n$-simplices for which $\sigma_{i-1}\cap \sigma_i$ is an $(n-1)$-simplex for each $i\in \{1,\ldots, m\}$. Since $\Delta_n^\boxplus$ is adjacently connected cubical $n$-complex and $T$ is adjacently connected in the simplicial sense, we conclude that $\Gamma(T^\boxplus)$ is connected. Since each $n$-cube in $K$ is either in $T^\boxplus$ or adjacent to an $n$-cube in $T^\boxplus$, we have that $\Gamma(K)$ is connected.

Let $S \subset |\partial K|$ be a component and let $\Sigma = K|_S$. Let $\Sigma' \subset \Star_K(S)$ be the link of $S$, that is, the subcomplex of $\Star_K(S)$ which does not meet $\Sigma$. By the construction of $K$, $\Sigma$ and $\Sigma'$ are isomorphic and $\Sigma' \subset \partial T^\boxplus$. Let $S'=|\Sigma'|$. Since $S$ is a component of $|\partial K|$,  $S'$ is a component of $|\partial T|$, hence an $(n-1)$-manifold. Since the simplicial adjacency graph of $T|_{S'}$ is connected,  $\Gamma(T^\boxplus|_{S'})$ is also connected. Thus $\Gamma(\Sigma)$ is connected.

To check condition \eqref{item:good-complex-regularity-a}, we observe that, since $|T|$ is an $n$-manifold,  each $(n-1)$-simplex in $T$ is a face of at most two $n$-simplices in $T$. Thus each $(n-1)$-cube in $T^\boxplus$ is a face of at most two $n$-cubes in $T^\boxplus$.

Finally, condition \eqref{item:good-complex-regularity-b} follows from Lemma \ref{lemma:star-cyclic}.
\end{proof}

\section{Spectrum of a subcomplex}
For the definition of an indentation, we introduce the notion of spectrum of a subcomplex.

\begin{definition}
\label{def:spectrum}
\index{spectrum of a subcomplex} \index{$\cS_*(P)$}
Let $K$ be a cubical $n$-complex. The \emph{spectrum $\cS_*(P)$} of an $n$-dimensional subcomplex $P$ of $\Refine^k(K)$ is a sequence 
\[
\cS_*(P) = (\cS_0(P), \cS_1(P),\ldots, \cS_k(P)),
\]
where $\cS_j(P)$ is a subcomplex of $\Refine^j(K)$ for each $j\in \{0,1,\ldots, k\}$, satisfying the following properties:
\begin{enumerate}
\item $P = \bigcup_{j=0}^k \Refine^{k-j}(\cS_j(P))$ and 
\item $\cS_j(P)$ does not contain the refinement complex $\Refine(Q)$ of any $n$-cube $Q\in \Refine^{j-1}(K)$.  \label{item:disjoint}
\end{enumerate}
We denote 
\[
\cS_*(P)^{[n]} = \bigcup_{j=0}^k \cS_j(P)^{[n]}
\]
call the $n$-cubes in this family \emph{spectral cubes of $P$}.
\end{definition}

\begin{figure}[htp]
\begin{overpic}[scale=.5,unit=1mm]{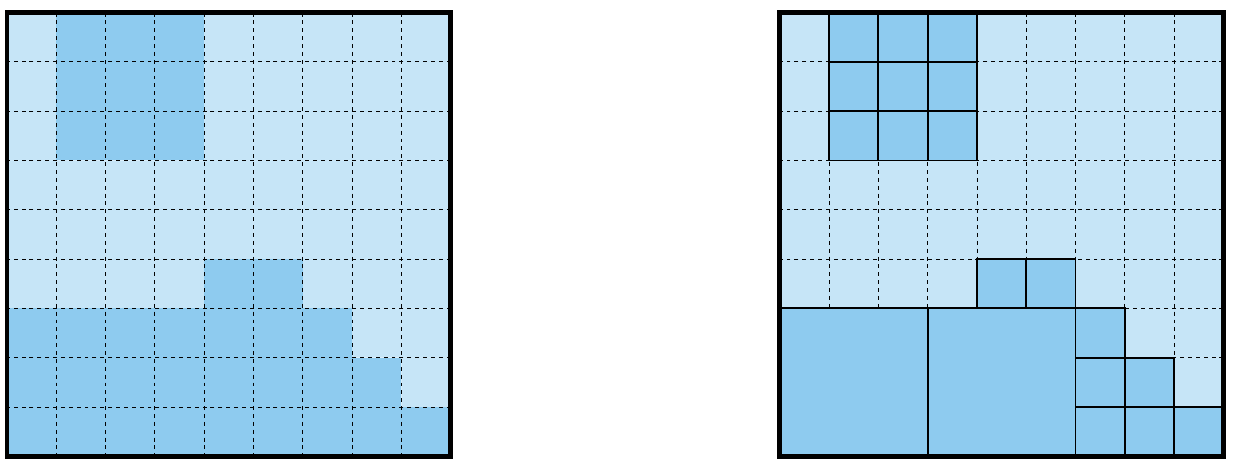} 
\put(16,7){{$P$}}
\put(10,33){{$P$}}
\end{overpic}
\label{fig:Spectrum}
\caption{Left: a cubical complex $K$ consisting of one $2$-cube and a subcomplex $P$ of $\Refine^2(K)$. Right: spectral cubes $\cS_*(P)^{[2]}$ of $P$.}
\end{figure}

\begin{remark} 
First, in view of condition \eqref{item:disjoint} in the definition, the complexes $\Refine^{k-j}(\cS_j(P))$, $j=0,1,\dots, k$, are essentially disjoint. Thus each subcomplex $P \subset \Refine^k(K)$ has unique spectrum $\cS_*(P)$.
Second,  the spectrum of a subcomplex is stable under refinements, that is, for $P \subset \Refine^k(K)$, we have $\cS_*(\Refine(P)) = \cS_*(P)$.
\end{remark}

Let $P \subset \Refine^k(K)$ be a non-empty subcomplex, and $k\geq 1$. We say that $\delta(P)\in \{0,\ldots, k\}$ is the \emph{depth of the spectrum of $P$} if $\delta(P)$ is the largest index $j$ for which $\cS_j(P)\ne \emptyset$. We also define a \emph{spectral reduction of $P$} to be the subcomplex \index{spectral reduction} \index{$P^\dagger$} 
\[
P^\dagger = \bigcup_{j=0}^{\delta(P)-1} \Refine^{k-j}(\cS_j(P))
\]
of $P$. Note that $\delta(P^\dagger) < \delta(P)$ and that $\cS_j(P^\dagger) = \cS_j(P)$ for $0\le j < \delta(P)$.

\section{Indentations}
\label{sec:indentation}

We are now ready to define indentations, that is, subcomplexes which lie in a spectral sense along the boundary of a given complex.

\begin{definition}
\label{def:indentation}
\index{indentation}
Let $K$ be a cubical $n$-complex with boundary. A cubical complex $D\subset \Refine^k(K)$ is an \emph{indentation in $K$} if the spectrum $\cS_*(D)$ has the following properties:
\begin{enumerate}
\item $\cS_0(D)=\cS_1(D) = \emptyset$; \label{item:indentation-1}
\item for each $j\in\{2,\ldots, k\}$, each $n$-cube $Q\in \cS_j(D)^{[n]}$ has exactly one  face  in $\Refine^j(\partial K)$; \label{item:indentation-2}
\item for $Q\in \cS_j(D)$ and $Q'\in \cS_{j'}(D)$, where $j\le j'$, the intersection $Q\cap Q'$ is either empty or a face of $Q'$ which, in the case of $j<j'$, intersects at most two faces of $Q$;
\label{item:indentation-3} 
\item for any three distinct cubes $Q,Q',Q''$ in $\cS_*(D)^{[n]}$, the intersection $Q\cap Q'\cap Q''$ is empty. \label{item:indentation-4}
\end{enumerate}
\end{definition}

Note that an indentation need not be connected by definition; each component is also an indentation by itself.  However each connected component of an indentation must be adjacently connected.

See Figure \ref{fig:bent-indentation} for examples of indentation $D$.

We record some simple facts about indentation which follow from the definition.

\begin{remark}\label{rmk:dent-cell}
Let $D\subset \Refine^k(K)$ be an indentation in a cubical $n$-complex. Suppose that $D$ is connected. Then $|D|$ is an $n$-cell, $|D \cap \partial K|$ is an $(n-1)$-cell, and $|\partial D - \partial K|$ is also an $(n-1)$-cell. 
\end{remark}

\begin{remark}
\label{rmk:indentation-1}
condition \eqref{item:indentation-3} in the definition yields the following. Suppose that $Q\in \cS_j(D)$ and $Q'\in \cS_{j'}(D)$,  $j\le j'$, have nonempty intersection. Then this face $Q\cap Q'$ of $Q'$ is not contained in $\partial K$.
\end{remark}

\begin{remark}
\label{rmk:indentation-2}
If an $n$-cube $Q\in \cS_j(D)^{[n]}$ meets two $n$-cubes $Q'$ and $Q''$ in $\cS_{j'}(D)^{[n]}$ for $j'\leq j$,  then, by \eqref{item:indentation-3}  and  \eqref{item:indentation-4}, $Q\cap Q'$ and $Q\cap Q''$ are opposite faces of $Q$.

An $n$-cube $Q\in \cS_j(D)^{[n]}$, however, may meet two or more smaller $n$-cubes $Q', Q''\in \cS_*(D)^{[n]}$ in the same face of $Q$.
\end{remark}

\subsection{Wedges}
Wedge is a notion of neighborhood adapted to the spectrum of  indentation. We give the definition of wedge here, however, more generally for all subcomplexes.

\begin{definition}
\label{def:wedge}
\index{wedge} \index{$\Wedge_K(P)$}
Let $K$ be a cubical $n$-complex with boundary, and $P \subset \Refine^k(K)$ be a subcomplex. We say that $x\in |K|$ is in the \emph{wedge of $P$} if there exists $p\in |P|$ for which 
\[
\dist_K(x,|P|) = d_K(x,p) \le \dist_K(p,|\partial K|)/4,
\]
where $d_K$ is the standard metric on $K$ (see Section \ref{sec:refinement-metric}). We call
\[
\Wedge_K(P)= \{ x\in |K| \colon x \text{ is in the wedge of } P\}
\]
the \emph{wedge of $P$ in $|K|$}.
\end{definition}

\begin{remark}
Observe that  $\Wedge_K(P)=\Wedge_{\Refine(K)}(\Refine(P))$.
\end{remark}

A key property of the wedge is the following non-overlapping property associated to the spectral decomposition of an indentation.

\begin{lemma}
\label{lemma:Wedge-intersections}
Let $k\geq 2$, and $D\subset \Refine^k(K)$ be an indentation. Then
\begin{enumerate}
\item if $Q$ and $Q'$ are cubes in $\cS_*(D)^{[n]}$ for which $Q\cap Q'=\emptyset$, then $\Wedge_K(Q)\cap \Wedge_K(Q')=\emptyset$; and \label{item:Wedge-intersections-1}
\item if $Q,Q',Q''$ are three distinct cubes in $\cS_* (D)^{[n]}$,  then $\Wedge_K(Q)\cap \Wedge_K(Q') \cap \Wedge_K(Q'')=\emptyset$. \label{item:Wedge-intersections-2}
\end{enumerate}
\end{lemma}
\begin{proof}
To prove \eqref{item:Wedge-intersections-1}, let $Q\in \cS_j(D)$ and $Q'\in \cS_{j'}(D)$ be two non-intersecting $n$-cubes. We may assume that
$j\leq j'$. Then the distance between $Q$ and $Q'$ is at least $3^{-j'}$.

Suppose towards contradiction that $\Wedge_K(Q) \cap \Wedge_K(Q')\ne \emptyset$ and let $x\in \Wedge_K(Q')\cap \Wedge_K(Q'')$. Let also $p\in Q$ and $p'\in Q'$ be points for which $\dist(x,Q) = d(x,p)\le \dist(p,|\partial K|)/4$ and $\dist(x,Q') = d(x,p')\le \dist(p',|\partial K|)/4$. 
Observe that
\begin{align*}
\dist(p,|\partial K|)&\leq d(p,x)+d(x,p')+ \dist(p',|\partial K|) \\
&\leq \dist(p,|\partial K|)/4+ 5\, \dist(p',|\partial K|)/4.
\end{align*}
Thus, $\dist(p,|\partial K|)\leq 5\, \dist(p',|\partial K|)/3$, and
\begin{align*}
\dist(Q,Q') &\le d(p,p') \le d(p,x) + d(x,p') \\ &\le \dist(p,|\partial K|)/4 + \dist(p',|\partial K|)/4 \\
&\le 2\dist(p',|\partial K|)/3 \le 2\cdot 3^{-j'-1}.
\end{align*}
This is a contradiction. Thus $\Wedge_K(Q')\cap \Wedge_K(Q'') = \emptyset$.

To prove \eqref{item:Wedge-intersections-2}, let $Q,Q',Q''$ be three different cubes in $\cS_*(D)^{[n]}$ and let $j,j',j''$ be indices for which $Q\in \Refine^{j}(K)$, $Q'\in \Refine^{j'}(K)$, and $Q''\in \Refine^{j''}(K)$, respectively. We may assume that $j\le j'\le j''$. By condition \eqref{item:indentation-1} in the definition of indentation, $Q\cup Q'\cup Q''$ is the contained in one, or the union two adjacent $n$-cubes in $K$. Therefore the argument below is Euclidean.

Suppose  that $\Wedge_K(Q)\cap \Wedge_K(Q') \cap \Wedge_K(Q'')$ is non-empty. Then, by the first part of this lemma, $q=Q\cap Q''\ne \emptyset$ and $q'=Q'\cap Q''\ne \emptyset$ are opposite faces of $Q''$. 
Since also $Q\cap Q'\ne \emptyset$, this is impossible by Euclidean geometry.  Hence $\Wedge_K(Q)\cap \Wedge_K(Q') \cap \Wedge_K(Q'')=\emptyset$. This is a contradiction.
\end{proof}

Since connected components of an indentation are indentations, we have the following non-overlapping property of their wedges.

\begin{corollary}
\label{cor:Wedge-intersections}
Let $D\subset \Refine^k(K)$ be an indentation, and let $D_1$ and $D_2$ be two different connected components of $D$. Then
\[
\Wedge_K(D_1)\cap \Wedge_K(D_2) =\emptyset.
\] 
\end{corollary}

\subsection{Flattening indentations} 
Indentations can be flattened. The basic property of flattening reads as follows.
\begin{proposition}
\label{prop:first-flattening-theorem}
There exists a constant $L=L(n)\ge 1$ for the following. Let $D \subset \Refine^k(K)$ be an  indentation. Then there exists a piecewise linear $L$-bilipschitz homeomorphism 
\[
\phi_D \colon |\Refine^k(K)-D| \to |\Refine^k(K)-D^\ddagger|,
\]
which is the identity on $|K|\setminus \Wedge_K(\cS_{\delta(D)}(D))$.
\end{proposition}

 We emphasize that the bilipschitz constant $L$ in the theorem depends only on the dimension $n$, not on either the refinement scale or the structure of the indentation $D$.

The idea behind the proof of Proposition \ref{prop:first-flattening-theorem} is based on the following simple observation.

\begin{lemma}
\label{lemma:flattening-key-lemma}
Let $D \subset \Refine^k(K)$ be an indentation and $Q\in \cS_{\delta(D)}(D)^{[n]}$. Then $C=Q\cap (\Refine^k(K)-D)$ and $\partial Q-C$ are $(n-1)$-cells.
\end{lemma}
\begin{proof}
Let $q_0 = Q\cap |\partial K|$ and let $t\in \{0,1,2\}$ be the number of $n$-cubes in $\cS_*(D)^{[n]}$ meeting $Q$. We have three cases. 

If $t=0$ and $Q$ does not meet other $n$-cubes in $\cS_*(D)^{[n]}$, then $\partial Q - C = q_0$, $C = \partial Q - q_0$ and the claim holds. 

If $t=1$ and $Q$ meets exactly one $n$-cube $Q_1\in \cS_*(D)^{[n]}$, then $q_1=Q\cap Q_1$ is a face of $Q$ and $q_1$ meets $q_0$ in an edge. Thus $\partial Q-C = q_0\cup q_1$ and $C=\partial Q - (q_0\cup q_1)$ are $(n-1)$-cells. 

Finally, if $t=2$ and $Q$ meets exactly two $n$-cubes $Q_1$ and $Q_2$ of $\cS_*(D)^{[n]}$, then $q_1 = Q\cap Q_1$ and $q_2=Q \cap Q_2$ are opposite faces of $Q$ each of which meets $q_0$ in an edge. Thus $\partial Q - C = q_0\cup q_1\cup q_2$ and $C = \partial Q - (q_0\cup q_1 \cup q_2)$ are $(n-1)$-cells. 
\end{proof}

\begin{proof}[Proof of Proposition \ref{prop:first-flattening-theorem}]
Fix an (arbitrary) enumeration $Q_1,\ldots$, $ Q_s$ of the $n$-cubes in $\cS_{\delta(D)}(D)$ for the order in which the cubes to be flattened. For each $i=1,\ldots, s$, let $E_i= Q_1\cup \cdots \cup Q_i$.
We also set $E_0 = \emptyset$. Then each $D-E_i$ is an indentation of $\Refine^k(K)$ and $D-E_s = D^\ddagger$. 

Let $\beta(Q_i)$ be the collection of $n$-cubes in $E_{i-1}$ which are adjacent to $Q_i$ and   $\alpha(Q_i)$ be the collection of $n$-cubes  in $\cS_{\delta(D)}(D)-E_i$ adjacent to $Q_i$.
Since $Q_i$ meets $0$, $1$, or $2$ $n$-cubes in $\cS_{\delta(D)}(D)-E_i$,  there are at most two cubes in $\alpha(Q_i) \cup \beta(Q_i)$.

By Lemma \ref{lemma:flattening-key-lemma}, there exists $L'=L'(n)\ge 1$ and, for each $i=1,\ldots, s$, a piecewise linear $L'$-bilipschitz map $\psi_i \colon |(\Refine^k(K)-D)\cup E_{i-1}| \to |(\Refine^k(K)-D)\cup E_i|$ satisfying the following conditions:
\begin{enumerate}
\item $\psi_i$ is an identity in the complement of $\Wedge_K(Q_i)$ and 
\item  $\psi_i \left(\Wedge_K(Q')\right) \, \cap\, \Wedge_K(Q'')=\emptyset$, for $Q'\in \beta(Q_i)$ and $Q''\in \alpha(Q_i)$.\label{item:wedge-image}
\end{enumerate}
Due to these conditions and also Lemma \ref{lemma:Wedge-intersections}, we have that for each $x\in|K|$, 
\[\psi_i(\psi_{i-1}\ldots (\psi_1(x)))\ne \psi_{i-1}(\psi_{i-2} \ldots( \psi_1(x)))\]
for at most two indices $i\in \{2,\ldots, s\}$. 

Thus the composition
\[
\psi_D = \psi_s \circ \cdots \circ \psi_1 \colon |\Refine^k(K)-D| \to |\Refine^k(K)-D^\ddagger|
\]
is an $(L')^2$-bilipschitz homeomorphism, which is an identity in the complement of $\Wedge_K(\cS_{\delta(D)}(D))$.
\end{proof}

For an indentation $D$,  its spectrum reduction $D^\ddagger$ is also an indentation. An iterative application of Proposition \ref{prop:first-flattening-theorem} yields a result on flattening indentation.

\begin{corollary}
\label{cor:first-flattening-theorem}
There exists $L=L(n)\ge 1$ for the following. Let $k\geq 2$ and $D\subset \Refine^k(K)$ be an indentation. Then there exists a piecewise linear $L$-bilipschitz homeomorphism
\[
\phi_D \colon |\Refine^k(K)-D|\to |\Refine^k(K)|,
\]
which is an identity on $|K| \setminus \Wedge_K(D)$.
\end{corollary}

\begin{proof}
It suffices to observe that, for each $0\le j \le k$, 
\[
\Wedge_K(\cS_j(D)) = \bigcup_{Q\in \cS_j(D)^{[n]}} \Wedge_K(Q).
\]
Thus, for any three different indices $j,j',j''$, 
\[
\Wedge_K(\cS_j(D)) \cap \Wedge_K(\cS_{j'}(D)) \cap \Wedge_K(\cS_{j''}(D)) = \emptyset.
\]
by Lemma \ref{lemma:Wedge-intersections}. An iterative application of   Proposition \ref{prop:first-flattening-theorem} and its proof yields a  composition of bilipschitz homeomorphisms
\begin{align*}
|\Refine^k(K)-D_0| \to |\Refine^k(K)-D_1| \to \cdots \to |\Refine^k(K) - D_t| \to |\Refine^k(K)|
\end{align*}
which is identity on $|\Refine^k(K)|\setminus \Wedge_k(D)$. Adapting condition \eqref{item:wedge-image}  in the proof of Lemma \ref{lemma:flattening-key-lemma}  to the spectra $D_1, \ldots, D_t$, we may construct individual $L(n)$-bilipschitz homeomorphisms for which their composition is $L^2(n)$-bilipschitz.
Here $D_0 = D$, $D_i = (D_{i-1})^\dagger$, and $t\in [1,k]$ is the number of indices $j$ for which $\cS_j(D)\ne \emptyset$. 
\end{proof}

\subsection{Partial star indentations}
\label{sec:partial-star-indentations}

We begin with a definition of a partial star.

\begin{definition}
Let $K$ be a cubical $n$-complex, and $e\in (\Refine^j(K))^{[n-2]}$. 
We call an $n$-subcomplex $S\subsetneq \Star_{\Refine^j(K)}(e)$  a \emph{partial star} if $S$ is an adjacently connected. We denote $\mu_e(S)$ the number of $n$-cubes in $S$. 
\end{definition}

Partial-star indentations are defined as follows. For an illustration, see Figure \ref{fig:bent-indentation}; type \eqref{item:boundary-star-indentation} on the right and type \eqref{item:inner-star-indentation} on the left.

\begin{definition}
\label{def:partial-star-indentation}\index{indentation!partial-star} 
Let $K$ be a cubical $n$-complex with boundary. A cubical subcomplex $V=\Refine^{k-j}(S) \subset \Refine^k(K)$, where $2\leq j\leq k$, is a \emph{partial-star-indentation of an edge $e\in \Refine^j(\partial K)$ in $K$} if either
\begin{enumerate}
\item $S \subset \Star_{\Refine^j(K)}(e)$ is a partial star for which  $S\cap \Refine^j(\partial K)$ contains precisely the two $(n-1)$-cubes in $\partial S$ containing $e$, or \label{item:boundary-star-indentation}
\item $S=\Star_{\Refine^j(K)}(e^\op)$, where $e^\op$ is the edge opposite to $e$ in an $n$-cube $Q\in \Refine^j(K)$. \label{item:inner-star-indentation}
\end{enumerate} 

For type \ref{item:inner-star-indentation}, we call the unique $n$-cube $Q$ in $S$, which does not meet $e$, a \emph{spare cube}. \index{spare cube}
\end{definition}

\begin{remark}
In type \eqref{item:inner-star-indentation}, the partial-star-indention could be also called as \emph{star-indentation}, but we do not make this distinction in what follows as these cases are discussed mostly simultaneously. Note that, formally, a partial-star-indentation is not an indentation in the sense of Definition \ref{def:indentation}. For this reason, we introduce the notion of bent indentations in the next section.
\end{remark}

\begin{remark}
In type \eqref{item:boundary-star-indentation}, the $n$-cubes of the partial star $S$ belong to different $n$-cubes of $K$, whereas, in type \eqref{item:inner-star-indentation}, the $n$-cubes of $S$ belong to the same $n$-cube of $K$. Note also that the number of $n$-cubes of $S$ in type \eqref{item:boundary-star-indentation} is between $1$ and $\mu(K)$. In type  \eqref{item:inner-star-indentation} we always have $\mu_e(S)=4$. In this case, $\Star_{\Refine^j(K)}(e)$ consists of only one $n$-cube, the $n$-cube $Q$ in the statement.
\end{remark}

\begin{remark}
\label{rmk:face-counting}
There is a distinction between type \eqref{item:boundary-star-indentation} and \eqref{item:inner-star-indentation} which we describe as follows.

In type \eqref{item:boundary-star-indentation} case, for each $(n-1)$-cube $q\in S\cap \partial \Refine^j(K)$, the opposite face $q^\op$, with respect to a unique cube in $S$, is also in $\partial S$. Also the mapping $q\mapsto q^\op$ is injective. In type \eqref{item:inner-star-indentation} case, the intersection $S\cap \partial \Refine^j(K)$ consists of four $(n-1)$-cubes $q$ and none of the opposite faces $q^\op$ are on the boundary of $S$, but the intersection $(S-Q)\cap Q$, where $S$ is the spare cube, consists of two of the opposite faces $q^\op$.

\end{remark}

\subsection{Bent indentation}
\label{sec:bent-indentations}

In the case of good complexes, we may generalize the notion of indentation to bent indentation.  This section may be read along with Section \ref{sec:indentation-lift}. Heuristically, a bent indentation is an indentation whose component are joined together by partial-star-indentations; see Figure \ref{fig:bent-indentation} for an illustration.

\begin{definition}
\label{def:bent-indentation}
\index{indentation!bent}
Let $K$ be a good cubical $n$-complex with boundary, and $k\geq 2$. A cubical complex $B= D\cup E \subset \Refine^k(K)$ is a \emph{bent indentation in $K$}, if 
$D$ is an indentation (possibly empty)  and $E$ is a collection of mutually disjoint partial-star-indentations in $K$, for which
 $D$ and $E$ have no common $n$-cubes and  for  each partial-star-indentation $V=\Refine^{k-j}(S_V)$ in $E$, the partial star $S_V \subset \Refine^{j}(K)$ has the properties:
\begin{enumerate}
\item $S_V$ is adjacent to at most two (possibly none) $n$-cubes in separate components of $\cS_j(D)$, 
\item any common face of $S_V$ and  $ \cS_j(D)$ does not intersect $e$, and $S_V$ does not meet $\cS_{j'}(D)$ for $j'\ne j$.
\end{enumerate}
\end{definition}

\begin{figure}[htp]
\centering
\begin{overpic}[scale=.65,unit=1mm]{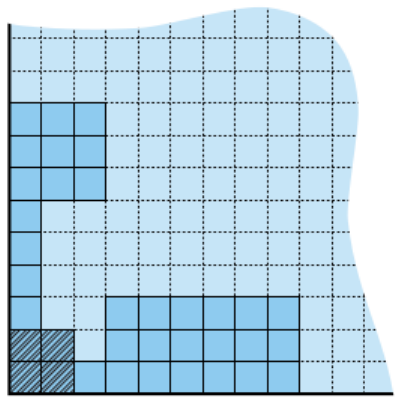} 
\put(5,5){\tiny {\color{white} $E$}}
\put(1.5,23){\tiny $D$}
\put(19,5){\tiny $D$}
\put(-.5,-.5){\tiny$e$}
\end{overpic}
\hfill
\begin{overpic}[scale=.23,unit=1mm]{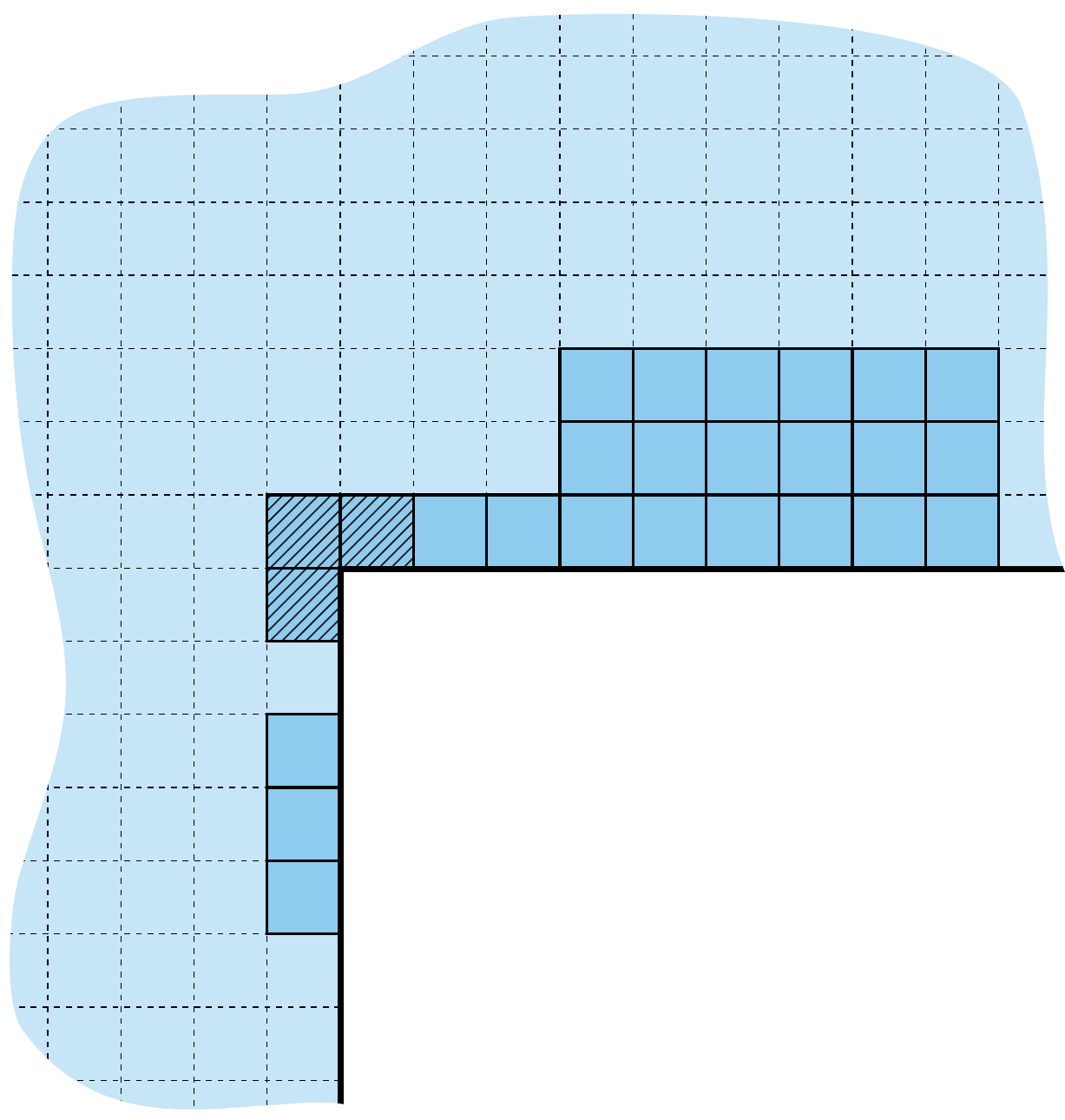} 
\put(12,12){\tiny $D$}
\put(29,29){\tiny $D$}
\put(12,25.5){\tiny {\color{white} $E$}}
\put(16,23){\tiny$e$}
\end{overpic}
\caption{Bent-indentations $D\cup E$ in $\Refine^k(K)$ as in Definition \ref{def:bent-indentation}; indentations $D$ in medium blue, and partial-star-indentations $E$ in shaded blue. On the right, $E=S \subset \Star_{\Refine^k(K)}(e)$ is of type (1). On the left, $E=S=\Star_{\Refine^k(K)}(e^\op)$ is of type (2).}
\label{fig:bent-indentation}
\end{figure}

Bent indentations admit a flattening similar to indentations. The bilipschitz constant of the flattening depends, however, on the local multiplicity $\mu(K)$ of $K$.

\begin{proposition}[Bent indentation flattening]
\label{prop:flattening-bent-indentation}
Let $K$ be a good cubical complex. Then there exists $ L = L(n,\mu(K))\ge 1$ for the following. For each bent indentation $B\subset \Refine^{k}(K)$,  there exists a piecewise linear $ L$-bilipschitz homeomorphism
\[
\phi_B \colon |\Refine^k(K)-B| \to |\Refine^k(K)|,
\]
which is an identity on $|\Refine^k(K)|\setminus \Wedge_K(B)$. 
\end{proposition}

\begin{proof}
Let $B = D\cup E$ be a union of $n$-complexes as in Definition \ref{def:bent-indentation}. Let $V$ be a component of $E$, and $S_V \subset \Refine^{j}(K)$ be the partial-star in $K$ for which  $V=\Refine^{k-j}(S_V)$. By definition,  $|V| \cap |\partial K|=|S_V \cap \Refine^j(\partial K)|$ is an $(n-1)$-cell.

Suppose that $q$ is the common face of a cube in ${S_V}^{[n]}$ and a cube in $\cS_j(D)^{[n]}$. Then $|\partial K|$ meets $q$ in an $(n-2)$-dimensional  face $\tau_q$ of $q$. Thus  $\tau_q \subset |V| \cap |\partial K|$, and $q \cup (|V| \cap |\partial K|)$  is an $(n-1)$-cell.

Denote now by $C_V$ the union of $|V| \cap |\partial K|$ and the (at most two) faces of the $n$-cubes in $S_V$ which are in common with some $n$-cubes in  $\cS_j(D)^{[n]}$. Let $C'_V = \cl(|\partial V|\setminus C_V )$. Then both $C_V$ and $C'_V$ are $(n-1)$-cells.

For distinct components $V$ and $V'$ of $E$, their intersection $V\cap V'=\emptyset$. Thus by the proof of Lemma \ref{lemma:Wedge-intersections}, the intersection of wedges $\Wedge_K(V)\cap \Wedge_K(V')=\emptyset$. Therefore there exists a piecewise linear $L(n,\mu(K))$-bilipschitz homeomorphism 
\[
\psi_E \colon |\Refine^k(K)-B|\to |\Refine^k(K)-D|,
\]
which is the identity outside $\Wedge_K(E) \cap |\Refine^k(K)-B|$.

Since $D$ is an indentation in $\Refine^k(K)$, we may  apply Proposition \ref{prop:first-flattening-theorem} to obtain an $L(n)$-bilipschitz homeomorphism 
\[\psi_D \colon |\Refine^k(K)-D|\to |\Refine^k(K)|,\]
which is the identity outside $\Wedge_K(D) \cap |\Refine^k(K)-D|$.

The claim follows now with the composition $ \phi_B = \psi_D \circ \psi_E$.
\end{proof}

In what follows, we need also the mapping in Proposition \ref{prop:flattening-bent-indentation} with an additional  local rigidity requirement. Since the proof is only a slight modification of that for Proposition \ref{prop:flattening-bent-indentation}, we omit the details.

\begin{corollary}
\label{cor:flattening-bent-indentation}
Let $K$ be a good cubical complex. Then there exists $\widetilde L = \widetilde L(n,\mu(K))\ge 1$ for the following. Let $B=D\cup E\subset \Refine^{k}(K)$ be a bent indentation, where $D$ is an indentation and $E$ is a collection of partial-star-indentations. Let $j\geq 2$ and $q\in \Refine^{k+j}(\partial D)^{[n-1]}$ be an $(n-1)$-cube which does not meet $\partial K$ or $E$. Then there exists a piecewise linear $\widetilde L$-bilipschitz homeomorphism
\[
\widetilde \phi_B \colon |\Refine^k(K)-B| \to |\Refine^k(K)|,
\]
which is an identity on $|\Refine^k(K)|\setminus \Wedge_K(B)$ and  which is an isometry on the unique $n$-cube $Q\in \Refine^j(\Refine^{k}(K)-B)$ having $q$ as a face.
\end{corollary}

\section{Reservoir-canal systems}
\label{sec:reservoir-canal-system} 

In this section, we formulate a concrete version of Theorem \ref{theorem:RC-vague} to be applied in the following chapters. The following proposition serves as a blueprint for the construction steps ahead and is  best to be read in parallel with the construction.

For the statement, we define the notion of a preference function.
\begin{definition}
\label{def:preference}
Let $U$ be a good cubical $n$-complex and $Y\subset U$ an $(n-1)$-subcomplex. A function $\rho \colon Y^{[n-1]}\to U^{[n]}$ is a \emph{preference function on $Y$ for the pair $(U,Y)$} if, for each $q\in Y^{[n-1]}$,  $\rho(q)$ is an $n$-cube having $q$ as its face. \index{preference function}
\end{definition}
Note that we neither assume a preference function $\rho$ to be injective nor that the $n$-cubes $\rho(q)$ and $\rho(q')$ to be adjacent in $U$ for adjacent $(n-1)$-cubes $q$ and $q'$ in $Y$.

\begin{proposition}
\label{prop:RC-existence}
Let $U$ be a good cubical $n$-complex, let $Y\subset U$ be an adjacently connected $(n-1)$-subcomplex, and let $\rho \colon Y^{[n-1]} \to U^{[n]}$ be a preference function. Then, given $m\ge 1$ and $\nu\in \N$ satisfying $3^\nu\ge 3^{10}m \mu(U)^2$, there exist an adjacently connected $n$-subcomplex $\sfRC=\sfRC_{\Refine^\nu(U)}(Y)$ of $\Refine^\nu(U)$
and an $(n-1)$-subcomplex $\Tr = \Tr_{\Refine^\nu(U)}(Y)$ of $\Refine^\nu(Y) \cup \sfRC^{(n-1)}$ having the following properties:
\begin{enumerate}
\item if a spectral cube $Q$ in $\sfRC$ meets an $(n-1)$-cube $q\in Y^{[n-1]}$ in a face, then $Q\subset \rho(q)$, \label{item:RC-existence-2}
\item $\sfRC\cap \partial Y = \emptyset$ and $\Refine^\nu(Y) \cup \partial \sfRC \subset \Tr$, \label{item:RC-existence-3}
\item graph $\Gamma(\sfRC; \Tr \cap \sfRC)$ has $m$ components $\tau_1,\ldots, \tau_m$, all of which are trees of length at most $\lambda=\lambda(\nu,\# Y^{[n-1]})$, \label{item:RC-existence-4}
\item for each $q\in Y^{[n-1]}$ and $i=1,\ldots, m$, there exits an $n$-cube in $\Span_{\Refine^\nu(U)}(\tau_i)$ having a face in $q$, \label{item:RC-existence-5}
\item the $n$-complex $\widetilde{\sfRC} \subset \Real_{\Refine^\nu(U)}(Y)$ satisfying $\pi_{(U,Y)}(\widetilde{\sfRC}) = \sfRC$ has the property that the intersection of $\widetilde{\sfRC}$ with any component of $\Real_{\Refine^\nu(U)}(Y)$ is a bent indentation of that component. \label{item:RC-existence-6}
\end{enumerate}
\end{proposition}

Having this result at our disposal, Theorem \ref{theorem:RC-vague} follows from Proposition \ref{prop:RC-existence} and Proposition \ref{prop:flattening-bent-indentation}. 

The notation $\sfRC$ stems from the terminology that $\sfRC$ consists of 'reservoirs' $\mathsf R$ (corresponding to spectral cubes in $\cS_2(\sfRC)$) and 'canals' $\mathsf C$ (corresponding to components of $\cS_\nu(\sfRC)$) connecting reservoirs. The name for complex $\Tr$ refers to terminology that $\Tr$ is a 'transformation' of $Y$. The preference function $\rho$ indicates the 'side' of $q\in Y^{[n-1]}$ in which we locate the complex.

In this section, we assume -- for the constructions of $\sfRC$ and $\Tr$ -- that we have made the following choices, which we record as standing assumptions.

\begin{standing}
\label{standing:rho-T}\index{Standing assumptions on $(U;Y, \rho,\cT)$}
Throughout Section \ref{sec:reservoir-canal-system}, we assume that we have fixed 
\begin{itemize}
\item a good cubical $n$-complex $U$ having a flat structure $\sF_U$,
\item an $(n-1)$-subcomplex $Y$ of $U$,
\item a preference function $\rho=\rho_Y\colon Y^{[n-1]} \to U^{[n]}$, 
\item a rooted spanning tree $\cT=\cT_Y$ of $\Gamma(Y)$ and a partial order $<_Y$ in $\cT$ having the root as the unique minimum,
\item the local multiplicity $\mu=\mu(U)$ of $U$, and
\item integers $m, \nu\geq 1$ satisfying $3^\nu \ge 3^{10}m \mu^2$.
\end{itemize}
\end{standing}

We also denote the connected components of $U$ separated by $Y$ as follows. Let $G_1,\ldots, G_r$ be the connected components of $\Gamma(U;Y)$ and, for $j=1,\ldots, r$, let
\[
U_j = \Span_U(G_j).
\]

\begin{remark}
We emphasize that the number $r$ of  connected components of $\Gamma(U;Y)$ depends on the data $(U, Y)$; it has no \emph{a priori} relation with the integers $m$, $\mu$, or $\nu$. The structures constructed below are indexed by $\{1, \ldots, m\}$.
\end{remark}

We begin now the construction of the complexes $\sfRC$ and $\Tr$ by introducing first auxiliary concepts of gates, pre-reservoir-blocks, canal sections, and connectors. The reservoirs and reservoir-canal systems are then finally defined in Section \ref{sec:sub-reservoir-canal-system} and transformations in Section \ref{sec:reservoir-canal-transformation}. The properties of $\sfRC$ and $\Tr$ stated in Proposition \ref{prop:RC-existence} are proved in the course of these sections and summarized in Section \ref{sec:RC-existence-summary}.

\subsection{Gates}
\label{sec:gates}\index{reservoir-canal system!gates}
Let $U$ be a good cubical $n$-complex and assume the Standing assumption \ref{standing:rho-T}. Let  $ \nu(U)=\nu$ be the refinement scale of $U$ from now on.

We fix first model gates in the unit cube $[0,1]^{n-2}$. Since $3^\nu \ge 3^{10}m \mu^2$, we may fix a family 
\[
\mathbf E=\{\zeta_1,\ldots, \zeta_m; \, \zeta_{m+1},\ldots,\zeta_{2m};\, \ldots; \,\zeta_{m (\mu^2-1)+1}\ldots, \zeta_{m\mu^2}\}
\]
of mutually disjoint $(n-2)$-cubes contained in $\Refine^{\nu-2}\left([\frac{4}{9},\frac{5}{9}]^{n-2}\times \{0\}\times \{0\}\right)$. Note that $[\frac{4}{9},\frac{5}{9}]^{n-2}\in \Refine^2([0,1]^{n-2})$, and hence 
\[
\Refine^{\nu-2}\left(\left[\frac{4}{9},\frac{5}{9}\right]^{n-2}\times \{0\}\times \{0\}\right) \subset \Refine^\nu([0,1]^{n-2}\times \{0\}\times \{0\})).
\]
The elements in $\mathbf E$ are called \emph{model gates}.

To define gates for the complex $U$, we use the flat structure $\sF_U$ of $U$ as follows. For each $(n-2)$-cube $e  \in U^{[n-2]}$, we fix an $n$-cube $Q\in U^{[n]}$ that contains $e$. Let $\phi_Q \colon Q\to [0,1]^n$ be the map associated to $Q$ in the flat structure $\sF_U$. By applying an isometry of $[0,1]^n$ if necessary, we may assume that $\phi_Q(e) = [0,1]^{n-2}\times \{0\}\times \{0\}$. We call  each 
\[
\zeta_{e,i}=  \phi_Q^*(\zeta_i),
\]
for $i \in \{1,\ldots, m \mu\}$, a \emph{gate on $e$}, and let 
\[
\sfE(e)= \{\zeta_{e,1},\ldots, \zeta_{e, m}; \, \zeta_{e, m+1},\ldots,\zeta_{e, 2m};\, \ldots; \,\zeta_{e,(\mu^2-1)m+1}\ldots, \zeta_{e,m\mu^2}\}.
\]

Note that gates in $\sfE(e)$ are contained in $\Refine^{\nu-2}(c^2(e))\subset \Refine^\nu(e)$, where $c^2(e) \in \Refine^2(e)$ is the center cube of the center cube $c(e)$ of $e$. Here, we have suppressed the dependency on the flat structure in the notations.

We denote 
\[
\sfE(U) = \bigcup_{e\in U^{[n-2]}} \sfE(e) \subset \Refine^\nu(U)
\]
the collection of \emph{all gates of $U$}. We emphasize that, despite the notation, the gates of $U$ are not elements of $U$ but elements of $\Refine^\nu(U)$.

Since the model gates are mutually disjoint $(n-2)$-cubes in the refinement $\Refine^\nu([0,1]^{n-2}\times \{0\}\times \{0\})$, the stars of gates in $\Refine^\nu(U)$ are mutually disjoint in the following sense. We omit a trivial argument.

\begin{lemma}
Let $\zeta$ and $\zeta'$ be gates in $U$, $\zeta\ne \zeta'$. Then 
\[
\Star_{\Refine^\nu(U)}(\zeta) \cap \Star_{\Refine^\nu(U)}(\zeta') = \emptyset.
\]
\end{lemma}

We distribute now, for each $(n-2)$-cube $e \in U^{[n-2]}$, the gates in $\sfE(e)$ among all pairs $\{q,q'\}$ of $(n-1)$-cubes in $\Star_U(e)^{[n-1]}$; note that for such pair, we have $q\cap q'=e$. For this, we fix for each $e\in U^{[n-2]}$, an injective labeling 
\[
\kappa_e \colon \{ \{q,q'\} \colon q, q' \in \Star_U(e)^{[n-1]}, q\cap q'=e \} \to \{1,\ldots, \mu^2\},
\]
and denote, for each $i\in \{1,\ldots, m\}$,
\[
\zeta_{e,\{q,q'\},i} = \zeta_{e,(\kappa_e(\{q,q'\})-1)m+i}\in \sfE(e).
\]
Let 
\[
\kappa(U) = \{\kappa_e \colon e\in U^{[n-2]}\}
\]
be the entire collection of labels associated to $U^{[n-2]}$.

\subsection{Pre-reservoir-blocks on $Y$}
\label{sec:pre-reservoir-block}

We begin by fixing a family of model reservoirs $\mathbf{R}_1, \ldots, \mathbf{R}_m$ in the unit $n$-cube $[0,1]^n$. For $i=1,\ldots, m$, let
\[
D_i= \left[\frac{4}{9}+\frac{i-1}{3^\nu}, \frac{5}{9}-\frac{i-1}{3^\nu}\right]^{n-1}   \times \left[0,\frac{1}{9}-\frac{i-1}{3^\nu}\right] \subset [0,1]^n
\]
be an $n$-dimensional rectangle in $[0,1]^n$. Rectangles $D_i$ are cubical in the sense that both $\Refine^\nu([0,1]^n)|_{D_i}$ and $\Refine^\nu([0,1]^n)|_{\cl(D_i\setminus D_{i+1})}$ are well-defined complexes for each $i$.  We set $\sfR_i\subset \Refine^\nu([0,1]^n)$, $i=1,\ldots, m$, to be the subcomplex, whose space is $\cl(D_i\setminus D_{i+1})$, that is,  
\[
\mathbf R_i = \Span_K(\{Q\in \Refine^\nu(U)^{[n]}\colon Q\subset \cl(D_i\setminus D_{i-1})\}),
\]
and $D_0 = \emptyset$.  Note that each $|\mathbf R_i|$ is an $n$-cell, and $\Gamma(\sfR_i)$ is connected. We call the cubical complexes $\sfR_i$ \emph{pre-reservoir-blocks}.

To construct pre-reservoir-blocks on $Y$, let $q\in  Y^{[n-1]}$ and  $Q = \rho(q)$. By applying an additional isometry of $[0,1]^n$ if necessary, we may assume that the map $\phi_Q \colon Q\to [0,1]^n$  associated to $Q$ in the flat structure $\sF_U$ has the property that $\phi_Q(q) = [0,1]^{n-1}\times \{0\}$. By the definition of the refinements, we have that $\Refine^\nu(q) = \phi_Q^*(\Refine^\nu([0,1]^n)$. Thus the subcomplexes 
\[
\sfR_{q,i} = \phi_Q^*(\mathbf R_i) \subset \Refine^\nu(Q)
\]
are well-defined. We call subcomplexes $\sfR_{q,1},\ldots, \sfR_{q,m}$ \emph{pre-reservoir-blocks adjacent to $q$}.   The name stems from the fact that \index{reservoir-canal-system!pre-reservoir block}\index{$\sfR_{q,i}$}
\[
\sfR_{q,i} \cap \Refine^\nu(q) = \phi_Q^*(\sfR_i \cap \Refine^\nu([0,1]^{n-1}\times \{0\}))
\]
is an adjacently connected cubical complex contained in $|q|$.

We call the union 
\[
\sfR_q =\bigcup_{i=1}^m \sfR_{q,i}
\]
a \emph{pre-reservoir-cube over $q$}, and observe that it is a refinement of a cube in $\Refine^n(Q)$.
\index{$\sfR_q$}

\subsection{Canal sections}
\label{sec:canal-section}

We use now the spanning tree $\cT$ to define, for each $q\in Y^{[n-1]}$, canal sections on $q$, which connect pre-reservoir-blocks to the stars of the gates. See Figure \ref{fig:Reservoir_Canal_Single} for an illustration.  

\begin{figure}[h!]
\begin{overpic}[scale=.56,unit=1mm]{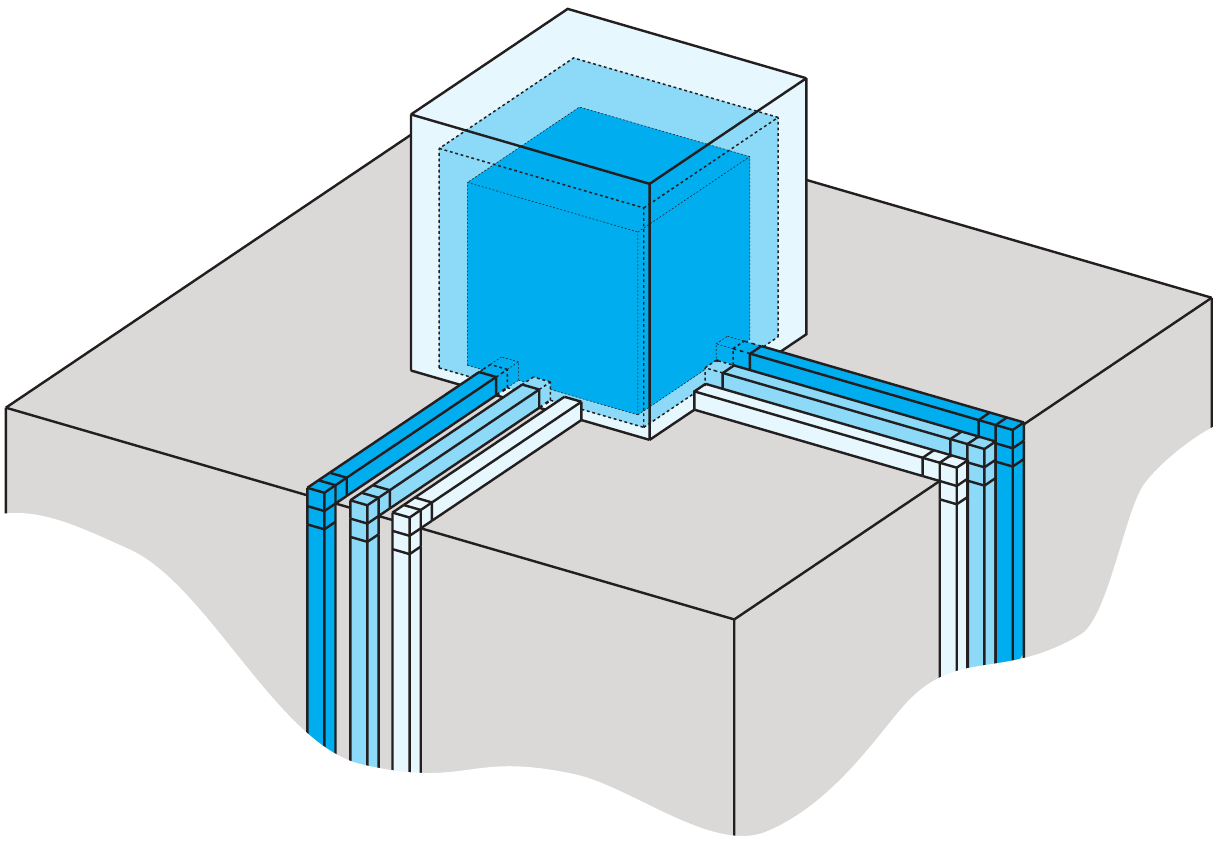}
\put(50,75){\tiny$\sfR_{q,1}$}
\put(50,70){\tiny$\sfR_{q,2}$}
\put(21,47){\tiny$q$}
\put(11,33){\tiny$q'$}
\put(80,20){\tiny$q''$}
\put(53,55){\tiny$\sfR_{q,3}$}
\put(38,35){\tiny$\sfC_{\{q,q'\},2}$}
\put(28,41){\tiny$\sfC_{\{q,q'\},3}$}
\put(40,24){\tiny$\sfC_{\{q,q'\},1}$}
\end{overpic}
\caption{Pre-reservoir-blocks and canal stretches over an $(n-1)$-cube $q\in  Y$ (not in scale).}
\label{fig:Reservoir_Canal_Single}
\end{figure}

Let $q\in Y^{[n-1]}$. For each edge $\{q,q'\}$ in $\cT$ and each index $i\in \{1,\ldots, m\}$, let $\sfC_{q,\{q,q'\},i}$ be the unique minimal $n$-subcomplex of $\Refine^\nu(\rho(q))$ having the following properties: \index{reservoir-canal system!canal section} \index{$\sfC_{q,\{q,q'\},i}$}
\begin{enumerate}
\item $\sfC_{q,\{q,q'\},i} \cap \sfR_{q,i}$ is an $(n-1)$-cube and 
\item $\sfC_{q,\{q,q'\},i}\cap \Star_{\Refine^\nu(U)}(\zeta_{q\cap q', \{q,q'\},i})$ is an $(n-1)$-cube. 
\end{enumerate}
We call $\sfC_{q,\{q,q'\}, i}$ \emph{a canal section connecting the pre-reservoir-block $\sfR_{q,i}$ to the star of gate $\zeta_{q\cap q', \{q,q'\},i}$}.
Let 
\[
\sfC_{q,i} = \bigcup_{\{q,q'\}\in \cT} \sfC_{q,\{q,q'\},i}.
\]

The following two observations follow immediately from the minimality and the fact that the complexes $\sfC_{q,i}$ do not meet the boundary of $q$.

\begin{lemma}
The complex $\sfC_{q,i} \subset \Star_{\Refine^\nu(U)}(\Refine^\nu(q))$. In fact, for each $\widehat Q\in (\sfC_{q,\{q,q'\},i})^{[n]}$, the intersection $\widehat Q \cap \Refine^\nu(q)$ is an $(n-1)$-cube.
\end{lemma}

\begin{lemma}
Complexes $\sfC_{q,i}$ and $\sfC_{q',i'}$ do not meet unless $q=q'$ and $i= i'$. Moreover, for $q\in Y^{[n-1]}$ and $i\in \{1,\ldots, m\}$, subcomplexes $\sfC_{q,\{q,q'\},i}$, $\{q,q'\}\in \cT$, are connected components of $\sfC_{q,i}$.
\end{lemma}

Due to the nested nature of the pre-reservoir-blocks, for $i'< i$, each $\sfC_{q,\{q,q'\},i}$ meets the pre-reservoir-blocks $\sfR_{q,i'}$ in an $n$-cube. 

\begin{lemma}
\label{lemma:canal-reservoir-intersection}
Let $q\in Y^{[n-1]}$, $\{q,q'\} \in \cT$, and $i,i'\in \{1,\ldots,m\}$. Then $\sfC_{q,\{q,q'\},i}$ meets $\sfR_{q,i'}$ in an $n$-cell if and only if  $1\leq i'< i\leq m$.
\end{lemma}

\begin{lemma}
For each $q\in Y^{[n-1]}$ and $i\in \{1,\ldots, m\}$, the cubical complex $\sfR_{q,i} \cup \sfC_{q,i}$ is adjacently connected.
\end{lemma}

\begin{proof}
For each $\{q,q'\}\in \cT$, the complex $\sfC_{q,\{q,q'\},i}$ meets the pre-reservoir-block $\sfR_{q,i}$ in an $(n-1)$-cube. Since each $\sfC_{q,\{q,q'\},i}$, and $\sfR_{q,i}$, are adjacently connected, their union is adjacently connected.
\end{proof}

\subsection{Connectors}\label{sec:connector}

We define now, for each edge $\{q,q'\}$ in $\cT$ a subcomplex $\sfJ_{\{q,q'\},i}$, which connects complexes $\sfR_{q,i}\cup \sfC_{q,i}$ and $\sfR_{q',i}\cup \sfC_{q',i}$. We have two cases, analogous to partial-star-indentations (Definition \ref{def:partial-star-indentation}), depending on the choice of preference cubes $\rho(q)$ and $\rho(q')$.

\begin{definition}
\label{def:connector} \index{reservoir-canal system!connector} \index{$\sfJ_{\{q,q'\},i}$}

Let $\{q,q'\}\in \cT$ be an edge and  $i\in \{1,\ldots, m\}$ be an index. We define \emph{a connector $\sfJ_{\{q,q'\},i}\subset \Refine^\nu(K)$  connecting $\sfR_{q,i}\cup \sfC_{q,i}$ and $\sfR_{q',i}\cup \sfC_{q',i}$ over gate $\zeta_{q\cap q', \{q,q'\},i}$} in two separate cases. 
\begin{itemize}
\item [Case 1.]  
For $\rho(q)\neq \rho(q')$, we set $\sfJ_{\{q,q'\},i}$ to be a minimal adjacently connected $n$-subcomplex of  $\Star_{\Refine^{\nu}(U)}(\zeta_{q\cap q', \{q,q'\},i})$ for which 
\[
\Star_{\Refine^\nu(U)}(\zeta_{q\cap q', \{q,q'\},i})\cap \left( \Refine^\nu(\rho(q)) \cup \Refine^\nu(\rho(q'))\right) \subset \sfJ_{\{q,q'\},i}.
\]
\item[Case 2.] For $\rho(q) = \rho(q')$, we set $\sfJ_{\{q,q'\},i} = \Star_{\Refine^\nu(U)}(\zeta^\op_{q\cap q', \{q,q'\},i})$, where $\zeta^\op_{q\cap q', \{q,q'\},i}$ is the edge opposite to $\zeta_{q\cap q', \{q,q'\},i}$ in the unique $n$-cube of $\Refine^\nu(\rho(q))$ containing $\zeta_{q\cap q', \{q,q'\},i}$.
\end{itemize}
\end{definition}

For a connector in case 2, let $\sfQ_{\{q,q'\},i}$ be the unique $n$-cube in $\sfJ_{\{q,q'\},i}$, which does not meet $\partial \rho(q)$; we call $\sfQ_{\{q,q'\},i}$ the \emph{spare cube in connector}. 
\index{reservoir-canal system!spare cube in connector} \index{$\sfQ_{\{q,q'\},i}$}

We make three remarks.

\begin{remark}
The connectors in Cases 1 and 2 are partial-star-indentations of types 1 and 2 at $\zeta_{q\cap q', \{q,q'\},i}$ in $K$, respectively.
\end{remark}

\begin{remark}[Adjustment of canal sections when $\rho(q)=\rho(q')$]
\label{rmk:canal-section-adjust}
If $\rho(q)=\rho(q')$ the connector $\sfJ_{\{q,q'\},i}$ meets  each of the two canal sections, $\sfC_{q,i}$ and $\sfC_{q',i}$, in an $n$-cube. To have an essentially disjoint partition, we replace
$\sfC_{q,i}$ by the subcomplex $\sfC_{q,i}-(\sfJ_{\{q,q'\},i}\cap \sfC_{q,i})$ and $\sfC_{q',i}$ by  $\sfC_{q',i}-(\sfJ_{\{q,q'\},i}\cap \sfC_{q',i})$. For convenience, the notations for these canal sections remain.
\end{remark}

\begin{remark}
If $\rho(q)\ne \rho(q')$, the minimality in the definition ensures that any proper subcomplex of a connector is not a connector.

The intersection $\Star_{\Refine^\nu(U)}(\zeta_{q\cap q', \{q,q'\},i})\cap \left( \Refine^\nu(\rho(q)) \cup \Refine^\nu(\rho(q'))\right)$ has two $n$-cubes. Since  star $\Star_{\Refine^{\nu}(U)}(\zeta_{q\cap q', \{q,q'\},i})$ is cyclic,  we may always fix a connector $\sfJ_{\{q,q'\},i}$ among the at most two possible choices. 
\end{remark}

Two canal sections are joined by a connector into a canal stretch.
\begin{definition}\index{$\sfC_{\{q,q'\},i}$} \index{reservoir-canal system!canal stretch}
For each edge $\{q,q'\}$ in $\cT$ and $i\in \{1,\ldots, m\}$, the complex 
\[
\sfC_{\{q,q'\},i} = \sfC_{q,\{q,q'\},i} \cup \sfJ_{\{q,q'\},i} \cup \sfC_{q',\{q,q'\},i}
\]
is called a \emph{canal stretch of index $i$ over the edge $\{q,q'\}$}.
\end{definition}

\begin{lemma}
\label{lemma:adjacency-canal-stretch}
For each edge $\{q,q'\}\in \cT$ and $i\in \{1,\ldots, m\}$, the canal stretch $\sfC_{\{q,q'\},i}$ and its boundary $\partial \sfC_{\{q,q'\},i}$ are adjacently connected.
\end{lemma}

\begin{proof}

Since the complex $\sfJ_{\{q,q'\},i}$ meets each of the  two complexes $\sfC_{q,\{q,q'\},i}$ and $\sfC_{q,\{q,q'\},i}$ in a face and all three complexes are adjacently connected,  their union
\[
\sfC_{q,\{q,q'\},i} \cup \sfJ_{\{q,q'\},i}  \cup \sfC_{q',\{q,q'\}, i}
\]
is adjacently connected. Thus the first claim follows.

The second claim, i.e.~the connectedness of the graph 
\[
\Gamma( \partial \sfC_{\{q,q'\},i}),
\]
follows from the first claim and the fact the each complex in the union has adjacently connected boundary.
\end{proof}

Before proceeding, we review some notations: canal section $\sfC_{q,\{q,q'\},i}$ connects the pre-reservoir-block $\sfR_{q,i}$ to the star of gate $\zeta_{q\cap q', \{q,q'\},i}$;  the collection $\sfC_{q,i}$ of canal sections connects  the pre-reservoir-block $\sfR_{q,i}$ to all its neighboring gates along $\cT$; the canal stretch $ \sfC_{\{q,q'\},i}$ over the edge $\{q,q'\}$ connects two pre-reservoir-blocks $\sfR_{q,i}$ and $\sfR_{q',i}$.

\subsection{Reservoir and canal systems}
\label{sec:sub-reservoir-canal-system}

We are now ready to define reservoir-canal systems.  We connect, for each index $i\in\{1,\ldots, m\}$,  pre-reservoir-blocks $\sfR_{q,i}$ and $\sfR_{q',i}$ for $\{q,q'\}\in \cT$ using the canal stretch $\sfC_{\{q,q'\},i}$. Since canal stretches of indices $i'>i$ intersect pre-reservoir-blocks of lower index $i$ (Lemma \ref{lemma:canal-reservoir-intersection}), we remove these canal stretches from the pre-reservoir-blocks.  See Figure \ref{fig:Reservoir_Canal_System} for an illustration.

\begin{figure}[h!]
\begin{overpic}[scale=.75,unit=1mm]{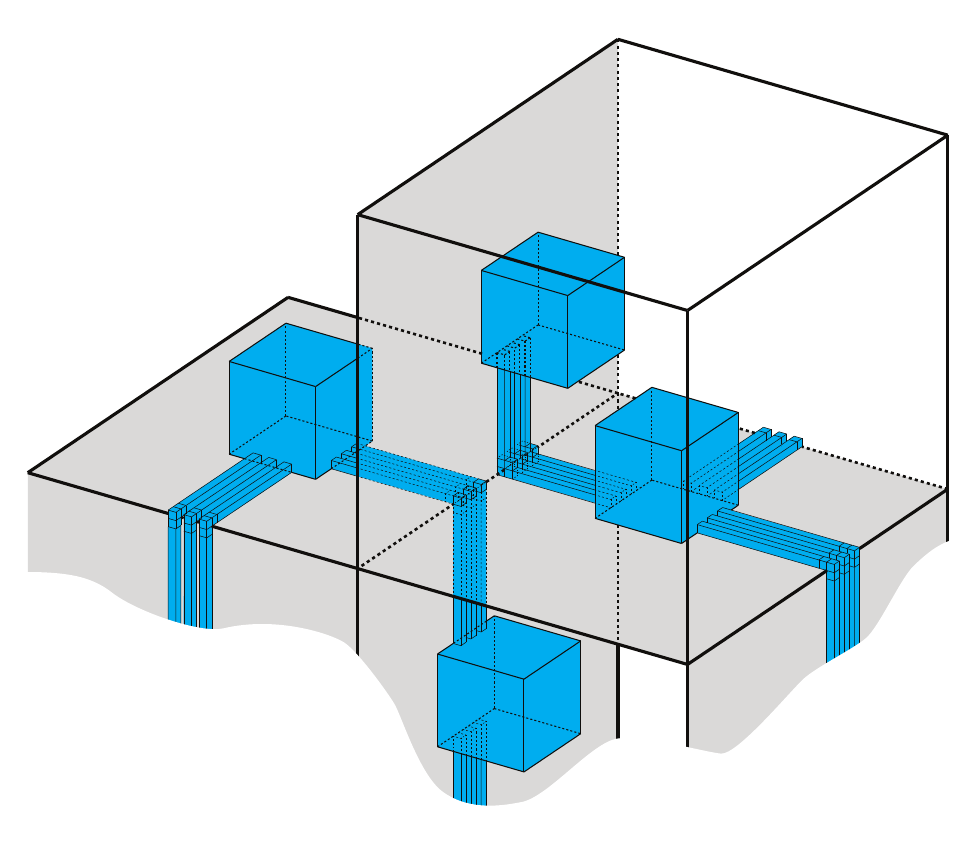}
\put(82,30){$q_1$}
\put(50,30){$q_5$}
\put(100,20){ $q_3$}
\put(10,35){ $q_6$}
\put(17,48){ $q_4$}
\put(73,85){ $q_2$}
\put(53,40){$e$}
\end{overpic}
\caption{Schematic figure of a space of a reservoir-canal system $\sfRC_{\Refine^\nu(U)}(Y)$ - a local picture (not in scale). Cubes in $Y$ are colored in grey.}
\label{fig:Reservoir_Canal_System}
\end{figure}

For each $q\in Y^{[n-1]}$ and $i\in \{1,\ldots, m\}$, we define 
\[
\widehat \sfR_{q,i} = \sfR_{q,i} - \bigcup_{i'>i} \bigcup_{\{q,q'\}\in \cT} \sfC_{q,\{q,q'\},i}
\]
to be the \emph{reservoir adjacent to $q$ of index $i$}.   \index{reservoir-canal system!reservoir} \index{$\widehat \sfR_{q,i}$}

\begin{lemma}
\label{lemma:adjacency-reservoirs}
Complexes $\widehat \sfR_{q,i}$ and $\partial \widehat \sfR_{q,i}$ are adjacently connected.
\end{lemma}
\begin{proof}
Since each canal stretch $\sfC_{\{q,q'\},i'}$ meets $\sfR_{q,i}$ in an $n$-cube if and only if $i' \ge i$, and these $n$-cubes are mutually disjoint, we conclude that the adjacency graphs $\Gamma(\widehat \sfR_{q,i})$ and $\Gamma(\partial \widehat \sfR_{q,i})$ are connected.
\end{proof}

\begin{definition}
\label{def:reservoir-canal-system}
\index{reservoir-canal system}   
\index{$\sfRC_{\Refine^\nu(U)}(Y)_i $}
We call the complex
\[
\sfRC_{\Refine^\nu(U)}(Y)_i = \left(\bigcup_{q\in Y^{[n-1]}} \widehat \sfR_{q,i} \right) \, \cup \left( \bigcup_{\{q,q'\}\in \cT} \sfC_{\{q,q'\},i} \right) \subset \Refine^\nu(U)
\]
the \emph{reservoir-canal system on $Y$ in $\Refine^\nu(U) $ of index $i$, with respect to the spanning tree $\cT$ and the preference function $\rho$}. We also call the subcomplex
\[
\sfRC_{\Refine^\nu(U)}(Y) =\bigcup_{i=1}^m \sfRC_{\Refine^\nu(U)}(Y)_i \subset \Refine^\nu(U)
\]
the \emph{entire reservoir-canal system on $Y$}. 
\end{definition}

We list some properties of reservoir-canal systems. We begin with connectedness properties, which follow from Lemmas \ref{lemma:adjacency-reservoirs} and \ref{lemma:adjacency-canal-stretch}.

\begin{corollary}
\label{cor:adjacency-RC}
Each reservoir-canal system $\sfRC_{\Refine^\nu(U)}(Y)_i$ on $Y$ is adjacently connected and has adjacently connected boundary. Moreover, the union $\sfRC_{\Refine^\nu(U)}(Y)$ is adjacently connected and  $\bigcup_{i=1}^m \partial \sfRC_{\Refine^\nu(U)}(Y)_i$ is also adjacently connected.
\end{corollary}

Second, we note that the spaces of the reservoir-canal systems are $n$-cells. 

\begin{corollary}
\label{cor:RC-is-a-cell}
The space of each $\sfRC_{\Refine^\nu(U)}(Y)_i$ is an $n$-cell. 
\end{corollary}
\begin{proof}
For each $q\in Y^{[n-1]}$, the space of $\widehat \sfR_{q,i}$ is an $n$-cell. Furthermore, for each $\{q,q'\}\in \cT$, the space of $\sfC_{\{q,q'\},i}$ is an $n$-cell and $\widehat \sfR_{q,i}\cap \sfC_{\{q,q'\},i}$ is an $(n-1)$-cube. Thus, for each $\{q,q'\}\in \cT$, the space of $\widehat \sfR_{q,i}\cup \sfC_{\{q,q'\},i}$ is an $n$-cell. Since the reservoirs and canal stretches are otherwise mutually disjoint, the claim follows now by induction in the tree $\cT$. 
\end{proof}

Third, since the complexes $\sfRC_{\Refine^\nu(U)}(Y)_i$ for distinct indices meet only in reservoirs, we have the following. 
\begin{corollary}
\label{cor:reservoir-canal-system-intersection}
For $i\ne i'$, the intersection $\sfRC_{\Refine^\nu(U)}(Y)_i \cap \sfRC_{\Refine^\nu(U)}(Y)_{i'}$ is an $(n-1)$-complex.
\end{corollary}
\begin{proof}
We may assume that $i'>i$. Then the intersection is the union of the intersections $\widehat \sfR_{q,i}\cap \sfC_{\{q,q'\},i}$ for $\{q,q'\}\in \cT$ and the intersections $\widehat \sfR_{q,i}\cap \widehat \sfR_{q,i'}$ for $q\in Y^{[n-1]}$. Since $i'>i$,  the first intersection is always non-empty. The latter intersection is non-empty if only if $|i-i'|=1$. Since each of these intersections is an $(n-1)$-complex, the claim follows. 
\end{proof}

Fourth, we observe that the reservoir-canal system $\sfRC_{\Refine^\nu(U)}(Y)$ is far from the boundary $\partial Y$ in a quantitative sense. 

\begin{lemma}
\label{lemma:border-of-Y} 
The graph distance, with respect to $\Gamma({\Refine^\nu(U)})$, from an $n$-cube in $\sfRC_{\Refine^\nu(U)}(Y)$ to any $n$-cube in ${\Refine^\nu(U)}$ that meets  $\Refine^\nu(\partial Y)$ is  at least  $3^{\nu-2}$. 
\end{lemma}
\begin{proof}
Since  $(n-2)$-cubes on the boundary of $Y$ are not edges of the the tree $\cT$, there are no gates on $\partial Y$. Therefore, there are no canal sections in $\sfRC_{\Refine^\nu(U)}(Y)$ connecting reservoirs to $\Star_{\Refine^\nu(U)}(\partial Y)$. The lemma follows from the fact that reservoirs are located along the center of the center of the cubes $q\in Y^{[n-1]}$. 
\end{proof}

\subsubsection{Partition  $\{ \sfRC_{\Refine^\nu(U)}(Y)_q  \colon q\in Y^{[n-1]}\}$ of $\sfRC_{\Refine^\nu(U)}(Y)$}
\label{sec:RC-partition-UY}

In preparation for a local channeling argument in Section \ref{sec:localization}, we fix a partition of a reservoir-canal system $\sfRC_{\Refine^\nu(U)}(Y)$  with respect to $q\in Y^{[n-1]}$. 

To begin, we divide each canal stretch into two parts by assigning the connector to one of the canal sections according to the partial order in $\cT$. For an edge $\{q, q'\}\in \cT$ with $q<q'$, we let 
\[
\widetilde \sfC_{q,\{q,q'\},i} =\sfC_{q,\{q,q'\},i} \cup J_{\{q,q'\},i} \quad \text{and}\,\,\, \widetilde \sfC_{q',\{q,q'\},i}= \sfC_{q',\{q,q'\},i}.
\]

We set 
\[
\sfRC_{\Refine^\nu(U)}(Y)_{q,i}= \widehat \sfR_{q,i} \cup \bigcup_{\{q,q'\}\in \cT} \widetilde \sfC_{q,\{q,q'\},i},
\]
and
\[
\sfRC_{\Refine^\nu(U)}(Y)_q = \bigcup_{i=1}^m \sfRC_{\Refine^\nu(U)}(Y)_{q,i}.
\]
The collection $\{ \sfRC_{\Refine^\nu(U)}(Y)_q \colon q\in Y^{[n-1]}\}$ is an essential partition of the reservoir-canal system $\sfRC_{\Refine^\nu(U)}(Y)$ in the sense that     \index{$\sfRC_{\Refine^\nu(U)}(Y)_q $} \index{reservoir-canal system!essential partition}
\[
\sfRC_{\Refine^\nu(U)}(Y)=   \bigcup_{q\in Y^{[n-1]}}  \sfRC_{\Refine^\nu(U)}(Y)_q 
\]
and that subcomplexes $\sfRC_{\Refine^\nu(U)}(Y)_q$ do not have common $n$-cubes. We also note that, by Lemmas \ref{lemma:adjacency-reservoirs} and \ref{lemma:adjacency-canal-stretch}, each $\sfRC_{\Refine^\nu(U)}(Y)_q$ is adjacently connected with adjacently connected boundary.

 \subsubsection{Well-located cubes in $\sfRC_{\Refine^\nu(U)}(Y)$}
 
For future references, we also fix the notion of a well-located cube in a reservoir-canal system.

\begin{definition}\label{def:well-located-cube}
An $n$-cube $C$ in the reservoir-canal system $\sfRC_{\Refine^\nu(U)}(Y)$ is \emph{well-located}, if $C$ belongs to a canal section and it is not adjacent to any pre-reservoir-cube or any connector.
\end{definition}
 
 \subsection{Transformation of $Y$}
\label{sec:reservoir-canal-transformation}
\index{separating complex!transformation}\index{reservoir-canal system!transformation}

We use now the reservoir-canal system $\sfRC_{\Refine^\nu(U)}(Y)$ to transform the subcomplex $\Refine^\nu(Y)$ in $\Refine^\nu(U)$, and call the new subcomplex  the \emph{reservoir-canal transformation of $Y$}.

For each $i\in \{1,\ldots, m\}$, fix a spanning tree $\tau_{Y,i} \subset \Gamma(\sfRC_{\Refine^\nu(U)}(Y)_i )$, and take  
\[
\tau_Y = \tau_{Y,1}\cup \cdots \cup \tau_{Y_m}.
\]
The collection of  $(n-1)$-cubes $Q\cap Q'$ representing edges $\{Q,Q'\}$ in $\tau_Y$,
\[
P_{\tau_Y} = \{ Q\cap Q' \in \sfRC_{\Refine^\nu(U)}(Y)^{[n-1]} \colon \{Q,Q'\}\in \tau_Y \},
\]
is called a family of \emph{$\tau_Y$-passages.}
\begin{definition}
\label{def:reservoir-canal-transformation}
The $(n-1)$-subcomplex  \index{$\Tr_{\Refine^\nu(U)}(Y)$} 
\[
\Tr_{\Refine^\nu(U)}(Y) = \Refine^\nu(Y) \cup \left( \sfRC_{\Refine^\nu(U)}(Y)^{(n-1)} -  P_{\tau_Y}\right).
\]
is called the \emph{reservoir-canal transformation of $Y$  
with respect to  $\tau_Y$}.
\end{definition}

\begin{remark}
Although not emphasized in this definition, $\Tr_{\Refine^\nu(U)}(Y)$ depends on Standing assumptions \ref{standing:rho-T}, especially on the spanning tree $\cT$, preference function $\rho$, and the selection of gates in Section \ref{sec:gates}. We omit these dependencies  as well as the choice of the trees $\tau_{Y,i}$ from the notation, as these choices have no role in what follows.
\end{remark}

\begin{remark}
\label{rmk:tau_i}
The spanning tree $\tau_{Y,i}$ is encoded into the reservoir-canal transformation $\Tr_{\Refine^\nu(U)}(Y)$. Indeed, for each $i=1,\ldots, m$, we have that 
\[
\tau_{Y,i} = \Gamma(\sfRC_{\Refine^\nu(U)}(Y)_i; \Tr_{\Refine^\nu(U)}(Y)\cap \sfRC_{\Refine^\nu(U)}(Y)_i).
\]
\end{remark}

We finish this section by recording two basic properties of the reservoir-canal transformation that are used in  the induction steps.

\begin{lemma}
\label{lemma:Y_sfRC-connected}
The transformation $\Tr_{\Refine^\nu(U)}(Y)$ is adjacently connected.
\end{lemma}
\begin{proof}
Since $Y$ is adjacently connected, the refinement $\Refine^\nu(Y)$ is adjacently connected. For each $i=1,\ldots, m$, let $P_i$ be the set of passages induced by the tree $\tau_{Y,i}$. Then $\{P_1,\ldots, P_m\}$ is a partition of the set $P_{\tau_Y}$.

We now check that $\Gamma (\Span_{\Refine^\nu(K)}(\sfRC_{\Refine^\nu(U)}(Y)^{[n-1]} \setminus P_{\tau_Y})$ is connected.
Note that 
\[
\sfRC_{\Refine^\nu(U)}(Y)^{[n-1]} \setminus P_{\tau_Y}  =  \bigcup_{i=1}^m \, (\sfRC_{\Refine^\nu(U)}(Y)_i ^{[n-1]} \setminus P_i).
\]
Let $i \in \{1,\ldots, m\}$. Since $\sfRC_{\Refine^\nu(U)}(Y)_i$ is an adjacently connected $n$-complex, $\Span_{\Refine^\nu(U)}(\sfRC_{\Refine^\nu(U)}(Y)_i ^{[n-1]})$ is an adjacently connected  $(n-1)$-complex. Since all cubes in $P_i$ are contained in the interior of $\sfRC_{\Refine^\nu(U)}(Y)_i $, the graph $\Span_{\Refine^\nu(U)}(\sfRC_{\Refine^\nu(U)}(Y)_i ^{[n-1]} \setminus P_i)$ is adjacently connected. 

Since complexes $\Span_{\Refine^{\nu}(U)}({\sfRC_{\Refine^\nu(U)}(Y)_i}^{[n-1]} \setminus P_i)$ 
of consecutive indices $i$ meet in single $(n-1)$-cubes, the adjacency graph $\Gamma( \sfRC_{\Refine^\nu(U)}(Y)^{[n-1]} \setminus (P_1\cup \cdots \cup P_m) )$ is connected.

Since the $n$-cubes in the canal stretches in $\sfRC_{\Refine^\nu(U)}(Y)$ have faces in $\Refine^\nu(Y)$ and these faces do not belong to $P_1\cup \cdots \cup P_m$, we have that the graphs $\Gamma(\sfRC_{\Refine^\nu(U)}(Y)^{[n-1]}\setminus (P_1 \cup \cdots \cup P_m))$ and $\Gamma(\Refine^\nu(Y)^{[n-1]})$ belong to the same connected component of $\Gamma(\Tr_{\Refine^\nu(U)}(Y))$. Thus $\Gamma(\Tr_{\Refine^\nu(U)}(Y))$ is connected.
\end{proof}

By Lemma \ref{lemma:border-of-Y}, the boundary of $Y$ remains unchanged  after the reservoir-canal transformation.

\begin{corollary}\label{cor:boundary-of-Y_sfRC}
The $(n-1)$-complexes $\Tr_{\Refine^\nu(U)}(Y)$ and $\Refine^\nu(Y)$ have the same boundary. 
\end{corollary}

\subsection{Receded complexes} 
\label{sec:receded-complex}
Reservoir-canal systems give rise to bent indentations in realizations $\Real_{\Refine^\nu(U_j)}(\Refine^\nu(Y);\Sigma)$. We now discuss the subcomplexes of $\Refine^\nu(U_j)$ receded by the removal of reservoir-canal system.

Recall that  $G_1,\ldots, G_r$ are the connected components of the cut-graph $\Gamma(U;Y)$, and subcomplexes $U_j = \Span_U(G_j)$ for $j\in \{1,\ldots, r\}$, are the subcomplexes of $U$ spanned by the vertices of $G_j$.

We say that a reservoir-canal system $\sfRC_{\Refine^\nu(U)}(Y)$ \emph{enters $U_j$} if the complex $\sfRC_{\Refine^\nu(U)}(Y)$ has $n$-cubes in $\Refine^\nu(U_j)$, and call, in this case, 
\[
\Rec_{\Refine^\nu(U)}(U_j;Y)= \Refine^\nu(U_j) - \sfRC_{\Refine^\nu(U)}(Y)
\]
a \emph{receded subcomplex of $\Refine^\nu(U_j)$}. 
\index{receded complex} \index{$\Rec_{\Refine^\nu(U)}(U_j;Y)$}

The complex $\Rec_{\Refine^\nu(U)}(U_j;Y)$ is adjacently connected.
We first show a local version of the connectedness property. The statement as well as the proof will be used  in Lemma \ref{lemma:Z_1-core-expanding} to study the relation between cores.

\begin{lemma}
\label{lemma:receded-connected}
For each $j\in \{1,\ldots, r\}$ and each $Q\in U_j^{[n]}$, the adjacency graph $\Gamma(\Refine^\nu(Q)\cap \Rec_{\Refine^\nu(U)}(U_j;Y))$ is connected. 
\end{lemma}

\begin{proof}
Let $Q\in U_j^{[n]}$. If $Q\cap |Y|=\emptyset$, then  $\Refine^\nu(Q) \subset \Rec_{\Refine^\nu(U)}(U_j;Y)$. Thus $\Gamma(\Refine^\nu(Q)\cap \Real(U_j;Y)) = \Gamma(\Refine^\nu(Q))$, which is connected. We assume from now on that $Q\cap |Y|\ne \emptyset$, and consider two cases. 

Suppose first that $Q$ is not in the image of the preference function $\rho$. Then $\Refine^\nu(Q)\cap \sfRC_{\Refine^\nu(U)}(Y)$ consists of $n$-cubes in connectors only. Each of these $n$-cubes meets  $\Refine^\nu(\partial Q)$ in an $(n-1)$-cell. Thus $\Gamma(\Refine^\nu(Q)\cap \Rec(U_j))$ is connected.

Suppose next that $Q$ is the image of one or more $(n-1)$-cubes in $Y$ under the preference function $\rho$. Then the space of $\Refine^\nu(Q)\cap \sfRC_{\Refine^\nu(U)}(Y)$ may be partitioned into essentially disjoint $n$-cells, each of which is either the space of an $n$-cube in $\Refine^2(Q)$ having a face contained in $\partial Q$, or the space of an $n$-cube in $\Refine^\nu(Q)$ having one or more faces contained in $\partial   Q$. Since 
\begin{align*}
&\Refine^\nu(Q) \cap \Rec_{\Refine^\nu(U)}(U_j;Y)\\
&\quad = \Span_{\Refine^\nu(U)} \left( \Refine^\nu(Q) -  (\Refine^\nu(Q) \cap \sfRC_{\Refine^\nu(U)}(Y)) \right),
\end{align*}
each $n$-cube in $P=\Refine^\nu(Q) \cap \Rec_{\Refine^\nu(U)}(U_j;Y)$ may be connected in $P$ to the $n$-cube at the center of $\Refine^\nu(Q)$.  
We conclude that $\Gamma(\Refine^\nu(Q) \cap \Rec(U_j))$ is connected. 
\end{proof}

\begin{corollary}\label{cor:receded-connected}
Each complex $\Rec_{\Refine^\nu(U)}(U_j;Y)$ is adjacently connected.
\end{corollary}

\begin{proof}
Let $\{Q,Q'\} \in \Gamma(U_j)$ and let $q=Q\cap Q'$ be the common face. Since the graphs $\Gamma(\Refine^\nu(Q)\cap \Rec_{\Refine^\nu(U)}(U_j;Y))$ and  $\Gamma(\Refine^\nu(Q')\cap \Rec_{\Refine^\nu(U)}(U_j;Y))$ are connected and can be connected to each other an $(n-1)$-cube in $\Refine^\nu(q)$,   graphs $\Gamma(\Refine^\nu(Q)\cap \Rec_{\Refine^\nu(U)}(U_j;Y))$ and $\Gamma(\Refine^\nu(Q')\cap \Rec_{\Refine^\nu(U)}(U_j;Y))$ belong to the same connected component of the graph $\Gamma(\Rec_{\Refine^\nu(U)}(U_j;Y))$. Since $\Gamma(U_j)$ is connected, we conclude that the graph  $\Gamma(\Rec_{\Refine^\nu(U)}(U_j;Y))$ is connected.
\end{proof}

By Remark \ref{rmk:tau_i} and Corollary \ref{cor:receded-connected}, we conclude that the graph $\Gamma(\Refine^\nu(U);\Tr_{\Refine^\nu(U)}(Y))$ has $r+m$ connected components. We state this observation as a corollary.

\begin{corollary}
\label{cor:reservoir-canal-components}
Let $\Tr_{\Refine^\nu(U)}(Y)$ be the reservoir-canal-transformation of $Y$ in $\Refine^\nu(U)$. Then the connected components of  $\Gamma(\Refine^\nu(U);\Tr_{\Refine^\nu(U)}(Y))$ are 
 \[
\Gamma(\Rec_{\Refine^\nu(U)}(U_1;Y)),\ldots,  \Gamma(\Rec_{\Refine^\nu(U)}(U_r;Y)), \tau_{Y,1},\ldots,\tau_{Y,m}.
\]
\end{corollary}

\subsection{Indentation in realizations}
\label{sec:indentation-lift}

Recall the notion of bent indentation from Section \ref{sec:bent-indentations}. We show now that the lift of a reservoir-canal system to a realization is a bent indentation. Let 
\[
\Dent(U_j,Y)= \pi_{(\Refine^\nu(U_j),\Refine^\nu(Y))}^{-1}\left(\sfRC_{\Refine^\nu(U)}(Y) \cap \Refine^\nu(U_j) \right).
\]
\index{$\Dent_{\Refine^\nu(\Real_{U_j}(Y))}$}

\begin{remark}
Although the reservoir-canal system $\sfRC_{\Refine^\nu(U)}(Y)$ is adjacently connected, its lift $\Dent(U_j,Y)$ to  $\Refine^\nu(\Real_{U_j}(Y))$, $1\leq j\leq r$, is typically not connected. The splitting of the lift of $ \sfRC_{\Refine^\nu(U)}(Y)$ occurs precisely where $Y$ enters the interior of $|\sfRC_{\Refine^\nu(U)}(Y)|$; in fact,  splitting occurs where $Y$ enters the interior of a connector $|\sfJ_{\{q,q'\},i}|,\, i\in \{i,\ldots, m\}$. 
\end{remark}

First we check that the lift of a canal stretch $\sfC_{\{q,q'\},i}$ is a bent indentation, for each $i\in\{1,\ldots,m\}$.

\begin{lemma}
\label{lemma:canal-stretch-as-bent-indentation}
For each $j\in \{1,\ldots, r\}$ and each $i\in\{1,\ldots,m\}$, the lift 
\[
\pi_{(\Refine^\nu(U_j),\Refine^\nu(Y))}^{-1}\left(\sfC_{\{q,q'\},i}\cap \Refine^\nu(U_j) \right)
\]
of the canal stretch $\sfC_{\{q,q'\},i}$, if nonempty, is a bent indentation in the realization $\Refine^\nu(\Real_{U_j}(Y))$. 
\end{lemma}
\begin{proof}
Fix an index $i$, and let $R \subset \sfC_{\{q,q'\},i}$ be the subcomplex whose adjacency graph is a connected component of the cut-graph $\Gamma(\sfC_{\{q,q'\},i};Y)$. Then $R$, if nonempty, has one of the three possible forms:
 \begin{enumerate}
\item  the entire $\sfC_{\{q,q'\},i}$, or  \label{item:P_adjacent_both}
\item the union of a partial star and a canal section in $\sfC_{\{q,q'\},i}$, or \label{item:P_adjacent_one}
\item  a partial star which is not adjacent to either  canal section in $\sfC_{\{q,q'\},i}$.  \label{item:P_isolated}
\end{enumerate}
Moreover, $\Gamma(R)$ is isomorphic to $\Gamma(\pi_{(\Refine^\nu(U_j),\Refine^\nu(Y))}^{-1}\left(R\cap \Refine^\nu(U_j)\right)$.
In all cases, $\pi_{(\Refine^\nu(U_j),\Refine^\nu(Y))}^{-1}\left(R\cap \Refine^\nu(U_j)\right)$, if nonempty, is either a bent indentation or a single partial-star-indentation;  see Definition \ref{def:bent-indentation}. Since disjoint unions of bent indentations are bent indentations, the claim in the lemma follows.
\end{proof}

We now check that lifts of truncated reservoir-canal systems are bent indentations in realizations. Indeed, since pre-reservoir-cubes are refinements of $n$-cubes in $\Refine^{2}(K)$,  we have, by Lemma \ref{lemma:canal-stretch-as-bent-indentation}, the following.

\begin{corollary}
\label{cor:reservoir-canals-as-bent-indentations}
For each $j\in \{1,\ldots,r\}$, the complex $\Dent(U_j;Y)$ is a bent indentation of $\Refine^\nu(\Real_{U_j}(Y))$. 
\end{corollary}

We denote 
\[
\widetilde \Rec_{\Refine^\nu(U)}(U_j;Y) = \pi_{(\Refine^\nu(U_j);\Refine^\nu(Y))}^{-1}(\Rec_{\Refine^\nu(U)}(U_j;Y)).
\]
the lift of the receded complex to the realization,

\medskip

Since $\Dent(U_j;Y)$ is a bent indentation, we have, by Corollary \ref{cor:reservoir-canals-as-bent-indentations} and Proposition \ref{prop:flattening-bent-indentation}, that there exists bilipschitz homeomorphisms from $|\widetilde \Rec_{\Refine^\nu(U)}(U_j;Y)|$  to $|\Real_{U_j}(Y)|$. We record this as a corollary.

\begin{corollary}
\label{cor:rec-bilipschitz}
There exist $L=L(n,K)\ge 1$ and an $L$-bilipschitz homeomorphism $|\widetilde \Rec_{\Refine^\nu(U)}(U_j;Y)|\to |\Real_{U_j}(Y)|$ which is the identity in the complement of the wedge $\Wedge_{\Refine^\nu(\Real_{U_j}(Y))}(\Dent(U_j;Y))$.  
\end{corollary}

\subsection{Proof of Proposition \ref{prop:RC-existence}}
\label{sec:RC-existence-summary}

It suffices to summarize the properties of complexes $\sfRC$ and $\Tr$ stated in Proposition \ref{prop:RC-existence}. 

We note first that $\sfRC = \sfRC_{\Refine^\nu(U)}(Y)$ and $\Tr = \Tr_{\Refine^\nu(U)}(Y)$ are adjacently connected by Corollary \ref{cor:adjacency-RC} and Lemma \ref{lemma:Y_sfRC-connected}. 

Condition \eqref{item:RC-existence-2} is satisfied, since spectral cubes in $\sfRC$ are either pre-reservoir-blocks, cubes in canal sections, or cubes in connectors. The choices of reservoirs and canals are such that pre-reservoir-blocks and cubes in canal sections are contained in $\Refine^\nu(\rho(Y))$ as well as those cubes in connectors which have a face in $Y$.

Condition \eqref{item:RC-existence-3} follows by Corollary \ref{cor:boundary-of-Y_sfRC}. Condition \eqref{item:RC-existence-4} on the number of components of $\Gamma(\sfRC;\Tr\cap \sfRC)$ is guaranteed by the choice of the graph $\tau_Y$. Components of  $\Gamma(\sfRC;\Tr\cap \sfRC)$ are trees by Remark \ref{rmk:tau_i} and their size depends only on $\nu$ and $\# Y^{[n-1]}$ by construction.

Condition \eqref{item:RC-existence-5} is guaranteed by the spanning tree $\cT$ of $\Gamma(Y)$ and the construction of canal sections along $\cT$. Finally, condition \eqref{item:RC-existence-6}, stating that the part of $\widetilde{\sfRC}$ in any component of $\Real_{\Refine^\nu(U)}(Y)$ is a bent indentation of that component, is proved in Corollary \ref{cor:reservoir-canals-as-bent-indentations}. 

This completes a summary of the proof of Proposition \ref{prop:RC-existence}. \qed


\chapter{Separating complexes}
\label{chap:Separating-complexes}

The aim of this chapter is to define a notion central to this article --  separating complexes -- and to show that all good complexes admit separating complexes, possibly after refinement. We also state a theorem on evolution of separating complexes, which serves as the first step in the proof of Quasiregular cobordism theorem.

\section{Definition of separating complexes}
\label{sec:Separating-complexes}

Heuristically, a separating complex $Z$ partitions the underlying complex $K$ into subcomplexes whose realizations resemble collars $\Star_K(\Sigma)$ of the boundary components $\Sigma \subset \partial K$. A controlled sequence $(Z_k)$, called evolution of a separating complexes, of such separating complexes in refinements $\Refine^{\nu k}(K)$ of $K$ yields geometrically uniform partitions of $K$. Here $\nu=\nu(K)$ is a constant associated to complex $K$ called the \emph{refinement scale of $K$}.
\begin{definition}
\label{def:nu}
\index{refinement scale $\nu$}
\index{$\nu(K)$}
Let $K$ be a good cubical $n$-complex with $m$ boundary component, $m\geq 1$. The \emph{refinement scale $\nu(K)\ge 1$ of $K$} is the smallest integer $\nu\ge 1$ satisfying 
\[
3^\nu \ge 3^{10}m \mu(K)^2,
\]
where $\mu(K)$ the multiplicity of $K$.
\end{definition}

For the definition of a separating complex, we refer to Section \ref{sec:adjacency-graph} for the terminologies associated to adjacency graphs and  recall here the following. Let $K$ be a cubical $n$-complex,  $Z\subset K$ an $(n-1)$-subcomplex whose space is in the interior of $|K|$, and $\Sigma$ a boundary component of $K$. Then the graph $\Gamma(K;Z,\Sigma)$ is the connected component of the cut-graph $\Gamma(K;Z)$ for which $\Sigma \subset \Span_K(\Gamma(K;Z,\Sigma))$, the subcomplex $\Comp_K(Z;\Sigma)\subset K$ is 
the $\Sigma$-component of $K$ separated by $Y$, and $\Real_K(Z; \Sigma)$ is the realization of the graph $\Gamma(K;Z,\Sigma)$; see Section \ref{sec:cut-graphs}.  \index{realization $\Real_K(Z;\Sigma)$}

In general, it may occur that $\Sigma\cup \Sigma'\subset \Comp_K(Z;\Sigma) =\Comp_K(Z;\Sigma')$ for two distinct boundary components $\Sigma$ and $\Sigma'$ of $K$. It may also occur that the cut-graph $\Gamma(K;Z)$ has a component whose vertices do not meet $\partial K$. The definition of a separating complex gives us a notion to avoid these issues.

\begin{definition} \index{separating complex}
\label{def:separating-complex}
Let $K$ be a good cubical $n$-complex. An $(n-1)$-subcomplex $Z \subset K$ is a \emph{separating complex of $K$} if 
\begin{enumerate}
\item $Z\cap \Star_K(\partial K) = \emptyset$; \label{item:separating-no-boundary}
\item $\Gamma(Z)$ is connected, that is, $Z$ is adjacently connected as a cubical $(n-1)$-complex;   \label{item:separating-strongly-connected}
\item for each boundary component $\Sigma \subset \partial K$, $\Comp_K(Z;\Sigma)\cap \partial K=\Sigma$, and the realization $\Real_K(Z;\Sigma)$ is carried by a space homeomorphic to $|\Sigma|\times [0,1]$; and  \label{item:separating-mu-1}
\item $K = \bigcup_{\Sigma} \Comp_K(Z;\Sigma)$. \label{item:separating-total}
\end{enumerate}\index{separating complex} 
\end{definition}

\begin{remark}
In condition \eqref{item:separating-mu-1}, we may not replace realization $\Real_K(Z;\Sigma)$ by the complex $\Comp_K(Z;\Sigma)$. In  general, neither spaces $|\Comp_K(Z;\Sigma)|$ and $|\Real_K(Z;\Sigma)|$, nor their interiors, are homeomorphic, respectively. Note, however, that $\interior (|\Comp_K(Z;\Sigma)|\setminus |Z|)$ is homeomorphic to  $\interior |\Real_K(Z;\Sigma)|$. Hence  $\interior (|K|\setminus |Z|)$ is homeomorphic to the interior of $\bigcup_\Sigma |\Real_K(Z;\Sigma)|$.
\end{remark}

In Figure \ref{fig:Separating_complex}, the cut-graph $\Gamma(K;Z)$ has four connected components, and each $\Real_K(Z;\Sigma_i), i=1,\ldots, 4,$ has two boundary components. Since all $2$-cubes adjacent to the $1$-subcomplex $\tau \subset Z$ are cubes in $\Comp_K(Z;\Sigma_1)$, 
subcomplex $\tau$ is contained in $\Comp_K(Z;\Sigma_1)$ but not in $\partial (\Comp_K(Z;\Sigma_1))$.
Thus  $\Comp_K(Z;\Sigma_1)$ has three boundary components.

\begin{figure}[h!]
\begin{overpic}[scale=.45,unit=1mm]{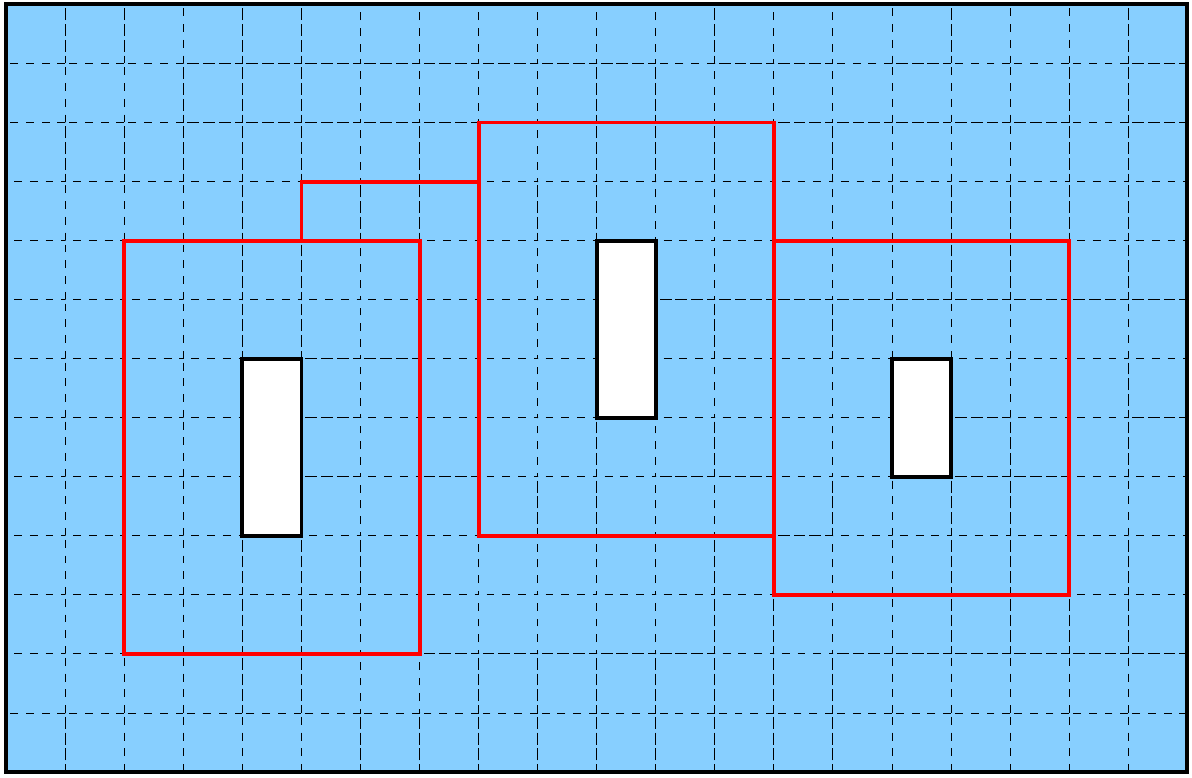}
\put(-4,50){\tiny $\Sigma_1$}
\put(15,29){\tiny $\Sigma_2$}
\put(42,38){\tiny $\Sigma_3$}
\put(64.5,29){\tiny $\Sigma_4$}
\put(25,46.5){\tiny $\tau$}
\put(60,11){\tiny $Z$}
\end{overpic}
\caption{Separating complex $Z$ in a complex $K$.}
\label{fig:Separating_complex}
\end{figure}

By condition  \eqref{item:separating-no-boundary}, a separating complex $Z$ of $K$ does not meet the star $\Star_K(\partial K)$ of the boundary $\partial K$ of $K$, thus $\Gamma(\Star_K(\Sigma)) \subset \Gamma(K;Z, \Sigma)$. Therefore, $\Sigma$ and $\Star_K(\Sigma)$ lift isomorphically into the realization $\Real_K(Z;\Sigma)$. In other words, the restriction 
\[
\pi_{(K,Z;\Sigma)}|_{\pi_{(K,Z;\Sigma)}^{-1}(\Star_K(\Sigma))} \colon \pi_{(K,Z;\Sigma)}^{-1}(\Star_K(\Sigma)) \to \Star_K(\Sigma)
\]
is an isomorphism. 

\begin{convention}
From here on, we identify, as we may, subcomplexes $\Sigma\subset \partial K$ and $\Star_K(\Sigma)$ with their preimages $\pi_{(K,Z;\Sigma)}^{-1}(\Sigma)$ and $\pi_{(K,Z;\Sigma)}^{-1}(\Star_K(\Sigma))$ in the realization $\Real_K(Z;\Sigma)$. 
\end{convention}

To this end, we note that separating complexes are stable under refinement. We record this fact as a lemma and omit the straightforward proof. 

\begin{lemma}
\label{lemma:Separating-complex-trivial}
If $Z\subset K^{(n-1)}$ is a separating complex of a cubical $n$-complex $K$, then $\Refine(Z)$ is a separating complex of $\Refine(K)$.
\end{lemma}

\section{Existence of separating complexes}
\label{sec:separating-complex-existence}

We show that -- after refinement -- each good complex admits a separating complex. Recall that subcomplexes of $K$, which do not meet the separating complex $Z \subset K$, are identified with their isomorphic copies in $\Real_K(Z)$.

\begin{theorem}
\label{thm:separating-complex-existence}
Let $n\geq 2$ and let $K$ be a good cubical $n$-complex having at least one boundary component $\Sigma' \subset \partial K$. Then there exists a separating complex $Z$ of $\Refine(K)$ satisfying the following conditions:
\begin{enumerate}
\item for each connected component $\Sigma\subset \partial K$, $\Sigma\ne \Sigma'$, we have that $\Comp_{\Refine(K)}(Z;\Sigma) = \Refine(\Star_K(\Sigma))$, \label{item:separating-1}
\item  there exists a spanning tree $\gamma$ of $\Gamma(\Refine(K - \bigcup_{\Sigma} \Star_K(\Sigma)))$, where the union is taken over all connected components $\Sigma$ of $\partial K$,  for which
\[ 
\Real_{\Refine(K)}(Z;\Sigma') = \Refine(\Star_K(\Sigma')) \cup \Real(\gamma),
\]
and $\Refine(\Star_K(\Sigma')) \cap \Real(\gamma)$ is an $(n-1)$-cube. \label{item:separating-2}
\end{enumerate}
\end{theorem}

\begin{remark}
By condition \eqref{item:separating-2} we further have that $|\Comp_{\Refine(K)}(Z;\Sigma')|\setminus |Z|$ is homeomorphic to $\Sigma' \cup \interior |\Refine(\Star_K(\Sigma'))|$. Thus condition \eqref{item:separating-1} together with Proposition \ref{prop:Riemannian-to-cubical} give a quantitative version of the following topological statement: \emph{on each Riemannin $n$-manifold $M$ with boundary and $n\geq 2$, there exist a closed set $C$ for which $M\setminus C$ is homeomorphic to $\partial M\setminus [0,1)$}.
\end{remark}

\begin{proof}[Proof of Theorem \ref{thm:separating-complex-existence}]
To simplify notations, let $\Sigma_1,\ldots, \Sigma_m$ be the boundary components of $\partial K$ and let $\Sigma'=\Sigma_1$. Since $K$ is a good cubical $n$-complex, the stars $\Star_K(\Sigma_i)$ are mutually disjoint and graphs $\Gamma(\Star_K(\Sigma_i))$ are connected.

Let  $C'$ be the cubical $n$-complex obtained from $K$ by removing all stars $\Star_K(\Sigma_i)$, that is, 
\[
C' =  K - \left( \Star_K(\Sigma_1)\cup \ldots \cup \Star_K(\Sigma_m)\right).
\]
Since the stars $\Star_K(\Sigma_i)$ are mutually disjoint and $\Gamma(K)$ is connected, complex $C'$ is adjacently connected. 
Let $B' = \Star_K(\Sigma_1)$, and
\[
K' = B' \cup C'.
\]
Then $K'$ is adjacently connected and $B'\cap C' =\partial(\Star_K(\Sigma_1))\setminus \Sigma_1$.    

Let $G'\subset \Gamma(C')$ be a spanning tree and let 
\[
P = \{ c(q) \in \Refine(C')^{[n-1]} \colon q\in (C')^{[n-1]}\,\,\, \text{an edge in}\,\, G'\}
\]
be the collection of the center cubes of the edges in $G'$, where $c(a)$ is the center cube of $\Refine(q)$. Let 
\[
Y= \Span_{\Refine(K)} \left(( \bigcup_{Q\in C'^{[n]}}  (\Refine(\partial Q))^{[n-1]}) \setminus \, P\right).
\]
Fix a spanning tree $\gamma \subset \Gamma(\Refine(C'))$ whose edges are not  in ${Y }^{[n-1]}$, and fix an $n$-cube $Q'\in (C')^{[n]}$ with a face $q'$ in $\Star_K(\Sigma_1)$ and let $\omega=c(q')$ be the center cube of $\Refine(q')$. 

We claim that
\[
Z=\Span_{\Refine(K)} \left( (\Refine(C'))^{[n-1]} \setminus (\gamma  \cup \{\omega\})\right)
\] 
is a separating complex in $\Refine(K)$. 

First, since $Z\subset \Refine(C')$, the intersection $Z\cap \Star_{\Refine(K)}(\Refine(\partial K)) =\emptyset$. Connectedness of $\Gamma(Z)$ is immediate from the construction.

Since $\Gamma(\Refine(K');Z)$ is connected, connected components of $\Gamma(\Refine(K);Z)$ are in one-to-one correspondence with the boundary components of $\partial K$. Condition \eqref{item:separating-total} in the definition of separating complex holds.

Finally, for $i=2,\ldots, m$,  the adjacency graph of the realization of $\Gamma(\Refine(\Star_K(\Sigma_i));Z)$ is isomorphic to $\Gamma(\Refine(\Star_K(\Sigma_i));Z)$.
The subgraph 
\[
G= \Gamma(\Refine(B')) \cup \gamma  \cup\{\omega\}
\]  
of $\Gamma(\Refine(K'))$ has a realization $R_{G}= B \cup  C $
and a $G$-quotient 
\[
\pi_{G} \colon R_{G} \to \Refine(K')
\]
for which $\pi_{G}|_B \colon  B\to \Refine(B')$ is an isomorphism, $C$ is a tunnel whose adjacency graph is isomorphic to $\gamma$, and $B \cap C$ is the $(n-1)$-cube $\omega$. Thus, $R_G$ is an extension of $B$ by $C$, and satisfies  Condition  \eqref{item:separating-mu-1} in the definition of separating complex. 

We conclude that $Z$ is a separating complex in $\Refine(K)$ having the properties stated in the proposition.
\end{proof}

We distinguish the particular separating complex constructed in Theorem \ref{thm:separating-complex-existence} in the form to be used in the future.

\begin{remark}
\label{rmk:separating-complex-special}
Let $K$ be a cubical $n$-complex, which is the refinement $\Refine(K_0)$ of a good $n$-complex $K_0$, and $\Sigma'$ be a connected component of $\partial K$. Then there exists a separating complex $\mathcal Z$ in $K$ such that for each component $\Sigma$ of $ \partial K$, $\Sigma\neq \Sigma'$,
\[\Comp_K(\mathcal Z;\Sigma)= C_\Sigma,\]
where $C_\Sigma$ is the collar of $\Sigma$ in $K$ which is isomorphic to $\Sigma \times [0,3]$.
\end{remark}
 
\section{Statement of Evolution theorem}
\label{sec:statement}

We state now, in the form of a theorem, our goal in  constructing a sequence $(Z_k)$ of separating complexes. 
It should be noted that, for the extension theorems in Part \ref{part:QR-extension}, more detailed descriptions of the separating complexes will  be needed. 

The proof of the main theorem of this section (Theorem \ref{theorem:evolution-short}) spans  the next  several sections and is finalized in  Section \ref{sec:evolution}. 

We denote by $\nu = \nu(K)\ge 1$ the refinement scale of a good cubical $n$-complex $K$.

\subsection{Quantitative aims}
\label{sec:evolution-short}

For the statement of Theorem \ref{theorem:evolution-short}, we formalize the aims as follows.

\subsubsection{Relative Wada property}
\label{sec:relative-Wada}
Our first aim may be described in terms  how the complementary components of the separating complexes are progressively intertwined. We call this a relative Wada property.

Given two separating complexes $Z\subset K$ and $Z'\subset \Refine^\nu(K)$, the requirement here is that  for each $q\in Z^{[n-1]}$, every component $\Comp_{\Refine^\nu(K)}(Z';\Sigma)$, $\Sigma\subset \partial K$, enters the union of the two $n$-cubes neighboring $q$. 

 \begin{definition}[Relative Wada property] 
\label{def:Wada}   
\index{Wada property!relative}
Let $K$ be a good cubical $n$-complex, $Z\subset K$ a separating complex, and $\nu=\nu(K)$ the refinement scale of $K$. A separating complex $Z' \subset \Refine^{\nu}(K)$ has the \emph{relative Wada property with respect to $Z$} if,  for each $q \in Z^{[n-1]}$ and each component $\Sigma$ of $\partial K$, there exists an $n$-cube $Q_{\Sigma,q}$ in $ \Comp_{\Refine^{\nu}(K)}(Z'; \Sigma)^{[n]}\cap {\Refine^\nu(Q\cup Q^*)}^{[n]}$, where $Q$ and $Q^*$ are the two $n$-cubes in $K$ having face $q$.
\end{definition}

The name of this property stems from the fact that a sequence $(Z_k)$ of separating complexes having relative Wada property yields a version of the classical Lakes of Wada sharing the same boundary. We discuss this in more detail in Section \ref{sec:Wada},

\subsubsection{Perturbation of a separating complex}
The second aim is a concrete perturbation procedure which yields a geometrical stability for the complementary components of the separating complexes.

\begin{definition}[Indented perturbation]
\index{perturbation!$\lambda$-indented}
\label{def:indented-perturbation}
\index{$\lambda$-tame}
A separating complex $Z'$ of $K'=\Refine^\nu(K)$ of a good cubical $n$-complex $K$, where $\nu = \nu(K)$, is an \emph{indented perturbation of a separating complex $Z$ of $K$} if 
there exist mutually essentially disjoint subcomplexes $T_\Sigma \subset \Refine^\nu(K)$, for boundary components $\Sigma\subset \partial K$,
whose union $T=\bigcup_{\Sigma} T_\Sigma$ has the following properties:
\begin{enumerate}
\item for each $\Sigma$, $\Comp_{K'}(Z';\Sigma) = \left( \Refine^\nu(\Comp_K(Z;\Sigma))-T\right) \cup T_\Sigma$; and \label{item:tame-perturbation-1}
\item the complex $\pi_{(K',Z';\Sigma)}^{-1}( \Refine^\nu(\Comp_K(Z;\Sigma))\, \cap\, T)$ is a bent indentation in the realization $\Real_{K'}(Z';\Sigma)$; \label{item:tame-perturbation-2}
\end{enumerate}
and, for each $\Sigma$, components of $\pi_{(K',Z';\Sigma)}^{-1}(T_\Sigma)$ are tunnels. 

We say that $Z'$ is a \emph{$\lambda$-indented perturbation of $Z$ for some $\lambda \ge 1$}, if $Z'$ is a indented perturbation of $Z$ and, for each component $\Sigma \subset \partial K$, each connected component of $T_\Sigma$ has at most $\lambda$ $n$-cubes.
\end{definition}

\subsubsection{Cores}\label{sec:cores}

Parts of complementary components of separating complexes are undisturbed by the iterative perturbation and intertwining.  We call these parts cores.

\begin{definition}\index{core}\index{$\Core_K(Z;\Sigma)$}
\label{def:core}
Let $K$ a good cubical $n$-complex. For each boundary component $\Sigma \subset \partial K$, we call the subcomplex
\[
\Core_K(Z;\Sigma) = \Comp_K(Z; \Sigma) - \Star_K(Z).
\]
the \emph{core of $\Comp_K(Z;\Sigma)$}.
\end{definition}

\begin{definition}
\label{def:core-expanding}
Let $K$ a good cubical $n$-complex and let $Z\subset K$ be a separating complex. A separating complex $Z' \subset \Refine^\nu(K)$ is \emph{core-expanding with respect to $Z$} if 
\[
|\Core_K(Z;\Sigma)| \subset |\Core_{\Refine^{\nu}(K)}(Z';\Sigma)|.
\]
for each boundary component $\Sigma$ of $K$. Furthermore, we say that a sequence $(Z_k)$ of separating complexes, where $Z_k \subset \Refine^{\nu k}(K)$, is  \emph{core-expanding} if, for each $k\in \N$, $Z_{k+1}$ is core-expanding with respect to $Z_k$.\index{core!expanding}
\end{definition}

For the problem of quasiregular extension in Part \ref{part:QR-extension}, we study bilipschitz classes and quasiconformal classes of the realizations $\Real_K(Z;\Sigma)$ instead of the components $\Comp_K(Z;\Sigma)$. For this reason, the following definitions are given in terms of realizations.

\begin{remark}\label{rmk:cores}
Since the core of the complex $\Comp_K(Z;\Sigma)$ lifts isomorphically into the realization $\Real_K(Z;\Sigma)$, we also consider $\Core_K(Z;\Sigma)$ as a subspace of $\Real_K(Z;\Sigma)$. In other words, cores of $\Real_K(Z;\Sigma)$ and $\Comp_K(Z;\Sigma)$ are so identified.
\end{remark}

\subsubsection{Bilipschitz and quasiconformal stability}
In the following definitions, the metric on the realization $|\Real_{\Refine^{\nu k}(K)}(Z_k; \Sigma)|$, for example, is the lifted  standard metric  induced by the flat metric $d_K$ on $K$; see Definition \ref{def:standard-metric} and  Convention \ref{convention:realization-metric}. 

\begin{definition}
\label{def:bilipschitz-expanding}
Let $K$ be a good cubical $n$-complex, let $Z\subset K$ be a separating complex, and $L\ge 1$. We say that a separating complex $Z'\subset \Refine^\nu(K)$ is an \emph{$L$-bilipschitz perturbation of $Z$}, if $Z'$ is core-expanding with respect to $Z$ and, for each boundary component $\Sigma \subset \partial K$, there exists an $L$-bilipschitz homeomorphism  \index{separating complex!bilipschitz perturbation}
\[
|\Real_{K}(Z; \Sigma)| \to |\Real_{\Refine^\nu(K)}(Z'; \Sigma)|,
\]
which is the identity on $\Core_K(Z; \Sigma)$. We say that  a sequence $(Z_k)$ of separating complexes in $(\Refine^{\nu k}(K))$ is an \emph{$L$-bilipschitz perturbation sequence} if, for each $k\in \N$, $Z_{k+1}$ is an $L$-bilipschitz perturbation of $Z_k$. 
\end{definition}

\begin{definition}\index{separating complex!quasiconformally stable}
\label{def:qc-stable}
Let $K$ a good cubical $n$-complex and $\sfK \ge 1$. We say that a core-expanding sequence $(Z_k)$ is \emph{$\sfK$-quasiconformally stable} if, for each boundary component $\Sigma \subset \partial K$ and each $k \geq 1$, there exists a $\sfK$-quasiconformal homeomorphism 
\[
 |\Real_K(Z_; \Sigma)| \to |\Real_{\Refine^{\nu k}(K)}(Z_{k}; \Sigma)|, 
\]
which is the identity on $|\Core_K(Z; \Sigma)|$.
\end{definition}

The main theorem of this part is the following statement.

\begin{theorem}[Evolution of Separating Complexes]\index{Evolution theorem of separating complexes}
\label{theorem:evolution-short} 
Let $n\geq 3, m\geq 2$, and let $K$ be a good cubical $n$-complex with $m$ boundary components and a flat metric $d_{\sF_K}$; let also $\nu=\nu(K)\ge 1$ be the refinement scale of $K$. Let $Z_0$ be a separating complex of $K$. Then there exist constants $\lambda=\lambda(n)\ge 1$, $L=L(K)\geq 1$ and $\sfK= \sfK(K) \geq 1$, and a sequence $(Z_k)$ of separating complexes, where $Z_k \subset \Refine^{\nu k}(K)$, such that for $k\ge 1$,
\begin{enumerate}
\item each $Z_k$ has the relative Wada property with respect to $Z_{k-1}$, \label{item:evolution-Wada}
\item each $Z_k$ is an $\lambda$-indented perturbation of $Z_{k-1}$, \label{item:evolution-tameness}
\item the sequence $(Z_k)$ is an $L$-bilipschitz perturbation sequence, and \label{item:evolution-Lipschitz}
\item the sequence $(Z_k)$ is $\sfK$-quasiconformally stable. \label{item:evolution-quasiconformal}
\end{enumerate}
\end{theorem} 

The main tension in this theorem lies in the simultaneous validity of the conditions  \eqref{item:evolution-Wada} and \eqref{item:evolution-tameness}. While condition  \eqref{item:evolution-Wada} requires all subcomplexes separated by $Z_k$ getting closer to every single $(n-1)$-cube in $Z_{k-1}$ measured in a uniform scale, condition \eqref{item:evolution-tameness} allows only a controlled perturbation from $Z_{k-1}$ to $Z_k$.

The choices made to obtain these properties, especially the construction of tunnels for \eqref{item:evolution-tameness}, yield as a by-product bilipschitz homeomorphisms, which we may use to obtain conditions \eqref{item:evolution-Lipschitz} and \eqref{item:evolution-quasiconformal}. A localized procedure is used, from second step on, to obtain a uniform control of the sizes of the tunnels, and thereby the quasiconformal stability.

We repeat here  that, for the quasiregular extension theorems in Part \ref{part:QR-extension}, a more detailed description of the construction is needed. 

We complete the proof of Theorem \ref{theorem:evolution-short} in Section \ref{sec:evolution}. The rest of this section is devoted to the discussion of the heuristics behind the construction of $Z_1$ and $Z_2$. Note that the complex $Z_0 = Z$ is given, and its existence follows from Theorem \ref{thm:separating-complex-existence}, possibly after refinement of the underlying complex. The construction of $Z_1$ is given formally in Sections \ref{sec:reservoir-canal-system} and \ref{sec:channeling}, and the construction of $Z_2$ is given in Section \ref{sec:localization}. 

\subsection{Idea of the proof}

Let $Z$ be a separating complex in $K$ and  $\Sigma_1,\ldots, \Sigma_m$ be the boundary components of $K$, that is, connected components of $\partial K$.

\subsubsection{Construction of $Z_1$ from $Z_0 = Z$}
For  $i=1,\ldots, m$, let $K_i = \Comp_K(Z;\Sigma_i)$
be the $\Sigma_i$-component separated by $Z$ and $\cL_i = \interior\,(|K_i|\setminus |Z|)$. We view the domains $\cL_i$ as lakes separated by $Z$, and sets $s_i = |K_i|\cap |Z|$ as shorelines of the lakes. Clearly, $|Z|=s_1 \cup \cdots \cup s_m$.

We rearrange these shorelines by fixing first a reservoir-canal system $\sfRC_{\Refine^\nu(K)}(Z)$ and a transformation $\Tr_{\Refine^\nu(K)}(Z)$ of $Z$ as in Proposition \ref{prop:RC-existence}. We create next openings between the components $\sfRC_{\Refine^\nu(K)}(Z)_i$ of $\sfRC_{\Refine^\nu(K)}(Z)$ and the receded subcomplexes $\Rec_{\Refine^\nu(K)}(\Comp_K(Z;\Sigma_i); Z)$ of the complementary components $\Comp_{\Refine^\nu(K)}(Z;\Sigma_i)$. This procedure produces  a new separating complex $Z_1$ in $\Refine^\nu(K)$ with new complementary components $\Comp_{\Refine^\nu(K)}(Z_1;\Sigma_i)$, which are near the corresponding complementary components of $Z$. 
This step is called \emph{channeling} in Section \ref{sec:channeling}.

\subsubsection{Construction of $Z_2$}
For the construction of the second separating complex $Z_2$ and onward, we may not repeat the same construction. Because the spanning tree of $\Gamma(Z_1)$ has graph size roughly $3^{\nu (n-1)}$ times the size of the spanning tree of $\Gamma(Z)$,  we would lose the uniformity in the perturbation condition \eqref{item:evolution-tameness} in Theorem \ref{theorem:evolution-short}.

To remedy this issue, we subdivide the separating complex $Z_1$ into subcomplexes $Z'_q$ of bounded graph size, encoded by the $(n-1)$-cubes $q\in Z^{[n-1]}$. The subcomplexes $Z'_q$ have qualitative properties similar to that of separating complexes, which allow us to repeat the channeling construction locally with respect to each $Z'_q$. Heuristically, we may view this step as perturbing  shorelines of the lakes  locally uniformly. This procedure is called \emph{localized channeling} in Section \ref{sec:localization}.

The cost of this 'localization' is three fold. First, it leads, in terms of complexes, to  multiple channelings instead of a single channeling. Second, an essential partition of the ambient space needs to be carefully defined  in order to address the compatibility of different local constructions near the boundaries of the complexes $Z'_q$. Third, it leads to an inductive process, where the construction of a complex $Z_k$ for $k\ge 2$ depends on the two previous complexes in the sequence, that is, on cubes in $Z_{k-2}$ for coding and on $Z_{k-1}$ for channeling. Nevertheless, the outcome of the local channeling is the separating complex $Z_2\subset \Refine^{2\nu}(K)$ and, after iteration,  the sequence $(Z_k)$ asserted in Theorem \ref{theorem:evolution-short}.


\chapter{Evolution of separating complexes}

The goal of this chapter is to prove the Evolution Theorem (Theorem \ref{theorem:evolution-short}).
\section{Channeling}
\label{sec:channeling}

In this section, we return to the context of a separating complex $Z$ in a cubical $n$-complex, and  make the following assumptions.

\begin{standing}\label{standing:Z}\index{Standing assumptions on $(K;Z)$}
Throughout Sections \ref{sec:channeling}--\ref{sec:evolution}, we denote
\begin{itemize}
\item by $K$ a cubical $n$-complex with $m$ boundary components, having a flat structure $\sF_K$, and admitting a separating complex,
\item by $Z$ a separating complex of $K$, 
\item by $\rho = \rho_Z \colon Z^{[n-1]}\to K^{[n]}$ a preference function on $Z$,
\item by $\cT_Z$ a spanning tree of $\Gamma(Z)$.
\end{itemize}
\end{standing}

\begin{remark}
\label{rmk:preference-cube-special}
For a complex $K=\Refine(K_0)$ and a separating complex  $Z=\mathcal Z$ of $K$ in Remark \ref{rmk:separating-complex-special}, we use a preference function $\rho=\rho_{\mathcal Z}$ for which all preference cubes $\rho(q), q\in {\mathcal Z}^{[n-1]}$, are contained in  $K- \bigcup_\Sigma C_\Sigma$, where $C_\Sigma=\Sigma \times [0,3]$ is a collar and the union is taken over all  connected component $\Sigma$ of $\partial K$.
\end{remark}

To align the complexes separated by $Z$  with the boundary components $\Sigma_1,\ldots, \Sigma_m$ of $K$,  we label the graphs $G_1,\ldots, G_m$ of the cut-graph $\Gamma(K;Z)$ (defined in Section \ref{sec:receded-complex}) so that 
\[ 
\Comp_K(Z;\Sigma_i) = \Span_K(G_i).
\]

We apply now Proposition \ref{prop:RC-existence}  to the pair $(U,Y)=(K, Z)$ and use tree $\cT_Z$ to obtain a reservoir-canal system $\sfRC_{\Refine^\nu(K)}(Z)$ and a transformation $\Tr_{\Refine^\nu(K)}(Z)$ of $Z$. We assume that $\sfRC_{\Refine^\nu(K)}(Z)$ and $\Tr_{\Refine^\nu(K)}(Z)$ are as constructed in Section \ref{sec:reservoir-canal-system}. 

In what follows, we also use other notations associated to these complexes. In particular, 
\[
\Tr_{\Refine^\nu(K)}(Z) = \Refine^\nu(Z)  \cup \left(\sfRC_{\Refine^\nu(K)}(Z)^{(n-1)} - P_{\tau_Z}\right)
\]
where  $P_{\tau_Z}$ are passages associated to the spanning trees $\tau_Z=(\tau_{Z,1},\ldots,\tau_{Z,m})$ of $\left(\Gamma(\sfRC_{\Refine^\nu(K)}(Z)_1), \ldots, \Gamma(\sfRC_{\Refine^\nu(K)}(Z)_m)\right)$.

\subsection{Channeling through openings}
\label{sec:channeling-for-Z-one}
Since we have labeled the complementary components of the separating complex $Z$ with respect to the boundary components of $K$, we change the notation in Section \ref{sec:receded-complex} slightly and denote 
\[
\Rec_{\Refine^\nu(K)}(Z;\Sigma_i)=\Refine^\nu( \Comp_K(Z;\Sigma_i)) - \sfRC_{\Refine^\nu(K)}(Z)
\] 
the \emph{receded complex of $\Refine^\nu( \Comp_K(Z;\Sigma_i))$} for $i \in \{1,\ldots, m\}$. 
\index{$\Rec_{\Refine^\nu(K)}(Z;\Sigma_i)$}

\begin{remark}
With respect to the notations in Section \ref{sec:receded-complex}, we have 
\[
\Rec_{\Refine^\nu(K)}(Z;\Sigma_i) = \Rec_{\Refine^\nu(K)}(\Comp_K(Z;\Sigma_i); Z).
\]
\end{remark}

By Corollary \ref{cor:reservoir-canal-components}, the graph $\Gamma(\Refine^\nu(K);\Tr_{\Refine^\nu(K)}(Z))$ has $2m$ connected components, which are 
\[
\Gamma(\Rec_{\Refine^\nu(K)}(Z;\Sigma_1),\ldots,  \Gamma(\Rec_{\Refine^\nu(K)}(Z;\Sigma_m)), \tau_{Z,1},\ldots,\tau_{Z,m}.
\]
We connect subgraphs $\tau_{Z,i}$ and $\Gamma(\Rec_{\Refine^\nu(K)}(Z;\Sigma_i))$ of $\Gamma(\Refine^\nu(K))$ by adding a connecting edge $\omega_{Z,i}\in {\Refine^\nu(K)}^{[n-1]}$ as follows.

For each $i$, we fix a pair of adjacent $n$-cubes $(C_{Z,i}, \Omega_{Z,i})$ where $C_{Z,i}\in  \sfRC_{\Refine^\nu(K)}(Z)_i$ is well-located and $ \Omega_{Z;i}\in\Rec_{\Refine^\nu(K)}(Z;\Sigma_i)$ according to the following:
\begin{enumerate}
\item choose $\Omega_{Z,i}$ to be a cube contained in the interior of a preference cube if possible; 
\item choose $\Omega_{Z,i}$ to be a cube in $\Refine^\nu(\Comp_K(Z;\Sigma_i))$ otherwise.
\end{enumerate}
Thus,  in the first case, $|C_{Z,i}| \cup |\Omega_{Z,i}|$ is in the interior of a preference cube, and in the second case, $C_{Z,i} \cap \Omega_{Z,i}$ is an $(n-1)$-cube in $\Refine^\nu(Z)$.

\begin{remark}\label{rmk:separating-complex-special-opening}
Suppose that  $Z=\mathcal Z$   is the separating complex in an $n$-complex $K$ as stated in Remark \ref{rmk:separating-complex-special}, and $\Sigma_1=\Sigma'\subset \partial K$ is the boundary component distinguished. The preference function $\rho_\mathcal Z$ chosen in Remark \ref{rmk:preference-cube-special} yields the fact that 
 $|\Omega_{Z,1}| \cap |\Sigma_1\times [0,3]|= \emptyset$ and 
 $\Omega_{Z,i} \subset\Refine^\nu(\Sigma_i\times [0,3])$ for $ i\neq 1$.
\end{remark}

We call the intersection 
\[
\omega_{Z,i}= C_{Z,i}\,  \cap \,\Omega_{Z;i} \in \Tr_{\Refine^\nu(K)}(Z)
\]
an \emph{opening} between  $\sfRC_{\Refine^\nu(K)}(Z)_i$ and $\Rec_{\Refine^\nu(K)}(Z;\Sigma_i)$, and  view $\Omega_{Z,i}$ as a roof cube for $\sfRC_{\Refine^\nu(K)}(Z)_i$. \index{separating complex!channeling!opening} \index{$\omega_{Z,i}$} Now 
\[
G_{Z,i} = \Gamma(\Rec_{\Refine^\nu(K)}(Z;\Sigma_i)) \cup \{ \omega_{Z,i}\} \cup \tau_{Z,i}
\]
is a connected subgraph of $\Gamma(\Refine^\nu(K))$ and 
\[
\Span_{\Refine^\nu(K)}(G_{Z,i}) = \Rec_{\Refine^\nu(K)}(Z;\Sigma_i) \cup \sfRC_{\Refine^\nu(K)}(Z)_i.
\]

\begin{definition}
\label{def:Channeling}
The subcomplex  \index{separating complex!channeling} \index{$\Channel_{\Refine^\nu(K)}(Z) $}
\[
\Channel_{\Refine^\nu(K)}(Z) = \Tr_{\Refine^\nu(K)}(Z) \setminus \{\omega_{Z,1},\ldots, \omega_{Z,m}\} \subset \Refine^\nu(K)
\]
is called a \emph{channeling transformation of $Z$ in $\Refine^\nu(K)$ with respect to the tree $\cT$ and the preference function $\rho\colon Z^{[n-1]}\to K^{[n]}$.}
\end{definition}\index{channeling}

\begin{remark}The channeling transformation
\[
\Channel_{\Refine^\nu(K)}(Z) = \left( \Refine^\nu(Z) \setminus \{\omega_{Z,1},\ldots, \omega_{Z,m}\} \right) \bigcup  \left(\sfRC_{\Refine^\nu(K)}(Z)^{(n-1)} - P_{\tau_Z}\right)
\]
consists of the entire $\Refine^\nu(Z)$ with only the openings removed, plus the $(n-1)$-cubes in $\sfRC_{\Refine^\nu(K)}(Z)$ with a collection of spanning trees in $\sfRC_{\Refine^\nu(K)}(Z)$ removed. Note that two parts in the union have $(n-1)$-cubes in common.
\end{remark}

An immediate consequence of the connectedness of the graphs $G_{Z,i}$ is that the channeling transformation separates $\Refine^\nu(K)$ into exactly $m$ connected components. 

\begin{corollary}
\label{cor:G_Z} The graph  $\Gamma(\Refine^\nu(K);\Channel_{\Refine^\nu(K)}(Z))$ has $m$ connected components, $ G_{Z,1},\ldots,G_{Z,m}$, and the components of $\Refine^\nu(K)$ separated by the channeling transformation $\Channel_{\Refine^\nu(K)}(Z)$ are
\[
\Comp_{\Refine^\nu(K)}(\Channel_{\Refine^\nu(K)}(Z); \Sigma_i) = \Span_{\Refine^\nu(K)}(G_{Z,i})\quad i=1,\ldots, m.
\]
\end{corollary}

\subsection{Structure of subcomplexes separated by channeling}
\newcommand{\Tunnel}{\mathsf{Tunnel}}

As a preliminary step in showing that the channeling $\Channel_{\Refine^\nu(K)}(Z)$ is a separating complex of $\Refine^\nu(K)$, we discuss the structures of the complexes $\Comp_{\Refine^\nu(K)}(\Channel_{\Refine^\nu(K)}(Z); \Sigma_i)$ and $\Real_{\Refine^\nu(K)}(\Channel_{\Refine^\nu(K)}(Z);\Sigma_i)$.

We denote 
\[
Z_1=\Channel_{\Refine^\nu(K)}(Z)
\] 
for simplicity. By the definition of the receded complexes, we have that
\[\Rec_{\Refine^\nu(K)}(Z;\Sigma_i) = \Refine^\nu(\Comp_K(Z;\Sigma_i)) - \sfRC_{\Refine^\nu(K)}(Z)\]
and that
\[
\Comp_{\Refine^\nu(K)}(Z_1;\Sigma_i) = \Rec_{\Refine^\nu(K)}(Z;\Sigma_i) \cup \sfRC_{\Refine^\nu(K)}(Z)_i.
\]
There is no inclusion relation between the complexes $\Comp_{\Refine^\nu(K)}(Z_1;\Sigma_i)$ and $\Comp_{K}(Z;\Sigma_i)$. However, we do have that
\[
\Rec_{\Refine^\nu(K)}(Z;\Sigma_i) \subset \Refine^\nu(\Comp_K(Z;\Sigma_i)) \cap \Comp_{\Refine^\nu(K)}(Z_1;\Sigma_i).
\]

We move now to discuss the realizations, and consider the lifts 
\[
\widetilde \Rec_{\Refine^\nu(K)}(Z;\Sigma_i) = \pi_{(\Refine^\nu(K),Z_1;\Sigma_i)}^{-1}(\Rec_{\Refine^\nu(K)}(Z;\Sigma_i))
\]
and 
\[
\Tunnel_{\Refine^\nu(K)}(Z_1;\Sigma_i) = \pi_{(\Refine^\nu(K),Z_1;\Sigma_i)}^{-1}(\sfRC_{\Refine^\nu(K)}(Z)_i),
\]
of the receded complexes and the reservoir-canal systems, respectively, to  $\Real_{\Refine^\nu(K)}(Z_1;\Sigma_i)$, for $i=1,\ldots, m$.
\index{$\widetilde \Rec_{\Refine^\nu(K)}(Z;\Sigma_i)$}
\index{$\Tunnel_{\Refine^\nu(K)}(Z_1;\Sigma_i)$}

The lift $\widetilde \Rec_{\Refine^\nu(K)}(Z;\Sigma_i)$ is isomorphic to $\Rec_{\Refine^\nu(K)}(Z;\Sigma_i)$, because
$\Rec_{\Refine^\nu(K)}(Z;\Sigma_i) \subset\Comp_{\Refine^\nu(K)}(Z_1;\Sigma_i)$ and none of the $(n-1)$-cubes in $Z_1$ is the common face of two adjacent $n$-cubes in $\Rec_{\Refine^\nu(K)}(Z;\Sigma_i)$.

Complex $\Tunnel_{\Refine^\nu(K)}(Z_1;\Sigma_i)$ is a tunnel, because it is the realization of a tree $\tau_{Z,i} = \Gamma(\sfRC_{\Refine^\nu(K)}(Z)_i;\Channel_{\Refine^\nu(K)}(Z))$.

Let now 
\[\widetilde \omega_{1;Z;i}=\pi_{(\Refine^\nu(K),Z_1;\Sigma_i)}^{-1}(\omega_{Z,i}), \quad i=1,\ldots, m, \] 
be the lifts of the openings $\omega_{Z,i}$ to $\Real_{\Refine^\nu(K)}(Z_1;\Sigma_i)$.  From the construction, we have
\[
\widetilde \omega_{1;Z;i} = \widetilde \Rec_{\Refine^\nu(K)}(Z;\Sigma_i) \cap \Tunnel_{\Refine^\nu(K)}(Z_1;\Sigma_i).
\]
Thus, the realization $\Real_{\Refine^\nu(K)}(Z_1;\Sigma_i)$ is obtained by attaching tunnels $\Tunnel_{\Refine^\nu(K)}(Z_1;\Sigma_i)$ to the receded subcomplex $\widetilde \Rec_{\Refine^\nu(K)}(Z;\Sigma_i)$ at the $(n-1)$-cube $\widetilde \omega_{1;Z;i}$.

We record these structural properties of the realizations in a proposition.

\begin{proposition}
\label{prop:Channel-realization-structure}
Let $Z_1 = \Channel_{\Refine^\nu(K)}(Z)$. Then, for each $i=1,\ldots, m$,
\[
\Real_{\Refine^\nu(K)}(Z_1;\Sigma_i)=\widetilde \Rec_{\Refine^\nu(K)}(Z;\Sigma_i) \cup \Tunnel_{\Refine^\nu(K)}(Z_1;\Sigma_i).
\]
Moreover,
\begin{enumerate}
\item $\widetilde \Rec_{\Refine^\nu(K)}(Z;\Sigma_i)$ is isomorphic to $\Rec_{\Refine^\nu(K)}(Z;\Sigma_i)$, 
\item $\Tunnel_{\Refine^\nu(K)}(Z_1;\Sigma_i)$ is a tunnel, and 
\item $\widetilde \Rec_{\Refine^\nu(K)}(Z;\Sigma_i) \cap \Tunnel_{\Refine^\nu(K)}(Z_1;\Sigma_i)$ is the $(n-1)$-cube $\widetilde \omega_{1;Z;i}$.
\end{enumerate}
\end{proposition}

Next we exhibit the role of  receded complex $\widetilde \Rec_{\Refine^\nu(K)}(Z;\Sigma_i)$ as a mediator in perturbation. 
For this,  we denote
\[
\Dent_{\Refine^\nu(K)}(Z;\Sigma_i) = \pi_{(\Refine^\nu(K), Z_1;\Sigma_i)}^{-1}(\Refine^\nu(\Comp_K(Z;\Sigma_i)) \cap \sfRC_{\Refine^\nu(K)}(Z)),
\]
which is a subcomplex of $\Refine^\nu(\Real_K(Z;\Sigma_i))$. 

Thus, arguing as for Corollary \ref{cor:rec-bilipschitz}, we obtain that there exists a bilipschitz homeomorphism
\[
|\widetilde \Rec_{\Refine^\nu(K)}(Z;\Sigma_i)| \to |\Real_K(Z;\Sigma_i)|.
\]
Since $\Tunnel_{\Refine^\nu(K)}(Z_1;\Sigma_i)$ is a tunnel attached to 
$\widetilde \Rec_{\Refine^\nu(K)}(Z;\Sigma_i)$ at the $(n-)$-cube $\widetilde \omega_{1;Z;i}$,
by Proposition \ref{prop:tunnel-contracting}, there is also a bilipschitz homeomorphism 
\[
|\widetilde \Rec_{\Refine^\nu(K)}(Z;\Sigma_i)|\to |\Real_{\Refine^\nu(K)}(Z_1;\Sigma_i)|,
\]
which is the identity outside the unique $n$-cube in $\widetilde \Rec_{\Refine^\nu(K)}(Z;\Sigma_i)$  having a face $\widetilde \omega_{1;Z;i}$. 

By composing these two maps, we have the following.

\begin{lemma}
\label{lemma:Z_1-bilipschitz-expanding}
There exist a constant $L=L(n,K)\ge 1$ and, for each boundary component $\Sigma_i$ of $K$,  an $L$-bilipschitz homeomorphism
\[
|\Real_{\Refine^\nu(K)}(Z_1;\Sigma_i)| \to |\Real_{K}(Z;\Sigma_i)|,
\]
which is the identity in the complement of  $\Wedge_{\Refine^\nu(K)}(\Dent_{\Refine^\nu(K)}(Z;\Sigma_i)) \cap |\Real_{\Refine^\nu(K)}(Z_1;\Sigma_i)|$, 
in particular, it is identity on $\Core_K(Z;\Sigma_i)$. 
\end{lemma}

In Section \ref{sec:evolution} we will construct a homeomorphism $|\Real_{\Refine^\nu(K)}(Z_1;\Sigma_i)| \to |\Real_K(Z;\Sigma_i)|$ with additional geometric properties, which is suitable for iterations. 

We now verify that $Z_1=\Channel_{\Refine^\nu(K)}(Z)$ is a separating complex of $\Refine^\nu(K)$.

\begin{proposition}
\label{prop:Channel-separating-complex}
The channeling transformation $\Channel_{\Refine^\nu(K)}(Z)$ of the separating complex $Z\subset K$ is a separating complex of $\Refine^\nu(K)$.
\end{proposition}

\begin{proof}
We verify the conditions in Definition \ref{def:separating-complex}. Since $\Channel_{\Refine^\nu(K)}(Z)$ does not enter $\Rec_{\Refine^\nu(K)}(Z;\Sigma_i)$ and $Z$ is a separating complex in $K$, we have that $\Channel_{\Refine^\nu(K)}(Z) \cap \Star_{\Refine^\nu(K)}(\partial \Refine^\nu(K))=\emptyset$. Thus condition \eqref{item:separating-no-boundary} holds. 

We observe, by Lemma \ref{lemma:Y_sfRC-connected}, that the graph $\Gamma(\Tr_{\Refine^\nu(K)}(Z))$ is connected. The removal of mutually disjoint $(n-1)$-cubes $\omega_{Z,1},\ldots, \omega_{Z,m}$ does not change the connectedness. 
Thus $\Gamma\left(\Channel_{\Refine^\nu(K)}(Z)\right)$ is connected, and 
condition  \eqref{item:separating-strongly-connected} holds.

It remains to check  \eqref{item:separating-mu-1} and \eqref{item:separating-total}. Let  $G$ be a connected component of $ \Gamma(\Refine^\nu(K);Z_1)$. By Corollary \ref{cor:G_Z}, $G = G_{Z,i}$ for some $i\in \{1,\ldots, m\}$. 
Thus,   $\Comp(G) = \Comp_{\Refine^\nu(K)}(Z_1;\Sigma_i)$ and $\Real(G) = \Real_{\Refine^\nu(K)}(Z_1;\Sigma_i)$. Since $K= \bigcup_G\Span_{\Refine^\nu(K)}(G)$,  condition \eqref{item:separating-total} holds true. 

From Lemma \ref{lemma:Z_1-bilipschitz-expanding}, the space  $|\Real_{\Refine^\nu(K)}(Z_1;\Sigma_i)|$ is bilipschitz homeomorphic to $|\Real_K(Z;\Sigma_i)|$.
Since $Z$ is a separating complex, $|\Real_K(Z;\Sigma_i)|$ is  homeomorphic to $|\Sigma_i|\times [0,1]$. condition \eqref{item:separating-mu-1} follows.
\end{proof} 

\subsection{Channeling as an evolution step}
We now collect the facts regarding  $\Channel_{\Refine^\nu(K)}(Z)$ as a perturbation of the separating complex $Z$. 

The first  is the relative Wada property.

\begin{proposition}
\label{prop:Channel-Wada}
The separating complex $\Channel_{\Refine^\nu(K)}(Z)$ has the relative Wada property  with respect to $Z$.
\end{proposition}
\begin{proof}
 Let $\rho$ be the preference function and $q\in Z^{[n-1]}$. Then there exists a reservoir contained in $\rho(Q)$, that is, $|\bigcup_{i=1}^m \sfR_{q,i}| \subset \rho(q)$. Thus, for each $i=1,\ldots, m$, we have that $\left( \Refine^\nu(\rho(q)) \cap \sfRC_{\Refine^\nu(K)}(Z)_i\right)^{[n]}\ne \emptyset$. We conclude that $\Channel_{\Refine^\nu(K)}(Z)$ has the relative Wada property with respect to $Z$. 
\end{proof}

The second property is that $\Channel_K(Z)$ is an indented perturbation of $Z$. 

\begin{proposition}
\label{prop:Channel-tame-perturbation}
The separating complex $Z_1=\Channel_{\Refine^\nu(K)}(Z)$ of $\Refine^\nu(K)$ is an $\lambda$-indented perturbation  of $Z$, where $\lambda = \lambda(n,K) \ge 1$.
\end{proposition}

\begin{proof}
Take $T=\sfRC_{\Refine^\nu(K)}(Z)$ and $T_{\Sigma_i}=\sfRC_{\Refine^\nu(K)}(Z)_i$ in Definition \ref{def:indented-perturbation}. The indentation property \eqref{item:tame-perturbation-2} is proved in Corollary \ref{cor:reservoir-canals-as-bent-indentations}, and the fact that complexes $\Tunnel_{\Refine^\nu(K)}(Z_1;\Sigma_i), i=1,\ldots,m,$ are tunnels follows from Proposition \ref{prop:Channel-realization-structure}. The lengths of these tunnels depend only on the size of the tree $\cT_Z$, hence on the refinement scale $\nu=\nu(K)$. Thus $\lambda$ depends only on $n$ and $K$.
\end{proof}

Thirdly, by Lemma \ref{lemma:Z_1-bilipschitz-expanding}, there exists $L=L(n,K)\ge 1$ for which the separating complex $\Channel_{\Refine^\nu(K)}(Z)$ is an $L$-bilipschitz perturbation of $Z$, as in Definition \ref{def:bilipschitz-expanding}.

Finally, we show that $Z_1=\Channel_{\Refine^\nu(K)}(Z)$ is core-expanding with respect to $Z$; recall  Definitions \ref{def:core} and \ref{def:core-expanding}.

\begin{lemma}
\label{lemma:Z_1-core-expanding}
The separating complex $Z_1$ is core-expanding with respect to $Z$, that is, for each $i=1,\ldots, m$, we have that
\[
| \Core_K(Z;\Sigma_i)|\subset |\Core_{\Refine^{\nu}(K)}(Z_1;\Sigma_i)|.
\]
Moreover, each $n$-cube in $P = \Core_{\Refine^{\nu}(K)}(Z_1;\Sigma_i) - \Refine^\nu(\Core_{K}(Z;\Sigma_i))$ may be connected, by a chain of linearly adjacently connected $n$-cubes in $P$ of graph length at most $3^{n \nu}$, to an $n$-cube in $P$ which has  a face in $\Refine^{\nu} \left( \Core_K(Z;\Sigma_i) \cap \Star_K(Z) \right)$. 
\end{lemma}

\begin{proof}
The core-expanding property follows from the inclusions 
\[
\sfRC_{\Refine^\nu(K)}(Z) \subset \Refine^\nu(\Star_K(Z))
\]
and 
\[
\Star_{\Refine^\nu(K)}(\sfRC_{\Refine^\nu(K)}(Z)) \subset \Refine^\nu(\Star_K(Z)).
\]

The difference $\Core_{\Refine^{\nu}(K)}(Z_1;\Sigma_i) - \Refine^\nu(\Core_{K}(Z;\Sigma_i))$ of the cores consists of precisely  those $n$-cubes in $\Rec_{\Refine^\nu(K)}(Z_1;\Sigma_i) \cap \Refine^\nu( \Star_K(Z))$ which do not meet $\partial( \Rec_{\Refine^\nu(K)}(Z_1;\Sigma_i))$.

Let $Q$ be an $n$-cube in
$\Core_{\Refine^{\nu}(K)}(Z_1;\Sigma_i) - \Refine^\nu(\Core_{K}(Z;\Sigma_i))$.
Let $Q'\in \Star_K(Z)^{[n]}$ be the $n$-cube for which $Q\in \Refine^\nu(Q')$ and $q'$ a face of $Q'$ in  $\Star_K(Z)\cap \Core_K(Z)$. A slight modification of the proof of Lemma \ref{lemma:receded-connected} yields a chain 
\[ Q=Q_1, Q_2, \ldots, Q_t\]
of (linearly) adjacently connected $n$-cubes 
in $\Rec_{\Refine^\nu(K)}(Z_1;\Sigma_i) \cap \Refine^\nu(Q')$ which does not meet $\partial( \Rec_{\Refine^\nu(K)}(Z_1;\Sigma_i))$ and for which $Q_p$ has a face in $\Refine^\nu(q')$. Clearly $t \leq 3^{n\nu}$.
\end{proof}

\subsection{Preference function for $\Channel_{\Refine^\nu(K)}(Z)$}
\label{sec:channeling-preference-function}
We finish this section on channeling by defining an induced preference function for $\Channel_{\Refine^\nu(K)}(Z)$, which may be viewed as a perturbation of the preference function $\rho \colon Z^{[n-1]}\to K^{[n]}$ in the construction of $\Channel_{\Refine^\nu(K)}(Z)$. We use this preference function in the iterative construction in the next section.

Note that the $(n-1)$-cubes $q^*$ in $\Channel_{\Refine^\nu(K)}(Z)$ belong to one of the two non-overlapping cases: either $q^*$ is contained in the reservoir-canal system $\sfRC_{\Refine^\nu(K)}(Z)$, or $q^*\in \Refine^\nu(q)$ for some $q\in Z^{[n-1]}$ and $q^*\not \in \sfRC_{\Refine^\nu(K)}(Z)^{[n-1]}$. In the first case, $q^*$ is a face of an $n$-cube in $\sfRC_{\Refine^\nu(K)}(Z)$ and this $n$-cube is in fact unique if $q^*$ belongs to the boundary of a canal stretch. In the second case, $q^*$ is a face of a unique $n$-cube in $\Refine^\nu(\rho(q))$. 
\begin{definition}
\label{def:channeling-preference-function}
\index{preference function!channeling transformation}
\index{$\rho_{\Channel}$}
A preference function \index{$\rho_{\Channel}$}
\[
\rho_{\Channel}\colon {\Channel_{\Refine^\nu(K)}(Z)}^{[n-1]}\to {\Refine^\nu(K)}^{[n]}, 
\]
is a \emph{channeling transformation of the preference function $\rho\colon Z^{[n-1]}\to K^{[n]}$} if 
\begin{enumerate}
\item $\rho_{\Channel}(q^*)\in {\sfRC_{\Refine^\nu(K)}(Z)}^{[n]}$ for $q^* \in {\sfRC_{\Refine^\nu(K)}(Z)}^{[n-1]}$, and
\item $\rho_{\Channel}(q^*) \in \Refine^\nu(\rho(q))$ for $q^* \subset q\in Z^{[n-1]}$ and $q^* \not \in {\sfRC_{\Refine^\nu(K)}(Z)}^{[n-1]}$.
\end{enumerate}
\end{definition}

The preference function $\rho_{\Channel}$ has the following  property. 
\begin{lemma}\label{lemma:preference-function-property}
Let $q^*$ and $q^{**}$ be adjacent $(n-1)$-cubes in $\Channel_{\Refine^\nu(K)}(Z)$. Then there are four (overlapping) possibilities:
\begin{enumerate}
\item $\rho_{\Channel}(q^*)=\rho_{\Channel}(q^{**})$, 
\item $\rho_{\Channel}(q^*)$ and $\rho_{\Channel}(q^{**})$ are adjacent, 
\item  $\rho_{\Channel}(q^*)\neq \rho_{\Channel}(q^{**})$ are not adjacent, but  there exists an $n$-cube $Q$ in union pre-reservoir cubes $\sfR_{\Refine^\nu(K)}(Z)$ which is adjacent to both $\rho_{\Channel}(q^*)$ and $\rho_{\Channel}(q^{**})$ and contains $q^*\cap q^{**}$, 
\item $\rho_{\Channel}(q^*)$ and $\rho_{\Channel}(q^{**})$ belong to the same connector in $\sfRC_{\Refine^\nu(K)}(Z)$.
\end{enumerate}
\end{lemma}
\begin{proof}
We examine the possible location of  $\rho_{\Channel}(q^*)$ according to Definition \ref{def:channeling-preference-function}. It belongs to either a canal section, or  a connector, or a pre-reservoir-block, or it does not belongs to the reservoir-canal system but has a face in an $(n-1)$-cube in $Z$; the same can be said for  $\rho_{\Channel}(q^{**})$. The claim follows by pairing up all possible cases.
\end{proof}

Lemma \ref{lemma:preference-function-property}
provides essential information about connectors in the second channeling step. We return to this discussion in Section \ref{sec:local-reservoir-canal}.

\section{Localized channeling}
\label{sec:localization}
\index{channeling!localized}

Given a separating complex $Z_0=Z$ of a cubical $n$-complex $K$, we have constructed by channeling a separating complex 
\[
Z_1=\Channel_{\Refine^\nu(K)}(Z)
\]
of $\Refine^\nu(K)$ which has the Wada property with respect to $Z_0$. We now construct a separating complex $Z_2$ of $\Refine^{2\nu}(K)$ from $Z_1$ using a \emph{localized} channeling transformation.
 
The reason for localization is  to ensure that all tunnels to be constructed are uniformly bounded in graph length and that the number of the tunnels, from one level to the next, increases geometrically. This is crucial for obtaining the quasiconformal stability and constructing quasiregular extensions.

The method of localization in terms of the separating complexes $Z_1$ and $Z$  yields a general procedure for constructing a sequence of separating complexes. This is discussed in more detail in Section \ref{sec:evolution}.

\begin{remark}
In this and the forthcoming section, to simplify discussion we define only reservoir-canal systems and receded complexes and not their truncated versions. Their truncated versions are defined analogously as before.
\end{remark}

\subsection{Partition of the separating complex  $Z_1$}
\label{sec:Z-localization}\index{localization!separating complex}
Let $Z$ be a separating complex in an $n$-complex,  $\rho \colon Z^{[n-1]}\to K^{[n]}$
a preference function,  and $\cT$ a spanning tree of $\Gamma(Z)$ which satisfy the Standing assumption \ref{standing:Z}. Let $q_0\in Z^{[n-1]}$ be the root of $\cT$, and  $<_Z$ the partial order for which $q_0 < q$ for all $q \in Z^{[n-1]}$. We also assume that $Z_1=\Channel_{\Refine^\nu(K)}(Z)$ is the separating complex in $\Refine^\nu(K)$ constructed in the previous section based on a reservoir-canal system $\sfRC_{\Refine^\nu(K)}(Z)$.

\subsubsection{Localization of $Z_1$ along $Z$}
We begin by dividing the complex $Z_1$ into two subcomplexes
\[
I(Z_1) = Z_1\cap \sfRC_{\Refine^\nu(K)}(Z)
\]
and
\[
II(Z_1) = \Refine^\nu(Z) - \sfRC_{\Refine^\nu(K)}(Z).
\]
Using the partition $\{\sfRC_{\Refine^\nu(K)}(Z)_q, q\in Z^{[n-1]}\}$, defined in Section  \ref{sec:RC-partition-UY}, we further subdivide complexes $I(Z_1)$ and $II(Z_1)$ into subcomplexes
\[
I_q(Z_1) =  Z_1\cap \sfRC_{\Refine^\nu(K)}(Z)_q
\]
and
\[
II_q(Z_1) = \Refine^\nu(q) - \sfRC_{\Refine^\nu(K)}(Z)_q
\]
with $q\in Z^{[n-1]}$. Clearly
\[
I(Z_1) = \bigcup_{q\in Z^{[n-1]}} I_q(Z_1) 
\quad \text{and} \quad
II(Z_1) = \bigcup_{q\in Z^{[n-1]}} II_q(Z_1).
\]
Since complexes $\sfRC_{\Refine^\nu(K)}(Z)_q, q\in Z^{[n-1]},$ do not have $(n-1)$-cubes in common, we have that $\{ I_q(Z_1)^{[n-1]} \colon q\in Z^{[n-1]}\}$ is an essential partition of $I(Z_1)^{[n-1]}$, and that $\{ II_q(Z_1)^{[n-1]} \colon q\in Z^{[n-1]}\}$ is an essential partition of $II(Z_1)^{[n-1]}$

\begin{definition} 
\label{def:localization}\index{separating complex!localization}
The subcomplex 
\[
Z_1(q) = I_q(Z_1) \cup II_q(Z_1)
\]
of $Z_1$ 
is called the \emph{localization of $Z_1$ at $q\in Z^{[n-1]}$}.\index{$Z_1(q)$}
\end{definition}
Clearly, $\{ Z_1(q)^{[n-1]} \colon q\in Z^{[n-1]}\}$ is a partition of $Z_1^{[n-1]}$. We also have that each subcomplex $Z_1(q) $ is adjacently connected.

\begin{lemma}\label{lemma:Z_sfRC_q-connectedness}
For each $q\in Z^{[n-1]}$, the graph $\Gamma(Z_1(q))$ is connected.
\end{lemma}
\begin{proof}
First we observe that each $(n-1)$-cube $q_1$ in $II_q(Z_1)$ may be connected to an $(n-1)$-cube $q_0\in  I_q(Z_1) \cap \Refine^\nu(q)$ through a chain $q_1, \ldots, q_\ell, q_0$ of adjacent $(n-1)$-cubes in which all, except cube $q_0$, are in  $II_q(Z_1)$. Since $I_q(Z_1)$ is adjacently connected (by e.g.~a localized version of the proof of Lemma \ref{lemma:Y_sfRC-connected}), $Z_1(q)$ is adjacently connected.  
\end{proof}

We fix now an ambient $n$-complex $U_{\Refine^\nu(K)}(q)$ of $Z_1(q)$ and perform a local channeling of $Z_1(q)$  in the refinement $\Refine^\nu(U_{\Refine^\nu(K)}(q))$. For each $q\in Z^{[n-1]}$, we take
\[
N_{\Refine^\nu(K)}(q) =\Span_{\Refine^\nu(K)}\left(\{\widehat Q \in \Refine^{\nu}(\rho(q)) \colon \widehat Q \cap q \neq \emptyset \} \right) - \sfRC_{\Refine^\nu(K)}(Z),
\]
and let $\{ \sfRC_{\Refine^\nu(K)}(Z)_q, q\in Z^{[n-1]}\}$ be the partition of 
$ \sfRC_{\Refine^\nu(K)}(Z)$ defined in Section \ref{sec:RC-partition-UY}.

\begin{definition}
\label{def:localized-ambient-space}
The subcomplex  \index{localization!ambient space}
\[
U_{\Refine^\nu(Z)}(q)=N_{\Refine^\nu(K)}(q) \cup \sfRC_{\Refine^\nu(K)}(Z)_q.
\]
is called the \emph{localized ambient space of $Z_1(q)$ in $\Refine^\nu(K)$}. \index{separating complex!localized ambient space}
\end{definition}

The subcomplex $U_{\Refine^\nu(K)}(q)$ is an ambient $n$-complex for $Z_1(q)$ in the following sense. The properties are immediate from the definition and we omit the proof.

\begin{lemma}
\label{lemma:Ch-in-U} 
For each $q\in Z^{[n-1]}$, the subcomplex $U_{\Refine^\nu(K)}(q)$ has the following properties:
\begin{enumerate}
\item $Z_1(q) \subset U_{\Refine^\nu(K)}(q)$, 
\item each $(n-1)$-cube in $Z_1(q)$ is a face of an $n$-cube in $U_{\Refine^\nu(K)}(q)$, and 
\item each $n$-cube in $U_{\Refine^\nu(K)}(q)$ has a face in $Z_1(q)$. 
\end{enumerate}
\end{lemma}

\begin{remark}
\label{rmk:ambient-meet}
In general, complexes $U_{\Refine^\nu(K)}(q)$,  $q\in Z^{[n-1]}$, are not mutually essentially disjoint. In fact, whenever $q$ and $q'$ are adjacent and $\rho(q)=\rho(q')$, we have that
\begin{align*}
U_{\Refine^\nu(K)}(q)^{[n]}\cap U_{\Refine^\nu(K)}(q')^{[n]} &= N_{\Refine^\nu(K)}(q)^{[n]}\cap N_{\Refine^\nu(K)}(q')^{[n]} \\
&=\{\widehat Q \in N_{\Refine^\nu(K)}(q) \colon \widehat Q \cap (q\cap q') \neq \emptyset \}.
\end{align*}
\end{remark}

\subsection{Local reservoir-canal system on $Z_1(q)$}
\label{sec:local-reservoir-canal} \index{localization!reservoir-canal system}
We construct the separating complex $Z_2$ in $\Refine^{2\nu}(K)$ from  $Z_1 = \Channel_{\Refine^\nu(K)}(Z)$ by a local channeling transformation of $Z_1(q)$ in $U_{\Refine^\nu(K)}(q)$ for  $q\in Z^{[n-1]}$.

The localized ambient complex $U_{\Refine^\nu(K)}(q)$ admits a flat structure inherited from $K$, and  $Z_1(q)$ is adjacently connected but is not a separating complex in $U_{\Refine^\nu(K)}(q)$. Before applying the rearrangement process in Section \ref{sec:reservoir-canal-system}, we need to fix a spanning tree $\cT_q$ and define on  $Z_1(q)$ a preference function $\rho_q$  satisfying the Standing assumptions \ref{standing:rho-T}.

Regarding the spanning trees, we set the following.
\begin{notation}
We fix for each $q\in Z^{[n-1]}$, a spanning tree $\cT_q$ of $\Gamma(Z_1(q))$ for which none of its edges is contained in the gates $\zeta_{q\cap q',\{q,q'\},i}$ for any  $\{q,q'\} \in\cT$ and $i\in \{1,\ldots, m\}$. 
\end{notation}

In other words, none of the edges of $\cT_q$ are the refinements of the edges of tree $\cT$ in the previous step. With respect to this tree, the preference cubes associated to an edge have the properties stated in Lemma \ref{lemma:preference-function-property}.
 
\begin{lemma}
\label{lemma:tree-size} \index{$\cT_q$}
There exists a constant $C=C(n, \nu)\geq 1$ for which the size of all spanning trees $\cT_q$, for $q\in Z^{[n-1]}$, are bounded above by $C$.
\end{lemma}
\begin{proof}
This estimate follows from the fact that the number of $(n-1)$-cubes in ${Z_1(q)}$ is determined by the refinement scale, and the number of $(n-1)$-cubes in  reservoirs, canal stretches and stars of gates.
\end{proof}

Regarding the preference functions, we take, for each $q\in Z^{[n-1]}$, the preference function $\rho_q \colon Z_1(q)^{[n-1]}\to U_{\Refine^\nu(K)}(q)^{[n]}$ to be the restriction of  $\rho_\Channel$ defined in Section \ref{sec:channeling-preference-function}, that is,
\[
\rho_q = \rho_{\Channel}|_{Z_1(q)} \colon Z_1(q)^{[n-1]} \to \Refine^\nu(K)^{[n]}.
\]
The image of $\rho_q$ is contained in $U_{\Refine^\nu(K)}(q)^{[n]}$ from the choice of the preference cubes in the definition of $\rho_{\Channel}$. \index{$\rho_q$} 

Following the procedure in Section \ref{sec:reservoir-canal-system}, with data 
\[
(U,Y, \rho,\cT)=\left(U_{\Refine^\nu(K)}(q),Z_1(q), \rho_q,\cT_q \right),
\]
we may construct a local reservoir-canal system 
\begin{align*}
&\sfRC_{\Refine^\nu(U_{\Refine^\nu(K)}(q))}(Z_1(q))\\
&\quad=\left(\sfRC_{\Refine^\nu(U_{\Refine^\nu(K)}(q))}(Z_1(q))_1 ,\ldots, \sfRC_{\Refine^\nu(U_{\Refine^\nu(K)}(q))}(Z_1(q))_m \right)
\end{align*}
in $\Refine^\nu(U_{\Refine^\nu(K)}(q))$.
 
Since $\Refine^\nu(U_{\Refine^\nu(K)}(q)) \subset \Refine^{2\nu}(K)$, there is no ambiguity in denoting 
\[
\sfRC_{\Refine^{2\nu}(K)}(Z_1(q)) = \sfRC_{\Refine^\nu(U_{\Refine^\nu(K)}(q))}(Z_1(q)), 
\]
thus 
\[
|\sfRC_{\Refine^{2\nu}(K)}(Z_1(q))|\subset |U_{\Refine^\nu(K)}(q)| \quad \text{for } q\in Z^{[n-1]}. 
\]

\begin{remark}
We emphasize that the construction of $\sfRC_{\Refine^{2\nu}(K)}(Z_1(q))$ involves three refinement stages $K$, $\Refine^\nu(K)$, and $\Refine^{2\nu}(K)$. More precisely, the system $\sfRC_{\Refine^{2\nu}(K)}(Z_1(q))$, which is a subcomplex of $\Refine^{2\nu}(K)$, is built over localizations $Z_1(q)\subset \Refine^\nu(K)$, which are indexed by $(n-1)$-cubes $q\in Z^{[n-1]} \subset K$.
\end{remark}

In view of Definition \ref{def:connector}, Lemma \ref{lemma:preference-function-property}, and the selection of the  tree $\cT_q$ above,  connectors in each reservoir-canal system $\sfRC_{\Refine^{2\nu}(K)}(Z_1(q))_i$ may be selected to have the following property.

\begin{corollary}
\label{cor:rho-Channel-regularity}
Let $q\in Z^{[n-1]}$.
Then there exists a reservoir-canal system  $\sfRC_{\Refine^{2\nu}(K)}(Z_1(q))$ having the property that, for each $i=1,\ldots, m$ and an edge $\{q^*, q^{**}\} \in \cT_q$, a connector $\sfJ_{\{q^*,q^{**}\},i} \subset \sfRC_{\Refine^{2\nu}(K)}(Z_1(q))_i$ is a partial-star consisting of either two, three or four $n$-cubes and satisfies
\[
\sfJ_{\{q^*,q^{**}\},i} \subset \Refine^\nu(\sfRC_{\Refine^\nu(K)}(Z))\, \cup \, \Refine^\nu(\rho_{\Channel}({Z_1(q)})).
\]
\end{corollary}

\begin{notation}
\label{notation:Euclidean-modules}
In these special cases, we may identify each connector $\sfJ_{\{q^*,q^{**}\},i}$ consisting of two $n$-cubes with
\[
\sfJ^2 = \left( \left([-1,0] \times [0,1] \right) \cup \left([0,1]\times [0,1]\right)\right)\times [0,1]^{n-2}.
\]
Similarly we identify each each connector $\sfJ_{\{q^*,q^{**}\},i}$ consisting of three or four $n$-cubes with 
\[
\sfJ^3 = \left( \left( [0,1] \times [0,1]\right) \cup  \left( [0,1]\times [1,2]\right) \cup \left([1,2]\times [0,1]\right)\right) \times [0,1]^{n-2}
\]
or
\[
\sfJ^4 = [0,2]^2 \times [0,1]^{n-2},
\]
respectively.
\end{notation}

As a final observation on local reservoir-canal systems, we note that, by Lemma \ref{lemma:border-of-Y} and the properties of the local preference functions $\rho_q$, local reservoir-canal systems are far from one another in graph distance.

\begin{corollary}\label{cor:RC-far}
Let $q,q'\in Z^{[n-1]}$, $q\ne q'$. Then the graph distance between subcomplexes $\sfRC_{\Refine^{2\nu}(K)}(Z_1(q))$ and $\sfRC_{\Refine^{2\nu}(K)}(Z_1(q'))$ in the adjacency graph $\Gamma(\Refine^{2\nu}(K))$ is at least $3^{\nu-2}$.
\end{corollary}

\subsection{Localized channeling transformation of $Z_1$}
\label{sec:channeling-global} \index{separating complex!localized channeling}
We fix, for each $q\in Z^{[n-1]}$, a collection of passages $P_{Z_1(q)} \subset \sfRC_{\Refine^{2\nu}(K)}(Z_1(q))^{[n-1]}$, 
as in Section \ref{sec:reservoir-canal-transformation}, and call 
\[
\Tr_{\Refine^{2\nu}(K)}(Z_1(q)) = \Refine^\nu({Z_1(q)}) \cup \left(\sfRC_{\Refine^{2\nu}(K)}(Z_1(q))^{(n-1)} -  P_{Z_1(q)}\right)
\]
a \emph{reservoir-canal transformation of $Z_1(q)$}. Here, the collection $P_{Z_1(q)}$ is composed of the edges of a graph 
$\tau_{Z_1(q)}= \tau_{Z_1(q),1}\cup \cdots \cup\tau_{Z_1(q),m}$, where $\tau_{Z_1(q),i} \subset \Gamma(\sfRC_{\Refine^{2\nu}(K)}(Z_1(q))_i)$ is a spanning tree.

Channeling is now constructed as in Section \ref{sec:channeling-for-Z-one} by creating openings between local reservoir-canal systems $\sfRC_{\Refine^{2\nu}(K)}(Z_1(q))_i$,  $i=1,\ldots, m$, and  the local receded complexes 
\begin{align*}
\Rec_{q,i} &= \left(\Refine^{\nu}(U_{\Refine^\nu(K)}(q)) \cap \Refine^\nu(\Comp_{\Refine^\nu(K)}(Z_1;\Sigma_i)) \right) \\
&\quad - \sfRC_{\Refine^{2\nu}(K)}(Z_1(q))
\end{align*}
of  $\Comp_{\Refine^\nu(K)}(Z_1;\Sigma_i)$.

For each $i\in \{1,\ldots, m\}$, we fix a pair $(C_{q,i}, \Omega_{q,i})$ of adjacent $n$-cubes, where $C_{q,i}$ is a well-located cube in  $\sfRC_{\Refine^{2\nu}(K)}(Z_1(q))_i$, and  $\Omega_{q,i}$ is a cube in $ \interior |\sfRC_{\Refine^\nu(K)}(Z)_i| \cap |\Rec_{q,i}|$. The existence of such a pair follows from the fact that $Z_1$ satisfies the relative Wada property with respect to $Z$.

\begin{remark}
\label{rmk:omega-q-i}
Since cubes $\Omega_{q,i}$ are well-located (see Definition \ref{def:well-located-cube}), it follows that
\[
|C_{q,i}| \cup |\Omega_{q,i}| \subset |\sfRC_{\Refine^\nu(K)}(Z)_i| \cap \interior |\rho(q)|.
\] 
\end{remark}

\bigskip

We call the $(n-1)$-cube
\[
\omega_{q,i}= C_{q,i}\,  \cap \,\Omega_{q,i}
\]
an \emph{opening between $\sfRC_{\Refine^{2\nu}(K)}(Z_1(q))_i$ and $\Rec_{q,i}$}, and $\Omega_{q,i}$ a \emph{roof cube for $\sfRC_{\Refine^{2\nu}(K)}(Z_1(q))_i$ in $\Rec_{q,i}$}. By the choice of  $\omega_{q,i}$, the graph
\[
\Gamma(\Rec_{q,i})\cup \tau_{Z_1(q),i} \cup \{\omega_{q,i}\} 
\]
is connected.

The \emph{local channeling transformation of $Z_1(q)$ associated to $q$}  is
\[
\Channel_{\Refine^{2\nu}(K)}(Z_1(q)) = \Tr_{\Refine^{2\nu}(K)}(Z_1(q)) \setminus \bigcup_{i=1}^m \{\omega_{q,i}\}.
\]

The local channeling does not alter the $(n-1)$-cubes in $\Refine^\nu(Z_1(q))$ meeting  the boundary  $\Refine^\nu(Z_1(q))$. Thus it follows from Corollary \ref{cor:RC-far} that Corollary \ref{cor:boundary-of-Y_sfRC} has its counterpart in the local setting.

\begin{corollary}
\label{cor:localization-conditioning}
For $q\in Z^{[n-1]}$, complex $Z_1(q)$ and its channeling transformation $\Channel_{\Refine^{2\nu}(K)}(Z_1(q))$ have the same $(n-2)$-dimensional boundary. In fact, two transformations in $\{\Channel_{\Refine^{2\nu}(K)}(Z_1(q))\colon  q \in Z^{[n-1]}\}$ can only meet at their boundaries. 
\end{corollary}

Together the local channeling transformations form a channeling transformation of $Z_1$. 

\begin{definition}
\label{def:localized-channeling-transformation}
The $(n-1)$-subcomplex 
\[
\Channel_{\Refine^{2\nu}(K)}\left(Z_1; \fL_1 \right) = \bigcup_{q\in Z^{[n-1]}} \Channel_{\Refine^{2\nu}(K)}(Z_1(q)) \subset \Refine^{2\nu}(K)
\]
is called a \emph{localized channeling transformation of  $Z_1$ associated to the family $\fL_1 = \{Z_1(q) \colon q\in Z^{[n-1]}$\}}. 
\end{definition}

We also denote
\[
\sfRC_{\Refine^{2\nu}(K)}(Z_1;\fL_1) = \bigcup_{q\in Z^{[n-1]}} \sfRC_{\Refine^{2\nu}(K)}(Z_1(q))
\]
and 
\[
\sfRC_{\Refine^{2\nu}(K)}(Z_1;\fL_1)_i = \bigcup_{q\in Z^{[n-1]}} \sfRC_{\Refine^{2\nu}(K)}(Z_1(q))_i.
\]

The argument before Definition \ref{def:channeling-preference-function} readily yields a well-defined transformation of the preference function $\rho_1 = \rho_{\Channel}\colon {Z_1}^{[n-1]}\to {\Refine^\nu(K)}^{[n]}$, which will be used for the local channeling in the next step.

\begin{definition}
\label{def:localized-channeling-preference-function}
A preference function
\[
(\rho_1)_{\Channel} \colon {\Channel_{\Refine^{2\nu}(K)}\left(Z_1; \fL_1 \right)}^{[n-1]}\to {\Refine^{2\nu}(K)}^{[n]}
\]
is a \emph{perturbation of $\rho_1 = \rho_{\Channel}\colon {Z_1}^{[n-1]}\to {\Refine^\nu(K)}^{[n]}$} if 
\begin{enumerate}
\item $(\rho_1)_{\Channel}(q^*)\in {\sfRC_{\Refine^{2\nu}(K)}(Z_1;\fL_1)}^{[n]}$ for $q^* \in {\sfRC_{\Refine^{2\nu}(K)}(Z_1;\fL_1)}^{[n-1]}$,
\item $(\rho_1)_{\Channel}(q^*) \in \Refine^\nu(\rho_1(q))$ for $(n-1)$-cube $q^* \subset q\in Z_1^{[n-1]}$ and $q^* \not \in {\sfRC_{\Refine^\nu(K)}(Z_1;\fL_1)}^{[n-1]}$.
\end{enumerate}
\end{definition}

\subsection{Structure of realizations after localized channeling}
\label{sec:Localized-channeling-structure}

Before showing that a localized channeling transformation 
\[
Z_2 = \Channel_{\Refine^{2\nu}(K)}\left(Z_1; \fL_1 \right)
\]
of  $Z_1$ is indeed a separating complex in $\Refine^{2\nu}(K)$,
we discuss, as a preparatory step, the structure of  $\Real_{\Refine^{2\nu}(K)}(Z_2;\Sigma_i)$. 
 
Note first by construction that  
\[
\Comp_{\Refine^{2\nu}(K)}(Z_2;\Sigma_i) = \Rec_{\Refine^{2\nu}(K)}(Z_1;\Sigma_i) \cup \sfRC_{\Refine^{2\nu}(K)}(Z_1;\fL_1)_i,
\]
where 
\[
\Rec_{\Refine^{2\nu}(K)}(Z_1;\Sigma_i) =  \Refine^\nu( \Comp_{\Refine^\nu(K)}(Z_1;\Sigma_i)) - \sfRC_{\Refine^{2\nu}(K)}(Z_1;\fL_1).
\]
Denote by 
\[
\widetilde \Rec_{\Refine^{2\nu}(K)}(Z_1;\Sigma_i) = \pi_{(\Refine^{2\nu}(K),Z_2;\Sigma_i)}^{-1}(\Rec_{\Refine^{2\nu}(K)}(Z_1;\Sigma_i))\]
the lift of the receded complex, and by
\[
\Tunnel_{\Refine^{2\nu}(K)}(Z_2;\Sigma_i) =  \pi_{(\Refine^{2\nu}(K),Z_2;\Sigma_i)}^{-1}(\sfRC_{\Refine^{2\nu}(K)}(Z_1;\fL_1)_i)
\]
the lift of the reservoir-canal system, to $\Real_{\Refine^{2\nu}(K)}(Z_2;\Sigma_i)$.

The reservoir-canal system $\sfRC_{\Refine^{2\nu}(K)}(Z_1;\fL_1)$ is an essentially disjoint union of localized systems $\sfRC_{\Refine^{2\nu}(K)}(Z_1(q)), q\in Z^{[n-1]}$. 
The connected components of 
$\Gamma(\sfRC_{\Refine^{2\nu}(K)}(Z_1(q)); Z_2)$ are trees $\tau_{Z_1(q),1}, \dots, \tau_{Z_1(q),m}$. The connected components of $\Tunnel_{\Refine^{2\nu}(K)}(Z_2;\Sigma_i)$ are the realizations of these trees, hence may be denoted canonically by  $\Tunnel_{\Refine^{2\nu}(K)}(Z_2;\Sigma_i)_q, q\in Z^{[n-1]}$.

For each $q\in Z^{[n-1]}$, the size of $\tau_{Z_1(q),i}$, hence the graph size of  tunnel $\Tunnel_{\Refine^{2\nu}(K)}(Z_2;\Sigma_i)_q$, is determined by the spanning tree $\cT_q$. Tree $\cT_q$ is contained in $Z_1(q)$, hence has graph size  bounded above by a number depending only on $n$ and $\nu$.

The intersection 
\[
\Tunnel_{\Refine^{2\nu}(K)}(Z_2;\Sigma_i)_q \cap \widetilde \Rec_{\Refine^{2\nu}(K)}(Z_1;\Sigma_i)= \widetilde \omega_{2;q;i} 
\]
is an $(n-1)$-cube, which is the lift 
\[
\widetilde \omega_{2;q;i}= \pi_{(\Refine^{2\nu}(K),Z_2;\Sigma_i)}^{-1}(\omega_{q,i})
\] 
of the opening $\omega_{q,i}$ to $\Real_{\Refine^{2\nu}(K)}(Z_2;\Sigma_i)$. 

The lift $\widetilde \Rec_{\Refine^{2\nu}(K)}(Z_1;\Sigma_i)$ is adjacently connected, because it is isomorphic to $ \Rec_{\Refine^{2\nu}(K)}(Z_1;\Sigma_i)$ which, by construction, is adjacently connected. From the choices of $\omega_{q,i}, q\in Z^{[n-1]}$, it follows that 
\[
\Real_{\Refine^{2\nu}(K)}(Z_2;\Sigma_i) = \widetilde \Rec_{\Refine^{2\nu}(K)}(Z_1;\Sigma_i) \cup \left(\bigcup_{q\in Z^{[n-1]}} \Tunnel_{\Refine^{2\nu}(K)}(Z_2;\Sigma_i)_q \right)
\]
and is adjacently connected.

\begin{remark}
\label{rmk:omega-i-2}
In view of Remark \ref{rmk:omega-q-i}, the $(n-1)$-cube $\widetilde \omega_{2;q;i}$ is a face of an $n$-cube $\widetilde \Omega_{2;q;i} \in \Refine^\nu(\Tunnel_{\Refine^\nu(K)}(Z_1;\Sigma_i))$. Heuristically, we may say that the tunnel $\Tunnel_{\Refine^{2\nu}(K)}(Z_2;\Sigma_i)_q$ is attached to the tunnel part of $\Real_{\Refine^\nu(K)}(Z_1;\Sigma_i)$.
\end{remark}

We summarize the discussions above in a proposition.  

\begin{proposition} 
\label{prop:Localized-realization-structure}
For each $i=1,\ldots, m$, 
\[
\Real_{\Refine^{2\nu}(K)}(Z_2;\Sigma_i) = \widetilde \Rec_{\Refine^{2\nu}(K)}(Z_1;\Sigma_i) \cup \Tunnel_{\Refine^{2\nu}(K)}(Z_2;\Sigma_i),
\]
where 
\begin{enumerate}
\item $\widetilde \Rec_{\Refine^{2\nu}(K)}(Z_1;\Sigma_i)$ is isomorphic to $\Rec_{\Refine^{2\nu}(K)}(Z_1;\Sigma_i)$, 
\item complex $\Tunnel_{\Refine^{2\nu}(K)}(Z_2;\Sigma_i)$ is composed of mutually disjoint tunnels  $\Tunnel_{\Refine^{2\nu}(K)}(Z_2;\Sigma_i)_q$,  $q\in Z^{[n-1]}$, where 
\[ \Tunnel_{\Refine^{2\nu}(K)}(Z_2;\Sigma_i)_q \cap \widetilde \Rec_{\Refine^{2\nu}(K)}(Z_1;\Sigma_i)  =\widetilde \omega_{2;q;i}\]
is an $(n-1)$-cube, and
\item the graph size of each tunnel  $ \Tunnel_{\Refine^{2\nu}(K)}(Z_2;\Sigma_i)_q, q\in Z^{[n-1]},$
is bounded by a constant $ \lambda_\loc$ depending only on $n$ and $\nu$. \index{$\lambda_\loc$}
\end{enumerate}
\end{proposition}

\index{separating complex!localized channeling!evolution}
We now show that $Z_2=\Channel_{\Refine^{2\nu}(K)}(Z_1;\fL_1)$ is a separating complex in $\Refine^{2\nu}(K)$.
\begin{proposition}
\label{prop:localization-separating-complex}
The complex $Z_2 = \Channel_{\Refine^{2\nu}(K)}(Z_1;\fL_1)$ is a separating complex of $\Refine^{2\nu}(K)$.
\end{proposition}

\begin{proof}
We check the conditions in Definition \ref{def:separating-complex} following the argument for Proposition \ref{prop:Channel-separating-complex}. Condition \eqref{item:separating-no-boundary} holds because $Z_2$ lies in the union of complexes $U_{\Refine^\nu(K)}(q)$ which do not meet the boundary of $\Refine^{2\nu}(K)$. Condition \eqref{item:separating-strongly-connected}, $\Channel_{\Refine^{2\nu}(K)}(Z_1;\fL_1)$ is adjacently connected, holds because $\Gamma(Z_1)$ is connected and each $\Channel_{\Refine^\nu(K)}(Z_1(q))$ is adjacently connected.

Since realizations $\Real_{\Refine^{2\nu}(K)}(Z_2;\Sigma_i)$ project onto $ \Comp_{\Refine^{2\nu}(K)}(Z_2;\Sigma_i)$, we conclude that 
$K=\bigcup_{i=1}^m \Comp_{\Refine^{2\nu}(K)}(Z_2;\Sigma_i)$  from the discussion  before Remark \ref{rmk:omega-i-2}. Thus condition \eqref{item:separating-total} holds.

By Proposition \ref{prop:tunnel-contracting} and an argument similar to that for Lemma \ref{lemma:Z_1-bilipschitz-expanding}, there exists a composition of two bilipschitz homeomorphisms 
\[|\Real_{\Refine^{2\nu}(K)}(Z_2;\Sigma_i)| \to  |\widetilde \Rec_{\Refine^{2\nu}(K)}(Z_1;\Sigma_i)| \to |\Real_{\Refine^\nu(K)}(Z_1;\Sigma_i)|.\] 
Since $Z_1$ is a separating complex in $\Refine^\nu(K)$, $|\Real_{\Refine^\nu(K)}(Z_1;\Sigma_i)|$ is homeomorphic to $|\partial \Sigma_i|\times [0,1]$. Hence condition \eqref{item:separating-mu-1} holds.
\end{proof}

The relative Wada  property for the separating complex $\Channel_{\Refine^{2\nu}(K)}(Z_1;\fL_1)$ follows immediately from the construction. 

\begin{corollary}
\label{cor:Z_2-Wada-properties}
The separating complex $Z_2$ has the relative Wada property with respect to $Z_1$. 
\end{corollary}

The separating complex $Z_2$ is core-expanding with respect to $Z_1$; the proof this fact follows from that of Lemma \ref{lemma:Z_1-core-expanding} almost verbatim.

\begin{corollary}
\label{cor:Z_2-core-expanding}
The separating complex $Z_2$ is core-expanding with respect to $Z_1$, in fact, 
\[
|\Core_{\Refine^\nu(K)}(Z_1;\Sigma_i)| \subset |\Core_{\Refine^{2\nu}(K)}(Z_2;\Sigma_i)|
\]
for each $i=1,\ldots, m$. 
Moreover, each $n$-cube in 
$P_2 = \Core_{\Refine^{2\nu}(K)}(Z_2;\Sigma_i) - \Refine^\nu(\Core_{\Refine^\nu(K)}(Z_1;\Sigma_i))$ may be connected, 
by a chain of linearly adjacently connected $n$-cubes in $P_2$ of graph length at most $3^{n \nu}$, to an $n$-cube in $P_2$ which has  a face in 
$\Refine^{2\nu} \left( \Core_{\Refine^\nu(K)}(Z_1;\Sigma_i) \cap \Star_{\Refine^\nu(K)}(Z_1) \right)$. 
\end{corollary}

The fact that $Z_2$ is an indented perturbation is built into the construction of reservoir-canal systems and Proposition \ref{prop:Localized-realization-structure}. We record this as follows.  
 
\begin{corollary}
\label{cor:Z_2-indented-perturbation}
The separating complex $Z_2$ is a $\lambda_\loc$-indented perturbation of $Z_1$, where $\lambda_\loc$ is the constant in Proposition \ref{prop:Localized-realization-structure}.
\end{corollary}

\subsection{Lifts of components of $\sfRC_{\Refine^{2\nu}(K)}(Z_1(q))$}

Let $q\in Z^{[n-1]}$. Although the local reservoir-canal system $\sfRC_{\Refine^{2\nu}(K)}(Z_1(q))$ is adjacently connected, its lift in $\Real_{\Refine^{2\nu}(K)}(Z_2;\Sigma_i)$ is typically not connected. The splitting of the lift of occurs where  $Z_1(q)$ enters the interior of a connector. We describe now the three cases, which are analogous to cases in Section \ref{sec:indentation-lift}. We return to these cases in Part \ref{part:QR-extension}.

\begin{lemma}
Let $j\in\{1,\ldots, m\}$ and let $\sfC=  \sfC_{\{q^*,q^{* *}\},j}$ be a canal stretch in $\sfRC_{\Refine^{2\nu}(K)}(Z_1(q))_j$, which the union of two canal sections $C^* = \sfC_{q^*,\{q^*,q^{* *}\},j}$, $C^{**} = \sfC_{q^{* *}, \{q^*,q^{* *}\},j}$ and a connector $\sfJ = \sfJ_{\{q^*,q^{* *}\},j}$. Let $R \subset \sfC$ be the subcomplex whose adjacency graph is a connected component of the cut-graph  $\Gamma(\sfC;Z_1(q))$. Then $R$ is either
 \begin{enumerate}
\item  the entire $\sfC$,  \label{item:P_adjacent_both-2}
\item the union of one of the canal sections, $C^*$ or $C^{**}$, and one or two, $n$-cubes in the connector $\sfJ$, or \label{item:P_adjacent_one-2}
\item  a single $n$-cube in $\sfJ$  which is not adjacent to either  canal section. \label{item:P_isolated-2}
\end{enumerate}
\end{lemma}

In each case, there is a unique index $i\in\{1,\ldots, m\}$, depending on data,  for which  $\pi^{-1}(R\cap \Refine^\nu(\Comp_{\Refine^\nu(K)}(Z_1;\Sigma_i)))$ is isomorphic to $R$. 

\subsection{Combined bent indentations}
\label{sec:combined-indentation} \index{indentation!combined}

By the proof of Proposition \ref{prop:localization-separating-complex}, there exists a composition of a tunnel-contracting bilipschitz homeomorphism and iterated dent-flattening bilipschitz homeomorphism,
\[
|\Real_{\Refine^{2\nu}(K)}(Z_2;\Sigma_i)| \to |\widetilde \Rec_{\Refine^{2\nu}(K)}(Z_1;\Sigma_i)| \to |\Real_{\Refine^\nu(K)}(Z_1;\Sigma_i)|;
\]
compare to Corollary \ref{cor:rec-bilipschitz}.

The bilipschitz constant of this homeomorphism is \emph{a priori} the product of the constants of the composed mappings. Without special care,
the bilipschitz constants of   $|\Real_{\Refine^{k\nu}(K)}(Z_k;\Sigma_i)| \to |\Real_{\Refine^\nu(K)}(Z_1;\Sigma_i)|$ might increase geometrically, so might the quasiconformal distortions.

A closer inspection of the complexes reveals that there is a better way to construct the mapping which allows a uniform control of the distortion after iteration. The key notion here is the combined bent indentations.

We consider now bent indentations in $\Refine^{2\nu}(\Real_K(Z_0;\Sigma_i))$ resulted from the removal of reservoir-canal systems in the first two steps combined. For this 
we set 
\[
\sfRC_*(Z_0,Z_1) = \Refine^\nu(\sfRC_{\Refine^\nu(K)}(Z_0)) \cup \sfRC_{\Refine^{2\nu}(K)}(Z_1;\fL_1).
\]

\begin{lemma}
\label{lemma:reservoir-canal-spectral-lemma}
The spectrum of $\sfRC_*(Z_0,Z_1)$ consists of spectral cubes in $\sfRC_{\Refine^\nu(K)}(Z_0)$ and those spectral cubes in $\sfRC_{\Refine^{2\nu}(K)}(Z_1;\fL_1)$ which are not contained 
$\sfRC_{\Refine^\nu(K)}(Z_0)$. Moreover, if a spectral cube of $\sfRC_{\Refine^{2\nu}(K)}(Z_1;\fL_1)$ is not contained in $\sfRC_{\Refine^\nu(K)}(Z_0)$ then it has a face in $Z_0$.
\end{lemma}
\begin{proof}
The first claim is immediate. The second  claim follows from the choice in Definition \ref{def:channeling-preference-function} for the preference function $\rho_{\Channel} \colon Z_1^{[n-1]}$\, $ \to \Refine^\nu(K)^{[n]}$. Indeed, if an $n$-cube in $\sfRC^\tr_{\Refine^{2\nu}(K)}(Z_1;\fL_1)$ is not contained in $\Refine^\nu(\sfRC_{\Refine^\nu(K)}(Z_0))$ then it either has a face in $Z_0$ or is contained in a pre-reservoir-block in $\sfRC_{\Refine^{2\nu}(K)}(Z_1;\fL_1)$. Thus each spectral cube, which is not contained in $\sfRC^\tr_{\Refine^\nu(K)}(Z_0)$ has a face in $Z_0$.
\end{proof}

Set next
\[
\Rec(Z_0,Z_1;\Sigma_i) = \Refine^{2\nu}(\Comp_K(Z_0;\Sigma_i)) - \sfRC_*(Z_0,Z_1),
\]
and
\[
\widetilde \Rec(Z_0,Z_1;\Sigma_i) = \pi_{(\Refine^{2\nu}(K),Z_2;\Sigma_i)}^{-1}(\Rec(Z_0,Z_1;\Sigma_i))
\]
the lift of receded complex to $\Real_{\Refine^{2\nu}(K)}(Z_2;\Sigma_i)$. Observe that 
\begin{align*}
&\widetilde \Rec(Z_0,Z_1;\Sigma_i) \subset\\
&\quad \Refine^{2\nu}(\Real_K(Z_0;\Sigma_i)) \cap \Refine^{\nu}(\Real_{\Refine^{\nu}(K)}(Z_1;\Sigma_i)) \cap \Real_{\Refine^{2\nu}(K)}(Z_2;\Sigma_i).
\end{align*}
We denote 
\begin{align*}
&\Dent_*(Z_0,Z_1;\Sigma_i) \\
&\quad =  \pi_{(\Refine^{2\nu}(K),\Refine^{2\nu}(Z);\Sigma_i)}^{-1}(\Refine^{2\nu}(\Comp_{K}(Z_0;\Sigma_i))\cap \sfRC_*(Z_0,Z_1))
\end{align*}
the combined bent indentations.

\begin{corollary}
\label{cor:Iterated-tunnel-is-bent-indentation}
The complex $\Dent_*(Z_0,Z_1;\Sigma_i)$ is a bent indentation in the realization $\Refine^{2\nu}(\Real_K(Z_0;\Sigma_i))$.
\end{corollary}
\begin{proof}
By Lemma \ref{lemma:reservoir-canal-spectral-lemma}, it suffices to observe that the intersection of any three spectral cubes of $\Dent_*(Z_0,Z_1;\Sigma_i)$ is empty, which follows immediately from the construction of canal stretches over complexes $Z_1(q)$.
\end{proof}

We consider next bent indentations in tunnels. Let
\begin{align*}
&\Dent_*(Z_1;\Sigma_i) = \\
&\quad \pi_{(\Refine^{2\nu}(K),\Refine^\nu(Z_1);\Sigma_i)}^{-1}\left(\Refine^\nu(\Comp_{\Refine^\nu(K)}(Z_1;\Sigma_i)) \cap \sfRC_{\Refine^{2\nu}(K)}(Z_1)\right).
\end{align*}
Then by an argument similar to that for Corollary \ref{cor:Iterated-tunnel-is-bent-indentation}, we have that 
\[ 
\Dent_*^{\sfT}(Z_1;\Sigma_i)=  \Dent_*(Z_1;\Sigma_i) \cap  \Refine^\nu(\Tunnel_{\Refine^\nu(K)}(Z_1;\Sigma_i))
\]
is a bent indentation in the complex $\Refine^\nu(\Tunnel_{\Refine^\nu(K)}(Z_1;\Sigma_i))$.  

By construction, complex $\Real_{\Refine^{2\nu}(K)}(Z_2;\Sigma_i)$ has an essential partition,
\begin{align*}
\Real_{\Refine^{2\nu}(K)}(Z_2;\Sigma_i) &= \widetilde \Rec_{\Refine^{2\nu}(K)}(Z_1;\Sigma_i) \cup \Tunnel_{\Refine^{2\nu}(K)}(Z_2;\Sigma_i) \\
&= \widetilde \Rec(Z_0,Z_1;\Sigma_i) \cup \, \Diff \,  \cup \Tunnel_{\Refine^{2\nu}(K)}(Z_2;\Sigma_i),
\end{align*}
where the difference 
\begin{align*}
\Diff &=\widetilde \Rec_{\Refine^{2\nu}(K)}(Z_1;\Sigma_i)-\widetilde \Rec(Z_0,Z_1;\Sigma_i) \\
&= \pi_{(\Refine^{2\nu}(K),Z_2;\Sigma_i)}^{-1}(\Refine^\nu(\sfRC_{\Refine^\nu(K)}(Z_0)_i) - \sfRC_{\Refine^{2\nu}(K)}(Z_1))\\
&= \Refine^\nu(\Tunnel_{\Refine^\nu(K)}(Z_1;\Sigma_i)) - \Dent_*^\sfT(Z_1;\Sigma_i).
\end{align*}

We show that  the space $|\Real_{\Refine^{2\nu}(K)}(Z_2;\Sigma_i)|$ is bilipschitz to an expansion of $|\Real_K(Z_0;\Sigma_i)|$ with tunnels, by flattening the combined bent indentations.

\begin{proposition}
\label{prop:tree-structure-of-realization}
There exists an $L(n,K)$-bilipschitz homeomorphism 
\[
\phi \colon |\Real_{\Refine^{2\nu}(K)}(Z_2;\Sigma_i)|\to |B_i| \cup |T_i| \cup \bigcup_{q\in Z^{[n-1]}} |T_{q;i}|,
\]
where 
\begin{enumerate}
\item $B_i =\Real_K(Z_0;\Sigma_i)$, 
\item $T_i$ is a  copy of tunnel $\Tunnel_{\Refine^\nu(K)}(Z_1;\Sigma_i)$ which meets $\Refine^\nu(B_i)$ in an $(n-1)$-cube, call it  $\widehat\omega_{1;Z;i}$ 
\item each $T_{q;i},\, q\in Z^{[n-1]},$ is a copy of  tunnel $\Tunnel_{\Refine^{2\nu}(K)}(Z_2;\Sigma_i)_q$ which meets $\Refine^\nu(T_i)$ in an $(n-1)$-cube, call it  $\widehat \omega_{2;q;i}$,
\end{enumerate}
and the homeomorphism $\phi$ 
\begin{enumerate}
\item [(4)] maps $\widetilde \omega_{1;Z;i}\mapsto \widehat \omega_{1;Z;i}$, and the unique $n$-cube in $\widetilde \Rec_{\Refine^\nu(K)}(Z_0;\Sigma_i)$ having $\widetilde \omega_{1;Z;i}$ as a face isometrically onto an $n$-cube in $\Refine^\nu(B_i)$, and maps, for each $q\in Z^{[n-1]}$, $\widetilde \omega_{2; q;i} \mapsto \widehat \omega_{2; q;i}$, and the unique $n$-cube in $\Diff$ having $\widetilde \omega_{2; q;i}$ as a face onto an $n$-cube in $\Refine^\nu(T_i)$, and  
\item [(5)] is the identity in the complement of the union $W_1\cup W_2$ of 
\[
W_1= \Wedge_{\Refine^{2\nu}(\Real_K(Z_0;\Sigma_i))}(\Dent_*(Z_0,Z_1;\Sigma_i))\bigcap |\Real_{\Refine^{2\nu}(K)}(Z_2;\Sigma_i)|,
\]
and 
\[
W_2= \Wedge_{\Refine^\nu(\Tunnel_{\Refine^\nu(K)}(Z_1;\Sigma_i))}(\Dent_*^\sfT(Z_1;\Sigma_i))\bigcap |\Real_{\Refine^{2\nu}(K)}(Z_2;\Sigma_i)|.
\]
\end{enumerate}
In particular, the homeomorphism $\phi$ is identity on $|\Core_{\Refine^{2\nu}(K)}(Z_2;\Sigma_i)|$.
\end{proposition}

\begin{proof}
For the proof, we follow the decomposition of the realization $\Real_{\Refine^{2\nu}(K)}(Z_2;\Sigma_i)$ stated before the proposition. From the locations of the selected openings, we have that 
\begin{align*}
&\widetilde \Rec(Z_0,Z_1;\Sigma_i) \cap \Refine^\nu(\Tunnel_{\Refine^\nu(K)}(Z_1;\Sigma_i)) = \\
&\quad \Refine^\nu(\widetilde \Rec_{\Refine^\nu(K)}(Z_0;\Sigma_i))\cap \Refine^\nu(\Tunnel_{\Refine^\nu(K)}(Z_1;\Sigma_i)) = \Refine^\nu(\widetilde \omega_{1;Z;i}),
\end{align*}
where $\widetilde \omega_{1;Z;i}=\pi_{\Refine^\nu(K), \Refine^\nu(Z)}^{-1}(\omega_{Z,i})$.

Since 
\[
\Dent_*(Z_0,Z_1;\Sigma_i) =\Refine^{2\nu}(\Real_K(Z_0;\Sigma_i))- \widetilde \Rec(Z_0,Z_1;\Sigma_i),
\] 
there exists, as in Corollary \ref{cor:rec-bilipschitz}, an $L(n,K)$-bilipschitz map 
\[
\phi_0 \colon |\widetilde \Rec(Z_0,Z_1;\Sigma_i)|\to |\Real_K(Z_0;\Sigma_i)|,
\]
which is the identity in the complement of $W_1$, and maps the unique $n$-cube in $\widetilde \Rec_{\Refine^\nu(K)}(Z_0;\Sigma_i)$ having $\widetilde \omega_{1;Z;i}$ as a face isometrically onto an $n$-cube in $\Refine^\nu(B_i)$.

The difference $\Diff=\widetilde \Rec_{\Refine^{2\nu}(K)}(Z_1;\Sigma_i)-\widetilde \Rec(Z_0,Z_1;\Sigma_i)$, in view of Remark \ref{rmk:omega-q-i}, meets tunnels $\Tunnel_{\Refine^{2\nu}(K)}(Z_2;\Sigma_i)$, $q\in Z^{[n-1]}$, in the $(n-1)$-cubes $\widetilde \omega_{2;q;i}$.

As before, there exists an $L(n,K)$-bilipschitz map 
\[
\phi_1 \colon |\Diff| \to |\Tunnel_{\Refine^\nu(K)}(Z_1;\Sigma_i)|,
\]  
which is  identity in the complement of $W_2$, in particular  identity on $\widetilde \omega_{1;Z;i}$, and  which maps, for each $q\in Z^{[n-1]}$, the unique $n$-cube in $\Diff$ having $\widetilde \omega_{2;q;i}$ as a face isometrically onto an $n$-cube in $\Refine^\nu(\Tunnel_{\Refine^\nu(K)}(Z_1;\Sigma_i))$.
 
We now glue a copy  $T_i$ of $\Tunnel_{\Refine^\nu(K)}(Z_1;\Sigma_i)$ to $\Refine^\nu(B_i)$ by identifying the $(n-1)$-cube $\widetilde \omega_{1;Z;i}$ in $\Tunnel_{\Refine^\nu(K)}(Z_1;\Sigma_i)$ 
with $\phi_0(\widetilde \omega_{1;Z;i})$ in $|B_i|$. Similarly we 
glue, for each $q\in Z^{[n-1]}$, a copy
 $T_{q;i}$ of tunnel $\Tunnel_{\Refine^{2\nu}(K)}(Z_2,\Sigma_i)_q$ to $\Refine^\nu(T_i)$ by
identifying the $(n-1)$-cube $\widetilde \omega_{2;q;i}\in \Tunnel_{\Refine^{2\nu}(K)}(Z_2,\Sigma_i)_q$  
with  $\phi_1(\widetilde \omega_{2;q;i})$ in  $|T_i|$. This completes the proof.
\end{proof}

\section{Evolution sequence}
\label{sec:evolution}
We now construct the sequence $(Z_k)$ of separating complexes in Theorem  \ref{theorem:evolution-short}. The starting point, as before, is a cubical $n$-complex $K$ which has boundary components $\Sigma_1,\ldots, \Sigma_m$, where $m\ge 2$, and a separating complex $Z$ which satisfy  Standing assumptions \ref{standing:Z}.

\subsection{Construction of  $(Z_k)$}\label{sec:construction-evolution}Let $K_0=K$, and 
\[
K_k=\Refine^{k\nu}(K), \quad \text{for} \quad k\geq 1.
\]
The existence of $Z_1$ and $Z_2$ have been shown in Sections \ref{sec:reservoir-canal-system},  \ref{sec:channeling}, and  \ref{sec:localization}. Complex $Z_2$ is constructed by a localized channeling transformation 
\[
Z_2 = \Channel_{K_2}(Z_1; \fL_1),
\]
where $\fL_1= \{Z_1(q)\colon q \in Z_0^{[n-1]}\}.$
 
We construct the sequence $(Z_k)$ by induction. Suppose that, for $k\ge 2$, we have already defined separating complexes $Z_0, Z_1,\ldots, Z_k$ of complexes $K_0, K_1,\ldots, K_k$, respectively. Let now \index{$\fL_k$} 
\[
\fL_k = \{Z_k(q) \colon q\in {Z_{k-1}}^{[n-1]}\}
\]
be the localization of $Z_k$ as in Definition \ref{def:localization}. We take, as in Definition \ref{def:localized-channeling-transformation},
 \[
Z_{k+1} = \Channel_{K_{k+1}}(Z_k; \fL_k).
\]
 to be a localized channeling transformation of $Z_k$  associated to data:  the localization $\fL_k$, a preference function $\rho_k = (\rho_{k-1})_{\Channel} \colon Z_k^{[n-1]} \to {K_k}^{[n]}$, and  local trees $\cT_q, q\in Z_{k-1}^{[n-1]}$, whose edges are not refinements of the edges of trees in previous steps.

We denote also by  $\sfRC_{K_{k+1}}(Z_k;\fL_k)$ the reservoir-canal system and by $\omega_{k+1;q;i}, q\in Z_{k-1}^{[n-1]}$ and $i\in \{1,\ldots, m\}$, for the openings from  $\sfRC_{K_{k+1}}(Z_k;\fL_k)$ to the receded complex $\Rec_{K_{k+1}}(Z_k;\fL_k)$,  used in the localized channeling.

\subsection{Properties of $(Z_k)$}
\label{sec:summary}
We begin with the properties of the individual complexes $Z_k$.

\begin{lemma}
\label{lemma:summary}
The sequence $(Z_k)$ has the following properties:
\begin{enumerate}
\item each separating complex $Z_k$ has the relative Wada property with respect to $Z_{k-1}$, 
\item each $Z_k$ is core-expanding with respect to $Z_{k-1}$,
\item each $Z_k$ is a $\lambda$-indented perturbation of $Z_{k-1}$, and 
\item each  $Z_k$ is an $L$-bilipschitz perturbation of $Z_{k-1}$.
\end{enumerate}
where $L=L(n, K) \geq 1$ and $\lambda=\lambda(n, K)\geq1$ are constants.
\end{lemma}

For the proof, we note as before that
\[
\Comp_{K_k}(Z_k;\Sigma_i) = \Rec_{K_k}(Z_{k-1};\Sigma_i) \cup \sfRC_{K_k}(Z_{k-1};\fL_{k-1})_i,
\]
where $\Rec_{K_k}(Z_{k-1};\Sigma_i)=  \Refine^\nu(\Comp_{K_{k-1}}(Z_{k-1}; \Sigma_i)) - \sfRC_{K_k}(Z_{k-1};\fL_{k-1})$, and
\[
\Real_{K_k}(Z_k;\Sigma_i) = \widetilde \Rec_{K_k}(Z_{k-1};\Sigma_i) \cup \Tunnel_{K_k}(Z_k;\Sigma_i),
\]
where 
$\widetilde \Rec_{K_k}(Z_{k-1};\Sigma_i) = \pi_{(K_k, Z_k;\Sigma_i)}^{-1}(\Rec_{K_k}(Z_{k-1};\Sigma_i))$, and 
\[
\Tunnel_{K_k}(Z_k;\Sigma_i) =  \pi_{(K_k,Z_k;\Sigma_i)}^{-1}(\sfRC_{K_k}(Z_{k-1};\fL_{k-1})_i).
\]

\begin{proof}[Proof of Lemma \ref{lemma:summary}]
Since $Z_k$ is obtained by a localized channeling transformation of $Z_{k-1}$, the separating complex $Z_k$ satisfies the relative Wada property with respect to $Z_{k-1}$ (see Corollary \ref{cor:Z_2-Wada-properties}) and that $Z_k$ is core-expanding with respect to $Z_{k-1}$ (see Corollary \ref{cor:Z_2-core-expanding}). 

We verify now that $Z_k$ is a $\lambda_\loc$-indented perturbation of $Z_{k-1}$, where $\lambda_\loc$ is the constant in Proposition \ref{prop:Localized-realization-structure}.  Since 
\[
\Comp_{K_k}(Z_k;\Sigma_i) = \left( \Refine^\nu(\Comp_{K_{k-1}}(Z_{k-1}; \Sigma_i) - \sfRC_{K_k}(Z_{k-1})\right)\cup \sfRC_{K_k}(Z_{k-1})_i,
\]
 the first condition for indented perturbation in Definition \ref{def:indented-perturbation} holds.

Note that 
\begin{align*}
&\Refine^\nu(\Real_{K_{k-1}}(Z_{k-1};\Sigma_i))- \widetilde \Rec_{K_k}(Z_{k-1};\Sigma_i)\\
&\quad =\pi_{(K_k,\Refine^\nu(Z_{k-1});\Sigma_i)}^{-1}(\Refine^\nu(\Comp_{K_{k-1}}(Z_{k-1};\Sigma_i))\cap \sfRC_{K_k}(Z_{k-1}));
\end{align*}
see the proof of Corollary \ref{cor:reservoir-canals-as-bent-indentations}). From a local version of  Proposition \ref{prop:Channel-tame-perturbation}, if follows that $Z_k$ is a $\lambda_\loc$-indented perturbation of $Z_{k-1}$.  

Finally, since the tunnels in the complex $\Tunnel_{K_k}(Z_k;\Sigma_i)$ have uniformly bounded size, we have, by Proposition \ref{prop:tunnel-contracting}, that $Z_k$ is an $L$-bilipschitz perturbation of $Z_{k-1}$, where $L$ depends only on $n$ and $\nu$.
\end{proof}

\begin{remark}\label{rmk:core-to-core}
By an argument similar to that for Lemma \ref{lemma:Z_1-core-expanding} and Corollary \ref{cor:Z_2-core-expanding},  the  cores $\Core_{K_k}(Z_k;\Sigma_i)$ satisfy the following property. Given an $n$-cube $Q_0\in \Comp_K(Z;\Sigma_i)$ and an  $n$-cube $Q_k$ in 
\[
\Core_{K_k}(Z_k;\Sigma_i) - \Refine^\nu(\Core_{K_{k-1}}(Z_{k-1};\Sigma_i)),
\] there exists a chain $\mathcal C_k$, 
\[
Q_k=Q_{k,1}, \ldots, Q_{k,p_k}; \,\,Q_{k-1,1},\ldots, Q_{k-1,p_{k-1}};\,\,  \ldots; \,\, Q_{0,1},\ldots, Q_{0,p_0}=Q_0,
\]
of $n$-cubes such that for each $\ell\in \{0,\ldots, k\}$,
\begin{enumerate}
\item $Q_{\ell,1}, \ldots, Q_{\ell,p_\ell}$ is an linearly adjacently connected sequence of $n$-cubes in $\Core_{K_\ell}(Z_\ell;\Sigma_i) -\Refine^\nu(\Core_{K_{\ell-1}}(Z_{\ell-1};\Sigma_i))$ for which $p_0\leq C( K)$, and  $p_\ell \leq 3^{n \nu}$ for $1\leq \ell \leq k$, and
\item $Q_{\ell,p_\ell}$ is adjacent to an $n$-cube in $\Refine^\nu(Q_{\ell-1,1})$ for  $1\leq \ell \leq k$.
\end{enumerate}
\end{remark}

We move now to the quasiconformal stability of the sequence $(Z_k)$. The stability, together with Lemma \ref{lemma:summary}, completes the proof of Theorem \ref{theorem:evolution-short}.

The realization $\Real_{K_k}(Z_k;\Sigma_i)$ has a tree structure similar to that of $\Real_{K_2}(Z_2;\Sigma_i)$ in Proposition \ref{prop:Localized-realization-structure}, which will be used to obtain the quasiconformal stability of the sequence $(Z_k)$.

To set the stage, we define, for each $k\ge 0$ and $0\le \ell \le k-1$, a combined reservoir-canal system \index{$\sfRC_*(Z_\ell,\ldots, Z_k)$} 
\[
\sfRC_*(Z_\ell,\ldots, Z_{k-1}) = \bigcup_{j=\ell+1}^k \Refine^{\nu(k-j)}(\sfRC_{K_j}(Z_{j-1};\fL_{j-1})). 
\]   
Note that $\sfRC_*(Z_\ell, \ldots, Z_{k-1})$ is a subcomplex of $K_k$, and it does not meet  core $\Refine^{\nu (k-\ell)}(\Core_{K_\ell}(Z_\ell;\Sigma_i))$.

The spectrum of the combined reservoir-canal system $\sfRC_*(Z_\ell, \ldots, Z_{k-1})$ satisfies a property similar to that described in Lemma \ref{lemma:reservoir-canal-spectral-lemma}.

\begin{lemma}
\label{lemma:RC_*}
For $k\ge 2$, each spectral cube of $\sfRC_{K_k}(Z_{k-1};\fL_{k-1})$ is either contained in a spectral cube of $\sfRC_*(Z_0,\ldots, Z_{k-2})$ or has a face in $Z_0$. 
\end{lemma}
\begin{proof}
By Lemma \ref{lemma:reservoir-canal-spectral-lemma} the claim holds for $k=2$. Suppose that $k\geq 3$.
Let $Q$ be a spectral cube in $\sfRC_{K_k}(Z_{k-1};\fL_{k-1})$. Applying the second statement in Lemma \ref{lemma:reservoir-canal-spectral-lemma} to complexes $Z_{k-1}$ and $Z_{k-2}$, we have that  $Q$ is either contained in $\sfRC_{K_{k-1}}(Z_{k-2};\fL_{k-2})$, or has a face in $Z_{k-2}$. In the first case, $Q$ is contained in $\sfRC_*(Z_0,\ldots, Z_{k-2})$.

In the second case, $Q$ is contained in the preference cube $\rho_{k-2}(q')$ of some $q'\in Z_{k-2}^{[n-1]}$. Since preference function $\rho_{k-2}=(\rho_{k-3})_\Channel$ is a channeling of $\rho_{n-3}$, the $n$-cube $\rho_{k-2}(q')$ is either contained in $\sfRC_{K_{k-2}}(Z_{k-3};\fL_{k-3})$ or has a face in $Z_{k-2} \cap \Refine^\nu(Z_{k-3})$. Thus $Q$ is either contained in $\sfRC_*(Z_0,\ldots, Z_{k-2})$ or  has a face in $Z_{k-3}$.

Continuing this line of argument, we conclude that $Q$ is either contained in $\sfRC_*(Z_0,\ldots, Z_{k-2})$ or has a face in $Z_0$.
\end{proof}

Let now for $k\geq 2$ and $\ell \in \{0,\ldots, k-1 \}$,
\begin{align*}
&\Dent_*(Z_\ell, \ldots, Z_{k-1};\Sigma_i) \\
&\quad = \pi_{(K_k, Z_k;\Sigma_i)}^{-1}\left( \sfRC_*(Z_\ell,\ldots, Z_{k-1})  \cap \Refine^{\nu(k-\ell)}(\Comp_{K_\ell}(Z_\ell; \Sigma_i)) \right).
\end{align*}
  \index{$\Dent_*(Z_\ell,\ldots, Z_k)$}
Let also  for $\ell \ge 1$,
\[
\Tunnel_{K_\ell}(Z_\ell;\Sigma_i) =  \pi_{(K_\ell,Z_\ell;\Sigma_i)}^{-1}(\sfRC_{K_\ell}(Z_{\ell-1})_i),
\]
and 
\begin{align*}
&\Dent_*^\sfT (Z_\ell, \ldots, Z_{k-1};\Sigma_i) \\
&\quad = \Dent_*(Z_\ell, \ldots, Z_{k-1};\Sigma_i) \cap \Refine^{\nu(k-\ell)}\Tunnel_{K_\ell}(Z_\ell;\Sigma_i).
\end{align*}

\begin{lemma}
\label{lemma:Dent_*}
Let $k \ge 2$. The complex $\Dent_*(Z_0,\ldots, Z_{k-1})$ is a bent indentation in $\Refine^{k \nu}(\Real_{K}(Z;\Sigma_i))$, and 
$\Dent_*^\sfT(Z_\ell,\ldots, Z_{k-1})$ is a bent indentation in $\Refine^{\nu(k-\ell)}(\Tunnel_{K_\ell}(Z_\ell;\Sigma_i))$ for $\ell =1,\ldots, k-1$.
\end{lemma}
\begin{proof}
Since the intersection of three distinct spectral cubes
is empty, we have, by Lemma \ref{lemma:RC_*}, that  $\Dent_*(Z_0,\ldots, Z_{k-1})$ is a bent indentation in $\Refine^{k \nu}(\Real_{K}(Z;\Sigma_i))$.  The proof for $\ell\geq 1$ is similar. 
\end{proof}

We now  state a version of Proposition \ref{prop:Localized-realization-structure} for a fixed $k\ge 2$. We first fix some notations. For $\ell \in\{0, \ldots, k-1\}$, let
\[
\Rec_*(Z_\ell,\ldots, Z_{k-1};\Sigma_i) = \Refine^{\nu (k-\ell)}(\Comp_{K_\ell}(Z_\ell;\Sigma_i))-\sfRC_*(Z_\ell,\ldots, Z_{k-1}),
\]
and 
\[
\widetilde \Rec_*(Z_\ell, \ldots, Z_{k-1};\Sigma_i) = \pi_{(K_k;Z_k;\Sigma_i)}^{-1}(\Rec_*(Z_\ell,\ldots, Z_{k-1};\Sigma_i)).
\]
Let also for $1\leq \ell \leq k-1$, \index{$\Diff_*(Z_\ell,\ldots, Z_{k-1};\Sigma_i)$}
\begin{align*}
&\Diff_*(Z_\ell, \ldots, Z_{k-1};\Sigma_i) \\
&\quad = \widetilde \Rec_*(Z_\ell,\ldots, Z_{k-1};\Sigma_i) - \widetilde \Rec_*(Z_{\ell-1},\ldots, Z_{k-1};\Sigma_i) \\
&\quad = \Refine^{\nu(k-\ell)}(\Tunnel_{K_\ell}(Z_\ell;\Sigma_i))- \Dent_*^\sfT(Z_\ell,\ldots, Z_{k-1};\Sigma_i),
\end{align*}   
where  $\Tunnel_{K_1}(Z_1;\Sigma_i)$ is a tunnel of graph size at most $\#(Z^{[n-1]})$, and
for each $2\leq \ell \leq k-1$, $\Tunnel_{K_\ell}(Z_\ell;\Sigma_i)$ is a collection of tunnels indexed by $Z_{\ell-2}^{[n-1]}$, in which each tunnel has graph size at most $\lambda_\loc$, where $\lambda_\loc$ is the constant in Proposition \ref{prop:Localized-realization-structure}.

\begin{proposition}
\label{prop:general-localized-realization-structure}
Let $k\ge 2$ and  $i=1,\ldots, m$. The realization $ \Real_{K_k}(Z_k;\Sigma_i)$ has an essential partition
\begin{align*}
\Real_{K_k}(Z_k;\Sigma_i) &= \widetilde \Rec_*(Z_0,\ldots, Z_{k-1};\Sigma_i) \\
&\cup \left(\bigcup_{\ell=1}^{k-1} \Diff_*(Z_\ell,\ldots, Z_{k-1};\Sigma_i)\right) 
\cup \Tunnel_{K_k}(Z_k;\Sigma_i),
\end{align*}
where 
\begin{enumerate}
\item $\widetilde \Rec_*(Z_0,\ldots, Z_{k-1};\Sigma_i)$ is isomorphic to $\Rec_*(Z_0,\ldots, Z_{k-1};\Sigma_i)$; 
\item $\Diff_*(Z_1,\ldots, Z_k;\Sigma_i)$  is a dented refined tunnel and
for  $2\leq \ell \leq k-1$, $\Diff_*(Z_\ell,\ldots, Z_k;\Sigma_i)$ is a collection of mutually disjoint complexes indexed by $Z_{\ell-2}^{[n-1
]}$, each of which is a dented refined tunnel, moreover, each tunnel, before refined, has graph size at most $\lambda_\loc$;
\item $\Tunnel_{K_k}(Z_k;\Sigma_i)$ is a collection of mutually disjoint tunnels
indexed by $Z_{k-2}^{[n-1]}$, in which each tunnel has graph size at most $\lambda_\loc$.
\end{enumerate}
\end{proposition}

\begin{proof}
The claim holds for $k=2$. Suppose that $k\ge 3$ and the claim holds for $k-1$. The realization $\Real_{K_k}(Z_k;\Sigma_i)$ is obtained from $\Refine^\nu(\Real_{K_{k-1}}(Z_{k-1};\Sigma_i))$ by removing 
$\pi_{(K_k,Z_k;\Sigma_i)}^{-1}(\sfRC_{K_k}(Z_{k-1})_i)$ and then adding  $\Tunnel_{K_k}(Z_{k-1};\Sigma_i)$. Since
\begin{align*}
&\Dent_*(Z_\ell,\ldots, Z_{k-1};\Sigma_i) \\
&\quad = \Dent_*(Z_\ell,\ldots,Z_{k-2};\Sigma_i) \\
&\qquad \cup  \pi_{(K_\ell,\Refine^{\nu(k-\ell)}(Z_\ell);\Sigma_i)}^{-1}(\Refine^{\nu(k-\ell)}(\Comp_{K_{\ell}}(Z_{\ell};\Sigma_i))\cap \sfRC_{K_k}(Z_{k-1})_i),
\end{align*}
the claim about the structure follows from the induction assumption and the basic properties of the localized channeling construction.

The claim about the size of tunnels follows from the same reasoning as in Proposition \ref{prop:Localized-realization-structure}.
\end{proof}

\begin{remark}\label{omega-labels}
For future reference to the tree structure, we label 
the intersection of two consecutive complexes in the sequence 
\begin{align*}
&\widetilde \Rec_*(Z_0,\ldots, Z_{k-1};\Sigma_i),\,\, \Diff_*(Z_1,\ldots, Z_{k-1};\Sigma_i),\ldots \\
& \Diff_*(Z_{k-1};\Sigma_i),\,\, \Tunnel_{K_k}(Z_k;\Sigma_i)
\end{align*}
as follows. The intersection $\widetilde \Rec_*(Z_0,\ldots, Z_{k-1};\Sigma_i) \bigcap \Diff_*(Z_1,\ldots, Z_{k-1};\Sigma_i)$ is  a refinement of the $(n-1)$-cube $\widetilde \omega_{1;Z:i}$.

For  $\ell\in \{2,\ldots, k-1\}$, the intersection of $ \Diff_*(Z_{\ell-1},\ldots, Z_{k-1};\Sigma_i)$ and $ \Diff_*(Z_\ell,\ldots, Z_{k-1};\Sigma_i)$ consists of  refinements of lifts $\widetilde \omega_{\ell;q;i}$ of the openings $\omega_{\ell;q;i}$ to  $\partial (\widetilde \Rec_{K_\ell}(Z_{\ell-1};\Sigma_i))$
indexed by $q\in Z_{\ell-2}$.

Finally, the intersection  $ \Diff_*(Z_{k-1};\Sigma_i) \bigcap \Tunnel_{K_k}(Z_k;\Sigma_i)$ consists of refinements of  lifts $\widetilde \omega_{k;q;i}$ of the openings $\omega_{k;q;i}$ to $ \partial (\widetilde \Rec_{K_k}(Z_{k-1};\Sigma_i))$ indexed by $ q\in Z_{k-2}$. 
\end{remark}

\begin{remark}
We observe that, for $1\leq k'\leq k$,
\[\Core_{K_{k'}}(Z_{k'};\Sigma_i) \subset \widetilde \Rec_*(Z_0,\ldots, Z_{k-1};\Sigma_i) 
\cup \bigcup_{\ell=1}^{k'-1} \Diff_*(Z_\ell,\ldots, Z_{k-1};\Sigma_i).\]
\end{remark}

Having Proposition \ref{prop:general-localized-realization-structure} at our disposal, we may flatten the combined bent indentations in one step, again by a dent-flattening homeomorphism. 

We first define a cubical complex $X_{k;i}$ as the target for the  flattening. 
For $i\in\{1,\ldots,m\}$ and $k\geq 2$, we consider a cubical $n$-complex,
\[X_{k;i}= \Refine^{k\nu}(B_i)\, \cup \, \bigcup_{\ell=1}^k \Refine^{(k-\ell)\nu}(\sfT_{\ell;i}),\]
where 
\begin{enumerate}
\item  $B_i=\Real_K(Z;\Sigma_i)$,  
\item $ \sfT_{\ell;i}$ is a copy of the tunnel complex  $\Tunnel_{K_\ell}(Z_\ell;\Sigma_i)$, and
\item only consecutive elements in the sequence $B_i, \, \sfT_{\ell;i} \ldots, \sfT_{k;i}$ may have nonempty intersections for which
\begin{enumerate}
\item $\Refine^\nu(B_i) \cap \sfT_{1;i}$ is an $(n-1)$-cube  $\widehat \omega_{1,i}$, and
\item for  $\ell \in \{2, \ldots,  k\}$, $\Refine^\nu(\sfT_{\ell-1;i}) \cap \,T_{\ell;i}$ is a collection of $(n-1)$-cubes $\widehat \omega_{\ell;q;i}$, $q\in Z_{\ell-2}^{[n-1]}$.
\end{enumerate}
\end{enumerate}

\begin{proposition}
\label{prop:general-tree-structure-of-realization} For each $i\in\{1,\ldots,m\}$ and $k\geq 2$,
there exists an $L(n,K)$-bilipschitz homeomorphism 
\[
\varphi_{k;i} \colon |\Real_{K_k}(Z_k;\Sigma_i)|\to |X_{k;i}|,
\]
which maps
\begin{align*} 
&|\widetilde \Rec_*(Z_0,\ldots, Z_{k-1};\Sigma_i)| \to |B_i|, \,\,\,\, |\Diff_*(Z_1,\ldots, Z_{k-1};\Sigma_i)| \to |\sfT_{1;i}|, \ldots,\\
&  |\Diff_*(Z_{k-1};\Sigma_i)| \to |\sfT_{k-1;i}|, \,\,\,\text{and}\,\,\,   |\Tunnel_{K_k}(Z_k;\Sigma_i)| \to |\sfT_{k;i}|,
\end{align*}
homeomorphically. Moreover,
\begin{enumerate}
\item $\varphi_{k;i}$ maps $\widetilde \omega_{1;Z;i}\mapsto 
 \widehat \omega_{1;Z;i}$ and  the unique $n$-cube in $\widetilde \Rec_{\Refine^\nu(K)}(Z_0;\Sigma_i)$ having $\widetilde \omega_{1;Z;i}$ as a face onto an $n$-cube in $\Refine^\nu(B_i)$, and 
for each $\ell \in \{2, \ldots,  k\}$ and $q\in Z_{\ell-2}^{[n-1]}$,
$\varphi_{k;i}$ maps $\widetilde \omega_{\ell;q;i}\mapsto \widehat \omega_{\ell;q;i}$ and 
 the unique $n$-cube in 
$\Refine^\nu(\Diff_*^\sfT(Z_{\ell-1},\ldots, Z_{k-1};\Sigma_i))$ 
having $\widetilde \omega_{\ell;q;i}$ as a face isometrically onto an $n$-cube in $\Refine^\nu(\sfT_{\ell-1,i})$, 

\item $\varphi_{k;i}$ is the identity in the complement of the union $W_1\cup \cdots, \cup W_{k}$ of sets, where
\begin{align*}
W_1= &\Wedge_{\Refine^{k\nu}(\Real_K(Z_0;\Sigma_i))}(\Dent_*^\sfT(Z_0,\ldots,Z_{k-1};\Sigma_1))\\
&\,\,\, \bigcap |\Real_{K_k}(Z_k;\Sigma_i)|,
\end{align*}
\begin{align*}
\quad \,\,\,\, W_\ell =& \Wedge_{\Refine^{\nu(k-\ell)}(\Tunnel_{K_\ell}(Z_\ell;\Sigma_i))}(\Dent_*^\sfT(Z_\ell,\ldots, Z_{k-1};\Sigma_i)) \\
&\,\,\, \bigcap |\Real_{K_k}(Z_k;\Sigma_i)|,
\end{align*}
for $\ell=2,\ldots, k$, and $W_\ell  \cap W_{\ell'} \neq \emptyset$ unless $|\ell-\ell'|\leq 1$.
\end{enumerate}
In particular, $\varphi_{k;i}$ is the identity on $|\Core_K(Z_0;\Sigma_i)|$.
\end{proposition}

The proof is almost verbatim to that of Proposition \ref{prop:tree-structure-of-realization}, we omit the details.

\section{Proof of the Evolution theorem}

To complete the proof of Theorem \ref{theorem:evolution-short}, it remains to check the quasiconformal stability of the sequence $(Z_k)$.

With the structure property in Proposition \ref{prop:general-tree-structure-of-realization}, the proof of the quasiconformal stability of the sequence $(Z_k)$ is standard.

\begin{theorem}\label{theorem:quasiconformal-stable}
There exists a constant $\sK=\sK(n, K)\ge 1$ for the following. For each $i=1,\ldots, m$ and $k\ge 1$, there exists a $\sK$-quasiconformal homeomorphism 
\[
f_{k;i} \colon |\Real_{K_k}(Z_k;\Sigma_i)|\to |\Real_K(Z;\Sigma_i)|,
\]
which is the identity on core $\Core_K(Z;\Sigma_i)$.
\end{theorem}

\begin{proof}
To determine the distortion constant $\sK$, we first make an observation and fix a parameter. The local channeling construction yields a constant $N=N(n, \nu)\ge 1$ for which at most $N$ tunnels in $\Tunnel_{K_{\ell+1}}(Z_{\ell+1};\Sigma_i)$ may be attached to the same tunnel 
in the complex  $\Tunnel_{K_\ell}(Z_\ell;\Sigma_i)$. 
Let $t\in \N$ be the smallest integer with the property that the refined $(n-1)$-cube $\Refine^{t \nu}([0,1]^{n-1})$ contains at least $N$ mutually disjoint $(n-1)$-cubes which do not meet the boundary of $[0,1]^{n-1}$. 

Recall that  tunnel $\Tunnel_{K_1}(Z_1;\Sigma_i)$ has graph size at most $\#(Z^{[n-1]})$ and that 
for each $\ell\geq 2$, all tunnels in complex $\Tunnel_{K_\ell}(Z_\ell;\Sigma_i)$ have  sizes at most $\lambda_\loc$. We also recall that the tunnel-contracting homeomorphism in Proposition \ref{prop:tunnel-contracting} has a bilipschitz constant depending only on $n$ and the size of the tunnel. The tunnel-contracting maps constructed below, subject to an additional requirement, have distortions depending on the number $N(n,\nu)$ as well. Everything considered, these constants depend only on $n$ and $K$.

As a preliminary step, let
\[
\varphi_{k;i} \colon |\Real_{K_k}(Z_k;\Sigma_i)|\to |B_i| \cup \bigcup_{\ell=1}^k |\sfT_{\ell;i}| 
\]
be an $L(n,K)$-bilipschitz dent-flattening homeomorphism defined in Proposition \ref{prop:general-tree-structure-of-realization} having all the properties listed therein. 
Hence each $\varphi_{k;i}$ is also $\sfK_\ind$-quasiconformal for a constant $\sfK_\ind\geq 1$ depending  only on $n$ and $K$. 

When $k=1$, by Proposition \ref{prop:tunnel-contracting} there exists a quasiconformal homeomorphism 
$\widehat \psi_{1;i} \colon |B_i|  \cup  |\sfT_{1;i}|   \to |B_i|$ which is the identity on $\Core_K(Z;\Sigma_i)$  and whose distortion depends only on $n$ and $K$. The composition $\widehat \psi_{1;i} \circ \varphi_{1;i}$ is the claimed map in the proposition.

Let now $k\geq 2$. We fix a tunnel-contracting bilipschitz, hence quasiconformal, homeomorphism 
\[
\psi_{k;i} \colon |B_i| \cup \bigcup_{\ell=1}^{k-1} |\sfT_{\ell;i}| \cup |\sfT_{k;i}| \to |B_i| \cup \bigcup_{\ell=1}^{k-1} |\sfT_{\ell;i}|
\]
as in Proposition \ref{prop:tunnel-contracting}, which 
\begin{enumerate}
\item maps each tunnel  $\sfT_{k;q;i}$, $q\in Z_{k-2}^{[n-1]}$,  in $\sfT_{k;i}$  into the unique $n$-cube $\Omega_{k;q;i} \in \Refine^\nu(\sfT_{k-1;i})$ adjacent to $\sfT_{k;q;i}$,  and 
\item is the identity in the complement of 
$\sfT_{k;i}\, \cup \, \bigcup_{q\in Z_{k-2}^{[n-1]}}\Omega_{k;q;i}$.
\end{enumerate}
Since the graph size of each tunnel $\sfT_{k;q;i}$ is at most $\lambda_\loc$, the bilipschitz constant as well as the quasiconformal distortion may be chosen to depend only on $n$ and $K$.
 
Define now  a tent  $\widehat \Omega_{k;q;i}$ over $\Omega_{k;q;i}$ to be the union of the $n$-cubes in $\Refine^\nu(\sfT_{k-1;i})$, including  $\Omega_{k;q;i}$ itself,  for which $(\widehat \Omega_{k;q;i}, \,\Omega_{k;q;i})$ is isomorphic to $([0,3]^n, [1,2]^{[n-1]}\times [0,1])$. 

Also, for each  $q'\in Z_{k-3}^{[n-1]}$, denote by $\Omega_{k-1;q';i}$ the unique $n$-cube in $\Refine^\nu(\sfT_{k-2,i})$ adjacent to tunnel  $\sfT_{k-1;q';i}$ in $\sfT_{k-1;i}$, and let $\widehat \Omega_{k-1;q';i}$ be  the tent over $\Omega_{k-1;q';i}$  in $|\sfT_{k-2;i}|$ defined analogously as before.

We fix as we may a $\sfK''$-quasiconformal homeomorphism
\[
\psi_{k-1;i}\colon |B_i| \cup \bigcup_{\ell=1}^{k-1} |\sfT_{\ell;i}| \to |B_i| \cup \bigcup_{\ell=1}^{k-2} |\sfT_{\ell;i}|,
\]
which 
\begin{enumerate}
\item  is the identity in the complement of  $\sfT_{k-1;i} \cup \bigcup_{q'\in Z_{k-3}^{[n-1]}} \widehat \Omega_{k-1;q';i}$,
 \item maps $\sfT_{k-1;q';i}$ into $n$-cube $\Omega_{k-1;q';i}$ for each  $q'\in Z_{k-3}^{[n-1]}$, 
\item $\psi_{k-1;i}$  is a scaling on each (previously defined) tent $ \widehat \Omega_{k;q;i}$, $q\in Z_{k-2}^{[n-1]}$, and 
 for each $q'\in Z_{k-3}^{[n-1]}$, $\psi_{k-1;i}$ maps every tent $\widehat \Omega_{k;q;i}$ contained in $|\sfT_{k-1;q';i}|$ onto a cube in 
$\Refine^{t \nu}(\Omega_{k-1;q';i})$ with a face  in $\partial |\sfT_{k-2;i}|$,
\end{enumerate}
where $\sfK'' $ is a constant depending only on $n$ and $K$ and, more precisely, for $k\ge 2$, $\sK''$ depends only on $n$ and $\nu$ and
is independent of  $k$.

Repeat this for $\ell= k-2, \ldots, 1$, we obtain a sequence of 
$\sfK''$-quasiconformal maps 
\[
\psi_{k;i},\,  \psi_{k-1;i}, \, \psi_{k-2;i},\, \ldots, \,\psi_{1;i},
\]
having one essential property. Namely,  for each $\ell\in \{1,\ldots, k\}$,  points in $\sfT_{\ell;i}$ are not moved by $\psi_{\ell+2;i}\circ\cdots \circ \psi_{k;i}$ and they are moved first slightly by $\psi_{\ell+1;i}$, then moved  by the tunnel-contracting map $\psi_{\ell;i}$ into tents situated in $ \sfT_{\ell-1;i}$.  From then on, tents are moved only by a scaling $\psi_{1;i} \circ\cdots \circ \psi_{\ell-1;i}$. In the above, we take  $\psi_{k+1;i}$ and  $\psi_{k+2;i}$  to be the identity map. Points in $|B_i|$ stay fixed under $\psi_{2;i}\circ \cdots, \circ \psi_{k-1;i} \circ \psi_{k;i}$ and are only moved by $\psi_{1;i}$.

Therefore the  composition  
\[f_{k;i}= \psi_{1;i} \circ \cdots, \circ \psi_{k;i} \circ \varphi_{k;i}\colon |\Real_{K_k}(Z_k;\Sigma_i)|\to |\Real_K(Z;\Sigma_i)|,\]
 is a $\sfK$-quasiconformal homeomorphism with  $\sK =  (\sK'')^2 \sK_\ind$ and is identity on $\Core_K(Z;\Sigma_i)$.
\end{proof}

\begin{proof}[Proof of  Theorem \ref{theorem:evolution-short}]
Quasiconformal stability of the sequence $(Z_k)$ follows by taking  the inverse $f_{k;i}^{-1}$ of the mapping $f_{k;i}$ in Theorem \ref{theorem:quasiconformal-stable}.

From this and Lemma \ref{lemma:summary}, Theorem \ref{theorem:evolution-short} follows. 
\end{proof}

\chapter{Lakes of Wada}
\label{sec:Wada}

Lakes of Wada continuum is a compact connected subset of $\bS^n$ which is the common boundary of its (at least three) complementary components. These complementary domains are called Lakes of Wada. The existence of planar Lakes of Wada was first shown by Brouwer \cite{Brouwer} in 1910, and by Yoneyama \cite{Yoneyama} in 1917. Lakes of Wada and the more general Wada basins occur as basins of attraction in the dynamics of nonlinear physical system. 
Existence of  Wada continuum as the common boundary of Fatou components in complex dynamics has been shown recently by Mart\'i--Pete, Rempe, and Waterman \cite{Marti-Rempe-Waterman}. 
Wada continua in higher dimensions were constructed by Luba\'nski \cite{Lubanski} in 1951. \index{Wada!continuum}

Evolution of separating complexes, repeated indefinitely,  yields the existence of Lakes of Wada in Riemannian $n$-manifolds in dimension $n\geq 3$, \emph{with controlled geometry}.  This result is contained implicitly in \cite{Drasin-Pankka}. 

\begin{theorem}[Wada continuum]
\label{intro-thm:Wada-Riemannian-manifold} 
\index{Lakes of Wada!quasiconformally controlled} 
Let $n\ge 3$, $m \ge 2$, and $(M,g)$ be a Riemannian $n$-manifold with boundary components $\Sigma_1,\ldots, \Sigma_m$.
Then there exists a compact connected continuum $S$ in $M$ whose complement  has exactly $m$ components, $ M_1,\ldots, M_m$, 
satisfying $\Sigma_i \subset  M_i$, for which
\[ 
\partial M_1\setminus \partial M = \cdots = \partial M_m \setminus \partial M =\bigcap_{i=1}^m\,  \overline{M_i} =S.
\]  
Moreover, there exists a constant $\mathsf K=\mathsf K(n,M)>1$ so that for each $ i\in\{1,\ldots, m\}$,   $M_i\setminus\partial M$ is $\mathsf K$-quasiconformal to $\Sigma_i\times (0,1)$.
\end{theorem}

\begin{proof}
For the proof we pass from the Riemannian manifold $(M,g)$ to a cubical complex  with a flat metric. In view of Proposition \ref{prop:Riemannian-to-cubical},  there exists a cubical $n$-complex $K$ having a flat metric $d_K$ in which each cube is isometric to a Euclidean unit cube, and for which $(|K|, d_K)$ is quasi-similar to $(M,g)$. 
Therefore, it suffices to prove Theorem \ref{intro-thm:Wada-Riemannian-manifold} in the setting of $(|K|, d_K)$. 

From now on, we continue to use $\Sigma_1, \ldots, \Sigma_m$ for the boundary components of $K$, and refer freely to the construction in Section \ref{sec:evolution}. 

Let $K_k= \Refine^{k \nu}(K)$, and $Z_0, Z_1,\ldots, Z_k, \ldots$, respectively, be the  separating complexes in $K_0, K_1,\ldots, K_k, \ldots $ in Theorem \ref{theorem:evolution-short}. 

Recall that the cores associated to the components separated by $Z_k$ are the complexes
\[
\Core_{K_k}(Z_k;\Sigma_i)= \Comp_{K_k}(Z_k;\Sigma_i) - \Star_{K_k}(Z_k)\quad i= 1,\ldots,m,
\]
and  they are expanding in a strong sense, 
\[
|\Core_{K_{k-1}}(Z_{k-1};\Sigma_i)| \subset \interior |\Core_{K_k}(Z_k;\Sigma_i)|.
\]
Thus the unions 
\[
M_i=\bigcup_{k=1}^\infty |\Core_{K_k}(Z_k;\Sigma_i)|,\quad i=1,\ldots, m,
\]
are mutually disjoint open sets in $M$. Let
\[
S= \bigcap_{k=1}^\infty |\Star_{K_k}(Z_k)|.
\]
We claim that $\partial M_1 \setminus \partial |K| = \ldots = \partial M_m\setminus \partial |K| =S$.

The inclusion $\bigcup_{i=1}^m (\partial M_i  \setminus \partial |K|) \subset S$ is  immediate.

We check next that  $S \subset \partial M_i \setminus \partial |K| $ for each $i\in \{1,\ldots,m\}$. Fix an index $i$ and let  $x\in S$. Then, for each $k\ge 1$, the point $x$ is contained in the union $Q_k \cup Q'_k$ of two adjacent $n$-cubes $Q_k, Q'_k \in \Star_{K_k}(Z_k)$ having a common face in $Z_k^{[n-1]}$.  
The relative Wada property of
$Z_{k+1}$  with respect to $Z_k$ stated in Theorem \ref{theorem:evolution-short} yields the existence of an $n$-cube $Q^*_{k+1}$ in $ \Comp_{K_{k+1}}(Z_{k+1};\Sigma_i)^{[n]} \cap \Refine^\nu(Q_k \cup Q'_k )^{[n]}$. By the definition of cores, the $n$-cube $Q^*_{k+1}$ necessarily contains a point $x_{k+1}$ in $\Core_{K_{k+1}}(Z_{k+1};\Sigma_i)$. Hence $x$ is a limit point of the sequence $(x_{k+1})$ in $ M_i$, thus a boundary point of $M_i$.
This proves that  
\[\partial M_1 \setminus \partial |K| = \ldots = \partial M_m\setminus \partial |K| =S.\]  

Before verifying the quasiconformality, we briefly recall the metrics defined in  Section \ref{sec:metric-realization}. Complex $K$ is equipped a flat metric $d_K$ induced by its flat structure. Distance between points in $M_i\subset |K|$ (also distance between points in $\interior(|\Comp_{K_k}(Z_k;\Sigma_i)|\setminus |Z_k|)$) is measured by $d_K$.
The  realization $\Real_{K_k}(Z_k;\Sigma_i)$ is equipped with the lifted standard metric, that is, the path metric for which each $n$-cube in $\Real_{K_k}(Z_k;\Sigma_i)$ is locally isometric to the corresponding $n$-cube in  $|\Comp_{K_k}(Z_k;\Sigma_i)|\subset |K|$ via lifting. Under these metrics,  $\interior (|\Comp_{K_k}(Z_k;\Sigma_i)|\setminus |Z_k|)$ and $\interior (|\Real_{K_k}(Z_k;\Sigma_i)|)$ are conformal.

Recall also that the  lift $\pi_{K_k, Z_k;\Sigma_i}^{-1}(\Core_{K_k}(Z_k;\Sigma_i))$ of core $\Core_{K_k}(Z_k;\Sigma_i)$ to $\Real_{K_k}(Z_k;\Sigma_i)$   is identified with the core itself;  we use the same notation for both.

For each $i\in \{1,\ldots, m\}$ and  $k\ge 2$, let
\[
g_{k;i}\colon |\Real_K(Z,\Sigma_i)|\to |\Real_{K_k}(Z_k;\Sigma_i)|
\]
be the inverse of the $\sfK$-quasiconformal mapping in Theorem \ref{theorem:quasiconformal-stable}, and let
\[
h_{k;i}\colon \interior (|\Comp_K(Z,\Sigma_i)|\setminus |Z|) \to \interior(|\Comp_{K_k}(Z_k;\Sigma_i)|\setminus |Z_k|)
\]
be the mapping associated to $g_{k;i}$ canonically via lifting.

For each fixed $i$, mappings $(h_{k;i})$ are identity on $|\Core_K(Z,\Sigma_i)|$, hence equicontinous, therefore they form a normal family. Therefore there is a subsequence $(h_{k_j;i})$ which converges locally uniformly to a $\sfK$-quasiconformal map 
\[h_i\colon \interior(|\Comp_K(Z,\Sigma_i)|\setminus |Z|) \to \ker\left(\interior(|\Comp_{K_{k_j}}(Z_{k_j};\Sigma_i)|\setminus |Z_{k_j}|)\right);\]
see V\"ais\"al\"a \cite{Vaisala-book}. Recall that the kernel $\ker (A_j)$ of a sequence of sets $(A_j)$  in $|K|$ is the set of all points in $|K|$ which has a neighborhood that is contained in all but finitely many sets $A_j$. 
Since 
\[
|\Core_{K_k}(Z_k;\Sigma_i)|\subset |\Comp_{K_k}(Z_k;\Sigma_i)|\setminus |Z_k| \subset |\Core_{K_k}(Z_k;\Sigma_i)|\cup |\Star_{K_k}(Z_k)|,
\]
and cores are expanding in a strong sense,  it is straightforward to check that 
\[
\ker\left(\interior(|\Comp_{K_{k_j}}(Z_{k_j};\Sigma_i)|\setminus |Z_{k_j}|)\right)= M_i \setminus \partial |K|.
\]

Since  $ \interior (|\Comp_K(Z,\Sigma_i)|\setminus |Z|)$ is $L(n,K)$-bilipschitz to $\Sigma_i\times (0,1)$, each $M_i\setminus \partial |K|$ is $\sfK$-quasiconformal to the space $\Sigma_i\times (0,1)$ for some constant $\sfK=\sfK(n,K)\geq 1$. 
 
This completes the proof of the theorem.
\end{proof}

\begin{remark}
Topological  Wada continua may be built by iterating the steps in Sections \ref{sec:reservoir-canal-system} and  \ref{sec:channeling} alone. The localized channeling procedure in Section \ref{sec:localization} and the evolution in Section \ref{sec:evolution} are used in constructing Wada continua which are also quasiconformally stable.
\end{remark}

\begin{corollary}[Lakes of Wada]
\label{cor:Wada-sphere} 
Let $n\ge 3$, $m \ge 2$, and $D_1,\ldots, D_m$ be mutually disjoint closed connected PL $n$-dimensional submanifolds in $\bS^n$. 
Then there exists a compact connected continuum $S$ in $\bS^n\setminus \bigcup_{i=1}^m D_i $, whose complement in $\bS^n$  has exactly $m$ components,  $ M_1,\ldots, M_m$ satisfying $D_i \subset \interior M_i$, for which 
\[ \partial M_1= \cdots = \partial M_m =\bigcap_{i=1}^m\,  \overline{M_i}=S.\]  
Moreover, there exists a constant $\mathsf K=\mathsf K(n)>1$  so that 
for each $i\in\{1,\ldots, m\},$  $M_i$  is $\mathsf K$-quasiconformal to $\interior D_i$. 
\end{corollary}

Theorem \ref{intro-thm:Wada-on-sphere} in Introduction follows immediately from Corollary \ref{cor:Wada-sphere}.

\medskip

Our construction in Theorem \ref{intro-thm:Wada-Riemannian-manifold} yields Lakes of Wada which are also John domains. We state this property as a corollary (Corollary \ref{cor:Wada-John-domain}) for  Riemannian manifolds.

The concept of John domains was originated by F. John \cite{John}, and the term John domain was due to  Martio and Sarvas  \cite{Martio-Sarvas}. 
John domains occur frequently in the study of elasticity and geometric analysis.
A proper subdomain $D$ of $\R^n, n\geq 2$ is a $C$-\emph{John domain}, if there exists a constant $C\geq 1$ such that any two points $x_1, x_2 \in D$ can be joined by a rectifiable arc $\gamma \subset D$ for which \index{John domain}
\[\min_{j=1, 2} \, s(\gamma(x_j,x))\leq C  \,d(x,\partial D)\quad \text{for all}\,\, x\in \gamma, \]
\noindent where $s(\gamma)$ is the length of $\gamma$, $\gamma(x_j,x)$ the part of $\gamma$ between $x_j$ and $x$, and $d(x,\partial D)$  the  distance from $x$ to $\partial D$. 

In manifolds, we take  the obvious generalization of the definition of $C$-John domains. The following statement is a straightforward consequence of  Remark \ref{rmk:core-to-core}; we omit the details.

\begin{corollary}\label{cor:Wada-John-domain} 
In Theorem \ref{intro-thm:Wada-Riemannian-manifold}, the continuum $S$ may be chosen  for which the open submanifolds $\interior (M_i)$ are all $C$-John domains, for a constant $C=C(n,M)\geq1$.
\end{corollary}

In dimension $n=2$ the situation changes. While Lakes of Wada in $\bS^2$ are all conformal to the unit disk provided that they are $2$-cells, they are never John domains.

\begin{proposition}\label{proposition:Wada_2dim} 
Let $m\geq 3$, and $\cL_1,\ldots, \cL_m$ be open $2$-cells in $\bS^2$ for which
\[ 
\partial \cL_1= \cdots = \partial \cL_m =\bigcap_{i=1}^m\,  \overline{\cL_i}.
\]  
Then none of $\cL_1,\ldots, \cL_m$ is a John domain.
\end{proposition}

\begin{proof}
Suppose that $\cL_1$ is a John domain. We transform the question to $\overline{\R^2}$ by applying a stereographic projection of $\bS^2$, and assume that $\cL_1$ is a  $C$-John domain in $\overline{\R^2}$ for some $C\geq 1$ and $\infty \in\cL_2$. 
Denote by $S=\bigcap_{i=1}^m\,  \overline{\cL_i}$. 

Let $z_0$ be a point in $\cL_1$, and  $a>0$ and $b>0$ be the smallest and the largest radius $r$, respectively, for which $S$ meets the circle $|z- z_0|=r$. Choose and fix points $p_a, p_b \in S$ for which $|p_a-z_0|= a$ and $|p_b -z_0|=b$. Thus  $\{z \colon |z-z_0|<a\}\subset \cL_1$ and  $\{z \colon |z-z_0|\leq a\}\subset \cL_1 \cup S$.
 
Let $\epsilon = \min\{ |p_a-p_b|/10, a/(10C), (b-a)/(10 C)\}$, $D_a$ and $ D_b$ be  disks $\{z \colon |z-p_a|\leq \epsilon\}$ and $\{z \colon |z-p_b|\leq \epsilon\}$, respectively, and  $\Omega =\{ z \colon a<|z-z_0| <b\}\setminus (D_a\cup D_b)$. Since $p_a$ and $p_b$ are points in the boundary of $ \cL_i$ for all $i$,  each $\cL_i$ enters the interior of $D_a$ and also the interior of D$_b$. 

Thus for each $j \in \{2,3\}$, there exists a PL Jordan arc $\Gamma_j\subset \cL_j \cap (\{z \colon |z-z_0|\geq a \} $ having one end points in  $\interior D_a$ and the other in $ \interior D_b$. Choose and fix an open subarc $\gamma_j$ of $\Gamma_j$ which is contained in the interior of $\Omega$, and  has  one end point on $\partial D_a$ and the other  on the outer boundary component of $\Omega$. Let $\Omega_1$ be the component of $\Omega\setminus (\gamma_a\cup \gamma_b)$ for which the intersection of $\partial \Omega_1 $ with the inner boundary component of $\Omega$ is an arc $\tau$ in $\{z \colon |z-p_a|=\epsilon\}$.

The intersection $\Omega_1\cap \{z \colon |z-z_0|= (a+b)/2\}$ necessarily contains a point in the common boundary $S$. Therefore there is necessarily a point $z_1\in \cL_1\cap \Omega_1$ whose distance to circle $\{z \colon |z-z_0|= (a+b)/2\}$ is less than $\epsilon$.

Now any arc $\gamma\subset \cL_1$ connecting $z_1$ to $z_0$ must intersect $\tau$, say at a point $p$. Since 
\[ 
\min\{|p-z_0|,|p-z_1|\} \geq \min\{ a, (b-a-2\epsilon)/2 \} >  2C|p-p_a|\geq 2 C \dist(p, S), 
\]
$\cL_1$ is not a $C$-John domain. The same argument applies to other $2$-cells.
\end{proof}


\part{Deformation}
\label{part:deformation}

\chapter{Quasiregular deformation theorem}

In this Part, we extend Rickman's planar deformation theory (\cite[Section 5]{Rickman_Acta}) to dimensions $n\ge 3$. In our setting deformation is used to simplify  the cubical Alexander maps on (shellable) cubical complexes by deforming them into quasiregular maps which are expansions of Alexander maps on  simpler cubical complexes; see Part \ref{part:QR-extension}.

Although Rickman does not state his deformation result explicitly in \cite[Section 5]{Rickman_Acta}, we may formulate it as follows. We refer to \cite[Section 3]{Rickman_Acta} for the definition of a mapping complex.

\begin{theorem*}[Rickman, 1985]
Suppose that $K$ is a $2$-dimensional mapping complex whose space is a surface $\Sigma$, $G$ is a subcomplex whose space is a $2$-cell, and $f\colon \Sigma\to \bS^2$ is an Alexander map whose restriction $f|_{|G|}$ evenly covers $\bS^2$. Then there exists a discrete and open map $F\colon \Sigma \times [0,1] \to \bS^2 \times [0,1]$ for which $F|_{\Sigma \times \{0\}} = f$ and $F|_{\Sigma\times \{1\}}$ is an Alexander map $f' \colon \Sigma \to \bS^2$ expanded by simple covers, associated to a mapping complex $K'$ on $\Sigma$ obtained by collapsing the $2$-complex $G$ to a $1$-complex $G'$ on a $1$-cell.
\end{theorem*}

Our main deformation result is a version of of Rickman's theorem for cubical Alexander maps on shellable cubical complexes in dimension $n\geq 3$.

\begin{definition}
\label{def:shellable}
A cubical complex $K$, having an $n$-cell as its space $|K|$, is \emph{shellable}\index{cubical complex!shellable} if there exists an order $q_1, q_2,\ldots, q_m$ of the $n$-cubes of $K$ for which \index{shellability}
\[
(q_1\cup \cdots \cup q_i) \cap q_{i+1}
\]
is an $(n-1)$-cell for each $i=1,\ldots, m-1$.
\end{definition}

We formulate our main result in terms of the star replacement $K^*$ of the cubical complex $K$.

\begin{definition}
\label{def:star-replacement}
\index{cubical complex!star-replacement $K^*$} \index{$K^*$}
Let $K$ be a cubical complex whose space is an $n$-cell. A simplicial complex $K^*$ is a \emph{star-replacement of $K^\Delta$} if $|K^*|=|K|$, $K^*|_{\partial |K|} = K^\Delta|_{\partial |K|}$, and $K^*$ has a unique vertex in the interior of $|K|$.
\end{definition}

\begin{figure}[h!]
\begin{overpic}[scale=0.8,unit=1mm]{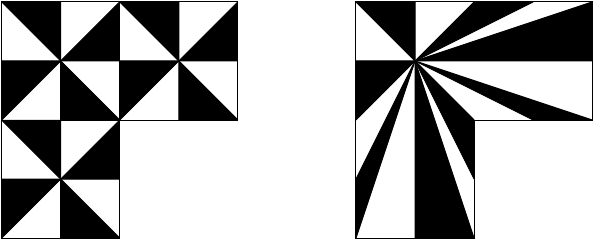} 
\end{overpic}
\caption{The barycentric triangulation $K^\Delta$ of a cubical complex $K$ in Figure \ref{fig:canonical} and its star-replacement $K^*$.}
\label{fig:star_replacement}
\end{figure}

\begin{restatable}[Quasiregular deformation theorem]{theorem}{Quasiregulardeformation} \index{Quasiregular deformation theorem}
\label{thm:QR-deformation-star} 
Let $n\ge 2$ and $K$ be a shellable cubical $n$-complex. Then there exist a star replacement $K^*$ of $K$ and constants $\sL^\dagger=\sL^\dagger(n, K^*)\ge 1$ and $\sL'=\sL'(n,K)\ge 1$ for the following.

Let $f \colon |K|\to \bS^n$ be an $\sL$-BLD-controlled expansion of a $K^\Delta$-Alexander map for $\sL\ge 1$. 
Then there exists a $\max\{ \sL', \sL\}$-BLD homotopy 
\[ F\colon |K|\times [0,1]\to \bS^n\times [0,1]\quad (\text{modulo}\,\,\, |\partial K|) \]
from $f$ to a $\max\{ \sL^\dagger, \sL\}$-BLD-controlled expansion $f'\colon |K|\to \bS^n$ of a $K^*$-Alexander map. 
In addition, if $|K|$ is a convex subset of $\R^n$ and the flat structure of $K$ consists of affine maps, then we may choose $K^*$ so that its flat structure also consists of Euclidean affine maps.
\end{restatable}


\chapter{Local deformation}

\section{Local deformation of Alexander stars}
\label{sec:locel_deformation_theory}

\newcommand{\DStar}{\mathrm{N}\Star}
\newcommand{\st}{\mathrm{st}}

We say that a pair $(K,f)$ is an \emph{Alexander $n$-complex} if $K$ is a simplicial $n$-complex and $f\colon |K|\to \bS^n$ is a $K$-Alexander map. 

Recall from the definition of Alexander maps in Section \ref{sec:Alexander-maps} that the target  $\bS^n$ has the structure of the complex $\Sigma^2(\Delta_n)$ (or $\Sigma^2(\Delta^\square_n)$ when the domain is a barycentric triangulation of a cubical complex), which consists of two $n$-simplices $B^n_+$ and $B^n_-$ having all their vertices, $w_0,\ldots, w_n \in \bS^{n-1}$, in common. 

A vertex $v\in K$ is called a \emph{$w_i$-vertex for $f$} if $f(v)=w_i$. \index{$w_i$-vertex}

Let $v$ be a $w_i$-vertex. For $j \ne i$, we define the $w_j$-neighborhood of $\Star_K(v)$ to be
\[
\DStar_f(v;w_j) = \bigcup_{u\in \Link_K(v)\cap f^{-1}(w_j)} \Star_K(u).
\]
where $\Link_K(v)$ is the \emph{link of $v$ in $K$}, that is, $\Link_K(v) = \Star_K(v) \cap (K - \Star_K(v))$. Recall that $K - \Star_K(v)$ is the subcomplex of $K$ consisting of all simplices of $K$ which do not meet $v$. 

Note that $\Star_K(v) \subset \DStar_f(v;w_j)$; see Figure \ref{fig:DStar} for examples.

\begin{figure}[htp]
\begin{overpic}[scale=.43,unit=1mm]{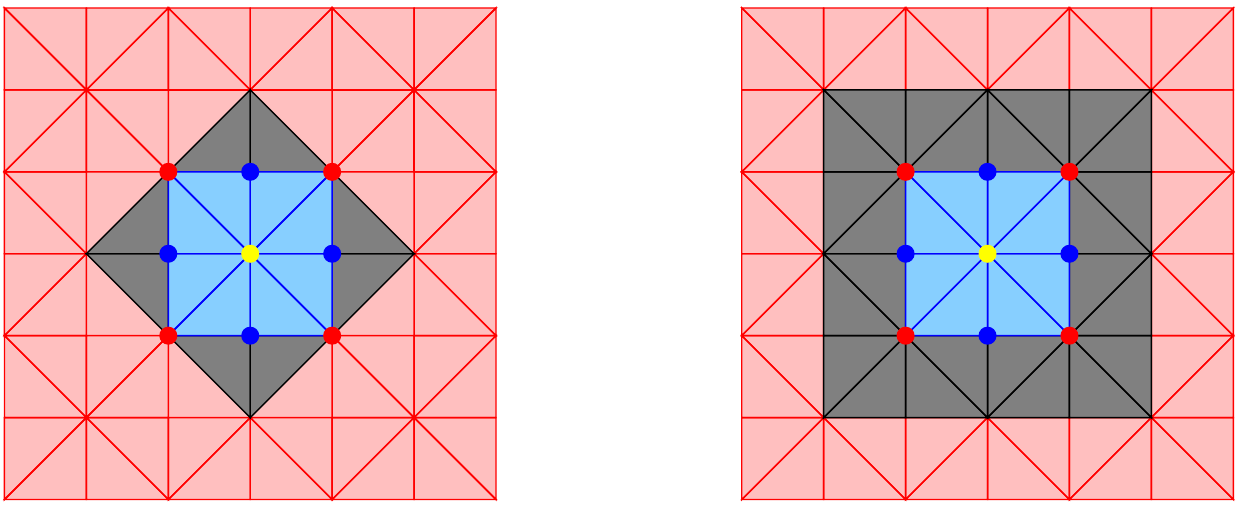} 
\put(17,20){\tiny $v$}
\put(25,18){\tiny $v_1$}
\put(25,26){\tiny $v_2$}
\put(17,26){\tiny $v_3$}
\put(9,26){\tiny $v_4$}
\put(9,18){\tiny $v_5$}
\put(9,10){\tiny $v_6$}
\put(17,10){\tiny $v_7$}
\put(25,10){\tiny $v_8$}

\put(70,20){\tiny $v$}
\put(78,18){\tiny $v_1$}
\put(78,26){\tiny $v_2$}
\put(70,26){\tiny $v_3$}
\put(62,26){\tiny $v_4$}
\put(62,18){\tiny $v_5$}
\put(62,10){\tiny $v_6$}
\put(70,10){\tiny $v_7$}
\put(78,10){\tiny $v_8$}

\end{overpic}
\caption{Neighborhoods $\DStar_f(v;w_1)$ and $\DStar_f(v;w_2)$ of the star $\Star_K(v)$ of a $w_0$-vertex $v$; here $f^{-1}(w_1) = \{ v_1,v_3,v_5,v_7\}$ and $f^{-1}(w_2) = \{ v_2,v_4,v_6,v_8\}$.}
\label{fig:DStar}
\end{figure}

In deformation, the notion of reduced star has a key role; see Figure \ref{fig:reduced_star} for two examples.

\begin{definition}[Reduced star]
The \emph{reduced star $\Star_f(v;w_j)$ of a $w_i$-vertex $v$ for $i\ne j$} is the subcomplex consisting of all the simplices in $\Star_K(v)$ which do not meet $f^{-1}(w_j)$. \index{reduced star}  
\end{definition}

We say that a $(n-1)$-simplex $\tau$ is  \emph{$w_j$-avoiding} if $\tau \cap f^{-1}(w_j) = \emptyset$. Note that an $(n-1)$-simplex in $\Star_K(v)$ is $w_j$-avoiding if and only if it is contained in $\Star_f(v;w_j)$.

\begin{remark}
In a simplicial $1$-complex, we define the reduced star of a vertex $v$ to be simply the vertex $v$ itself.
\end{remark}

\begin{figure}[htp]
\begin{overpic}[scale=.43,unit=1mm]{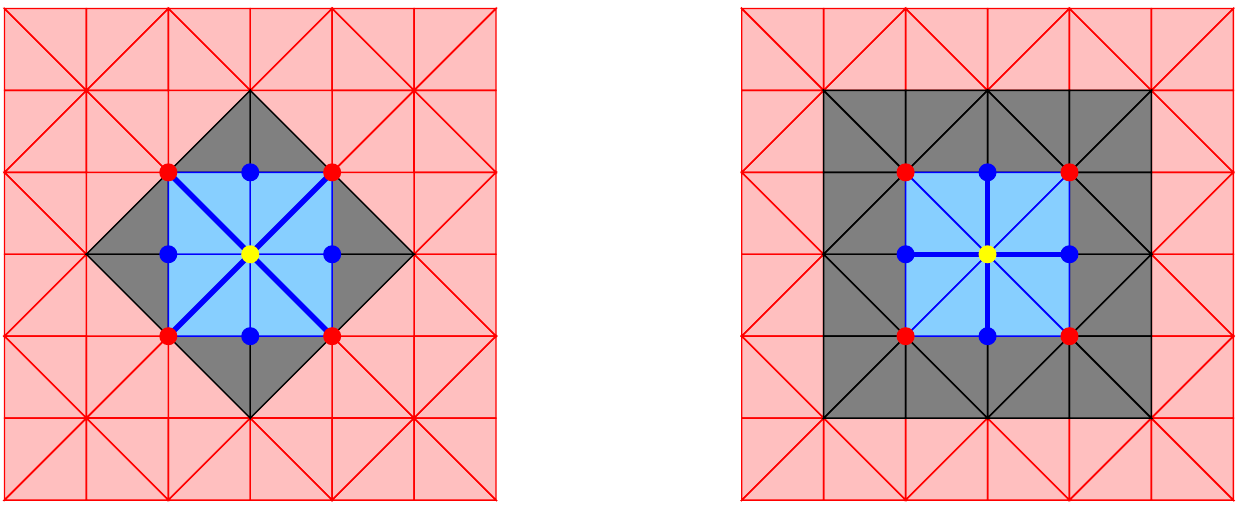} 
\put(17,20){\tiny $v$}
\put(25,18){\tiny $v_1$}
\put(25,26){\tiny $v_2$}
\put(17,26){\tiny $v_3$}
\put(9,26){\tiny $v_4$}
\put(9,18){\tiny $v_5$}
\put(9,10){\tiny $v_6$}
\put(17,10){\tiny $v_7$}
\put(25,10){\tiny $v_8$}

\put(70,20){\tiny $v$}
\put(78,18){\tiny $v_1$}
\put(78,26){\tiny $v_2$}
\put(70,26){\tiny $v_3$}
\put(62,26){\tiny $v_4$}
\put(62,18){\tiny $v_5$}
\put(62,10){\tiny $v_6$}
\put(70,10){\tiny $v_7$}
\put(78,10){\tiny $v_8$}
\end{overpic}
\caption{Reduced stars $\Star_f(v;w_1)$ and $\Star_f(v;w_2)$ of a star $\Star_K(v)$ marked in bold; here $f^{-1}(w_1) = \{ v_1,v_3,v_5,v_7\}$ and $f^{-1}(w_2) = \{ v_2,v_4,v_6,v_8\}$.}
\label{fig:reduced_star}
\end{figure}

\subsection{Collapse of a complex at a vertex}

We define now the reduction of the complex $K$ at vertex $v$. First we make an elementary observation. 

\begin{lemma}
\label{lemma:link-collapse}
Let $v\in K$ be a $w_i$-vertex, $L=\Link_K(v)$, and $u \in L$ a $w_j$-vertex with $j \ne i$. Then the star $\Star_L(u)$ is isomorphic to the join $\Link_L(u) \star \{ v \} \subset \Star_f(v;w_j)$. 
\end{lemma}

An immediate consequence of this observation is the following proposition on collapse of the complex $K$. For the statement, we give the  definition of a tame vertex. In Figure \ref{fig:DStar}, the vertex $v$ is a $w_j$-tame $w_0$-vertex for $j=1$ and $2$.

\begin{definition}
\label{def:tame-vertex} \index{$w_j$-tame vertex}
A $w_i$-vertex $v\in K$ in an Alexander $n$-complex $(K,f)$ is \emph{$w_j$-tame in $K$ with $j\ne i$} if the following conditions are satisfied:
\begin{enumerate}
\item $v\in \interior K$ and $\Star_K(v)$ is isomorphic to a simplicial complex $P$ in $\R^n$ for which $|P|$ is homeomorphic to $\bar B^n$, and 
\item each $w_j$-vertex $u$ of $\Link_K(v)$ belongs to $\interior K$ and $(\Star_K(u), \Link_K(v)\cap \Star_K(u))$ is isomorphic to a pair $(Q,R)$ of  simplicial complexes for which $|Q| \subset\R^n$,  $|R|\subset \R^{n-1}\times \{0\}$, and $(|Q|,|R|)$ is homeomorphic to  $(\bar B^n, \bar B^{n-1})$.
\end{enumerate}
\end{definition}

\begin{proposition}
\label{prop:star-collapse}
Let $(K,f)$ be an Alexander $n$-complex, let $v\in K$ be a $w_j$-tame $w_i$-vertex in $K$. Then there exists a continuous map $\Phi_{f,j} \colon |K - \Star_K(v)|\to |K|$ having the following properties:
\begin{enumerate}
\item $\Phi_{f,j}$ is the identity on $K-\DStar_f(v;w_j)$; 
\item $\Phi_{f,j}$ is a simplicial map from $\Link_K(v)$ onto $\Star_f(v;w_j)$;
\item $\Phi_{f,j}$ is an $\sL$-BLD map with respect to the metric $d_K$ and homeomorphism  $\interior |K - \Star_K(v)|\to |K|\setminus |\Star_f(v;w_j)|$, where $\sL=\sL(n,\DStar_f(v;w_j))$; and
\item the image 
\[
K^\st_{v,j} = (\Phi_{f,j})_*(K - \Star_K(v))
\]
of $K - \Star_K(v)$ under $\Phi_{f,j}$ is a simplicial complex on $|K|$ with $(n-1)$-skeleton 
\[
(K^\st_{v,j})^{(n-1)} = \left( (\Phi_{f,j})_*(K - \Star_K(v))\right)^{(n-1)} \supset \, \Star_f(v;w_j).
\]
\end{enumerate}  
Moreover, there exists a $K^\st_{v,j}$-Alexander map $f_{v,j} \colon |K^\st_{v,j}| \to \bS^n$ satisfying 
\[
f_{v,j} \circ \Phi_{f,j}= f,
\] 
for which $v$ is a $w_j$-vertex.
\end{proposition}

\begin{definition}
We call $\Phi_{f,j}$ a \emph{collapse map of $K$ at $v$}, and $K^\st_{v,j}$ a \emph{collapse of $K$ at $v$}. \index{collapse map} \index{$K^\st_{v,j}$}
\end{definition}

\begin{proof}[Proof of Proposition \ref{prop:star-collapse}]
Since $v$ is $w_j$-tame, we may combine the local isomorphisms from Lemma \ref{lemma:link-collapse} to a well-defined simplicial map $\Link_K(v) \to \Star_f(v;w_j)$. We then extend its realization $|\Link_K(v)|\to |\Star_f(v;w_j)|$ to a non-collapsing simplicial map $|K - \Star_K(v)|\to |K|$, which is identity on $n$-simplices not meeting $\Link_K(v)$ or only meeting $\Link_K(v)$  in a vertex, and is bilipschitz on the $n$-simplices meeting the link $\Link_K(v)$. Indeed, let $u\in \Link_K(v)$ be a $w_j$-vertex. Since $v$ is $w_j$-tame, both $|\Star_K(v)\cap \Star_K(u)|$ and $|\Star_K(u)-\Star_K(v)|$ are $n$-cells, and $|\Star_K(v)\cap \Star_K(u)| \cap |\Star_K(u)-\Star_K(v)| = |\Link_K(v)\cap \Star_K(u)|$ is an $(n-1)$-cell. Thus there exists a bilipschitz homeomorphism $|\Star_K(u)-\Star_K(v)| \to |\Star_K(u)|$ extending the already defined map.

The existence of the Alexander map $f_{v,j}$ follows immediately from the observation that, in the complex $(\Phi_{f,j})_*(K - \Star_K(v))$, the vertex $v$ is a well-defined $w_j$ vertex. 
\end{proof}

By Proposition \ref{prop:star-collapse}, we may collapse a simplicial complex at a $w_j$-tame vertex for which this collapse produces a new Alexander map $f_{v,j} \colon |K|\to \bS^n$ with respect to the new complex. The mappings $f$ and $f_{v,j}$, however, have different degrees. More precisely,
\[
\deg f = \deg f_{v,j} + (\# \Star_K(v)^{[n]})/2.
\]
Thus, if $|K|$ is a closed manifold, there is no BLD homotopy  $F\colon |K|\times [0,1]\to \bS^n \times [0,1]$ from $f$ to $f_{v,j}$. 

In the rest of this chapter, we construct a map $\widetilde f \colon |K|\to \bS^n$ which is a BLD expansion of $f_{v,j}$ and is homotopic to mapping $f$. More precisely, we prove the following.

\begin{theorem}[Local deformation]
\label{thm:local-deformation}
Let $(K,f)$ be an Alexander complex of dimension $n\ge 2$. Let also $v\in K$ be a $w_j$-tame $w_i$-vertex of $K$ for $j\ne i$. Then there exist $\Ldef = \Ldef(v,\DStar_f(v;w_j)) \ge 1$ and an $\Ldef$-BLD map $F_{v,j} \colon |K|\times [0,1]\to |K|\times [0,1]$ having the following properties:
\begin{enumerate}
\item $F_{v,j}|_{|K|\times \{0\}} = f$ and
\item $F_{v,j}|_{|K|\times \{1\}} \colon |K|\times \{1\} \to \bS^n \times \{1\}$ is an $\Ldef$-BLD-controlled expansion of a $K^\st_{v,j}$-Alexander map, where $K^\st_{v,j}$ is a collapse of $K$ at $v$ with respect to $w_j$.
\end{enumerate}
Moreover, if $|K| \subset \R^n$, $|\DStar_f(v;w_j)|$ is a convex subset, and the flat structure of $K$ consists of affine maps, then  $K^\st_{v,j}$ may be chosen so that its flat structure consists of affine maps.
\end{theorem}

\begin{remark}
\label{rmk:local-deformation}
The proof shows that the constant $\Ldef$ depends, in fact, only on the dimension $n$ and the isometry class of the complex $\DStar_f(v;w_j)$.
\end{remark}

We discuss the proof of Theorem \ref{thm:local-deformation} in two steps: first  a local deformation of complex $K$ in star $\Star_K(v)$ into a  weakly simplicial complex, called clover complex $\Clover_K(v;w_j)$, and an associated $\Clover_K(v;w_j)$-Alexander map, then in the second step, an extraction of the leaves in $\Clover_K(v;w_j)$ into the complex  $K^\st_{v,j}$ and  a controlled expansion $|K|\to \bS^n$ of a $K^\st_{v,j}$-Alexander map.

\subsection{Clovers and weakly simplicial complexes} 
\label{sec:leaves_and_clovers} 

We define now the clover complexes used in  deformation. We begin with an observation on pairing $n$-simplices. 
In what follows, $(K,f)$ is an Alexander $n$-complex for $n\ge 2$.

Let $v_0$ be a $w_i$-vertex.
Since the star $\Star_K(v_0)$ admits an Alexander map, the $n$-simplices in $\Star_K(v_0)$ may be paired for each $j=0,\ldots, n, j\neq i,$ as follows. Given an $(n-1)$-simplex $\tau$ in the reduced star $\Star_f(v_0;w_j)$, there are exactly two $n$-simplices $T$ and $T'$ having $\tau$ as a face. Since $\Star_K(v_0)$ is simplicial, $T\cap T' = \tau$. The  vertices $x_T$ and $x_{T'}$ of $T$ and $T'$, respectively, which are not contained in $\tau$, are distinct for which $f(x_T) = f(x_{T'}) = w_j$, $T=\tau * \{x_T\}$ and $T' = \tau * \{x_{T'}\}$; here the join is taken in the simplicial complex $\Star_K(v_0)$. 

Since $f|_{\interior |T\cup T'|}\colon \interior |T\cup T'| \to \bS^n$ is an embedding, we call $\{T, T'\}$ a \emph{simple pair sharing $\tau$}. \index{simple pair}
 Since each $n$-simplex has a $w_j$-avoiding face, all $n$-simplices in $\Star_K(v_0)$ can be paired, and the pairing is unique.
We denote $\scrS_f(v_0;w_j)$ the family of all simple pairs in $\Star_K(v_0)$ over $w_j$-avoiding faces.

We now define leaves and then clover complexes; see Figure \ref{fig:3-leaved_clover} for an illustration.

\begin{definition}
\label{def:leaf}
\index{leaf complex}
An \emph{$n$-leaf $L$} is an $n$-dimensional CW-complex consisting of two $n$-simplices $T\cup T'$ which have a common vertex $v_L$, and whose intersection $T\cap T'$ is the star of $v_L$ in $\partial T$ and also the star of $v_L$ in $\partial T'$. 
\end{definition}

The space $|L|$ of an $n$-leaf is an $n$-cell and the union of the boundaries $\partial |T| \cup \partial |T'|$  is homeomorphic to $\bS^{n-1} \cup B^{n-1}$.

\begin{definition}
\label{def:clover} \index{clover complex $\Clover_f(v_0;w_j)$}

Let $v_0$ be a $w_i$-vertex. For $j=0, 1,\ldots, n, j\neq i$, a \emph{clover complex $\Clover_f(v_0;w_j)$} of a reduced star $\Star_f(v_0;w_j)$ is a CW-complex, which is a union of leaves for which each $(n-1)$-simplex $\sigma$ in $\Star_f(v_0;w_j)$ is a midrib $r_{L_\sigma}$ of a unique leaf $L_\sigma$ in $\Clover_f(v_0;w_j)$, and the leaves $L_\sigma$ are essentially mutually disjoint in the sense: $L_\sigma \cap L_{\sigma'} = \sigma \cap \sigma'$ for $(n-1)$-simplices $\sigma$ and $\sigma'$ in $\Star_f(v_0;w_j)$.
\end{definition}

\begin{figure}[h!]
\begin{overpic}[scale=.5,unit=1mm]{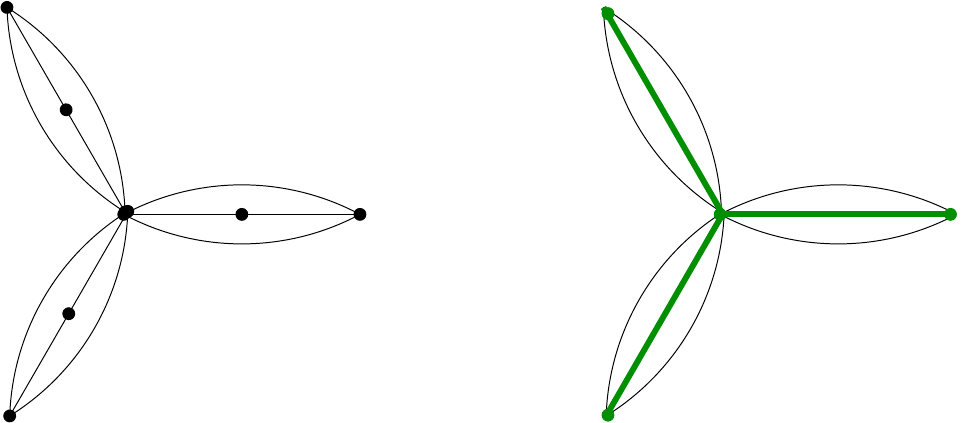} 
\end{overpic}
\caption{A $2$-dimensional three-leaved clover complex of a reduced star.}
\label{fig:3-leaved_clover}
\end{figure}

Finally, we define a clover map $f^{\cl}\colon |\Clover_f(v_0;w_j)| \to \bS^n$ associated to a $K$-Alexander map $f\colon |K|\to \bS^n$ as follows.

\begin{definition} \index{clover map $f^{\cl}$}
\label{def:clover-map}Let $v_0$ be a $w_i$-vertex in a $(K,f)$-Alexander $n$-complex, and 
$\Clover_f(v_0;w_j)$ be a clover complex of a reduced star $\Star_f(v_0;w_j)$ for $j\ne i$. A simplicial map 
\[
f^{\cl}\colon \Clover_f(v_0;w_j) \to \bS^n
\]
is a \emph{clover map associated to $f$ at $v_0$} if 
\begin{enumerate}
\item $f^\cl$ is an Alexander map in the sense that, for a leaf $L=T\cup T'$ in $\Clover_f(v_0;w_j)$, the images $f^\cl(T)$ and $f^\cl(T')$ are opposite hemispheres,
\item $(f^{\cl})^{-1}(w_i)$ is the set of centers of the leaves in $\Clover(v_0;w_n)$, 
\item $(f^{\cl})^{-1}(w_j) = \{v_0\}$, and $(f^{\cl})^{-1}(w_\ell) =f ^{-1}(w_\ell)$ for $\ell\ne i,j$.
\end{enumerate}
\end{definition}

\begin{remark}
The restriction of a clover map $f^{\cl} \colon |\Clover_f(v_0;w_j)|\to \bS^n$ to a leaf $L$ of $\Clover_f(v_0;w_j)$ is a topological simple cover in the sense that $f^{\cl}(|L|) = \bS^n$ and that $f^{\cl}_{\interior |L|} \colon \interior |L|\to \bS^n$ is an embedding.
\end{remark}

Clover complexes are examples of weakly simplicial complexes.

\begin{definition}
\label{def:weakly-simplicial-complex}\index{weakly simplicial complex}
A $\CW$-complex $P$ in an $n$-manifold $M$ is a \emph{weakly simplicial $n$-complex} if  
\begin{enumerate}
\item $M$ is the union of $n$-cells in $P$,
\item for each $n$-cell $\sigma\in K$,  $P|_\sigma$ is an $n$-simplex,
\item the restriction $P|_{T\cap T'}$ to the intersection of any two adjacent $n$-simplices $T$ and $T'$  is a well-defined $(n-1)$-dimensional simplicial complex, and
\item the $(n-1)$-skeleton $P^{(n-1)}$ is a simplicial complex.
\end{enumerate}
\end{definition}

Since a weakly simplicial $n$-complex $P$ consists of standard $n$-simplices, we say that a mapping $h\colon |P| \to \bS^n$ is $K$-Alexander if $h$ satisfies the usual definition for Alexander maps, i.e., $h$ is a simplicial map which maps adjacent $n$-simplices to opposite hemispheres; see Definition \ref{def:simplicial-Alexander-map}. In what follows, we say that the pair $(P,h)$ is a \emph{weakly simplicial Alexander $n$-complex}.

We may now formulate a version of Proposition \ref{prop:star-collapse} in which a star collapses to a clover complex. Since the proof is analogous to the proof of Proposition \ref{prop:star-collapse}, we omit the details.

\begin{proposition}
\label{prop:star-clover-collapse}
Let $(K,f)$ be a weakly simplicial Alexander $n$-complex, $n \ge 2$, $v\in K$ a $w_j$-tame $w_i$-vertex for $j\ne i$. Then there exists $\sL=\sL(n,\DStar_f(v;w_j))\ge 1$ and an $\sL$-BLD mapping $\Psi_{f,j} \colon |K - \Star_K(v)|\to |K|\setminus \interior |\Clover_f(v_0;w_j)|$ having the following properties:
\begin{enumerate}
\item $\Psi_{f,j}$ is the identity on $K-\DStar_f(v;w_j)$; 
\item $\Psi_{f,j}$ is a simplicial map from $\Link_K(v)$ onto $\partial \Clover_f(v_0;w_j)$; and
\item the union $K^\cl_{v,j} = \Clover_f(v_0;w_j) \cup (\Psi_{f,j})_*(K - \Star_K(v))$ is a well-defined weakly simplicial $n$-complex with space $|K|$.
\end{enumerate}
\end{proposition}

\begin{figure}[htp]
\begin{overpic}[scale=.45,unit=1mm]{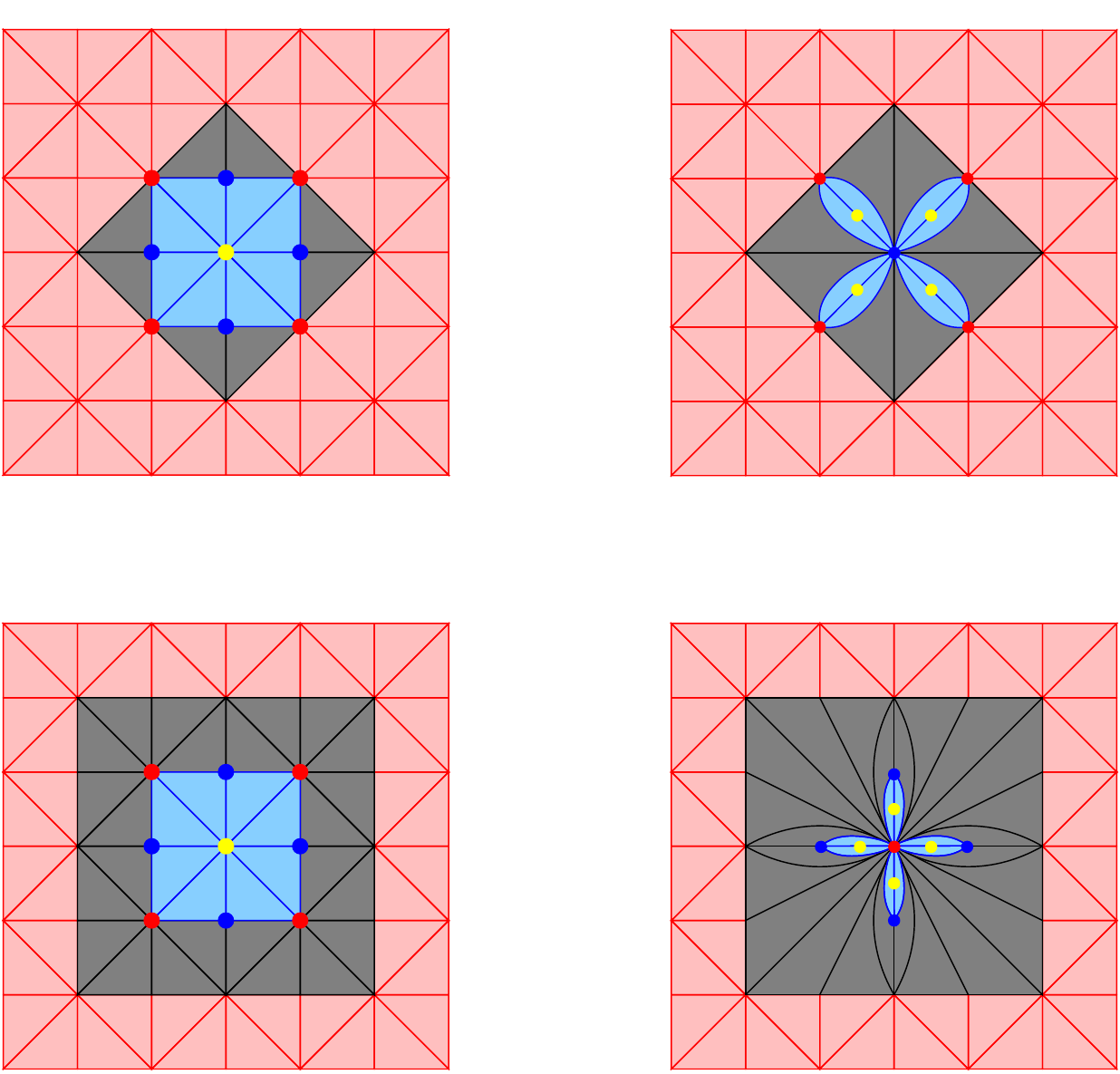} 
\put(38,35){\tiny $K$}
\put(95,35){\tiny $K^\cl_{v,2}$}
\put(38,85){\tiny $K$}
\put(95,85){\tiny $K^\cl_{v,1}$}
\end{overpic}
\caption{Complexes $K^\cl_{v,1}$ and $K^\cl_{v,2}$ associated to neighborhoods $\DStar_f(v;w_1)$ and $\DStar_f(v;w_2)$ of vertex $v$, respectively, in complex $K$ in Figure \ref{fig:DStar}.} 
\label{fig:deformation-ambient}
\end{figure}

We call the mapping $\Psi_{f,j} \colon |K - \Star_K(v)|\to |K|\setminus \interior |\Clover_f(v_0;w_j)|$ a \emph{pinch map at the vertex $v$} and the complex $K^\cl_{v,j}$ a \emph{clover reduction of $K$ at $v$}.\index{pinch map}

\begin{remark}
We do not state in Proposition \ref{prop:star-clover-collapse} that the complex $K^\cl_{v,j}$ admits an Alexander map $f^\cl_{v,j} \colon |K|\to \bS^n$ associated to the mapping $\Psi_{f,j}$. This is proven separately using deformation construction. 
\end{remark}

\subsection{Deformation in a single star}
\label{sec:local_deformation_star}

In the following proposition, we identify a mapping $X\times \{s\}\to Y\times \{s\}$ with a mapping $X\to Y$. We also say that a $w_i$-vertex $v$ in a weakly simplicial complex $K$ is \emph{$w_j$-tame} if $\DStar_f(v)$ is simplicial and $v$ is $w_j$-tame in the sense of Definition \ref{def:tame-vertex}.

\begin{proposition}
\label{prop:deformation-from-star-to-clover}
Let $(K,f)$ be a weakly simplicial Alexander $n$-complex, $n \ge 2$. Let also $v_0\in K$ be a $w_j$-tame $w_i$-vertex for $j\ne i$. Then there exist a constant $\sL''=\sL''(n,\DStar_f(v_0))\ge 1$, a clover reduction $K^\cl_{v_0,j}$ of $K$ at $v_0$, and a level preserving $\sL''$-BLD mapping $F \colon |K|\times [0,1] \to \bS^n\times [0,1]$ for which 
\begin{enumerate}
\item $F|_{|K|\times \{0\}} = f$ and
\item $F|_{|K|\times \{1\}} \colon |K|\to \bS^n$ is an $\sL''$-BLD $K^\cl_{v_0,j}$-Alexander map.
\end{enumerate}
\end{proposition}

We say that the mapping $f^\cl_{v_0,j} = F|_{|K|\times \{1\}} \colon |K|\to \bS^n$ is obtained from $f$ by \emph{pinching}.

We construct the mapping $F$ using a flow method for $n$-simplices in $\Star_K(v_0)$; see Figure \ref{fig:Deformation_3d_v2} for an illustration. The method does not produce simplicial complexes on levels $|K|\times \{s\}$.

\begin{figure}[htp]
\begin{overpic}[ scale=.68,unit=1mm]{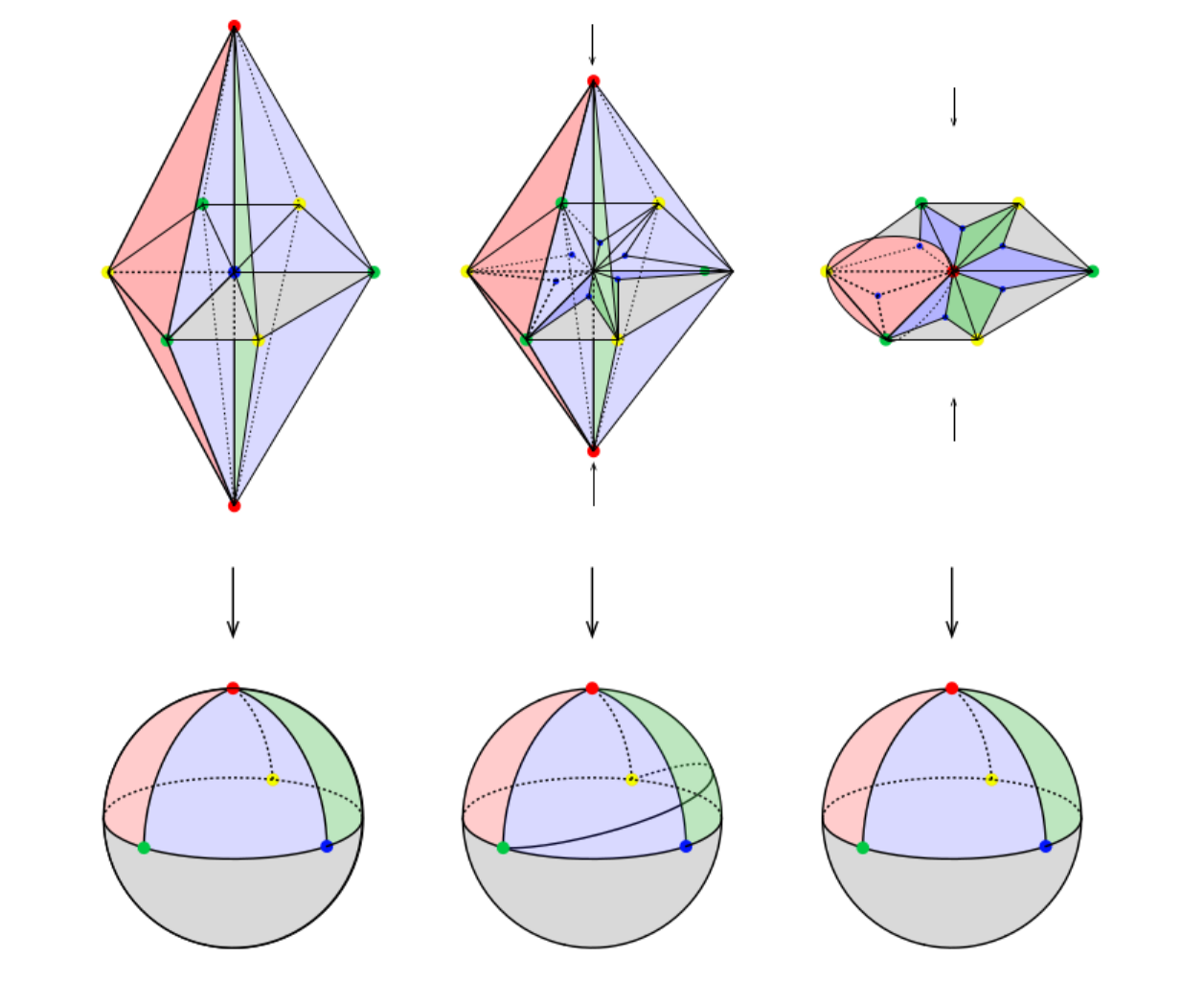} 
\put(8,28){\tiny $\bS^3_+$}
\put(26,22){\tiny $\bS^3_-$}
\put(41,10){\tiny $\bS^2$}

\put(30,47){\tiny $t=0$}
\put(73,47){\tiny $t=1/2$}
\put(116,47){\tiny $t=1$}
\end{overpic}
\caption{Deformation of a $3$-dimensional Alexander map on a star with two $w_3$-vertices, to an Alexander map on  a clover complex. On the domain side, $10$ of the $12$ outer faces (in red) are removed for viewing.}
\label{fig:Deformation_3d_v2}
\end{figure}

\begin{proof}[Proof of Proposition \ref{prop:deformation-from-star-to-clover}]

We begin the proof by introducing flows $s\mapsto T(s)$ of simplices $T=T(0)\in \DStar_f(v_0;w_j)^{[n]}$, which transform the simple pairs $\scrS_f(v_0;w_j)$ in $\Star_K(v_0)$ simultaneously to leaves of $\Clover(v_0,w_j)$. In order to simplify notations, we fix $i=0$ and $j=n$ throughout this construction. The complex $K^\cl_{v,n}$ and the mapping $F$ are given by this flow.

\medskip
\noindent
\emph{Step 1: (Flow on $n$-simplices in $\Star_K(v_0)$)} 
Let $T$ be an $n$-simplex in $\Star_K(v_0)$. We denote $x_T$ the vertex of $T$ for which $f(x_T)=w_n$ and  $\tau=[v_0,x_1,\ldots, x_{n-1}]$ the $w_n$-avoiding face in $T$. We denote also 
\[
y_\tau = (v_0+x_1+\cdots +x_{n-1})/n
\]
the barycenter of $\tau$.

For each $s\in [0,1]$,  let
\[
y_\tau(s) = (1-s)v_0 + s y_\tau\in \tau,
\] 
and $h \colon \{x_T\} \times [0,1] \to \Star(v_0)$ be the map $(x_T,s)\mapsto (1-s)x_T + s v_0$.
Set also, for each $s\in [0,1]$,
\[
\tau_T(s) = [y_\tau(s),x_1,\ldots, x_{n-1}] \,\subset\, \tau
\]
and, for each $j\in \{1,\ldots, n-1\}$,
\[
\begin{split}
\sigma_{T,j}(s) = & [v_0, x_1,\ldots, x_{j-1}, y_\tau(s), x_{j+1},\ldots, x_{n-1}] \\
& \cup [v_0, x_1,\ldots, x_{j-1}, h(x_T,s), x_{j+1},\ldots, x_{n-1}].
\end{split}
\]
Note that, for each $j\in \{1,\ldots, n-1\}$, $\sigma_{T,j}(s)$ is an $(n-1)$-cell, and the set
\[
\beta_T(s) = \tau_T(s) \cup (\bigcup_{j=1}^{n-1}\sigma_{T,j}(s))
\]
is also an $(n-1)$-cell.

\begin{figure}[h!]
\begin{overpic}[scale=.75,unit=1mm]{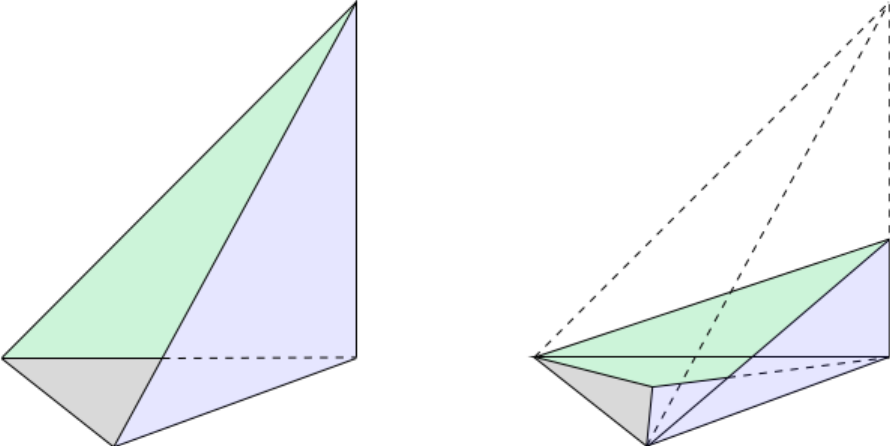} 
\put(-2,10){\tiny $x_1$}
\put(13,-2){\tiny $x_2$}
\put(46,10){\tiny $v_0$}
\put(46,57){\tiny $x_T$}

\put(66,10){\tiny $x_1$}
\put(81,-2){\tiny $x_2$}
\put(113,10){\tiny $v_0$}
\put(113,57){\tiny $x_T$}

\put(113,26){\tiny $h(x_T,s)$}
\put(81,9){\tiny $y_\tau(s)$}
\end{overpic}
\caption{Boundary of a $3$-dimensional $ T$ with face $[x_1,x_2, x_T]$ removed; cells $\tau_T(s)$, $\sigma_{T,1}(s)$, and $\sigma_{T,2}(s)$.}
\label{fig:y_tau}
\end{figure}

Finally, we choose, for each $s\in [0,1]$,  an $(n-1)$-cell $\sigma_T(s)$ contained in $T$ so that the boundary of $\sigma_T(s)$ coincides with the boundary of $\beta_T(s)$ and that the union
\[
\eta(s)=\tau_T(s) \cup \left( \bigcup_{j=1}^{n-1} \sigma_{T,j}(s) \right) \cup \sigma_T(s)
\]
is the  boundary of an $n$-cell contained in $T$. Then $\eta(s)$ is homeomorphic to an $(n-1)$-sphere and hence bounds an $n$-cell, which we denote $T(s)$. 

To fulfill these requirements, we let $z=(x_T+x_1+\cdots + x_{n-1})/n$ and $z(s)= (1-s)x_T +s z$. We may take, for example,  $\sigma_T(s)$ to be the graph of a PL map over $\tau$, which is affine on  each of these simplices,  $\tau_T(s)$ and  $[v_0, x_1,\ldots, x_{j-1}, y_\tau(s), x_{j+1},\ldots, x_{n-1}]$, $j=1,\ldots,n-1$ and which maps points $x_1,x_2,\ldots, x_{n-1}, v_0, y_\tau(s)$ to  points $x_1,x_2,\ldots, x_{n-1}, h(x_T,s)$ and 
\[
\frac{s}{2} y_\tau(s)+(1-\frac{s}{2}) z(s),
\]
respectively. Since $\sigma_T(s)$ is chosen as a graph of a PL map over $\tau$, we have that each $T(s)$ is an $n$-cell and, further, choose $\sigma_T(s)$ so that $T(s)$ is $L$-bilipschitz homeomorphic to $T$ with $L=L(n,\Star_K(v_0))$.

By these choices the $n$-cells $T(s)$ for $s\in [0,1]$ we have that $T(0)=T$ and that $T(s)\supset T(s')$ for $0 \le s < s' \le 1$. Moreover, there exists a level preserving $\check L$-bilipschitz embedding 
\[
\varphi_T \colon |T|\times [0,1]\to |T|\times [0,1]
\]
 satisfying $\varphi_T(|T|\times \{s\}) = T(s)\times \{s\}$ for each $s\in [0,1]$, where $\check L = \check L(n)$.

For $s\in [0,1]$, we give now the $n$-cell $T(s)$ a simplicial structure with vertices $y_\tau(s)$, $x_1,\ldots, x_{n-1}$, and $h(x_T,s)$; see Figure \ref{fig:y_tau}.

Although that the $n$-simplices $T(s)$ do not induce even a weakly simplicial structure on $\bigcup_{T\in \Star_K(v)^{[n]}} T(s)$, the above choices yield that the neighboring $n$-simplices in $\bigcup_{T} T(s)$ for each $s\in [0,1]$ are \emph{compatible} in the following sense:
\begin{enumerate}[label=(C\arabic*)]
\item   \label{item:neighboring1} If $T=[v_0,x_1,\ldots,x_n]$ and $T' \in \Star_K(v_0)$ is an $n$-simplex  which shares an $w_n$-avoiding $(n-1)$-simplex $\tau$ with $T$, then
\begin{enumerate}
\item $\tau_T(s) = \tau_{T'}(s)$, and
\item $\sigma_{T,j}(s) \cap \tau = \sigma_{T',j}(s) \cap \tau \quad \text{for} \,j =1,\ldots, n-1.$
\end{enumerate}
\item \label{item:neighboring2}
If $T=[v_0,x_1,\ldots,x_n]$ and $T'' =[v_0,x_1,\ldots,x_{\ell-1},x''_{\ell}, x_{\ell+1},\ldots, x_n]$  are $n$-simplices sharing a face $\xi = [v_0,x_1,\ldots,x_{\ell-1},x_{\ell+1},\ldots, x_n]$, then
\[
\sigma_{T,\ell}(s) \cap \partial T'' = \sigma_{T'',\ell}(s)\cap \partial T.
\]
\item \label{item:neighboring3} If $T$ and $T''$ have a common face $\xi$, we may further choose $\varphi_T$ and $\varphi_{T''}$ so that $\varphi_T|_{|\xi|\times [0,1]} = \varphi_{T''}|_{|\xi|\times [0,1]}$.
\end{enumerate}

This defines the flow $s\mapsto T(s)$ of the simplex $T$. For completeness, we denote $s\mapsto T(s)|_\sigma$ the restriction of this flow to a face $\sigma$ of $T$. Note that, by construction, adjacent $n$-simplices induce identical flows on the common face.

By construction, we now have that, for $s=1$, 
\[
\Clover_f(v_0;w_n) = \bigcup_{T\in \Star_K(v_0)} T(1).
\]

\medskip
\noindent
\emph{Step 2: (Flow of $K$)} 
Since flow  of $\Star(v_0)$ has already been defined, we now consider $n$-simplices $T$ in $K - \Star_K(v)$. 

Suppose that $T$ belongs to $K-\DStar_f(v_0;w_n)$. We take $s\mapsto T(s)$ to be the constant flow $s\mapsto T$. 

Suppose next that $T$ belongs to $\DStar_f(v_0;w_n)- \Star(v_0)$ and 
has a face in $\Link_K(v_0)$. Then there exists a unique $n$-simplex $T'$ in $\Star_K(v_0)$ having a common $(n-1)$-face $\sigma$ with $T$. We may define the flow $s\mapsto T(s)$ of $T$ by requiring that the restriction of the flows $s\mapsto T(s)|_\sigma$ agrees with the already defined flow $s\mapsto T'(s)|_\sigma$. Then each $T(s)$ is an $n$-cell and there exists a level preserving $L'$-bilipschitz embedding $\varphi_T \colon |T|\times [0,1]\to |T\cup T'|\times [0,1]$ for which $\varphi_T(|T|\times \{s\}) = T(s)\times \{s\}$ for each $s$, where $L'=L'(n,\DStar_f(v_0;w_n))$.

Finally we consider the $n$-simplices $T$ in $\DStar_f(v_0;w_n) - \Star(v_0)$ which do not have face in $\Link_K(v_0)$, and denote the collection of these $n$-simplices by $\mathcal T$.  Since flows of all vertices in $K$ have already been fixed, we extend the flows of vertices into each $T\in \mathcal T$  to obtain a flow $s\mapsto T(s)$ with the properties that for each $T'\in K^{[n]}$ adjacent to $T$,  $s\mapsto T(s)$ agrees with $s\mapsto T'(s)$ on their common face and that  there exists 
 a level preserving $L''$-bilipschitz embedding 
 \[\varphi_{\mathcal T} \colon \cup_{T\in \mathcal T} |T|\times [0,1]\to |K|\times [0,1],\] where $L''=L''(n,\DStar_f(v_0;w_n))$.

Since the flows of adjacent $n$-simplices $T$ and $T'$ in $K$ agree on the common face, the flows on individual $n$-simplices define a flow $s\mapsto K(s)$ of the complex $K$, where $K(s) = \bigcup_{T\in K^{[n]}} T(s)$. 

We conclude that, for $s=1$, $K(1)$ is a weakly simplicial complex, which is a clover reduction of $K$ at $v$. More precisely, the flow $s\mapsto K(s)$ at $s=1$ induces a pinch map $\Psi_{f,n}\colon |K-\{v_0\}| \to K\setminus \interior |\Clover_f(v_0;w_n)|$ by formula $\Psi_{f,n}(x) = \varphi_T(x,1)$ for $x\in |T|$ for each $T\in (K-\{v_0\})^{[n]}$. Thus  
\[
K(1) = \Clover_f(v_0;w_n) \cup (\Psi_{f,n})_*(K-\{v_0\})
\]
is a clover reduction $K^\cl_{v,j}$ of $K$ at $v$.

\medskip
\noindent
\emph{Step 3: (Construction of the mapping $F$)} For each $s\in (0,1]$, we fix a family $(\alpha_{T,s} \colon T(s)\to \bS^n)_{T\in K^{[n]}}$ of $\widetilde L$-bilipschitz embeddings by the formula 
\[\alpha_{T,s}(\pr_1(\varphi_T(x,s)) = f(x), \quad (x,s)\in |T|;\]
 here $\widetilde L=\widetilde L(n,L, L', L'')$, and $\pr_1$ is the projection to the first coordinate. By the compatibility conditions for embeddings $\varphi_T$, we have, for adjacent $T$ and $T'$, that $\alpha_{T,s}|_{T(s)\cap T'(s)} = \alpha_{T',s}|_{T(s)\cap T'(s))}$.

We define a map $F_s \colon |K| \to \bS^n$ by the formula $F_s|_{T(s)} = \alpha_{T,s}$ for each $T\in K^{[n]}$. Then $F_s$ is $\widetilde L$-BLD. By changing the parametrization of $n$-simplices in $K^\cl_{v,j}$, we may assume that $F_1 \colon |K|\to \bS^n$ is a $K^\cl_{v,j}$-Alexander map.

We define now a mapping
\[
F \colon |K|\times [0,1]\to \bS^n \times [0,1],
\]
by $F(x,s) = F_s(x)$ for each $(x,s)\in |K|\times [0,1]$. Then  $F$ is a BLD mapping.
This completes the proof.
\end{proof}

\subsection{Detachment of simple covers}
\label{sec:star_clover_simple_cover}

We may detach now all the leaves of the clover $\Clover_f(v_0;w_j)$ from $K^\cl_{v,j}$, and continue to detach if more simple covers appear, to obtain a collapse $K^\st_{v,j}$ of $K$ and a controlled BLD expansion of a $K^\st_{v,j}$-Alexander map; see Figure \ref{fig:clover-detachment} 
for an illustration.

\begin{figure}[htp]
\begin{overpic}[scale=.45,unit=1mm]{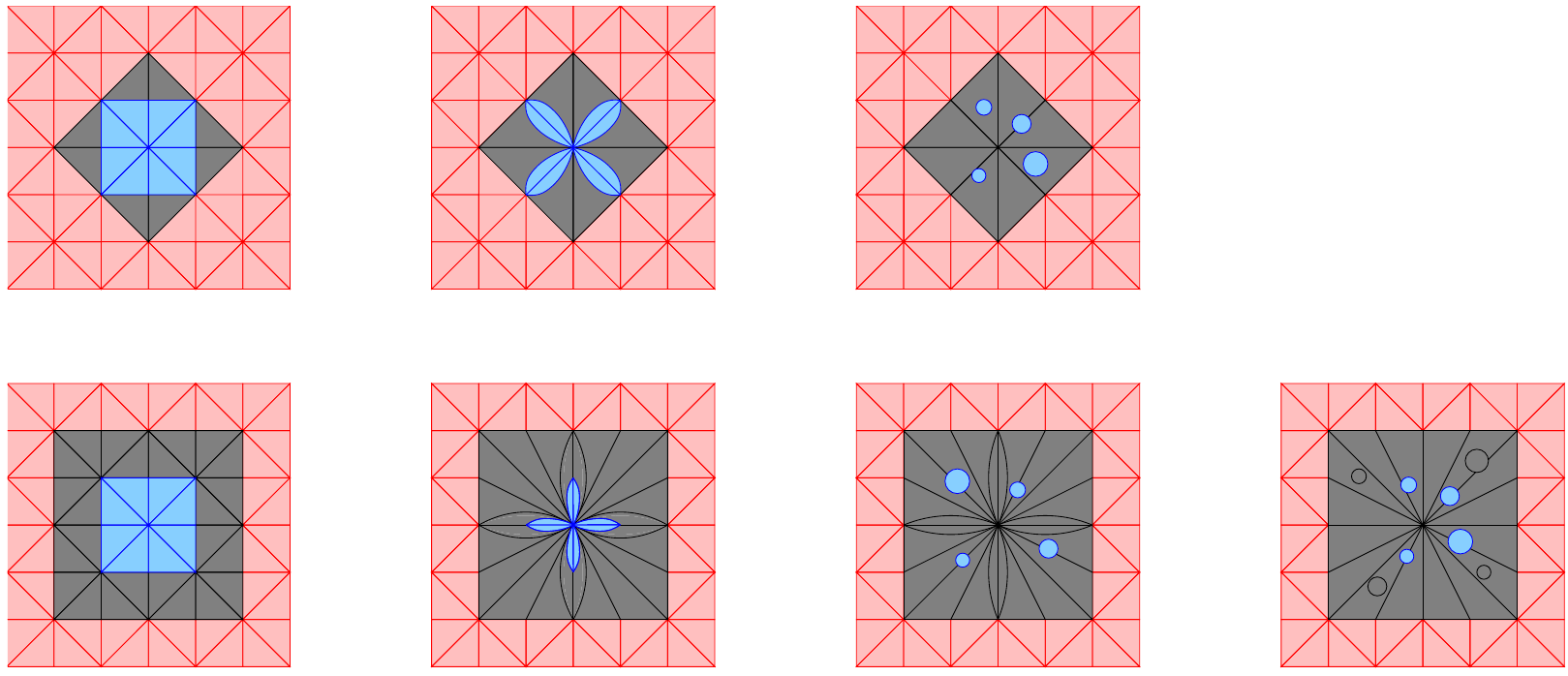} 
\put(24,21){\tiny $K$}
\put(57,21){\tiny $K^\cl_{v,2}$}
\put(91,50){\tiny $K^\st_{v,1}$}
\put(24,50){\tiny $K$}
\put(57,50){\tiny $K^\cl_{v,1}$}
\put(124,21){\tiny $K^\st_{v,2}$}
\end{overpic}
\caption{Reduction from $K$ to $K^\st_{v,1}$ and $K^\st_{v,2}$ through complexes $K^\cl_{v,1}$ and $K^\cl_{v,2}$, respectively, for neighborhoods $\DStar_f(v;w_1)$ and $\DStar_f(v;w_2)$ in Figure \ref{fig:DStar}. Here the simple covers are not arranged to give a controlled expansion due to the limitation of planar illustration.} 
\label{fig:clover-detachment} 
\end{figure}

\newcommand{\sLdetach}{\sL_{\mathrm{detach}}}   

In the following statement, we suppress the dependence of the complex $K^\cl_{v,j}$ on the pinching map $\Psi_{f,j}\colon |K - \Star_K(v)|\to |K|$.

\begin{proposition}
\label{prop:clover-detachment}
Let $(K,f)$ be a weakly simplicial Alexander $n$-complex, $n \ge 2$,
$v\in K$ a $w_j$-tame $w_i$-vertex for $j\ne i$. Let also $f^\cl_{v,j} \colon |K|\to \bS^n$ an $\sL$-BLD $K^\cl_{v,j}$-Alexander map obtained from $f$ by pinching. Then there exist $\sLdetach=\sLdetach(n,\sfK, \Star_{K^\cl_{v,j}}(v)))\ge 1$ and a level preserving $\sLdetach$-BLD map $F\colon |K|\times [0,1]\to \bS^n \times [0,1]$ for which
\begin{enumerate}
\item $F|_{K\times \{0\}} = f$ and
\item $F|_{K\times \{1\}} \colon |K|\to \bS^n$ is an $\sLdetach$-BLD-controlled expansion of a $K^\st_{v,j}$-Alexander map $|K|\to \bS^n$.
\end{enumerate}
Moreover, if $|K| \subset \R^n$, $|\DStar_f(v)|$ is a convex subset, and the flat structure of $K$ consists of affine maps, then $K^\st_{v,j}$ may be chosen so that its flat structure consists of affine maps.
\end{proposition}

\begin{proof}
It suffices to observe that leaves $L$ of the clover $\Clover_f(v;w_j)$ form a simple cover in the sense that $f^\cl_{v,j}|_{\interior |L|} \colon \interior |L|\to \bS^n$ is embedding and $f^\cl_{v,j}|_{\partial |L|} \colon \partial |L|\to \bS^n$ is a tame PL $(n-1)$-cell in $\bS^n$. Thus, by fixing a framing for each leaf, we may shrink each leaf by Proposition \ref{prop:shrink} 
to obtain an isotopy from $f^\cl_{v,j}$ to a controlled BLD expansion $\widetilde f \colon |K|\to \bS^n$ of a $\sK^\st_{v,j}$-Alexander map $f^\st_{v,j} \colon |K|\to \bS^n$. 
The isotopy yields a level preserving BLD map $F\colon |K|\times [0,1] \to \bS^n \times [0,1]$ in the statement. Since $\Star_K(v)^{[n]}$ has finitely many elements, it suffices to repeat this detachment of leaves finitely many times if necessary to obtain complex $K^\st_{v,j}$. 

Under the additional assumptions, we may choose the collapse map $|K - \Star_K(v)|\to |K|$ to be affine on simplices. This induces the wanted flat structure on $K^\st_{v,j}$.
\end{proof}

We finish this section by observing that Propositions  \ref{prop:deformation-from-star-to-clover} and \ref{prop:clover-detachment} together immediately yield Theorem \ref{thm:local-deformation}. 

\begin{proof}[Proof of Theorem \ref{thm:local-deformation}]
It suffices to combine the BLD mappings $F_1 \colon |K|\times [0,1]\to \bS^n \times [0,1]$ and $F_2 \colon |K|\times [0,1]\to \bS^n \times [0,1]$ given by Propositions \ref{prop:deformation-from-star-to-clover} and \ref{prop:clover-detachment}, respectively, into a BLD mapping $F\colon |K|\times [0,1]\to \bS^n \times [0,1]$ given by formula 
\[
(x,t) \mapsto \left\{ \begin{array}{ll}
F_1(x,2t). & t\in [0,1/2] \\
F_2(x,2t-1), & t\in [1/2,1].
\end{array}\right.
\]
This concludes the proof.
\end{proof}

In what follows, we use notations $K\searrow K^\st_{v,j}$ and $f \searrow f^\st_{v,j}$, when we refer to a collapse of the complex $K$ at $v$ and to the corresponding deformation of a $K$-Alexander map $f \colon |K|\to \bS^n$ to a BLD-controlled expansion of a $K^\st_{v,j}$-Alexander map provided by Theorem \ref{thm:local-deformation}; see Figure \ref{fig:simple_Alexander_merge} for an illustration. 

\begin{figure}[h!]
\begin{overpic}[scale=0.6,unit=1mm]{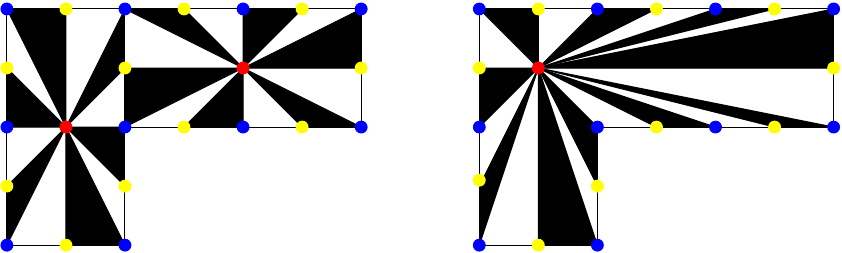} 
\put(13,20){\tiny $v$}
\put(55,21){\tiny $v$}
\end{overpic}
\caption{Collapse $K \searrow K^\st_{v,j}$ of a simplicial complex at a vertex; location of the vertex $v$ moved due to due to the limitation of planar illustration.}
\label{fig:simple_Alexander_merge}
\end{figure}

\section{Local deformation of Alexander star pairs}
\label{sec:merge-Alexander-pair}

In what follows, we say that an Alexander complex $(K,f)$ is an \emph{Alexander star} if $K$ is a simplicial complex $K=\Star_K(v)$ of a vertex $v\in K$.

\begin{definition}\label{def:Alexander_pair} 
A pair $\{(K_1,f_1),(K_2,f_2)\}$ of $n$-dimensional Alexander stars is an \emph{Alexander star pair} \index{Alexander star pair} if
\begin{enumerate}
\item $K_1\cap K_2$ is a simplicial complex whose space $|K_1\cap K_2|$  is an $(n-1)$-cell, 
\item $K_1\cap K_2$ has a vertex in $\interior |K_1\cap K_2|$, and
\item the mapping $f = f_1 \# f_2 \colon |K_1\cup K_2|\to \bS^n$, defined by $f|_{|K_i|}=f_i$ for $i=1,2$, is a well-defined Alexander map. \label{item:Alexander_pair-3}
\end{enumerate}
In this case, we say  $\{(K_1,f_1),(K_2,f_2)\}$ is a \emph{Alexander star pair} if $K_1\cap K_2$ is a star.
\end{definition}

We discuss below two methods of simplifying the Alexander complex $(K_1\cup K_2,f_1\# f_2)$. First is a merge of Alexander stars, which reduces the complex to a single Alexander star. The second is a reduction of the intersection $K_1\cap K_2$.

\subsection{Merge of Alexander stars}

The pair $(K_1\cup K_2, f_1\# f_2)$ is an Alexander complex, but is not an Alexander star. However, there is a star having the same boundary as $K_1\cup K_2$ which admits an Alexander map, hence is an Alexander star; see Figure \ref{fig:simple_Alexander_merge}
for an illustration. We record this as follows.

\begin{definition}
An Alexander star $(K,f)$ is called a \emph{merge of the Alexander star pair $\{(K_1,f_1),(K_2,f_2)\}$} if $|K|=|K_1\cup K_2|$, $K|_{\partial |K|} = (K_1\cup K_2)|_{\partial |K|}$, and $f|_{|\partial K|} = (f_1\# f_2)_{|\partial K|}$. \index{Alexander star pair!merge}
\end{definition}

Each Alexander star pair admits a merge.

\begin{lemma}
\label{lemma:star_pair_map}
Let $\{(K_1,f_1),(K_2,f_2)\}$ be an Alexander star pair, and let $K =\Star_K(v_K)$ be a star satisfying $|K|=|K_1\cup K_2|$ and $K|_{\partial |K|}=(K_1\cup K_2)|_{\partial |K|}$. Then there exists a $K$-Alexander map $f\colon |K|\to \bS^n$ for which $f|_{|\partial K|} = (f_1\# f_2)|_{|\partial (K_1\cup K_2)|}$. In particular, $(K,f)$ is an Alexander star.
\end{lemma}

\begin{proof}
Let $K_1 = \Star_{K_1}(v_{K_1})$ and $K_2=\Star_{K_2}(v_{K_2})$. Since $f_1$ and $f_2$ agree on $K_1\cap K_2$, we have that $f(v_{K_1}) = f(v_{K_2})$. Now, let $\theta \colon K^{[n]} \to (K_1\cup K_2)^{[n]}$ defined by the formula $\theta(T)\cap \partial (K_1 \cap K_2) = T\cap \partial K$.

We define now a $K$-simplicial map $g \colon |K|\to \bS^n$ defined by formulas $g(v_K)= f(v_{K_1})$, $g|_{\partial |K|} = f|_{\partial |K|} \colon \partial |K|\to \bS^n$, and $g(|T|) = (f_1\# f_2)(|\theta(T)|)$ for each $T\in K^{[n]}$. The mapping $g$ is well-defined, since $K=\Star_K (v_K)$ is a star. 

To show that $g$ is a $K$-Alexander map, let $T$ and $T'$ be adjacent $n$-simplices in $K$. We have two cases. If $\theta(T)$ and $\theta(T')$ are adjacent either in $K_1$ or in $K_2$, we have that $g(|T|)$ and $g(|T'|)$ are mapped to opposite hemispheres. Suppose now that $\theta(T)$ and $\theta(T')$ do not belong to the same complex $K_1$ or $K_2$. We may assume that $\theta(T)\in K_1$ and $\theta(T')\in K_2$. Since $T$ and $T'$ are adjacent, also $T\cap \partial K_1$ and $T'\cap \partial K_2$ are adjacent. Thus there exists an $(n-1)$-simplex $\sigma\in K_1\cap K_2$ which is adjacent to both $T\cap \partial K_1$ and $T'\cap \partial K_2$. Let $T_1$ and $T_2$ be $n$-simplices in $K_1$ and $K_2$, respectively, for which $T_1 \cap \partial K_1 = \sigma = T_2\cap \partial K_2$. Since $f_1\# f_2$ is an Alexander map, we have that $(f_1\#f_2)(T_1)$ and $(f_1 \# f_2)(T_2)$ are opposite hemispheres. Since $T_1$ is adjacent to $T$ and $T_2$ is adjacent to $T'$, we conclude that $g(|T|)$ and $g(|T'|)$ are opposite hemispheres. Thus $g$ is an Alexander map.
\end{proof}

A merge $(K,f)$ of an Alexander star pair $\{ (K_1,f_1),(K_2,f_2)\}$ can be realized as a collapse $K_1\cup K_2 \searrow K$ if $K_1\cap K_2$ is a star.

\begin{lemma}
\label{lemma:merge-of-Alexander-pair}
Let $\{(K_1,f_1),(K_2,f_2)\}$ be an Alexander star pair, where $K_1 = \Star_{K_1}(v_1)$, $K_2=\Star_{K_2}(v_2)$, and $v_1$ and $v_2$ are $w_j$-vertices. Suppose that $K_1\cap K_2 = \Star_{K_1\cap K_2}(v)$, where $v\in K_1\cap K_2$ is a $w_j$-tame $w_i$-vertex, and that  $(K,f)$ is a merge of $\{(K_1,f_1),(K_2,f_2)\}$. Then $(K_1\cup K_2)^\st_{v,j}$ is isomorphic to $K$. In addition, if $|K_1\cup K_2|$ is a convex subset of $\R^n$ and the flat structure of $K_1\cup K_2$ consists of affine maps, then we may choose $K^\st_{v,j}$ so that its flat structure consists of affine maps.
\end{lemma}

\begin{proof}
Since $v_1$ and $v_2$ are the only $w_j$-vertices in $K_1\cup K_2$, we have that $\Star_{K_1\cap K_2}(v)$ is the reduced star $\Star_{f_1\# f_2}(v, w_j)$. Thus $(K_1\cup K_2)^\st_{v,j}$ consists of $\Star_{K_1\cap K_2}(v)$ and the image of $(K_1\cup K_2)-\{v\}$ under a collapse map. Since the collapse map maps $v_1$ and $v_2$ onto $v$ and vertices $v_1$, $v_2$ and $v$ are the only vertices in the interior of $K_1\cup K_2$, we conclude that $v$ is the only vertex in the interior of $(K_1\cup K_2)^\st_{v,j}$. Since $(K_1\cup K_2)^\st_{v,j}$ coincides with $K_1\cup K_2$ on the boundary, we conclude that $(K_1\cup K_2)^\st_{v,j}$ is isomorphic to $K$. 

The second claim is a direct consequence of Proposition \ref{prop:clover-detachment}.
\end{proof}

In particular, we have the following corollary of Theorem \ref{thm:local-deformation};  the last claim here follows from the uniqueness of simplicial Alexander maps. 

\begin{corollary}
\label{cor:merge-of-Alexander-pair}
Let $\{(K_1,f_1),(K_2,f_2)\}$ be an Alexander star pair, where $K_1 = \Star_{K_1}(v_1)$ and $K_2=\Star_{K_2}(v_2)$, and $v_1$ and $v_2$ are $w_j$-vertices. Suppose that $K_1\cap K_2 = \Star_{K_1\cap K_2}(v)$, where $v\in K_1\cap K_2$ is a $w_j$-tame $w_i$-vertex. Let  $(K,f)$ be a merge of $\{(K_1,f_1),(K_2,f_2)\}$. Then there exist $\sL=\sL(n,K_1,K_2)\ge 1$ and an $\sL$-BLD mapping $F\colon |K|\times [0,1]\to \bS^n\times [0,1]$ for which $F|_{|K|\times \{0\}} = f_1\# f_2$ and $F|_{|K|\times \{1\}}$ is an $\sL$-BLD-controlled expansion of a $\widetilde K$-Alexander map, where $\widetilde K$ is isomorphic to $K$. In addition, if $|K_1\cup K_2|$ is a convex subset of $\R^n$ and the flat structure of $K_1\cup K_2$ consists of affine maps, then we may choose $\widetilde K = K$ and $F_{|K|\times \{1\}} = f$.
\end{corollary}

\subsection{Reduction of the intersection}

In another perspective, we may simplify an Alexander star pair $K_1\cup K_2$ by collapsing the intersection $K_1\cap K_2$ at a vertex; see Figure \ref{fig:Alexander_star_pair_reduction}. We use this observation to prove  reduction theorems by induction in dimensions. This observation is a direct corollary of Proposition \ref{prop:star-collapse} and we omit the details.

\begin{figure}[h!]
\begin{overpic}[scale=0.5,unit=1mm]{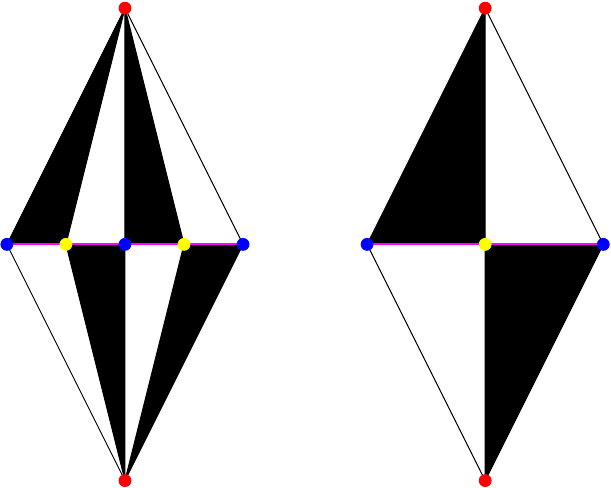}
\put(-1,30){\tiny $K_1$}
\put(-14,20){\tiny $K_1\cap K_2$}
\put(-1,10){\tiny $K_2$}
\put(11,18){\tiny $v$}
\put(39,18){\tiny $v$}
\put(54,20){\tiny $Q$}
\put(49,30){\tiny $P$}
\end{overpic}
\caption{Reduction $K_1\cap K_2\searrow Q$  yields a reduction \, $\{v_{K_1}, v_{K_2}\}* (K_1 \cap K_2) \searrow \{v_{K_1}, v_{K_2}\}*Q$ of one dimension higher,  as in Lemma \ref{lemma:Alexander_star_pair_reduction}}
\label{fig:Alexander_star_pair_reduction}
\end{figure}

\begin{lemma}
\label{lemma:Alexander_star_pair_reduction}
Let $\{(K_1,f_1),(K_2,f_2)\}$ be an $n$-dimensional Alexander star pair, where $K_1 = \Star_{K_1}(v_1)$, $K_2=\Star_{K_2}(v_2)$ and  $v_1$ and $v_2$ are $w_j$-vertices. Let $k\ne i,j$, and $v$ be a $w_k$-tame $w_i$-vertex in $K_1\cap K_2$ for which $|\Star_{K_1\cap K_2}(v)| \subset \interior |K_1\cap K_2|$. Then there exists a collapse $(K_1\cup K_2)^\st_{v,k}$ of $K_1\cup K_2$ at $v$ and a collapse map $\Phi_{f,k} \colon |(K_1\cup K_2)-\{v\}|\to |K_1\cup K_2|$ for which the restriction $\Phi_{f,k}|_{|(K_1\cap K_2)-\{v\}|} \colon |(K_1\cap K_2)-\{v\}| \to |K_1\cap K_2|$ is well-defined and $(K_1\cup K_2)^\st_{v,j}|_{|K_1\cap K_2|} = (\Phi_{f,k})_*((K_1\cap K_2)-\{v\})\, \supset \Star_f(v;w_k)$. 

In addition, if $|K_1\cup K_2|$ is a convex subset of $\R^n$ and the flat structure of $K_1\cup K_2$ consists of affine maps, then we may choose $K^\st_{v,k}$ so that its flat structure consists of affine maps.
\end{lemma}

Similarly as for the merge of stars, we have the following corollary of Theorem \ref{thm:local-deformation}.

\begin{corollary}
\label{cor:Alexander_star_pair_reduction}
Let $\{(K_1,f_1),(K_2,f_2)\}$ be an $n$-dimensional Alexander star pair, where $K_1 = \Star_{K_1}(v_1)$, $K_2=\Star_{K_2}(v_2)$ and both $v_1$ and $v_2$ are $w_j$-vertices. Let $K=K_1\cup K_2$. Let also $k\ne i,j$ and  $v$  a $w_k$-tame $w_i$-vertex in $K_1\cap K_2$ for which $|\Star_{K_1\cap K_2}(v)| \subset \interior |K_1\cap K_2|$. Then there exist $\sL=\sL(n,K_1,K_2)$ and an $\sL$-BLD map $F\colon|K|\times [0,1]\to \bS^n\times [0,1]$ for which $F|_{|K|\times \{0\}} = f_1\# f_2$ and $F|_{|K|\times \{1\}} \colon |K| \to \bS^n$ is an $\sL$-BLD-controlled expansion of a $K^\st_{v,k}$-Alexander map. In addition, if $K$ is a convex subset of $\R^n$ and the flat structure of $K$ consists of affine maps, then we may choose $K^\st_{v,k}$ so that its flat structure consists of affine maps.
\end{corollary}


\chapter{Deformation on shellable cubical complexes}
\label{sec:deformation_shellable_cubical_complexes}

As said in the Introduction, the flat structure on cubical complexes gives the stability needed for deformation of Alexander maps on triangulated cubical complexes locally. 
The shellability of cubical complexes  is a global condition which allows the underlying complexes and the maps to be reduced inductively.

\section{Shellable cubical complexes}
\label{sec:shellable_complexes}

In this section, we verify shellability of a cubical complex in three different cases: (1) cubical complexes of a $2$-cell, (2) refinements of an $n$-cube, and (3) cubical subcomplexes of $\partial ([0,1]^n)$ whose spaces are $(n-1)$-cells.
We also show from (3) that the boundary of a tunnel, modulo a single $(n-1)$-cube, is a shellable cubical $(n-1)$-complex. In the future, we use these complexes as building blocks of more complicated shellable complexes. We begin with the two dimensional result.

\begin{proposition}
\label{prop:reconstructibility_2-dim}
Let $K$ be a cubical complex for which $|K|$ is a $2$-cell. Then $K$ is shellable.
\end{proposition}

\begin{proof}
Let $\mathcal C$ be the collection of all $2$-cubes $q\in K^{[2]}$ whose intersection $q\cap \partial |K|$ with $\partial |K|$ contains at least one $1$-cube. We first show that there is a cube $q_1 \in \mathcal C$ for which $q_1 \cap \partial |K| $ is connected. It follows then that $\overline{|K|\setminus  q_1}$ is a $2$-cell.

Suppose that for all $q\in \mathcal C$ the intersection $q \cap \partial |K|$ is not connected, and let $q'\in \mathcal C$. Then $|K|\setminus q'$ has more than one components each of which meets $\partial |K|$ in a $1$-dimensional set. Let $D_1$ be any one of these components, and  $q''\in \mathcal C$ be a $2$-cube contained in $\overline{D_1}$.

Since $q''\cap \partial |K|$ is disconnected, $|D_1|\setminus q''$  has a component $D_2$ which does not meet $q'$ and whose intersection with $\partial |K|$ is $1$-dimensional. Let $q''' \in \mathcal C$ be a $2$-cube contained in $\overline{D_2}$. Since this process may be continued indefinitely and $\mathcal C$ is finite, we conclude that there exists a cube $q_1\in \mathcal C$ for which $q_1\cap \partial |K|$ is connected.

Now let $m=\# K^{[2]}$, $K_m=K$, and $K_{m-1}$ the subcomplex of $K$ having space $\overline{|K|\setminus  q_1}$. The claim now follows by induction.
\end{proof}

The iterative application of the following lemma gives shellability of all refinements $\Refine^k(Q)$ of an $n$-cube $Q$.

\begin{lemma}
\label{lemma:refinement_shellability}
Let $n\geq 1$ and $Q$ be an $n$-cube. Then the refinement $\Refine(Q)$ of $Q$ is a shellable cubical complex.
\end{lemma}
\begin{proof}
Since $\Refine(Q)$ is isomorphic to the Euclidean $n$-complex $[0,3]^n$ in which $n$-cubes are unit cubes whose vertices are integer lattice points, it suffices to prove the shellability for the complex $P=[0,3]^n$.

Let $n\geq 1$, and $q_1=[1,2]^n$ be the center $n$-cube in the complex $P=[0,3]^n$. Let, for each $n$-cube $q \in P^{[n]}$,  $\delta(q)$ be the dimension of the intersection $q\cap q_1$. Observe that any ordered sequence,
\[
q_1, q_2, \ldots, q_{3^n-1}, q_{3^n},
\]
of the $n$-cubes in $P$ which satisfies $\delta(q_{i+1})\leq \delta(q_i)$ for $i=1,\ldots, 3^n-1$, gives the defining property for the shellability.
\end{proof}

\begin{lemma}
\label{lemma:shellability_faces}
Let $K_Q$ be a cubical complex on an $n$-cube $Q$ that is isomorphic to the standard cubical structure of the Euclidean cube $[0,1]^n$. Let $P$ be a subcomplex of $K_Q|_{\partial Q}$ whose space $|P|$ is an $(n-1)$-cell. Then $P$ is a shellable cubical $(n-1)$-complex.
\end{lemma}
\begin{proof}
For each $(n-1)$-cube $q$ in $K_Q|_{\partial Q}$, denote $\hat q$ the unique cube in $(K_Q|_{\partial Q})^{(n-1)}$ opposite to $q$, that is, $q\cap \hat q = \emptyset$. Then for each $q'\in (K_Q|_{\partial Q})^{(n-1)}\setminus \{q,\hat q\}$, the  intersection $q\cap q'$ is an $(n-2)$-cube. Since $|P|$ is an $(n-1)$-cell, there exists a face $q_0$ of $Q$ for which $q_0\in P$ but $\hat q_0\not \in P$.

We fix now a labeling, $q_0,q_1,\ldots,q_m$, of the $(n-1)$-cubes in $P$. Since all other cubes connect to $q_0$ in an $(n-1)$-face, the sequence $q_0, \ldots,  q_m$ satisfies the defining property for the shellability.
\end{proof}

\begin{corollary}
\label{cor:reconstructibility_faces}
Let $T$ be a tunnel and let $q\subset \partial T$ be an $(n-1)$-cube. Then $\partial T - q$ is a shellable cubical $(n-1)$-complex. 
\end{corollary}
\begin{proof}
We prove the claim by induction on the size of adjacency graph $\Gamma(T)$ of $T$. Let $Q_0\in T^{[n]}$ be the (unique) $n$-cube having $q$ as a face. Since $\Gamma(T)$ is tree, we may give $\Gamma(T)$ a partial order in which $Q_0$ is the root. 

Let $Q$ be a leaf of $\Gamma(T)$ and $T'$ be the tunnel $T'=T-Q$. Since $\# \Gamma(T') = \# \Gamma(T) -1$, we have by induction assumption that $\partial T'-q $ is shellable. Let $q'_0,\ldots, q'_k$ be a labeling of $(n-1)$-cubes in $T'$ for which $q'_k \cap (q'_0\cup \ldots \cup q'_{k-1})$ is an $(n-2)$-cell for each $k$. Let also $\ell$ be such that $q'_\ell=T'\cap Q$. 

We label now the $(n-1)$-cubes $q''_1,\ldots, q''_{2n-1}$ of $\partial Q-q'$ so that $q''_1$ has a face in $q'_\ell$, that $q''_2$ does not meet $q'_\ell$, and that $q''_{2n-1}$ has a face in $q'_{\ell+1}$. Then the sequence 
\[
q'_0, q'_1,\ldots, q'_{\ell-1}, q''_1,q''_2,\ldots, q''_{2n-1}, q'_{\ell+1},\ldots, q'_k
\]
satisfies the shellability property. Thus
\[
\partial T - q  = (\partial T' - q) \cup (\partial Q - (T'\cap Q))
\]
is shellable. 
\end{proof}

\begin{remark}\label{rmk:non_shellable}
Shellability for simplicial complexes has a long history.
It is well-known that not all triangulations of a tetrahedron are shellable; see  M.\;Rudin \cite{Rudin_MaryEllen}. To the other direction, it is a result of Frankl \cite{Frankl} that every triangulation of a $2$-cell is shellable, and a result of Sanderson \cite{Sanderson-PAMS} that every triangulation of a $3$-cell has a shellable subdivision.

Not all cubical $n$-complexes of cells are shellable for any $n\geq 3$. See  Bing \cite[Example 2]{Bing_Poincare} for a $3$-dimensional example;  higher dimensional examples may be obtained by taking the product of Bing's example with Euclidean cubes.
\end{remark}

\section{Proof of the Quasiregular deformation theorem}

We are now ready to prove Theorem \ref{thm:QR-deformation-star} stated in the beginning of this part.

\Quasiregulardeformation*

\begin{convention}
We assume in what follows, as we may, that a cubical Alexander map $f\colon |K|\to \bS^n$ has the property that its $w_k$-vertices are barycenters of $k$-cubes; see Section \ref{sec:triangulation}. In particular, vertices at the centers of $n$-cubes are $w_n$-vertices.
\end{convention}

The proof of Theorem \ref{thm:QR-deformation-star} is by double induction -- first on the dimension and then on the complexity of the cell. We first prove the theorem in dimension two, which  serves as the initial step in the induction by dimension.

\begin{lemma}
\label{lemma:cubical_deformation_dim_2}
Let $K$ be a cubical $2$-complex whose space is a cell. Then there exist a star replacement $K^*$ of $K$ and constants $\sL^\dagger=\sL^\dagger(K^*)\ge 1$ and $\sL'=\sL'(K)\ge 1$ for the following.
Let  $f \colon |K|\to \bS^2\times \{0\}$ be an $\sL$-BLD-controlled expansion of a $K^\Delta$-Alexander map for $\sL\ge 1$.
Then there exists a $\max\{ \sL', \sL\}$-BLD map $F\colon |K|\times [0,1]\to \bS^2\times [0,1]$ for which 
\begin{enumerate}
\item $F|_{|K|\times \{0\}} = f$, 
\item $F(x,t)=f(x,0)$ for $(x,t)\in |\partial K|\times [0,1]$, and 
\item $F|_{|K|\times \{1\}}\colon |K|\to \bS^2\times \{1\}$ is $\max\{\sL^\dagger,\sL\}$-BLD-controlled expansion of a $K^*$-Alexander map. 
\end{enumerate}
In addition, if $|K|$ is a convex subset of $\R^2$ and flat structure of $K$ consists of affine maps, then we may choose $K^*$ so that its flat structure consists of affine maps.
\end{lemma}

\begin{remark}
For every $\ell\in \N$, there exists finitely many isometry classes of cubical $2$-complexes $K$ for which $\# K^{[2]} = \ell$. Thus the constants $\sL^\dagger(K^*)$ and $\sL'(K)$ depend only on the size of $K$.  
\end{remark}

\begin{proof}[Proof of Lemma \ref{lemma:cubical_deformation_dim_2}]
We prove the claim by induction on the number $\ell = \# K^{[2]} \ge 1$ of $2$-cubes in $K$. The claim clearly holds when $\# K^{[2]}=1$ and, in this case, $K^* = K^\Delta$.

Assume now that $\ell\ge 2$ and that the claim holds for all complexes $K'$ satisfying $\# (K')^{(2)} \le \ell$. Let $K$ be a cubical complex with $\# K^{[2]} = \ell+1$.

\medskip
\noindent
\emph{Case 1:}
Suppose that $f\colon |K|\to \bS^2$ is a $K^\Delta$-Alexander map. By Proposition \ref{prop:reconstructibility_2-dim}, $K$ is shellable.
Hence there exist a subcomplex $\widetilde K$ of $K$ and a $2$-cube $\widetilde q \in K$ for which $K = \widetilde K \cup \widetilde q$ and $|\widetilde K\cap \widetilde q|$ is an $1$-cell. Thus $\widetilde q\cap \partial K$ contains either one,  two, or three faces of $\widetilde q$.  Let $\widetilde f = f|_{|\widetilde K|} \colon |\widetilde K|\to \bS^2$ and note that $\widetilde f$ is a $\widetilde K^\Delta$-Alexander map.

Since $\#\widetilde K^{[2]} = \ell$, the induction assumption yields a star replacement ${\widetilde K}^*$ and an $\sL'(\widetilde K)$-BLD mapping $\widetilde F \colon |\widetilde K|\times [0,1]\to \bS^2\times [0,1]$ as in the statement. In particular, $\widetilde F|_{|\widetilde K|\times \{1\}}$ is an $\sL^\dagger({\widetilde K}^*)$-BLD-controlled expansion of a ${\widetilde K}^*$-Alexander map.

Let ${\widetilde f}^* \colon |\widetilde K| \to \bS^2$ be the ${\widetilde K}^*$-Alexander map underlying the BLD map $\widetilde F|_{|\widetilde K|\times \{1\}} \colon |\widetilde K|\times \{1\} \to \bS^2\times \{1\}$. Then $({\widetilde K}^*,{\widetilde f}^*)$ is an Alexander star and $\{ ({\widetilde K}^*,{\widetilde f}^*), (\widetilde q^\Delta, \widetilde f|_{|\widetilde q|})\}$ is an Alexander star pair. We extend $\widetilde F$ to a map $|\widetilde K\cup \widetilde q| \times [0,1]\to \bS^2\times [0,1]$, by   $(x,t) \mapsto (x,t)$ on $|\widetilde q|\times [0,1]$.

Denote by $\widetilde v$ and $v$, respectively,  the only vertices of ${\widetilde K}^*$ and $\widetilde q^\Delta$ contained in $\interior |\widetilde K|$ and $\interior |\widetilde q|$, respectively. Then $L = {\widetilde K}^* \cap \widetilde q^\Delta$ is a $1$-complex having only $w_0$ and $w_1$ vertices and the $w_1$ vertices are contained in the interior of $|L|$. See Figure \ref{fig:three_cases_joins} for an illustration on all isomorphism classes of $L$.

\begin{figure}[h!]
\begin{overpic}[scale=0.75,unit=1mm]{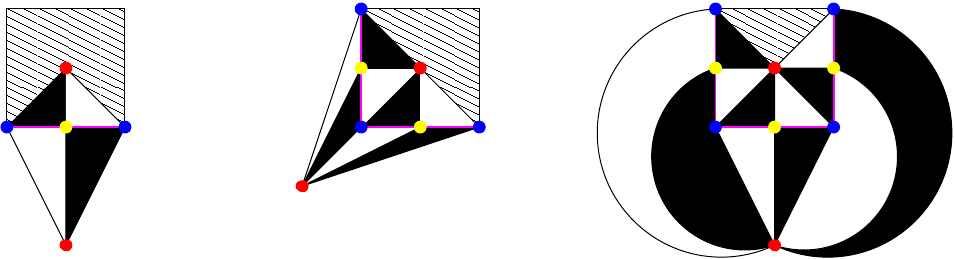} 
\put(36,7){\tiny $\widetilde v$}
\put(10,0){\tiny $\widetilde v$}
\put(98,-1){\tiny $\widetilde v$}
\put(9,25){\tiny $v$}
\put(54,25){\tiny $v$}
\put(98,26){\tiny $v$}
\put(11,28){\tiny $\widetilde q$}
\put(56,28){\tiny $\widetilde q$}
\put(100,29){\tiny $\widetilde q$}
\end{overpic}
\caption{Joins of $L$ with vertices $\widetilde v$ and $v$ in all cases}.
\label{fig:three_cases_joins}
\end{figure}

There are two cases. Suppose first that $|L|$ consists of exactly one face of $\widetilde q$. Then $L^\Delta$ has a single $w_1$-vertex $u\in L^\Delta$. By Lemma \ref{lemma:merge-of-Alexander-pair}, we may merge the stars ${\widetilde K}^*$ and $\widetilde q^\Delta$ at $u$ to a complex $ K^*$ for which $|K^*| = |{\widetilde{K}}^*\cup \widetilde q^\Delta|$, $u$ is the only vertex in the interior of $ K^*$.

By Corollary \ref{cor:merge-of-Alexander-pair}, there exists a constant $\sL=\sL(\widetilde K, \widetilde q)\ge 1$ and an $\sL$-BLD map $F^*\colon |\widetilde K \cup \widetilde q|\times [1,2] \to \bS^2\times [1,2]$ for which $F^*|_{|\widetilde K \cup \widetilde q|\times \{1\}} = \widetilde F|_{|\widetilde K \cup \widetilde q|\times \{1\}}$ and $F^*|_{|\widetilde K \cup \widetilde q|\times \{2\}} \colon |\widetilde K \cup \widetilde q|\to \bS^2$ is an $\sL$-BLD-controlled expansion of a $ K^*$-Alexander map. Fix a ball $B$ for which $2B$ is contained in a $2$-simplex in ${\widetilde q}^\Delta$.  We extend now $F^*$ to the interval $[2,3]$ and move all new simple covers in $|{\widetilde K}^*|$ into $B$ by a BLD isotopy.

Finally, we define a mapping $F\colon |K|\times [0,1]\to \bS^2\times [0,1]$ by
\[
(x,t) \mapsto \left\{ \begin{array}{ll}
\widetilde F(x,3t), & t\le 1/3 \\
F^*(x, 3t-1), & t\ge 1/3.
\end{array}\right.
\]
The first part of the claim in lemma follows in this case.

Suppose now that $|L|$ is a union of two or three faces of $q$; see Figure \ref{fig:three_cases_joins}. In this case, we first apply  Lemma \ref{lemma:Alexander_star_pair_reduction} and Corollary \ref{cor:Alexander_star_pair_reduction} 
to reduce the complex $L^\Delta$, and simultaneously the complex ${\widetilde K}^*\cup \widetilde q^\Delta$, to a complex of having a single $w_1$-vertex in the interior; see Figure \ref{fig:Dim_2_intersections} for an example. After this, we complete the proof by following the case of a single face. Since the proof is otherwise analogous, we omit the technical details. 

In both cases we obtained a BLD map $F\colon |K|\times [0,1]\to \bS^2\times [0,1]$ which is a homotopy from $f$ to a BLD expansion of a $K^*$-Alexander map $|K|\to \bS^2$. Since the distortion constants depend only distortion constants $\sL^\dagger(\widetilde K)$, $\sL'(\widetilde K)$, and distortion constants of Corollaries \ref{cor:merge-of-Alexander-pair} and \ref{cor:Alexander_star_pair_reduction}, we conclude that $F$ and its restriction $F|_{|K|\times \{1\}}$ are $\sL''(K)$ and $\sL^\dagger(K^*)$-BLD, respectively. This completes the proof of the first claim in the lemma for the case that $f$ is a $K^\Delta$-Alexander map.

We now verify the second claim on affine structure, and suppose that $|K| \subset \R^2$ is convex and $K$ has an affine flat structure. Write $K=\widetilde K \cup \widetilde q$ as before and note that, before applying the second part of  the induction hypothesis to $\widetilde K$, we need to map $\widetilde K$ isomorphically to a complex whose space is convex.

For this we fix a piecewise linear $2$-cell $q'$ in $\R^2$ for which $q'\cap |K|= |\partial K \cap \widetilde q|$ and $|q'\cup \widetilde q|$ is convex.
Fix  also  a piecewise linear homeomorphism $\psi \colon |K| \to |K\cup q'|$ for which  
\begin{enumerate}
\item $\psi(|\widetilde K|)= |K|$, $\psi(|\widetilde q|)=|q'|$,
\item $\psi(|\widetilde K \cap \widetilde q|) = |\partial K \cap \widetilde q|$, $\psi(|\partial K \cap \widetilde q|) = |q'\cap \partial (K\cup q')|$, and 
\item $\psi$ is affine on simplices of $\widetilde K^\Delta \cup \widetilde q^\Delta$.
\end{enumerate}
Then  $\psi_*K$ is a cubical complex on $|K\cup q'|$ isomorphic to $K$ and $\psi_*K^\Delta$ is a simplicial complex on $|K\cup q'|$ isomorphic to $K^\Delta$.
Moreover,  complex $\psi_*\widetilde K$ has now a convex space $|\psi_*\widetilde K|=|K|$ and, because $\psi$ is piecewise linear, an affine structure. 

The induction hypothesis now yields a star replacement $(\psi_*\widetilde K)^*$ having an affine structure, and a bilipschitz homotopy  $\widehat F \colon |\psi_*\widetilde K|\times [0,1]\to \bS^2\times [0,1]$ from $\widetilde f \circ \psi^{-1}$ to a BLD-controlled expansion ${\widehat f}^*$ of a 
$(\psi_*\widetilde K)^*$-Alexander map. As before, $\widehat F$ may be extended  canonically to a homotopy  $\widehat F \colon |\psi_*(\widetilde K \cup \widetilde q)| \times [0,1]\to \bS^2\times [0,1]$. 

After repeating the previous deformation steps, almost verbatim, to 
$(\psi_*\widetilde K)^* \cup (\psi_*\widetilde q)^\Delta$ and $\widehat F|_{|\psi_*(\widetilde K \cup \widetilde q)| \times \{1\}}$, we arrive at a star replacement $(\psi_*K)^*$ of  $\psi_*K$ having an affine flat structure, and a BLD homotopy from ${\widehat f}^*$ to a BLD expansion of a  $(\psi_*K)^*$-Alexander map on $|K\cup q'|$. Since $\psi$ is piecewise linear, the second part of the claim on the affine structures follows immediately by pulling the homotopy and the star replacement $(\psi_*K)^*$ back to $K$ by ${\psi_*}^{-1}$.

\medskip
\noindent
\emph{Case 2: (General case)} Suppose now that $f \colon |K|\to \bS^2$ is an $\sL$-BLD-controlled expansion of a $K^\Delta$-Alexander map $|K|\to \bS^2$ and let $\widetilde K$ and $\widetilde q$ be as in Case 1. 

We fix also a $2$-simplex $\sigma \in \widetilde q^\Delta$, which has a face in $\partial K$, and a Euclidean ball $B \subset \sigma$ satisfying $2B \subset \sigma$. Since $f$ is a controlled expansion, we may assume, by applying Proposition \ref{prop:move} if necessary, that all simple covers are contained in $B$. Since the isotopy of simple covers has distortion constant depending only on $K$, we may increase the constant $\sL''(K)$ if necessary, so that the isotopy $|K|\times [0.1]\to \bS^2\times [0,1]$ is $\sL''(K)$-BLD. 

After these preparatory steps, we may apply Case 1 to mapping $f$ in place of an Alexander map $|K|\to \bS^2$ with a requirement that  $B$ is moved by throughout  the process.  This yields a $2\sL''(K)$-BLD mapping $F\colon |K|\times [0,1]\to \bS^2\times [0,1]$ for which $F|_{|K|\times \{0\}} = f$, and $F_{|K|\times \{1\}} \colon |K|\to \bS^2$ is a $\max\{ \sL^\dagger(K^*), \sL\}$-BLD expansion of a $K^*$-Alexander map. This concludes the proof of the first claim for  the general case. Proof of the claim on affine structure follows also from that for Case 1 almost verbatim. 
\end{proof}

\begin{figure}[h!]
\begin{overpic}[scale=0.50,unit=1mm]{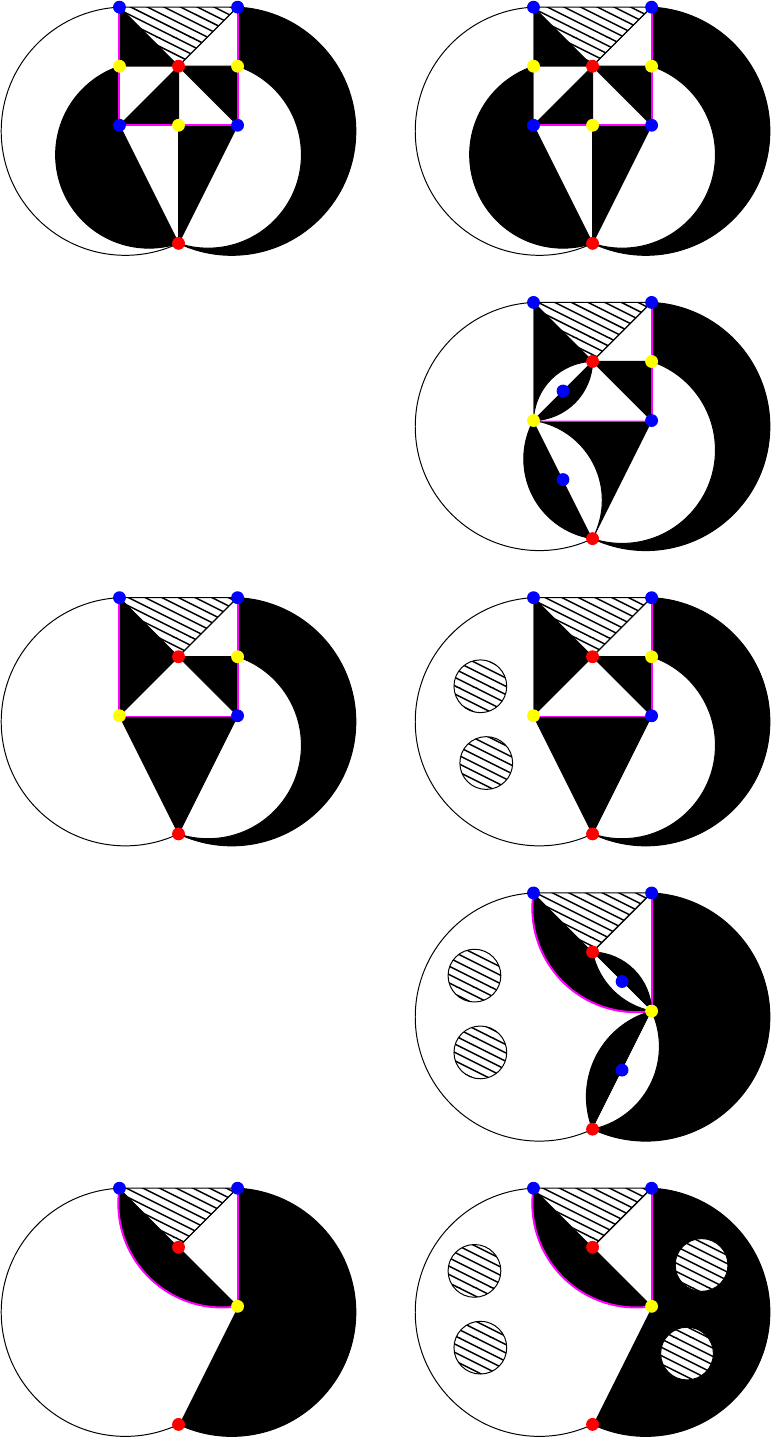} 
\end{overpic}
\caption{Deformation of an Alexander star pair to a simple Alexander star pair expanded by simple coves, for the third picture in Figure \ref{fig:three_cases_joins}. \emph{On the left:} Alexander maps associated to the  reduction steps of the original complex. \emph{On the right:} Branched covers associated to the deformation steps.} 
\label{fig:Dim_2_intersections}
\end{figure}

\begin{proof}[Proof of Theorem \ref{thm:QR-deformation-star}]
We start with the induction on the dimension and then on the number of top dimensional cubes.

The induction assumption on the dimension states: \emph{Suppose that  $n\ge 2$ and that, for each shellable Alexander $n$-complex $(K,f)$, there exist a reduction sequence $K^\Delta \searrow \cdots \searrow K^*$, a constant $\sL'=\sL'(n,K)\ge 1$, and an $\sL'$-BLD mapping $F'\colon |K|\times [0,1]\to \bS^n\times [0,1]$ for which $F'|_{|K|\times \{0\}} = f$ and $F'|_{|K|\times \{1\}} \colon |K|\to \bS^n$ is an $\sL'$-BLD-controlled expansion of a $K^*$-Alexander map. Moreover, if $|K|\subset \R^n$ is convex and $K$ has an affine flat structure, then $K^*$ has an affine flat structure.} 

This statement holds for $n=2$ by Lemma \ref{lemma:cubical_deformation_dim_2}. Let $n \ge 2$ be a dimension for which the induction assumption on the dimension holds. 

Assume next that in dimension $n+1$ the following holds: \emph{Suppose that $\ell\ge 1$ and that, for each shellable Alexander $(n+1)$-complex $(K,f)$ and $\# K^{[n+1]} = \ell$, there exist a reduction sequence $K^\Delta \searrow \cdots \searrow K^*$, a constant $\sL''=\sL''(n,K)\ge 1$, and an $\sL''$-BLD mapping $F''\colon |K|\times [0,1]\to \bS^n\times [0,1]$ for which $F''|_{|K|\times \{0\}} = f$ and $F''|_{|K|\times \{1\}} \colon |K|\to \bS^n$ is an $\sL''$-BLD-controlled expansion of a $K^*$-Alexander map. Moreover, if $|K|\subset \R^n$ is convex and $K$ has an affine flat structure, then $K^*$ has an affine flat structure.} This statement holds trivially for $\ell=1$. 

Suppose now that the induction assumption holds for all $(n+1)$-complexes $K$ satisfying $\# K^{[n+1]} = \ell$ and,  
for the induction step, suppose that $(K,f)$ is a shellable Alexander $(n+1)$-complex for which $\# K^{[n+1]}=\ell+1$. We proceed as in Lemma \ref{lemma:cubical_deformation_dim_2}.

Let $f\colon |K|\to \bS^n$ be an $\sL$-BLD-controlled expansion of a $K^\Delta$-Alexander map $h \colon |K|\to \bS^n$.

By the shellability of $K$, there exist a shellable subcomplex $\widetilde K$ of $K$, whose space is an $(n+1)$-cell, and an $(n+1)$-cube $\widetilde q\in K$ for which  $K = \widetilde K \cup \widetilde q$, $|\widetilde K\cap \widetilde q|$ is an $n$-cell, and $\widetilde q^\Delta = {\widetilde q}^*$. The cubical complex $\widetilde K\cap \widetilde q$ is shellable by Lemma \ref{lemma:shellability_faces}, and complex $\widetilde K$ is shellable since $K$ is shellable. 

Again, as in the two dimensional case, we may assume that supports of the  simple covers in $f$ expanding $h$ are contained in a small ball $B$ in an $(n+1)$-simplex $\sigma\in \widetilde q^\Delta$ which has a face in $\partial K$. Let $\widetilde f = f|_{|\widetilde K|} \colon |\widetilde K|\to \bS^{n+1}$. Then $\widetilde f$ is an Alexander map.

By the (second) induction assumption on the number of top dimensional cubes, there exists a reduction sequence $\widetilde K^\Delta \searrow \cdots \searrow {\widetilde K}^*$, hence also a reduction sequence  $\widetilde K^\Delta\cup {\widetilde q}^\Delta \searrow \cdots \searrow {\widetilde K}^*\cup {\widetilde q}^\Delta$.

Also  by induction there exist $\sL'=\sL'(n,K)$ and an $\sL'$-BLD map $\widetilde F\colon |\widetilde K|\times [0,1]\to \bS^{n+1}\times [0,1]$ for which $\widetilde F|_{|\widetilde K|\times \{0\}} = \widetilde f$ and $\widetilde F|_{|\widetilde K|\times \{1\}} \colon |\widetilde K|\to \bS^{n+1}$ is an $\sL'$-BLD-controlled expansion of a ${\widetilde K}^*$-Alexander map. We extend  $\widetilde F$ to $|\widetilde K \cup \widetilde q|\times [0,1]$ canonically as in the two dimensional case, then move the new simple covers into ball $B$. We denote the new mapping also by $\widetilde F \colon |K|\times [0,1]\to \bS^{n+1}\times [0,1]$.

Let $L = {\widetilde K}^* \cap \widetilde q^\Delta$, and  $\widetilde v$ and $v$ be the centers of the stars ${\widetilde K}^*$ and $\widetilde q^\Delta$, respectively. From the labeling of  the vertices it follows  that $L$ admits an $L$-Alexander map $|L|\to \bS^n$. Since $L$ is shellable by Lemma \ref{lemma:shellability_faces}, by the induction assumption on the  dimension, we obtain a reduction sequence 
\[
L = L_0 \searrow \cdots \searrow L_m,
\]
where $L_m$ is a star of a $w_{n-1}$-vertex $u$. By coning the elements in complex $L_\ell$ with vertices $\widetilde v$ and $v$, respectively, the sequence $L_0 \searrow \cdots \searrow L_m$ induces a reduction sequence 
\[
{\widetilde K}^* \cup \widetilde q^\Delta = P_0 \searrow \cdots \searrow P_m,\]
where 
\[P_m=  v * (( \partial {\widetilde q}^\Delta - L)\cup L_m) \,\,   \bigcup \,\,  \widetilde v * ((\partial {\widetilde K}^* - L) \cup L_m )
\]
is a simple Alexander star pair of dimension $n+1$, and $v$ and $\widetilde v$ are $w_{n+1}$-vertices.
By  Lemma \ref{lemma:merge-of-Alexander-pair}, we obtain an Alexander star $K^*= ({\widetilde K}^*\cup \widetilde q^\Delta)^*$ by taking  the collapse $(P_m)^\st_{u,n+1}$.

By Corollaries \ref{cor:Alexander_star_pair_reduction} and \ref{cor:merge-of-Alexander-pair}, and the induction assumption, there exist  a constant $\sL''=\sL''(n,K)$ and an $\sL''$-BLD map $F' \colon |\widetilde K \cup \widetilde q|\times [1,2]\to \bS^{n+1}\times [1,2]$ for which $F'|_{|\widetilde K \cup \widetilde q|\times \{1\}} = F'|_{|\widetilde K \cup \widetilde q|\times \{1\}}$ and $F'|_{|\widetilde K \cup \widetilde q|\times \{2\}} \colon |\widetilde K \cup \widetilde q|\to \bS^{n+1}$ is an $\sL''$-BLD-controlled expansion of a $K^*$-Alexander map. Again, we extend the map $F'$ to $|K|\times [2,3]\to \bS^{n+1}\times [2,3]$ by an isotopy of simple covers, which moves the simple covers into $B$. We use the notation $F'$ also for the extension after rescaling the interval.

We finish the construction of the homotopy by combining the maps $\widetilde F$ and $F'$ to a BLD map $F\colon |K|\times [0,1]\to \bS^{n+1}\times [0,1]$.

To check that $K^*$ may be chosen to have an affine structure when $|K|\subset \R^n$ is convex and $K$ has affine structure, we need to transform  the space of $\widetilde K$ to a convex set as in Lemma \ref{lemma:cubical_deformation_dim_2}. As in the two dimensional case, 
we fix an $(n+1)$-cube $q'$ and a PL homeomorphism $\psi \colon |K|\to |K\cup q'|$ for which $|\psi (\widetilde K)| =|K|$, hence convex,  $|\psi(\widetilde q)|= |q'|$,  and $\psi_* \widetilde K$ consists of affine simplices. By performing the deformation of the mapping $\widetilde f \circ \psi^{-1}$ on $\psi_*\widetilde K \cup \psi_* \widetilde q$, we  obtain an affine star replacement $(\psi_*\widetilde K)^*$ of $\psi_*\widetilde K$  and a BLD homotopy from $\widetilde f \circ \psi^{-1}$  to a BLD-controlled expansion of a $(\psi_*\widetilde K)^*$-Alexander map. 
 The conclusion follows by pulling back the BLD homotopy by  $(\psi_*)^{-1}$ and setting a $K^* = (\psi_*)^{-1}(\psi_* \widetilde K)^*$, as in the two dimensional case. Since the argument is analogous to that in Lemma \ref{lemma:cubical_deformation_dim_2}, we omit the details. 
 
This completes the induction step and the proof.
\end{proof}

\begin{remark}
For each $n\in \N$ and every $\ell\in \N$, there exists finitely many isometry classes of cubical $n$-complexes $K$ for which $\# K^{[n]} = \ell$. Hence the constants $\sL^\dagger(K^*)$ and $\sL'(K)$ in Theorem \ref{thm:QR-deformation-star} depend only dimension $n$ and the size of $K$.  
\end{remark}

\subsection{Deformation in specific cubical models}

We record an immediate corollary of Theorem \ref{thm:QR-deformation-star}. This observation follows from the fact that $q^\Delta$, for a cube $q$, is a star.

\newcommand{\Ldeformshellable}{\sL_{\mathrm{deform}}}

\begin{corollary}
\label{cor:extension-over-cube} 
Let $q=[0,1]^{n-1}$ be an $(n-1)$-cube, and 
$K$ be a shellable cubical complex having an affine flat structure,  space $|K|=|q|$, and boundary $\partial K = \partial q$. Then there exists $\Ldeformshellable=\Ldeformshellable(n, K) \ge 1$ for the following. 

Let $f\colon |K|\to \bS^{n-1}$ be an $\sL$-BLD-controlled expansion of a $K^\Delta$-Alexander map for $\sL\ge 1$, and $\sigma$ be an $(n-1)$-simplex in $q^\Delta$. Then there exists a $\max\{ \Ldeformshellable, \sL\}$-BLD map $F \colon |q|\times [0,1]\to \bS^{n-1}\times [0,1]$ for which 
\begin{enumerate}
\item $F|_{|q|\times \{0\}} = f$,
\item $F(x,t) = (f(x),t)$ for $(x,t)\in (|\sigma|\cup \partial |q|) \times [0,1]$, and
\item $F|_{|q|\times \{1\}} \colon |q|\times \{1\} \to \bS^{n-1}\times \{1\}$ is a $\max\{ \Ldeformshellable, \sL\}$-BLD-controlled expansion of a $q^\Delta$-Alexander map.
\end{enumerate}
\end{corollary}



Another typical case is a merge in a bands of cubes. To illustrate, let $P$ be a cubical $(n-1)$-complex whose space is a closed $(n-1)$-manifold, and let $K=P\times [0,2]$ be a cubical $n$-complex with $n$-cubes $K^{[n]} = \{ q\times [0,1], q\times [1,2] \colon q\in P^{[n-1]}\}$. Then for each $k=0,\ldots, n-1$ and each $\xi\in P^{[k]}$, there exists exactly two $(k+1)$-cubes $\xi_- \subset P\times [0,1]$ and $\xi_+ \subset P\times [1,2]$ for which $\xi_- \cap \xi_+ = \xi\times \{1\}$. In particular, for each $q\in P^{[n-1]}$, the union  $(q_+)^\Delta \cup (q_-)^\Delta$ is an Alexander star pair in $K^\Delta$. Since these pairs have mutually disjoint interiors, we may merge all  such $n$-dimensional Alexander star pairs individually. After that we continue, inductively in dimension $k=n-2,\ldots, 0$, to merge $(k+1)$-dimensional Alexander star pairs having a common face in $P^{[k]}$. This iterative merge of Alexander stars yields a complex $h_*\left((P\times [0,1])^\Delta\right)$, where $h \colon |P|\times [0,1]\to |P|\times [0,2]$ is the map $(x,t)\mapsto (x,2t)$. Note that spaces $|h_*\left((P\times [0,1])^\Delta\right)| =|K|$.
We record this deformation as a corollary.

\newcommand{\Band}{\mathrm{Band}}

\begin{corollary}
\label{cor:Band-merge}
For $n\ge 2$, there exists a constant $\sL_\Band = \sL_\Band(n)\ge 1$ for the following. Let $P$ be a cubical $(n-1)$-complex whose space is a closed $(n-1)$-manifold and let $K=P\times [0,2]$ be a cubical $n$-complex with $n$-cubes $K^{[n]} = \{ q\times [0,1], q\times [1,2] \colon q\in P^{[n-1]}\}$. Then there exists an $\sL_\Band$-BLD map $F\colon |K| \times [0,1] \to \bS^{n-1}\times [0,1]$ for which $F|_{|K|\times \{0\}} \colon |K|\to \bS^{n-1}$ is the orientation preserving $K^\Delta$-Alexander map and $F_{|K|\times \{1\}} \colon |K|\to \bS^{n-1}$ is an $\sL_\Band$-BLD-controlled expansion of an $h_*\left((P\times [0,1])^\Delta\right)$-Alexander map.   
\end{corollary}



\chapter{Hopf theorem for Alexander maps}
\label{sec:uniqueness_theorems}

The Quasiregular deformation theorem has a consequence in terms of homotopy, which may be considered as a Hopf theorem for Alexander maps.

 \begin{theorem}[Homotopy theorem of Alexander maps]
\label{thm:cell_uniqueness}
Let $n\geq 2$, and $K_1$ and $K_2$ be two shellable cubical complexes on an $n$-cell $E$ for which $K_1|_{\partial E} = K_2|_{\partial E}$. Let $f_1 \colon E\to \bS^n$ be a $K_1^\Delta$-Alexander map and $f_2\colon E\to \bS^n$ be a $K_2^\Delta$-Alexander map and  let, for each
 $i=1,2$,  $\check f_i \colon E \to \bS^n$ be a BLD-controlled expansion of $f_i$ for which  $\deg\, \check f_1=\deg\, \check f_2$. Then $\check f_1$ and $\check f_2$ are quasiregularly homotopic.
\end{theorem}

\begin{proof}
Since $E$ is a cell and $K_1|_{\partial E} = K_2|_{\partial E}$, the star-replacements $K_1^*$ and $K_2^*$ of $K_1^\Delta$ and $K_2^\Delta$ are isomorphic. We may therefore assume that $K_1^*=K_2^*$. Then, by the Quasiregular deformation theorem, each map $f_i$ is a quasiregularly homotopic to a BLD-controlled expansion of a $K_i^*$-Alexander map $h_i \colon E\to \bS^n$. Since $f_1$ and $f_2$ are either both orientation preserving or both orientation reversing and $K_1^*$ is a star, we have that $h_1 = h_2$. Thus both $\check f_1$ and $\check f_2$ are quasireguarly homotopic to the same quasiregular controlled expansion of an $K_1^*$-Alexander map $E\to \bS^n$. 
\end{proof}

As a direct consequence of Theorem \ref{thm:cell_uniqueness}, we obtain the following theorem stated in the introduction.

\introthmdeformation*

We also obtain immediately a version of the Hopf degree theorem for Alexander maps. For the statements, we say that a cubical complex $K$, for which $|K|\approx \bS^n$, is \emph{shellable} if there exists an $n$-cube $Q\in K^{(n)}$ for which $K' = K\setminus \{Q\}$ is shellable cubical complex.

\index{Hopf theorem for Alexander maps}

\introthmHopftheoremvone*

In these theorems, the level mappings  in the intermediate stages of the homotopy are merely branched covers and  not Alexander maps.

Finally we record also two normal forms for cubical Alexander maps $\bS^n \to \bS^n$. The first is a reduction to the identity map, and the second to a winding map.

\begin{theorem}[Normal forms for cubical Alexander maps]
\label{thm:Hopf_theorem_v2}
Let $n\geq 2$ and $(K,f)$ be a shellable cubical Alexander $n$-complex on $\bS^n$ for which the Alexander map $f \colon \bS^n\to \bS^n$ is orientation preserving. Then $f$ is quasiregularly homotopic 
\begin{enumerate}
\item to a controlled expansion of the identity map $\bS^n \to \bS^n$ by simple covers, and \label{item:expansion-simple-cover}
\item to a winding map $\bS^n \to \bS^n$. \label{item:winding}
\end{enumerate}
\end{theorem}

\begin{figure}
\begin{overpic}[scale=0.65,unit=1mm]{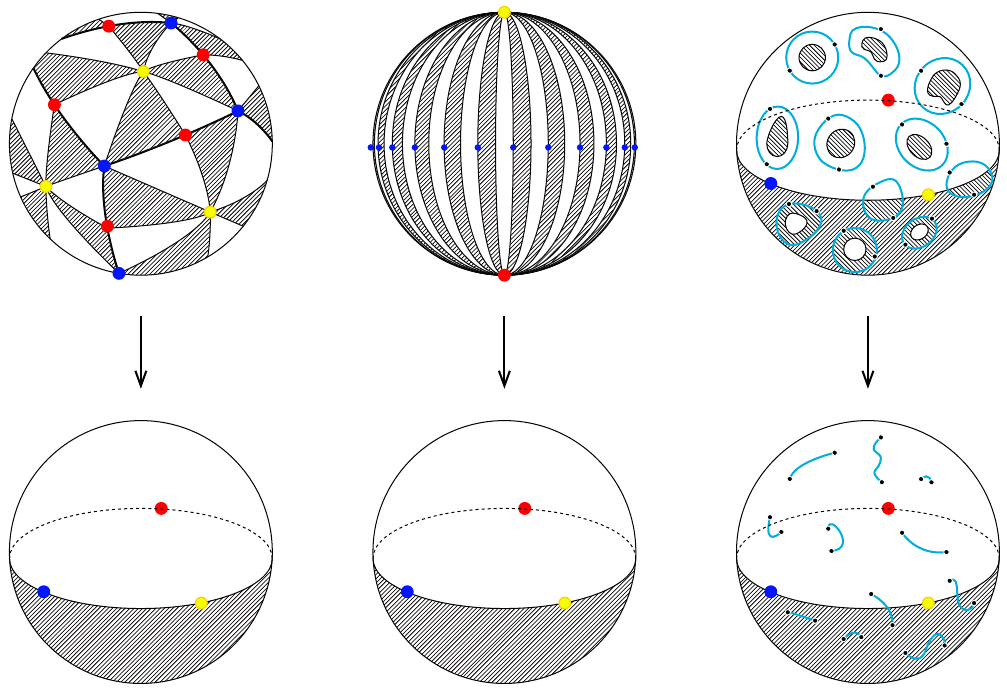} 
\put(25,28){\tiny $\bS^2$}
\put(65, 28){\tiny $\bS^2$}
\put(105,28){\tiny $\bS^2$}
\end{overpic}
\caption{A cubical Alexander map, a winding map, and an identity map expanded by simple covers in two-dimension, as in Theorem \ref{thm:Hopf_theorem_v2}.}
\label{fig:Hopf}
\end{figure}

In the  two theorems above, only part \eqref{item:winding} of Theorem \ref{thm:Hopf_theorem_v2} requires a proof. 

We first recall the notion of a winding map. For each $k\in \N$, let  $re^{i\theta} \mapsto re^{ik\theta}$ be the standard \emph{winding map of order $k$ in $\R^2$}. We also call its restriction to $\bS^1\to \bS^1$ a winding map on $\bS^1$. Inductively, for each $n\ge 2$, we call a map $F = f \times \id \colon \R^{n-1}\times \R \to \R^{n-1}\times \R$  the \emph{winding map of order $k$ on $\R^n$} if $f$ is a winding of order $k$ in $\R^{n-1}$. Similarly, we call the restriction $F|_{\bS^{n-1}} \colon \bS^{n-1}\to \bS^{n-1}$ a winding map of order $k$ on $\bS^{n-1}$.

\begin{proof}[Proof of part \eqref{item:winding} of Theorem \ref{thm:Hopf_theorem_v2}]
The proof is by induction on dimension. In dimension $n=2$, it suffices to observe that a winding map $\bS^2\to \bS^2$ is quasiregularly homotopic to a BLD-controlled expansion of the identity map $\id \colon \bS^2 \to \bS^2$. Thus, for $n=2$, the claim follows from \eqref{item:expansion-simple-cover}.

Suppose now that the claim holds on $\bS^n$ for some $n\ge 2$. Let $K$ be a shellable cubical complex on $\bS^{n+1}$ and let $f$ be a $K^\Delta$-Alexander map.  Let $Q\in K$ be an $(n+1)$-cube for which $K' = K \setminus \{Q\}$ is shellable. By Theorem \ref{thm:QR-deformation-star}, we may reduce $(K')^\Delta$ to its star-replacement, which is the barycentric triangulation $(Q')^\Delta$ of a single $(n+1)$-cube $Q'$ having space $|K'|$, and deform the $K^\Delta$-Alexander map $f$ to a quasiregular controlled expansion of a $(Q\cup Q')^\Delta$-Alexander map. Since $|Q\cap Q'|$ is an $n$-sphere, a $(Q\cup Q')^\Delta$-Alexander map is a suspension of a $(Q\cap Q')^\Delta$-Alexander map. Since a winding map $\bS^{n+1} \to \bS^{n+1}$ is a suspension of a winding map $\bS^n \to \bS^n$, the claim now follows from the induction assumption.
 \end{proof}


\part{Extension}
\label{part:QR-extension}

\chapter{Quasiregular extension theorem}
\label{sec:QRextension-statement}

In this part, we study  quasiregular extension for complexes associated to the partition by a separating complex.  The deformation theorems from Part \ref{part:deformation} have a critical role in constructing maps.  The main result is stated in Theorem \ref{thm:Real-extension-problem}.

\medskip

In what follows, we work with the following notion of evolution sequence. 

\begin{definition}
\index{evolution sequence of separating complexes}
A sequence $(Z_k)$ of separating complexes,  $Z_k \subset \Refine^{\nu k}(K)$, $k\geq 0$, is an \emph{evolution sequence of separating complexes} if
\begin{itemize}
\item $K$ is the refinement $\Refine(K_0)$ of a good cubical $n$-complex $K_0$ with boundary,
\item $Z_0=\mathcal Z$ is the separating complex in $K$ formulated in Remark \ref{rmk:separating-complex-special} based on the construction  in Theorem \ref{thm:separating-complex-existence}, and 
\item sequence $(Z_k)$ is constructed inductively by the evolution process  in Section \ref{sec:evolution}. 
\end{itemize}
Here $\nu = \nu(K)\ge 1$ is the refinement scale of $K$.
\end{definition}

Note that in the construction of $(Z_k)$, some  other structure constants, such as the length of tunnels $\lambda_\loc$, have also been fixed.

Since $Z_k$ rarely admits a well-defined Alexander map, instead of working on complexes $\Comp_{\Refine^{\nu k}(K)}(Z_k;\Sigma)$, 
we consider the extension problem on the realizations
\[
\Real_k(\Sigma) = \Real_{\Refine^{\nu k}(K)}(Z_k;\Sigma).
\]
Let $\pi_k \colon \Real_k(\Sigma) \to \Comp_k(\Sigma)$ be the canonical projection, and denote by
\[
\Upsilon_k(\Sigma) = \pi_k^{-1}(\Comp_k(\Sigma) \cap Z_k)
\]
the inner boundary component of $\Real_k(\Sigma)$. For a fixed $k$, complexes  $\Real_k(\Sigma)$ are equipped flat metrics lifted from the flat metric in the cubical complex $\Refine^{\nu k}(K)$.

Our main extension theorem reads as follows.

\begin{restatable}[Quasiregular extension theorem]{theorem}{Realextensionproblem} \index{Quasiregular extension theorem}
\label{thm:Real-extension-problem}
Let $n\geq 3$. Then there exists $\sL^\Diamond=\sL^\Diamond(n)\ge 1$ for the following. Let $K=\Refine(K_0)$ the refinement of a good cubical $n$-complex $K_0$ with boundary, and $(Z_k)$ be an evolution sequence of separating complexes in $K$.  Let $\sL\ge \sL^\Diamond$ and, for each $k\ge 0$ and a boundary component $\Sigma$ of $K$, an $\sL$-BLD-controlled expansion 
\[
\widetilde f_{k,\Sigma} \colon |\Upsilon_k(\Sigma)|\to \bS^{n-1}\times \{0\}
\]
of a $\Upsilon_k(\Sigma)^\Delta$-Alexander map, 

Then there exists $\sL'=\sL'(n,K;\sL)\ge 1$ and an $\sL'$-BLD extension 
\[
\widetilde F_{k,\Sigma} \colon |\Real_k(\Sigma)|\to \bS^{n-1}\times [0,3^{\nu k}]
\]
of $\widetilde f_{k,\Sigma}$ for which the restriction $\widetilde F_{k,\Sigma}|_{|\Sigma|} \colon |\Sigma|\to \bS^{n-1}\times \{3^{\nu k}\}$ is an $\sL$-BLD-controlled expansion of a $\Refine^{\nu k}(\Sigma)^\Delta$-Alexander map.
\end{restatable}

\begin{remark}
It should be noted that we do not fix a simplicial structure on $\bS^{n-1}\times [0,3^{\nu k}]$, and that the extension $F_{k,\Sigma}$ in Theorem \ref{thm:Real-extension-problem} is neither a simplicial map nor an expansion of such.
\end{remark}

\begin{remark}
For each fixed $k$, mappings 
$F_{k,\Sigma}$ in Theorem \ref{thm:Real-extension-problem}  are the building blocks for constructing a well-defined quasiregular map $F_k\colon |K|\to \bS^{n-1}\times \R$. We address the method of combining $F_{k,\Sigma}$  for all components $\Sigma\subset \partial K$ in Part \ref{part:Alexander-Rickman}.
\end{remark}

The idea of the extension is based on an observation: each $\Real_k(\Sigma)$ has a hierarchical decomposition into essentially disjoint subcomplexes by evolution. Indeed, we have from Proposition \ref{prop:general-localized-realization-structure} that
\begin{align*}
\Real_k(\Sigma) = \,\,&\widetilde \Rec_*(Z_0,\ldots, Z_{k-1};\Sigma) \\
&\cup \left(\bigcup_{\ell=1}^{k-1} \Diff_*(Z_\ell,\ldots, Z_{k-1};\Sigma)\right)\\
&\cup \Tunnel_{\Refine^{\nu k}(K)}(Z_k;\Sigma).
\end{align*}
For the purpose of recording, we denote by $\mathcal{Real}_k(\Sigma)$ the collection 
\[
\mathcal{Real}_k(\Sigma) = \{ \widetilde \Rec_k(\Sigma) \} \cup \cdiff_1(\Sigma) \cup \cdots \cup \cdiff_{k-1}(\Sigma) \cup \cT_k(\Sigma),
\]
where
\begin{enumerate}
\item $\widetilde \Rec_k(\Sigma) = \widetilde \Rec_*(Z_0,\ldots,Z_k;\Sigma)$,
\item  $\cdiff_\ell(\Sigma)$ is the collection of the connected components of the complex $\Diff_*(Z_\ell,\ldots, Z_{k-1};\Sigma)$, and
\item $\cT_k(\Sigma)$ is the collection of connected components of the complex $\Tunnel_{\Refine^{\nu k}(K)}(Z_k;\Sigma)$.
\end{enumerate}
We set also $\cdiff_k(\Sigma) = \cT_k(\Sigma)$.

The collection $\mathcal{Real}_k(\Sigma)$ is an essential partition of $\Real_k(\Sigma)$, that is, for $E\ne E'$ in $\mathcal{Real}_k(\Sigma)$, the intersection $E\cap E'$ has no interior.

\subsection{Hierarchy of components} 
The decomposition of $\Real_k(\Sigma)$ above, together with Remark \ref{rmk:omega-i-2},  reveals an intrinsic hierarchy.
We  associate two functions to the collection $\mathcal{Real}_k(\Sigma)$: a  level function $\ell$ of the elements and a hierarchy function $\alpha$ based on the intersections; together, these functions give  $\mathcal{Real}_k(\Sigma)$ a natural tree structure with root $\widetilde \Rec_*(\Sigma)$.

\begin{notation}
Let $\ell \colon \mathcal{Real}_k(\Sigma) \to \{0,\ldots, k\}$ be the function defined by $\ell(\widetilde \Rec_*)=0$,  by $\ell(E)=\ell$ for $E\in \cdiff_\ell(\Sigma)$, and by $\ell(E)=k$ for $E\in \cT_k(\Sigma)$.
\end{notation}

\begin{notation}\label{notation:alpha}
Let  $\alpha \colon \mathcal{Real}_k(\Sigma) \to \mathcal{Real}_k(\Sigma)$ be the (unique) function for which $\alpha(\widetilde \Rec_*(\Sigma)) = \widetilde \Rec_*(\Sigma)$ and, for $\ell(E)>0$, let $\alpha(E)\in \mathcal{Real}_k(\Sigma)$ be the unique element satisfying $\ell(\alpha(E))= \ell(E)-1$ and  $\alpha(E) \cap E \ne \emptyset$.
\end{notation}

For the construction of extension of $\widetilde f_{k,\Sigma}$, we begin with an extension over  $\Tunnel_{\Refine^{\nu k}(K)}(Z_k;\Sigma)$ in Section \ref{sec:extension-tunnels}. We move next to the differences in $\widetilde \Diff_*(Z_\ell,\ldots, Z_{k-1};\Sigma)$ in the order of decreasing  $\ell$ in Chapter \ref{chap:extension-differences};  this is  the most demanding step. Finally, we extend the mapping over $\widetilde \Rec_*(Z_0,\ldots, Z_{k-1};\Sigma)$ and complete the proof of Theorem \ref{thm:Real-extension-problem} in Chapter \ref{chap:final-extension}.

The construction follows a general principle. Extension is constructed in the essentially disjoint elements in $\mathcal{Real}_k(\Sigma)$ separately. Within each element, the extension is constructed with the help of bilipschitz mappings akin to the partition. Every point is moved by non-conformal maps at most a bounded number of times, with the number depending only on $n$ and $K$, before being mapped to the target.  Therefore the distortion of the final extension remains bounded by a constant depending only on $n$ and $K$, in particular independent of $k$.

\chapter{Roof-admissible maps}

Quasiregular extension of the Alexander map $|\Upsilon_k(\Sigma)|\to \bS^{n-1}\times \{0\}$ 
from an inner boundary component of  $\Real_k(\Sigma)$ over to $|\Real_k(\Sigma)|$  is  to be constructed inductively on the elements in the partition $\mathcal{Real}_k(\Sigma)$. 

In this chapter we discuss the methods of constructing extension on complexes modeled on the elements in $\mathcal{Real}_k(\Sigma)$.
The proof of Theorem \ref{thm:Real-extension-problem} is given in Chapter \ref{chap:final-extension}.

\section{Roof-adjusted complexes}

\subsection{Dented refined tunnels} We begin by defining dented refined tunnels. Let 
\[
\cT(\lambda)  = \{ T \colon T \text{ is a tunnel with } \# T^{[n]} \le \lambda\}
\]
be the collection of all tunnel of length at most $\lambda\in \mathbb{N}$; a tunnel of length zero is an $(n-1)$-cube.
We also denote, for integer  $r\geq 0$, 
\[
\cA_r(\lambda) = \{ \Refine^{\nu r}(T) \colon T\in \cT(\lambda)\},
\]
and
\[
\cD_r(\lambda) = \{A - D \colon A\in \cA_r(\lambda),\ D \text{ bent indentation}\}.
\]
We call elements of $\cA_r(\lambda)$ \emph{refined tunnels} and elements of $\cD_r(\lambda)$ \emph{dented refined tunnels}.

For each $E\in \cD_r(\lambda)$, there exists a unique $A_E \in \cA_r(\lambda)$ and a unique bent indentation $D_E \subset A_E$ for which $E = A_E- D_E$; moreover, there exists a unique $T_E\in \cT(\lambda)$ for which $A_E = \Refine^{\nu r}(T_E)$. We call the triple $(T_E,A_E,D_E)$ an \emph{encoding of $E$}.

\begin{remark}\label{rmk:mu_D}
Refinement does not increase the local multiplicity of a complex. Thus $\mu(D_E) \leq \mu(A_E)=\mu(T_E)$  and $\mu(E)\leq \mu(T_E)$.
\end{remark}

\begin{notation}[Metric] \label{notation:side-length}
From now on, for  $E= A_E -D_E \in \cD_r(\lambda)$, we consider $E$ as a subcomplex of the cubical complex $(A_E ,d_{A_E})$  with a flat metric. 
For a subcomplex $C\subset A_E$ whose space $|C|$ is a cube, we denote by $\SL_E(C)$ the side length of  $|C|$.
In particular, the side length of an $n$-cube $Q\in T_E$  is $3^{\nu r}$.
\end{notation}

\subsection{Roof complexes}

We next define roof cubes and roof complexes for elements in $\cT(\lambda)$, $\cA_r(\lambda)$, and $\cD_r(\lambda)$.

We say that an $n$-cube $\Omega$ is a \emph{roof cube of $T\in \cT(\lambda)$} if $T^* = T\cup \Omega$ is a tunnel in $ \cT(\lambda+1)$. We call $\omega(T;\Omega) = T\cap \Omega$ the \emph{roof face} and call $\omega^+(T;\Omega)$, the face opposite to $\omega(T;\Omega)$, the \emph{roof top}. Note that, for  $\lambda=0$, we have $T\cup \Omega = \Omega$ and $\omega(T;\Omega) = T$.

Similarly, we say that a cubical complex $\Omega$ is a \emph{roof complex of a refined tunnel $A= \Refine^{\nu r}(T) \in \cA_r(\lambda)$} of a tunnel $T\in \cT(\lambda)$, if there exists an $n$-cube $Q$ for which  $\Omega = \Refine^{\nu r}(Q)$ and $T\cup Q\in \cT(\lambda+1)$. In this case, we have 
\[A^* = A\cup \Omega \in \cA_r(\lambda+1).\]
We denote $\omega(A;\Omega)=\Refine^{\nu r}(\omega(T;Q))$ and $\omega^+(A;\Omega) = \Refine^{\nu r}(\omega^+(T;Q))$ the \emph{roof face} and the \emph{roof top} of the roof complex $\Omega$, respectively.

We state now the most general case as follows.
\begin{definition}
\index{roof complex}
A cubical complex $\Omega$ is a \emph{roof complex of $E\in \cD_r(\lambda)$}, with $E=A-D$, $A\in \cA_r(\lambda)$ and $D$ a bent indentation, provided that $\Omega$ is a roof complex of $A$ for which  $\Omega \cap D = \emptyset$.
\end{definition}

The roof face $\omega(E;\Omega)$ and roof top $\omega^+(E;\Omega)$ for the dented $E$ are defined analogously. For  $E\in \cD_r(\lambda)$ and a roof complex $\Omega$ of $E$, we define the  \emph{roof cylinder} to be
\[
\sigma(E;\Omega) = \partial \Omega - (\omega(E;\Omega)\cup \omega^+(E;\Omega)).
\]

We also define  inner roof complexes as follows. 

\begin{definition}\label{inner-roof-complex}Let $E= A-D\in \cD_r(\lambda)$.
A subcomplex $\Omega' \subset E$, whose space $|\Omega'|$ is a cube, is a \emph{inner roof complex} if \index{roof complex!inner}
\begin{enumerate}
\item  $\omega = |\Omega'|\cap |\partial A|$ is a face of $|\Omega'|$ and, for the face $\omega^+$ of $|\Omega'|$ opposite to $\omega$, $\dist_{d_A}(\omega^+, |D|\cup |\partial A|)\ge \SL_E(|\Omega'|)$, or
\item there exists a spectral cube $C$ of $D$ and a face $c$ of $C$ for which  $c\cap |\partial A| = \emptyset$ and $|\Omega' \cap D| \subset c$.
\end{enumerate}
\end{definition}

If $\Omega'$ is a cube, we call it \emph{inner roof cube}. Again the \emph{inner roof face} $\omega(E;\Omega')$, the \emph{inner roof top} $\omega^+(E;\Omega')$, and \emph{inner roof cylinder} $\sigma(E;\Omega')$ are similarly defined.

\section{Definition of roof-admissible maps}

For the discussion, we fix some notations.

\begin{notation}
For $E\in \cD_r(\lambda)$, we denote by $\cD^*_r(\lambda;E)$ the collection of all \emph{roof-adjusted dented refined tunnels} 
\[
E^*=(E\cup \Omega) -  (\Omega_1\cup \cdots \cup \Omega_m),
\]
where $E\in \cD_r(\lambda)$, $\Omega$ is a roof complex of $E$, and $\Omega_1,\ldots, \Omega_m$ are inner roof complexes of $E$. 
We also denote
\[
\cD^*_r(\lambda)= \bigcup_{E\in \cD_r(\lambda)} \cD^*_r(\lambda;E).
\]
\end{notation}

Let  $E\in \cD_r(\lambda)$. We define now a class of roof-adjusted maps $|\partial E^*|\to \bS^{n-1}\times \R$, which lift the given BLD expansions of $(\partial E)^\Delta$-Alexander maps $|\partial E|\to \bS^n$ to  the roof-adjusted complexes 
$E^*$. 

For convenience, we set $\Omega_0 = \Omega$, and for each $j=0,\ldots, m$, we let 
\[
\pi_{E,\Omega_j} \colon |\Omega_j|\to |\omega(E;\Omega_j)|
\]
be the unique orthogonal projection which maps $|\omega^+(E;\Omega_j)|$ isometrically to $|\omega(E;\Omega_j)|$. Let also $\delta_{E,\Omega_j} \colon |\Omega_j|\to \R$ be the function $x\mapsto \dist(x,\pi_{E,\Omega_j}(x))$.

\begin{definition}
\label{def:BLD roof-adjustment}
Let $E\in \cD_r(\lambda)$. A Lipschitz mapping $f^* \colon |\partial E^*| \to \bS^{n-1}\times \R$ is a \emph{roof adjustment of a BLD expansion $f\colon |\partial E| \to \bS^{n-1}$} of a $(\partial E)^\Delta$-Alexander map if there exist constants
\[0 \le a < b_1,\ldots, b_m < b_0\]
for which
\begin{enumerate}
\item $f^*|_{|\partial E^* \cap \partial E|} \colon |\partial E^* \cap \partial E| \to \bS^{n-1}\times \{a\}$ is the mapping $x\mapsto (f(x),a)$,  
\item for each $j=0,\ldots, m$, $f^*|_{\omega^+(E;\Omega_j)} \colon |\omega^+(E;\Omega_j)| \to \bS^{n-1}\times \{b_j\}$ is the mapping $x\mapsto (f(\pi_{E,\Omega_j}(x)) ,b_j)$, where $\Omega_0=\Omega$, and
\item for each $x\in \sigma(E;\Omega_j)$, 
\[
f^*(x) = (f(\pi_{E,\Omega_j}(x)),a+(b_j-a)\delta_{E,\Omega_j}(x)/\SL_E(|\Omega_j|)).
\]
\end{enumerate}
We call the sequence $(a;b_0,\ldots, b_m)$ the \emph{levels of $f^*$}. 
\end{definition}

We emphasize that, roof adjustments are assumed to be Lipschitz instead of BLD. The reason  is that the image of a roof adjustment $f^*$ and the target do not have the same dimension, thus $f^*$ is not open hence not quasiregular.

\begin{definition}
\label{def:roof-admissible-map}\index{roof-admissible map}
Let $E\in \cD_r(\lambda)$. A Lipschitz map $\widetilde f\colon |\partial E^*|\to \bS^{n-1}\times \R$ is \emph{roof-admissible} if $\widetilde f$ is a roof adjustment of a BLD-controlled expansion  $f\colon |\partial E|\to \bS^{n-1}$ of a $P$-Alexander map, where $P$ is a simplicial complex having space $|P|=|\partial E|$.
\end{definition}

\begin{definition}\label{def:roof-admissible-equal-modulo}\index{roof-admissible map!equal modulo roof top}
Let $E \in \cD_r(\lambda)$. Two roof-admissible maps $ \widetilde f\colon |\partial E^*|\to \bS^{n-1}\times \R$ and $ \widetilde f' \colon |\partial E^*|\to \bS^{n-1}\times \R$ are \emph{equal modulo $\omega^+(E;\Omega)$} (modulo roof top), denoted by $\widetilde f\doteq \widetilde f'$, if $\widetilde f|_{|\partial E^* - \omega^+(E;\Omega)|} =  \widetilde f'|_{|\partial E^*-\omega^+(E;\Omega)|}.$  \index{$\doteq$}
\end{definition}

We make two remarks on the levels $(a; b_0, b_1,\ldots, b_m)$.

\begin{remark}
The Lipschitz constant $\sL^*$ of $f^*$ depends on the BLD constant $\sL$ of $f$ and the levels $a,b_0,\ldots,b_m$ and \emph{vica versa}. Indeed, $\sL^*\ge \SL_E$ and 
\[
\frac{1}{\sL^*} \le \frac{b_j - a}{\SL_E(|\Omega_j|)} \le \sL^*.
\]
\end{remark}

\begin{remark}
The formula 
\[
x\mapsto (f(\pi_{E,\Omega_j}(x)),a+(b_j-a)\delta_{E,\Omega_j}(x)/\SL_E(|\Omega_j|))
\]
gives a BLD extension $|\Omega_j|\to \bS^{n-1}\times \R$ of $f^*|_{|\Omega_j|\cap |\partial E^*|}$. 
\end{remark}

\section{Extension of roof-admissible maps}
In this section we discuss the meaning of extending a roof-admissible map. Formally, the extension does not necessarily  yield an extension of the given roof-admissible map, but an extension of an 'equivalent' roof-admissible map. For the statement, we give the following definitions.

\begin{definition}
\label{def:roof-admissible-domination}\index{roof-admissible map!domination}
We say that a roof-admissible map $\widetilde f\colon |\partial E^*|\to \bS^{n-1}\times \R$ \emph{dominates the roof-admissible map}  $\widetilde f' \colon |\partial E^*|\to \bS^{n-1}\times \R$, denoted $\widetilde f' \ll \widetilde f$, if  the levels $(a;b_0,\ldots, b_m)$ and $(a';b_1',\ldots, b_m')$ of $\widetilde f$ and $\widetilde f'$, respectively, satisfy $a < a'< b'_0 <b_0$, and $b_j < b_j'$ for $j=1,\ldots, m$.
\end{definition}

The following lemma shows that roof adjustments $f^*_0$ and $f^*_1$ are BLD homotopic if  $f^*_1 \ll f^*_0$ and the underlying maps $f_0$ and $f_1$ are BLD homotopic. 

\begin{lemma}
\label{lemma:roof-admissible-homotopy}
Let $E\in \cD_r(\lambda)$ and $E^*=(E\cup \Omega)-(\Omega_1 \cup \cdots \cup \Omega_m)$, where $\Omega$ is a roof complex of $E$ and $\Omega_1,\ldots, \Omega_m$ inner roof complexes of $E$. For $i=0,1$, let $f_i\colon |\partial E|\to \bS^{n-1}$ be a BLD-controlled expansion of a $(\partial E)^\Delta$-Alexander map,  $f_i^* \colon |\partial E^*|\to \bS^{n-1}\times\R$ a roof adjustment of $f_i$ with levels 
\[(a^{(i)}; b_0^{(i)},\ldots, b_m^{(i)})\]
 for which $f_1^* \ll f_0^*$, and let $F\colon |\partial E| \times [s_0,s_1] \to \bS^{n-1}\times [s_0,s_1]$ be a level preserving BLD homotopy from $f_0$ to $f_1$. Then there exists a BLD homotopy $\widetilde F\colon |\partial E^*| \times [s_0,s_1]\to \bS^{n-1}\times \R$ from $f^*_0$ to $f^*_1$. The result is quantitative in the sense that the BLD constant of $\widetilde F$ depends only on the BLD constant of $F$ and the ratios $(a^{(1)}-a^{(0)})/(s_1-s_0)$ and $(b_j^{(1)}-b_j^{(0)})/(s_1-s_0)$ for $j=0,\ldots, m$.
\end{lemma}

\begin{proof}
We may assume that $s_0=0$. Let also $\lambda \colon [0,s_1]\to \R$ be the linear map $t\mapsto a^{(1)} + (a^{(0)}-a^{(1)})t/3^{\nu r}$ and, for each $j=0,\ldots, m$, let $\lambda_j \colon [0,s_1] \to \R$ be the linear map $t\mapsto b^{(1)}_j - (b^{(1)}_j-b^{(0)}_j)t/s_1$.

We define first an auxiliary mapping $\widehat F\colon |\partial E^*|\times [0,s_1] \to \bS^{n-1}\times \R$ as follows. For $(x,t)\in |\partial E\cap \partial E^*|\times [0,s_1]$, we set $\widehat F(x,t) = (\id \times \lambda)F(x,t)$. For each $j=0,\ldots, m$ and $(x,t)\in |\omega(E;\Omega_j)|\times [0,s_1]$, we set $\widehat F(x,t) = (\id \times \lambda_j)F(\pi_j(x),t)$. Finally, for each $t\in [0,s_1]$, we define $\widehat F$ on each $\sigma(E;\Omega_j)\times \{t\}$ so that the restriction $\widehat F|_{|\partial E^*|\times \{t\}} \colon |\partial E^*|\to \bS^{n-1}\times \R$ is a roof adjustment of $F|_{|\partial E|\times \{t\}} \colon |\partial E|\to \bS^{n-1}\times \R$. We denote $\widehat f \colon |\partial E^*|\to \bS^{n-1}$ and $\widehat h \colon |\partial E^*|\to \R$ the coordinate mappings of $\widehat F$, that is, $\widehat F=(\widehat f,\widehat h)$.

Let $\tau$ be the $(n-1)$-simplex in $\Sigma^2(\Delta^\square_{n-1})$ which does not contain $w_{n-1}$. Then $\widehat F\left(\sigma(E;\Omega_j)\times [0,s_1]\right) \subset \tau \times \R$. Since this image has dimension lower than $\sigma(E;\Omega_j)\times [0,s_1]$, the mapping $\widehat F$ is not a BLD map on $\sigma(E;\Omega_j)\times (0,s_1)$. We use a perturbation of $\widehat F$ in the spherical coordinate to obtain $\widetilde F$. 

Let $\Psi \colon \bS^{n-1} \times [0,s_1]\to \bS^{n-1}\times [0,s_1]$ be an orientation preserving and level preserving bilipschitz isotopy from the identity $\id_{\bS^{n-1}}$ back to the identity $\id_{\bS^{n-1}}$ for which $\Psi_t(y) \ne \Psi_s(y)$ for each $y\in \tau$ and $0 < s < t < s_1$, where, for each $t\in [0,s_1]$, $\Psi_t \colon \bS^{n-1}\to \bS^{n-1}$ is the mapping satisfying $\Psi(y,t)=(\Psi_t(y),t)$ for all $y\in \bS^{n-1}$.

We define now the map $\widetilde F \colon |\partial E^*|\times [0,s_1] \to \bS^{n-1}\times \R$ by formula $\widetilde F(x,t) = (\Psi_t(\widehat f(x,t)), \widehat h(x,t))$ for $(x,t)\in |\partial E^*|\times [0,s_1]$.
\end{proof}

The BLD homotopy in the previous lemma is used to reduce the extension problem of a roof-admissible map to a simpler problem. Before stating the next proposition, we make a remark to explain its role in constructing extensions.

\begin{remark}
Proposition \ref{prop:X_E-extension} below divides the construction of an extension of a roof-admissible map $|\partial E^*|\to \bS^{n-1}\times \R$ into two separate steps: construction of a collar map $H$, which reduces the extension problem to a simpler problem, and the solution $G$ to this  more accessible extension problem. In applications, the collar map $H$ is a branched cover homotopy constructed by applying the Quasiregular deformation theorem (Theorem \ref{thm:QR-deformation-star}) and its corollaries.
\end{remark}

\begin{proposition}[Extension of roof-admissible maps]
\label{prop:X_E-extension}
Given $\sL\geq 1$ and $\lambda\ge 1$, there exists $\Lfundamental=\Lfundamental(n,\nu,\lambda; \sL)\ge 1$ for the following. 
Let $E\in \cD_r(\lambda)$ and $\widetilde f\colon |\partial E^*|\to \bS^{n-1}\times \R$ be a roof-admissible map. Suppose that there exist
\begin{enumerate}
\item  a roof-admissible map $\widetilde f' \colon |\partial E^*|\to \bS^{n-1}\times \R$ satisfying $\widetilde f' \ll \widetilde f$,
\item an $\sL$-BLD map $H\colon |\partial E^*|\times [0,3^{\nu r} ] \to \bS^{n-1}\times \R$ satisfying 
\[
H|_{|\partial E^*|\times \{0\}} \doteq \widetilde f \,\,\,\, \text{and}\,\,\, H|_{|\partial E^*|\times \{3^{\nu r}\}} \doteq \widetilde f',
\] 
and \label{item:H-G-1}
\item an $\sL$-BLD map $G \colon |E^*| \to \bS^{n-1} \times \R$ satisfying $G|_{|\partial E^*|} \doteq \widetilde f'$.\label{item:H-G-2}
\end{enumerate}
Then there exists an $\Lfundamental$-BLD map  $\widetilde F\colon |E^*|\to \bS^{n-1}\times \R$ for which
\[
\widetilde F|_{|\partial E^*|}\colon  |\partial E^*|\to  \bS^{n-1} \times \R
\]
is a roof-admissible map equal to $\widetilde f$ modulo $\omega^+(E)$.
\end{proposition}

\begin{remark}The  number $3^{\nu r}$ in the proposition is the common side length of the $n$-cubes in $T_E$, hence also the difference of the top level and the bottom level of the admissible maps on $\partial E^*$.
\end{remark}

\begin{proof}[Proof of Proposition \ref{prop:X_E-extension}] 
Let 
\[
X_E = \left( |\partial E^*|\times  [0,3^{r \nu}] \right) \cup \left( |E^*|\times \{3^{r \nu}\} \right).
\]

We begin with an observation that there exists an $L$-bilipschitz homeomorphism $\theta \colon X_E\to |E^*|\times \{0\}$ which is the identity on $|\partial E^*|\times \{0\}$, for a constant $L=L(n,\lambda)\ge 1$. 

Indeed, by Proposition \ref{prop:flattening-bent-indentation}, $|E^*|$ is  bilipschitz equivalent to $|A_E|$, and thus also to $[0,3^{r \nu}]^n$ with a constant depending only on $n$ and $\lambda_\loc$. Moreover, there is a level preserving bilipschitz homeomorphism from $|\partial E^*|\times  [0,3^{r \nu}]$ to $\partial ([0,3^{r \nu}]^n) \times [0,3^{r \nu}]$. Thus $X_E$ is bilipschitz equivalent to $|\partial Q-q|$, where $Q=[0,3^{r \nu}]^{n+1}$ and $q=[0,3^{r \nu}]^n \times \{0\}$. Since all  bilipschitz constants depend only on $n$ and $\lambda$, there exists a bilipschitz homeomorphism $\theta$ as claimed.

Let $(a;b_0,\ldots, b_m)$ and $(a';b_0',\ldots, b'_m)$ be the levels of $\widetilde f$ and $\widetilde f'$. Since $\widetilde f' \ll \widetilde f$, we have that $a<a'$, $b_0>b_0'$, and $b_j<b'_j$ for $j=1,\ldots, m$.

Since $G|_{|\partial E^*|}$ is roof-admissible and equivalent to $\widetilde f'$ modulo $\omega^+(E)$, we have that $G(|\omega^+(E)|) = \bS^{n-1}\times \{a'\}$. Let now $f\colon |\omega^+(E)|\to \bS^{n-1}$ be the map defined by $G(x)=(f(x),a')$ for $x\in |\omega^+(E)|$.

Since $G|_{|\partial E^*-\omega^+(E)|}= H|_{|\partial E^*-\omega^+(E)|\times \{3^{\nu r}\}}$, the mapping $\widetilde H \colon X_E\to \bS^{n-1}\times \R$, 
defined by 
\begin{align*}
\widetilde H (x,t)
= \begin{cases}
 H(x,t) &\text{on} \,\,\,  (|\partial E^*-\omega^+(E)| \times [0,3^{r \nu}]),\\
G(x)  &\text{on} \,\,\,  |E^*|\times \{3^{r \nu}\},\\
\left(f(x), b(t) \right)  &\text{on}\,\,\,  |\omega^+(E)|\times [0,3^{r \nu}],
\end{cases}
\end{align*}
where $b\colon  [0,3^{r \nu}] \to [b'_0,b_0]$ is the function $t\mapsto b_0' + (b_0-b_0')\frac{3^{r \nu}-t}{3^{r \nu}}$, is well-defined and BLD. 
Clearly, $\widetilde H|_{|\partial E^*|\times \{0\}}= H|_{|\partial E^*|\times \{0\}}\doteq \widetilde f$.

The extension $\widetilde F \colon |E^*|\to \bS^{n-1} \times \R$ of $\widetilde f$ modulo $\omega^+(E)$ may now be defined as the composition
\[
\widetilde F = \widetilde H \circ \theta^{-1} \colon |E^*|\times \{0\}\to \bS^{n-1}\times \R,
\]
where we identify $|E^*|\times \{0\}$ with $|E^*|$. Mapping $\widetilde F$ is $\Lfundamental$-BLD for a constant $\Lfundamental$ depending only on $n, \nu, \lambda$, and $\sL$. 
\end{proof}

\chapter{Extensions over undented tunnels}

In the next two sections, we give extensions over tunnels and refined tunnels. These two cases can be viewed as a preliminary for the general case of dented refined tunnels.

\section{Extension over tunnels}
\label{sec:extension-tunnels}

Our extension result for tunnels reads as follows.

\begin{proposition}
\label{prop:extension-over-tunnels}
There exists $\sL^\Diamond=\sL^\Diamond(n) \ge 1$ for the following.
Let  $T \in \cT(\lambda)$ be a tunnel with a flat metric, $\Omega$ a roof cube of $T$, and $T^*= T\cup \Omega$. Let $\sL\ge \sL^\Diamond$, and let $f\colon |\partial T|\to \bS^{n-1}$ be an $\sL$-BLD-controlled expansion of a $T^\Delta$-Alexander map and  $\widetilde f \colon |\partial T^*|\to \bS^{n-1}\times \R$ be an $\sL$-BLD roof adjustment of $f$.  Then there exist  $\sL'=\sL'(n, \nu, \lambda; \sL)\geq 1$ and an $\sL'$-BLD extension 
\[
\widetilde F \colon |T^*|\to \bS^{n-1}\times [0,1]
\]
for which $\widetilde F|_{|\partial T^*|}\colon  |\partial T^*|\to  \bS^{n-1} \times \R$ is an $\sL$-BLD roof-admissible map satisfying 
\begin{enumerate}
\item $\widetilde F|_{|\partial T^*|} \doteq \widetilde f$, and
\item $\widetilde F|_{|\omega^+(T;\Omega)|}\colon  |\omega^+(T;\Omega)| \to  \bS^{n-1} \times \{1\}$ is an $\sL$-BLD-controlled expansion of an $(\omega^+(T;\Omega))^\Delta$-Alexander map, of degree $\deg(\widetilde f|_{|\partial T^* \cap\, \partial T|})$.
\end{enumerate}
\end{proposition}

\begin{proof}
The proof is by induction on $\lambda$. Recall that a tunnel of length zero is defined to be an $(n-1)$-cube.

Clearly, the claim holds for $\lambda=0$. Suppose that the claim holds for all tunnels $T$ of size $\lambda -1$, that is, for a tunnel $T$ of size $\lambda-1$, a roof cube $\Omega$ of $T$, and a roof $\sL$-admissible map $f\colon |\partial T^*|\to \bS^{n-1}\times \R$ there exists an $\sL'(n,\lambda-1;\sL)$-BLD map $F\colon |\partial T^*|\times [0,1]\to \bS^{n-1}\times [0,1]$ satisfying the properties of the claim.

Let now $T$ be a tunnel of length $\lambda$, $\Omega$ be a roof cube of $T$, and let $\tilde f \colon |\partial (T\cup \Omega)|\to \bS^{n-1}\times \R$ be a roof-admissible map.

Let $Q\in T^{[n]}$ be a leaf of $T$ for which $T' = T-Q$ is a tunnel with a roof cube $\Omega$. Let $q = T'\cap Q$ and let $q'$ be the face of $Q$ opposite to $q$. We fix now a piecewise affine map $\phi \colon |\partial Q - q| \to |q|$ for which $\phi|_{q'}$ is a scaling by $1/3$ onto the center cube $c_q$ of $q$ and that all other faces of $Q$ are mapped affinely. Let also $\widetilde f' \colon |\partial T'|\to \bS^{n-1}\times \R$ be the mapping defined by $\widetilde f'|_{|\partial T-q|} = \widetilde f$ and $\widetilde f'|_{|q|} = \widetilde f \circ \phi^{-1}$. By Lemma \ref{lemma:roof-admissible-homotopy}, there exists an $\sL'(n,\lambda;\sL)$-BLD map $F' \colon |\partial T'|\times [0,1]\to \bS^{n-1}\times \R$ from $\widetilde f'$ to an $L$-Lipschitz roof-admissible map on $|\partial T'|$. Thus, by the induction assumption, there exists an $\sL''(n,\lambda;\sL)$-BLD map $\widetilde F' \colon |T'|\to \bS^{n-1}\times \R$ which is an extension of the roof-admissible map  $F'|_{|\partial T'|\times \{1\}}$ modulo $\omega^+(T;\Omega)$ as in the claim.

We extend now $\phi \colon |\partial Q-q|\to |q|$ to a bilipschitz homeomorphism $\Phi \colon |T|\to |T'|$ with a distortion constant depending only on $n$ and define $\widehat F\colon |T|\to \bS^{n-1}\times \R$ to be the mapping $\widehat F = \widetilde F \circ \Phi$. The claim follows now from an application of Proposition \ref{prop:X_E-extension} to $T\cup \Omega$. This completes the construction.
\end{proof}

\section{Extension over refined tunnels}

We give first an extension result for refined tunnels in $\cA(\lambda, \mu)$.

Let $A\in \cA_r(\lambda)$ be the refinement $\Refine^{\nu r}(T)$ of a tunnel $T$. 
As before,  $\SL_A(Q)$ denotes the side length of a cube $Q$ in $|A|$ and that cubes in $T$, hence the roof complexes, have side length $3^{\nu r}$.

\begin{proposition}
\label{prop:extension-over-refined-tunnels}
There exists $\sL^\Diamond = \sL^\Diamond(n)\ge 1$ for the following. Let $A\in \cA_r(\lambda)$ be a refined tunnel with a flat metric,  $\Omega$ be a roof complex for $A$, and $\Omega_1,\ldots, \Omega_m$ be inner roof complexes in $A$ satisfying, for $j\ne i$, $\dist(|\Omega_i|,|\Omega_j|) \ge \max\{ \SL_A(|\Omega_i|), \SL_A(|\Omega_j|)\}$; let 
\[
A^*=(A\cup \Omega)-(\Omega_1,\ldots, \Omega_m).
\]
Let also $\sL\ge \sL^\Diamond$, $f\colon |\partial A|\to \bS^{n-1}$ an $\sL$-BLD-controlled expansion of an $(\partial A)^\Delta$-Alexander map, and let 
\[
\widetilde f\colon |\partial A^*|\to \bS^{n-1}\times \R
\]
be an $\sL$-BLD roof adjustment of $f$ with levels 
\[
(0; \SL_A(|\Omega|), \SL_A(|\Omega_1|), \ldots, \SL_A(|\Omega_m|)).
\]
Then there exist $\sL'=\sL'(n,\nu, \lambda; \sL)\ge 1$ and an $\sL'$-BLD extension $\widetilde F\colon |A^*|\to \bS^{n-1} \times [0,3^{\nu r}]$ of $\widetilde f$ modulo $\omega^+(A;\Omega)$  for which $\widetilde F|_{|\partial A^*|} \colon |\partial A^*|\to \bS^{n-1}\times \R$ is an $\sL$-BLD roof-admissible map satisfying
\begin{enumerate}
\item $\widetilde F|_{|\partial A^*|} \doteq \widetilde f$, and
\item $\widetilde F|_{|\omega^+(A;\Omega)|}\colon  |\omega^+(A;\Omega)| \to  \bS^{n-1} \times \{3^{\nu r}\}$ is an $\sL$-BLD-controlled expansion of an $(\omega^+(A;\Omega))^\Delta$-Alexander map.
\end{enumerate}
\end{proposition}

We emphasize that, in the Proposition,  the BLD constant $\sL'$ does not depend on the refinement index $r$.

\begin{proof}
The proof is an induction on $\lambda$. Clearly, the claim holds for $\cA^*_r(0)$. 

Assume now that the claim holds for all complexes in $\cA^*_r(\lambda-1)$. 
Let $A\in \cA_r(\lambda)$, $\Omega$ be a roof complex of $A$, and $\Omega_1\ldots, \Omega_m$ be inner roof complexes in $A$. Let  $\widetilde f \colon |\partial A^*|\to \bS^{n-1}\times \R$ be the $\sL$-BLD roof-admissible map, with levels $(a; b_0,\ldots, b_m)$, obtained by roof adjustment of a BLD expansion $f\colon |\partial A|\to \bS^{n-1}$ of an Alexander map.
 Since $A$ has no bent indentation, all inner roof complexes $\Omega_j$ have faces $|\omega(A;\Omega_j)|$ in $|\partial A|$. 

In the following, we construct first an extension   $ F\colon |A \cup \Omega|\to \bS^{n-1}\times \R$ of $f$ (modulo $\omega^+(A;\Omega)$), then show that this extension induces an extension $\widetilde F \colon |A^*|\to \bS^{n-1}\times \R$ of $\widetilde f$ (modulo $\omega^+(A;\Omega)$).

\medskip
\noindent \emph{Step 0:}
As a preparatory step for the induction, let $A=\Refine^{\nu r}(T)$ be a refinement of tunnel $T\in \cT(\lambda)$, and $Q_0\in T^{[n]}$ be the $n$-cube for which $\Refine^{\nu r}(Q_0) \cap \Omega = \omega(A;\Omega)$. We  give $\Gamma(T)$ a partial order for which $Q_0$ is the root, and  fix  a leaf $Q\in T^{[n]}$ in this partial order. Then $q=Q\cap (T-Q)$ is a face of $Q$,  $T'=T-Q$ is a tunnel in $\cT(\lambda-1)$, and 
\[A' = \Refine^{\nu r}(T')\in \cA_r(\lambda-1).\]
 Let  $q^+$ be the face of $Q$ opposite to $q$. 

We may assume that $\sptMod f \cap |\partial A\cap Q|=\emptyset$, that is, $ f|_{|\partial A\cap Q|}$ is an $\partial A^\Delta$-Alexander map. Indeed, since $f$ is a controlled expansion of an $\partial A^\Delta$-Alexander map,  we first apply Proposition \ref{prop:move} to move all the simple covers in $\partial A \cap Q$ to $\partial A \cap \partial A'$ by a BLD homotopy $|\partial A|\times [0,3^{\nu r}]\to \bS^{n-1}\times [0,3^{\nu r}]$ from $f$ to a BLD-controlled expansion $|\partial A|\to \bS^{n-1}$ with all simple covers in $\partial A \cap \partial A'$.

We define now bands of cubes in cylinder $\sigma = \Refine^{\nu r}(\partial Q-(q\cup q^+))\subset \partial A$ as follows. We say that an $(n-1)$-cube $c\in \sigma^{[n-1]}$ belongs to $B_1 = \Band_1(\sigma;q^+)$ if $c\cap q^+\ne \emptyset$. Next we define inductively that $c$ be belongs to $B_j = \Band_j(\sigma;q^+)$ if $c$ shares a face with an $(n-1)$-cube $c'$ which belongs to $B_{j-1} = \Band_{j-1}(\sigma;q^+)$. Note that this $(n-1)$-cube $c'$ is unique and we denote $c_\Band=c'$ in what follows. Under this definition,  $(n-1)$-cubes $c\in \sigma^{[n]}$ which meet $q$ belong to band $B_{3^{\nu r}}=\Band_{3^{\nu r}}(\sigma;q^+)$. 

As in the proof of Proposition \ref{prop:extension-over-tunnels}, we fix a piece-wise affine map 
\[
\phi \colon |\partial Q-q|\to |q|
\]
 which maps the face $q^+$ of $Q$, opposite to $q$, by scaling $1/3$ onto the center cube $c_q$ of $q$ and that all other faces of $Q$ are mapped affinely. We extend $\phi$ as a map $\phi \colon |\partial A|\to |\partial A'|$ by setting it to be the identity map in the complement of $|\partial Q-q|$. 

We define
\[ 
f' \colon |\partial A'|\to \bS^{n-1}\times \R
\]
by setting $ f'|_{|\partial A-q|} = f$, and $ f'|_{|q|} =  f \circ \phi^{-1}$. 

In the following, we deform  $f'$ to a BLD expansion of an $(\partial A')^\Delta$-Alexander map in $3^{\nu r}$ (essentially identical) steps. In fact, the deformation is first performed for $f$ on bands in $|\partial Q - (q\cup q^+) |$, a pair at a time.

\medskip
\noindent
\emph{Step $1$ (Merge of $B_1$ and $B_2$):} 
Since $B_1\cup B_2 = \Band_1(\sigma;q^+)\cup \Band_2(\sigma;q^+)$ is isometric to complex $(\partial q^+)\times [-1,1]$, we may fix an isometry $\Psi_1 \colon (\partial q^+)\times [-1,1]\to |\Band_1(\sigma;q^+)\cup \Band_2(\sigma;q^+)|$ for which $\Psi_1((\partial q^+)\times [-1,0])=|\Band_2(\sigma;q^+)|$. Let  $\tau_1 \colon (\partial q^+)\times [-1,0] \to (\partial q^+)\times [-1,1]$ be the map $(x,t)\mapsto (x,2t+1)$, and denote by $\tau = \Psi_1 \circ \tau_1 \circ \Psi_1^{-1} \colon |B_2| \to |B_1\cup B_2|$ the stretching map. 

By Corollary \ref{cor:Band-merge} on deformation and Lemma \ref{lemma:roof-admissible-homotopy}, there exists an $\sL^\Diamond$-BLD homotopy 
\[F_1 \colon |B_1\cup B_2|\times [0,1]\to \bS^{n-1}\times [0,1] \,\,\,(\text{modulo}\,\,  |\partial (B_1\cup B_2)|)\]  from $ f|_{|B_1\cup B_2|}$ to an $\sL^\Diamond$-BLD-controlled expansion $ f_1 \colon |B_1\cup B_2|\to \bS^{n-1}$ of an $\tau_*(B_2^\Delta)$-Alexander map.

We define now a BLD homotopy 
$F'_1 \colon |\partial A'|\times [0,1]\to \bS^{n-1}\times [0,1]$  by 
\[
(x,t) \mapsto \left\{ \begin{array}{ll}
 F_1\circ (\phi \times \id)^{-1}(x,t), & x\in |\phi_*(B_1\cup B_2)| \\
(f'(x),t), & x\not \in  |\phi_*(B_1\cup B_2)|
\end{array}\right.
\]
 
We fix next a level preserving bilipschitz isotopy $\psi_1\colon |\partial A'|\times [1,2]\to |\partial A'|\times [1,2]$ from the identity map to the bilipschitz map $h\colon |\partial A'|\to |\partial A'|$, which is the identity in the complement of $\phi(|q^+\cup B_1\cup B_2|)$ and which maps, by scaling, $\phi(|q^+|)$ to $\phi(|q^+\cup B_1|)$, and $\phi(|B_1\cup B_2|)$ to $\phi(|B_2|)$ for which $h_*\circ \phi_*\circ \tau_*(B_2^\Delta) = \phi_*(B_2^\Delta)$; the structure on $\phi(|q^+\cup B_1|)$ is now $h_*\circ \phi_*((q^+\cup B_2)^\Delta)$.

We extend  $F'_1$, by $\psi_1$, to a BLD homotopy $|\partial A'|\times [0,2]\to \bS^{n-1}\times [0,2]$, and then rescale $[0,2]$ to the interval $[0,1]$ by $(x,t)\mapsto (x,t/2)$. The map thus obtained, again denoted by $F'_1$, is a homotopy from $f'$ to a BLD-controlled expansion $f'_1  \colon |\partial A'|\to \bS^{n-1}$ of an Alexander map on complex  
\[
(\partial A'-q)^\Delta \cup \phi_*(B_{3^{\nu r}} \cup \cdots \cup B_2)^\Delta \cup (h_1)_*\phi_*(q^+)^\Delta.
\]
Applying  Proposition \ref{prop:move} again, we move simple covers on $\sptMod f'_1 \cap |\phi_*(B_2)|$ by homotopy into $|\phi_*(B_4)|$, and may assume that $\sptMod f'_1|_{|\phi_*(B_2)|}$ is a $\phi_*(B_2)^\Delta$-Alexander map.

To keep the distortion under control, we assume, as we may, that  BLD homotopy $F'_1$ moves the existing simple covers by similarities only.

\medskip
\noindent
\emph{Steps $2$-$3^{\nu r}$ (Merge of $B_{j-1}$ and $B_j$):} By repeating Step 1 inductively for $j=2,\ldots, 3^{\nu r}$, we merge the band $B_{j-1}$ with $B_j$ and scale the space of the complex $\phi_*(q^+)$ from 
$|\phi_*(B_{j-2} \cup \cdots \cup B_1 \cup q^+)|$ to $|\phi_*(B_{j-1}\cup \cdots \cup B_1 \cup q^+)|$. 

This produces, for each $j$, a BLD homotopy 
\[
F'_j \colon |\partial A'|\times [j-1,j]\to \bS^{n-1}\times [j-1,j]
\]
from $f'_{j-1}=F'_{j-1}|_{|\partial A'|\times \{j-1\}} \to \bS^{n-1}$ to a BLD-controlled expansion $f'_j \colon |\partial A'|$\, $\to \bS^{n-1}$ of an Alexander map.  Again for the construction, at  the end of each step, we move the simple covers of $f'_{j-1}$ on $|\phi_*(B_{j-1}\cup B_j)|$ into $|\phi_*(B_{j+2})|$; finally into $|\partial A'-q|$ for $j\ge 3^{\nu r}-1$. Also, all simple covers, after created, are moved by similarities only.

\medskip
\noindent
\emph{Final step (for $f'$ and $f$):} Putting  these steps together, we obtain a BLD homotopy 
\[
F'\colon |\partial A'|\times [0,3^{\nu r}] \to \bS^{n-1}\times [0,3^{\nu r}]
\]
from $f'$ to a BLD-controlled expansion $f'' \colon |\partial A'|\to \bS^{n-1}$ of an $(\partial A')^\Delta$-Alexander map. By Lemma \ref{lemma:roof-admissible-homotopy}, homotopy $F'$ gives rise to a  BLD homotopy $H \colon |\partial (A'\cup \Omega)|\times [0,3^{\nu r}] \to \bS^{n-1}\times \R$ from the roof-admissible  $\widetilde {f'}$ to a roof adjustment $\widetilde {f''}$ of $f''$.

Since $A'\in \cA_r(\lambda -1)$, by the induction assumption, there exists a BLD map $G \colon |A'\cup \Omega|\to \bS^{n-1}\times \R$ satisfying $G_{|\partial (A'\cup \Omega)|} \doteq  \widetilde {f''}$. Thus, by Proposition \ref{prop:X_E-extension}, there exists a BLD map $\widetilde {F'} \colon |A'\cup \Omega| \to \bS^{n-1}\times \R$ for which $\widetilde {F'}|_{|\partial (A'\cup \Omega)|} \doteq \widetilde {f'}$. It suffices now to extend $\phi\colon |\partial A| \to |\partial A'|$ to a bilipschitz homeomorphism $\Phi \colon |A\cup \Omega|\to |A'\cup \Omega|$ and define $ F \colon |A\cup \Omega| \to \bS^{n-1}\times \R$ by $ F = \widetilde {F'} \circ \Phi$.

\medskip
\noindent
\emph{Final step (for $\widetilde f$):} We now adjust  the  map $F$, in a neighborhood of each inner roof complex, to obtain a mapping  $\widetilde F \colon |A^*| \to \bS^{n-1}\times \R$ which extends $\widetilde f$ modulo $\omega^+(A;\Omega)$. This is done by precomposing $F$ with a bilipschitz homeomorphism.

By assumption, $\dist(|\Omega_j|,|\Omega_i|) \ge \max\{ \SL_A(|\Omega_j|), \SL_A(|\Omega_i|)\}$ for $j\ne i$. Since also $\dist(|\omega^+(A;\Omega_j)|, |\partial A|) = \SL_A(|\Omega_j|)$ for each $j$, we may fix mutually disjoint cubes $Q_1,\ldots, Q_m$ in $|A|$ containing cubes $|\Omega_1|,\ldots, |\Omega_m|$, respectively, which satisfy $\dist(|\Omega_j|, |A|\setminus Q_j)=\SL_A(|\Omega_j|)/2$ for each $j$. Since each $|\Omega_j|$ has a face in $|\partial A|$, so does each $Q_j$. Let $q_j = Q_j\cap |\partial A|$.

As a preparatory step, let, for each $j$, $H_j \subset Q_j$ be the convex hull of $|\Omega_j|\cup q_j$. 
We also fix an orthogonal projection $\pi_j \colon Q_j \to q_j$ and denote $\varphi_j \colon Q_j \to Q_j$ the map, which is the identity on $\partial Q_j$ and, for each $x\in q_j$, maps segments $\pi_j^{-1}(x)\cap \cl(Q_j\setminus H_j)$ linearly to segments $\pi^{-1}_j(x)$ and segment $\pi_j^{-1}(x)\cap H_j$ to point $x$; see Figure \ref{fig:refined-tunnels-convex-hull} for an illustration. The mapping $\varphi_j$ is bilipschitz with a constant depending only on $n$.

\begin{figure}[htp]
\begin{overpic}[scale=0.5,unit=1mm]{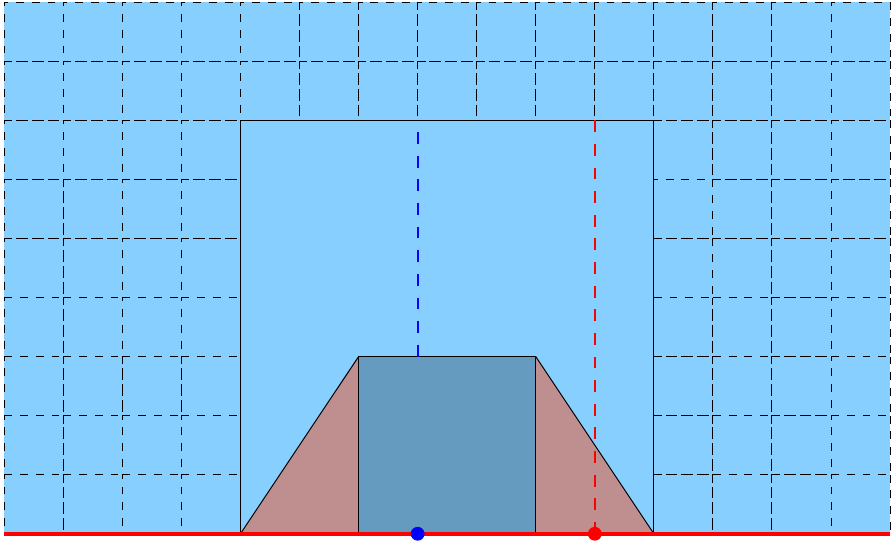} 
\put(40,10){\tiny $\Omega_j$}
\put(21,31){\tiny $Q_j$}
\put(25,3){\tiny $H_j$}
\put(36,30){\tiny $L_j(x)$}
\put(36,-2){\tiny $x$}

\put(51,-2){\tiny $y$}
\put(41, 21){\tiny $\pi_j^{-1}(y)$}
\end{overpic}
\caption{Illustration on the configuration for the map $\varphi_j \colon Q_j\to Q_j$, where $L_j(x)=\pi_j^{-1}(x)\cap \cl(Q_j\setminus H_j)$ and $x,y\in q_j$; cubes $\Omega_j$ and $Q_j$ not in scale.}
\label{fig:refined-tunnels-convex-hull}
\end{figure}

We redefine now the mapping $F$ in each $Q_j$ so that the new map satisfies the conditions of the roof map on $\partial \Omega_j\setminus q_j$. 

Write $F=(F_1, F_2) \colon |A\cup \Omega|\to \bS^{n-1}\times \R$, where  $F_1 \colon |A\cup \Omega |\to \bS^{n-1}$ and $F_2 \colon |A\cup \Omega|\to \R$ are the component mappings of $F$. By precomposing $F$ with a bilipschitz homeomorphism of $|A|$, which maps $|\omega^+(A;\Omega_j)|$ closer to $\Omega$, we may assume that $F_2(x)\ge 3 b_j$ for $x\in |\omega^+(A;\Omega_j)|$. Note that, we may choose this homeomorphism so that the bilipschitz constant depends only on $\lambda$ and $\sL$. 

We set first $g_j=F_1 \circ \varphi_j \colon Q_j \to \bS^{n-1}$. It remains to find a function $h_j \colon Q_j \to \R$ for which the map $(g_j,h_j) \colon Q_j \to \bS^{n-1}\times \R$ agrees with $F$ on $\partial Q_j$ and satisfies the conditions the roof map. Thus it remains to find a function $h_j$ which agrees with $F_2$ on $\partial Q_j$, $h_j$ maps $|\omega^+(A;\Omega_j)|$ to a fixed level $b_j$ and maps the cyclinder $\sigma(A;\Omega_j)$ affinely. Thus the function $h_j$ is prescribed on $|\omega^+(A;\Omega_j) \cup \sigma(A;\Omega_j)|$, up to an affine scaling. Since $h_j$ is constant function $a$ on $q_j$, we may further affinely extend $h_j$ to the convex hull $H_j$. We now extend $h_1$ to $\cl(Q_j\setminus H_j)$ affinely on segments $\pi_j^{-1}(x) \cap \cl(Q_j\setminus H_j)$. This completes the construction of the function $h_j$.

We define  $\check F \colon |A\cup \Omega|\to \bS^{n-1}\times \R$ by setting $\check F= F$ in the complement of $Q_1\cup \cdots \cup Q_m$, and  set $\check F|_{Q_j} = ( g_j,  h_j)$ for each $Q_j$. Finally, take $\widetilde F\colon |A^*|\to \bS^{n-1}\times \R$ to be the restriction of $\check F$ to $|A^*|$. This concludes the construction of the mapping in the statement.
\end{proof}

\chapter{Extension over dented refined tunnels}
\label{chap:extension-differences}

\section{Extension statement}

The general strategy of extension over roof-adjusted dented refined tunnels is to restore the dents, that is, a reduction to Proposition \ref{prop:extension-over-refined-tunnels}.

The proof is based on a case study on a classification of $n$-cubes in the indentation $D_E$ of a dented refined tunnel $E$. To avoid complications with the classification and to avoid unnecessary cases from the point of view of the Quasiregular cobordism theorem, we restrict in the proof to the of dented refined tunnels appearing in the reservoir-canal-construction. For this reason, we state the result for the subset 
\[
\cD_k(\Sigma) =  \left(\bigcup_{\lambda=1}^\infty \cD_k(\lambda)\right) \cap \mathcal{Real}_k(\Sigma),
\]
instead of all dented refined tunnels in $\cD_k(\lambda)$ for some $\lambda\ge 1$.

We also assume that each complex $E$ in $\cD_k(\Sigma)$ we have fixed a roof cube $\Omega_E$ and inner roof cubes $\Omega_{E,1},\ldots, \Omega_{E,s_E}$, according to the hierarchy of the elements of $\mathcal{Real}_k(\Sigma)$. We denote 
\[
E^* = (E\cup \Omega_E)-(\Omega_{E,1}\cup \cdots \cup \Omega_{E,s_E}) 
\]
this roof adjusted dented refined tunnel, and $\cD^*_k(\Sigma)$ their collection.
\index{$\cD_k(\Sigma)$}
\index{$\cD^*_k(\Sigma)$}

\begin{proposition}
\label{prop:extension-over-differences}
There exists $\sL^\Diamond = \sL^\Diamond(n)\ge 1$ for the following. Let $\sL \ge \sL^\Diamond$, let $k\ge 1$ and $E\in \cD_k(\Sigma)$ be a dented refined tunnel with a flat metric from $\Real_k(\Sigma)$, and let
\[
E^*=(E\cup \Omega_E)-(\Omega_{E,1}\cup \cdots \cup \Omega_{E,s_E}) \in \cD^*_k(\Sigma)
\]
be the associated roof-adjusted dented refined tunnel. Let also $f\colon |\partial E|\to \bS^{n-1}$ be an $\sL$-BLD-controlled expansion of a $(\partial E)^\Delta$-Alexander map, and  $\widetilde f\colon |\partial E^*|\to \bS^{n-1}\times \R$ be an $\sL$-Lipschitz roof adjustment of $f$.

Then there exist a constant $\sL'=\sL'(n,\nu,\lambda,\mu; \sL)$ and an $\sL'$-BLD extension 
\[
\widetilde F\colon |E^*| \to \bS^{n-1}\times [0,3^{\nu k}]
\]
of $\widetilde f$ modulo $\omega^+(E;\Omega)$, for which $\widetilde F|_{|\partial E^*|}\colon  |\partial E^*|\to  \bS^{n-1} \times \R$ is an $\sL$-Lipschitz roof-admissible map for which
\begin{enumerate}
\item $\widetilde F|_{|\partial E^*|} \doteq \widetilde f$, and
\item $\widetilde F|_{|\omega^+(E;\Omega)|}\colon  |\omega^+(E;\Omega)| \to  \bS^{n-1} \times \{3^{\nu k}\}$ is an $\sL$-BLD-controlled  expansion of an $(\omega^+(E;\Omega))^\Delta$-Alexander map. 
\end{enumerate}
\end{proposition}

\begin{remark}
To simplify the discussion, we write the proof in the case $k \ge 2$. The case $k=1$ is similar.
\end{remark}

For the proof, we introduce a peeling process for bent indentations, which reduces the extension problem to the extension problem for refined tunnels. For the peeling process, we introduce the notion of layers of bent indentations and classify cubes in $D_E$ in terms of their position in layers.

\section{Layers of bent indentation}

For the proof of Proposition \ref{prop:extension-over-differences}, we describe a peeling process, based on layers, for a bent indentation $D_E$ of $E$. Let $E\in\cD_k(\Sigma)$ and let $(T_E,A_E,D_E)$ be its representation,  where $D_E\subset A_E$ is a bent indentation.

\begin{notation}
\label{notation:layer}\index{layer} \index{$\Layer_j(D_E)$}
We define $\Layer_j(D_E) \subset D_E$ to be the subcomplex of $D_E$ spanned by $n$-cubes $Q\in D_E^{[n]}$ determined as follows:
\begin{enumerate}
\item $Q$ belongs to the $1$st layer $\Layer_1(D_E)$ if $Q\cap \partial A_E\ne \emptyset$, and 
\item $Q$ belongs to the $j$th layer $\Layer_j(D_E)$ for $j\ge 2$ if $Q$ has a non-empty intersection with an $n$-cube in $\Layer_{j-1}(D_E)$.
\end{enumerate}
For completeness, we set also $\Layer_0(D_E) = D_E\cap \partial A_E$.  
\end{notation}

We denote $j_E\ge 1$ the number of non-empty layers, that is, the smallest number $j\ge 1$ for which $\Layer_{j+1}(D_E)=\emptyset$. So, $1\leq j_E\leq 3^{\nu r}$. For completeness, we set also $\Layer_0(D_E) = D_E\cap \partial A_E$ and $\Layer_{j_E+1}(D_E)=\emptyset$.

\subsection{Trichotomy of $(n-1)$-cubes in $\partial \Layer_j(D_E)$}

Before discussing the case of $n$-cubes, we divide the $(n-1)$-cubes on the boundary of $D_E$ into three classes. For the complex $\partial \Layer_j(D_E), 1\leq j\leq j_E$, we have a partition
\[
\partial \Layer_j(D_E)=  (\partial \Layer_j(D_E))^{\floor}\cup (\partial \Layer_j(D_E))^{\ceiling} \cup (\partial \Layer_j(D_E))^{\wall}
\]
into \emph{floor}, \emph{ceiling}, and \emph{wall}, respectively, where  
\begin{enumerate}
\item $(\partial \Layer_j(D_E))^{\floor}=  \Layer_j(D_E) \cap \Layer_{j-1}(D_E)$,
\item $(\partial \Layer_j(D_E))^{\ceiling}$ is the subcomplex of $\partial \Layer_j(D_E)$ spanned by the $(n-1)$-cubes not intersecting $|(\Layer_j(D_E))^{\floor}|$, and 
\item $(\partial \Layer_j(D_E))^{\wall}$ is the subcomplex of $\partial \Layer_j(D_E)$ spanned by the $(n-1)$-cubes not in  $(\partial \Layer_j(D_E))^{\floor}\bigcup (\partial \Layer_j(D_E))^{\ceiling}$. 
\end{enumerate}
\index{floor}
\index{ceiling}
\index{wall}

Regarding individual $(n-1)$-cubes in $\partial \Layer_j(D_E)$, we use the following terminology. 

\begin{definition} \index{floor cube} \index{ceiling cube} \index{wall cube}
We say that $q\in \partial \Layer_j(D_E)$ is a \emph{floor cube}, \emph{ceiling cube}, or \emph{wall cube}, if $q$ belongs to $(\partial \Layer_j(D_E))^{\floor}$, $ (\partial \Layer_j(D_E))^{\ceiling}$, or $(\partial \Layer_j(D_E))^{\wall}$, respectively.
\end{definition}
We also say that a floor cube (resp. ceiling cube) $q$  of $\Layer_j(D_E)$ is an \emph{inner cube} of the said type  if $q$ does not meet the boundary of  $(\partial \Layer_j(D_E))^\floor$ (resp. $(\partial \Layer_j(D_E))^\ceiling$). If $q$ is not an inner cube, we say that $q$ is an \emph{edge cube}. Note that edge cubes meet walls.
\index{edge cube}
\index{inner cube}

\subsection{Corner and transition cubes in $\Layer_j(D_E)$}
\label{sec:corners-and-transition-cubes}

We define first corner cubes in $\Layer_j(D_E)$; see Figure \ref{fig:corner-cube-types} for an illustration.

\begin{definition}
\label{def:corner-cube}
An $n$-cube $Q\in \Layer_j(D_E)$ is called a \emph{corner cube} if one of the three (mutually exclusive) subcases holds:
 \begin{enumerate}
\item [(i)] $e_Q=Q\cap \Layer_{j-1}(D_E)\ne \emptyset$ does not contain $(n-1)$-cubes, 
\item [(ii)] $e_Q=Q\cap \Layer_{j+1}(D_E)\neq \emptyset$ does not contain $(n-1)$-cubes, 
\item [(iii)] $Q\cap \Layer_{j+1}(D_E)= \emptyset$ and $Q\cap \Layer_{j-1}(D_E)$ contains two adjacent $(n-1)$-cubes in $Q$.
\end{enumerate} 
\end{definition}

\begin{figure}[h!]
\begin{overpic}[scale=0.8,unit=1mm]{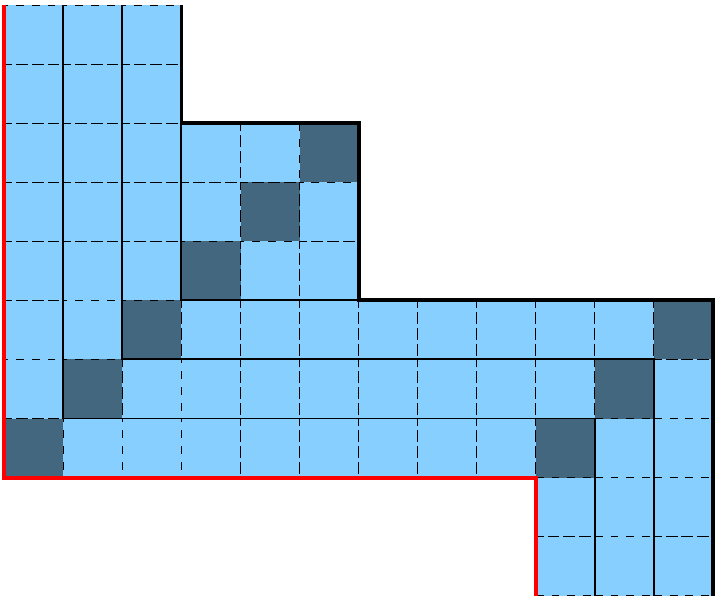} 
\put(30,13){\tiny $\Layer_0(D_E)$}
\put(58,42){\tiny $(\partial \Layer_3(D_E))^{\ceiling}$}

\put(4,76){\tiny $1$}
\put(12,76){\tiny $2$}
\put(20,76){\tiny $3$}


\put(11,28){\tiny {\color{white}$Q_{ii}$}}
\put(4,36){\tiny $1$}
\put(4,28){\tiny $1$}
\put(4,20){\tiny {\color{white}$1$}}
\put(12,20){\tiny $1$}
\put(20,20){\tiny $1$}

\put(12,36){\tiny $2$}
\put(20,28){\tiny $2$}
\put(20,36){\tiny {\color{white}$3$}}

\put(76,20){\tiny {\color{white}$Q_{i}$}}
\put(68,20){\tiny $1$}
\put(76,12){\tiny $1$}
\put(84,28){\tiny {\color{white}$2$}}
\put(84,20){\tiny $2$}
\put(76,28){\tiny $2$}


\put(42,60){\tiny {\color{white}$Q_{iii}$}}
\put(28,60){\tiny $4$}
\put(28,52){\tiny $4$}
\put(36,60){\tiny $5$}
\put(36,52){\tiny {\color{white}$5$}}
\put(44,52){\tiny $5$}
\put(44,44){\tiny $4$}
\put(36,44){\tiny $4$}
\put(28,44){\tiny {\color{white}$4$}}

\end{overpic}
\caption{$\Layer_0(D_E)$ marked with red and layers $1$--$5$ by numbers. Cubes $Q_i$, $Q_{ii}$, and $Q_{iii}$ are corner cubes of the corresponding type in layers $1$, $2$, and $6$, respectively.}
\label{fig:corner-cube-types}
\end{figure}

By the structure of the reservoir-canal construction, a corner cube $Q$ always belongs to a refinement of a connector;  type (i) corner cubes belong to connectors of Case 1 in Definition \ref{def:connector} and types (ii) and (iii) corner cubes belong to connectors of Case 2. Moreover, the intersection $e_Q$ of type (i) or (ii) in Definition \ref{def:corner-cube} is an $(n-2)$-cube. 

\begin{notation}
For a corner cube $Q\in \Layer_j(D_E)$ of type (iii), we set $e_Q$ to be the unique $(n-2)$-cube 
in $Q$ which does not meet  $\Layer_{j-1}(D_E)$.
\end{notation}

For cubes in refinements of canal sections and reservoirs, we introduce another special class of cubes, \emph{transition cubes}. For this reason, we define first edge and inner cubes.

\begin{definition}
\label{def:edge-cube-n-dim}\index{edge cube} \index{inner cube}
An $n$-cube $Q\in \Layer_j(D_E)$ is an \emph{edge cube} if $Q\cap (\partial \Layer_j(D_E))^{\wall}\ne \emptyset$; other $n$-cubes in $\Layer_{j-1}(D_E)$ are inner cubes. 
\end{definition}

Note that a corner cube of type (iii) is always an edge cube. We define also the following subclass of edge cubes.

\begin{definition}
\label{def:transition-cube}\index{transition cube}
An edge cube $Q\in D_E$ is a \emph{transition cube (between two spectra)} if $Q$ is in a refinement of spectrum $\cS_i(D_E)$ and it meets an $n$-cube in a different spectrum $\cS_{i'}(D_E)$ for $i'\neq i$.
\end{definition}

Each transition cube $Q$ in $\Layer_j(D_E)$ is associated to two other transition cubes in $\Layer_j(D_E)$ as follows. Let $Q\in \Layer_j(D_E)$ be a transition cube. Then there are exactly two transition cubes $Q'\in \Layer_j(D_E)$ and $Q''\in \Layer_j(D_E)$ for which $Q\cap Q'\cap Q''$ is an $(n-2)$-cube and two of the cubes $Q,Q',Q''$ are edge cubes. We call the triple $\{Q,Q',Q''\}$ a \emph{cluster of $Q$}. Note that the cluster $\{Q, Q', Q''\} $ of $Q$ forms the shape of the letter $L$.
\index{cluster}
\index{$\{Q,Q',Q''\}$}

\subsection{Classification of $n$-cubes in $\Layer_j(D_E)$}
\label{sec:classification}

We classify now the $n$-cubes in $\Layer_j(D_E)$. We consider the following mutually exclusive categories of $n$-cubes in $\Layer_j(D_E)$ based on the properties of reservoirs, canal sections and connectors. In each case we associate subcomplexes $\Theta(Q) \subset \Layer_j(D_E)$ and $\theta(Q) \subset (\partial \Layer_j(D_E))^\floor$ to each $n$-cube $Q\in \Layer_j(D_E)$ considered.
\index{$\Theta(Q)$}
\index{$\theta(Q)$}

\medskip
\noindent {\bf Category I: Neighborhoods of corners.} 

\medskip
\noindent {\bf{Category I.1A}} Let $Q$ be a corner cube of type (i) in $\Layer_j(D_E)$. We set
\[
\Theta(Q)=\Star_{\Layer_j(D_E)}(e_Q) 
\] 
and observe that $|\Theta(Q)|$ is an $n$-cell. The intersection of $\Theta(Q)$ with the floor,
\[
\theta(Q) = \Theta(Q) \cap (\partial \Layer_j(D_E))^{\floor},
\] 
is an $(n-1)$-cell. Since $Q$ is of type (i), then the numbers of $n$-cubes and $(n-1)$-cubes in $\Theta(Q)$ and $\theta(Q)$ are $3$ and $2$, respectively, due to the assumption $k\ge 2$; in the case $k=1$, the numbers of $n$ and $(n-1)$-cubes in $\Theta(Q)$ and $\theta(Q)$ depend on $\mu(K)$.

\medskip
\noindent {\bf{Category I.1B}} Let $Q$ be a corner cube of type (ii) in $\Layer_j(D_E)$. We set
\[
\Theta(Q)=\Star_{\Layer_j(D_E)}(e_Q) \quad \text{and}\quad \theta(Q) = \Theta(Q) \cap (\partial \Layer_j(D_E))^{\floor}.
\] 
Again $|\Theta(Q)|$ is an $n$-cell and $\theta(Q)$ is an $(n-1)$-cell. Moreover, since $Q$ is of type (ii), complexes $\Theta(Q)$ and $\theta(Q)$ contain three $n$-cubes and four $(n-1)$-cubes, respectively.

\medskip
\noindent {\bf{Category I.1C}} Let $Q$ be a corner cube of type (iii) in $\Layer_j(D_E)$; in this case $\Star_{\Layer_j(D_E)}(e_Q)=Q$. We set 
\[
\Theta(Q)=Q \quad \text{and}\quad \theta(Q) = Q \cap (\partial \Layer_j(D_E))^{\floor}. 
\] 
Then $|\Theta(Q)|$ is an $n$-cube and $\theta(Q)$ is an $(n-1)$-cell which contains two $(n-1)$-cubes.

\medskip

\noindent {\bf{Category I.2.}} $Q$ is an $n$-cube in $\Layer_j(D_E)$, which is not a corner cube, but is adjacent to a corner cube $Q'$ in $\Layer_j(D_E)$ for which $Q\in \Theta(Q')$. In this case, we take
\[
\Theta(Q)=\Theta(Q') \quad \text{and}\quad \theta(Q)= \theta(Q').
\]

\medskip

\noindent{\bf Category II: Cubes away from corners.} Away from the corners, there is a one-to-one correspondence between the $n$-cubes in
\[
(\Layer_j(D_E))^\dagger =\Layer_j(D_E) - \bigcup \{ \Theta(Q)\colon Q \text{ a corner cube in }\Layer_j(D_E)
\}, 
\]
and the $(n-1)$-cubes in 
\begin{align*}
((\partial \Layer_j(D_E))^{\floor})^\dagger 
&= (\Layer_j(D_E))^{\floor} \\
&\quad - \bigcup \{ \theta(Q)\colon Q \text{ a corner cube in }\Layer_j(D_E)\}.
\end{align*}
For an $(n-1)$-cube $q$ in $((\partial \Layer_j(D_E))^{\floor})^\dagger$,  denote by $Q_q$  the unique $n$-cube in $(\Layer_j(D_E))^\dagger $ having $q$ as a face. Conversely, for an $n$-cube $Q \in (\Layer_j(D_E))^\dagger $,  denote by $q_Q$  its $(n-1)$-face in $((\partial \Layer_j(D_E))^{\floor})^\dagger$.

We formulate now the complexes $\Theta(\cdot)$ and $\theta(\cdot)$ in the remaining cases.

\medskip
\noindent {\bf{Category II.1}} If $Q \in \Layer_j(D_E)^\dagger$ is an inner cube, we set $\Theta(Q)=Q$ and $\theta(Q)=q_Q$.

\medskip
\noindent {\bf{Category II.2A.}} If $Q\in \Layer_j(D_E)^\dagger$ is an edge cube, which is an transition cube, we set
\[
\Theta(Q)= Q\cup Q' \cup Q''
\quad \text{and} \quad
\theta(Q)=\Theta(Q)\cap (\partial \Layer_j(D_E))^{\floor},
\]
where $\{Q,Q',Q''\}$ is the cluster of $Q$.

\medskip 

\noindent {\bf{Category II.2B.}} If $Q \in \Layer_j(D_E)^\dagger$ is an edge cube, which is not a transition cube, we set
\[
\Theta(Q)=Q \quad \text{and} \quad \theta(Q)= \Theta(Q)\cap (\partial \Layer_j(D_E))^{\floor}.
\]

\subsection{Equivalence relation based on the classification}
\label{sec:equivalence-relation}

From the classification, we have that the collection $\{\Theta(Q) \colon Q \in \Layer_j(D_E)\}$ is an essential partition of $\Layer_j(D_E)$, and the collection $\{\theta(Q) \colon Q\in \Layer_j(D_E)\}$ is an essential partition of $(\partial \Layer_j(D_E))^{\floor}$. Before we move on, we record a simple observation on these partition.

\begin{lemma}
Let $Q,Q'\in \Layer_j(D_E)$. Then $\Theta(Q)=\Theta(Q')$ if and only if $\theta(Q)=\theta(Q')$.
\end{lemma}
\begin{proof}
Let $Q\in \Layer_j(D_E)$. If $Q$ belongs to Category I and $Q'\in \Layer_j(D_j)$ has the property that $\Theta(Q')=\Theta(Q)$. Then there exits a corner cube $Q''\in \Layer_j(D_E)$ for which $\Theta(Q'')=\Theta(Q') =\Theta(Q)$. Thus also $\theta(Q'')=\theta(Q')=\theta(Q)$.

Suppose now that $Q$ belongs to Category II. If $Q$ is as in Category II.1 or II.2B, then $\Theta(Q)=Q$ and the claim holds. The claim clearly holds in Category II.2A. 
\end{proof}

Let now $\sim_{j-1}$ be the equivalence relation in $(\partial \Layer_j(D_E))^{\floor}$ having equivalence classes $\theta(Q)$ for $Q\in \Layer_j(D_E)$. That is,  $q\sim_{j-1} q'$ if and only if $q$ and $q'$ belong to the same $\theta(Q)$ for some $Q\in \Layer_j(D_E)$. We denote $[q]_{j-1}= \theta(Q)$ the equivalence class of $q$ under $\sim_{j-1}$ and denote 
\[ 
|[q]_{j-1}|=\bigcup \{q'\colon q'\in [q]_{j-1}\} = |\theta(Q)|.
\]

\begin{definition}
\label{def:tent} \index{tent} \index{$\tau([q]_{j-1})$}
A \emph{tent $\tau([q]_{j-1})$ over an equivalence class $[q]_{j-1}$} for $q\in (\partial \Layer_j(D_E))^{\floor}$ is
\[
\tau([q]_{j-1})=  \Theta(Q_q) \cap \left( (\partial \Layer_j(D_E))^{\ceiling} \bigcup (\partial \Layer_j(D_E))^{\wall}\right),
\]
where $Q_q$ is an $n$-cube in $\Layer_j(D_E)$ for which $[q]_{j-1}=\theta(Q_q)$. 
\end{definition}

We have the following finiteness property for tents, which we record as a lemma.

\begin{lemma}
\label{lemma:tents-classes}
The collection 
\[
\{\tau([q]_{j-1}) \colon q\in (\partial \Layer_j(D_E))^{\floor}\}
\]
is an essential partition of $(\partial \Layer_j(D_E))^{\ceiling} \cup (\partial \Layer_j(D_E))^{\wall}$, and collection 
\[
\{([q]_{j-1}, \tau([q]_{j-1})) \colon q\in (\partial \Layer_j(D_E))^{\floor}\}
\]
contains only a bounded number of isomorphic classes with the number depending only on $n$ and $\mu(K)$.
\end{lemma}

As a final observation related to the classification of cubes, we observe that tents are shellable cubical complexes.

\begin{lemma}
\label{lemma:tau-shellable}
Let $1\leq j\leq j_E$ and let $q\in (\partial \Layer_j(D_E))^{\floor}$ be an $(n-1)$-cube. Then the tent $\tau([q]_{j-1})$ is a shellable $(n-1)$-complex whose space $|\tau([q]_{j-1})|$ is an $(n-1)$-cell. 
\end{lemma}

\begin{proof}
Let $Q_q$ be the cube in $\Layer_j(D_E)$ having $q$ as a face. We consider now the five cases of $Q_q$.

\medskip
\noindent \emph{Categories I.1 and I.2.} In these cases $[q]_{j-1}=\theta(Q_q)$ has a linear adjacency graph. Similarly $\Theta(Q_q) \cap \partial \Layer_j(D_E))^\ceiling$ has a linear adjacency graph and $|\Theta(Q_q) \cap \partial \Layer_j(D_E))^\ceiling|$ is an $(n-1)$-cell. Since the complex $\Theta(Q_q)\cap (\partial \Layer_j(D_E))^{\wall}$ is either empty or has a linear adjacency graph, we conclude that $|\Theta(Q_q) \cap \partial \Layer_j(D_E))^\ceiling|$ is either empty or an $(n-1)$-cell which meets in a complex $\Theta(Q_q) \cap \partial \Layer_j(D_E))^\ceiling$, whose space is an $(n-2)$-cell. Thus $\tau([q]_{j-1})$ is shellable and $|\tau([q]_{j-1})|$ is an $(n-1)$-cell.

\medskip
\noindent\emph{Category II.1.} In this case $\tau([q]_{j-1})$ is a single $(n-1)$-cube opposite to $q$ and the both claims hold trivially.

\medskip
\noindent\emph{Category II.2A.} In this case the argument is analogous to Category I.1.

\medskip
\noindent\emph{Category II.2B.} The tent $\tau([q]_{j-1})$ consists of two adjacent $(n-1)$-cubes. Thus the claims hold trivially.
\end{proof}

\begin{remark}
\label{rmk:corner-cube-type-2}
For the forthcoming discussion, we note that, for $Q\in \Layer_j(D_E)$, the tent $\tau(\theta(Q))$ contains as many $(n-1)$-cubes as $\theta(Q)$, that is, $\# (\tau(\theta(Q))^{[n-1]} \ge \# (\theta(Q))^{[n-1]}$, except when $Q$ belongs to Category I.1B. This  exception is the source of the construction in Section \ref{sec:peeling-spare-cubes}.
\end{remark}

We finish this section by recording a summary on the intersections of the tents.
\begin{lemma}
\label{lemma:tau-intersections}
Let $1\leq j\leq j_E$ and let $q$ and $q'$ in $(\partial \Layer_j(D_E))^{\floor}$ be adjacent $(n-1)$-cubes for which $[q]_{j-1}\neq[q']_{j-1}$. Then the intersection $\tau([q]_{j-1})\cap \tau([q']_{j-1})$ is either an $(n-2)$-cube or a pair of two adjacent $(n-2)$-cubes.
\end{lemma}
\begin{proof}
Let $Q$ and $Q'$ be $n$-cubes for which $[q]_{j-1}=\theta(Q)$ and $[q']_{j-1}=\theta(Q')$. We may assume that $Q$ and $Q'$ are adjacent. We consider the cases by categories. Note that cases are symmetric in $Q$ and $Q'$.

Suppose first that $Q$ is in Category I. Then $Q'$ is either in Category II.1 or in Category II.2B and $Q$ is in Category I.2. In the first case, the intersection $\tau([q]_{j-1})\cap \tau([q']_{j-1})$ is a single $(n-2)$-cube and in the second case, the intersection $\tau([q]_{j-1})\cap \tau([q']_{j-1})$ consists of two adjacent $(n-2)$-cubes. Thus in both cases the intersection is shellable with a space which is an $(n-2)$-cell.

Suppose now that both cubes $Q$ and $Q'$ are in Category II. We are left with two cases. Suppose that $Q$ is in Category II.2A, that is, $Q$ is an edge cube which is also an transition cube. Then $\tau([q]_{j-1})$ is a shellable complex consisting of five $(n-1)$-cubes; two wall cubes and three ceiling cubes on the boundary of $\Theta(Q)$ with a pair of $(n-1)$-cubes, forming an $(n-1)$-cell, meeting cube $Q'$. Regarding cube $Q'$, we have two cases. Either $Q'$ is an inner cube in Category II.1 or a wall cube in Category II.2B. Thus the intersection $\tau([q]_{j-1})\cap \tau([q']_{j-1})$ is either an $(n-2)$-cube or a pair of adjacent $(n-2)$-cubes.

Suppose finally that $Q$ is in Category II.1. Then $Q'$ is either also in Category II.1 or in Category II.2B. Also in this case, in both cases, the intersection $\tau([q]_{j-1})\cap \tau([q']_{j-1})$ is a single $(n-2)$-cube.
\end{proof}

\section{Peeling off layers}

We define
\[
D_j = \bigcup_{i=0}^j \Layer_i(D_E) \quad\text{and}\quad E_j = E\cup \bigcup_{i=j+1}^{j_E} \Layer_i(D_E)
\]
for  $j=0,\ldots, j_E$ and call the sequences
\[
D_E = D_{j_E} \supset \cdots \supset D_0=D_E\cap \partial A_E
\quad \text{and} \quad
E = E_{j_E} \subset \cdots \subset E_0 = A_E
\]
a \emph{peeling of $D_E$} and a \emph{filling of $E$}, respectively.

We call the set 
\[
\partial^+ D_j = D_j \cap E_j.
\]
the \emph{positive boundary of $D_j$}. We have the following properties of $\partial^+D_j$.
\begin{lemma}
For each $j\ge 1$, we have that
\begin{enumerate}
\item $\partial D_j =\partial^+ D_j \cup (D_j\cap \partial A_E)$, 
\item $\partial E_j = \partial^+ D_j \cup (E_j\cap \partial A_E)$,
\item $\partial^+D_j  = (\partial^+D_j \cap D_{j-1}) \cup ( (\partial \Layer_j(D_E))^{\ceiling} \cup (\partial \Layer_j(D_E))^{\wall})$, and
\item $\partial^+D_{j-1} =  (\partial^+ D_j \cap D_{j-1})  \cup (\partial \Layer_j(D_E))^{\floor}$.
\end{enumerate}
\end{lemma}

The discussion on tents $\tau([q]_{j-1})$ gives now the following lemma.

\begin{lemma}[Peeling a layer]
\label{lemma:psi_j}
There exists, for each $j\ge 1$, an $\sL(n)$-bilipschitz map 
\[
\psi_j \colon |\partial^+ D_j|\to |\partial^+ D_{j-1}|
\]
for which 
\begin{enumerate}
\item $\psi_j|_{(\partial^+ D_j) \cap D_{j-1}} = \id$, 
\item for each $(n-1)$-cube $q\in (\partial \Layer_j(D_E))^{\floor}$, $\psi_j(|\tau([q]_{j-1})|) = |[q]_{j-1}|$ and the restriction $\psi_j|_{|\tau([q]_{j-1})|} \colon |\tau([q]_{j-1})| \to |[q]_{j-1}|$ is affine on simplices of $\tau([q]_{j-1})^\Delta$, and
\item for each $q'\in (\partial \Layer_j(D_E))^\ceiling$, which is an inner cube,
the restriction of $\psi_j|_{|q'|}$ is a translation through $|\Layer_j(D_E)|$. 
\end{enumerate}
\end{lemma}
\begin{proof}
It suffices to observe that for each $q\in (\partial \Layer_j(D_E))^{\floor}$, $\tau([q]_{j-1})$ is a cubical $(n-1)$-cell for which $\tau([q]_{j-1})\cap [q]_{j-1}$ is an $(n-2)$-cell. The existence of $\psi_j$ now follows from local considerations in each case.
\end{proof}

Mappings $\psi_j$ in Lemma \ref{lemma:psi_j} may be extended to bilipschitz homeomorphisms $|A_E|\to |A_E|$. We omit the details. 
 
\begin{lemma}[Filling a layer]
\label{lemma:psi_j-extension}
There exists $\sL=\sL(n)\ge 1$ for the following. For $1\leq j \leq j_E$ there exists an $\sL$-bilipschitz homeomorphism 
\[
\Psi_j \colon |A_E|\to |A_E|
\]
for which 
\begin{enumerate}
\item $\Psi_j|_{|\partial^+ D_j|} = \psi_j \colon |\partial^+ D_j|\to |\partial^+ D_{j-1}|$, 
\item $\Psi_j|_{|\partial A_E|} = \id$, 
\item $\Psi_j|_Q = \id$ if $Q\in E_j^{[n]}$ and $Q\cap |\Layer_j(D_E)| = \emptyset$, and
\item for each inner roof cube $\Omega_{E,\ell}$, the restriction $\Psi_j|_{\Omega_{E,\ell}}$ is a translation through $|\Layer_j(D_E)|$.
\end{enumerate}
In particular, $\Psi_j(|D_j|) = \Psi_j(|D_{j+1}|)$ and $\Psi_j(|E_j|) = |E_{j-1}|$.
\end{lemma}
We also set $\Psi_j= \id \colon |A_E|\to |A_E|$ for $j\geq j_E +1$, for convenience.

\section{Spare simple covers}
\label{sec:peeling-spare-cubes}

We recall that our goal is to develop a method to deform the mapping $\widetilde F \colon |E^*|\to \bS^{n-1}\times [0,3^{\nu k}]$, where $E^* = (E\cup \Omega_E)-(\Omega_{E,1}\cup \cdots \cup \Omega_{E,s_E})$ to a mapping $\widehat F \colon |E\cup \Omega_E|\to \bS^{n-1}\times [0,3^{\nu k}]$ and then use extension over undented refined tunnels (Proposition \ref{prop:extension-over-refined-tunnels}) to finish the extension in this case.

The main issue in application of the filling lemma (Lemma \ref{lemma:psi_j-extension}) and deformation methods to obtain $\widehat F$, lies in 
the existence of corner cubes of Category I.1B. 
As pointed out in Remark \ref{rmk:corner-cube-type-2}, for an $(n-1)$-cube $q\in (\partial \Layer_j(D_E))^\floor$ for which $[q]_{j-1}=\theta(Q)$ for an $n$-cube $Q$ in Category I.1B, we have that  $\# ([q]_{j-1})^{[n-1]} > \# (\tau([q]_{j-1]})^{[n-1]}$. Thus it is not possible to deform an Alexander map on the simplicial complex $\Psi_j(\tau([q]_{j-1]})^\Delta)$ to an Alexander map on the simplicial complex $([q]_{j-1})^\Delta$ unless we have additional simple covers which may be used to build the missing $(n-1)$-simplices.

With  this particular issue in mind, we defined connectors of Case 2 in Definition \ref{def:connector}. These connectors are stars, each of which has three $n$-cubes meeting the boundary in faces and a spare cube contained completely in the interior.

To set the stage, let $\sfR$ be a refinement of a connector $\sfJ$ of Case 2 in $D_E$, let $\sfQ$ be a spare cube in $\sfJ$ and $\sfJ^\tr = \sfJ-\sfQ$ be its truncation. Denote $\sfR_{\sfQ}= \sfR|_{|\sfQ|}$ and $\sfR_{\sfJ^\tr}=\sfR|_{|\sfJ^\tr |}$.

Since $\sfR$ is a refined connector, it is isometric to $3^{\nu r}([-1,1]^2 \times [0,1]^{n-2})$ for some $r\in \{0,1\ldots, k\}$, we may make the following isometric identifications and choices:
\begin{enumerate}
\item $|\sfR| = |\sfJ| = 3^{\nu r}([-1,1]^2 \times [0,1]^{n-2})$,
\item $|\sfR_{\sfQ}| = 3^{\nu r} ([0,1]^2 \times [0,1]^{n-2})$, and
\item $|\sfR_{\sfJ^\tr}|=3^{\nu r} ([-1,1]\times [-1,0] \times [0,1]^{n-2}\bigcup  [-1,0]\times [0,1] \times [0,1]^{n-2})$.
\end{enumerate}

With this identification, the walls of $\sfR_{\sfQ}$ correspond to $(n-1)$-cubes 
\begin{align*}
& \left(3^{\nu r}([0,1]^2 \times [0,1]^{n-3}\times \{0\})\, \bigcup \,3^{\nu r}([0,1]^2 \times [0,1]^{n-3} \times \{1\})\right)\\
&\cup \,\left(3^{\nu r}([0,1]\times\{1\} \times [0,1]^{n-2})\, \bigcup \,3^{\nu r}(\{1\}\times [0,1] \times [0,1]^{n-2}) \right)\\
& \equiv V \cup H,
\end{align*}
where $V$ stands for 'vertical' and $H$ stands for 'horizontal'.

The issue at hand is solved by deforming walls of $\sfQ$ to floor of $\sfQ$ in  refinements. Since, in any refinement, the number of $(n-1)$-cubes in the walls of $\sfQ$ is exactly twice as many as that in the floor. The deformation  allows us to add one  $(n-1)$-cube worth of additional simple covers to each $(n-1)$-cube in the floor of $\sfQ$. Since the number of $(n-1)$-cubes in the floor of $\sfJ$ is exactly twice of that in the floor of $\sfQ$. With these added simple covers, the deformation on the floor  of $\sfQ$ may be continued.

The deformation on the walls is an isotopy on the horizontal part $H$, which does not produce any simple covers. Deformation on the vertical walls $V$ can be seen as follows. Figure \ref{fig:peeling-off-corners-first}  presents the first step of the deformation of the barycentric subdivision of the vertical wall. 
Note that, in all dimensions $n\ge 3$, we have that a barycentric subdivision in four $(n-1)$-cubes occupying the space $[-1,1]^2\times [0,1]^{n-3}$ is taken to barycentric subdivision of a complex of three $(n-1)$-cubes with space isometric to $\left( ([-1,0] \times [-1,0]) \cup ([-1,0] \times [0,1]) \cup ([0,1]\times [-1,0])\right)\times [0,1]^{n-3}$. Thus Figure \ref{fig:peeling-off-corners-first} is representative.
For a general peeling argument and the proof of the associated deformation, see the proof of Proposition \ref{prop:extension-over-differences}.

Although the peeling of spare cubes will be done simultaneously with the rest of the peeling of the complex $D_E$, we indicate two of the subsequent steps in  Figures \ref{fig:peeling-off-corners-continuation-I} and \ref{fig:peeling-off-corners-continuation-II}, since these steps exhibits the transportation of the additional simple covers produced by peeling of layers. We call these additional simple covers  \emph{spare simple covers}. 

We remark that, in each step, some of the spare simple covers are kept on the vertical part, but most of them are transported to the deformed horizontal part. 
More precisely, we move spare simple covers in $(n-1)$-cubes which do not meet a corner cube to the horizontal part after each deformation.

As mentioned earlier, the mere fact that the number of $(n-1)$-cubes in the walls of $\sfR_{\sfQ}$ is exactly twice of that in the floor of  $\sfR_{\sfQ}$ allows the simple covers be distributed, at the end of the deformation,  for which each $(n-1)$-cube in the floor of $\sfR_{\sfQ}$ has exactly one $(n-1)$-cube worth of additional simple covers.



\begin{convention}\label{convention:restored}
In the forthcoming proof of Proposition \ref{prop:extension-over-differences}, we assume that all spare cubes $\sfQ\in \cQ_E$ have been peeled and spare simple covers have been created and transported to the floor of $\sfQ$ as described by the process indicated in this section. 

In the following, we say that an expansion of the Alexander map satisfies Condition $\cC^\spare$ (or its variation) provided that spare simple covers are created and transported according to (a variation of) this process and the number of spare simple covers in each $(n-1)$-cube is bounded by a constant $N(n,\mu)$. 
\end{convention}

\begin{figure}[h!]
\centering
\begin{overpic}[scale=0.6,unit=1mm]{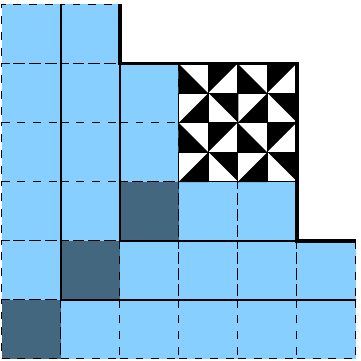} 
\end{overpic}
\hfill
\begin{overpic}[scale=0.6,unit=1mm]{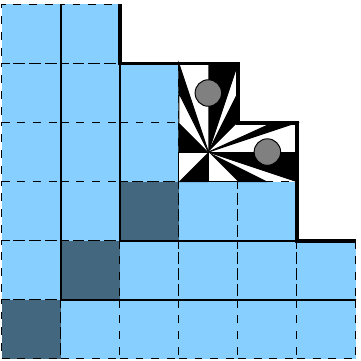} 
\end{overpic}
\hfill
\begin{overpic}[scale=0.6,unit=1mm]{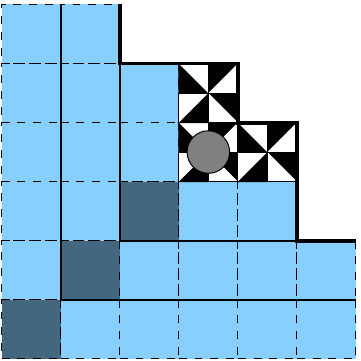} 
\put(20.5,20.5){\tiny {\color{white}$1$}}
\end{overpic}

\caption{Deformation of the Alexander map on the barycentric subdivision of a vertical wall of $\sfR_{\sfQ}$. The disks represent clusters of  well-placed simple covers. The number in the disk indicates the number of clusters of simple covers; each cluster has $\# ([0,1]^{n-1})^\Delta/2$, i.e., the number of simplices in a barycentric subdivision of an $(n-1)$-cube.}
\label{fig:peeling-off-corners-first}
\end{figure}

\begin{figure}[h!]
\centering
\begin{overpic}[scale=0.6,unit=1mm]{Part4-layer-corner-deformation-local-3.pdf} 
\put(20.5,20.5){\tiny {\color{white}$1$}}
\end{overpic}
\hfill
\begin{overpic}[scale=0.6,unit=1mm]{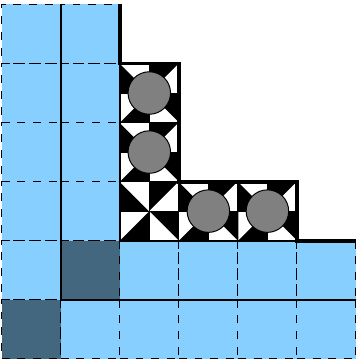} 
\put(14.5,26.5){\tiny {\color{white}$1$}}
\put(14.5,20.5){\tiny {\color{white}$1$}}
\put(20.5,14.5){\tiny {\color{white}$1$}}
\put(26.5,14.5){\tiny {\color{white}$1$}}
\end{overpic}
\hfill
\begin{overpic}[scale=0.6,unit=1mm]{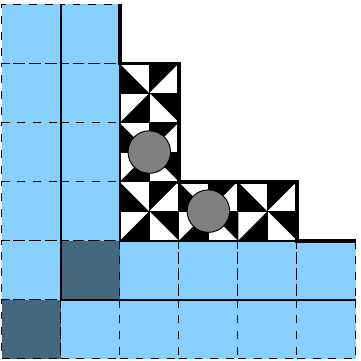} 
\put(14.5,20.5){\tiny {\color{white}$1$}}
\put(20.5,14.5){\tiny {\color{white}$1$}}
\end{overpic}
\caption{Second step in the deformation of the Alexander map on the barycentric subdivision of a wall of $\sfR_{\sfQ}$. On the left: Barycentric subdivision after first step. In the middle: Barycentric subdivision after deformation. On the right: On a vertical wall, after two clusters of the spare simple covers are moved to non-visible horizontal part $H$.}
\label{fig:peeling-off-corners-continuation-I}
\end{figure}

\begin{figure}[h!]
\centering
\begin{overpic}[scale=0.6,unit=1mm]{Part4-layer-corner-deformation-local-4-v2.pdf} 
\put(14.5,20.5){\tiny {\color{white}$1$}}
\put(20.5,14.5){\tiny {\color{white}$1$}}
\end{overpic}
\hfill
\begin{overpic}[scale=0.6,unit=1mm]{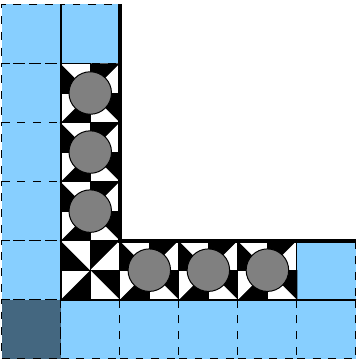} 
\put(8.5,26.5){\tiny {\color{white}$1$}}
\put(8.5,20.5){\tiny {\color{white}$1$}}
\put(8.5,14.5){\tiny {\color{white}$2$}}
\put(14.5,8.5){\tiny {\color{white}$1$}}
\put(20.5,8.5){\tiny {\color{white}$1$}}
\put(26.5,8.5){\tiny {\color{white}$1$}}
\end{overpic}
\hfill
\begin{overpic}[scale=0.6,unit=1mm]{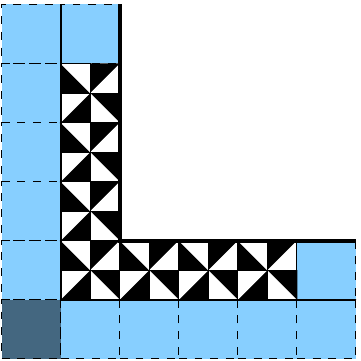} 
\put(8.5,20.5){\tiny {\color{white}$1$}}
\put(8.5,14.5){\tiny {\color{white}$1$}}
\put(14.5,8.5){\tiny {\color{white}$1$}}
\put(20.5,8.5){\tiny {\color{white}$1$}}
\end{overpic}
\caption{Third step in the deformation of the Alexander map on the barycentric subdivision of a wall of $\sfR_{\sfQ}$. At the end of peeling off a spare cube $\sfQ$, all simple covers are moved to the non-visible horizontal part $H$.}
\label{fig:peeling-off-corners-continuation-II}
\end{figure}


\section{Proof of the extension proposition}

Before giving the proof, we  fix, for each $1\leq j\leq j_E$, a bilipschitz homeomorphism $\Psi_j \colon |A_E|\to |A_E|$ as in Lemma \ref{lemma:psi_j-extension}, and also a structure 
\[
P_{j-1} = (\Psi_j)_*((\partial E_j)^\Delta)
\]
on $|\partial E_{j-1}|$.

In the proof, when we say that simple covers are collected in (or, moved into) a ball $B$ in $q$, we mean  that  simple covers are contained in (or, moved into) the center half $(1/2) B$ of a ball $B$, which is contained in the interior of  $|q|$ and of radius $3^{-n-5} \SL_E(q)$, where $\SL_E(q)$ is the side length of $|q|$. When we refer to \emph{balls}, we mean balls which contain simple covers.

\begin{proof}[Proof of Proposition \ref{prop:extension-over-differences}] 

Let $f\colon |\partial E|\to \bS^{n-1}$ be an $\sL$-BLD-controlled expansion of a $(\partial E)^\Delta$-Alexander map and $\widetilde f\colon |\partial E^*|\to \bS^{n-1}\times \R$  an $\sL$-Lipschitz roof adjustment of $f$  in the statement of the proposition. Let also $\mu = \mu(K)$.

Regarding the peeling process in spare cubes, we refer to the previous section on the discussion of creation and transportation of spare simple covers. Note that, in this peeling process, we always peel off a corner cube of type (iii) before encountering a corner cube of type (ii). Thus we always have the necessary spare simple covers for the deformation.

\medskip

\noindent
\emph{Step 1:} In this step, we define inductively  a  sequence $f_j \colon |\partial E_j|\to \bS^{n-1}, j=j_E, \ldots, 1,$ of $\sL'(n,\sL)$-BLD-controlled expansion of $(\partial E_j)^\Delta$-Alexander maps and the roof adjustments 
$\widetilde f_j \colon |\partial E_j^*|\to \bS^{n-1}\times \R$.

Let $f_{j_E} = f$ and $\widetilde f_{j_E} = \widetilde f_{j_E}$. We may assume that all simple covers of $f_{j_E}$ are contained in balls  $B_{j_E,q}\subset \interior |q|$, $q \in (\partial E_{j_E})^{[n-1]}$, by applying an isotopy $|\partial E_{j_E}|\times [0,1] \to |\partial E_{j_E}|\times [0,1]$. We may further assume, after a second isotopy, that all  balls (containing simple covers) in $(\Layer_{j_E}(D_E))^{\wall} \cup (\Layer_{j_E}(D_E))^{\ceiling}$ are in inner cubes of $(\Layer_{j_E}(D_E))^{\ceiling}$.

For the construction of $f_{j_E-1}$, let 
\[
\mathbf{B}= \bigcup \{\tau([q]_{j_E-1})\colon q\in (\Layer_j(D_E))^{\floor}, \, q \text{ not an inner cube}\}.
\]
Recall that tents $\tau([q]_{j_E-1})$ are shellable by Lemma \ref{lemma:tau-shellable} and the collection $\{([q]_{j-1}, \tau([q]_{j-1})) \colon q\in (\Layer_j(D_E))^{\floor}\}$ contains only a bounded number of isomorphism and isometry classes. We may put $\mathbf{B}$ in the place of the band and $\tau([q]_{j_E-1})$ the place of $q\times [0,2]$ in Corollary \ref{cor:Band-merge}.

We deform now $\mathbf B$ back to the original complex in three steps. First, using shellability, we deform in individual tents to obtain stars. Then, using shellability again, we deform on common boundaries of adjacent tents. Finally, in the third step, we combine stars replacing the tents and obtained simple covers to back to the original cubical structure in $|\mathbf B|$. In case of a corner cube $Q$ of type (ii), we use the excess simple covers given by the peeling of the spare cubes to build the cubical structure. After these steps, we obtain a $\sL(n,\mu, \sL)$-BLD homotopy 
\[
H_{j_E-1}^{(1)} \colon |\partial E_{j_E-1}|\times [0,1] \to |\partial E_{j_E-1}|\times [0,1]
\]
from $f_{j_E} \circ (\Psi_{j_E}|_{|\partial E_{j_E}|})^{-1} \colon |\partial E_{j_E-1}|\to \bS^{n-1}$, a BLD-controlled expansion of an $(P_{j_E-1})^\Delta$-Alexander map, to a BLD-controlled expansion of a $(\partial E_{j_E-1})^\Delta$-Alexander map $f_{j_E-1}' \colon |\partial E_{j_E-1}|\to \bS^{n-1}$, whose simple covers are in balls contained in cubes in $(\partial E_{j_E-1})^{[n-1]}$.

After applying another isotopy $H_{j_E-1}^{(2)}$, we may assume that  
\begin{enumerate}
\item all balls in $ (\Layer_{j_E-1}(D_E))^{\wall} \cup (\Layer_{j_E-1}(D_E))^{\ceiling}$ have been moved into the inner cubes of $(\Layer_{j_E-1}(D_E))^{\ceiling}$, and  \label{item:move-to-inner}
\item there are at most $N(n,\mu)$ balls  in each inner cube $q$ of the ceiling. \label{item:move-to-center}
\end{enumerate}
The reason for adding Condition \eqref{item:move-to-center} is to not over-crowd those inner cubes near the wall with balls. The concatenation 
\[
H_{j_E-1}=H_{j_E-1}^{(2)} * H_{j_E-1}^{(1)}\colon |\partial E_{j_E-1}|\times [0,1] \to |\partial E_{j_E-1}|\times [0,1]
\]
is a BLD homotopy from $f_{j_E}$ to $f_{j_E-1}$, an BLD-controlled expansion of an $(\partial E_{j_E-1})^\Delta$-Alexander map whose simple covers have the said properties above.  Let $\widetilde f_{j_E-1}$ be the roof-adjusted modification of $f_{j_E-1}$.

\medskip
 
Suppose that $f_{j_E}, \ldots, f_{j+1}$, as well as $\widetilde f_{j_E}, \ldots, \widetilde f_{j+1}$, have been defined for some $j\geq 1$. Following the construction of $H_{j_E-1}$ almost verbatim, we obtain a homotopy 
\[
H_j \colon |\partial E_j |\times [0,1] \to |\partial E_j |\times [0,1]
\]
from $f_{j+1}$ to $f_j$ for which 
\begin{enumerate}
\item all  balls in $ (\partial \Layer_j(D_E))^{\wall} \cup (\partial \Layer_j(D_E))^{\ceiling}$ have been moved into the inner cubes of $(\partial \Layer_j(D_E))^{\ceiling}$, and  
\item there are at most $N(n,\mu)$ balls in each inner cube  of the ceiling.
\end{enumerate}
The second condition is achieved by moving, for each spectral cube $P$ of $D_E$ (reps. each partial star $S$), the balls  in $C=(\partial \Layer_j(D_E))^{\ceiling} \cap \widehat P$ (reps. balls in $C^*=(\partial \Layer_j(D_E))^{\ceiling} \cap \widehat S$\,) by at most  a bounded graph distance towards the center of $C$ (resp. $C^*$); the upper bound for the  distance and the number $N(n,\mu)$ are mutually dependent.  

Continuing this process inductively,  we obtain a BLD  homotopy 
\[
H_1 \colon |\partial E_1|\times [0,1] \to |\partial E_1|\times [0,1]
\]
from $f_1$ to a BLD-controlled expansion  $f_0$ of a $(\partial A_E)^\Delta$-Alexander map which satisfies  a version of Condition $\cC^{\spare}$ that puts an upper bound, $N(n,\mu)$, on the number of simple covers in each $(n-1)$-cube. Note that complex $A_E$ is undented.

\medskip

Indeed, a version of Condition $\cC^\spare$
can be achieved for the following reason. Since the precise moves are similar to the ones indicated in Figures \ref{fig:peeling-off-corners-continuation-I} and \ref{fig:peeling-off-corners-continuation-II}, we only indicate the different cases. Let $P$ be a spectral cube of $\cS_\ell(D_E)$,  $\widehat P=\Refine^{\nu r- \ell}(P)$, and $J= 3^{(\nu r -\ell)(n-1)}$. Then $\partial \widehat P \cap \partial E$ has at most $(2n-1) J$ many $(n-1)$-cubes, hence the number of balls containing the simple covers of the original mapping $f_E$ is at most $(2n-1) J$. The number of simple covers created by the deformation in the subsequent steps is precisely half of the number of $(n-1)$-simplices in $(\partial \widehat P \cap \partial E)^\Delta$ on its 'wall', which is at most $c(n) J$. Packing $(2n+c(n))J$ balls into  into $J$ many $(n-1)$-cubes in $\widehat P \cap \partial A_E$ for which 
a version of Condition $\cC^\spare$
holds can easily be done, provided that we diligently move the balls towards the center of $\widehat P$ in each step. 

Suppose next that $\widehat S=\Refine^{\nu r- \ell}(S)$ is a partial-star-indentation. 
Thus a version of Condition $\cC^\spare$
may be achieved on $\widehat S \cap \partial A_E$ for a similar reason, only this time constants depend also on $\mu$.

\medskip

\noindent \emph{Step 2:} In this step, we pull back  homotopies $H_j$ onto the boundary $|\partial A_E|$ of the undented $A_E$ by conjugation with suitable homeomorphisms, and  combine them into a single homotopy.

Let $g\colon |E|\to |A_E|$ be a bilipschitz indentation flattening map as in Proposition \ref{prop:flattening-bent-indentation} for which $g$ is the identity in the complement of wedge $\Wedge(D_E)$ of $|D_E|$. We may further assume that the restriction $g|_{\Omega_{E,i}}$  on each inner roof cube $\Omega_{E,i}$  is a scaling. Set $g_{j_E}=g$. 
 
 A closer inspection of the proof of Proposition \ref{prop:flattening-bent-indentation} shows that there exist $\sL(n)$-bilipschitz homeomorphisms
\[
g_j \colon |E_j|\to |A_E|, \quad j= j_E-1,\ldots, 1,
\]
for which
 \begin{enumerate}
 \item $g_j$ is the  identity in the complement of wedge $\Wedge(D_j)$ of $D_j$,
 \item  $g_j \circ \psi_{j+1} = g_{j+1}|_{|\partial^+ D_{j+1}|}$, 
 \item  $g_j|_Q = g_{j-1}|_Q$ for each $Q\in E_j^{[n]}$ satisfying $Q\cap |\partial \Layer_j(D_E)|=\emptyset$, and 
 \item  $g_j|_{\Omega_{E,i}}$ is a scaling of each inner roof cube $\Omega_{E,i}$. 
\end{enumerate}
The  conjugation
\[
F_j = (g_j \times \id) \circ H_j \circ (g_j^{-1}\times \id) \colon |\partial A_E|\times [0,1] \to |\partial A_E|\times [0,1]
\]
 is a homotopy from 
\[
f_{j+1} \circ \Psi_{j+1}^{-1} \circ g_j^{-1}|_{|\partial A_E|} = f_{j+1}\circ (g_j \circ \Psi_{j+1})^{-1}|_{|\partial A_E|} = f_{j+1} \circ g_{j+1}^{-1}|_{|\partial A_E|}
\]
to $f_j \circ g_j^{-1}$. 

Thus the mapping
\[
F \colon |\partial A_E|\times [0,j_E] \to |\partial A_E|\times [0,j_E]
\]
defined by 
\[
F(x,t) = F_{j_E-j} (x,t-(j-1))
\]
for $(x,t)\in |\partial A_E|\times [j-1,j]$ and $1\leq j \leq j_E$,  is a homotopy from $f_{j_E}\circ g_{j_E}^{-1}=f\circ g^{-1}$ to a BLD-controlled expansion  $f_0 \colon |\partial A_E|\to \bS^{n-1}$ of a $(\partial A_E)^\Delta$-Alexander map.

\medskip
\noindent \emph{Step 3:} Let ${A_E}^* = (A_E \cup \Omega) - (g(\Omega_{E,1}) \bigcup \cdots \bigcup g(\Omega_{E,s_E}))$ be the roof adjustment of $A_E$ in which dent $D_E$  has been restored. 

Let ${\widetilde f}\,^{^{*}}  \colon |\partial {A_E}^*|\to \bS^{n-1}\times \R$ be the roof-admissible map defined  by ${\widetilde f}\,^{^{*}}|_{|\partial {A_E}^*|\cap |\partial A_E|} = \widetilde f \circ g^{-1}|_{|\partial {A_E}^*|\cap |\partial A_E|}$ and ${\widetilde f}\,^{^{*}} = \widetilde f$ elsewhere.

Fix now a roof adjustment  $\widetilde f_0 \colon |\partial {A_E}^*|\to \bS^{n-1}\times \R$ of $f_0$ dominated by ${\widetilde f}\,^{^{*}}$, $\widetilde f_0 \ll {\widetilde f}\,^{^{*}}$, as in Definition \ref{def:roof-admissible-domination}. 
By Lemma \ref{lemma:roof-admissible-homotopy}, $F$ gives rise to a $\sL(n, \nu, \sL)$-BLD homotopy $H \colon |\partial {A_E}^*|\times [0,j_E] \to \bS^{n-1}\times [0,j_E]$ from ${\widetilde f}\,^{^{*}}$ to  $\widetilde f_0$. Extending $H$ trivially over $[j_E, 3^{\nu r}]$, 
we get now a homotopy 
\[
H\colon |\partial {A_E}^*| \times [0, 3^{\nu r} ] \to \bS^{n-1}\times [0, 3^{\nu r}]
\] 
from ${\widetilde f}\,^{^{*}}$ to  $\widetilde f_0$.

Since ${A_E}^*$ is undented, by Proposition \ref{prop:extension-over-refined-tunnels}, there exists a $\sL'(n,\nu, \mu,\lambda, \sL)$-BLD extension $G \colon |{A_E}^*|\to \bS^{n-1}\times \R$ of $\widetilde f_0$ modulo $\omega^+(A_E; \Omega)$. Therefore, by Proposition \ref{prop:X_E-extension}, ${\widetilde f}\,^{^{*}}$ has a BLD extension $ F' \colon | {A_E}^*|\to \bS^{n-1}\times \R$ modulo $\omega^+(E;\Omega)$. 

Since the homeomorphism $g\colon |E| \to |A_E|$ is identity outside  wedge $\Wedge(D_E)$, it induces a natural homeomorphism $g^*\colon |E^*| \to |{A_E}^*|$ for which $g^* =g$ on $|E^*| \cap |E|$ and $g^*|_{\Omega}$ is the identity.

The mapping $\widetilde F =  F' \circ g^* \colon |E^*|\to \bS^{n-1}\times \R$ is a $\sL(n,\nu,\mu, \lambda,\sL)$-BLD extension of $\widetilde f$ modulo $\omega^+(E;\Omega)$. This concludes the proof.
\end{proof}



\chapter{Proof of the Quasiregular extension theorem}
\label{chap:final-extension}

We now return to the realization $\Real_k(\Sigma)$ and its  partition 
\[
\mathcal{Real}_k(\Sigma) = \{ \widetilde \Rec_k(\Sigma) \} \cup \cdiff_1(\Sigma) \cup \cdots \cup \cdiff_{k-1}(\Sigma) \cup \cT_k(\Sigma).
\]
For convenience, write $\cdiff_k(\Sigma)=\cT_k(\Sigma)$.

By the construction of $Z_k$, lengths of the tunnels in $\cT_k(\Sigma), k\geq 2,$ are at most $\lambda_\loc$ - a number fixed in Proposition \ref{prop:Localized-realization-structure}; thus $\cT_k(\Sigma) \subset \cT(\lambda_\loc)$. For $k=0$ and $k=1$, $\cT_0(\Sigma)$ and $\cT_1(\Sigma)$  are contained in $\cT(\lambda_0)$ for a number $\lambda_0$ depending on the complex $K_0$. 
We assume, as we may, that 
\[
\lambda_0 \ge \lambda_\loc.
\]
Cases $k=0,1$ can be easily deduced as special cases; the control of distortion begins with $k=2$.

\begin{convention}
From here on, we fix an integer $k\geq 2$ and a flat metric on the cubical complex  $\Real_k(\Sigma)$  in which each $n$-cube has side length $1$. 
\end{convention}

Under this convention, $\cT_k(\Sigma)\subset \cD_0(\lambda_0)$.
Each $E\in \cdiff_\ell(\Sigma)$ has a representation $E = A_E- D_E$, where  $A_E= \Refine^{\nu (k-\ell)}(T_E) \in \cA_{k-\ell}(\lambda_0)$, $T_E\in \cT(\lambda_0)$ is a tunnel whose cubes $Q$ have side length $\SL_{\Real_k(\Sigma)}(Q)=3^{\nu(k-\ell)}$, and  $D_E \subset A_E$ is a bent indentation. Hence  
\[
\cdiff_\ell(\Sigma)\subset \cD_{k-\ell}(\lambda_0).
\]

Each $E\in \mathcal{Real}_k(\Sigma)\setminus \{ \widetilde \Rec_*(\Sigma)\}$ has a unique roof complex determined by the hierarchy defined in Notation \ref{notation:alpha}. Indeed, let $E\in \cdiff_\ell(\Sigma)$ and  $\alpha(E)$ be the unique $n$-cube in $\Real_k(\Sigma)$ for which $\ell(\alpha(E))=\ell -1$. Then $\alpha(E)\cap E = \alpha(E)\cap A_E = \Refine^{\nu(k-\ell)}(q_E)$ for an $(n-1)$-cube $q_E$  in $\partial T_E$. Let $Q_E$ be the unique $n$-cube in $\alpha(E)$ having $q_E$ as a face. Then $\Omega(E) = \Refine^{\nu (k-\ell)}(Q_E)$ is the roof complex of $E$ satisfying $\omega(E;\Omega(E))=\Refine^{\nu(k-\ell)}(q_E)$. 
For simplicity, we denote $\omega(E) = \omega(E;\Omega(E))$, $\omega^+(E) = \omega^+(E;\Omega(E))$, and $\sigma(E) = \sigma(E;\Omega(E))$ the roof face, roof top, and roof cylinder of $E$, respectively. 

Note that  $\Omega(E)$ is an inner roof cube of $\alpha(E)$ and, by the construction,  the inner roof cubes of $\alpha(E)$ are mutually disjoint. 
We set 
\[
E^* = (E \cup \Omega(E)) - \bigcup_{\alpha(E')=E} \Omega(E')
\]
for each  $E\in \mathcal{Real}_k(\Sigma)\setminus \{\widetilde \Rec_*(\Sigma)\}$.
Complexes $E^*$ yield a roof-adjusted essential partition 
\[
\mathcal{Real}^*_k(\Sigma) = \{ (\widetilde \Rec_*(\Sigma))^* \} \cup \{ E^* \colon E\in \mathcal{Real}_k(\Sigma)\setminus \{\widetilde \Rec_*(\Sigma)\} \}
\]
of $\Real_k(\Sigma)$.

Finally, we define \emph{roof-adjustment of $\Upsilon_k(\Sigma)$} by 
\[
\Upsilon_k^*(\Sigma) = \left( \bigcup_{E^*\in \mathcal{Real}^*_k(\Sigma)} \partial E^* \right) - \Refine^{\nu k}(\Sigma).
\]
Note that
\[ 
Y^*_k(\Sigma) = \Upsilon_k(\Sigma) \cup \bigcup_{E\in \mathcal{Real}_k(\Sigma)} (\sigma(E) \cup \omega^+(E) ).
\]

We are now ready to construct a BLD extension $|\Real_k(\Sigma)|\to \bS^{n-1}\times \R$ of a given quasiregular controlled expansion $|\Upsilon_k(\Sigma)|\to \bS^{n-1}$. We recall first the statement of the theorem. 

\Realextensionproblem*

\begin{proof}
Let $\widetilde f_{k,\Sigma} \colon |\Upsilon_k(Z)|\to \bS^{n-1}\times \{0\}$ be an $\sL$-BLD-controlled expansion of a $(\Upsilon_k(\Sigma))^\Delta$-Alexander map. To simplify the discussion, we may assume  that 
\[
\sL\geq \sL^\Diamond.
\]
The extension of $\widetilde f_{k,\Sigma}$ is constructed in two steps: first over tunnels and differences and then, separately, over $(\widetilde \Rec_*(\Sigma))^*$. The extension over tunnels and differences is constructed recursively following the reverse ordering of the level $\ell(E)$. 

\medskip
\noindent \emph{Step 1.1: Extension over tunnels.}
Let $T\in \mathcal{Real}_k(\Sigma)$ be a tunnel, that is, $\ell(T)=k$, and set $\widetilde f_T = \widetilde f_{k,\Sigma}|_{|\partial T \cap \Upsilon_k(\Sigma)|}$. Mapping $\widetilde f_T$ has a natural extension over the $(n-1)$-cube $\omega(T)$ as a $\omega(T)^\Delta$-Alexander map; this extension is denoted again as $\widetilde f_T$. 
We extend $\widetilde f_T$ as a product over $\Omega(T)$ , thus $\widetilde f_T|_{|\omega^+(T)|} \colon |\omega^+(T)|\to \bS^{n-1} \times \{1\}$ is an $\omega^+(T)^\Delta$-Alexander map. Since $\sL\geq \sL^\Diamond $, the  restriction ${\widetilde f}^*_T= \widetilde f_T|_{\partial T^*}$ of the extended $\widetilde f_T$  is an  $\sL$-BLD roof-admissible map on $\partial T^*$.

Proposition \ref{prop:extension-over-tunnels}  yields an $\sL''$-BLD, $\sL''=\sL''(n,\nu,\lambda_0;\sL)$, extension
\[
F_T \colon |T^*|\to \bS^{n-1}\times [0,1]
\]
of ${\widetilde f}^*_T$, modulo $\omega^+(T)$, for which 
$F_T|_{|\omega^+(T)|}\colon  |\omega^+(T)| \to  \bS^{n-1} \times \{1\}$ is an $\sL$-BLD-controlled expansion of an $\omega^+(T)^\Delta$-Alexander map.

Since roof-adjusted tunnels $T^*$ in $\mathcal{Real}^*_k(\Sigma)$ are mutually disjoint,  extensions $F_T$ over  $T^*$'s may be constructed simultaneously and independently for all $T$. Assume  that $\widetilde f_{k,\Sigma}$ has been extended over all $T^*$, modulo roof tops, and denote this extension by $F_k$.

\medskip
\noindent \emph{Step 1.2: Extension over differences.}Let
\[
V_\ell =  \cdiff_\ell (\Sigma) \,\, \cup \cdots \cup \,\, \cdiff_k(\Sigma), \quad \ell=1,\ldots, k,
\]
where $\cdiff_k(\Sigma)=\cT_k(\Sigma)$.

Assume that $\ell \in \{1,\ldots, k-1\}$ and that mapping $\widetilde f_{k,\Sigma}$ has been extended to an $\sL''$-BLD map
\[ 
F_{\ell+1} \colon V_{\ell+1}  \to \bS^{n-1}\times [0,3^{\nu(k-\ell-1)}],
\]
such that, for each $E'\in \cdiff_{\ell+1}(\Sigma)$, the restriction $F_{\ell+1}|_{|\partial (E')^*|} $ is an $\sL$-BLD roof-admissible map.

Let $E\in \cdiff_\ell(\Sigma)$    
and   $\widetilde f_E = \widetilde f_{k,\Sigma}|_{|\Upsilon_k(\Sigma) \cap E|}$. 
We first extend $\widetilde f_E|_{|\partial E\cap \partial E^*|}$ to a roof-admissible map ${\widetilde f}^*_E$ on $\partial E^*$ that is compatible with the already defined  $F_{\ell+1}$ as follows. For each $E'$ satisfying $\alpha(E') = E$, we extend $\widetilde f_E|_{|\partial E\cap \partial E^*|}$ over $|\partial E^*\cap (E')^*|$ as $\widetilde F_{\ell+1}|_{|\partial E^*\cap (E')^*|}$, and  extend $\widetilde f_E|_{|\partial E\cap \partial E^*|}$ over $|\partial E^*\cap \Omega(E)|$ as in the case of tunnels. The extension ${\widetilde f}^*_E$ is $\sL$-BLD roof-admissible on $\partial E^*$. Following  the proof of  Proposition \ref{prop:extension-over-differences}, we may construct an $\sL'''$-BLD ($\sL'''=\sL'''(n,\nu, \lambda_0;\sL)$)  extension 
\[F_{\ell,E}\colon |E^*|\to \bS^{n-1}\times [3^{\nu(k-\ell-1)},3^{\nu(k-\ell)}]\] 
of ${\widetilde f}^*_E$ modulo $\omega^+(E)$, 
 for which the restriction $F_{\ell,E}|_{\omega^+(E)}$ is an $\sL$-BLD expansion of an $\omega^+(E)^\Delta$-Alexander map.

Since the extension may be constructed simultaneously for  all differences of level $\ell$, we obtain an $\sL'''$-BLD extension $\widetilde F_\ell \colon V_\ell \to \bS^{n-1}\times [0, 3^{\nu(k-\ell)}]$ from the induction assumption.
  
Applying this process recursively for $\ell=k-1,\ldots, 1$, we obtain  an $\sL'''$-BLD  extension 
 \[
F_1\colon \Real_k(\Sigma) - (\widetilde \Rec_*(\Sigma))^* \to \bS^{n-1}\times [0, 3^{\nu(k-1)}]
 \]
of $\widetilde f_{k,\Sigma}$ apart from the space of the receded complex $(\widetilde \Rec_*(\Sigma))^*$.

\medskip

\noindent \emph{Step 2: Extension over $(\widetilde \Rec_*(\Sigma))^*$.} 
In the final part of extension, recall that  complex $K$ and the initial separating complex $Z_0=\mathcal Z$  are the ones constructed in   Theorem \ref{thm:separating-complex-existence}  and  fixed in Remark \ref{rmk:separating-complex-special}. Recall also  from  Remark \ref{rmk:preference-cube-special} about properties of the preference function and from Remark \ref{rmk:separating-complex-special-opening} about the choices of the openings in the channeling construction of $Z_1$ from $Z_0$.
For simplicity, we denote
\[
R^*=(\widetilde \Rec_*(\Sigma))^*.
\]
The construction of $\mathcal Z$ in  Theorem \ref{thm:separating-complex-existence}  reveals that $R^*$ has two possible forms.

 In the first case, $\Sigma$ is the distinguished boundary component $\Sigma'$ in the construction. In this case, the receded complex has the structure
\[
R^* = \Refine^{\nu k}( C_\Sigma \cup T_\Sigma)-D_\Sigma = \Refine^{\nu k}(C_\Sigma) \cup (\Refine^{\nu k}(T_\Sigma) - D_\Sigma),
\]
where $C_\Sigma$ is the collar of $\Sigma$ in $K$ isomorphic to $\partial \Sigma \times [0,3]$, $T_\Sigma\subset K$ is a tunnel whose length is  bounded by the number of $n$-cubes in $K$, and $D_\Sigma\subset \Refine^{\nu k}(T_\Sigma)$ is a bent indentation. We take $E=\Refine^{\nu k}(T_\Sigma)-D_\Sigma$ and treat $E$ as a difference. Let $\Omega_\Sigma = \Omega(E) \subset \Refine^{\nu k}(C_\Sigma)$ be the roof complex for $E$, and set $E^* = E\cup \Omega(E)$. We  fix an extension of $\widetilde f_{k,\Sigma}$ over $|E^*|$ as in Step 1.2, and continue to denote the extended map by $F_1$. It remains to extend  $\widetilde f_{k,\Sigma}$ over $\Refine^{\nu k}(C_\Sigma) - \Omega_\Sigma$.

In the second case, $\Sigma\neq \Sigma'$, the receded complex 
\[
R^* = \Refine^{\nu k}(C_\Sigma) - \Omega_\Sigma,
\]
where  $C_\Sigma$ is the collar of $\Sigma$ in $K$ isometric to $\Sigma \times [0,3]$, and $\Omega_\Sigma$ is the roof complex of the unique difference $E$ in $\mathcal{Real}^*_k(\Sigma)$, which meets $R^*$. Since  $\widetilde f_{k,\Sigma}$ has already been extended to the map $F_1$ over $|E^*|$, thus also already over $|\Omega_\Sigma|$.  It remains to extend $\widetilde f_{k,\Sigma}$  over $\Refine^{\nu k}(C_\Sigma) - \Omega_\Sigma$.

Given a boundary component $\Sigma$, take $j=0$ in the first case and $j=1$ in  the second case, respectively. 
Let $Q_\Sigma$ be the cube in $\Refine^{\nu j}(C_\Sigma)$ for which $|\Omega_\Sigma|=|Q_\Sigma|$, thus
\[
R^* = \Refine^{\nu k}(C_\Sigma) - \Refine^{\nu (k-j)}(Q_\Sigma).
\]
We now  extend  $\widetilde f_{k,\Sigma}^*$ over every complex $\Refine^{\nu (k-j)}(Q)$, where $Q \in \Refine^{\nu j}(C_\Sigma)\setminus \{Q_\Sigma\}$ is an $n$-cube having a face $\omega_Q$ in $\Refine^{\nu j}(\Sigma_-)$ and $\Sigma_- = \partial C_\Sigma - \Sigma$. 
The extension of $\widetilde f_{k,\Sigma}^*$ over each $\Refine^{\nu (k-j)}(Q)$ is the natural product extension $\Refine^{\nu j}(Q) \to \bS^{n-1}\times [0,3^{\nu j}]$, $(x,t) \mapsto (\widetilde f_{k,\Sigma}(x), t)$, given by the identification $\omega_Q \times [0,3^{\nu j}] = \Refine^{\nu j}(Q)$ with $\omega_Q \times \{0\} = \omega_Q$.

This construction  yields an $\sL''''$-BLD mapping 
\[
F_{P_\Sigma} \colon |P_\Sigma|\to \bS^{n-1}\times \R
\]
on the collar $P_\Sigma = \bigcup_{Q} \Refine^{\nu(k-j)}(Q)$ of $\Refine^{\nu k}(\Sigma_-)$ in $\Refine^{\nu k}(C_\Sigma)$ for a number $\sL''''$ depending only on  $n$ and $\sL'''$, for which the restriction on $\Sigma_+ = \partial P_\Sigma - \Refine^{\nu k}(\Sigma_-)$, $F_{P_\Sigma}|_{\Sigma_+} \colon |\Sigma_+|\to \bS^{n-1}\times \{ 3^{\nu j}\}$, is an $\sL$-BLD-controlled expansion of an $(\Sigma_+)^\Delta$-Alexander map.

Since $\Refine^{\nu k}(C_\Sigma)-P_\Sigma$ is a product, to finalize the proof,  we extend  $F_{P_\Sigma}|_{\Sigma_+}$ over $\Refine^{\nu k}(C_\Sigma)-P_\Sigma$ by taking the product extension. The extended map $F_{P_\Sigma}$ is  $\sL'''''$-BLD for a number $\sL'''''$ depending only on $n$ and $\sL''''$, and its restriction to the boundary $\Sigma$ is an $\sL$-BLD-controlled expansion of a $\Refine^{\nu k}(\Sigma)^\Delta$-Alexander map.

Together mappings $F_1$ and $F_{P_\Sigma}$ yield the $\sL'$-BLD extension 
\[\widetilde F_{k,\Sigma}\colon |\Real_k(\Sigma)|\to \bS^{n-1}\times [0,3^{\nu k}]\] of $\widetilde f_{k,\Sigma}$, with $\sL'=\sL'(n, K; \sL)$,  claimed in the theorem.
\end{proof}

\part{Weaving}
\label{part:Alexander-Rickman}

\chapter{Weaving theorem}

In this part, we provide the last step in the proof of the Quasiregular cobordism theorem, i.e.,  combining  quasiregular maps $\widetilde F_{k,\Sigma} \colon |\Real_k(\Sigma)|\to \bS^{n-1}\times [0,3^{\nu k}]$ constructed in  the Quasiregular extension theorem (Theorem \ref{thm:Real-extension-problem})
into a map $|K| \to \bS^n\setminus \interior (B_1\cup \cdots \cup B_p)$. 

Essential  to this combination is a method of weaving. 
This method stems from Rickman's sheet construction  \cite[Section 7]{Rickman_Acta} and also, in another version, from  \cite[Section 7]{Drasin-Pankka}.

\section{Statement of Weaving theorem}
For the statement of the Weaving theorem, we fix some terminologies related to partitions of the domain and the target of the mappings.

\subsection{Partition of the domain}

Let $K$ be a good cubical complex with $m$ boundary components, and  $\cC(K)$ be the collection of the boundary components of $K$. 
Let $\nu \in \N$ be the refinement scale associated to $K$ fixed in Definition \ref{def:nu}.

For a closed set $X\subset \interior |K|$, let  $\cU_K(X)$ be the family of connected components of $|K|\setminus X$, and $\overline{\cU}_K(X)$ be the family of closures (in $|K|$) of the elements in $\cU_K(X)$.
Since $X$ is contained in the interior of $|K|$, functions 
\[\beta_X \colon \cC(K)\to \cU_K(X)\quad \text{and} \quad \overline{\beta}_X \colon \cC(K) \to \overline{\cU}_K(X),\]
 given by formulas $|\Sigma|\subset \beta_X(\Sigma) \subset \overline{\beta_X}(\Sigma)$ for $\Sigma \in \cC(K)$, are well-defined. \index{$\beta_X(\Sigma)$}

Suppose that  $Z$ is a separating complex of $\Refine^{\nu k}(K)$. Then functions $\beta_{|Z|}$ and $\overline{\beta}_{|Z|}$ are bijections.
Retaining the notations  in Chapter \ref{chap:Separating-complexes}, we get
\[
\beta_{|Z|}(\Sigma) = |\Comp_{\Refine^{\nu k}(K)}(Z;\Sigma)| \setminus |Z| \quad\text{and}\quad
\overline{\beta}_{|Z|}(\Sigma) = |\Comp_{\Refine^{\nu k}(K)}(Z;\Sigma)|.
\]
We denote also for $\Sigma\in \cC(K)$, $\pi_\Sigma \colon \Real_{\Refine^{\nu k}(K)}(Z;\Sigma) \to \Comp_{\Refine^{\nu k}(K)}(Z;\Sigma)$ the canonical projection, and  \index{$\Upsilon(Z;\Sigma)$} 
\[
 \Upsilon_{\Refine^{\nu k}(K)}(Z;\Sigma) = \pi_\Sigma^{-1}(Z \cap \Comp_{\Refine^{\nu k}(K)}(Z;\Sigma)),
\]
the inner boundary component of $ \Real_{\Refine^{\nu k}(K)}(Z;\Sigma)$.

\begin{definition}
\label{def:weaved-approximation}
\index{weaved approximation}
A closed set $X\subset \interior |K|$ is a \emph{weaved approximation of a separating complex $Z \subset \Refine^{\nu k}(K)$} provided that
\begin{enumerate}
\item  $\beta_X \colon \cC(K)\to \cU_K(X)$ is a bijection, and
\item for each $\Sigma\in \cC(K)$, there exists a homeomorphism 
\[\psi_\Sigma \colon |\Real_{\Refine^{\nu k}(K)}(Z;\Sigma)|\setminus |\Upsilon_{\Refine^{\nu k}(K)}(Z;\Sigma)| \to \beta_X(\Sigma),\]
 which is bilipschitz with respect to the inner metrics of the spaces.
\end{enumerate}
\end{definition}

Observe that  $\overline{\beta}_X(\Sigma)  = \beta_X(\Sigma)  \cup (X \cap \overline{\beta}_X(\Sigma))$. Hence  a homeomorphism $\psi_\Sigma \colon |\Real_{\Refine^{\nu k}(K)}(Z;\Sigma)|\setminus |\Upsilon_{\Refine^{\nu k}(K)}(Z;\Sigma)| \to \beta_X(\Sigma)$, which is bilipschitz with respect to the inner metrics, extends continuously as a Lipschitz map 
\[\overline{\psi}_\Sigma \colon |\Real_{\Refine^{\nu k}(K)}(Z;\Sigma)|\to \overline{\beta}_X(\Sigma).\] 
We call $\overline{\psi}_\Sigma$ a \emph{weaving map}. Note that $\overline{\psi}_\Sigma$ need not be injective.

Although not formally defined, we also may view $\overline{\beta}_X(\Sigma)$ as a weaved approximation of $\overline{\beta}_{|Z|}(\Sigma) = |\Comp_{\Refine^{\nu k}(K)}(Z;\Sigma)|$, and $\overline{\psi}_\Sigma$ as a weaved approximation of $\pi_\Sigma$.

\subsection{Partition of the target}

We discuss now a partition of the target of the map, that is, $\bS^n$. 

\begin{definition}
Let $S\subset \bS^n$ be an $(n-1)$-sphere. We say 
$\cS$ is  a \emph{$\Sigma^2(\Delta^\square_{n-1})$-structure of $S$} if there exists a homeomorphism $h \colon \bS^{n-1}\to S$ for which $\cS = h_*(\Sigma^2(\Delta^\square_{n-1}))$. 
We all the pair  $(S,\cS)$ a \emph{structured $(n-1)$-sphere}.
\end{definition}

Two structured $(n-1)$-spheres $(S,\cS)$ and $(S',\cS')$ in $\bS^n$ are said to be \emph{compatible} if either $S\cap S'$ is an $(n-1)$-cell or an $(n-2)$-sphere, and $\cS|_{S\cap S'} = \cS'_{S\cap S'}$.

\begin{definition}
\index{pillow partition}   
\label{def:pillow-partition}
An essential partition $\{D_1,\ldots, D_p\}$ of $\bS^n$ into bilipschitz $n$-cells is a \emph{pillow partition of $\bS^n$}, if spheres $\partial D_i$ have pairwise compatible $\Sigma^2(\Delta^\square_{n-1})$-structures.
\end{definition}

The Weaving theorem asserts the existence of a weaved approximation and weaving maps.

\begin{restatable}[Weaving theorem]{theorem}{Quasiregularcombinationtheorem}\index{Weaving theorem}
\label{thm:combination-I}
Let $K$ be a good cubical $n$-complex with $m$ boundary components, where $n\geq 3$ and  $m\ge 2$. Let also $k\in \N$, $Z\subset \Refine^{\nu k}(K)$ a separating complex, obtained by a localized channeling transformation of a separating complex in $\Refine^{\nu (k-1)}(K)$. Then, given $p\in \{2,\ldots, m\}$ and a surjective function $c \colon \cC(K) \to \{1,\ldots, p\}$ on the collection $\cC(K)$ of boundary components of $K$, there exist
\begin{enumerate}
\item a weaved approximation $X\subset |K|$ of $Z$, 
\item a pillow partition $\{D_1,\ldots, D_p\}$ of $\bS^n$,
\item a map $\varphi \colon X\to \bigcup_{j=1}^p \partial D_j$, and
\item for each $\Sigma\in \cC(K)$, a weaving map 
\[
\overline{\psi}_\Sigma \colon |\Real_{\Refine^{\nu k}(K)}(Z;\Sigma)| \to \overline{\beta}_X(\Sigma),
\]
whose restriction 
\[
\varphi \circ \overline{\psi}_\Sigma|_{|\Upsilon(Z;\Sigma)|} \colon |\Upsilon(Z;\Sigma)|\to \partial D_{c(\Sigma)}
\]
on $\Upsilon(Z;\Sigma)$, the inner boundary component of $\Real_{\Refine^{\nu k}(K)}(Z;\Sigma)$, is a BLD-controlled expansion of an $\Upsilon(Z;\Sigma)^\Delta$-Alexander map. 
\end{enumerate}
The statement in quantitative in the sense that all constants depend only on $n$ and $K$.
\end{restatable}

The proof of this theorem spans the rest of this part. In Chapter \ref{chap:pillows} we introduce notions pillows and pillow maps, which serve as preliminaries for  $X$ and $\varphi$. The weaved approximation $X$ and the map $\varphi$, called Alexander--Rickman map, are discussed in Chapter \ref{chap:Alexander--Rickman}.

\section{Proof of the Quasiregular cobordism theorem}

We first prove a version of Theorem \ref{intro-thm:qrcobordism} for cubical complexes.

\begin{theorem}[Cubical quasiregular cobordism theorem] \index{Cubical quasiregular cobordism theorem}
\label{intro-thm:qrcobordismII}  
Let $n\geq 2$,  $m\ge p\geq 2 $, $K$ be a good cubical $n$-complex with $m$ boundary components, and $\cC(K)$  the collection of boundary components of $K$.  Let $B_1,\ldots, B_p$ be pairwise disjoint closed Euclidean balls in $\bS^n$, and
$c \colon \cC(K) \to \{1,\ldots, p\}$ a surjection. Then there are constants $\sK=\sK(n, K, B_1,\ldots, B_p)$ and $\sL=\sL(n,K,B_1,\ldots,B_p)$ for the following.

For each $k\geq 2$, 
there exists a $\sK$-quasiregular map 
\[
f \colon |K|\to \bS^n\setminus \interior (B_1 \cup \cdots \cup B_p)
\]
for which  $f(|\Sigma|) = \partial B_{c(\Sigma)}$ for $\Sigma\in \cC(K)$.
Moreover, there exists a $\Sigma^2(\Delta^\square_{n-1})$-structure on each $\partial B_{c(\Sigma)}$ for which the restriction $f|_{|\Sigma|}\colon |\Sigma|\to \partial B_{c(\Sigma)}$ is an $\sL$-BLD-controlled expansion of an $(\Refine^{\nu k}(\Sigma))^\Delta$-Alexander map. 
\end{theorem}

\begin{proof}Fix an integer $k\geq 2$, and also a pillow partition $\{D_1,\ldots, D_p \}$ of $\bS^n$ for which $D_j$'s are pairwise isometric and $\bigcup_{j=1}^p \partial D_j$ is a branched sphere. Let $Z$ be a separating complex of $\Refine^{\nu k}(K)$.

Let now  $X\subset \interior |K|$ be the weaved approximation,   $\varphi \colon X\to \bigcup^p_{i=1} \partial D_j$  the map, and $\overline{\psi}_\Sigma \colon |\Real_{\Refine^{\nu k}}(Z;\Sigma)|\to \overline{\beta}_X(\Sigma), \Sigma\in \cC(K)$ the weaving maps,  in the statement of Theorem \ref{thm:combination-I}.

For each $D_j$, we fix a point $x_j\in \interior D_j$, a (common) radius $r_0=r_0(k)>0$, and a quasiconformal map $\theta_j \colon D_j \setminus \{x_j\} \to \bS^{n-1}\times [0,\infty)$ for which $\theta_j(\partial D_j) = \bS^{n-1}\times \{0\}$, $\theta_j^{-1}(\bS^{n-1} \times [3^{\nu k},\infty)) = B^n(x_j,r_0)\subset \interior D_j$ is a ball, and the restriction $\theta_j^{-1}|_{\bS^{n-1}\times [3^{\nu k},\infty)}$ is a conformal map $(x,t) \mapsto x_j + r_0 3^{-t}x$. We may choose $r_0(k)$ and points $x_j$  so that mappings $\theta_j$ are quasiconformal with a constant depending only on $n$ and $p$. Fix now a quasiconformal mapping 
\[g \colon \bS^n \, \setminus \interior (\bigcup_{j=1}^p B(x_j, r_0)) \to \bS^n \setminus \interior (B_1\cup \cdots \cup B_p),\]
 which maps  $\partial B(x_j,r_0)$ to $\partial B_j$ with distortion depending only on  $n, B_1, \ldots, B_p$.

For each $\Sigma \in \cC(K)$, mapping  
\[
f_\Sigma = \theta_\Sigma \circ \varphi \circ \overline{\psi}_\Sigma|_{\Upsilon_{\Refine^{\nu k}(K)}(Z;\Sigma)} \colon \Upsilon_{\Refine^{\nu k}(K)}(Z;\Sigma) \to \bS^{n-1}\times \{0\}
\]
is a BLD-controlled expansion of an Alexander map. By the Quasiregular extension theorem (Theorem \ref{thm:Real-extension-problem}), $f_\Sigma$ admits a BLD extension 
\[F_\Sigma \colon |\Real_{\Refine^{\nu k}(K)}(Z;\Sigma)|\to \bS^{n-1}\times [0,3^{\nu k}]\]
 for which $F_\Sigma|_{|\Sigma|} \colon |\Sigma|\to \bS^{n-1}\times \{3^{\nu k}\}$ is a BLD-controlled expansion of a $\Refine^{\nu k}(\Sigma)^\Delta$-Alexander map.

We define now $\widetilde F_\Sigma \colon \overline{\beta}_X(\Sigma) \to D_{c(\Sigma)}\setminus \interior B^n(x_j,r_0)$ to be the unique map satisfying
\[
\widetilde F_\Sigma \circ \overline \psi_\Sigma = {\theta_{c(\Sigma)}}^{-1} \circ F_\Sigma.
\]
Thus $\widetilde F_\Sigma$ is a quasiregular extension of $\varphi$ into $\beta_X(\Sigma)$. In other words, $\widetilde F_\Sigma|_{\beta_X(\Sigma)} \colon \beta_K(X)\to \interior (D_{c(\Sigma)}\setminus B^n(x_j,r_0))$ is $\sK$-quasiregular with  $\sK$ depending only on $n$ and $K$, and  the restriction $\widetilde F_\Sigma|_{\overline{\beta}_X(\Sigma) \cap X} = \varphi_{\overline{\beta}_X(\Sigma) \cap X}$.

Finally we define 
\[
f\colon |K|\to \bS^n\setminus (B_1\cup \cdots \cup B_p)
\]
by the formula
\[
f|_{\overline{\beta}_X(\Sigma)} = g\circ \widetilde F_\Sigma
\]
for $\Sigma \in \cC(K)$. Since $\{\bar \beta_X(\Sigma)\colon \Sigma\in \cC(K)\}$ is an essential partition of $|K|$,  mapping $f$ is quasiregular on $|K|$ with a constant depending only on the $n, K$, and $B_1,\ldots, B_p$. Furthermore, for each $\Sigma \in \cC(K)$, $f|_{|\Sigma|}\colon |\Sigma|\to \partial B_{c(\Sigma)}$ is a BLD-controlled expansion of a $\Refine^{\nu k}(\Sigma)^\Delta$-Alexander map 
\end{proof}

As the reader may have observed, the Quasiregular cobordism theorem (Theorem \ref{intro-thm:qrcobordism}) is an immediate consequence of Theorem \ref{intro-thm:qrcobordismII}. \index{Quasiregular cobordism theorem}

\introthmqrcobordism*

\begin{proof} 
Let $g$ be the Riemannian metric on $M$. By Proposition \ref{prop:Riemannian-to-cubical}, $M$ admits a cubical structure $K$ and a flat metric $d_K$ for which $(|K|, d_K)$ is quasi-similar to $(M, g)$.  Let $d_0\in \N$ and let $c\colon \cC(K)\to \{1,\ldots, p\}$ be a subjection.

We fix an evolution sequence $(Z_k)$ of separating complexes as in Section \ref{sec:evolution}, and an index $k\in \N$ for which $\# (\Refine^{\nu k}(\Sigma)^{[n-1]} \ge d_0$ for each $\Sigma\in \cC(K)$. 

By Theorem \ref{intro-thm:qrcobordismII}, there exist a constant $\sK= \sK(n,K, N)=\sK(n,M,N)$,
and a $\sK$-quasiregular map $f\colon |K|\to \bS^n\setminus (B_1\cup \cdots \cup B_p)$ for which $f|_{|\Sigma|} \colon |\Sigma|\to \partial B_{c(\Sigma)}$ is an Alexander expanded by simple covers. Since $\#\Refine^{\nu k}(\Sigma)^{[n-1]} \ge d_0$,  mappings $f|_{|\Sigma|}$ have degree at least $d_0$. Thus $f$ also has degree at least $d_0$. This completes the proof.
\end{proof}

\chapter{Pillow expansion of complexes}
\label{chap:pillows}

To discuss the proof of Weaving theorem \ref{thm:combination-I}, we define  branched spheres, pillow cells, pillow spheres, and pillow maps. After that we discuss pillow expanded complexes and  pillow maps.

\medskip

\noindent{\bf{Notation.}} In this chapter, $\sigma$ is an $(n-1)$-simplex isometric to an $(n-1)$-simplex in a barycentric triangulation of a cubical $n$-complex. 

\section{Branched spheres}\label{sec:branched-sphere}

Let 
\[
\Sigma^\infty(\sigma) = (\sigma \times \Z_+)\Big/{\sim}
\]
be the quotient space, where $\sim$ is the minimal equivalence relation for which $(x,j) \sim (x,0)$ for $x\in \partial \sigma$ and $j\in \Z_+$. We call $\Sigma^\infty(\Sigma)$ the \emph{infinite branched $(n-1)$-sphere}. Let also 
\[\Pi \colon (\sigma \times \Z_+) \to \Sigma^\infty(\sigma)\]
 be the quotient map $(x,j) \mapsto [(x,j)]$ and let $\widehat \Pi \colon \Sigma^\infty(\sigma) \to \sigma$ be the projection $[(x,j)] \mapsto x$.

For each $\rho\geq 2$, we call the subspace
\[
\Sigma^\rho(\sigma) = \Pi(\sigma \times \{1,\ldots, \rho\}) \subset \Sigma^\infty(\sigma)
\]
the \emph{$\rho$-branched $(n-1)$-sphere} and $\rho$ the \emph{rank of $\Sigma^\rho(\sigma)$}.\index{branched sphere}
Note that $\Sigma^1(\sigma) \approx \sigma$ and $\Sigma^2(\sigma) \approx \bS^2$.  For $\rho \geq 3$, the space $\Sigma^\rho(\sigma)$ is not a manifold. In what follows, we identify $\Sigma^1(\sigma) = \sigma$. 

As a space, $\Sigma^2(\sigma)$ is isometric with $\Sigma^2(\Delta^\square_{n-1})$. As a complex, $\Sigma^\rho(\sigma)$ carries a non-simplicial  structure with $(n-1)$-simplices $\Pi(\sigma\times \{j\})$ for $j\in \{1,\ldots, \rho\}$. This structure on $\Sigma^2(\sigma)$ is isomorphic to  $\Sigma^2(\Delta^\square_{n-1})$.

For each $i\in \Z_+$, we denote
\[
\Sigma^\infty(\sigma)_i = \Pi(\sigma\times \{i,i+1\})
\]
and note that we may identify $\Sigma^\infty(\sigma)_i$ with $\Sigma^2(\sigma)$. Indeed, for each $i\in \Z_+$, we have an embedding 
\[
\iota_i \colon \Sigma^2(\sigma) \to \Sigma^\rho(\sigma)
\]
for which the diagram
\[
\xymatrix{
\sigma \times \{1,2\} \ar[rr]^{(x,j) \mapsto (x,i+j-1)} \ar[d]_\Pi & & \sigma\times \{i,i+1\} \ar[d]^\Pi \\
\Sigma^2(\sigma) \ar[rr]^{\iota_i} & & \Sigma^\infty(\sigma)
}
\]
commutes.

We call $(n-1)$-cells $\Pi(\sigma\times \{i\})$ \emph{sheets of $\Sigma^\infty(\sigma)$}. Similarly, we call  $\Pi(\partial \sigma \times \{i\}) = \Pi(\partial \sigma \times \{1\})$ the \emph{rim of $\Sigma^\infty(\sigma)$}. Note that the rim carries a well-defined structure of the boundary of an $(n-1)$-simplex.

\section{Pillows and pillow maps}

\subsection{Pillow cells and pillow spheres}
\label{sec:pillow}

A  pillow $P(\sigma)$ is merely a metric $n$-cell whose boundary has the structure of the sphere $\Sigma^2(\sigma)$. To give a pillow well-defined a metric structure, we give it a triangulation as follows.

Suppose  $\sigma$ is an $(n-1)$-simplex in a barycentric triangulation, let $\tau=[v_0,\ldots, v_n]$ be an $n$-simplex in the said triangulation having $\sigma$ as its face. We denote by   
 \[
\mathrm{proj}_\sigma\colon \tau \to \sigma
\]
the affine projection which maps the vertex $v_n$ to the barycenter of $\sigma$.

We glue two copies of $\tau$ together along the part of boundary $\partial \tau$ outside  $\interior\;\sigma$ by setting \index{$P(\sigma)$}
\[
P(\sigma) = (\tau\times \{-1,1\}) \Big/{\sim},
\] 
where $\sim$ is the minimal equivalence relation for which $(x,-1)\sim (x,1)$ for $x\in (\partial \tau)\setminus \interior \sigma$; here the interior is taken in the relative topology of $\partial \tau$.  
The quotient $P(\sigma)$ carries a structure consisting of  two $n$-simplices $\tau_{-1}= \tau\times \{-1\}$ and $\tau_1=\tau \times \{1\}$ for which   $\tau_{-1}\cap \tau_1$ is the union of $n$ numbers of $(n-1)$-simplices. 

The space $P(\sigma)$ is an $n$-cell whose boundary $\partial P(\sigma)$ carries an $\Sigma^2(\Delta_{n-1}^\square)$-structure with $(n-1)$-simplices $\partial_+ P(\sigma) = \tau_{+1} \cap \partial P(\sigma)$ and $\partial_{-} P(\sigma) = \tau_{-1} \cap \partial P(\sigma)$.

To give $P(\sigma)$ a metric structure, we give both $\tau_{-1}$ and $\tau_1$ the metric structure induced by identification with $\tau$ and give $P(\sigma)$ the inner metric induced by these identifications. This makes $P(\sigma)$ a metric $n$-cell, bilipschitz equivalent to $\bar B^n$. This metric structure identifies $\partial P(\sigma)$ isometrically with $\Sigma^2(\sigma)$. We call the metric cell $P(\sigma)$ a \emph{pillow associated to $\sigma$}. 
\index{pillow} 
\index{$P(\sigma)$}

We glue now pillows along their boundaries to branched spheres. Recall that we have identified $\partial P(\sigma)$ with $\Sigma^2(\sigma)$. We define now 
\[
\iota^\infty \colon \partial P(\sigma)\times \Z_+ \to \Sigma^\infty(\sigma)
\]
to be the mapping which identifies $\partial P(\sigma)\times \{i\}$ with $\Sigma^\infty(\sigma)_i$ by formula $(x,i) \mapsto \iota_i(x)$.
The space
\[
P^\infty(\sigma) = \left( P(\sigma)\times \Z_+ \right)\Big/{\sim_{\iota^\infty}}
\]
where $\sim_{\iota^\infty}$ is the minimal equivalence relation induced by the map $\iota^\infty$, is called the \emph{infinite pillow space}. Let $\pi_{P^\infty(\sigma)} \colon P(\sigma)\times \Z_+ \to P^\infty(\sigma)$ be the canonical projection. We denote $P^\infty(\sigma)_i = \pi_{P^\infty(\sigma)}(P(\sigma)\times \{i\})$ for each $i\in \Z_+$. Note that $P^\infty(\sigma)_i$ is merely a copy of $P(\sigma)$ and we have a natural projection 
\[\Theta \colon P^\infty(\sigma) \to \sigma, [(x,i)] \mapsto \mathrm{proj}_\sigma(x).\]

For each $\rho\in \Z_+$, we set $\PC^\rho(\sigma) \subset P^\infty(\sigma)$ to be the subspace
\[
\PC^\rho(\sigma) = \pi_{P^\infty(\sigma)}(P(\sigma)\times \{1,\ldots, \rho\}).
\]
Topologically $\PC^\rho(\sigma)$ is an $n$-cell and $\PC^\rho(\sigma) = P^\infty(\sigma)_1 \cup \cdots \cup P^\infty(\sigma)_\rho$. We call $\PC^\rho(\sigma)$ a \emph{pillow cell of rank $\rho$}.

Pillow sphere $\PS^\rho(\sigma)$ is obtained from $\PC^{\rho+1}(\sigma)$ by identifying the pillows $P^\infty(\sigma)_1$ and $P^\infty(\sigma)_{\rho+1}$; see Figure \ref{fig:intro-bs}. 
\begin{figure}[htp]
\begin{overpic}[scale=0.12,unit=1mm]{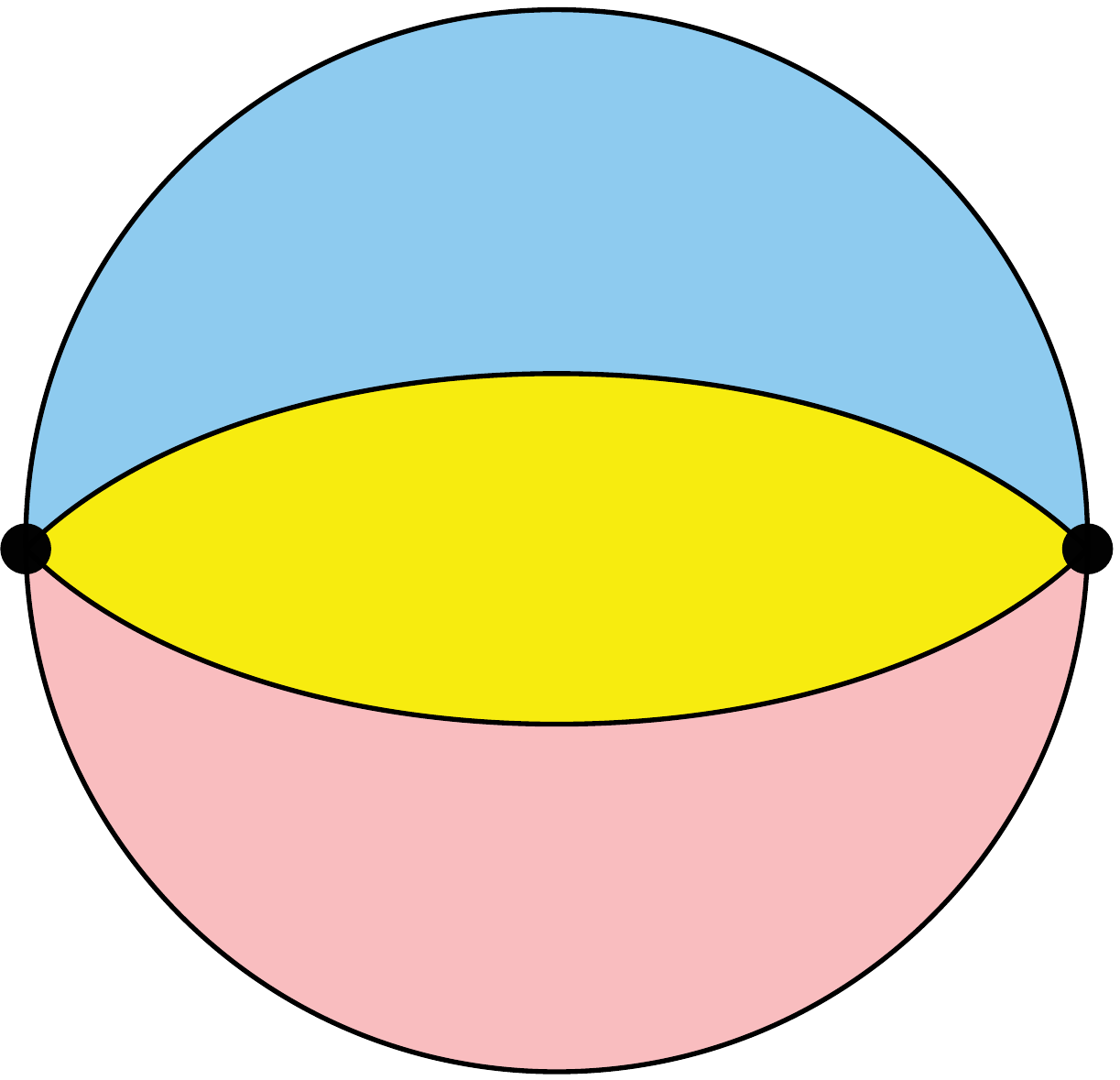} 
\put(22,21){\tiny$P_1$}
\put(11,3){\tiny$P_2$}
\put(11,11){\tiny$P_3$}
\put(11,19){\tiny$P_4$}
\end{overpic}
\caption{$\bS^2$ as a pillow-sphere of rank $4$, containing a codimension one embedded branched sphere of rank $4$.}
\label{fig:intro-bs}
\end{figure}
More precisely, let $\sim_{\PS}$ be the minimal equivalence relation for which $\pi_{P^\infty(\sigma)}(x,1) \sim_{\PS} \pi_{P^\infty(\sigma)}(x,\rho+1)$ for $x\in P(\sigma)$. We define  \index{pillow sphere}
\[
\PS^\rho(\sigma) = \PC^{\rho+1}(\sigma)\Big/{\sim_{\PS}}
\]
and let $\pi_{\PS} \colon \PC^{\rho+1}(\sigma) \to \PS^\rho(\sigma)$ be the canonical projection. 

We make some observations on $\pi_{\PS}$. Since, for each $i=1,\ldots, \rho$, the restriction $\pi_{\PS}|_{P^\infty(\sigma)_i} \colon P^\infty(\sigma)_i \to \PS^\rho(\sigma)$ is an embedding, we identify the image $\pi_{\PS}(P^\infty(\sigma)_i)$ with $P^\infty(\sigma)_i$ and call these $n$-cells \emph{ pillows of $\PS^\rho(\sigma)$}. Note also  that $\pi_{\PS} \colon \Sigma^\rho(\sigma) \to \PS^\rho(\sigma)$ is an embedding onto the set $\bigcup_{i=1}^\rho \partial P^\infty(\sigma)_i$. Following the convention for pillows, we again denote this image  by $\Sigma^\rho(\sigma)$.

We also note that the projection $\Theta \colon \PC^{\rho+1}(\sigma) \to \sigma$ factors through $\PS^\rho(\sigma)$ as $\Theta = \Theta^\rho \circ \pi_{\PS}$, where $\Theta^\rho \colon \PS^\rho(\sigma) \to \sigma$.

The role of the pillow sphere is to induce a structured pillow partition on $\bS^n$ for Theorem \ref{thm:combination-I}.  We record this as a remark.

\begin{remark}
\label{rmk:pillow-sphere-partition}
A pillow sphere $\PS^\rho(\sigma)$ induces a pillow partition of $\bS^n$ in the sense of Definition \ref{def:pillow-partition}. Indeed, we may fix a $\sL(n, \rho)$-bilipschitz homeomorphism $\varphi^\rho \colon \bS^n \to \PS^\rho(\sigma)$. Since $\partial \pi_{\PS}(P^\infty(\sigma)_i)$ has an $\Sigma^2(\Delta^\square_{n-1})$-structure, the collection $\{ D_1,\ldots, D_\rho\}$, where $D_i = (\varphi^\rho)^{-1}(\pi_{\PS}(P^\infty(\sigma)_i))$, is a pillow partition on $\bS^n$. 
\end{remark}

\subsection{Pillow expansions}
\label{sec:pillow-expansions-pillow-maps}

We move from pillows associated to a single simplex to the notion of pillows associated to a subcomplex. 

Let $Y$ be a cubical $(n-1)$-complex and let $\varrho\colon Y^\Delta \to \Z_+$ be a \emph{rank} function on $Y^\Delta$.

We call 
\[
\Sigma^\varrho(Y^\Delta) = \bigcup_{\sigma\in Y^\Delta} \Sigma^{\varrho(\sigma)}(\sigma)
\]
the \emph{$\varrho$-sheet over $Y$} 
\index{$\Sigma^\varrho(Y^\Delta)$}\index{$\varrho$-sheet} 
and
\[
\PC^\varrho(Y^\Delta) = \bigcup_{\sigma \in Y^\Delta} \PC^{\varrho(\sigma)}(\sigma).
\]
the \emph{$\varrho$-pillow over $Y^\Delta$}.  
\index{$P^\varrho(Y^\Delta)$}  
\index{$\varrho$-pillow} 
Note that
\[
\PC^{\varrho(\sigma)}(\sigma)\cap \PC^{\varrho(\sigma')}(\sigma') =\Sigma^{\varrho(\sigma)}(\sigma)\cap \Sigma^{\varrho(\sigma')}(\sigma') = \sigma\cap \sigma'
\]
for all $\sigma,\sigma'\in Y^{\Delta}$. Moreover, $\PC^\varrho(Y^\Delta)\cap |Y|=\Sigma^\varrho(Y^\Delta)\cap |Y|$ is the space of the $(n-2)$-skeleton of $Y$.

We denote  
\[
\Theta^\varrho \colon \PC^\varrho(Y^\Delta) \to Y,
\] 
the map for which $\Theta^\varrho|_{\PC^{\varrho(\sigma)}(\sigma)}$ is the pillow projection $\Theta^{\varrho(\sigma)}\colon \PC^{\varrho(\sigma)}(\sigma) \to \sigma$, for each $\sigma\in (Y^\Delta)^{[n-1]}$. 

\subsection{Pillow maps}
We now define a class of mappings $\PC^\rho(Y^\Delta) \to \PS^p(\Delta^\square_{n-1})$, given  $p\ge 2$, a cubical $(n-1)$-complex $Y$, and a rank function $\varrho \colon (Y^\Delta)^{[n-1]} \to \Z_+$. Since space $P^{\varrho(\sigma)}(\sigma)$ is an $n$-cell, we may assume, for each $\sigma\in (Y^\Delta)^{[n-1]}$, that the open $n$-cell $\interior P^{\varrho(\sigma)}(\sigma) \subset P^\varrho(Y^\Delta)$ has one of the two possible orientations.  We take the standard orientation on the $n$-sphere $\PS^p(\Delta^\square_{n-1})$.

\begin{definition}
\label{def:pillow-map}
\index{pillow map}
Let $p\ge 2$ and $\varphi \colon (Y^\Delta)^{[n-1]}\to \Z_+$ be a rank function on $Y^\Delta$. A mapping $F\colon \PC^\varrho(Y^\Delta) \to \PS^p(\Delta^\square_{n-1})$ is a \emph{pillow map} if
\begin{enumerate}
\item for each $\sigma\in (Y^\Delta)^{[n-1]}$ and each $i\in \{1,\ldots, \varrho(\sigma)\}$, the restriction $F|_{P^\infty(\sigma)_i} \colon P^\infty(\sigma)_i \to \PS^p(\Delta^\square_{n-1})$ is a simplicial embedding onto a pillow of $\PS^p(\Delta^\square_{n-1})$, \label{item:pillows-1}
\item for each $\sigma\in (Y^\Delta)^{[n-1]}$ and each $i\in \{1,\ldots, \varrho(\sigma)\}$, the restriction $F|_{\partial P^\infty(\sigma)_i} \colon \partial P^\infty(\sigma)_i \to \partial F(P^\infty(\sigma)_i)$ is an isomorphism of $\Sigma^2(\Delta^\Delta_{n-1})$-structures, \label{item:pillows-2}
\item for each $\sigma\in (Y^\Delta)^{[n-1]}$ and each $i\in \{1,\ldots, \varrho(\sigma)-1\}$, $F(P^\infty(\sigma)_i)\ne F(P^\infty(\sigma)_{i+1})$, and  also $F(P^\infty(\sigma)_{\varrho(\sigma)})\ne F(P^\infty(\sigma)_1)$.
 \label{item:orientation}
\end{enumerate} 
\end{definition}

By conditions \eqref{item:pillows-1} and \eqref{item:orientation}, a pillow map is a discrete, open, and oriented  map. By \eqref{item:pillows-2}, the restriction 
\[
F|_{\Sigma^\varrho(Y^\Delta)} \colon \Sigma^\varrho(Y^\Delta) \to \Sigma^p(\Delta^\square_{n-1})
\]
is well-defined.

\section{Pillow expansion}
\label{sec:pillow-expansion-complex}

Let $U$ be a good cubical $n$-complex and $Y$ an $(n-1)$-subcomplex of $U$ for which the graph $\Gamma(U;Y)$ has $r$ connected components $G_1,\ldots, G_r$.

For $i\in\{1,\ldots, r\}$, let $U_i=\Span_U(G_i)$ be the subcomplex of $U$ spanned by the vertices of graph $G_i$  and $\Real(U;Y)_i$  be the realization of $G_i$. Denote by \index{realization $\Real(U;Y)$}
\[
\Real(U;Y)=\Real(U;Y)_1\bigsqcup \Real(U;Y)_2 \bigsqcup \cdots \bigsqcup \Real(U;Y)_r
\]
the disjoint union of the realizations, and by  
\[
\pi_{(U;Y)} \colon \Real(U;Y) \to U
\]
the map for which 
\[
\pi_i=\pi_{(U;Y)}|_{\Real(U;Y)_i} \colon \Real(U;Y)_i \to U_i
\]
are the canonical quotient maps. Here, we assume that spaces of $\Real(U;Y)_i$ are oriented so that  mappings $\pi_i$ are orientation preserving. 

Let 
\[
\Upsilon_Y=\Upsilon_1\bigsqcup \Upsilon_2 \bigsqcup \cdots \bigsqcup \Upsilon_r
\]
be the preimage of $Y$ in $\Real(U;Y)$, where, for each $i$,  
\[
\Upsilon_i= \pi_i^{-1}(U_i\cap Y)
\]
is  the \emph{inner boundary component of $\Real(U;Y)_i $}. 
\index{$\Upsilon_i$} 
If there is no ambiguity, notations $U_i$ and $\Real(U;Y)_i$ of complexes are sometime used to denote the underlying spaces.

We fix now a rank function $\varrho \colon (Y^\Delta)^{[n-1]}\to \Z_+$ and denote $\PC^\varrho(Y^\Delta)$ the $\varrho$-pillow over $Y^\Delta$ as before. We glue components $\Real(U;Y)_i, i=1,\ldots, r,$ together along $\PC^\varrho(Y^\Delta)$ as follows. For this we make the following observation.

For each  $\sigma \in (Y^\Delta)^{[n-1]}$, the preimage $\pi_{(U;Y)}^{-1}(\sigma)$  consists of exactly two copies, $\sigma_1$ and $\sigma_2$, of $\sigma$ in different components of $\Real(U;Y)$. Also in each  $\Sigma^{\varrho(\sigma)}(\sigma)$, there are exactly two $(n-1)$-simplices, namely $\sigma^{\varrho(\sigma)}_\BOT = \Pi^{\varrho(\sigma)}(\sigma\times \{1\})$ and $\sigma^{\varrho(\sigma)}_\TOP = \Pi^{\varrho(\sigma)}(\sigma\times \{\varrho(\sigma)+1\})$, each of which meets only one pillow in $\PC^{\varrho}(Y^\Delta)$. 

Let 
\[
G_{(U;Y)} \colon |\Upsilon_Y| \to \partial \PC^\varrho(Y^\Delta)\subset \Sigma^\varrho(Y^\Delta)
\]
be the map constructed as follows:
\begin{enumerate}
\item for each $\sigma \in (Y^\Delta)^{[n-1]}$, switching the labels of $\sigma_1$ and $\sigma_2$ if needed, mappings $G_{(U;Y)}|_{\sigma_1}$ and $G_{(U;Y)}|_{\sigma_2}$ are orientation preserving,
\item  images $G_{(U;Y)}(\sigma_1) = \sigma^{\varrho(\sigma)}_\BOT$ and  $G_{(U;Y)}(\sigma_2) = \sigma^{\varrho(\sigma)}_\TOP$, and
\item $\Theta^\varrho \circ G_{(U;Y)} = \pi_{(U;Y)}|_{\Upsilon_Y}$.
\end{enumerate}

In particular, the diagram 
\[
\xymatrix{
|\Upsilon_Y| \ar[r]^{G_{(U;Y)}} \ar[d]_{\pi_{(U;Y)}|_{\Upsilon_Y}} & \Sigma^\varrho(Y^\Delta) \ar[d]^{\Theta^\varrho} \\
|Y| \ar[r]^{\id} & |Y| 
}
\]
commutes.

\begin{definition}\label{def:varrho-expansion}
The quotient space
\[
\Exp^\varrho(U;Y) = \Real(U;Y) {\bigsqcup}_{G_{(U;Y)}} \PC^\varrho(Y^\Delta)
\]
is called the \emph{$\varrho$-expansion of $U$ over $Y$}.  \index{$\Exp^\varrho(U;Y) $} 
\end{definition}

\begin{remark}
It is understood that  $\Sigma^\varrho(Y^\Delta)$ is an  $(n-1)$-dimensional structure embedded in  $\Exp^\varrho(U;Y)$.
This structure has an important role in the forthcoming discussion. 
\end{remark}

The closures of the complementary components of $ \Sigma^\varrho(Y^\Delta)$ in $\Exp^\varrho(U;Y)$ are  
\[
\{\Real(U;Y)_i \colon 1\leq i \leq r\}\,  \cup \, \{P(\sigma)_j \colon \sigma \in (Y^\Delta)^{[n-1]}, \,1\leq j\leq \varrho(\sigma)\}. 
\]
Let 
\[
\Pi^\varrho_{G_{(U;Y)}} \colon  \Real(U;Y) \bigsqcup \PC^\varrho(Y^\Delta) \to \Real(U;Y) \, {\bigsqcup}_{G_{(U;Y)}} \PC^\varrho(Y^\Delta)
\]
 be the canonical quotient map.

\begin{remark} 
We view $\Real(U;Y)_i, i=1,\ldots, r,$ and $\PC^\varrho(Y^\Delta)$ as subspaces of $\Exp^\varrho(U;Y)$ whenever it is convenient. Since  $|\Real(U;Y)| \setminus |\Upsilon_Y|$ and $|U|\setminus |Y|$ are homeomorphic, we may also consider $|U|\setminus |\Star_U(Y)|$ as a subspace of $|\Real(U;Y)| \subset \Exp^\varrho(U;Y)$.
\end{remark}

\subsection{Metric on $\Exp$}\label{sec:metric-Exp}

On the pillow $P(\Delta_{n-1})$, there is a canonical length metric $d_{P(\Delta_{n-1})}$ associated to the Euclidean metric on the $n$-simplices $(\Delta_n)_+$ and $(\Delta_n)_-$ isomorphic to the two $n$-simplices in $P(\Delta_{n-1})$. 
The pillow cell $\PC^\rho(\Delta_{n-1})$, resp. the pillow-sphere  $\PS^\rho(\Delta_{n-1})$,  may be equipped with the length metric $d_{\PC^\rho(\Delta_{n-1})}$, resp. $d_{\PS^\rho(\Delta_{n-1})}$, induced by the metrics on the individual pillows $P(\Delta_{n-1})_i$.

Associated to each $(n-1)$-simplex $\sigma$ in a barycentric triangulation of an $n$-cube, 
there is a non-degenerate simplicial affine map $\eta_\sigma \colon P(\sigma) \to P(\Delta_{n-1})$ (unique modulo  the permutations of the vertices of $\Delta_{n-1}$).
Metrics $d_\sigma$, $d_{P(\sigma)}$, $d_{\PC^\rho(\sigma)}$, and $d_{\PS^\rho(\sigma)}$ on  simplex $\sigma$, pillow $P(\sigma)$, pillow cell $\PC^\rho(\sigma)$, and pillow-sphere $\PS^\rho(\sigma)$, respectively, may be defined by the pull-backs by $\eta_\sigma$ of the Euclidean metric on $\Delta_{n-1}$,  and the metrics on $d_{P(\Delta_{n-1})}$, $d_{\PC^\rho(\Delta_{n-1})}$, and $d_{\PS^\rho(\Delta_{n-1})}$, respectively.

Thus, in the case of  $Y$ being an $(n-1)$-subcomplex of a good cubical $n$-complex, there is a canonical length metric on the pillow expansion $P^\varrho(Y^\Delta)$ induced by the metrics on the individual pillows.

The space $\Exp^\varrho(U;Y)$ has a length metric induced by the individual metrics on $\Real(U;Y)_i$, $ i=1,\ldots, r$ and $\PC^\varrho(Y^\Delta)$. By unraveling the identification, we observe that this metric is determined by the flat metric on $U$. 
In fact, there exists a constant $L'=L'(n,\varrho)\ge 1$,  depending only on the dimension $n$ and the rank function $\varrho$, so that each $n$-simplex in $\Real(U;Y)^\Delta \bigcup \PC^\varrho(Y^\Delta)$ is $L'$-bilipschitz equivalent to any $n$-simplex in $U^\Delta$.
 A fortiori, we have the following lemma; we omit the straightforward proof. 
 
\begin{lemma}
\label{lemma:metric-Exp} 
There exist a constant $L=L(n,\varrho)\ge 1$ and an $L$-bilipschitz homeomorphism $\phi \colon \Exp^\varrho(U;Y) \to |U|$ for which $h(\PC^\varrho(Y^\Delta)) \subset |\Star_U(Y)|$, and $h^{-1}|_{|U|\setminus |\Star_U(Y)|} = \id$.
\end{lemma}

\section{Pillow maps on expanded complexes}
\label{sec:pillow-expansion-Alexander-maps}

Having the notion $\Exp^\varrho(U;Y)$ at our disposal, we state the main result of this chapter.

\begin{restatable}{theorem}{Pillowadjustment}
\label{thm:pillow-adjustment}
Let $U$ be an orientable good $n$-complex and $Y$ an $(n-1)$-subcomplex of $U$ for which the graph $\Gamma(U;Y)$ has $r$ connected components and 
\[
\Real(U;Y)=\Real(U;Y)_1\bigsqcup \Real(U;Y)_2 \bigsqcup \cdots \bigsqcup \Real(U;Y)_r.
\]
Let also $2 \le p \le r$ and $c \colon \{1,\ldots, r\}\to \{1,\ldots, p\}$ a surjection. Then there exist a rank function $\varrho \colon (Y^\Delta)^{[n-1]}\to \{p,\ldots, 2p\}$ and an expanded pillow map 
\[\cP \colon \Exp^\varrho(U;Y) \to \PS^p(\Delta^\square_{n-1})\]
 for which 
\begin{enumerate}
\item $\cP|_{\PC^\varrho(Y^\Delta)} \colon \PC^\varrho(Y^\Delta) \to \PS^p(\Delta^\square_{n-1})$ is a pillow map and 
\item $\cP(\Real(U;Y)_i) = P^\infty(\Delta^\square_{n-1})_{c(i)}$.
\end{enumerate}
\end{restatable}

\begin{figure}[h!]
\begin{overpic}[scale=0.25,unit=1mm]{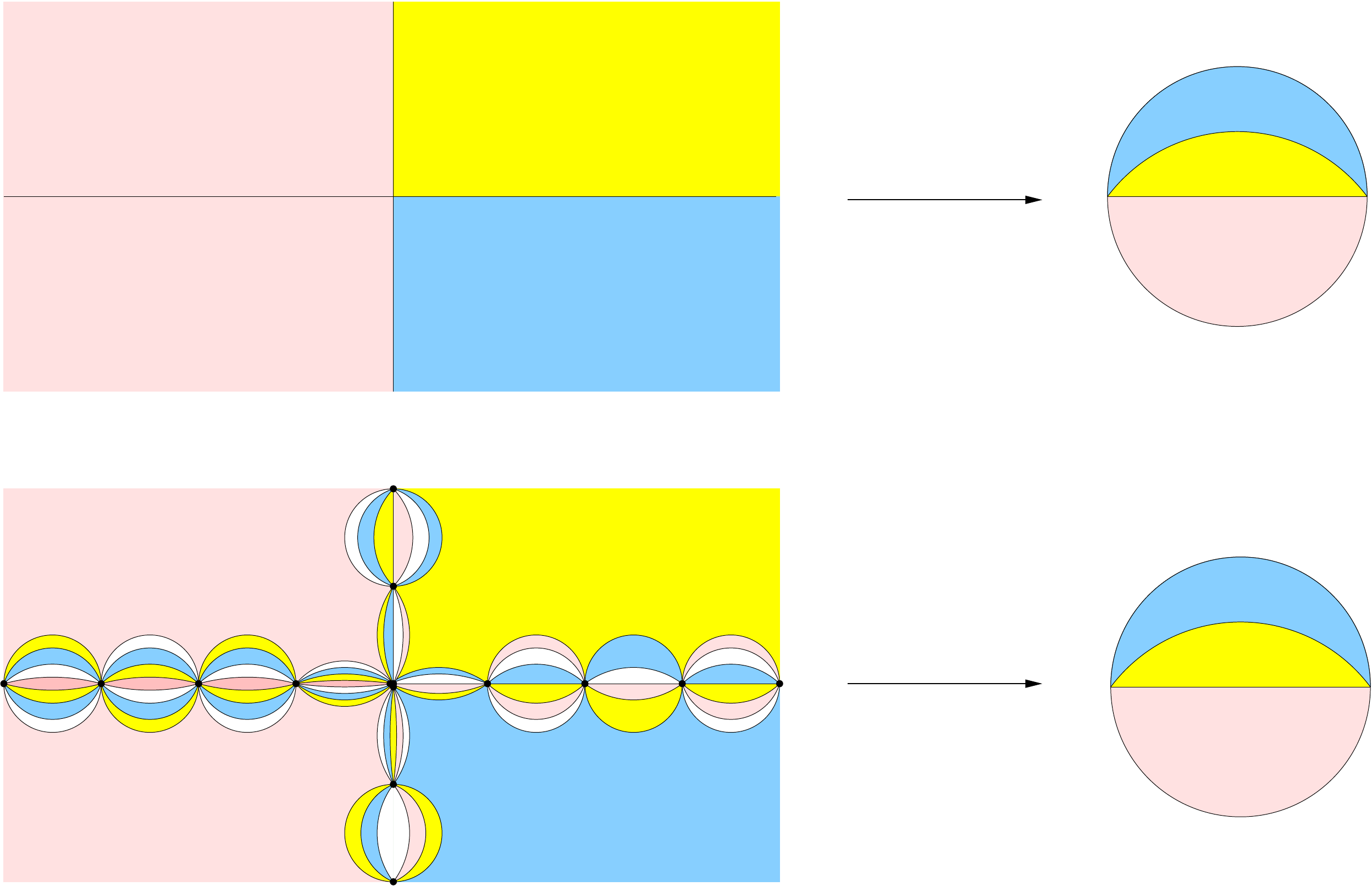} 
\end{overpic}
\caption{Example of a space $\Exp^\varrho(U;Y)$ and an expanded pillow map $\cP \colon \Exp^\varrho(U;Y) \to \PS^p(\Delta^\square_{n-1})$.}
\label{fig:Old-Part2-Map_f_v3}
\end{figure}

We refer to Figure \ref{fig:Old-Part2-Map_f_v3} for an illustration. Before the proof, we make a few comments on the statement. 

\begin{remark} The mapping $\cP$ maps $n$-simplices in $\Exp^\varrho(U;Y)$ to $n$-simplices in $\PS^p(\Delta^\square_{n-1})$. On  $\PC^\varrho(Y^\Delta)$, mapping $\cP$ resembles an Alexander map which maps adjacent simplices to adjacent simplices. On $\Real(U;Y)$ however $\cP$ is a folding map, that is not open.
\end{remark}

\begin{remark}
Since $\cP|_{\PC^\varrho(Y^\Delta)} \colon \PC^\varrho(Y^\Delta) \to \PS^p(\Delta^\square_{n-1})$ is a pillow map, we have that $\cP|_{|\partial \Real(U;Y)_i|} \colon |\partial \Real(U;Y)_i|\to |\partial P^\infty(\Delta^\square_{n-1})_{c(i)}|$ is an Alexander map.
\end{remark}

In general, the complex $Y^\Delta$ need not admit an Alexander map $|Y^\Delta|\to \bS^{n-1}$.
In particular, a collection $\{ A_i \colon |\partial \Real(U;Y)_i| \to |\partial P^\infty(\Delta^\square)_{c(i)}| \colon i=1,\ldots, r\}$ of Alexander maps need not descend to a well-defined map $|Y^\Delta|\to |\Sigma^p(\Delta^\square_{n-1})|$.

We prove Theorem \ref{thm:pillow-adjustment} in two steps. We discuss first Alexander maps on the boundaries $\partial \Real(U;Y)_i$ and extend them over $\Real(U:Y)$. Then we discuss the choice of the rank function $\varrho$ and the pillow map $\PC^\rho(Y^\Delta) \to \PS^p(\Delta^\square_{n-1})$.

\subsection{Alexander maps on boundaries of realizations}

Although $Y^\Delta$ need not admit an Alexander map, the complex $\Upsilon_i^\Delta$ does admit an Alexander map onto $\Sigma^2(\Delta^\square_{n-1})$. This Alexander map admits an extension over $\Real(U;Y)_i$ which is an Alexander map onto the pillow $P(\Delta^\square_{n-1})$. We record this observation as a lemma.

\begin{lemma}
\label{lemma:existence-of-Alexander-maps}
Let $U$ be a good cubical $n$-complex whose space is an oriented $n$-manifold  and $Y \subset U$ be a separating complex. Then, for each $i\in \{1,\ldots, r\}$, 
\begin{enumerate}
\item the inner boundary $\Upsilon_i$ of the realization $\Real(U;Y)_i$ admits an Alexander map \label{item:Alexander}
\[
A_{\Upsilon_i} \colon \Upsilon_i^\Delta \to \Sigma^2(\Delta_{n-1});
\]
\item  for each $\sigma\in Y^\Delta$ with $\sigma\in \partial U_{i_{-}}(\sigma)$ and $\sigma\in \partial U_{i_{+}}(\sigma)$, mappings  $A_{\Upsilon_{i_{-}(\sigma)}}$ and $ A_{\Upsilon_{i_{+}(\sigma)}}$ agree on the $(n-2)$-skeletons of $\sigma$ in the following sense:\label{item:compatible}
\[
A_{\Upsilon_{i_{-}(\sigma)}}\circ \pi_{i_{-}(\sigma)}^{-1}|_{\sigma^{[n-2]}} = 
A_{\Upsilon_{i_{+}(\sigma)}}\circ \pi_{i_{+}(\sigma)}^{-1}|_{\sigma^{[n-2]}};
\]
\item the Alexander map $A_{\Upsilon_i}$ has an extension to a simplicial folding map 
 \[
A_{\Real(U;Y)_i} \colon \Real(U;Y)_i^\Delta \to P(\Delta^\square_{n-1}).
\]
 \end{enumerate}
\end{lemma}

By the definition of separating complex (Definition \ref{def:separating-complex}), each $|\Upsilon_i|$ is an $(n-1)$-manifold. Claim \eqref{item:Alexander} in Lemma \ref{lemma:existence-of-Alexander-maps} follows from Alexander's Theorem.
We outline the proof to emphasize the connection between orientation on $|U|$ and that on the realizations.

\begin{proof}
Recall that  $\Real(U;Y)_i$ is oriented so that the quotient map $\pi_i$ is orientation preserving. We retain also the argument and the notations immediately before the statement of the lemma.

To prove \eqref{item:Alexander}, let $\sigma\in \partial U_i^\Delta \cap Y^\Delta$ be an $(n-1)$-simplex. In the case  $i_+(\sigma)\neq i_-(\sigma)=i $ or $i_-(\sigma)\neq i_+(\sigma)=i $, the simplex $\sigma$ has an orientation induced from the unique $n$-cube in $U_i$ containing $\sigma$. 

Suppose  that $i_+(\sigma)=i_-(\sigma)=i$ and that $\tau_+$ and $\tau_-$ are the $n$-simplices in $U$ for which $\tau_+\cap \tau_- = \sigma$, $p_U(\tau_+)=1$, and $p_U(\tau_-)=-1$. The simplex $\sigma$ has two lifts $\sigma_1 \ne \sigma_2$ in $\partial \, \Real(U;Y)_i$, and the lifts $\sigma_1$ and $\sigma_2$, respectively, are the faces of $n$-simplices $\tau_1$ and $\tau_2$, respectively. After relabeling if needed, we may assume that $\pi_i(\tau_1) = \tau_+$ and $\pi_i(\tau_2)=\tau_-$; thus $\tau_1$ and $\tau_2$ are positively and negatively oriented, respectively, in $\Real(U;Y)_i$.  We define $p_i(\sigma_1)=1$ and $p_i(\sigma_2)=-1$. 
This yields a parity function $p_i \colon \Upsilon_i^\Delta \to \{\pm 1\}$. 
The parity function $p_i$ gives rise to an Alexander map $A_{\Upsilon_i} \colon \Upsilon_i^\Delta \to \Sigma^2(\Delta_{n-1})$.

The compatibility \eqref{item:compatible} follows from the fact that quotients $\pi_i$ are orientation preserving.

Since ${\Real(U;Y)_i}$ is a cubical complex, the vertices, $v(\tau)_0,\ldots, v(\tau)_n$, of any $n$-simplex $\tau\in {\Real(U;Y)_i}^\Delta$ may be labeled so that $v(\tau)_k$ is the barycenter of an $k$-cube in ${\Real(U;Y)_i}$. On the target side, two $n$-simplices  in the pillow $P(\Delta_{n-1})$ have all their vertices in common. Denote by $w_n$ the  vertex  in the interior of their common face, and let other vertices $w_0, w_1,\ldots, w_{n-1}$ be labeled so that the Alexander map $A_{\Upsilon_i} \colon \Upsilon_i^\Delta \to \Sigma^2(\Delta_{n-1})$ is index preserving.
Thus $A_{\Upsilon_i}$ may be extended to a index preserving simplicial map 
\[
A_{\Real(U;Y)_i} \colon \Real(U;Y)_i^\Delta \to P(\Delta^\square_{n-1}).
\]
\end{proof}

\subsection{Proof of Theorem \ref{thm:pillow-adjustment}}

As a preparatory step, let $2\leq p\leq r$ and $c\colon \{1,\ldots r\} \to \{1,\ldots, p\}$ be a surjection as in the statement of Theorem \ref{thm:pillow-adjustment}. For each $i\in \{1,\ldots, r\}$, let 
\[
A_{i,c(i)} = \iota_{c(i)} \circ A_{\Real(U;Y)_i} \colon \Real(U;Y)^\Delta_i \to \PS^p(\Delta^\square_{n-1}),
\]
where $\iota_j \colon P(\Delta^\square_{n-1}) \to \PS^p(\Delta^\square_{n-1})$ is the simplicial extension of the embedding $\iota_j \colon \Sigma^2(\Delta^\square_{n-1}) \to \Sigma^p(\Delta^\square_{n-1})$ onto the pillow $P^\infty(\Delta^\square_{n-1})_j$ of $\PS^p(\Delta^\square_{n-1})$.
 
Let also 
\[
A_c \colon \Real(U;Y)^\Delta\to \PS^p(\Delta^\square_{n-1})
\]
be the map defined by $A_c|_{\Real (U;Y)_i} = A_{i,c(i)}$ for  $i\in \{1,\ldots, r\}$.

\begin{remark}
Suppose $\varrho \colon (Y^\Delta)^{[n-1]} \to \Z_+$ is a rank function and consider spaces $\Real(U;Y)_i$ as subspaces of $\Exp^\varrho(U;Y)$. For $i\ne i'$, the intersection $Y_{i, i'} = {\Real (U;Y)_i}^\Delta\cap \Real (U;Y)_{i'}^\Delta$ in $\Exp^\varrho(U;Y)$ is contained in the $(n-2)$-skeleton of either complex. By Lemma \ref{lemma:existence-of-Alexander-maps}, $A_{i,c(i)}$ and $A_{i',c((i')}$ agree on $Y_{i, i'}$, hence mapping $A_c \colon \Real(U;Y)^\Delta \to \PS^p(\Delta^\square_{n-1})$ is well-defined also for $\Real(U;Y)^\Delta \subset \Exp^\varrho(U;Y)$. 
\end{remark}

\begin{proof}[Proof of Theorem \ref{thm:pillow-adjustment}]
We define a rank function $\varrho \colon (Y^\Delta)^{[n-1]}\to \{p,\ldots, 2p\}$ as follows. 

Let $\sigma\in (Y^\Delta)^{[n-1]}$, and let $i_+ = i_+(\sigma)$ and $i_-=i_-(\sigma)$ be the indices in $\{1,\ldots, r \}$ for which $\sigma$ is positively oriented with respect to $U_{i_+(\sigma)}$, and is negatively oriented with respect to $U_{i_-(\sigma)}$. Let  $\sigma_+\in \Upsilon_{i_+}$ and $\sigma_-\in \Upsilon_{i_-}$, respectively, be the lifts of $\sigma$ in the corresponding realizations. Let $\tau_+$ and $\tau_-$ be the $n$-simplices in $\Real(U;Y)_{i_+}$ and $\Real(U;Y)_{i_-}$, respectively,
 having $\sigma_+$ and $\sigma_-$ as its face. 

Let $\rho_\sigma\in \{p,\ldots, 2p\}$ be the smallest integer for which there exists a pillow map $A_{\sigma} \colon \PC^{\rho_\sigma}(\sigma) \to \PS^p(\Delta^\square_{n-1})$ having the properties:
\begin{enumerate}
\item pillow $A_{\sigma}(P(\sigma)_\TOP)$ is adjacent to the pillow in $\PS^p(\Delta^\square_{n-1})$ that contains $A_c(\tau_+)$, \label{item:top}
\item pillow $A_{\sigma}(P(\sigma)_\BOT)$ is adjacent to the pillow in $\PS^p(\Delta^\square_{n-1})$ that contains $A_c(\tau_-)$, and   \label{item:bot}
\item $A_\sigma|_{\partial \sigma}=A_c|_{\partial \sigma}$.
\end{enumerate}

Since pillows in the pillow-sphere $\PS^p(\Delta^\square_{n-1})$ form a cycle of length $p$, it follows from Lemma \ref{lemma:existence-of-Alexander-maps} that such pillow map $A_\sigma$ and number $\rho_\sigma$ exist. 

Define the rank function $\varrho \colon (Y^\Delta)^{[n-1]}\to \{p,\ldots, 2p\}$ by $\sigma \mapsto \rho_\sigma$, and a pillow map $A^\varrho \colon \PC^\varrho(Y^\Delta) \to \PS^p(\Delta^\square_{n-1})$ by the condition $A^\varrho|_{P^{\varrho(\sigma)}}= A_\sigma$. 

Since maps $A_c$ and $A^\varrho$ agree on the boundary $\partial \PC^\varrho(Y^\Delta)$, the mapping $\cP \colon \Exp^\varrho(U;Y) \to S^\varrho(\Delta^\square_{n-1})$ for which 
\[
\cP|_{\Real(U;Y)} = A_c, \quad\text{and}\quad \cP|_{P^\varrho(Y^\Delta)} = A^\varrho
\]
is well-defined. 
Since $A_c$ and $A^\varrho$ are Alexander maps, also $\cP$ is an Alexander map. Conditions (1) and (2) for $\cP$ are satisfied by choices of $A_c$ and $A^\varrho$. This completes the construction. 
\end{proof}

\begin{remark}\label{rmk:coloring-function}\index{coloring function}
Heuristically, we may regard $\{1,\ldots,p\}$ as a collection of colors and $c \colon \{1,\ldots, r\}\to \{1,\ldots,p\}$ a coloring function. 
Let  $\PS^p(\Delta^\square_{n-1})$ be the pillow-sphere of rank $p$.
As the name suggests, we associate to each pillow $\PS^\infty(\Delta_{n-1})_j,  j=1,\ldots,p,$ in $\PS^p(\Delta^\square_{n-1})$  the color of its index.  In view of  Theorem \ref{thm:pillow-adjustment}, we may also associate  to the realization $\Real (U;Y)_i$ the color $c(i)$ and to each pillow $\tau$ in $\PC^\varrho(Y^\Delta)$ the color of its image $\cP(\tau)$. 
\end{remark}

For the forthcoming construction, we need to fix, for  each $j\in \{1,\ldots,p\}$, a component  $\Real(U;Y)_{ i_j}$ of  $\Real(U;Y)$ which is mapped by $F_{\cT_{Y^\Delta}}$ to  the pillow $P(\Delta_{n-1})_j$ of color $j$. To simplify the notation, we rearrange the indices in the components  of $\Real(U;Y)$, if necessary, and make the following convention.

\begin{convention}\label{convention:coloring}From now on, we assume the surjection $c\colon \{1,\ldots, r\} \to \{1,\ldots,p\}$ satisfies 
\[ c(j)=j, \quad \text{for} \,\,\, j=1,\ldots p.\]
\end{convention}

\chapter{Rickman expansion of complexes}
\label{chap:Alexander--Rickman}

In this chapter we begin the construction of  a weaved approximation $X$ of a separating complex $Z\subset \Refine^{\nu k}(K)$, and a mapping $\varphi \colon X\to \Sigma^p(\Delta^\square_{n-1})$ for Theorem \ref{thm:combination-I}. We complete the proof of Theorem \ref{thm:combination-I} in Chapter \ref{chap:proof-of-combination-I}. 

We do not construct set $X$ for $(K;Z)$ directly, instead, construct a simplicial  set $\X \subset U$ and a map $f \colon \X \to \Sigma^\varrho(\Delta^\square_{n-1})$ in the more general setting used in Theorem \ref{thm:pillow-adjustment}, that is, $U$ is an orientable good $n$-complex and $Y$ is an $(n-1)$-subcomplex of $U$ for which the graph $\Gamma(U;Y)$ has $r$ connected components. 

\subsubsection{Alexander-Rickman maps}
We give first a definition for $(n-1)$-dimensional simplicial sets.
\begin{definition}
\label{def:simplicial-set}
\index{simplicial set}
A closed space $X$ is \emph{$(n-1)$-dimensional simplicial set} 
if there exists a finite collection $\cF$ of mutually essentially disjoint $(n-1)$-simplices which cover $X$ and have the property that, whenever two simplices $\sigma$ and $\sigma'$ in $\cF$ meet, 
there exist a face $\xi$  of $\sigma$ and  a face $\xi'$  of $\sigma'$ for which
\begin{enumerate}
\item  $\sigma\cap \sigma'= \xi\cap\xi'$, and 
\item $(\xi\cup \xi') \setminus (\sigma\cap\sigma') \subset  \interior \,(\xi \cup\xi')$.
\end{enumerate}
We call $\cF$ a \emph{simplicial structure of $X$} and the pair $(X,\cF)$ a \emph{simplicial $(n-1)$-space}. Simplices in $\cF$ are \emph{adjacent} if their intersection has dimension $n-2$.
\end{definition}

We define Alexander--Rickman maps from a simplicial space to a branched sphere as follows.

\begin{definition}
\label{def:Alexander-Rickman-map}
\index{Alexander-Rickman map}
A mapping $f\colon X\to \Sigma^p(\Delta^\square_{n-1})$ from a simplicial $(n-1)$-space $(X,\cF)$ to a branched sphere $\Sigma^p(\Delta^\square_{n-1})$ is an \emph{Alexander--Rickman map} if any two adjacent $(n-1)$-simplices in $X$ are mapped to the opposite hemispheres of the boundary $\partial (\PS^p(\Delta^\square_{n-1})_j)$ of a pillow $\PS^\infty(\Sigma_{n-1})_j$ in $\PS^p(\Sigma_{n-1})$.
\end{definition}

Following the scheme of Chapter \ref{chap:pillows}, we construct an auxiliary space $\Exp^\varrho(U;Y,\cP)$ over $\Exp^\varrho(U;Y)$ with respect to an expanded pillow map $\cP \colon \Exp^\varrho(U;Y)\to \PS^p(\Delta^\square_{n-1})$,  a simplicial set $\X$ in  $\Exp^\varrho(U;Y,\cP)$ corresponding to $X$, and a map $ \X \to \Sigma^\varrho(\Delta^\square_{n-1})$ with analogous properties to $\varphi$.

The space $\Exp^\varrho(U;Y,\cP)$ is constructed over $\Exp^\varrho(U;Y)$ for which there exists a canonical quotient map $\Exp^\varrho(U;Y,\cP) \to \Exp^\varrho(U;Y)$. Heuristically, this quotient closes the openings created between components in $\Exp^\varrho(U;Y) \setminus \Sigma^\varrho(Y^\Delta)$ of the same color associated to $\cP$. In addition, we may view $\Exp^\varrho(U;Y, \cP)$ as a version of $\Exp^\varrho(U;Y)$, where $\Sigma^\varrho(Y^\Delta)$ is replaced with an $(n-1)$-dimensional simplicial set $\X$ which has the same number of complementary components as $Y$.

In what follows, let $(U;Y)$ be a pair as in Chapter \ref{chap:pillows}, that is, $U$ is a good complex and $Y$ an $(n-1)$-dimensional subcomplex of $U$ having the property that each $(n-1)$-cube in $Y$ is a face of two $n$-cubes of $U$. Let again $r$ be the number of complementary components of $Y$ in $U$. Let also $\varrho \colon (Y^\Delta)^{[n-1]}\to \Z_+$ be the rank function of $\Exp^\varrho(U;Y)$. 

\section{Rickman pillows}
\label{sec:Rickman-pillows}

We begin with a procedure of combining pillows in a pillow cell. 
Let $\sigma$ be an $(n-1)$-simplex, $\PC^\rho(\sigma)$ a pillow cell on $\sigma$, and $\cP \colon \PC^\rho(\sigma) \to \PS^p(\Delta^\square_{n-1})$ be a pillow map, where $p \le \rho \le 2p$. 
We construct cells $R^\rho(\sigma)_1,\ldots, R^\rho(\sigma)_p$, one for each $j\in \{1,\ldots, p\}$, which form an essential partition of $\PC^\rho(\sigma)$. We call the cells constructed  \emph{Rickman pillows}. \index{Rickman pillow $R^\rho(\sigma)_j$}

Let $\beta\ge 10p$  and $\mathsf T_\sigma$ be a $\beta$-iterated barycentric subdivision of $\sigma$. Let $\zeta_\sigma$ be the unique face of $\sigma$ whose vertices are barycenters of cubes in $Y$ of non-zero dimensions. 
Choose and fix next $p$ mutually disjoint $(n-2)$-simplices $\zeta_{\sigma,1},\ldots, \zeta_{\sigma,p}$ in the triangulation $\mathsf T_\sigma|_{\zeta_\sigma}$ of $\zeta_\sigma$.  For simplicity, we omit $\sigma$ in the indices and write $\zeta_j = \zeta_{\sigma,j}$, and denote for each $j$, by $\xi_j$   the unique $(n-1)$-simplex in $\mathsf T_{\sigma}$ having $\zeta_j$ as its face.  Here the subdivision $\mathsf T_\sigma$ is introduced for the purpose of marking cuts on the face of $\sigma$ only; the simplex $\sigma$ is not actually triangulated.

Let $j\in \{1,\ldots, p\}$. Since $p \le \rho \le 2p$, there exists at least one and at most two indices $t\in \{1,\ldots, \rho-1\}$ for which $\cP(P^\infty(\sigma)_t) = \PS^p(\Delta^\square_{n-1})_j$; we call these pillows $P^\infty(\sigma)_t$ the \emph{pillows of color $j$}. Let $1\le t_j^{\min} \le t_j^{\max} \le \rho-1$ be such indices. Note that either $t_j^{\max} = t_j^{\min}$ or $t_j^{\max}-t_j^{\min} = p$. 
We set 
\[
\cI_j(\cP) = \left\{ \begin{array}{ll}
\{ t_j^{\min}+1,\ldots, t_j^{\max}-1\}, & \text{ if } t_j^{\min} < t_j^{\max} \\
\emptyset, & \text{ if } t_j^{\min} = t_j^{\max},
\end{array}\right.
\]
and note that $\cI_j(\cP) \subset \{2,\ldots, \rho-2\}$. 
The union 
\[
V_j(\cP) = \bigcup_{t\in \cI_j(\cP)} P^\infty(\sigma)_t,
\]
of pillows in $\PC^\rho(\sigma)$ 'between' two color $j$ pillows, if nonempty, is a pillow cells of rank $(p-1)$.

Let $j\in \{1,\ldots, p\}$. If $t_j^{\min} = t_j^{\max}$, no modification is made to the pillow of  color $j$. If $t_j^{\min} < t_j^{\max}$, we combine the interiors of two color $j$ pillows, $P^\infty(\sigma)_{t_j^{\min}} $ and $P^\infty(\sigma)_{t_j^{\max}}$ following these steps:
\begin{itemize}
\item[(1)] first cut along $\interior \,\zeta_j$ on the rim of $\Sigma^\rho(\sigma)$ to obtain two copies $(\zeta_j)_+$ and $(\zeta_j)_-$ of $\zeta_j$ whose union is an $(n-2)$-sphere,
\item[(2)] then seal the opening by sewing $(\Sigma^\rho(\sigma)\cap V_j(\cP)) \setminus \interior \, \zeta_j$ along $\interior (\zeta_j)_+$,
\item[(3)] then seal the other opening by sewing $((\Sigma^\rho(\sigma)\setminus V_j(\cP))\cup \partial \sigma) \setminus \interior \, \zeta_j$ along $\interior (\zeta_j)_-$.
\end{itemize}
This operation produces an $(n-1)$-dimensional simplicial set $\Sigma^\rho(\sigma;\cP)$ in $\PC^\rho(\sigma)$, which has $p$ complementary components $\mathring{R}(\sigma; \cP)_1,\ldots, \mathring{R}(\sigma; \cP)_p$ for which  $\mathring{R}(\sigma; \cP)_j$ is an open $n$-cell of color $j$ with  closure $R^\rho(\sigma; \cP)_j$  a tame closed $n$-cell. 

We call  $n$-cells $R^\rho(\sigma;\cP)_1, \ldots, R^\rho(\sigma;\cP)_p$  \emph{Rickman pillows of $\cP$} and the set $\Sigma^\rho(\sigma;\cP)$  a \emph{Rickman sheet of $\cP$}.

The identification of the copies $(\zeta_j)_\pm$ of $\zeta_j$ with $\zeta_j$ itself yields a projection $\lambda_{\sigma,\cP} \colon \Sigma^\rho(\sigma;\cP) \to \Sigma^\rho(\sigma)$ for which the composition 
\[
f_{\sigma,\cP} = \cP \circ \lambda_{\sigma,\cP} \colon \Sigma^\rho(\sigma;\cP) \to \Sigma^p(\Delta^\square_{n-1})
\]
is an Alexander--Rickman map satisfying $f_{\sigma,\cP}(\partial R^\rho(\sigma;\cP)_j) = \partial \PS^p(\Delta^\square_{n-1})_j$. 

We may now give each $n$-cell $R^\rho(\sigma;\cP)_j$ a simplicial structure by coning the complex on the boundary with a one interior point. The map $f_{\sigma,\cP}$ extends now to a coning map $\cR_{\sigma,\cP} \colon \PC^\rho(\sigma) \to \PS^p(\Delta^\square_{n-1})$ satisfying $\cR_{\sigma,\cP}(R^\rho(\sigma;\cP)_j) = \PS^p(\Delta^\square_{n-1})_j$ for each $j=1,\ldots, p$. In each Rickman pillow, the mapping $\cR_{\sigma,\cP}$ is either $1$-to-$1$ or $2$-to-$1$.

\subsection{Rickman sheets}

Let now $\cP \colon \PC^\varrho(Y^\Delta) \to \PS^p(\Delta^\square_{n-1})$ be a pillow map. Then 
\[
\bigcup_{\sigma\in (Y^\Delta)^{[n-1]}} ( \bigcup_{j=1}^p  \,\,R^{\varrho(\sigma)}(\sigma; \cP)_j) = \bigcup_{\sigma\in (Y^\Delta)^{[n-1]}} \PC^{\varrho(\sigma)}(\sigma) = \PC^\varrho(Y^\Delta).
\]

We call the union of Rickman sheets 
\index{$\Sigma^\varrho(Y; \cP)$}
\[
\Sigma^\varrho(Y; \cP)=\bigcup_{\sigma\in (Y^\Delta)^{[n-1]}} \Sigma^{\varrho(\sigma)}(\sigma; \cP) \, \subset \PC^\varrho(Y^\Delta)  
\]
the \emph{Rickman sheet induced by $\cP$}.

Together the simplicial structures on  $\Sigma^{\varrho(\sigma)}(\sigma; \cP)$"s  give the space $\Sigma^\varrho(Y;\cP)$ a structure as a simplicial space which admits a simplicial map $\Sigma^\varrho(Y;\cP) \to Y^\Delta$. 

Also, as  in the case of individual pillow cell, there exists a pillow map $ |\PC^\varrho(Y^\Delta)|\to \PS^p(\Delta^\square_{n-1})$ which agrees with $\cP$ on $\partial \PC^\varrho(Y^\Delta)$ and maps each $R^{\varrho(\sigma)}(\sigma; \cP)_j$ onto $\PS^p(\Delta^\square_{n.1})_j$.

\section{Unified Rickman pillows}
\label{sec:unified-Rickman-pillows}

Next we combine Rickman pillows of the same color associated to all $(n-1)$-simplices in $\sigma\in Y^\Delta$. For this, we fix a spanning tree $\cT_{Y^\Delta}$ in the adjacency graph $\Gamma(Y^\Delta)$.

Let $\beta = \beta(p,Y)\ge 1$ be the smallest integer for which a $\beta$-fold iterated barycentric subdivision $\mathsf T_\sigma$ of $\sigma$ has at least
\[
10\, p \left(\max_{\sigma\in (Y^\Delta)^{[n-1]}} \# \{ \sigma' \in (Y^\Delta)^{[n-1]} \colon \sigma' \text{ is adjacent to } \sigma\}\right)
\]
mutually disjoint $(n-2)$-simplices on a face of $\sigma$. 

Let $\cP \colon \PC^\varrho(U;Y) \to \PS^p(\Delta^{\square}_{n-1})$ be a pillow map and let $\Sigma^\varrho(Y;\cP)$ be a Rickman sheet induced by $\cP$, which is constructed using $\beta$.

 For each edge $\{\sigma,\sigma'\}\in \cT_{Y^\Delta}$, we fix a collection $\{\delta_{\{\sigma,\sigma'\},1}, \ldots, \delta_{\{\sigma,\sigma'\},p}\}$ of mutually disjoint $(n-2)$-simplices in $\mathsf T_{\sigma} \cap \mathsf T_{\sigma'}$.
An $(n-2)$-simplex $e$ in $Y^\Delta$ may represent several edges in $\cT_{Y^\Delta}$, that is, there are distinct edges $\{\sigma,\sigma'\}$ and $\{\sigma'',\sigma'''\}$ for which $\sigma\cap\sigma'=\sigma''\cap \sigma'''=e$. Hence we require, in addition, all  $(n-2)$-simplices chosen with respect to the edges in $\cT_{Y^\Delta}$  be mutually disjoint.
We further assume that none of these $(n-2)$-simplices  meet any  
$\zeta_{\sigma,1},\ldots, \zeta_{\sigma,p}$, for $\sigma \in Y^\Delta,$ previously chosen in Section \ref{sec:Rickman-pillows}. These additional requirements may be fulfilled in view of the choice of number $\beta$.
Let
\[
D(\{\sigma,\sigma'\}) = \interior ( \delta_{\{\sigma,\sigma'\},1}\cup \cdots \cup \delta_{\{\sigma,\sigma'\},p}),
\]
where the interior is taken with respect to $\sigma\cap \sigma'$.

Given an edge $\{\sigma,\sigma'\}$ of $\cT_{Y^\Delta}$,  we create now, for each $j\in \{1,\ldots, p\}$, an opening between adjacent Rickman pillows $R^{\varrho(\sigma)}(\sigma; \cP)_j$ and $R^{\varrho(\sigma')}(\sigma; \cP)_j$ along $\delta_{\{\sigma,\sigma'\},j}$.

 \begin{figure}[h!]
\begin{overpic}[scale=0.4,unit=1mm]{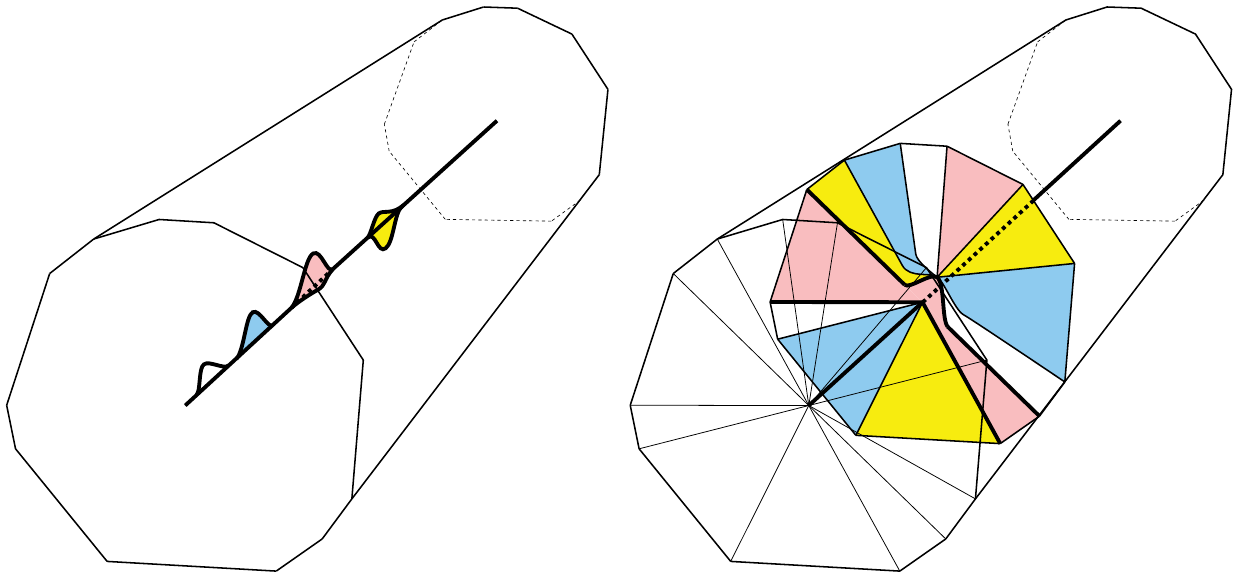} 
\end{overpic}
\caption{Towards a merge of Rickman pillows -- a local picture.}
\label{fig:unified Rickman_pillows}
\end{figure}
 
Let 
\[
\Sigma^\varrho(Y^\Delta; \cP, \cT_{Y^\Delta}) = \left( \bigsqcup_{\sigma \in Y^{[n-1]}} \sigma \times \{1,\ldots, \varrho(\sigma)\}\right) \Big/_{\sim_{\cT_{Y^\Delta}}},
\]
be the quotient space for which  $\sim_{\cT_{Y^\Delta}}$ is the smallest equivalence relation satisfying the following conditions. \index{$\Sigma^\varrho(Y; \cP, \cT_{Y^\Delta})$}
\begin{enumerate} 
\item Let $\sigma \in (Y^\Delta)^{[n-1]}$ and $(x,t)\in \sigma \times \{1,\ldots, \varrho(\sigma)\}$. 
\begin{enumerate}
\item [(i)] Suppose that $x\not \in D(\{\sigma, \sigma'\})$ for any $\{\sigma,\sigma'\}\in \cT_{Y^\Delta}$. Set the equivalence class $[(x,t)]_{\cT_{Y^\Delta}}$ to be $[(x,t)]_\cP$;
\item [(ii)] Suppose that $x\in \interior \delta_{\{\sigma,\sigma'\},j}$ for some 
$\{\sigma,\sigma'\}\in \cT_{Y^\Delta}$. Set $(x,t)\sim_{\cT_{Y^\Delta}} (x,t^{\min}_j(\sigma))$ for all $t\le t^{\min}_j(\sigma)$, and set $(x,t)\sim_{\cT_{Y^\Delta}} (x, 1+t^{\min}_j(\sigma))$ for all $t \geq 1+t^{\min}_j(\sigma)$. Under this, two copies $[(\delta_{\{\sigma,\sigma'\},j},t^{\min}_j(\sigma))] $ and $[(\delta_{\{\sigma,\sigma'\},j},1+ t^{\min}_j(\sigma))] $ of $\delta_{\{\sigma,\sigma'\},j}$ form an $(n-2)$-sphere which encloses an $(n-1)$-cell $\xi_{(\sigma,\sigma'),j}$. 
\end{enumerate}
\item For $\sigma, \sigma'  \in (Y^\Delta)^{[n-1]}$ with $\{\sigma,\sigma'\}\in \cT_{Y^\Delta}$, we also identify the $(n-1)$-cells $\xi_{(\sigma,\sigma'),j}$ and $\xi_{(\sigma',\sigma),j}$, and denote it  by $\xi_{\{\sigma',\sigma\},j}$.
\end{enumerate}
We call   $ \Sigma^\varrho(Y^\Delta; \cP, \cT_{Y^\Delta})$ a  \emph{unified Rickman sheet}.

The quotient map
\[
\Pi_{\cT_{Y^\Delta}} \colon \bigsqcup_{\sigma \in Y^{[n-1]}} \sigma \times \{1,\ldots, \varrho(\sigma)\} \to \Sigma^\varrho(Y^\Delta; \cP, \cT_{Y^\Delta})
\]
yields a natural triangulation and also a canonical length metric, induced by the metrics on individual simplices $\sigma\times \{t\}$, on the space $\Sigma^\varrho(Y^\Delta; \cP, \cT_{Y^\Delta})$.

Since the equivalence classes of $\sim_{\cT_{Y^\Delta}}$ are contained in the equivalence classes of $\sim_{\cP}$, there exists a natural projection 
\[
\lambda_{\cT_{Y^\Delta}} \colon \Sigma^\varrho(Y^\Delta; \cP, \cT_{Y^\Delta}) \to \Sigma^\varrho(Y; \cP), \quad [(x,t)]_{\cT_{Y^\Delta}} \mapsto [(x,t)]_{\cP},
\]
induced by the inclusion. Furthermore, since $\lambda_{\cT_{Y^\Delta}}$ maps the $(n-1)$-simplices in $\Sigma^\varrho(Y^\Delta; \cP,\cT_{Y^\Delta})$ one-to-one onto the $(n-1)$-simplices in $\Sigma^\varrho(Y^\Delta; \cP)$, the composition 
 \[
\cP_{\cT_{Y^\Delta}} = \cP \circ \lambda_{\cT_{Y^\Delta}} \colon \Sigma^\varrho(Y^\Delta; \cP,\cT_{Y^\Delta}) \to \Sigma^p(\Delta^\square_{n-1})
\]
is a simplicial branched covering map. 

We describe now the $n$-cell bounded by 
$\Sigma^\varrho(Y^\Delta; \cP,\cT_{Y^\Delta})$ and the structure of the map $\cP_{\cT_{Y^\Delta}}$.
By construction, for  $\{\sigma,\sigma'\}\in \cT_{Y^\Delta}$ and $1\leq j\leq p$, the closure $S_{\{ \sigma,\sigma'\} ,j}$ of $(\lambda_{\cT_{Y^\Delta}})^{-1}(\delta_{\{\sigma, \sigma'\},j})$ is an $(n-2)$-sphere bilipschitz equivalent to a Euclidean $(n-2)$-sphere of radius $r_{\sigma,\sigma'}>0$. We glue 
a Euclidean $(n-1)$-ball $\bar B^{n-1}(r_{\sigma,\sigma'})$ along $S_{\{\sigma,\sigma'\},j}$ and denote by $\xi_{\{\sigma,\sigma'\},j}$ the obtained $(n-1)$-cell. Let 
\[
\check \Sigma^\varrho(Y; \cP,\cT_{Y^\Delta})=  \Sigma^\varrho(Y; \cP,\cT_{Y^\Delta}) \cup  \bigcup_{\{\sigma,\sigma'\}\in \cT_{Y^\Delta}} \left(\bigcup_{j=1}^p\,  \xi_{\{\sigma,\sigma'\},j}\right),
\]
and let 
\[
\check \lambda_{\cT_{Y^\Delta}} \colon \check \Sigma^\varrho(Y; \cP, \cT_{Y^\Delta}) \to \Sigma^\varrho(Y; \cP)
\]
be an extension of $\lambda_{\cT_{Y^\Delta}}$ for which  $\check \lambda_{\cT_{Y^\Delta}}(\xi_{\{\sigma,\sigma'\},j})=\delta_{\{\sigma,\sigma'\},j}$ for each $j$.

For each $\sigma\in Y^\Delta$, the preimage $(\check \lambda_{\cT_{Y^\Delta}})^{-1}(R^{\varrho(\sigma)}(\sigma; \cP)_j \, \cap\, \Sigma^\varrho(Y; \cP))$ of the boundary of a Rickman pillow
is an $(n-1)$-sphere bilipschitz equivalent to a Euclidean $(n-1)$-sphere $r_\sigma>0$. We glue a Euclidean $n$-ball $\bar B^n(r_\sigma)$ into this sphere and denote by $\cR^{\varrho(\sigma)}(\sigma;\cP, \cT_{Y^\Delta})_j$ the puffed Rickman pillow enclosed by the sphere.

The boundary of the puffed $\cR^{\varrho(\sigma)}(\sigma;\cP, \cT_{Y^\Delta})_j$ consists of either two or four $(n-1)$-simplices in $\Sigma^\varrho(Y; \cP, \cT_{Y^\Delta})$ and one $(n-1)$-cell $\xi_{\{\sigma,\sigma'\},j}$ for each $\{\sigma,\sigma'\}\in \cT_{Y^\Delta}$.

We call the union 
\index{Rickman pillow!unified $\cR^\varrho(Y; \cP, \cT_{Y^\Delta})_j$}
\[
\cR^\varrho(Y^\Delta; \cP, \cT_{Y^\Delta})_j = \bigcup_{\sigma\in (Y^\Delta)^{[n-1]}}  \cR^{\varrho(\sigma)}(\sigma;\cP, \cT_{Y^\Delta})_j 
\]
of the puffed Rickman pillows the \emph{unified Rickman pillow of color $j$}.

\begin{lemma}
\label{lemma:Rickman-n-cells}
The unified Rickman pillows 
\[
\cR^\varrho(Y; \cP, \cT_{Y^\Delta})_1,\ldots, \cR^\varrho(Y; \cP, \cT_{Y^\Delta})_p
\]
are $n$-cells. Furthermore, $\bigcup_{j=1}^p \partial \cR^\varrho(Y; \cP, \cT_{Y^\Delta})_j = \Sigma^\varrho(Y; \cP, \cT_{Y^\Delta})$. 
\end{lemma}
\begin{proof}
If $n$-cells $\cR^{\varrho(\sigma)}(\sigma;\cP, \cT_{Y^\Delta})_j$ and $\cR^{\varrho(\sigma')}(\sigma';\cP, \cT_{Y^\Delta})_j$ are adjacent in $\cT_{Y^\Delta}$, they meet in $\xi_{\{\sigma,\sigma'\},j}$. Since $\cT_{Y^\Delta}$ is a tree, the first claim follows.  Furthermore, since the interior of each $\xi_{\{\sigma,\sigma'\},j}$ is contained in the interior of $\cR^\varrho(Y^\Delta; \cP, \cT_{Y^\Delta})_j$, the second claim follows.
\end{proof}

The $(n-1)$-simplices in $\partial \cR^{\varrho(\sigma)}(\sigma,\cP,\cT_{Y^\Delta})_j$ give $\Sigma^\varrho(Y^\Delta; \cP,\cT_{Y^\Delta})$ the  structure of a simplicial set. The map $\cP_{\cT_{Y^\Delta}} \colon \Sigma^\varrho(Y^\Delta; \cP,\cT_{Y^\Delta}) \to \Sigma^p(\Delta^\square_{n-1})$ defined above is an Alexander-Rickman map with respect to this simplicial structure. The following lemma gives a geometrical description for the structure  on $\cP_{\cT_{Y^\Delta}}$.

Heuristically, a unified Rickman pillow is similar to the space of a tunnel. The proof of the following lemma uses the idea  of deforming Alexander maps on the boundaries of tunnels by removing one leaf at time, as in Proposition \ref{prop:extension-over-tunnels}.

\begin{lemma}
\label{lemma:contraction-to-leaf}
Let $\sigma_0\in \cT_{Y^\Delta}$ be a leaf. Then, for each $j\in \{1,\ldots, p\}$, there exist
\begin{enumerate}
\item a bilipschitz map 
\[h_j \colon \cR^{\varrho(\sigma_0)}(\sigma_0;\cP,\cT_{Y^\Delta})_j \to \cR^\varrho(Y^\Delta;\cP, \cT_{Y^\Delta})_j,\]
from a puffed Rickman pillow to a unified Rickman pillow, which is the identity on $\cR^{\varrho(\sigma_0)}(\sigma_0;\cP,\cT_{Y^\Delta})_j \cap \Sigma^\varrho(Y^\Delta;\cP,\cT_{Y^\Delta})$, and
 \item a bilipschitz map
 \[g_j \colon R^{\varrho(\sigma_0)}(\sigma_0;\cP,\cT_{Y^\Delta})_j \to \cR^\varrho(\sigma_0;\cP, \cT_{Y^\Delta})_j,\]
from a Rickman pillow to a puffed Rickman pillow, which is the identity on all $(n-1)$-simplices in $\partial R^{\varrho(\sigma_0)}(\sigma_0;\cP,\cT_{Y^\Delta})_j $ except on one, namely $\tau_0$, that is adjacent to $\xi_{\{\sigma_0,\widehat \sigma_0\},j}$, and maps $\tau_0$ onto $\tau_0\cup \xi_{\{\sigma_0,\widehat \sigma_0\},j}$,
 \end{enumerate}
 for which 
\[
\cP_{\cT_{Y^\Delta}} \circ h_j \circ g_j|_{\partial R^{\varrho(\sigma_0)}(\sigma_0;\cP,\cT_{Y^\Delta})_j} \colon  \partial R^{\varrho(\sigma_0)}(\sigma_0;\cP,\cT_{Y^\Delta})_j \to \Sigma^p(\Delta^\square_{n-1})_j
\]
is BLD homotopic to a BLD-controlled expansion of a ${\partial R^{\varrho(\sigma_0)}(\sigma_0;\cP,\cT_{Y^\Delta})_j}^\Delta$-Alexander map. 

The result is quantitative in the sense that the bilipschitz constants of $h_j$ and $g_j$ depend only on the dimension $n$, subdivision index $\beta$, and the size of the tree $\cT_{Y^\Delta}$.
\end{lemma}

\begin{proof}
Let $\sigma'\neq \sigma_0$ be another leaf of $\cT_{Y^\Delta}$, and $\cR_{j,\sigma'}$ be the closure of $\cR^\varrho(Y^\Delta;\cP, \cT_{Y^\Delta})_j\setminus \cR^{\varrho(\sigma')}(\sigma';\cP,\cT_{Y^\Delta})_j$ in the unified Rickman pillow $\cR^\varrho(Y^\Delta;\cP, \cT_{Y^\Delta})_j$. Then $\cR_{j,\sigma'}\cap \cR^{\varrho(\sigma')}(\sigma';\cP,\cT_{Y^\Delta})_j$ is the $(n-1)$-cell $\xi_{\{\sigma',\sigma''\},j}$ previously fixed in the construction, where $\{\sigma',\sigma''\}\in \cT_{Y^\Delta}$.  
Let  $m_{\sigma'}$  be the number of $(n-1)$-simplices in $\cR^{\varrho(\sigma')}(\sigma';\cP,\cT_{Y^\Delta})_j\cap \Sigma^\varrho(Y^\Delta; \cP, \cT_{Y^\Delta})$, and $m_{\sigma''}$ be the number of $(n-1)$-simplices in $\cR^{\varrho(\sigma'')}(\sigma'';\cP,\cT_{Y^\Delta})_j\cap \Sigma^\varrho(Y^\Delta; \cP, \cT_{Y^\Delta})$.

We subdivide $\xi_{\{\sigma',\sigma''\},j}$ into $m_{\sigma'}$ $(n-1)$-simplices for which there exists a  bilipschitz contraction map 
\[h_{j,\sigma'} \colon \cR^\varrho(Y^\Delta;\cP, \cT_{Y^\Delta})_j \to \cR_{j,\sigma'}, \]
that is the identity map on $\cR_{j,\sigma'}\cap \Sigma^\varrho(Y^\Delta; \cP, \cT_{Y^\Delta})$ and is simplicial from $\cR^{\varrho(\sigma')}(\sigma';\cP,\cT_{Y^\Delta})_j \cap \Sigma^\varrho(Y^\Delta; \cP, \cT_{Y^\Delta})$ to $\xi_{\{\sigma',\sigma''\}}$.
 
 Since $\cP|_{\partial R^{\varrho(\sigma')}(\sigma';\cP)} \colon  \partial R^{\varrho(\sigma')}(\sigma';\cP) \to \Sigma^p(\Delta^\square_{n-1})_j$ is an Alexander map, the composition $\cP_{\cT_{Y^\Delta}} \circ h_{j,\sigma'}^{-1}|_{
 \cR_{j,\sigma'} \cap \Sigma^\varrho(Y^\Delta; \cP, \cT_{Y^\Delta})}$
 is an Alexander-Rickman map.

We finish this step by giving  $\partial \cR^{\varrho(\sigma'')}(\sigma'';\cP, \cT_{Y^\Delta})_j \cap \cR_{j,\sigma'}$  
another structure by attaching 
the $(n-1)$-structure $\xi_{\{\sigma',\sigma''\},j}$ to an $(n-1)$-simplex $\tau$ in $\cR^{\varrho(\sigma'')}(\sigma'';\cP, \cT_{Y^\Delta})_j \cap \Sigma^\varrho(Y^\Delta; \cP, \cT_{Y^\Delta})$ for which $\tau \cap \xi_{\{\sigma',\sigma''\},j}$ is an $(n-2)$-cell. 
More precisely, let $\tau'$ be the  $(n-1)$-simplex having space  $|\tau'|= |\tau \cup \xi_{\{\sigma',\sigma''\},j}|$ and  having all its $(n-2)$-faces, except one, the same as that of   $\tau$. Then $\tau \cup \xi_{\{\sigma',\sigma''\},j}$ is an expansion of the simplex $\tau'$ by $ \xi_{\{\sigma',\sigma''\},j}$. Associated to this local modification, $\partial \cR^{\varrho(\sigma'')}(\sigma'';\cP, \cT_{Y^\Delta})_j \cap (\cR_{j,\sigma'} \cup \xi_{\{\sigma',\sigma''\},j})$ is an expansion of an underlying complex having $m_{\sigma''}$ $(n-1)$-simplices by the complex $\xi_{\{\sigma',\sigma''\}}$ with $m_{\sigma'}$ simplices. 
  
By passing to an inner collar of $\cR_{j,\sigma}$ we may now homotope, with a level preserving BLD homotopy, the mapping $\cP_{\cT_{Y^\Delta}} \circ h_{j,\sigma'}^{-1}$ to a BLD-controlled expansion of an Alexander-Rickman map. Thus, we may assume that $h_{j,\sigma'}^{-1}$ is a bilipschitz map for which $\cP_{\cT_{Y^\Delta}} \circ h_{j,\sigma'}^{-1}$ is a BLD-controlled expansion of an Alexander map. Mappings may be chosen so that  the bilipschitz constant of $h_{j,\sigma'}$ depends only on $n$ and $\beta$.

We iterate now the step above along the tree $\cT_{Y^\Delta}$ to obtain a bilipschitz map $h_j$ as a composition of the mappings $h_{j,\sigma'}^{-1}$ above in such a way that the support, i.e., the union of simple covers, of the BLD-controlled expansion is mapped by a scaling at each step. The bilipschitz constant of $h_j$ depends only on $n$, $\beta$, and size of the tree $\cT_{Y^\Delta}$.

The construction of the last mapping $g_j$ for the Lemma follows the previous steps almost verbatim. This completes the proof.
\end{proof}

We set 
\[
\Omega^\varrho(Y^\Delta; \cP, \cT_{Y^\Delta}) =  \bigcup_{j=1}^p \cR^\varrho(Y^\Delta; \cP, \cT_{Y^\Delta})_j,
\] 
and view $\Omega^\varrho(Y^\Delta; \cP, \cT_{Y^\Delta})$ as an expansion of $\PC^\varrho(Y^\Delta;\cP)$. Similarly as for Rickman pillows, we may give $\Omega^\varrho(Y^\Delta; \cP, \cT_{Y^\Delta})$ a simplicial structure as a star of $n$-simplices. With respect to this structure, the map $\cP \circ \lambda_{\cT_{Y^\Delta}} \colon \Sigma^\varrho(Y^\Delta; \cP,\cT_{Y^\Delta}) \to \Sigma^p(\Delta^\square_{n-1})$ extends to a simplicial map 
\[
F_{\cT_{Y^\Delta}} \colon \Omega^\varrho(Y^\Delta; \cP, \cT_{Y^\Delta}) \to \PS^p(\Delta^\square_{n-1}).
\]

\subsection{Expansion of a complex by unified Rickman pillows}

We now glue components of $\Real(U;Y)$ together along  $\Omega^\varrho(Y; \cP, \cT_{Y^\Delta})$. This yields an expansion $\Exp^\varrho(U;Y; \cP, \cT_{Y^\Delta})$ of $U$ in the spirit of  the expansion $\Exp^\varrho(U;Y)$.

Let $G_{\cT_{Y^\Delta}} \colon \Upsilon_Y \to \partial \Omega^\varrho(Y; \cP, \cT_{Y^\Delta})$ be the map induced by the gluing map $G_{(U;Y)} \colon \Upsilon_Y \to \partial \PC^\varrho(Y^\Delta)\subset \Sigma^\varrho(Y;\cP)$, used in the definition of $\Exp^\varrho(U;Y)$, which satisfies $G_{(U;Y)} = \check \lambda_{\cT_{Y^\Delta}} \circ G_{\cT_{Y^\Delta}}$. 
We define
\index{$\Exp^\varrho(U;Y; \cP, \cT_{Y^\Delta})$}
\[
\Exp^\varrho(U;Y; \cP, \cT_{Y^\Delta}) = \Real(U;Y) \bigsqcup_{G_{\cT_{Y^\Delta}}} \Omega^\varrho(Y; \cP, \cT_{Y^\Delta}).
\]

We do not repeat the details similar to that in  the case  $\Exp^\varrho(U;Y)$. We merely emphasize that  $\Omega^\varrho(Y; \cP, \cT_{Y^\Delta})$ and $\Real(U;Y)\setminus |\Upsilon_Y|$ may be regarded as subspaces of  $\Exp^\varrho(U;Y; \cP, \cT_{Y^\Delta})$. In particular, the Alexander map $F_{\cT_{Y^\Delta}}\colon \Omega^\varrho(Y; \cP, \cT_{Y^\Delta}) \to \PS^p(\Delta^\square_{n-1})$ agrees with $\cP$ on the boundary of $\Omega^\varrho(Y; \cP, \cT_{Y^\Delta})$ though gluing. Thus we may extend $F_{\cT_{Y^\Delta}}$ over $\Real(U;Y)$ to a simplicial map 
\[
F_{\cT_{Y^\Delta}} \colon \Exp^\varrho(U;Y; \cP, \cT_{Y^\Delta}) \to \PS^p(\Delta^\square_{n-1})
\]
by defining $F_{\cT_{Y^\Delta}} = \cP$ on $\Real(U;Y)$.

Finally, we comment that $\Exp^\varrho(U;Y; \cP, \cT_{Y^\Delta})$ has a well-defined natural metric structure.

First, there is a length metric on the unified Rickman sheet  $\Sigma^\varrho(Y; \cP, \cT_{Y^\Delta})$ induced by the metrics on the individual $(n-1)$-simplices. Second, for each $j\in\{1,\ldots,p\}$, consider $\cR^{\varrho(\sigma)}(Y; \cP, \cT_{Y^\Delta})_j$ as a cone over $\partial (\cR^{\varrho(\sigma)}(Y; \cP, \cT_{Y^\Delta})_j)$ with respect to an arbitrary, but fixed, base point in the interior. We define in $\cR^{\varrho(\sigma)}(Y; \cP, \cT_{Y^\Delta})_j$ a conical metric with respect to the metric on the boundary $\partial \cR^{\varrho(\sigma)}(Y; \cP, \cT_{Y^\Delta})_j$. We may choose these metrics so that these $n$-cells are bilipschitz equivalent to one another other with a uniform constant depending only on $n$ and $\# Y^{[n-1]}$.

The metric on $\Omega^\varrho(Y; \cP, \cT_{Y^\Delta})$ is now the length metric induced by these conical metrics. The space $\Exp^\varrho(U;Y; \cP, \cT_{Y^\Delta})$ is given the length metric $d_{\Exp}$ induced by the length metrics on $\Real(U;Y)$ and  $\Omega^\varrho(Y; \cP, \cT_{Y^\Delta})$. Since we glue $\Omega^\varrho(U;Y,\cP,\cT_{Y^\Delta})$ along its boundary to $\Real(U;Y)$, we have, with respect to this metric, that the expansion $\Exp^\varrho(U;Y; \cP, \cT_{Y^\Delta})$ is bilipschitz to $|U|$ similarly as $\Exp^\varrho(U;Y;\cP)$ is bilipschitz to $|U|$. We record this observation as a lemma.

\begin{lemma}
\label{lemma:metric-Exp-2}
There exists $L'=L'(n, \varrho)\ge 1$ and an $L'$-bilipschitz homeomorphism 
\[
\widetilde h \colon \Exp^\varrho(U;Y; \cP, \cT_{Y^\Delta}) \to |U|
\]
for which $\widetilde h(\Omega^\varrho(Y; \cP, \cT_{Y^\Delta})) \subset |\Star_U(Y)|$ and $\widetilde h|_{|U|\setminus  |\Star_U(Y)|} = \id$.
\end{lemma}

\section{Simplicial subspace $\X$}

We are now ready to construct simplicial space $\X$. The construction is similar to that for Rickman pillows. In this case, we connect the preimage components of the open pillows in $\PS^p(\Delta^\square_{n-1})$, of the same color, under $F_{\cT_{Y^\Delta}}$.

Recall from Convention \ref{convention:coloring} that the coloring function $c\colon \{1,\ldots, r\} \to \{1,\ldots,p\}$ satisfies $c(j)=j, \quad \text{for} \,\,\, j=1,\ldots p$. 
 Thus, for  $j\in \{1,\ldots,p\}$, the component $\Real(U;Y)_{j}$ of  $\Real(U;Y)$  is mapped by $F_{\cT_{Y^\Delta}}$ to  the pillow $P(\Delta_{n-1})_j$. 

Fix  for each $j$, an $(n-1)$-simplex $\sigma_j$ in the inner boundary component $\Upsilon_{j}$ of $\Real(U;Y)_{ i_j}$ whose projection  $\pi(\sigma)$ is a leaf of  $\cT_{Y^\Delta}$.
Observe that tree $\cT_{Y^\Delta}$ could have been chosen so that leaves with this property exist.
Fix next an $(n-2)$-simplex $\eta_j $  contained in $\sigma_j$ whose projection in $(Y^\Delta)^{[n-2]}$ is not an edge of tree $\cT_{Y^\Delta}$. Simplices $\eta_1,\ldots, \eta_p$,  may be chosen for which their projections on $Y^\Delta$ are  mutually disjoint. We combine the unified Rickman pillow $\cR^{\varrho}(Y^\Delta; \cP, \cT_{Y^\Delta})_j$ with  $\Real(U;Y)_{j}$ by creating an opening at $\eta_j$ between
 $\cR^{\varrho(\sigma_j)}(\sigma_j; \cP, \cT_{Y^\Delta})_j$ and $\Real(U;Y)_{j}$. Since this procedure is formally the same as that in Section \ref{sec:unified-Rickman-pillows},  we give only an outline.

For an index $j\in \{1,\ldots, p\}$, we fix an $(n-2)$-simplex $\eta'_j$, having the same boundary as $\eta_j$, which together with $\eta_j$  bounds an $(n-1)$-cell $\xi_j$ in $\Real(U;Y)_{j}$.

We then connect $\cR^\varrho(Y^\Delta; \cP, \cT_{Y^\Delta})_j$ to $\Real(U;Y)_{j}$ by moving the unified Rickman sheet $\Sigma^\varrho(Y^\Delta;\cP,\cT_{Y^\Delta})$ locally along each cell $\xi_j$ to create openings as in the construction of Rickman pillows. We denote \index{$\X$}
\[
\X = \widetilde \Sigma^\varrho(Y^\Delta;\cP,\cT_{Y^\Delta})
\]
the simplicial subspace obtained after the move, and  $\widetilde \cR^\varrho(Y^\Delta; \cP, \cT_{Y^\Delta})_j, j=1,\ldots, p$, the modified unified Rickman pillows.

Clearly, $\X$ has $r$ complementary components in $\Exp^\varrho(Y; \cP, \cT_{Y^\Delta})$, namely, 
\[
\sfR^\varrho(U; Y; \cP, \cT_{Y^\Delta})_j =  \left\{
\begin{array}{lr}
|\Real(U;Y)_{j}| \cup \widetilde \cR^\varrho(Y; \cP, \cT_{Y^\Delta})_j  & j=1,\ldots,p,\\
| \Real(U;Y)_j|  & j=p+1,\ldots, m.
\end{array}\right.
\]

We fix a map $\widetilde \lambda \colon \widetilde \Sigma^\varrho(Y^\Delta;\cP,\cT_{Y^\Delta}) \to \Sigma^\varrho(Y^\Delta;\cP,\cT_{Y^\Delta})$ which, for each $j\in \{1,\ldots, p\}$, closes the openings $\xi_j$, that is, $\widetilde \lambda$ is bilipschitz with respect to  the inner metric of the complement of spheres $\eta_j\cup \eta'_j$ and maps each $\eta'_j$ bilipschitzly to $\eta_j$, respectively.
Since the cell $\xi_j$ is bilipschitz equivalent to the Euclidean $(n-1)$-ball and $\cR^\varrho(Y^\Delta; \cP, \cT_{Y^\Delta})_j$ is bilipschitz equivalent to a Euclidean ball, we conclude that, for each $j\in \{1,\ldots, p\}$, there exists a Lipschitz map
\[
\bar \psi_j \colon \Real(U;Y)_{j} \to \sfR^\varrho(U;Y; \cP, \cT_{Y^\Delta})_j
\]
which is bilipschitz with respect to the inner metrics when restricted to the interiors of the spaces, the identity on $\partial \Real(U;U)_{j}\setminus \interior\;\xi_j$, and a bilipschitz embedding on $\xi_j$.

Since $\widetilde \lambda$ is a simplicial isomorphism from $\X$ to $\Sigma^\varrho(Y^\Delta;\cP,\cT_{Y^\Delta})$, we conclude that the map
\[
\varphi = F_{\cT_{Y^\Delta}} \circ \widetilde \lambda \colon \X \to \Sigma^p(\Delta^\square_{n-1})
\]
is a Alexander--Rickman map.

We upgrade now Lemma \ref{lemma:contraction-to-leaf} to a contraction statement for complementary components of $\X$. Recall that $\Upsilon_i$ is the inner boundary component of $\Real(U;Y)_i$.

\begin{proposition}
\label{prop:contraction-weaving}
For each $j\in \{1,\ldots, p\}$ there exists a weaving map 
\[\bar \psi_j \colon |\Real(U;Y)_{j}| \to \sfR^\varrho(U; Y; \cP, \cT_{Y^\Delta})_j\]
 for which  $\bar \psi_j(|\Upsilon_{j}|) = \X \,\cap \, \sfR^\varrho(U;Y; \cP, \cT_{Y^\Delta})_j$,  
and   $\varphi \, \circ \, \bar \psi_j|_{|\Upsilon_{j}|} \colon  |\Upsilon_{j}| \to \Sigma^p(\Delta^\square_{n-1})_j$ is a BLD-controlled expansion of an $\Upsilon_{j}^\Delta$-Alexander map.
\end{proposition}

\begin{remark}\label{rmk:contraction-weaving}
For $j\in\{p+1,\ldots, m\}$, Proposition \ref{prop:contraction-weaving} holds with $\sfR^\varrho(U; Y; \cP, \cT_{Y^\Delta})_j =|\Real(U;Y)_{j}|$ 
and $\bar \psi_j$ the identity map.
\end{remark}

\begin{proof}
We fix first $j\in \{1,\ldots, p\}$. Let $\widetilde \cR_j$ be the closure of the interior of $\widetilde \cR^\varrho(Y; \cP, \cT_{Y^\Delta})_j$ in the inner metric and let 
\[\lambda_j \colon \widetilde \cR_j\cup |\Real(U;Y)_{j}| \to \widetilde \cR^\varrho(Y; \cP, \cT_{Y^\Delta})_j\cup |\Real(U;Y)_{j}|\]
 be the natural identification map obtained as the completion of the identity of the interior. We denote 
 \[\widetilde \X_j = \partial \widetilde \cR_j \cup |\Upsilon_{j}| \setminus \interior \xi_j=\partial (\widetilde \cR_j \cup |\Real(U;Y)_{j}|)\setminus \Sigma_i,\] 
 and $\widetilde \varphi_j = \varphi \circ \lambda_j \colon \widetilde \X_j \to \Sigma^p(\Delta^\square_{n-1})_j$.

Since the change of simplicial structure from $\partial \cR^{\varrho(\sigma_j)}(\sigma_j; \cP, \cT_{Y^\Delta})_j$ to $\partial \widetilde \cR^{\varrho(\sigma_j)}(\sigma_j; \cP, \cT_{Y^\Delta})_j$ has no role in the proof of Lemma \ref{lemma:contraction-to-leaf}, we may first apply Lemma \ref{lemma:contraction-to-leaf} to contract $\widetilde \cR_j$ to a copy of $\widetilde \cR^{\varrho(\sigma_j)}(\sigma_j; \cP, \cT_{Y^\Delta})_j$ in $\widetilde \cR_j$,  then use the same argument to contract this copy of  $\widetilde \cR^{\varrho(\sigma_j)}(\sigma_j; \cP, \cT_{Y^\Delta})_j$ into $|\Real(U;Y)_{j}|$. Now a simple adaptation of the argument of  Lemma \ref{lemma:contraction-to-leaf} yields a bilipschitz homeomorphism $h_j \colon |\Real(U;Y)_{j})|\to \widetilde \cR_j \cup |\Real(U;Y)_{j}|$ for which $h_j(|\Upsilon_{j}^\Delta|) = \widetilde X_j$
 and $\widetilde \varphi_j \circ h_j \colon |\Upsilon_{j}^\Delta|\to \Sigma^p(\Delta^\square_{n-1})_j$ is a BLD-controlled expansion of an $\Upsilon_{j}^\Delta$-Alexander map. 

Now the composition $\bar \psi_j = \lambda_j \circ h_j \colon |\Real(U;Y)_{j}| \to \sfR^\varrho(U;Y;\cP,\cT_{Y^\Delta})_j$ is a weaving map having the property that 
\[
\varphi \circ \bar \psi_j = \widetilde \varphi_j \circ h_j \colon |\Upsilon_{j}^\Delta|\to \Sigma^p(\Delta^\square_{n-1})_j
\]
is a BLD-controlled expansion of an $\Upsilon_{j}^\Delta$-Alexander map.
\end{proof}

\chapter{Proof of the Weaving theorem}
\label{chap:proof-of-combination-I}

We now give the proof of Theorem \ref{thm:combination-I}. We name this short section  a chapter because of its role in the final part of the proof of the Quasiregular cobordism theorem.

We begin by recalling the statement of the Weaving theorem.

\Quasiregularcombinationtheorem*

\begin{proof}
We begin by introducing an essential partition of $Z$ into subcomplexes of uniformly bounded sizes. More precisely, since $Z$ is a obtained by a localized channeling transformation of a separating complex $Z'$ in $\Refine^{\nu(k-1)}(K)$, that is, $Z = \Channel_{\Refine^{\nu k}(K)}(Z';\fL_1)$, we may fix a family $\{(U_q,Y_q)\}_{q \in (Z')^{[n-1]}}$ of pairs where $U_q \subset \Refine^{\nu k}(K)$ is a good complex, $Y_q \subset U_q$, and $\{Y_q\}_{q\in (Z')^{[n-1]}}$ is an essential partition of $Z$ into adjacently connected subcomplexes for which $\max \{\# Y_q^{[n-1]}\colon q\in (Z')^{[n-1]}\}$ depends only on $n$ and $\nu$, hence only on $n$ and $K$. See Sections \ref{sec:Z-localization} and \ref{sec:construction-evolution}.

For each $q\in (Z')^{[n-1]}$, we apply Theorem \ref{thm:pillow-adjustment} to obtain a rank function $\varrho_q \colon (Y^\Delta_q)^{[n-1]} \to \Z$ which admits a pillow expansion $\Exp^{\varrho_q}(U_q,Y_q)$ and pillow map $\cP_q \colon \PC^{\varrho_q}(Y_q^\Delta) \to \PS^p(\Delta^\square_{n-1})$.

We fix, for each $Y^\Delta_q$, a spanning tree $\cT_q$ and pass, as in Section \ref{sec:unified-Rickman-pillows}, from $\Exp^{\varrho_q}(U_q,Y_q)$ and $\cP_q$ via $\cT_q$ to space $\Exp^{\varrho_q}(U_q; Y_q; \cP_q; \cT_q)$ and map $F_{\cT_q} \colon \Exp^{\varrho_q}(U_q; Y_q; \cP_q; \cT_q) \to \PS^p(\Delta^\square_{n-1})$. Let also $\X_q=\widetilde \Sigma^{\varrho_q}(Y_q^\Delta; \cP_q, \cT_{Y^\Delta_q})$ be a simplicial set admitting an Alexander--Rickman map $\varphi_q \colon \X_q \to \Sigma^p(\Delta^\square_{n-1})$ and, for $j=1,\ldots, p$, $\sfR^{\varrho_q}(U_q;Y^\Delta_q; \cP,\cT_q)_j$ complementary components of $\X_q$. For each $j\in \{1,\ldots, p\}$, let $\bar \psi_{q,j} \colon |\Real(U_q;Y_q)_{j}|\to \sfR^{\varrho_q}(U_q;Y^\Delta_q; \cP_q, \cT_q)_j$ be a weaving map as in Proposition \ref{prop:contraction-weaving}. 

Since $\{ Y^\Delta_q\}_q$ is a partition of $Y^\Delta$, we may now define a rank function $\varrho \colon (Y^\Delta)^{[n-1]}\to \Z_+$ by $\varrho|_{(Y^\Delta_q)^{[n-1]}} = \varrho_q$, and set $\sfR^\varrho = \bigcup_q\sfR^{\varrho_q}(U_q;Y_q^\Delta; \cP_q, \cT_q)$. 

Since the space $\sfR^{\varrho_q}(U_q;Y^\Delta_q; \cP_q, \cT_q)$ expands the subcomplex $U_q$ and agrees with $U_q$ on $\partial U_q$,  we obtain a composite expansion $\EXP^\varrho(\Refine^{\nu k}(K); Z)$ of $\Refine^{\nu k}(K)$ induced by $\sfR^\varrho$.

Let $\X = \bigcup_q \X_q \subset \EXP^\varrho(\Refine^{\nu k}(K);Z)$. The mappings $\bar \psi_{q,j}$ combine now to weaving maps $\bar \psi_{\Sigma,\X} \colon |\Real_{\Refine^{\nu k}(K)}(Z;\Sigma)| \to \bar \beta_\X(\Sigma)$ for the boundary components $\Sigma$, and mappings $\varphi_q$ to an Alexander--Rickman map $\varphi_\X \colon \X \to \Sigma^p(\Delta^\square_{n-1})$.

By Lemma \ref{lemma:metric-Exp-2}, each expanded space $\Exp^{\varrho_q}(U_q; Y_q; \cP_q; \cT_q)$ is bilipschitz equivalent to $|U_q|$. Thus $\EXP^\varrho(\Refine^{\nu k}(K);Z)$ is bilipschitz equivalent to $|\Refine^{\nu k}(K)|$. Let $\theta \colon |\Refine^{\nu k}(K)|\to \EXP^\varrho(\Refine^{\nu k}(K); Z)$ be one such bilipschitz homeomorphism. 

Take $X = \theta^{-1}\X$ and set $\bar \psi_\Sigma = \theta^{-1} \circ \bar \psi_{\Sigma,\X} \colon  |\Real_{\Refine^{\nu k}(K)}(Z;\Sigma)| \to \bar \beta_X(\Sigma)$. 
Then $\bar \psi_\Sigma$ is a weaving map for which  
\[\varphi_\X \circ \theta \colon \bar\psi_\Sigma|_{|\Upsilon(Z;\Sigma)|} \colon |\Upsilon(Z;\Sigma)|\to \Sigma^p(\Delta^\square_{n-1})\] is a BLD-controlled expansion of a $\Upsilon(Z;\Sigma)^\Delta$-Alexander map. Set $\varphi= \varphi_\X \circ \theta$. This concludes the proof.
\end{proof}


\part{Applications}
\label{part:Applications}

In this part, we discuss applications of  quasiregular deformation theorems and the Quasiregular cobordism theorem (Theorem \ref{intro-thm:qrcobordism}). Before studying higher dimensional versions of Rickman's large local index theorem and the Heinonen--Rickman theorem on wild branching, we discuss two immediate applications.

\section{Quasiregular mappings with prescribed preimages}
\label{sec:prescribed_preimages}

As discussed in the introduction, an immediate consequence of Theorem \ref{intro-thm:qrcobordism} is an observation that, for $n\ge 3$, there exist quasiregular maps $\bS^n \to \bS^n$ with arbitrarily assigned preimages. The corresponding statement is false in dimension $n=2$ as Remark \ref{rmk:preimage-failure-dim2} shows.

\introthmpreimages*


\begin{proof}
Let $B_1,\ldots, B_p\, (p\geq 2)$ be mutually disjoint $n$-cells containing points $z_1,\ldots, z_p$ in their interiors, respectively. We fix, for each $x \in \bigcup_i Z_i$,
a PL $n$-cell  $G_x \subset \bS^n$ with $x\in \interior G_x$ so that cells in $\{G_x\colon x \in \bigcup_i Z_i\}$ are pairwise disjoint.
For each $i=1,\ldots, p$, let $G_i = \bigcup_{x\in Z_i} G_x$

By Theorem \ref{intro-thm:qrcobordism}, there exists a quasiregular map
\[
f \colon \bS^n\setminus \interior \bigcup_{i=1}^p G_i \to \bS^n\setminus \interior \bigcup_{i=1}^p B_i,
\]
 which maps $\partial G_x$ onto $\partial B_i$ for each $x\in Z_i$.  We extend $f$ to a map $f\colon \bS^n\to \bS^n$ by taking the cone extension $G_x \to B_i$ from  $f|_{\partial G_x} \colon \partial G_x \to \partial B_i$, for each $x\in Z_i$ and $i=1,\ldots, p$. The extension is also quasiregular.  
 
The assertion in the case $p=1$ is a direct consequence of that in $p=2$.
\end{proof}

\begin{remark}
By Picard constructions (\cite{Rickman_Acta}, \cite{Drasin-Pankka}), we may assign empty pre-images at finitely many points in $\bS^n$ for  quasiregular maps $\R^n \to \bS^n$. Theorem \ref{intro-thm:QR_preimages} complements this result by showing that we may also prescribe  finite pre-images at finitely many points for  quasiregular maps $\bS^n \to \bS^n$.
\end{remark}

\begin{remark}
\label{rmk:preimage-failure-dim2}
Theorem \ref{intro-thm:QR_preimages} does not hold in dimension two, which we may see as follows.

There is no entire function $f\colon \R^2\to \R^2$ whose preimages of three distinct points $z_1, z_2$, and $z_3$, consist of one, one, and two points, respectively. Indeed, if there were an entire $f$ with this property then,  by the Big Picard Theorem, $f$ would be a polynomial of degree $p\geq 2$. Then the derivative $f'$ would have at least $3p-4$ zeros counting multiplicities; hence $p-1=\deg f'\geq 3p-4$, which is impossible.

In view of the Stoilow factorization theorem \cite[Theorem 5.5.1]{Astala-Iwaniec-Martin-book}, every quasiregular map $ \R^2\to \R^2$ is composition $h\circ \varphi$ of a quasiconformal homeomorphism $\varphi\colon \R^2\to \R^2$ and an entire function $h\colon \R^2\to \R^2$. Therefore,
there is no quasiregular map $f\colon \bS^2\to \bS^2$ whose preimages of four distinct points $z_1, z_2,z_3$, and $z_4$ consist of one, one, one, and two points, respectively.
\end{remark}

\section{Radial limits of bounded quasiregular mappings}

A classical theorem of Fatou states that bounded analytic functions on the unit disk in $\R^2$ have radial limits almost everywhere on the unit circle. In comparison, the picture for boundary behavior of bounded spatial quasiregular maps in higher dimensions is  incomplete.

There are several growth conditions under which a spatial quasiregular map has radial limits; see, for example, Martio and Rickman  \cite{Martio-Rickman_boundary} and Rajala \cite{Rajala_Radial}. As for the non-existence, Martio and Srebro \cite{Martio-Srebro_locally_injective}, and Heinonen and Rickman \cite[Section 9.2]{Heinonen-Rickman_Duke} have constructed bounded spatial quasiregular mappings on the unit ball for which the radial limit does not exist on a set of Hausdorff dimension arbitrarily close to the dimension of the boundary.

The proof of the $3$-dimensional theorem of Heinonen and Rickman \cite[Section 9.2]{Heinonen-Rickman_Duke} applies almost verbatim to all dimensions $n\geq 3$, provided that the  theorem of Berstein and Edmonds used in their argument  is replaced by Theorem \ref{intro-thm:cobordism_short}.
For this reason, we merely state the result.

\begin{corollary}
\label{cor:radial}
Let $n\geq 3$, and $\Gamma$ be a geometrically finite torsion free Kleinian group without parabolic elements acting on the $(n-1)$-sphere $\bS^{n-1}$ whose limit set $\Lambda_\Gamma$ is not the whole sphere. Then there exists  a bounded quasiregular map $f\colon B^n\to B^n$ such that $f $ has no radial limit at points in $\Lambda_\Gamma$.
\end{corollary}

\section{Rickman's large local index theorem}

In this section we prove a version of Rickman's large local index theorem in \cite{Rickman-AASF} for quasiregular mappings $\bS^n\to \bS^n$ in dimensions $n\ge 3$.\index{Large local index theorem}

\introthmRickmanAASF*

Here $\deg(F)$ is the degree of the map $F$ and $i(x,F)$ the \emph{local index of $F$ at the point $x\in \bS^n$}. We refer to \cite{Rickman_book} for the formal definition of the local index and recall that $i(x,F)$ has the property that, each $x\in \bS^n$ has a neighborhood $U$ for which
\[
i(x,F) = \max_{y} \#(V \cap F^{-1}(y))
\]
for each neighborhood $V$ of $x$ contained in $U$; see \cite[Section I.4.2]{Rickman_book}.

The Cantor set $E_F$ in Theorem \ref{intro-thm:Rickman-AASF} is tame, i.e., there is a homeomorphism $h\colon \bS^n\to \bS^n$ for which $h(E_F)$ is the standard ternary Cantor set contained in $\bS^1\subset \bS^n$. In Section \ref{sec:Heinonen-Rickman}), we construct a similar example in dimension $n=4$ where the large local index set is a wild Cantor set. For the definition of wild Cantor sets, see Definition \ref{def:wild}.

The outline of the proof of Theorem \ref{intro-thm:Rickman-AASF} follows the $3$-dimensional theorem of Rickman in \cite{Rickman-AASF}.
At the core of the construction, however, is the Quasiregular cobordism theorem (Theorem \ref{intro-thm:qrcobordism}).

The essential part of  Theorem \ref{intro-thm:Rickman-AASF} is contained in the following proposition. Having this proposition at our disposal, a standard Schottky group argument yields Rickman's large local index theorem in all dimensions $n\geq 3$.

\begin{proposition}
\label{prop:Rickman-AASF}
Let $n\ge 3$ and $p\ge 2$. Then there exists a constant $\sK=\sK(n,p)\ge 1$ for the following: for each $c>0$, there exist pairwise disjoint balls $B^n(x_1,r), \ldots, B^n(x_p,r)$ in $\bS^n$  and a $\sK$-quasiregular mapping $f \colon \bS^n \to \bS^n$ for which $i(x_1,f) = \cdots = i(x_p,f) \ge c$ and which, for each $j=1,\ldots, p$, satisfies $f^{-1}f(B^n(x_j,r)) = B^n(x_j,r)$.
\end{proposition}

\begin{proof}[Proof of Theorem \ref{intro-thm:Rickman-AASF}]
Let $f\colon \bS^n \to \bS^n$ be a quasiregular mapping as in Proposition \ref{prop:Rickman-AASF} and let, for each $j=1,\ldots, p$, $\gamma_j\colon \bS^n \to \bS^n$ be the reflection with respect to the sphere $\partial B^n(x_j,r)$ in $\bS^n$, respectively. Let $\Gamma$ be the (Schottky) groups generated by the set $\{\gamma_1,\ldots, \gamma_p\}$. For each word $w=i_1\cdots i_k$ with letters $i_j \in \{1,\ldots, p\}$, let $\gamma_w = \gamma_{i_1}\circ \cdots \circ \gamma_{i_k} \in \Gamma$.

Let $M_0 = \bS^n\setminus \bigcup_{j=1}^p B^n(x_j,r)$ and, for each $k\in \Z_+$, let $M_k = M_{k-1} \cup \bigcup_{i=1}^p \gamma_i M_{k-1}$. Let now $E_f$ be the Cantor set $E_f = \bigcap_{k=1}^\infty (\bS^n \setminus M_k)$.  Then there exists a unique map $F\colon \bS^n \to \bS^n$ for which $F|_{M_0} = f$ and, for each word $w=i_1\cdots i_k$, we have $F\circ \gamma_w|_{M_0} = \gamma_w \circ F|_{M_0}$. The mapping $F$ is $\sK$-quasiregular and satisfies the conclusions of the claim; we refer to \cite{Rickman-AASF} for details.
\end{proof}

\begin{proof}[Proof of Proposition \ref{prop:Rickman-AASF}]

We reduce the proof to the construction of a separating complex $Z$, and an application of the Cubical quasiregular cobordism theorem (Theorem \ref{intro-thm:qrcobordismII}).

Let $c>0$ and $p\ge 2$. We construct now a quasiregular mapping $f \colon \bS^n \to \bS^n$ which has local index $i(\cdot,f)\ge c$ at $x_1,\ldots, x_p\in \bS^n$.

\medskip
\emph{Step 1: Initial configuration.} 
Let $Q$ be an $n$-cube, $k\in \N$ with $3^{k} \ge p+1$, and $\Refine^{k}(Q)$ the $k$-th refinement of $Q$. Let  $S$ be the topological sphere obtained by gluing two copies of $Q$ together along the boundary. In other words, $S$ is the quotient space $(Q\times \{1,2\})/{\sim}$,  where $\sim$ is the smallest equivalence relation satisfying $(x,1)\sim (x,2)$ for all $x\in \partial Q$.

Let  $K$ be the natural double of $\Refine^{k}(Q)$ in $S$, of which $\Refine^{k}(Q)$ is  a subcomplex of $K$. Let also $q\subset Q$ be a face of $Q$, and $Q_1,\ldots, Q_{p-1}$ be adjacent $n$-cubes in $\Refine^{k}(Q)$ such that each $Q_j$ has a face $q_j$ contained in $q$ and their union $\bigcup_{i=1}^{p-1} Q_j$, in the metric $d_{\Refine^{k}(Q)}$, is isometric to a Euclidean cube $[0,1]^{(n-1)}\times [0,p-1]$.

For each $j=1,\ldots, p-1$, let $Q'_j$ be the unique $n$-cube in $K\setminus \Refine^{k}(Q)$  for which $q_j=Q_j \cap Q'_j$ is a common face of $Q_j$ and $Q'_j$. Also, for  each $j=1,\ldots, p-1$, let $Q''_j=Q_j$ if $j$ is odd, and $Q''_j=Q'_j$ if $j$ is even, and let $\widetilde Q_j$ be the center cube of $Q''_j$ in the refinement $\Refine (K)$. We fix an $n$-cube $\widetilde Q_p \in \Refine(K)$  which meets neither $q$ nor any one of the cubes $Q''_1,\ldots, Q''_{p-1}$.

Let now $\widetilde K= \Refine(K) \setminus \bigcup_{i=1}^p \widetilde Q_j$ be the subcomplex of $\Refine(K)$ with $n$-cubes $\widetilde Q_1,\ldots, \widetilde Q_p$ removed, and  $Z$ be the subcomplex  of $ \widetilde K$ for which
\[
|Z| =  \bigcup_{i=1}^{p-1} \partial Q''_j .
\]
Then $Z$ is a separating complex of $\widetilde K$.

\medskip
\emph{Step 2: Application of the Quasiregular cobordism theorem to  $\widetilde K$.}
In view of Theorem \ref{intro-thm:qrcobordism},  there exist a refinement index
$k\in \N$, satisfying $3^k\geq p+1$ and $3^{k(n-1)}\geq c$,  a constant $\sK'=\sK'(n,p)$, and
a $\sK'$-quasiregular map $\widetilde f\colon |\widetilde K|\to \bS^n \setminus \left( B_1\cup \cdots \cup B_p\right)$ such that every restriction $\widetilde f|_{\partial \widetilde Q_j} \colon \partial \widetilde Q_j \to \partial B_j$ is a  $\Refine^k(\widetilde K)^\Delta$-Alexander map expanded by properly placed standard simple covers, where $B_1,\ldots, B_p$ are pairwise disjoint Euclidean balls, and  $\sK'\ge 1$ is a constant depending only on the dimension $n$ and the number  $p$.

Since the sphere $S= |K|$, in the metric $d_K$, is bilipschitz to the Euclidean sphere $\bS^n$, the sphere $S$ in the metric $d_{\Refine^k(K)}$ is $(3^k, L)$-quasi-similar to $\bS^n$ for a constant $L=L(n,p)$ independent of $k$.
In particular, there exist pairwise disjoint Euclidean balls $B'_j=B^n(x_j,r), j=1,\ldots,p,$ of the same size, and a quasiconformal homeomorphism
\[
F \colon (|\widetilde K|, d_{\widetilde K}) \to \bS^n \setminus \interior \left( B'_1\cup \cdots \cup B'_p\right),
\]
with a distortion constant depending only on $n$ and $p$.

\medskip

\noindent \emph{Step 3: Extension over $\bS^n$.} It remains to extend the mapping
\[f = \widetilde f \circ F^{-1} \colon \bS^n\setminus (B'_1 \cup \cdots \cup B'_p) \to \bS^n \setminus (B_1 \cup \cdots \cup B_p)\]
 to a quasiregular map $\bS^n \to \bS^n$. For simplicity, we assume that all balls are contained in $\R^n \subset \bS^n$. We may also assume that,  for each $j=1,\ldots, p$, the mapping $f|_{\partial B'_j} \colon \partial B'_j \to \partial B_j$ is an Alexander map expanded by properly placed standard simple covers.

Then,  with $\beta_j= \deg F|_{\partial \widetilde Q_j}$, the radial extensions $B'_j \to B_j$
\[
x\mapsto \left(\frac{|x-x_j|}{r}\right)^{\beta_j} \left( f\left(r\frac{x-x_j}{|x-x_j|} + x_j\right)-x_j \right)+x_j,
\]
of $f$ into $B'_j$, yield the desired $\sK$-quasiregular extension $\bS^n \to \bS^n$ with $\sK = \sK(\sK',n)$; see \cite[Example I.3.2]{Rickman_book} for the model case. Note that the degree $\deg F|_{\partial \widetilde Q_j}$ is independent on $j$ and
in fact is equal to the degree of $f$.

By the choice of refinement index $k$, the local index of $f$ satisfies $i(x_j,f)\ge 3^{k(n-1)}\geq c$ for each $j=1,\ldots,p$. This completes the proof.
\end{proof}

\section{Wildly branching quasiregular maps in dimension $4$}
\label{sec:Heinonen-Rickman}

In this section, we prove a $4$-dimensional version of a theorem of Heinonen and Rickman on wild branching of quasiregular maps \cite[Theorem 2.1]{Heinonen-Rickman_Topology} and discuss its applications to the Jacobian problem and to the Fatou problem. The main theorem reads as follows.\index{Wildly branching quasiregular map}

\introthmHeinonenRickmanTopologyfourdim*

\begin{definition}\label{def:wild}
A Cantor set $X\subset \bS^n$ is \emph{tame}, if there is a homeomorphism $h \colon \bS^n \to \bS^n$ for which $h(X) \subset \bS^1 \subset \bS^n$. It is \emph{wild} if there exists no such homeomorphism. \index{Cantor set!tame} \index{Cantor set!wild}
\end{definition}

There are wild Cantor sets in all dimensions. The  classical construction of Antoine \cite{Antoine} gives an example, the so-called \emph{Antoine's necklace}, in dimension three. The (topological) construction of Antoine's necklace was generalized to dimensions $n \ge 4$ by Blankinship \cite{Blankinship}. In the proof of Theorem \ref{intro-thm:Heinonen-Rickman_Topology_fourdim}, we use a quasi-self-similar version of  Blankinship's wild Cantor set; this (geometrical) construction in dimension $4$ is discussed in the appendix. The dimension restriction in Theorem \ref{intro-thm:Heinonen-Rickman_Topology_fourdim} stems from this dimension constraint. In fact, we are not aware of the existence of quasi-self-similar Cantor sets in dimensions $n\ge 5$.

We mention in passing, that the $3$-dimensional result of Heinonen and Rickman in \cite{Heinonen-Rickman_Topology} is    based on a self-similar Antoine's necklace; see also Semmes \cite{Semmes}.

\subsection{Jacobians of quasiregular maps}

Theorem \ref{intro-thm:Heinonen-Rickman_Topology_fourdim} yields also an example of a quasiregular map $f\colon \bS^4\to \bS^4$ for which the Jacobian $J_f$ is not comparable to the Jacobian of any quasiconformal map $\bS^4\to \bS^4$. The argument follows almost verbatim from that of Heinonen and Rickman for dimension $3$ in \cite{Heinonen-Rickman_Topology}.

\begin{corollary}
There exists  a quasiregular map $f\colon \bS^4\to \bS^4$ for which the Jacobian $J_f$ of the map is not comparable to the Jacobian of any quasiconformal map $\bS^4\to \bS^4$.
\end{corollary}

\subsection{Proof of Theorem \ref{intro-thm:Heinonen-Rickman_Topology_fourdim}}

We prove the theorem by reducing the question to an application of the Cubical quasiregular cobordism theorem (Theorem \ref{intro-thm:qrcobordismII}) and an iterative process as in the proof of Theorem \ref{intro-thm:Rickman-AASF}. We begin by describing the construction of the Cantor set and the corresponding separating complex.

As described in more detail in the appendix, Blankinship's Cantor set $X$ is obtained as the intersection
\[
X = \bigcap_{k=0}^\infty X_k,
\]
where
$X_0=T$ is homeomorphic to $B^2 \times \bS^1\times \bS^1$ and for each $k\ge 1$,
$X_k$ is a compact manifold with boundary, consisting of $m^k$ components $T_w$ each of which is homeomorphic to $B^2 \times \bS^1\times \bS^1$, where $m \in \N$ is a fixed even integer and $w$ is a word of length $k$ in letters $\{1,\ldots, m\}$.
For each $k\ge 0$, the pair $(X_k, X_{k+1})$ is a union of  mutually disjoint pairs $(T_w, T_w \cap X_{k+1})$ each of which is homeomorphic to the pair $\left((B^2 \times \bS^1)\times \bS^1, \mathcal A\times \bS^1\right)$, where  $\mathcal A$ is a union of $m$ mutually disjoint  solid $3$-tori $A_1,\ldots, A_m$  linked in $B^2 \times \bS^1$ as in the construction of Antoine's necklace in $\R^3$.

From now on, we assume that $m$ is large and that the solid tori $A_1,\ldots, A_m$ are isometric to each other and are similar to an embedded solid $3$-torus $A$ for which $X_0 = A\times \bS^1$. Set $A_0=A$.  We may further assume that the pairs $(T_w, T_w\cap X_{k+1})$ belong to two similarity classes for all  words $w$ of length $k$ and all $k\ge 1$; see the appendix for more discussion.

The main part of the proof is to define a suitable cubical complex $K$ on $\left( A \times \bS^1, \mathcal A\times \bS^1 \right)$ which admits a separating complex $Z$. After that, we apply Theorem \ref{intro-thm:qrcobordism} to construct a quasiregular branched covering map  $f$ on $(A \times \bS^1)\setminus \interior ( \mathcal A\times \bS^1)$ associated
 to $(K,Z)$, and then  appeal to the quasi-similarity in the construction to obtain the map $\bS^4 \to \bS^4$ claimed.

\medskip
\noindent \emph{Step 1: Cubical $2$-complexes $\{R_i\colon 0\leq i\leq m\}$ on $2$-tori $\{\partial A_i\colon 0\leq i\leq m\}$.} We set
\[
A_0 = A = B^2\times \bS^1 \subset \R^3.
\]
Given $\ell\ge 2$, let $C_\ell$ be a cubical complex on $\bS^1$ having $\ell$\, $1$-simplices of equal length, e.g.\;with $1$-simplices
\[
\sigma_{k,\ell} = \{ e^{i\theta}\in \bS^1 \colon \theta \in [2\pi (k-1)/\ell, 2\pi k]\},\,\,\,\, \text{for}\,\, k=1,\ldots, \ell.
\]

Let now $R_0$ be a cubical $2$-complex on $\partial A = \partial B^2 \times \bS^1$ isomorphic to $C_4  \times C_{m/2}$, to be fixed more precisely later. This structure subdivides the longitudinal direction of $\partial A$ into $m/2$ equal parts. For each $k=1,\ldots, m/2$, we set
\[
G_k = B^2 \times \sigma_{k,m/2}
\]
to be a $3$-cell.

We also identify (topologically) each solid $3$-torus $A_i$  with $B^2\times \bS^1$. For each even index $i\in \{1,\ldots, m\}$, let $R_i$ be a cubical $2$-complex  on $\partial A_i$  isomorphic to $C_2\times C_2$, to be fixed later. We consider those $A_i$ with even indices as the \emph{vertical rings}. For an odd index $i$, let $R_i$ be a $2$-complex on $\partial A_i$ isomorphic to $C_2\times C_4$. These rings are considered as \emph{horizontal rings}.

Observe that, for each $i=0,\ldots, m$, the cubical complex on the surface $\partial A_i$ is isomorphic to a shellable cubical refinement of the  model complex $C_2\times C_2$.

As a preparation for the next step, we arrange the solid $3$-tori $A_1,\ldots, A_m$ in such a way that each even-indexed torus $A_{2k}$ is contained in $G_k$ and the odd-indexed tori $A_{1+ 2k}$ are symmetric with respect to the $2$-disk $G_k \cap G_{k+1}$ for $k=1,\ldots, m/2$, where we identify $G_{m/2 +1} = G_1$.

\medskip
\emph{Step 2: Cubical $3$-complexes $\{P_i\}$ on the collars $\{M_i\}$ for $\{\partial A_i\}$, and cubical $2$-complexes $\{Y_i\}$ on the filling disks $\{D'_i\}$.} We fix  an inner collar $M_0$ for $\partial A$ in $A$ and mutually disjoint outer collars $M_1,\ldots, M_m$ for $A_1,\ldots, A_m$, respectively, in such a way that  $A'_0= A \setminus \interior M_0 = B^2(1-\varepsilon) \times \bS^1$ for some $\varepsilon>0$ and that each solid $3$-torus $A'_i = A_i \cup M_i$ is contained in the interior of $A'_0$. Note that configuration $(A'_0,A'_1,\ldots, A'_m)$ is homeomorphic to $(A,A_1,\ldots, A_m)$.

Since collars $M_0$ and $M_1,\ldots, M_m$ are product spaces, we may associate to each collar $M_i$,  for $i=0,\ldots, m$, a natural product structure $P_i$ isomorphic to $R_i \times [0,1]$ for which $P_i|_{\partial A_i} = R_i$. In particular, $P_i|_{\partial A'_i}$ is also isomorphic to $R_i$.

By making the tori $A'_1,\ldots, A'_m$ uniformly bilipschitz to solid tori $B^2(r)\times \bS^1(t)$ for some parameters $r$ and $t$ if necessary and rearranging their placement, we may  fix $2$-cells $D_1,\ldots, D_m$, uniformly bilipschitz to the Euclidean disk $B^2(t-2r)$, with the properties that
\begin{enumerate}
\item $D_i \cap \partial A'_i = \partial D_i$, and
\item intersections $D_i \cap A'_{i-1}$ and $D_i \cap A'_{i+1}$ are $2$-disks in the interior of $D_i$, and $D_i \cap A'_j= \emptyset$ for $j\ne i-1,i,i+1$ modulo $m$; and
\item $A'_0\setminus (A'_1\cup \cdots \cup A'_m\cup D_1\cup\cdots \cup D_m)$ is homeomorphic to $\partial A'_0 \times [0,1)$.
\end{enumerate}

For each $i=1,\ldots, m$, let
\[
D'_i = D_i \setminus (A'_{i-1} \cup A'_{i+1}).
\]
We give $D'_i$ a cubical structure $Y_i$ isomorphic to the cubical structure illustrated
in Figure \ref{fig:Cubical_disks}, for the even-indexed  $D'_i$ on the left and and the  odd-indexed  $D'_i$ on the right.

\begin{figure}[h!]
\begin{overpic}[scale=.25,unit=1mm]{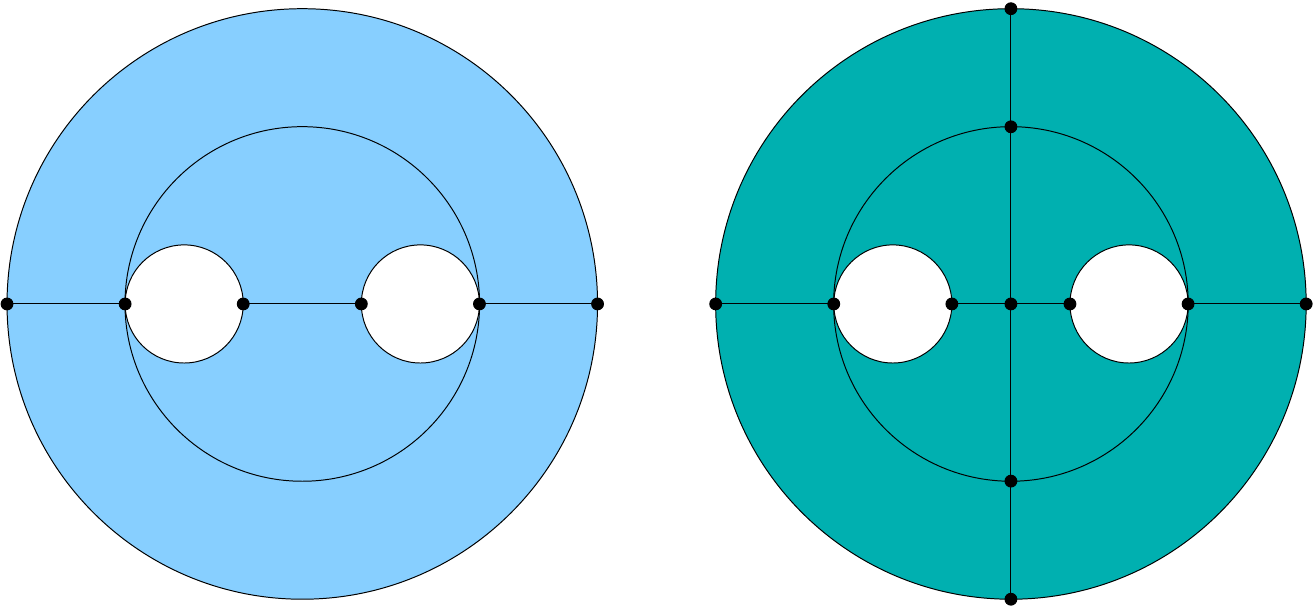} 
\end{overpic}
\caption{Cubical $2$-complexes for surfaces $D'_i$; the even-indexed on the left and the odd-indexed on the right.}
\label{fig:Cubical_disks}
\end{figure}

Further, we may assume
that the complexes $Y_i$ and $P_j$ agree on  $D_i \cap A'_j$ whenever the intersection is non-empty. Now there is a  well-defined $2$-complex $Y$ on
\[
\Sigma = \bigcup_{i=1}^m (D'_i \cup \partial A'_i) ,
\]
which contains each $P|_{\partial A'_i}$ and each $Y_i$ as a subcomplex. We note that at this stage $A\setminus \Sigma$ has $m+1$ components, each of which is homeomorphic to the product $(\bS^1\times \bS^1) \times [0,1)$.

\medskip
\noindent \emph{Step 3: Cubical $3$-complex $P$ on $A_0\setminus \interior (A_1\cup \cdots \cup A_m)$.} Having the preliminary cubical structures on $\partial A'_0$ and $\Sigma$, we now define a cubical structure on $A_0\setminus \interior (A_1\cup \cdots A_m)$ as follows.

Recall that, for each $k=1,\ldots, m/2$, the $3$-cell $G_k$ contains the vertical ring $A'_{2k}$ and two horizontal half rings $A'_{2k-1}\cap G_k$ and $A'_{2k+1}\cap G_k$. Moreover, we may assume that the cubical structure $Y$ on $\Sigma$ has been chosen so that each subcomplex $X|_{\Sigma \cap G_k}$ is a well-defined cubical subcomplex of $Y$  isomorphic to the first refinement $\Refine(P_0|_{G_k \cap \partial A'_0})$ of the complex $P_0$ on $G_k \cap \partial A'_0$. Furthermore, these isomorphisms are induced by the maps $\pi_k \colon G_k \cap \partial A'_0 \to \Sigma\cap G_k$; see  Figure \ref{fig:Antoine-slice} for an illustration. For this reason, we now replace the complex $R_0$ on $A_0$ by its first refinement $\Refine(R_0)$, and change the cubical complex $P_0$ accordingly.

We may now define the cubical complex $P$ on $A_0 \setminus \interior (A_1\cup \cdots \cup A_m)$ to be the unique cubical $3$-complex (up to isomorphism) for which
\begin{enumerate}
\item $P|_{A_0\setminus A'_0} = P_0$,
\item $P|_{A'_i\setminus \interior A_i} = P_i$ for each $i=1,\ldots, m$,
\item $P|_{A'_0\setminus \interior (A'_1\cup \cdots \cup A'_m\cup D'_1\cup \cdots \cup D'_m)}$ is isomorphic to the product $P_0|_{\partial A'_0} \times [0,1)$, and
\item $P|_{\Sigma}= Y$.
\end{enumerate}
Then $Y$ is a separating complex for $P$.

\begin{figure}[htp!]
\begin{overpic}[scale=.12,unit=1mm]{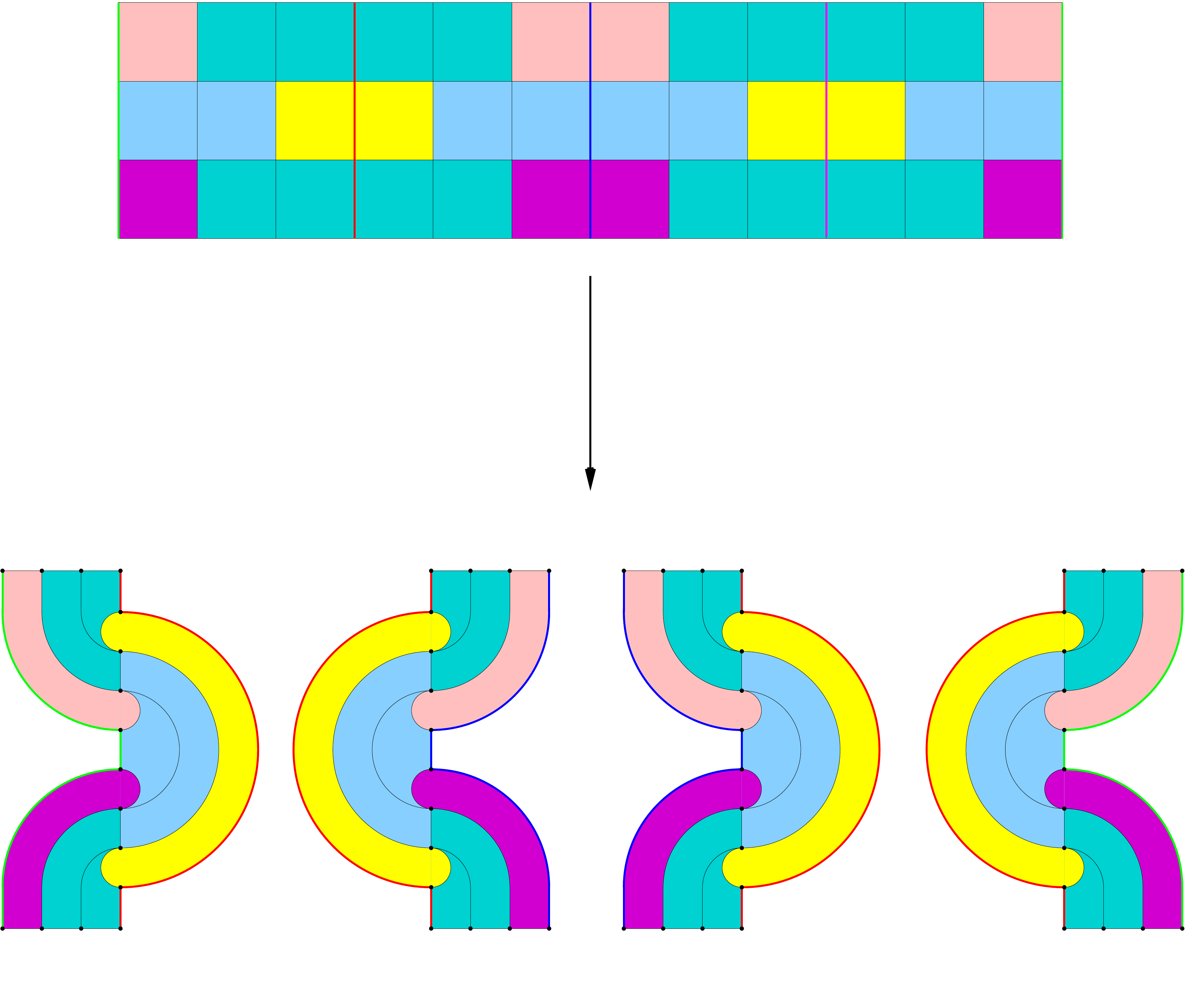} 
\put(37,37){\tiny $\pi_k$}
\end{overpic}
\caption{Cubical complex $\Refine(P_0|_{ \partial A'_0\cap B_k})$ unfolded (above); cubical complex $X|_{\Sigma\cap G_k}$ unfolded and sliced (below).}
\label{fig:Antoine-slice}
\end{figure}

\medskip
\noindent\emph{Step 4: Cubical $4$-complex $K$ on $A\times \bS^1$, and the separating complex $Z$.}
We now consider the $4$-dimensional pair $(A \times \bS^1, \mathcal A\times \bS^1)$. Recall that $\mathcal A= A_1\cup \cdots \cup A_m$.

Let
\[
K = P \times C_2
\]
be a cubical $4$-complex on $A\times \bS^1$ and set
\[
Z = Y \times C_2 \subset K.
\]
Then $Z$ is a separating complex for $K$

\FloatBarrier

\medskip
\noindent\emph{Step 5: Quasiregular branched cover on $(A\times \bS^1) \setminus \interior (\mathcal A\times \bS^1) $}
By Theorem \ref{intro-thm:qrcobordismII}, there exist a constant $\sK_1=\sK(n,K)\ge 1$, a refinement $\Refine^k(K)$ of $K$ for a sufficiently large $k \in \N$,
and a $\sK_1$-quasiregular mapping
\[
f\colon (A\times \bS^1) \setminus \interior (\mathcal A\times \bS^1) \to B^4 \setminus \interior (B_1\cup \cdots \cup B_m),
\]
where $B_1,\ldots, B_m$ are pairwise disjoint Euclidean $4$-balls in the interior of $B^4$, for which the restrictions $f|_{\partial (A\times \bS^1)} \colon\partial (A\times \bS^1) \to \partial B^4$ and $f|_{\partial A_i} \colon \partial A_i \to \partial B_i$, for $i=1,\ldots, m$, are all of the same degree, call it $c_0$,  and each of which is a $(\Refine^k(K))^\Delta$-Alexander map expanded by properly placed standard simple covers. Note that  we may arrange for the multiplicity of the map $f$ to be greater than any given $c\geq c_0$ by further refinement.

\FloatBarrier

\medskip
\noindent\emph{Step 6: Iterative construction.}
To complete the proof, we set $X_0=T=A\times \bS^1$, and
index the components $\{T_w\}$ of $X_k$ by words $w$ in letters $\{1,\ldots, m\}$ of length $k$, in such a way that
$T_{wi} \subset T_w$ for each word $w$ and $i\in \{1,\ldots,m\}$. In view of Remark \ref{rmk:Cantor_Semmes_package} on the construction of the Cantor set $X$, we may fix, for each word $w$ of length $k$, a sense-preserving $(b^k,L)$-quasi-similarity map $\varphi_w \colon A \times \bS^1 \to T_w$, which  is a homeomorphism between pairs  $(A \times \bS^1, \mathcal A\times \bS^1)$ and $(T_w, T_w \cap X_{k+1})$ and where $c>0, 0<b<1,$ and $L\geq1$ are constants.

Recall  that, after a partition based on the (model) complex $C_2\times C_2$,
the cubical complex on each surface $\partial A_i$, for $i=0,\ldots, m$, is a totally shellable refinement of $C_2\times C_2$.
Thus, we may fix these homeomorphisms $\varphi_w$ inductively so that, for each word $w$ and letter $i$, the induced complex $\varphi_{wi}(K)|_{\partial T_{wi}}$ is a totally shellable refinement of the complex $K|_{\partial T_{wi}}$, and that each $\varphi_w$ is  $\sK_2$-quasiconformal, with   $\sK_2\geq 1$ a constant independent of $w$ and $i$.

Fix for each $j=1,\ldots, m,$ a sense-preserving similarity $\lambda_j\colon B^4\to B_j$, and for a given word $w=i_1\cdots i_k$, write $\lambda_w = \lambda_{i_1}\circ \cdots \circ \lambda_{i_k}$.

Let $f\colon (A\times \bS^1) \setminus \interior (\mathcal A\times \bS^1) \to B^4 \setminus \interior (B_1\cup \cdots \cup B_m)$
be the map defined in Step 4. Since the cubical complexes on  $\partial A_i$ for $i=0,\ldots, m$, are totally shellable refinements of the same complex $C_2\times C_2$, the Alexander maps $f|_{\partial T_i}$ and  $\lambda_i \circ f \circ \varphi_0 \circ \varphi_i^{-1}|_{\partial T_i}$ from $\partial T_i$ to $\partial B_i$ are in fact BLD branched cover homotopic, with a BLD constant depending only on $n$ and $m$. Thus we may assume that the restrictions of
\[
f \colon T \setminus \interior (T\cap X_1) \to B^4 \setminus \interior(B_1\cup\cdots\cup B_m)
\]
to boundary components are conjugate to each other by (uniformly) quasi-similarities of the domains and similarities of the targets.

Repeating the construction of the map $f$ inductively on the length of the word $w$,
we obtain  a constant $\sK'\ge 1$ depending only on $\sK_1$, $\sK_2$,  and,
 for each $k\ge 1$ and a word $w$ of length $k$, a $\sK'$-quasiregular mapping
\[
f_w \colon T_w \setminus \interior (T_w\cap X_{k+1}) \to \lambda_w(B^4) \setminus (\cup_{j=1}^m \interior \lambda_{wi}(B^4))
\]
having the property that, for each $i=1,\ldots, m$, the mappings $f_w$ and $f_{wi}$ agree on $\partial T_{wi}$ and the map
\[
f \colon X_0 \setminus X\to B^4,
\]
defined by the condition $f|_{T_w\setminus \interior (T_w \cap X_{k+1})} = f_w$ for each $w$, is  well-defined and $\sK'$-quasiregular.

The quasiregular map $f\colon X_0\setminus X \to B^4$ extends now over to the Cantor set $X$ as a $\sK'$-quasiregular map.
\medskip

\noindent\emph{Step 7: Final extension.}
A similar construction using a separating complex, for example, extends $f|_{\partial X_0} \colon \partial X_0 \to \partial B^4$  to a $\sK''$-quasiregular map $\bS^4\setminus \interior X_0 \to \bS^4 \setminus \interior B^4$ for some constant $\sK''\ge 1$ depending only on $n$ and the original complex $K$.
By combining these two parts, we obtain a $\sK$-quasiregular map $f\colon \bS^n \to \bS^n$ with $\sK=\max\{ \sK',\sK''\}$.

By the construction, the local index  $i(x,f)=c'$ for each $x\in X$ for some number $c' $ which is at least as large as $c_0=\deg (f|_{\partial (A\times \bS^1)})$.
On the other hand, in the complement of the Cantor set $X$, the local degree of $f$ is determined by the cubical complex $K$ on $A\times \bS^1$, and hence there exists $m_0\ge 1$ for which $i(x,f)\le m_0$ for $x\not \in X$. This completes the construction.

\medskip
\noindent\emph{Step 7: The Jacobian estimates.} To verify the Jacobian estimates in the theorem, we make the following observations. First, since the restriction $F|_{X_0\setminus \interior X_1}$ is BLD, the  estimates hold in $X_0\setminus \interior X_1$. Second,
quasi-similarities $\varphi_w \colon T \to T_w$, for words $w$ of length $k$,  have scaling constant $b^k$  by the construction of the Blankinship necklace.
Third, after applying a  quasiconformal mapping of $\R^4$ which is identity outside $B^4$, we may assume the balls $B_1,\ldots, B_m$ are centered on the axis $\{(0,0,0)\}\times \R$ and have the same diameter $1/(2m)$.
Thus, similarities $\lambda_1,\ldots, \lambda_m$ have the same scaling constant $1/(2m)$.

It is now easy to observe that, for each $k\ge 1$ and a word $w$ of length $k$, the Jacobian $J_f|_{T_w\setminus X_{k+1}}$ is comparable to $\left((1/(2mb))^k\right)^4$. On the other hand, for $x\in T_w\setminus X_{k+1}$,  the distance $\dist(x,X)$ is comparable to $b^k$. Let
\[
s=-4\log(2mb)/\log(b).
\]
Then $\dist(x,X)^s$ and $J_f(x)$ are comparable at each point of differentiability in $T_w\setminus X_{k+1}$ for the mapping $f$.

This completes the proof of Theorem \ref{intro-thm:Heinonen-Rickman_Topology_fourdim}. \qed

\section{Bilipschitz and BLD parametrizations of metric spheres}\label{sec:intro_BLD_parametrization}\index{Bilipschitz parametrization of spheres}

It is known that, for each $n\geq 3$ and $n\neq 4$, there exists a topological sphere $(S,d)$, nearly indistinguishable from $\bS^n$ by classical analysis in the sense advocated by Semmes \cite{Semmes}, which is  not a bilipschitz copy of $\bS^n$ but  may be mapped onto $\bS^n$ by a BLD map.
Theorem \ref{intro-thm:Heinonen-Rickman_Topology_fourdim} may be used to furnish an example of this type for dimension $n=4$.

At the core of these examples is a theorem of Martio, Rickman and V\"ais\"al\"a \cite[III.5.1]{Rickman_book}: \emph{If $A$ is a closed subset of $\bS^n$ of zero $(n-2)$-dimensional Hausdorff measure, then  $\bS^n\setminus A$ is simply connected.}

\medskip

\emph{The case $n\geq 5$.}
Siebenmann and Sullivan  \cite{Siebenmann-Sullivan} observed that, for each $n\geq 5$, there exist finite $n$-dimensional polyhedra which are homeomorphic to the standard  $\bS^n$ but are not bilipschitz to $\bS^n$. Their assertion is  based on a deep work of Cannon \cite{cannon:recognition} and Edwards \cite{edwards:icm}, see also \cite{Edwards_arXiv}, which asserts that the double suspension $\Sigma^2 H^{n-2}$ of any $(n-2)$-dimensional homology sphere $H^{n-2}$ is homeomorphic to $\bS^n$, where the polyhedron $\Sigma^2 H^{n-2}$ is equipped with a canonical barycenter metric associated to a fixed triangulation of $H^{n-2}$.
Note that, by the theorem of Alexander, there exists a PL branched covering map $\Sigma^2 H^{n-2}\to \bS^n$.

The double suspension $\Sigma^2 H^{n-2}$ may be considered as the join $\bS^1 * H^{n-2}$.  The complement of the suspension circle $\Gamma$ in $\Sigma^2 H^{n-2}$ is not simply connected; therefore every homeomorphism $f\colon \Sigma^2 H^{n-2}\to \bS^n$ maps $\Gamma$ onto a curve $ f(\Gamma)$ whose complement in $\bS^n$ is not simply connected. Thus, by the theorem of Martio, Rickman and V\"ais\"al\"a,
$f(\Gamma)$ has positive $(n-2)$-dimensional Hausdorff measure. Therefore $f$ can not be H\"older continuous of order greater than $1/(n-2)$, in particular, not
bilipschitz. It was asked by Siebenmann and Sullivan in \cite{Siebenmann-Sullivan} \emph{whether $\Sigma^2 H^{n-2}$ and $\bS^n$ are quasisymmetrically equivalent.} This question seems inaccessible at the moment.

Since homology spheres are true spheres in dimensions one and two, the argument above is restricted to dimensions $n\geq 5$.

\medskip

\emph{The cases $n=3$ or $4$.} The argument, given here, leading to Corollary \ref{cor:parametrization}  is not new; it combines the discussions in \cite{Heinonen-Rickman_Topology}, \cite{Heinonen-Rickman_Duke}, and \cite{Semmes}.

David and Semmes introduced the notion of strong $A_{\infty}$-weights  in \cite{David-Semmes_A_infinity_1990} and \cite{Semmes_A_infinity_1993}. A strong $A_{\infty}$-weight $w$ is a nonnegative locally integrable function  in $\R^n$, which is doubling and for which  the distance function $d_w\colon \R^n \times \R^n \to [0,\infty)$ defined by
\[
d_w(x,y)=\left(\int_{B_{x,y}} w \,dx \right)^{1/n}
\]
for $x,y\in \R^n$, where $B_{x,y}$ is the unique $n$-ball containing $x, y\in \R^n$ with diameter $|x-y|$,
is comparable to a metric.

If $f\colon \bS^n \to \bS^n$ is a quasiregular map, then its Jacobian $J_f$ is a strong $A_{\infty}$-weight; see \cite{Heinonen-Koskela_Acta}.
Moreover, if $D_J$ is a metric comparable to  $d_{J_f}$, then the map $f\colon (\bS^n, D_J)\to \bS^n$ is BLD; see \cite[Proposition 3.1]{Heinonen-Rickman_Topology}.

Let $X$ be a self-similar Antoine's necklace in $\bS^3$ and $s>0$.
Semmes showed in \cite{Semmes} that
the function $w\colon \bS^n \to [0,\infty)$,
\[
x \mapsto \min\{1,\dist(x,X)^s\},
\]
is a strong $A_{\infty}$-weight for which the distance function $d_w$ is comparable to a metric $D_w$, and that the space $(\bS^3, D_w)$ is linearly locally contractible, has a Hausdorff measure comparable to the Lebesgue measure on $\bS^3$, and supports Sobolev and Poincar\'e inequalities. However, $(\bS^3, D_w)$ is not
bilipschitz to $\bS^3$ when $s> 3$. The argument follows  from the fact that, when $s>3$, $X$ has Hausdorff dimension less than one in $(\bS^3, D_w)$, but $\bS^3 \setminus X$ is not simply-connected.

There is nothing special about dimension $3$. Semmes' argument works for any quasi-self-similar Cantor set $X$  in $\bS^4$ as well. When we specialize the Cantor set $X$ and the quasiregular map $f$ to those in Theorem \ref{intro-thm:Heinonen-Rickman_Topology_fourdim}, the metrics $D_{J_f}$ and $D_w$ are comparable. Therefore, we have the following.

\begin{corollary}
\label{cor:parametrization}
There exists a metric $d$ on $\bS^4$ for which $(\bS^4,d)$ is not bilipschitz equivalent to $\bS^4$, but there is a BLD map $f\colon (\bS^4,d) \to \bS^4$. On the other hand, the space $(\bS^4,d)$ is linearly locally contractible, Ahlfors $4$-regular, and supports Sobolev and Poincar\'e inequalities.
\end{corollary}

If a tame Cantor set  is used in Semmes' argument instead,  then the space $(\bS^n, D_w)$ is bilipschitz equivalent to $\bS^n$ for any $s>0$; see \cite[Remark 4.24]{Semmes}.

\section{UQR mappings with wild Julia sets in dimension $4$}

Recall that a quasiregular self-mapping $f\colon M \to M$ of a Riemannian manifold $M$ is uniformly quasiregular (UQR), if there exists a constant $\sK>1$ for which $f$ and all its iterates are $\sK$-quasiregular. There are  UQR maps in $\bS^3$ whose Julia sets are wild Cantor sets \cite{Fletcher-Wu}.

Wild Cantor sets are known to exist in all dimensions $n\geq 3$. However, geometrically self-similar wild Cantor sets are known only in dimension three. In order for a Cantor set to be a potential candidate for the Julia set of a UQR map, it needs to be at least quasi-self-similar (Definition \ref{def:quasi_self_similar}). We are able to construct such a Cantor set which satisfies a sharper condition (Remark \ref{rmk:Cantor_Semmes_package}) sufficient for this purpose in $\R^4$. The Hopf theorem (Theorem \ref{thm:Hopf_theorem_v2}) is then used to build a UQR map, in dimension $4$, whose Julia set is the above wild Cantor set.

\introthmwildJuliasetdimfour*

We refer to \cite{Fletcher-Wu} for more discussion on the role of wild Julia sets in complex dynamics. In the following proof, we identify, using stereographic projection, $\bS^4$ with the one point compactification $\bar \R^4$ of $\R^4$.

\begin{proof}[Proof of Theorem \ref{intro-thm:wild_Julia_set_dim4}]
Let
\[
X = \bigcap_{k=0}^\infty X_k
\]
be the quasi-self-similar wild Cantor set constructed  in Appendix,  associated to the  parameters $b$, $m$, and $\rho$ subject to the relations \eqref{eq:rho_b} and \eqref{eq:m}.
We retain, from here on, all notations from this particular construction.

We assume that the initial $4$-tube $X_0$ is chosen to have pattern $T$ and that the number of linked $4$-tubes in $X_1$  is $m=48 d^3$ for some even integer $d$.
Denote by $B(r)$ the closed ball $B^4(0,r)$ in $\R^4$, and note from the construction that
\[
X_1= \bigcup_{j=1}^m X_{1,j} \subset X_0\subset B(2) \subset \R^4.
\]

We have now the essential partitions
\[
\overline{\R^4}=  (\overline{\R^4}\setminus B(4))  \cup (B( 4)\setminus B( 3))  \cup (B(3)\setminus X_1) \cup X_1
\]
and
\[
\overline{\R^4}=  (\overline{\R^4}\setminus B( 4^d)) \cup (B( 4^d)\setminus B( 3)) \cup (B(3)\setminus X_0) \cup X_0
\]
of $\overline{\R^4}$. As in \cite{Fletcher-Wu}, we construct a UQR map $f\colon \overline{\R^4} \to \overline{\R^4}$ for which
$f(X_1)= X_0$,
$f(B(3)\setminus X_1) = B(3)\setminus X_0$,
$f(B( 4)\setminus B( 3))= B( 4^d)\setminus B( 3)$, and
$f(\overline{\R^4}\setminus B(4)) = \overline{\R^4}\setminus B( 4^d)$.

\medskip

\noindent
\emph{Step 1:} Let $\varphi_j \colon X_0 \to X_{1,j}$ be the homeomorphisms  in Appendix, and let $f|_{X_1} \colon X_1\to X_0$ be the $m$-fold covering map satisfying
\[
f|_{X_{1,j}}=\varphi^{-1}_j
\]
for each $j=1,\ldots, m$.

\medskip
\noindent
\emph{Step 2:} We define next $f|_{B(3)\setminus X_1} \colon B(3)\setminus X_1 \to B(3) \setminus X_0$ to be the composition of two winding maps $\omega$ and $\omega'$,  and bilipschitz homeomorphisms of $\R^4$ as follows.

Let $\omega \colon \R^4 \to \R^4$ be the degree $m/2$ winding map
\[
 (x_1,x_2, r,\theta)\mapsto (x_1,x_2, r, \theta m/2),
\]
where $(r,\theta)$ are the polar coordinates in $\R^2$.

The winding map $\omega$ is a BLD map, which maps the triple $( B(3),X_0,X_1)$ onto the triple $(B(3), X_0, \omega(X_{1,1}) \cup \omega(X_{1,2}))$. Sets $\omega(X_{1,1})$ and $\omega (X_{1,2})$ remain linked inside $X_0$ and we may straighten them by a bilipschitz homeomorphism $\xi \colon \R^4 \to \R^4$, which is identity on $\R^4\setminus \interior X_0$, in such a way that the involution $\iota \colon X_0\to X_0$,
\[
(x_1,x_2,x_3,x_4) \mapsto (-x_1, x_2,x_3,-x_4),
\]
interchanges the images $\omega(X_{1,1})$ and $\omega(X_{1,2})$. In particular,
\[
\iota (\xi\circ \omega(X_0))=X_0,\,\, \iota (\xi \circ \omega (X_{1,1}))=\xi \circ \omega (X_{1,2}), \,\,\iota (\xi \circ \omega (X_{1,2}))=\xi \circ \omega (X_{1,1}).
\]

Let now $\omega' \colon \R^4 \to \R^4$ to be the winding map
\[
(x_1, r'\cos \theta', r'\sin \theta',x_4)\mapsto (x_1, r'\cos (2 \theta'), r'\sin (2\theta'),x_4),
\]
where $(r',\theta')$ are the polar coordinates in $\R^2$. So $\omega'$ is a degree $2$ sense preserving BLD map under which
\[
\omega' \circ \xi \circ \omega (X_{1,1})=\omega'\circ \xi \circ \omega (X_{1,2}),\,\quad  \omega'\circ \xi \circ \omega ( B(3))= B(3),
\]
and
\[
(\omega' \circ \xi \circ \omega (X_0),\, \omega' \circ \xi \circ \omega (X_{1,1}) )
\approx (B^3 \times \bS^1,\,\tau \times \bS^1),
\]
where $\tau\approx B^2 \times \bS^1$ is a solid torus contained in the interior of $B^3$.

Let now $\eta \colon \R^4 \to \R^4$ be a bilipschitz homeomorphism of $\R^4$ which is identity outside $B(3)$ and for which
\[
\eta\left( \omega' \circ \xi \circ \omega(X_{1,1})\right) =\eta \left( \omega' \circ \xi \circ \omega (X_{1,2}) \right)= X_0.
\]
Thus $\eta\circ \omega' \circ \xi \circ \omega|_{ X_{1,j}}\colon X_{1,j} \to X_0$ for all $j=1,\ldots, m$. We may further adjust the bilipschitz map $\eta$ in $\interior B(3)\setminus X_1$ and find a $4$-manifold $P\approx B^3\times \bS^1$ satisfying $X_0\subset \interior P\subset P \subset \interior B(2)$, and $P=\eta\left( \omega' \circ \xi \circ \omega(X_0)\right)$.

The composition
\[
f|_{ {B(3)}}= \eta\circ \omega' \circ \xi \circ \omega|_{ {B(3)}}\colon ( {B(3)}, X_1)\to ( {B(3)},  X_0)
\]
is a BLD map of degree $m$ which extends the already defined BLD map $f|_{X_1} \colon X_1 \to X_0$, and maps
\[
\left(B(3)\setminus P,\, P\setminus X_0,\, X_0\setminus X_1\right) \mapsto \left(B(3)\setminus B(2),\, B(2)\setminus P,\, P\setminus X_0\right).
\]

\medskip
\noindent \emph{Step 3:} Before defining the map $f \colon \overline{\R^4}\setminus  B( 4) \to \overline{\R^4} \setminus B( 4^d)$, we recall first so-called the \emph{Mayer's power map}. Mayer constructed, for any $n\geq 3$ and $d \geq 2$, a UQR map $p\colon \overline{\R^n} \to \overline{\R^n}$ of degree $d^{n-1}$, whose Julia set is the unit sphere $\bS^{n-1}$.
The restriction $p|_{\R^n \setminus \{0\}} \colon \R^n \setminus \{0\} \to \R^n\setminus \{0\}$ of $p$ is the higher dimensional counterpart of the degree $d^{n-1}$ power map. In fact, $p$ is derived from the Zorich map on $\R^{n-1}\times \R$ using a cylindrical structure $K\times \R$, where $K$ is the standard cubical partition of  $\R^{n-1}$; see \cite[Theorem 2]{Mayer} for the detailed construction.

As a preparation for the Hopf Theorem (Theorem \ref{thm:Hopf_theorem_v2}), we consider a modified Mayer's map $q\colon \overline{\R^4} \to \overline{\R^4}$ associated to a modified Zorich map given by a refined cylindrical structure $K^\Delta \times \R$, where  $K^\Delta$ is the canonical triangulation of $K$. In this triangulation a unit cube is subdivided into $48$ $3$-simplices. Thus our modified Mayer's map $q$ has degree $m= 48 d^3$.
We leave details to the interested reader.

We define now $f \colon \overline{\R^4} \to \overline{\R^4}$ in the complement of $B(4)$ to agree with $q$. Note that the sphere $\partial B(4)$ is mapped onto $\partial B(4^d)$ under the map $q$, and that  $f|_{\partial B(4)}=q|_{\partial B(4)}$ is a sense preserving cubical Alexander map of degree $m$.

\medskip
\noindent
\emph{Step 4:}
It remains to find a quasiregular mapping $B(4) \setminus B(3) \to B(4^d)\setminus B(3)$ extending the already defined parts of $f$.

Recall that $f|_{\partial B(3)}\colon \partial B(3)\to \partial B(3)$ is the composition of a degree $m/2$ winding map $\omega$ and a degree $2$ winding map $\omega'$, modulo bilipschitz adjustments. In view of Theorem \ref{thm:Hopf_theorem_v2}, we may deform $f|_{\partial B(3)}$ to a winding map on $\partial B(3)$ through a BLD map homotopy.
The homotopy yields a BLD extension $B(4)\setminus B(3) \to B(4^d)\setminus B(3)$ of the restrictions $f|_{\partial B(3)} \colon \partial B(3) \to B(3)$ and $f|_{\partial B(4)} \colon \partial B(4) \to \partial B(4^d)$.
This completes the construction of the quasiregular map $f\colon \overline{\R^4} \to \overline{\R^4}$.

\medskip
\noindent \emph{Final step:} The uniform quasiregularity of $f$ follows from two observations. First, by the construction of the Cantor set $X$, there exist $\lambda \ge 1$ and $L\ge 1$ for which the iterated compositions $f_k \circ \cdots \circ f_{k+\ell-1} \colon X_{k+\ell} \to X_k$ of the maps $f_j=f|_{X_{j+1}} \colon X_{j+1}\to X_j$ are $(\lambda^\ell,L)$-quasi-similarities. Second, the Mayer's map $q$ is UQR. The fact that $f$ has the wild Cantor set $X$ as its Julia set follows from the argument in \cite{Fletcher-Wu} almost verbatim.

This completes the proof of Theorem \ref{intro-thm:wild_Julia_set_dim4}.
\end{proof}

\part*{Appendix}

\section{A quasi-self-similar wild Cantor set in dimension $4$}
\label{sec:wild_Cantor}\index{Wild Julia set}

A Cantor set $X$ in $\R^n$ is \emph{tame} if there is a homeomorphism of $\R^n$ that maps $X$ onto the standard ternary Cantor set contained in a line. A Cantor set in $\R^n$ is \emph{wild} if it is not tame. Antoine constructed the first wild Cantor set in $\R^3$, which is now known as Antoine's necklace. Blankinship \cite{Blankinship} extended Antoine's construction to produce wild Cantor sets in $\R^n$ for every $n\geq 4$. \index{wild Cantor set} Constructions of wild Cantor sets are abundant in geometric topology, see e.g.~Bing \cite{Bing52} and Daverman and Edwards \cite{Daverman-Edwards} for more examples.

While there exist geometrically self-similar Antoine's necklaces allowing a sufficiently large number of tori at each stage \cite{Heinonen-Rickman_Topology}, \cite{Zeljko}, \cite{Pankka-Wu}, it is unknown whether geometrically self-similar wild Cantor sets exist in $\R^n$ for $n\geq 4$; see Garity and Repov{\v s} \cite[pp. 675-679]{Pearl_open_problems_topology}.

In this section we construct a wild Cantor set in $\R^4$ which is quasi-self-similar; we do not know whether quasi-self-similar wild Cantor sets exist in dimensions five or higher. The notion of quasi-self-similarity was introduced by McLaughlin \cite{McLaughlin}; see also \cite{Falconer}.

\begin{definition}\label{def:quasi_self_similar} \index{quasi-self-similarity}
 A nonempty set $X$ in a  metric space $(S,d)$ is \emph{$L$-quasi-self-similar for $L\ge 1$}, if there exists a radius $r_0>0$ such that, given any ball $B$ of radius $r<r_0$, there exists a map $f_B \colon B\cap X \to X$ satisfying
 \[
\frac{1}{L} \frac{r_0}{r} d(y,z) \le d(f_B(y),f_B(z)) \le L \frac{ r_0}{r} d(y,z)\]
for all $y,z\in B\cap X$.
A metric space $(X,d)$ is \emph{self-similar} if it is $1$-quasi-self-similar.
\end{definition}

\begin{theorem}\label{thm:AB_4_dim}
There exist geometrically quasi-self-similar wild Cantor sets in $\R^4$.
\end{theorem}

Topologically the Cantor set that we construct is an Antoine-Blankinship's necklace. We recall first the terminologies of the construction.
A solid $n$-tube in $\R^n,\, n\geq 3$, is a topological space homeomorphic to $B^2 \times (\bS^1)^{n-2}$.
Consider the embedding of $m$ linked solid $4$-tubes $T_1, T_2, \dots, T_m$  in an unknotted solid tube $T$ in $\R^n$ as  in Antoine \cite{Antoine} for $n=3$, or as in Blankinship \cite{Blankinship} for $n\geq 4$.
The Cantor set in question is the intersection
\[
X= \bigcap_{k=0}^\infty X_k,
\]
 where $X_k$ is a collection of $m^k$ disjoint solid $4$-tubes and for each tube $\tau$  in $X_k$ the triple $(\R^n, \tau, \tau\cap X_{k+1})$ is homeomorphic to $(\R^n, T, \bigcup_1^m T_j)$.  Chosen with care, the diameters of the components of $X_k$ approach zero as $k\to \infty$; hence $X$ is a Cantor set in $\R^n$.  See also \cite{Pankka-Vellis} for another description of Blankinship's construction.

 In dimension $3$, geometric self-similarity of $X$ can be reached by choosing $m$ sufficiently large;
 see \cite{Pankka-Wu} for related discussion.

In dimension $4$, we  construct a  wild Cantor set $X=\bigcap_{k=1}^\infty X_k$  in $\R^4$ whose difference sets $\{\tau\setminus X_{k+1}\colon \tau\in X_k \,\,\text{and}\,\, k\geq 0 \}$ belong to precisely  two  (geometric) similarity classes.
Because our construction follows that of Blankinship topologically, the
 wildness of $X$ follows from    \cite[Section 2]{Blankinship}.

\subsection{Construction}
For the proof of Theorem \ref{thm:AB_4_dim}, our  construction differs geometrically from that of Blankinship.

We first give an overview of the construction. To fit the steps together, we need Proposition \ref{prop:parameters} below.

Let  $\Phi\colon \R^4\to \R^4$ and $\Psi\colon \R^4\to \R^4$ be isometries defined by $x\mapsto A_\Phi x + e_3$ and $x\mapsto A_\Psi x + e_3$, respectively, where the linear mappings $A_\Phi$ and $A_\Psi$ are given by matrices
\[
A_\Phi = \left[ \begin{array}{cccc}
1 & 0 & 0 & 0 \\
0 & 0 & 0 & 1 \\
0 & 0 & -1 & 0 \\
0 & 1 & 0 & 0
\end{array} \right]
\quad \text{and} \quad
A_\Psi = \left[ \begin{array}{cccc}
0 & 0 & 1 & 0 \\
0 & 0 & 0 & 1 \\
1 & 0 & 0 & 0 \\
0 & 1 & 0 & 0
\end{array}\right]
\]
in the standard basis, and  $e_3$ is the vector $(0,0,1,0)$. In particular,
given $(x_1,x_2,x_3,x_4)$, we have
\[
\Phi(x_1,x_2,x_3,x_4)= (x_1,x_4,1-x_3,x_2)
\]
and
\[
\Psi(x_1,x_2,x_3,x_4)= (x_3,x_4,1+x_1,x_2),
\]
and that both isometries $\Phi$ and $\Psi$ are orientation preserving.
Heuristically, $A_\Phi$ is essentially an exchange of coordinates in $\{0\}\times \R
\times \{0\}\times \R$, and $A_\Psi$ exchanges the $\R^2$
factors of $\R^4=\R^2\times \R^2$.

Let  $\varrho \colon \R^4\to \R^4$ be the rotation
\[
(x_1,x_2,r,\theta) \mapsto (x_1,x_2,r, \theta + 2\pi/m)
\]
where $(r,\theta)$ are the polar coordinates in $\R^2$, that is,
\[(x_1,x_2,x_3, x_4) \mapsto \left(x_1,\, x_2,\, x_3 \cos(\frac{2\pi}{m})-x_4\sin(\frac{2\pi}{m}),\,  x_3\sin(\frac{2\pi}{m})+x_4\cos(\frac{2\pi}{m})\right)\]
in Cartesian coordinates.
Let also, for $j\in \mathbb Z_+$,
$\varrho^j\colon \R^4\to \R^4$ be the $j$th iterate of $\varrho$, that is, the map
\[
(x_1,x_2,r,\theta) \mapsto (x_1,x_2,r, \theta + 2j\pi/m).
\]

Let $b_0, b_1, c_0,$ and $ c_1 $ be constants in $(0,1)$, whose values will be fixed later in Lemma \ref{lem:dist}.
We denote $\rho = \min\{ c_0,c_1\}/10$, and fix constants $b\in (0,1)$ and $m\in 2\Z_+$ which satisfy
\begin{equation}\label{eq:rho_b}
  0<b<\min\{ b_0,b_1,\rho/10\},
\end{equation}
and
\begin{equation}\label{eq:m}
4 b^2/3 \le {2\pi}/{m} \le 3b^2/2.
\end{equation}
Let  $\lambda \colon \R^4\to \R^4$ be the scaling map
\[x\mapsto bx.\]

In the construction of the Cantor set, we iterate two geometric model configurations $T$ and $\widetilde T$ each of which is a solid $4$-tube. In what follows, $m$ is the number of $4$-tubes inside the initial configuration, and  $b$ is the scaling constant. Note  that
$b$ is not comparable to the reciprocal of $m$. We will return to this particular point later.

We now discuss the iteration process formally assuming that the $4$-tubes $T$ and $\widetilde T$ have been fixed.
We set for each $j=1,\ldots, m$,
\[
X_{1,j} = \left\{ \begin{array}{ll}
(\varrho^j \circ \Phi \circ \lambda)T, & j\ \text{even},\\
(\varrho^j \circ \Phi \circ \lambda) \widetilde T, & j\ \text{odd},
\end{array}\right.
\quad
\text{and}
\quad
\widetilde X_{1,j} = \left\{
\begin{array}{ll}
(\varrho^j \circ \Psi \circ \lambda) T, & j\ \text{even},\\
(\varrho^j \circ \Psi \circ \lambda) \widetilde T, & j\ \text{odd}.
\end{array}\right.
\]
Let also
\[
X_1 = \bigcup_{j=1}^m X_{1,j}\quad \text{and}\quad \widetilde X_1 = \bigcup_{j=1}^m \widetilde X_{1,j}.
\]
By suitable choices of $c_0, c_1, b_0$, and $b_1$, sets $X_1$ and $\widetilde X_1$ are unions of mutually disjoint $4$-tubes. Moreover in $X_1$ and also in $\widetilde X_1$, half of the $4$-tubes are similar to the model tube $T$ and half are similar to  the model tube $\widetilde T$. This choice allows us to establish the quasi-self-similarity (instead of the self-similarity) of the Cantor set.

For the iteration, we need the following proposition.

\begin{proposition}\label{prop:parameters}
There exist constants $ b_0, b_1, c_0,$ and $c_1 $ in $(0,1)$, and bilipschitz equivalent $4$-tubes $T$ and $\widetilde T$ for which the following holds: for each $j=1,\ldots, m$,
\begin{enumerate}
\item[(i)] $X_{1,j} \subset \interior T$ and $\widetilde X_{1,j} \subset \interior \widetilde T$;
\item[(ii)] for $i\ne j$, $X_{1,i}\cap X_{1,j}=\emptyset$ and $\widetilde X_{1,i} \cap \widetilde X_{1,j} = \emptyset$;
\item[(iii)] $X_{1,j}$ and $X_{1,j+1}$ (similarly  $\widetilde X_{1,j}$ and $\widetilde X_{1,j+1}$) are disjoint tubes, Hopf linked in $\R^4$; here $m+1 \equiv 1 (\text{mod}\,m)$.
\end{enumerate}
\end{proposition}

Assuming for now the validity of the proposition, we fix a bilipschitz homeomorphism (of pairs) $H \colon (T,X_1) \to (\widetilde T, \widetilde X_1)$ having the property that for each $j=1,\ldots,m$, the restriction $H|_{X_{1,j}}$ satisfies
\[
H|_{X_{1,j}} = \varrho^j \circ \Psi \circ \Phi^{-1} \circ \varrho^{-j}.
\]
Note that, in particular, each restriction $H|_{X_{1,j}} \colon X_{1,j} \to \widetilde X_{1,j}$ is an isometry.

For each $j=1,\ldots, m$, let $\varphi_j \colon T \to X_{1,j}$ be the homeomorphism
\[
\varphi_j = \left\{ \begin{array}{ll}
\varrho^j \circ \Phi \circ \lambda,& j\ \text{even},\\
\varrho^j \circ \Phi \circ \lambda \circ H, & j\ \text{odd}.
\end{array}\right.
\]
Then $X_1 = \bigcup_{j=1}^m \varphi_j T.$
We set, for each $k\ge 1$,
\[
X_{k+1} = \bigcup_{|\alpha|=k} \bigcup_{j=1}^m (\varphi_{\alpha_1}\circ \cdots \circ \varphi_{\alpha_k} \circ \varphi_j) T,
\]
where $\alpha = (\alpha_1,\ldots, \alpha_k) \in \{1,\ldots, m\}^k$, and set  $X_0 = T$ for completeness.
 Note again that, for each $k\in {\mathbb Z}_+$, half of the connected components of $X_k$ are similar to $T$ and half are similar to $\widetilde T$.

The intersection
\[
X = \bigcap_{k=0}^\infty X_k
\]
is a wild Cantor set. The quasi-self-similarity of $X$ follows from the fact that each map $\varphi_{\alpha_1}\circ \cdots \circ \varphi_{\alpha_k}$ is
$(b^k,L)$-quasi-similar, where $L$ is a bilipschitz constant for the mapping $H$. The wildness follows now directly from \cite[Section 2]{Blankinship}.

For the proof of Theorem \ref{thm:AB_4_dim}, it remains to  verify Proposition \ref{prop:parameters}.

\subsection{Proof of Proposition \ref{prop:parameters}}
Let $b\in(0,1/10)$ be the constant in \eqref{eq:rho_b} to be determined, and $m$ be an even integer satisfying \eqref{eq:m}.

Since a solid $4$-tube $B^2\times \bS^1\times \bS^1$ may be considered as the regular neighborhood of its core $\{0\} \times \bS^1\times \bS^1$ in $\R^4$, we first construct the toroidal cores.

Heuristically, starting with a torus $\bS^1(b)\times \bS^1(1)$ in $\R^4$, we  place $m$ tori comparable in size to
 $\bS^1(b) \times \bS^1(b^2)$ in such a way that their smaller generating circles $ \bS^1(b^2)$ are linked in $\R^3$, and go around the larger generating circle $\bS^1(1)$ of the initial torus. In the next step, corresponding to each one of these $m$ tori, we place $m$ smaller tori comparable in size to
$\bS^1(b^3) \times \bS^1(b^2)$ so that their  $ \bS^1(b^3)$-circles are linked in a Euclidean $3$-space and go around the  $\bS^1(b)$-circle of their predecessor. The construction may be continued by induction.
We note here that the scaling constant is roughly $b$, and we need $m\approx 1/b^2$ small circles to go around the previous generating circle. The condition \eqref{eq:m} stems from this observation.

\smallskip

\emph{Four circles.}
We fix four (families of) circles which are meridians and longitudes of the cores of two (families of) tubes.

For the first two families of tori, we set
\[
\gamma \equiv \gamma(b) = \{(0,x_2,x_3,0)\in \R^4\colon x_2^2+(x_3-1)^2=b^2\}
\]
and
\[
l \equiv l(b) = \{ (0,0,x_3,x_4) \in \R^4\colon x_3^2+x_4^2=(1-b)^2\}.
\]
For the next two families, we set
\[
\widetilde \gamma \equiv \widetilde \gamma(b) = \bS^1(b)\times \{(1,0)\}
\]
and
\[
\widetilde l \equiv \widetilde l(b) = \{(b,0)\} \times \bS^1.
\]
The circles $\gamma$ and $\widetilde \gamma$ are orthogonal and  they have a common center $e_3=(0,0,1,0)$, and circles $l$ and $\widetilde l$ are invariant under the rotation $\varrho$.

\smallskip

\emph{Toroidal cores.}
Let $\kappa$ be the round torus obtained by revolving $\gamma$ in $\R^4$ with respect to the hyperplane $P = \R^2\times \{(0,0)\}$ in $\R^4$, that is,
\[
\kappa\equiv \kappa(b) =\{(0,b  \cos \phi, (1+b\sin \phi)\cos\theta,   (1+b\sin \phi)\sin \theta)) \colon \phi, \theta \in [0, 2\pi] \}.
\]
Then $\gamma$ is a meridian and $l$ is a longitude of $\kappa$; they will be designated as the \emph{marked meridian} and the \emph{marked longitude} of $\kappa$, respectively.

Let $\widetilde \kappa=  \bS^1(b) \times \bS^1 \subset \R^2\times \R^2 $ be the \emph{flat} torus in $\R^4$, that is,
\[\widetilde \kappa \equiv \widetilde \kappa(b)=\{(b  \cos \psi, b \sin\psi, \cos \theta, \sin \theta  \colon \psi, \theta \in [0, 2\pi] \}.
\]
Then $\widetilde \gamma$ is a meridian and $\widetilde l$  is a longitude of $\widetilde \kappa$. In fact, $\widetilde \kappa$ is the surface of revolution of $\widetilde \gamma$ with respect to $P$. We call $\widetilde \gamma$ and $\widetilde l$ the \emph{marked meridian} and the \emph{marked longitude} of $\widetilde \kappa$, respectively.

We summarize the properties of toroidal cores $\kappa$ and $\widetilde \kappa$ with respect to the properties of $\Phi$ and $\Psi$ as follows.
Since $ \Phi(P)=\R\times \{(0,1)\}\times \R $, we have  that
\begin{enumerate}
\item the embedded $2$-tori $(\Phi\circ \lambda)\kappa$ and $(\Phi\circ \lambda)\widetilde \kappa$ are surfaces of revolution, with respect to the hyperplane $\R\times \{(0,1)\}\times \R$, of the marked meridians
\[
(\Phi\circ \lambda)\gamma=\{(0,0,x_3,x_4)\colon  (x_3-(1-b))^2+x_4^2 =b^4\},
\]
and
\[
(\Phi\circ \lambda)\widetilde \gamma=\{(x_1,0,1-b,x_4)\colon  x_1^2+x_4^2 =b^4\},
\]
respectively;
\item the marked meridians  $(\Phi\circ \lambda) \gamma$  and $(\Phi\circ \lambda)\widetilde \gamma$
have a common center $(0,0,1-b,0)$, which lies on the longitude $l$ of the torus $\kappa$, and the axle of tori $(\Phi\circ \lambda)\kappa$ and $(\Phi\circ \lambda)\widetilde \kappa$, by which we mean
the circle of revolution of the center $(0,0,1-b,0)$ with respect to the hyperplane $\Phi(P)$,  is the marked meridian $\gamma$ of  $\kappa$; and
\item
\[
\max\{\dist(y, \kappa)\colon y\in (\Phi\circ \lambda)\kappa \} \leq b^2,
\]
and
\[
\max\{\dist(y, \kappa)\colon y\in (\Phi\circ \lambda) \widetilde \kappa \} \leq b^2.
\]
\end{enumerate}
Similarly, regarding the embedding $\Psi \circ \lambda$, we have that
\begin{enumerate}
\item tori $(\Psi\circ \lambda)\kappa$ and $(\Psi\circ \lambda)\widetilde \kappa$ are the surfaces of revolution, with respect to the hyperplane $\{(0,0)\}\times \R^2=\Psi(P)$, of the meridians
\[
(\Psi\circ \lambda)\gamma =\{(x_1,0,1,x_4)\colon ( x_1-b)^2+x_4^2 =b^4 \},
\]
and
\[
(\Psi\circ \lambda)\widetilde \gamma=\{(b,0,x_3,x_4)\colon  (x_3-1)^2+x_4^2 =b^4\},
\]
respectively;
\item these two meridians  $(\Psi\circ \lambda) \gamma$  and $(\Psi\circ \lambda)\widetilde \gamma$ have a common  center $(b,0,1,0)$, which lies on the longitude  $\widetilde l$ of $\widetilde \kappa$, and
the  axle of $(\Psi\circ \lambda)\kappa$ and $(\Psi\circ \lambda)\widetilde \kappa$, i.e.,
the circle of revolution of the center $(b,0,1,0)$ with respect to $\Psi(P)$, is the marked meridian $\widetilde \gamma$ of  $\widetilde \kappa$; and
\item we have
\[\max\{\dist(y, \widetilde \kappa)\colon y\in (\Psi\circ \lambda)\kappa \}\leq b^2, \]
and
 \[ \max\{\dist(y, \widetilde \kappa)\colon y\in (\Psi\circ \lambda) \widetilde \kappa \}\leq b^2.\]
\end{enumerate}
\smallskip

\emph{Cyclically linked cores.}
For each $j=1,\ldots, m$, let $\sigma_j \subset \R^4$ be the circle
\[
\sigma_j = \left\{ \begin{array}{ll}
(\varrho^j\circ \Phi\circ \lambda)\gamma, & j\ \text{even}, \\
(\varrho^j\circ \Phi\circ \lambda)\widetilde \gamma, & j\ \text{odd}.
\end{array}\right.
\]
The circles $\sigma_1,\ldots, \sigma_m$ form a necklace chain in $\R\times \{0\}\times \R^2$. More precisely,
\begin{enumerate}
\item circles $\sigma_1,\ldots, \sigma_m$ are pairwise disjoint, and their centers $z_1,\ldots,z_m$ are equally spaced on the longitude $l$ of $\kappa$,
\item $\varrho^2 (\sigma_j)=\sigma_{j+2}$ for each $j=1,\ldots, m-2$, $\varrho^2 (\sigma_{m-1})=\sigma_1$, and $\varrho^2 (\sigma_m)= \sigma_2$, and
\item circles $\sigma_i$ and $\sigma_j$ are (Hopf) linked in  $\R\times \{0\}\times \R^2$ if and only if $i-j \equiv \pm 1 (\text{mod}\,m)$.
\end{enumerate}

Similarly circles
\[
\widetilde \sigma_j = \left\{
\begin{array}{ll}
(\varrho^j\circ \Psi\circ \lambda)\gamma, & j\ \text{even}, \\
(\varrho^j\circ \Psi\circ \lambda)\widetilde \gamma, & j\ \text{odd},
\end{array}\right.
\]
for $j= 1,\ldots, m$, also form a necklace chain in $\R\times \{0\}\times \R^2$.

Let $\tau_1,\ldots,\tau_m$ and $\widetilde \tau_1,\ldots,\widetilde \tau_m$, where
\[
\tau_j = \left\{ \begin{array}{ll}
(\varrho^j\circ \Phi\circ \lambda)\kappa, & j\ \text{even}, \\
(\varrho^j\circ \Phi\circ \lambda)\widetilde\kappa, & j\ \text{odd},
\end{array}\right.
\quad
\text{and}
\quad
\widetilde \tau_j = \left\{ \begin{array}{ll}
(\varrho^j\circ \Psi\circ \lambda)\kappa, & j\ \text{even}, \\
(\varrho^j\circ \Psi\circ \lambda)\widetilde\kappa, & j\ \text{odd},
\end{array}\right.
\]
be two sets of cyclically linked $2$-tori.
For small enough $b$, tori $\tau_1,\ldots,\tau_m$ are mutually disjoint; the same holds for tori $\widetilde \tau_1,\ldots,\widetilde \tau_m$. The distance between these tori may be estimated quantitatively as follows.

\begin{lemma}
\label{lem:dist}
There exist absolute constants $b_0>0$ and $c_0>0$ with the property  that if $0<b<b_0$, then
\[
\dist (\tau_i, \tau_j)\geq c_0 b^2
\]
for $i,j\in \{1,\ldots, m\}$, $i\ne j$.
Similarly, there exist absolute constants $b_1>0$ and $c_1>0$ so that if $0<b<b_1$ then
\[
\dist(\widetilde \tau_i, \widetilde \tau_j) \geq c_1 b^2
\]
for all $i,j \in \{1,\ldots, m\}$ and $i\ne j$.
\end{lemma}
\begin{proof}
Note that each set of distance estimates claimed in the lemma is rotational invariant. For the first claim, it suffices to check that
\begin{equation}
\label{eq:dist_1}
\dist ((\Phi\circ \lambda)\kappa, (\varrho\circ \Phi\circ \lambda)\widetilde\kappa) > c_0 b^2,
\end{equation}
for some absolute constants $b_0$ and $c_0$ in $(0,1)$ and for $0 <b< b_0$.

Recall that $\varrho$ is a rotation by an angle $2\pi/m$ and that  
\[4b^2/3 < 2\pi/m< 3b^2/2;\]
denote, in the following,  $\beta=2\pi/m$. Then
\[
(\Phi \circ \lambda)\kappa = p ([0,2\pi)^2)
\]
and
\[
(\varrho\circ \Phi \circ \lambda)\widetilde \kappa = q ([0,2\pi)^2),
\]
where $p\colon [0,2\pi)^2 \to \R^4$ is the mapping
\[
p(\phi,\theta) = \left(0, \, (b+b^2\sin \phi)\sin \theta, \, 1-(b+b^2\sin \phi)\cos\theta,\,  b^2 \cos \phi\right)
\]
for $\phi,\theta \in [0,2\pi]$, and $q\colon [0,2\pi)^2 \to \R^4$ is the mapping
\[
\begin{split}
q(\psi,\theta) = ( & b^2 \cos \psi,\,  b\sin\theta,\, (1-b\cos \theta)\cos \beta- b^2 \sin \psi \sin \beta,  \\
&(1-b\cos \theta) \sin \beta + b^2 \sin \psi \cos \beta )
\end{split}
\]
for $\psi,\theta \in [0,2\pi]$.

Since $\sum_{i=1}^4|x_i|\leq 2 (\sum_{i=1}^4 |x_i|^2)^{1/2}$  for $x\in \R^4$, a direct computation, using the fact that $\sin \beta= \beta + O(\beta^3)$ and $\cos \beta=1+O(\beta^2)$ as $\beta\to 0$,  yields the following estimates
\begin{align*}
&2 |p(\phi,\theta)-q(\psi,\theta')|\\
&\quad \geq b^2 |\cos \psi|+| b\sin\theta' -(b+b^2\sin \phi)\sin\theta |\\
&\qquad + |(1-b\cos \theta')\cos \beta- b^2 \sin \psi \sin \beta- (1- (b+b^2\sin \phi)\cos \theta)|\\
&\qquad +| (1-b\cos \theta') \sin \beta + b^2 \sin \psi \cos \beta- b^2 \cos \phi |\\
&\quad =  b^2 |\cos \psi|+| b\sin\theta' -(b+b^2\sin \phi)\sin\theta |
+ | b\cos \theta' -(b+b^2\sin \phi)\cos\theta | \\
&\qquad +| \beta + b^2 (\sin \psi -  \cos \phi )|+O(b^3)\\
&\quad \geq  b^2 |\cos \psi|+ b (| \sin\theta' -(1+ b \sin \phi) \sin\theta |^2 \\
&\qquad + |\cos \theta' -(1+ b\sin \phi)\cos \theta|^2)^{1/2}
 +| \beta + b^2 (\sin \psi -  \cos \phi )|+O(b^3) \\
&\quad = b^2 |\cos \psi|+ b(1 +(1-2(1+b\sin \phi) \cos(\theta -\theta') +b\sin \phi)^2)^{1/2}\\
&\qquad +| \beta + b^2 (\sin \psi -  \cos \phi )|+O(b^3)\\
&\quad \geq b^2 |\cos \psi|+b^2 |\sin \phi| +| \beta + b^2 (\sin \psi -  \cos \phi )|+O(b^3).
\end{align*}

If $| \beta + b^2 (\sin \psi -  \cos \phi )|\geq b^2/4$, then \eqref{eq:dist_1} holds trivially. Otherwise, $| \beta + b^2 (\sin \psi -  \cos \phi )| < b^2/4$, which yields
\[
{13}/{12}< \cos \phi -  \sin \psi < {7}/{4}.
\]
Thus $ \cos \phi >0$ and $\sin \psi <0.$
Hence, at least one of the two inequalities $0< \cos \phi  < {7}/{8}$ and $0< -\sin \psi  < {7}/{8} $ holds. As a consequence, either  $|\sin \phi| > 3/8$ or $| \cos \psi| > {3}/{8}$. In either case \eqref{eq:dist_1} holds true.

The estimate of the distances between tori $\widetilde \tau_1, \ldots, \widetilde \tau_m$ is similar. This completes the proof of the lemma.
\end{proof}

\emph{The $4$-tubes.} We are now ready to choose the $4$-tubes $T$ and $\widetilde T$, and to verify the claims of Proposition \ref{prop:parameters}
Let $b_0, b_1, c_0,$ and $ c_1$ be the constants in Lemma \ref{lem:dist}, and let $b, \rho$ and $m$ be constants satisfying \eqref{eq:rho_b} and \eqref{eq:m}.
We now define the model $4$-tubes by
\[
T \equiv T(\rho,b) = \{ x\in \R^4 \colon \dist(x,\kappa) \le \rho b\}
\]
and
\[
\widetilde T \equiv \widetilde T(\rho,b) = \{ x\in \R^4 \colon \dist(x,\widetilde \kappa) \le \rho b\}.
\]
Note that $T$ and $\widetilde T$ are not isometric, since $\kappa$ and $\widetilde \kappa$ are not isometric. So the $4$-tubes in $X_{1,1}, \ldots,X_{1,m}, \widetilde X_{1,1}, \ldots,\widetilde X_{1,m}$ are similar  to either $T$ or  $\widetilde T$. Since these tubes have cores $\tau_1,\ldots,\tau_m, \widetilde \tau_1,\ldots, \widetilde \tau_m$, respectively.
From \eqref{eq:rho_b} and Lemma \ref{lem:dist}, it follows that the $4$-tubes $X_{1,1}, \ldots,X_{1,m}$ are pairwise disjoint; the same holds
 for  $4$-tubes $\widetilde X_{1,1}, \ldots,\widetilde X_{1,m}$. So statement $(ii)$ in the claim of proposition holds.

Since $X_{1,1},\ldots, X_{1,m}$ are cyclically (Hopf) linked in $\R^4$ by the construction and the same holds for $4$-tubes $\widetilde X_{1,1}, \ldots,\widetilde X_{1,m}$. Thus statement $(iii)$ also holds.

Since the core tori $\tau$ and $\widetilde \tau$ are bilipschitz equivalent, the $4$-tubes $T$ and $\widetilde T$ are bilipschitz equivalent as claimed in the proposition.

\smallskip

It remains to verify statement $(i)$ in the proposition.
\smallskip

\begin{lemma} Under the conditions in \eqref{eq:rho_b}, $X_1 \subset \interior X_0$ and $\widetilde X_1 \subset \interior \widetilde X_0$.
\end{lemma}

\begin{proof} Recall that the core $\tau_m$ of $X_{1,m}$ is the torus $(\Phi\circ \lambda)\kappa$, and the axle of $\tau_m$ is the meridian $\gamma$ of $\kappa$. Recall also that $\max\{\dist(y, \kappa)\colon y\in (\Phi\circ \lambda)\kappa \} \leq b^2$. Therefore,
for $x\in X_{1,m}$, we have
\begin{eqnarray*}
\dist(x, \kappa) &\leq&  \dist(x, \tau_m)+\max\{\dist(y,\kappa)\colon y\in \tau_m\} \\
&\leq& \rho b^2 +b^2  <\rho b/5.
\end{eqnarray*}
Thus, $X_{1,m} \subset \interior X_0$ and, by rotation, $X_{1,j} \subset \interior X_0$ for all even $j$; the proof of inclusion for odd $j$ is similar. Hence $ X_1 \subset \interior  X_0$.
Similarly $\widetilde X_1 \subset \interior \widetilde X_0$.
\end{proof}

This completes the proof of Proposition \ref{prop:parameters} and the proof of Theorem \ref{thm:AB_4_dim}.

\smallskip

\begin{remark}
\label{rmk:Cantor_Semmes_package}
The wild Cantor set $X$ constructed for Theorem \ref{thm:AB_4_dim}
satisfies a condition sharper than the quasi-self-similarity. In fact, the Cantor set
\[
X=\bigcap_{k=0}^\infty \bigcup_{|\alpha|=k} \varphi_\alpha T,
\]
where, for each $\alpha\in \{1,\ldots, m\}^k$,  $\varphi_\alpha$ is a $(b^k,L)$-quasi-similarity,
and $0<b<1$ and $L\geq 1$ are constants.
\end{remark}

\begin{remark}
In the construction above, $T$ and $\widetilde T$ are both homeomorphic to $B^2\times (\bS^1)^2$ but are  not geometrically similar to each other. Therefore the wild Cantor set $X$ constructed above is only quasi-self-similar.

We do not know whether  the $4$-tubes $T$ and $\widetilde T$ may be chosen to be the same set. Such a choice, if possible,
would yield a self-similar (instead of quasi-self-similar) wild Cantor set.
\end{remark}

\bibliographystyle{abbrv}
\bibliography{Deformation}

\printindex

\end{document}